\documentclass[final,onefignum,onetabnum]{siamart171218}

\usepackage{amsopn}
\usepackage{amsfonts}
\usepackage{amsmath,amsbsy,amsgen,amscd,amssymb,bm}
\usepackage{mathtools}
\usepackage{dsfont}
\usepackage{stmaryrd}
\usepackage{pifont}

\usepackage[scaled=1.2]{urwchancal}
\usepackage{mathrsfs}

\usepackage{graphicx}
\graphicspath{{figures/}}

\usepackage{tikz}

\usepackage{epstopdf}

\ifpdf
  \DeclareGraphicsExtensions{.eps,.pdf,.pdf,.jpg}
\else
  \DeclareGraphicsExtensions{.eps}
\fi

\definecolor{dark-gray}{gray}{0.3}
\definecolor{dkgray}{rgb}{.4,.4,.4}
\definecolor{dkblue}{rgb}{0,0,.5}
\definecolor{medblue}{rgb}{0,0,.75}
\definecolor{paleblue}{rgb}{0.85,1,1}
\definecolor{rust}{rgb}{0.5,0.1,0.1}

\usepackage{multicol}
\usepackage{float}
\usepackage{caption} %
\usepackage{subcaption}
\usepackage[font=small,margin=0.25in,labelfont={sc},labelsep={colon}]{caption}
\usepackage{colortbl}

\usepackage{enumitem}
\setlist[enumerate]{leftmargin=.5in}
\setlist[itemize]{leftmargin=.5in}

\usepackage[noend]{algpseudocode}
\usepackage{algorithm,algorithmicx}

\algrenewcommand\alglinenumber[1]{\sf\scriptsize\color{black}{#1}}
\algrenewcommand\algorithmicrequire{\textbf{Input:}}
\algrenewcommand\algorithmicensure{\textbf{Output:}}

\algdef{SE}[SUBALG]{Indent}{EndIndent}{}{\algorithmicend\ }%
\algtext*{Indent}
\algtext*{EndIndent}

\newsiamthm{claim}{Claim}
\crefname{claim}{Claim}{Claims}

\newsiamthm{fact}{Fact}
\crefname{fact}{Fact}{Facts}

\newsiamremark{remark}{Remark}
\crefname{remark}{Remark}{Remarks}

\newsiamremark{warning}{Warning}
\crefname{warning}{Warning}{Warnings}

\newsiamremark{example}{Example}
\crefname{example}{Example}{Examples}

\newcommand{\R}{\mathbb{R}}
\newcommand{\C}{\mathbb{C}}
\newcommand{\F}{\mathbb{F}}

\newcommand{\eps}{\varepsilon}
\newcommand{\econst}{\mathrm{e}}

\newcommand{\vct}[1]{\bm{#1}}
\newcommand{\mtx}[1]{\bm{#1}}

\newcommand{\Id}{\mathbf{I}}

\newcommand{\trace}{\operatorname{tr}}
\newcommand{\rank}{\operatorname{rank}}
\newcommand{\range}{\operatorname{range}}
\newcommand{\diag}{\operatorname{diag}}

\newcommand{\abs}[1]{\vert #1 \vert}
\newcommand{\norm}[1]{\Vert #1 \Vert}

\newcommand{\fnorm}[1]{\norm{#1}_{2}}
\newcommand{\fnormsq}[1]{\fnorm{#1}^2}

\newcommand{\lowrank}[2]{\llbracket {#1} \rrbracket_{#2}}

\newcommand{\Var}{\operatorname{Var}}
\newcommand{\Expect}{\operatorname{\mathbb{E}}}
\newcommand{\Prob}[1]{\mathbb{P}\left\{ #1 \right\}}

\newcommand{\err}{\mathrm{err}}

\title{Streaming Low-Rank Matrix Approximation \\ with an Application to Scientific Simulation%
\thanks{Date: 22 March 2017. Revised: 16 July 2018 and 22 February 2019.}
\funding{JAT was supported in part by ONR Awards N00014-11-1002, N00014-17-1-214, N00014-17-1-2146, and the Gordon \& Betty Moore Foundation.
MU was supported in part by DARPA Award FA8750-17-2-0101.
VC has received funding from the European Research Council (ERC) under the European Union's Horizon 2020 research and innovation programme under the grant agreement number 725594 (time-data) and the Swiss National Science Foundation (SNSF) under the grant number 200021\_178865.}}

\headers{Streaming Matrix Approximation}
{Tropp, Yurtsever, Udell, and Cevher}

\author{Joel A.~Tropp%
\thanks{California Institute of Technology, Pasadena, CA (\email{jtropp@cms.caltech.edu}).}%
\and Alp Yurtsever%
\thanks{{\'E}cole Polytechnique F{\'e}d{\'e}rale de Lausanne, Lausanne, Switzerland (\email{alp.yurtsever@epfl.ch}, \email{volkan.cevher@epfl.ch}).}%
\and Madeleine Udell%
\thanks{Cornell University, Ithaca, NY (\email{udell@cornell.edu}).}
\and Volkan Cevher\footnotemark[3]}

\ifpdf
\hypersetup{
  pdftitle={Streaming Low-Rank Matrix Approximation with an Application to Scientific Simulation},
  pdfauthor={J.~A.~Tropp, A.~Yurtsever, M.~Udell, and V.~Cevher}
}
\fi

\begin{document}

\maketitle

\begin{abstract}
This paper argues that randomized linear sketching is a natural
tool for on-the-fly compression of data matrices that arise from
large-scale scientific simulations and data collection.
The technical contribution consists in a new algorithm for
constructing an accurate low-rank approximation of
a matrix from streaming data.  This method is
accompanied by an \emph{a priori} analysis that allows the
user to set algorithm parameters with confidence
and an \emph{a posteriori} error estimator that allows the
user to validate the quality of the reconstructed matrix.
In comparison to previous techniques, the new method
achieves smaller relative approximation errors
and is less sensitive to parameter choices.
As concrete applications, the paper outlines how
the algorithm can be used to compress a
Navier--Stokes simulation and a sea surface temperature dataset.
\end{abstract}



\section{Motivation}

Computer simulations of scientific models often generate data matrices that
are too large to store, process, or transmit in full.  This challenge arises
in a huge number of fields, including weather and climate
forecasting~\cite{woodring2011revisiting,drake2014climate,baker2014methodology},
heat transfer and fluid flow~\cite{patankar1980numerical, bejan2013convection},
computational fluid dynamics %
\cite{baurle2004modeling,garnier2009large},
and aircraft design~\cite{menter2003ten,sagaut2006large}.
Similar exigencies can arise with automated methods
for acquiring large volumes of scientific data~\cite{CERN19:What-Record}.

In these settings, the data matrix often has a decaying singular value spectrum,
so it admits an accurate low-rank approximation.
For some downstream applications,
the approximation serves as well as---or even better than---the full matrix~\cite{son2014data,calhoun2018exploring}.
Indeed, the approximation is easier to manipulate,
and it can expose latent structure.
This observation raises the question of how best to compute
a low-rank approximation of a matrix of scientific data
with limited storage, arithmetic, and communication.

The main purpose of this paper is to argue that sketching methods
from the field of randomized linear algebra~\cite{WLRT08:Fast-Randomized,CW09:Numerical-Linear,
HMT11:Finding-Structure,Mah11:Randomized-Algorithms,Woo14:Sketching-Tool,BWZ16:Optimal-Principal,
GLPW16:Frequent-Directions,TYUC17:Practical-Sketching,TYUC17:Fixed-Rank-Approximation}
have tremendous potential in this context.
As we will explain, these algorithms can inexpensively
maintain a summary, or \emph{sketch}, of the data as it is being generated.
After the data collection process terminates, we can extract
a near-optimal low-rank approximation from the sketch.
This approximation is accompanied by an \emph{a posteriori}
error estimate.

The second purpose of this paper is to design, analyze, and test
a new sketching algorithm that is suitable for handling scientific data.
We will build out the theoretical infrastructure needed for
practitioners to deploy this algorithm with confidence.
We will also demonstrate that the method is effective for
some small- and medium-scale examples, including a computer
simulation of the Navier--Stokes equations and a
high-resolution sea surface temperature dataset~\cite{NOAA19:Sea-Surface-Temperature}.

\subsection{Streaming, Sketching, and Matrix Approximation}

Let us begin with a brief introduction to streaming data and sketching,
as they apply to the problem of low-rank matrix approximation.
This abstract presentation will solidify into a concrete algorithm
in \cref{sec:sketching,sec:aposterior}.  The explanation borrows heavily
from our previous paper~\cite{TYUC17:Practical-Sketching},
which contains more details and context.

\subsubsection{Streaming}

We are interested in acquiring a compressed representation %
of an enormous matrix $\mtx{A} \in \F^{m \times n}$ where $\F = \R$ or $\F = \C$.
This work focuses on a setting where the matrix
is presented as a long sequence of ``simple'' linear updates:
\begin{equation} \label{eqn:data-stream}
\mtx{A} = \mtx{H}_1 + \mtx{H}_2 + \mtx{H}_3 + \cdots.
\end{equation}
In applications, each innovation $\mtx{H}_i$ is sparse,
low-rank, or enjoys another favorable structure.  
The challenge arises because we do not wish to store the full matrix $\mtx{A}$,
and we %
cannot revisit the innovation $\mtx{H}_i$ after processing it.
The formula~\cref{eqn:data-stream} describes a particular type of
\emph{streaming data model}~\cite{Mut08:Data-Streams,CW09:Numerical-Linear,Woo14:Sketching-Tool}.

\subsubsection{Sketching}

To manage the data stream~\cref{eqn:data-stream},
we can use a \emph{randomized linear sketch}~\cite{AMS96:Space-Complexity,AGMS02:Tracking-Join}.
Before any data arrives, we draw and fix a random linear map
$\mathscr{S} : \F^{m \times n} \to \F^d$, called a \emph{sketching operator}.
Instead of keeping $\mtx{A}$ in working memory,
we retain only the image $\mathscr{S}(\mtx{A})$.
This image is called a \emph{sketch} of the matrix.
The dimension $d$ of the sketch is much smaller
than the dimension $mn$ of the matrix space,
so the sketching operator compresses the data matrix. 
Nonetheless, because of the randomness, a well-designed
sketching operator is likely to yield useful
information about any matrix $\mtx{A}$ that is statistically independent
from $\mathscr{S}$.

Sketches and data streams enjoy a natural synergy.
If the matrix $\mtx{A}$ is presented via the data stream~\cref{eqn:data-stream},
the linearity of the sketching operator ensures that
$$
\mathscr{S}(\mtx{A}) = \mathscr{S}(\mtx{H}_1) + \mathscr{S}(\mtx{H}_2) + \mathscr{S}(\mtx{H}_3) + \cdots.
$$
In other words, we can process an innovation $\mtx{H}_i$ by forming
$\mathscr{S}(\mtx{H}_i)$ and adding it to the current value of the sketch.
This update can be performed efficiently when $\mtx{H}_i$ is structured.
It is a striking fact~\cite{LNW14:Turnstile-Streaming} that
randomized linear sketches are essentially the only mechanism
for tracking a general data stream of the form~\eqref{eqn:data-stream}.

\subsubsection{Matrix Approximation}

After the updating process terminates, we need to extract a low-rank approximation
of the data matrix $\mtx{A}$ from the sketch $\mathscr{S}(\mtx{A})$.
More precisely, we report a rank-$r$ matrix $\hat{\mtx{A}}_r$, in factored
form,
that satisfies
\begin{equation} \label{eqn:approx-err-intro}
\fnorm{ \mtx{A} - \hat{\mtx{A}}_r }
	\approx \min_{\rank \mtx{B} \leq r}
	\fnorm{ \mtx{A} - \mtx{B} },
\end{equation}
where $\fnorm{\cdot}$ is the Schatten 2-norm (also known as the Frobenius norm).
We can also exploit the sketch to compute an \emph{a posteriori} estimate
of the error:
$$
\err_2( \hat{\mtx{A}}_r ) \approx \fnorm{ \mtx{A} - \hat{\mtx{A}}_r }.
$$
This estimator helps us to select the precise rank $r$
of the approximation.

\subsubsection{Goals}

Our objective is to design a sketch rich enough
to support these operations.
Since a rank-$r$ matrix has about $2r(m+n)$ degrees
of freedom, the sketch ideally should have size $d = \Theta( r(m+n) )$.
We want to compute the approximation using $\mathcal{O}( r^2 (m + n))$
floating-point operations, the cost of orthogonalizing $r$ vectors.

Existing one-pass SVD algorithms
(e.g.,~\cite{WLRT08:Fast-Randomized,CW09:Numerical-Linear,HMT11:Finding-Structure,Woo14:Sketching-Tool,
BWZ16:Optimal-Principal-STOC,Upa16:Fast-Space-Optimal,TYUC17:Practical-Sketching})
already meet these desiderata.  Nevertheless, there remains room for
improvement~\cite[Sec.~7.6]{TYUC17:Practical-Sketching}.

\subsubsection{Contributions}

The main technical contribution of this paper is a new sketch-based
 algorithm
for computing a low-rank approximation of a matrix from streaming data.
The new algorithm is a hybrid of the methods from~\cite[Thm.~12]{Upa16:Fast-Space-Optimal}
and~\cite[Alg.~7]{TYUC17:Practical-Sketching} that improves on the performance
of its predecessors.  Here are the key features of our work: %

\vspace{0.5pc}

\begin{itemize} \setlength{\itemsep}{0.5pc}
\item The new method can achieve a near-optimal relative approximation~\cref{eqn:approx-err-intro}
when the input matrix has a decaying singular value spectrum.  In particular,
our approach is more accurate than existing methods, especially when the storage budget is small.
As a consequence, the new method delivers higher-quality estimates of leading singular vectors.
(\Cref{sec:numerics})

\item	The algorithm is accompanied by \emph{a priori} error bounds that help
us set the parameters of the sketch reliably.  The new method is less sensitive
to the choice of sketch parameters and to the truncation rank, as compared with existing methods.
(\Cref{sec:theory,sec:numerics})

\item	Our toolkit includes an \emph{a posteriori} error estimator
for validating the quality of the approximation. %
This estimator also provides a principled mechanism
for selecting the precise rank of the final approximation.
(\Cref{sec:aposterior})

\item The method treats the two matrix dimensions symmetrically.  As a consequence,
we can extend it to obtain an algorithm for low-rank Tucker approximation of a
tensor from streaming data.  See our follow-up paper~\cite{SGTU19:Low-Rank-Tucker}.
\end{itemize}

\vspace{0.5pc}

For scientific simulation and data analysis,
these advances are significant because they
allow us to approximate the truncated singular value
decomposition of a huge matrix accurately and with
minimal resource usage.

\subsection{Application to Scientific Simulation}
\label{sec:vignette}

As we have mentioned, it is often desirable to reduce scientific data
before we submit it to further processing.
This section outlines some of the techniques that are commonly used for this purpose,
and it argues that randomized linear sketching may offer a better solution.

\subsubsection{Dynamical Model for a Simulation}
\label{sec:dynamical-model}

In many cases, we can model a simulation as a process that computes
the state $\vct{a}_{t+1} \in \F^m$ of a system at time $t + 1$ from
the state $\vct{a}_t \in \F^m$ of the system at time $t$.
We may collect the data generated by the simulation into a matrix
$\mtx{A} = [ \vct{a}_1, \dots, \vct{a}_n ] \in \F^{m \times n}$.
In scientific applications, it is common that this matrix has a
decaying singular value spectrum.

The dimension $m$ of the state typically increases with the resolution of the simulation,
and it can be very big.
The time horizon $n$ can also be large, especially for problems involving
multiple time scales and for ``stiff'' equations that have high sensitivity
to numerical errors.  In some settings, we may not even know the time horizon $n$
or the dimension $m$ of the state variable in advance.

\subsubsection{On-the-Fly Compression via Sketching}

Let us explain how sketching interacts with the dynamical model
from~\cref{sec:dynamical-model}
For simplicity, assume that the dimensions $m$ and $n$ of
the data matrix $\mtx{A} \in \F^{m \times n}$ are known.
Draw and fix a randomized linear sketching operator
$\mathscr{S} : \F^{m \times n} \to \F^d$.

We can view the dynamical model for the simulation
as an instance of the data stream~\cref{eqn:data-stream}:
$$
\mtx{A} = \vct{a}_1 \mathbf{e}_1^* + \vct{a}_2 \mathbf{e}_2^* + \vct{a}_3 \mathbf{e}_3^* + \cdots.
$$
Here, $\mathbf{e}_i$ is the $i$th standard basis vector in $\F^n$.
The sketch $\vct{x} = \mathscr{S}(\mtx{A}) \in \F^d$ evolves as
$$
\vct{x} = \mathscr{S}(\vct{a}_1 \mathbf{e}_1^*) + 
	\mathscr{S}(\vct{a}_2 \mathbf{e}_2^*) + \mathscr{S}(\vct{a}_3 \mathbf{e}_3^*) + \cdots.
$$
Each time the simulation generates a new state $\vct{a}_t$,
we update the sketch $\vct{x}$ to reflect the innovation
$\vct{a}_t\mathbf{e}_t^*$ to the data matrix $\mtx{A}$.
We can exploit the fact that the
innovation is a rank-one matrix to ensure that this computation has negligible
incremental cost.  After sketching the new state, we write it to
external memory or simply discard it. %
Once the simulation is complete, we can extract a provably good low-rank
approximation from the sketch, along with an error estimate.

\subsubsection{Compression of Scientific Data: Current Art}

At present, computational scientists rely on several other strategies
for data reduction.  %
One standard practice is to collect the full data matrix %
and then to compress it.  Methods include direct computation of a low-rank
matrix or tensor approximation~\cite{zhou2014decomposition, austin2016parallel}
or fitting a statistical model~\cite{castruccio2016compressing, guinness2017compression, malik2018principal}.
These approaches have high storage costs, %
and they entail communication of large volumes of data. %

There are also some techniques for compressing simulation output %
as it is generated.  One approach is to store only a subset of the columns of
the data matrix (``snapshots'' or ``checkpointing''),
instead of keeping the full trajectory~\cite{GW00:Algorithm799-Revolve,HWS06:Optimal-Memory-Reduced}.
Another approach is to maintain a truncated singular value decomposition (SVD)
using a rank-one updating method~\cite{Bra06:Fast-Low-Rank,RPW18:Geometric-Subspace}.
Both techniques have the disadvantage that they do not preserve
a complete, consistent view of the data matrix.  The rank-one updating
method also incurs a substantial computational cost at each step.

\subsubsection{Contributions}

We believe that randomized linear sketching resolves many of the shortcomings
of earlier data reduction methods for scientific applications.
We will show by example (\cref{sec:numerics}) that our new sketching algorithm
can be used to compress scientific data drawn from several applications:

\vspace{0.5pc}

\begin{itemize} \setlength{\itemsep}{0.5pc}
\item	We apply the method to a 430 Megabyte (MB) data matrix from a direct numerical simulation,
via the Navier--Stokes equations, of vortex shedding
from a cylinder in a two-dimensional channel flow. %

\item	The method is used to approximate a 1.1 Gigabyte (GB) temperature dataset
collected at a network of weather stations in the northeastern United States.

\item	We can invoke the sketching algorithm as a module in an optimization algorithm
for solving a large-scale phase retrieval problem that arises in microscopic imaging
via Fourier ptychography.  The full matrix would require over 5 GB of storage.

\item	As a larger-scale example, we show that our methodology allows us to compute
an accurate truncated SVD of a sea surface temperature dataset, which requires over 75 GB
in double precision.
This experiment is performed without any adjustment of parameters or other retrospection.
\end{itemize}

\vspace{0.5pc}

These demonstrations support our assertion that sketching is
a powerful tool for managing large-scale data from scientific simulations and measurement
processes.  We have written this paper to motivate computational scientists
to consider sketching in their own applications.

\subsection{Roadmap}

In \cref{sec:sketching}, we give a detailed presentation of
the proposed method and its relationship to earlier work.
We provide an informative mathematical analysis that explains the behavior
of our algorithm (\cref{sec:theory}), and we describe how to construct
\emph{a posteriori} error estimates (\cref{sec:aposterior}).
We also discuss implementation issues (\cref{sec:implementation}),
and we present extensive numerical experiments on real and simulated data (\cref{sec:numerics}).

\subsection{Notation}

We use $\F$ for the scalar field, which is real $\R$ or complex $\C$.
The symbol ${}^*$ refers to the (conjugate) transpose of a matrix or vector.
The dagger ${}^\dagger$ denotes the Moore--Penrose pseudoinverse.
We write $\norm{\cdot}_p$ for the Schatten $p$-norm for $p \in [1, \infty]$.
The map $\lowrank{\cdot}{r}$ returns any (simultaneous)
best rank-$r$ approximation of its argument with respect
to the Schatten $p$-norms~\cite[Sec.~6]{Hig89:Matrix-Nearness}.

\section{Sketching and Low-Rank Approximation of a Matrix}
\label{sec:sketching}

Let us describe the basic procedure for sketching a matrix
and for computing a low-rank approximation from the sketch.
We discuss prior work in \cref{sec:provenance}.
See \cref{sec:implementation} for implementation,
\cref{sec:theory} for parameter selection, and
\cref{sec:aposterior} for error estimation.

\subsection{Dimension Reduction Maps}

We will use dimension reduction to collect information about an input matrix.
Assume that $d \leq N$.  A \emph{randomized linear dimension reduction map}
is a random matrix $\mtx{\Xi} \in \F^{d \times N}$ with the property that
\begin{equation} \label{eqn:embedding}
\Expect \fnormsq{ \mtx{\Xi} \vct{u} } = \mathrm{const} \cdot \fnormsq{\vct{u}}
\quad\text{for all $\vct{u} \in \F^{N}$.}
\end{equation}
In other words, the map reduces a vector of dimension $N$ to dimension $d$,
but it still preserves Euclidean distances on average.
It is also desirable that we can store the map $\mtx{\Xi}$
and apply it to vectors efficiently.
See \cref{sec:dim-red-maps} for concrete examples.

\subsection{The Input Matrix}

Let $\mtx{A} \in \F^{m \times n}$ be an arbitrary matrix
that we wish to approximate.  In many applications where
sketching is appropriate, the matrix is presented implicitly
as a sequence of linear updates; see \cref{sec:updates}.

To apply sketching methods for low-rank matrix approximation,
the user needs to specify a target rank $r_0$. %
The target rank $r_0$ is a rough estimate for the final rank of
the approximation, and it influences the choice of the sketch size.
We can exploit \emph{a posteriori} information to select the
final rank; see~\cref{sec:diagnose-spectrum}.

\begin{remark}[Unknown Dimensions]
For simplicity, we assume the matrix dimensions are known in advance.
The framework can be modified to handle matrices with growing dimensions,
such as a simulation with an unspecified time horizon.
\end{remark}

\subsection{The Sketch}

Let us describe the sketching operators we use
to acquire data about the input matrix.
The sketching operators are parameterized by
a ``range'' parameter $k$ and a ``core'' parameter $s$
that satisfy
$$
r_0 \leq k \leq s \leq \min\{m, n\},
$$
where $r_0$ is the target rank.
The parameter $k$ determines the maximum rank of an
approximation.
For now, be aware that the approximation scheme is more sensitive
to the choice of $k$ than to the choice of $s$.
In \cref{sec:sketch-size-theory},
we offer specific parameter recommendations that are
supported by theoretical analysis.  In \cref{sec:oracle-performance},
we demonstrate that these parameter choices are effective in practice.

Independently, draw and fix four randomized linear dimension reduction maps:
\begin{equation} \label{eqn:test-matrices}
\begin{aligned}
\mtx{\Upsilon} &\in \F^{k \times m} \quad\text{and}\quad
\mtx{\Omega} \in \F^{k \times n}; \\ %
\mtx{\Phi} &\in \F^{s \times m} \quad\text{and}\quad
\mtx{\Psi} \in \F^{s \times n}. %
\end{aligned}
\end{equation}
These dimension reduction maps are often called \emph{test matrices}.
The sketch itself consists of three matrices:
\begin{gather} \label{eqn:range-sketch}
\mtx{X} := \mtx{\Upsilon} \mtx{A} \in \F^{k \times n} \quad\text{and}\quad
\mtx{Y} := \mtx{A} \mtx{\Omega}^* \in \F^{m \times k}; \\ %
\label{eqn:core-sketch}
\mtx{Z} := \mtx{\Phi} \mtx{A} \mtx{\Psi}^* \in \F^{s \times s}.
\end{gather}
The first two matrices $(\mtx{X}, \mtx{Y})$ capture %
the co-range and the range of $\mtx{A}$.  The core sketch $(\mtx{Z})$ contains fresh information
that improves our estimates of the singular values and singular vectors of $\mtx{A}$;
it is responsible for the superior performance of the new method.

\begin{remark}[Prior Work]
The paper~\cite[Sec.~3]{Upa16:Fast-Space-Optimal} contains the insight that a
sketch of the form~\cref{eqn:range-sketch,eqn:core-sketch} can support better
low-rank matrix approximations, but it proposes a reconstruction algorithm
that is less effective.
Related (but distinct) sketches appear in the
papers~\cite{WLRT08:Fast-Randomized,CW09:Numerical-Linear,HMT11:Finding-Structure,
Woo14:Sketching-Tool,CEM+15:Dimensionality-Reduction,BWZ16:Optimal-Principal-STOC,TYUC17:Randomized-Single-View-TR,
TYUC17:Practical-Sketching}.
\end{remark}

\subsection{Linear Updates}
\label{sec:updates}

In streaming data applications, the input matrix
$\mtx{A} \in \F^{m \times n}$ is presented as a sequence of linear updates of the form
\begin{equation} \label{eqn:linear-update}
\mtx{A} \gets \eta \mtx{A} + \nu \mtx{H}
\end{equation}
where $\eta, \nu \in \F$ and the matrix $\mtx{H} \in \F^{m \times n}$.

In view of the construction~\cref{eqn:range-sketch,eqn:core-sketch},
we can update the sketch $(\mtx{X}, \mtx{Y}, \mtx{Z})$ of the matrix
$\mtx{A}$ to reflect the innovation~\cref{eqn:linear-update}
by means of the formulae
\begin{equation} \label{eqn:sketch-update}
\begin{aligned}
\mtx{X} &\gets \eta \mtx{X} + \nu \mtx{\Upsilon} \mtx{H} \\
\mtx{Y} &\gets \eta \mtx{Y} + \nu \mtx{H} \mtx{\Omega}^* \\
\mtx{Z} &\gets \eta \mtx{Z} + \nu \mtx{\Phi} \mtx{H} \mtx{\Psi}^*.
\end{aligned}
\end{equation}
When implementing these updates, it is worthwhile to exploit favorable
structure in the matrix $\mtx{H}$, such as sparsity or low rank.

\begin{remark}[Streaming Model]
For the linear update model~\cref{eqn:linear-update}, randomized linear
sketches are more or less the only way to track the input matrix~\cite{LNW14:Turnstile-Streaming}.
There are more restrictive streaming models (e.g., when the columns of the matrix are presented
in sequence) where it is possible to design other types of
algorithms~\cite{FVR16:Dimensionality-Reduction,GLPW16:Frequent-Directions}.
\end{remark}

\begin{remark}[Linearly Transformed Data] \label{rem:linear-transforms}
We can use sketching to track any matrix that depends
linearly on a data stream.
Suppose that the input data $\vct{a} \in \R^d$,
and we want to maintain the matrix $\mathscr{L}(\vct{a})$
induced by a fixed linear map $\mathscr{L} : \F^{d} \to \F^{m \times n}$.
If we receive an update $\vct{a} \gets \eta\vct{a} + \nu\vct{h}$,
then the linear image evolves as
$\mathscr{L}(\vct{a}) \gets \eta\mathscr{L}(\vct{a}) + \nu\mathscr{L}(\vct{h})$.
This update has the form~\eqref{eqn:linear-update},
so we can apply the matrix sketch~\cref{eqn:range-sketch,eqn:core-sketch}
to track $\mathscr{L}(\vct{a})$ directly.
This idea has applications to physical simulations where
a known transform $\mathscr{L}$ exposes structure in the data
\cite{MSTM13:Model-Based-Scaling}.
\end{remark}

\subsection{Optional Step: Centering}
\label{sec:centering}

Many applications require us to center the data matrix to remove a trend,
such as the temporal average.  Principal component analysis (PCA)
also involves a centering step~\cite{Jol02:Principal-Component}.
For superior accuracy, it is wise to perform this operation \emph{before}
sketching the matrix.

As an example, let us explain how to compute and remove the mean value of each
row of the data matrix in the streaming setting.
We can maintain an extra vector $\vct{\mu} \in \F^{m}$ that tracks the mean value of each row.
To process an update of the form~\cref{eqn:linear-update}, we first apply the steps
\begin{equation*} \label{eqn:detrend}
\vct{h} \gets n^{-1} \mtx{H} \mathbf{e}
\quad\text{and}\quad
\mtx{H} \gets \mtx{H} - \vct{h} \mathbf{e}^*
\quad\text{and}\quad
\vct{\mu} \gets \eta \vct{\mu} + \nu \vct{h}.
\end{equation*}
Here, $\mathbf{e} \in \F^n$ is the vector of ones.  Afterward, we update the sketches
using~\cref{eqn:sketch-update}.  The sketch now contains the centered data matrix, where
each row has zero mean.

\subsection{Computing Truncated Low-Rank Approximations}

Once we have acquired a sketch $(\mtx{X}, \mtx{Y}, \mtx{Z})$ of
the input matrix $\mtx{A}$, we must produce a good low-rank approximation.
Let us outline the computations we propose.
The intuition appears below in \cref{sec:intuition},
and \Cref{sec:theory} presents a theoretical analysis.

The first two components $(\mtx{X}, \mtx{Y})$ of the sketch are used
to estimate the co-range and the range of the matrix $\mtx{A}$.
Compute thin \textsf{QR} factorizations:
\begin{equation} \label{eqn:range-corange}
\begin{aligned}
\mtx{X}^* =: \mtx{P} \mtx{R}_1
	\quad\text{where}\quad \mtx{P} \in \F^{n \times k}; \\
\mtx{Y} =: \mtx{Q} \mtx{R}_2
	\quad\text{where}\quad \mtx{Q} \in \F^{m \times k}.
\end{aligned}
\end{equation}
Both $\mtx{P}$ and $\mtx{Q}$ have orthonormal columns; discard the triangular parts $\mtx{R}_1$ and $\mtx{R}_2$.

The third sketch $\mtx{Z}$ is used to compute the core approximation $\mtx{C}$, which describes
how $\mtx{A}$ acts between $\range(\mtx{P})$ and $\range(\mtx{Q})$:
\begin{equation} \label{eqn:linkage-matrix}
\mtx{C} := (\mtx{\Phi Q})^\dagger \mtx{Z} ((\mtx{\Psi} \mtx{P})^\dagger)^* \in \F^{k \times k}.
\end{equation}
This step is implemented by solving a family of least-squares problems.

Next, form a rank-$k$ approximation $\hat{\mtx{A}}$ of the input matrix $\mtx{A}$ via
\begin{equation} \label{eqn:Ahat}
\hat{\mtx{A}} := \mtx{Q} \mtx{C} \mtx{P}^*
\end{equation}
We refer to $\hat{\mtx{A}}$ as the ``initial'' approximation.
It is important to be aware that the initial approximation can contain spurious
information (in its smaller singular values and the associated singular vectors).

To produce an approximation that is fully reliable, we must truncate the rank
of the initial approximation~\cref{eqn:Ahat}.
For a truncation parameter $r$, we construct a rank-$r$ approximation
by replacing $\hat{\mtx{A}}$ with its best rank-$r$ approximation%
\footnote{The formula~\cref{eqn:Ahat-fixed} is an easy consequence
of the Eckart--Young Theorem~\cite[Sec.~6]{Hig89:Matrix-Nearness} and the fact
that $\mtx{Q}, \mtx{P}$ have orthonormal columns.}
in Frobenius norm:
\begin{equation} \label{eqn:Ahat-fixed}
\lowrank{\hat{\mtx{A}}}{r} = \mtx{Q} \lowrank{\mtx{C}}{r} \mtx{P}^*.
\end{equation}
We refer to $\lowrank{\hat{\mtx{A}}}{r}$ as a ``truncated''
approximation.  \Cref{sec:diagnose-spectrum} outlines some ways to
use \emph{a posteriori} information to select the truncation rank $r$.

The truncation~\cref{eqn:Ahat-fixed} has an appealing permanence property:
$\lowrank{ \lowrank{\hat{\mtx{A}}}{\varrho} }{r} = \lowrank{ \hat{\mtx{A}} }{r}$
for all $\varrho \geq r$.  In other words, the rank-$r$ approximation persists
as part of all higher-rank approximations. 
In contrast, some earlier reconstruction methods are unstable
in the sense that the rank-$r$ approximation varies wildly with $r$;
see~\cref{app:flow-field-reconstruction}.

\begin{remark}[Extensions]
We can form other structured approximations of $\mtx{A}$ by
projecting $\hat{\mtx{A}}$ onto a set of structured matrices.
See~\cite[Secs.~5--6]{TYUC17:Practical-Sketching}
for a discussion of this idea in the context of another sketching technique.
For brevity, we do not develop this point further.
See our paper~\cite{TYUC17:Fixed-Rank-Approximation} for a sketching
and reconstruction method designed specifically for positive-semidefinite matrices.  
\end{remark}

\begin{remark}[Prior Work]
The truncated approximation~\cref{eqn:Ahat-fixed} is new,
but it depends on insights from our previous
work~\cite{TYUC17:Randomized-Single-View-TR,TYUC17:Practical-Sketching}.
Upadhyay~\cite[Thm.~12]{Upa16:Fast-Space-Optimal} proposes
a different reconstruction formula for the same kind of sketch.
The papers~\cite{WLRT08:Fast-Randomized,CW09:Numerical-Linear,HMT11:Finding-Structure,
Woo14:Sketching-Tool,CEM+15:Dimensionality-Reduction,BWZ16:Optimal-Principal-STOC,Upa16:Fast-Space-Optimal}
describe other methods for low-rank matrix approximation from a randomized linear sketch.
The numerical work in \cref{sec:numerics}
demonstrates that~\cref{eqn:Ahat-fixed} matches or improves on
earlier techniques.
\end{remark}

\subsection{Intuition}
\label{sec:intuition}

The approximations \cref{eqn:Ahat,eqn:Ahat-fixed} are based on some well-known insights
from randomized linear algebra~\cite[Sec.~1]{HMT11:Finding-Structure}.
Since $\mtx{P}$ and $\mtx{Q}$ capture the co-range and range of the
input matrix, we expect that
\begin{equation} \label{eqn:rdm-range-finder}
\mtx{A} \approx \mtx{Q} (\mtx{Q}^* \mtx{A} \mtx{P}) \mtx{P}^*
\end{equation}
(See \cref{lem:compression-error} for justification.)
We cannot compute the core matrix $\mtx{Q}^*\mtx{A}\mtx{P}$ directly from
a linear sketch because $\mtx{P}$ and $\mtx{Q}$ are functions of $\mtx{A}$.
Instead, we estimate the core matrix using the core sketch $\mtx{Z}$.
Owing to the approximation~\cref{eqn:rdm-range-finder},
$$
\mtx{Z} = \mtx{\Phi} \mtx{A} \mtx{\Psi}^*
	\approx (\mtx{\Phi Q}) (\mtx{Q}^* \mtx{A} \mtx{P}) (\mtx{P}^* \mtx{\Psi}^*).
$$
Transfer the outer matrices to the left-hand side to discover that
the core approximation $\mtx{C}$, defined in~\eqref{eqn:linkage-matrix}, satisfies
\begin{equation} \label{eqn:linkage-intuition}
\mtx{C} = (\mtx{\Phi Q})^\dagger \mtx{Z} ((\mtx{\Psi} \mtx{P})^\dagger)^*
	\approx \mtx{Q}^* \mtx{A} \mtx{P}.
\end{equation}
In view of \cref{eqn:rdm-range-finder,eqn:linkage-intuition}, we arrive at the relations
$$
\mtx{A} \approx \mtx{Q} (\mtx{Q}^* \mtx{A} \mtx{P}) \mtx{P}^*
	\approx \mtx{Q} \mtx{C} \mtx{P}^*
	= \hat{\mtx{A}}.
$$
The error in the last relation depends on the error in the best rank-$k$
approximation of $\mtx{A}$.  When $\mtx{A}$ has a decaying spectrum,
the rank-$r$ truncation of $\mtx{A}$ for $r \ll k$ agrees closely with
the rank-$r$ truncation of the initial approximation $\hat{\mtx{A}}$.
That is,
$$
	\lowrank{\mtx{A}}{r}
	\approx \lowrank{\hat{\mtx{A}}}{r}
	= \mtx{Q} \lowrank{\mtx{C}}{r} \mtx{P}^*.
$$
\Cref{thm:low-rank-error-bound,cor:fixed-rank-error-bound} justify these heuristics
completely for Gaussian dimension reduction maps.
\Cref{sec:diagnose-spectrum} discusses \emph{a posteriori} selection of $r$.

\subsection{Discussion of Related Work}
\label{sec:provenance}

Sketching algorithms are specifically designed for the streaming model;
that is, for data that is presented as a sequence of updates.
The sketching paradigm is attributed to~\cite{AMS96:Space-Complexity,AGMS02:Tracking-Join}; see
the survey~\cite{Mut08:Data-Streams} for an introduction
and overview of early work.

Randomized algorithms for low-rank matrix approximation were proposed
in the theoretical computer science (TCS) literature in the late
1990s~\cite{PRTV00:Latent-Semantic,FKV04:Fast-Monte-Carlo}.
Soon after, numerical analysts developed practical versions of these
algorithms~\cite{MRT11:Randomized-Algorithm,WLRT08:Fast-Randomized,RST09:Randomized-Algorithm,
HMT11:Finding-Structure,HMST11:Algorithm-Principal}.  For more background
on the history of randomized linear algebra,
see~\cite{HMT11:Finding-Structure,Mah11:Randomized-Algorithms,Woo14:Sketching-Tool}.

The paper~\cite{WLRT08:Fast-Randomized} contains the first
one-pass algorithm for low-rank matrix approximation; it was
designed to control communication and arithmetic costs,
rather than to handle streaming data.
The first general treatment of numerical linear algebra in the
streaming model appears in~\cite{CW09:Numerical-Linear}.
Recent papers on low-rank matrix approximation in the streaming model
include~\cite{BWZ16:Optimal-Principal-STOC,Upa16:Fast-Space-Optimal,
FVR16:Dimensionality-Reduction,GLPW16:Frequent-Directions,
TYUC17:Practical-Sketching,TYUC17:Fixed-Rank-Approximation}.

\subsubsection{Approaches from NLA}

The NLA literature contains a number of
papers~\cite{WLRT08:Fast-Randomized,HMT11:Finding-Structure,TYUC17:Practical-Sketching}
on low-rank approximation from a randomized linear sketch.
These methods all compute the range matrix
$\mtx{Q}$ and the co-range matrix $\mtx{P}$
using the randomized range finder~\cite[Alg.~4.1]{HMT11:Finding-Structure},
encapsulated in~\cref{eqn:range-sketch,eqn:range-corange}.

The methods differ in how they construct a core matrix
$\tilde{\mtx{C}}$ so that $\mtx{A} \approx \mtx{Q}\tilde{\mtx{C}}\mtx{P}$.
Earlier papers reuse the range and co-range sketches ($\mtx{X}$, $\mtx{Y}$)
and the associated test matrices ($\mtx{\Upsilon}$, $\mtx{\Omega}$)
to form $\tilde{\mtx{C}}$.  Our new algorithm is based on an insight
from~\cite{BWZ16:Optimal-Principal-STOC,Upa16:Fast-Space-Optimal}
that the estimate $\mtx{C}$ from~\cref{eqn:linkage-matrix} is more reliable because it
uses a random sketch $\mtx{Z}$ that is statistically independent from ($\mtx{X}$, $\mtx{Y}$).
The storage cost of the additional sketch is negligible
when $s^2 \ll k (m + n)$. %

The methods also truncate the rank of the approximation
at different steps.  The older papers~\cite{WLRT08:Fast-Randomized,HMT11:Finding-Structure}
perform the truncation \emph{before} estimating the core matrix (cf.~\cref{app:hmt}).
One insight from~\cite{TYUC17:Practical-Sketching} is
that it is beneficial to perform the truncation \emph{after} estimating
the core matrix.  Furthermore, an effective truncation mechanism %
is to report a best rank-$r$ approximation of the initial estimate.
We have adopted the latter approach.

\subsubsection{Approaches from TCS}
\label{sec:tcs-discuss}

Most of the algorithms in the TCS literature~\cite{CW09:Numerical-Linear,Woo14:Sketching-Tool,CEM+15:Dimensionality-Reduction,BWZ16:Optimal-Principal-STOC,Upa16:Fast-Space-Optimal}
are based on a framework called ``sketch-and-solve'' that is attributed
to Sarl{\'o}s~\cite{Sar06:Improved-Approximation}.
The basic idea is that the solution to a constrained
least-squares problem (e.g., low-rank matrix approximation in Frobenius norm)
is roughly preserved when we solve the problem after randomized linear dimension reduction.

The sketch-and-solve framework sometimes leads to the same algorithms
as the NLA point of view; other times, it leads to different approaches.
It would take us too far afield to detail these derivations,
but we give a summary of one such method~\cite[Thm.~12]{Upa16:Fast-Space-Optimal}
in \cref{app:upa}.
Unfortunately, sketch-and-solve algorithms are often unsuitable for high-accuracy computations;
see~\cref{sec:numerics} and~\cite[Sec.~7]{TYUC17:Practical-Sketching} for evidence.

A more salient criticism is that the TCS literature does not attend to
the issues that arise if we want to use sketching algorithms in practice.
We have expended a large amount of effort to address these challenges,
which range from parameter selection to numerically sound implementation.
See~\cite[Sec.~1.7.4]{TYUC17:Practical-Sketching} for more discussion.

\section{Randomized Linear Dimension Reduction Maps}
\label{sec:dim-red-maps}

In this section, we describe several randomized linear dimension reduction maps
that are suitable for implementing sketching algorithms for low-rank matrix approximation.
See~\cite{Lib09:Accelerated-Dense,HMT11:Finding-Structure,Woo14:Sketching-Tool,
TYUC17:Practical-Sketching,sun2018tensor}
for additional discussion and examples.  The class template for a dimension
reduction map appears as \cref{alg:dim-redux}; the algorithms for specific
dimension reduction techniques are postponed to the supplement.

\begin{algorithm}[t]
  \caption{\textsl{Dimension Reduction Map Class.}  %
  \label{alg:dim-redux}}
  \begin{algorithmic}[1]
  \vspace{0.5pc}

  \State \textbf{class} \textsc{DimRedux} ($\mathbb{F}$)
  	\Comment Dimension reduction map over field $\F$

  \Indent

  \Function{\textsc{DimRedux}}{$d, N$}
  	\Comment Construct map $\mtx{\Xi} : \F^N \to \F^d$
  \EndFunction

  \Function{\textsc{DimRedux.mtimes}}{\texttt{DRmap}, $\mtx{M}$}
  		\Comment Left action of map
  \EndFunction

  \Function{\textsc{DimRedux.mtimes}}{$\mtx{M}, \texttt{DRmap}^*$}
  		\Comment Right action of adjoint
		\State \Return $(\textsc{DimRedux.mtimes}(\texttt{DRmap}, \mtx{M}^*))^*$
		\Comment Default behavior
  \EndFunction

  \EndIndent

  \vspace{0.25pc}

\end{algorithmic}
\end{algorithm}

\subsection{Gaussian Maps}
\label{sec:gauss}

The most basic dimension reduction map is simply a Gaussian matrix.
That is, $\mtx{\Xi} \in \F^{d \times N}$ is a $d \times N$
matrix with independent standard normal entries.%
\footnote{A real standard normal variable follows the Gaussian distribution
with mean zero and variance one.  A complex standard normal variable
takes the form $g_1 + \mathrm{i} g_2$,
where $g_i$ are independent real standard normal variables.}

\Cref{alg:dim-redux-gauss} describes an implementation
of Gaussian dimension reduction.
The map $\mtx{\Xi}$ requires storage of $dN$ floating-point numbers
in the field $\F$.  The cost of applying the map to a vector is
$\mathcal{O}(dN)$ arithmetic operations.

Gaussian dimension reduction maps are simple, and they are effective
in randomized algorithms for low-rank matrix approximation~\cite{HMT11:Finding-Structure}.
We can also analyze their behavior in full detail; see \cref{sec:theory,sec:aposterior}.
On the other hand, it is expensive to draw a large number of
Gaussian random variables, and the cost of storage and arithmetic renders
these maps less appealing when the output dimension $d$ is large.

\begin{remark}[Unknown Dimension] \label{rem:unknown-dim}
Since the columns of a Gaussian map $\mtx{\Xi}$ are statistically independent,
we can instantiate more columns if we need to apply $\mtx{\Xi}$ to a longer vector.
Sparse maps (\cref{sec:sparse}) share this feature.
This observation is valuable in the streaming setting,
where a linear update might involve coordinates heretofore unseen,
forcing us to enlarge the domain of the dimension reduction map.
\end{remark}

\begin{remark}[History]
Gaussian dimension reduction has been used as an algorithmic tool
since the paper~\cite{IM98:Approximate-Nearest} of Indyk \& Motwani.
In spirit, this approach is quite similar to the earlier theoretical
work of Johnson \& Lindenstrauss~\cite{JL84:Extensions-Lipschitz},
which performs dimension reduction by projection onto a random subspace.
\end{remark}

\subsection{Scrambled SRFT Maps}
\label{sec:ssrft}

Next, we describe a structured dimension reduction map, called a
\emph{scrambled subsampled randomized Fourier transform} (SSRFT).
We recommend this approach for practical implementations.

An SSRFT map takes the form%
\footnote{Empirical work suggests that it is not necessary to iterate
the permutation and trigonometric transform twice, but this duplication can increase reliability.}
$$
\mtx{\Xi} = \mtx{RF\Pi F\Pi}' \in \F^{d \times N}.
$$
The matrices $\mtx{\Pi}, \mtx{\Pi}' \in \F^{N \times N}$ are signed permutations,%
\footnote{A signed permutation matrix has precisely one nonzero entry in each row and column,
and each nonzero entry of the matrix has modulus one.}
drawn independently and uniformly at random.
The matrix $\mtx{F} \in \F^{N \times N}$ denotes a discrete cosine transform $(\F = \R)$
or a discrete Fourier transform $(\F = \C)$.
The matrix $\mtx{R} \in \F^{d \times N}$ is a restriction to $d$ coordinates,
chosen uniformly at random.

\Cref{alg:dim-redux-ssrft} presents an implementation of an SSRFT.
The cost of storing $\mtx{\Xi}$ is just $\mathcal{O}(N)$ numbers.
The cost of applying $\mtx{\Xi}$ to a vector is $\mathcal{O}(N \log N)$
arithmetic operations, using the Fast Fourier Transform (FFT)
or the Fast Cosine Transform (FCT).  According to~\cite{WLRT08:Fast-Randomized},
this cost can be reduced to $\mathcal{O}(N \log d)$, but the improvement is
rarely worth the implementation effort.

In practice, SSRFTs behave slightly better than Gaussian matrices,
even though their storage cost does not scale with the output dimension $d$.
On the other hand, the analysis~\cite{AC09:Fast-Johnson-Lindenstrauss,Tro11:Improved-Analysis,BG13:Improved-Matrix}
is less complete than in the Gaussian case~\cite{HMT11:Finding-Structure}.
A proper implementation requires fast trigonometric transforms.
Last, the random permutations and FFTs require data movement,
which could be a challenge in the distributed setting.

\begin{remark}[History]
SSRFTs are inspired by the work of Ailon \& Chazelle~\cite{AC09:Fast-Johnson-Lindenstrauss}
on fast Johnson--Lindstrauss transforms.  For applications in randomized linear algebra,
see the papers~\cite{WLRT08:Fast-Randomized,Lib09:Accelerated-Dense,
HMT11:Finding-Structure,Tro11:Improved-Analysis,BG13:Improved-Matrix}.
\end{remark}

\subsection{Sparse Sign Matrices}
\label{sec:sparse}

Last, we describe another type of randomized dimension reduction
map, called a \emph{sparse sign matrix}.  We recommend
these maps for practical implementations where data movement is a concern.

To construct a sparse sign matrix $\mtx{\Xi} \in \F^{d \times N}$,
we fix a sparsity parameter $\zeta$ in the range $2 \leq \zeta \leq d$.
The columns of the matrix are drawn independently at random.
To construct each column, we take $\zeta$ iid draws from the
$\textsc{uniform}\{ z\in \F : \abs{z} = 1\}$ distribution,
and we place these random variables in $p$ coordinates,
chosen uniformly at random.
Empirically,
we have found that $\zeta = \min\{ d, 8 \}$
is a very reliable parameter selection
in the context of low-rank matrix approximation.%
\footnote{Empirical testing supports more aggressive choices, say, $\zeta = 4$
or even $\zeta = 2$ for very large problems.
On the other hand, the extreme $\zeta = 1$ is disastrous, so we have excluded it.}

\Cref{alg:dim-redux-sparse} describes an implementation of sparse
dimension reduction.  Since the matrix $\mtx{\Xi} \in \F^{d \times N}$
has $\zeta$ nonzeros per column,
we can store the matrix with $\mathcal{O}(\zeta N \log(1 + d/\zeta))$ numbers
via run-length coding.  The cost of applying the map to a vector is
$\mathcal{O}(\zeta N)$ arithmetic operations.

Sparse sign matrices can reduce data movement because
the columns are generated independently and the matrices can be applied
using (blocked) matrix multiplication.  They can also
adapt to input vectors whose maximum dimension $N$ may be unknown,
as discussed in \cref{rem:unknown-dim}.
One weakness is that we must use sparse data structures
and arithmetic to enjoy the benefit of these maps.

\begin{remark}[History]
Sparse dimension reduction maps are inspired
by the work of Achlioptas~\cite{Ach03:Database-Friendly}
on database-friendly random projections.
For applications in randomized linear algebra,
see~\cite{CW13:Low-Rank-Approximation,MM13:Low-Distortion-Subspace,
NN13:OSNAP-Faster,NN14:Lower-Bounds,BDN15:Toward-Unified}.
See~\cite{Coh16:Nearly-Tight} for a theoretical analysis
in the context of matrix approximation.
\end{remark}

\section{Implementation and Costs}
\label{sec:implementation}

This section contains further details about the implementation
of the sketching and reconstruction methods from \cref{sec:sketching},
including an account of storage and arithmetic costs.
We combine the mathematical notation from the text
with \textsc{Matlab R2018b} commands (typewriter font).
The electronic materials include a \textsc{Matlab}
implementation of these methods.

\subsection{Sketching and Updates}

\Cref{alg:sketch,alg:linear-update} contain the pseudocode for initializing
the sketch and for performing the linear update \cref{eqn:linear-update}.
It also includes optional code %
for maintaining an error sketch (\cref{sec:aposterior}).

The sketch requires storage of four dimension reduction maps
with size $k \times m$, $k \times n$, $s \times m$, $s \times n$.
We recommend using SSRFTs or sparse sign matrices to minimize the
storage costs associated with the dimension reduction maps.

The sketch itself consists of three matrices with dimensions
$k \times n$, $m \times k$, and $s \times s$.  In general, the
sketch matrices are dense, so they require $k (m+n) + s^2$
floating-point numbers in the field $\F$.

The arithmetic cost of the linear update $\mtx{A} \gets \eta \mtx{A} + \tau \mtx{H}$
is dominated by the cost of computing $\mtx{\Phi H}$ and $\mtx{H \Psi}$.
In practice, the innovation $\mtx{H}$ is low-rank, sparse, or structured.
The precise cost of the update depends on how we exploit the structure of $\mtx{H}$
and the dimension reduction map.

\begin{algorithm}[t] 
  \caption{\textsl{Sketch Constructor.}
  Implements~\cref{eqn:test-matrices,eqn:range-sketch,eqn:core-sketch,eqn:error-sketch}
  \label{alg:sketch}}
  \begin{algorithmic}[1]

  \Require{Field $\F$; input matrix dimensions $m \times n$;
  approximation sketch size parameters $k \leq s \leq \min\{m, n\}$;
  error sketch size parameter $q$}
  \Ensure{Draw test matrices for the approximation~\cref{eqn:test-matrices} and the error estimate~\cref{eqn:error-test};
  form the sketch~\cref{eqn:range-sketch,eqn:core-sketch,eqn:error-sketch} %
  of the zero matrix $\mtx{A} = \mtx{0}$}

\vspace{0.5pc}

	\State	\textbf{class} \textsc{Sketch}
  \Indent
	\State \textbf{local variables} $\mtx{\Upsilon}, \mtx{\Omega}, \mtx{\Phi}, \mtx{\Psi}$ (\textsc{DimRedux})
    \State \textbf{local variables} $\mtx{X}, \mtx{Y}, \mtx{Z}$ (matrices)	

	\State \textbf{local variables} $\mtx{\Theta}$ (\textsc{GaussDR}), $\mtx{W}$ (matrix)
		\Comment{[opt] For error estimation} %
  \vspace{0.5pc}
	\Function{Sketch}{$m, n, k, s, q$; \textsc{DR}}
		\Comment{Constructor; \textsc{DR} is a \textsc{DimRedux}}
    \State	$\mtx{\Upsilon} \gets \textsc{DR}(k, m)$
    	\Comment{Construct test matrix for range}
    \State	$\mtx{\Omega} \gets \textsc{DR}(k, n)$
    	\Comment{Test matrix for co-range}
    \State	$\mtx{\Phi} \gets \textsc{DR}(s, m)$
    	\Comment{Test matrices for core}
    \State	$\mtx{\Psi} \gets \textsc{DR}(s, n)$
    \State	$\mtx{\Theta} \gets \textsc{GaussDR}(q, m)$
    	\Comment{[opt] Gaussian test matrix for error}
	\State	$\mtx{X} \gets \texttt{zeros}(k, n)$ %
		\Comment{Approximation sketch of zero matrix}
	\State	$\mtx{Y} \gets \texttt{zeros}(m, k)$ %
	\State	$\mtx{Z} \gets \texttt{zeros}(s, s)$ %
	\State	$\mtx{W} \gets \texttt{zeros}(q, n)$ %
		\Comment{[opt] Error sketch of zero matrix}

	\EndFunction
	
	  \EndIndent

	\vspace{0.25pc}

\end{algorithmic}
\end{algorithm}

\begin{algorithm}[t] %
  \caption{\textsl{Linear Update to Sketch.}
  Implements~\cref{eqn:linear-update,eqn:error-sketch-update}. %
  \label{alg:linear-update}}
  \begin{algorithmic}[1]

  \Require{Innovation $\mtx{H} \in \F^{m \times n}$; scalars $\eta, \nu \in \F$}
  \Ensure{Modifies sketch %
  to reflect linear update $\mtx{A} \gets \eta \mtx{A} + \nu \mtx{H}$}
  \vspace{0.5pc}

  \Function{Sketch.LinearUpdate}{$\mtx{H}; \eta, \nu$}
	\State	$\mtx{X} \gets \eta \mtx{X} + \nu \mtx{\Upsilon H}$
		\Comment{Update range sketch}
	\State	$\mtx{Y} \gets \eta \mtx{Y} + \nu \mtx{H \Omega}^*$
		\Comment{Update co-range sketch}
	\State	$\mtx{Z} \gets \eta \mtx{Z} + \nu (\mtx{\Phi H}) \mtx{\Psi}^*$
		\Comment{Update core sketch}
	\State	$\mtx{W} \gets \eta \mtx{W} + \nu \mtx{\Theta H}$
		\Comment{[opt] Update error sketch}
  \EndFunction

	\vspace{0.25pc}

\end{algorithmic}
\end{algorithm}

\subsection{The Initial Approximation}

\Cref{alg:low-rank-approx} lists the pseudocode for computing a rank-$k$
approximation $\hat{\mtx{A}}$ of the matrix $\mtx{A}$ contained in the sketch;
see~\cref{eqn:Ahat}.

The method requires additional storage of $k (m+n)$ numbers for the
orthonormal matrices $\mtx{P}$ and $\mtx{Q}$, as well as $\mathcal{O}(sk)$ numbers
to form the core matrix $\mtx{C}$.  The arithmetic cost is usually dominated
by the computation of the \textsf{QR} factorizations of
$\mtx{X}^*$ and $\mtx{Y}$, which require $\mathcal{O}(k^2(m+n))$
operations.  When the parameters satisfy $s \gg k$, it is possible that
the cost $\mathcal{O}(k s^2)$ of forming the core matrix
$\mtx{C}$ will dominate; bear this in mind when setting the parameter $s$.

\subsection{The Truncated Approximation}

\Cref{alg:fixed-rank-approx} presents the pseudocode for computing a rank-$r$
approximation $\lowrank{\hat{\mtx{A}}}{r}$ of the matrix $\mtx{A}$ contained
in the sketch; see~\cref{eqn:Ahat-fixed}.  The parameter $r$ is an input;
the output is presented as a truncated SVD.

The working storage cost $\mathcal{O}(k(m+n))$ is dominated by the call
to \cref{alg:low-rank-approx}.  Typically, the arithmetic cost
is also dominated by the $\mathcal{O}(k^2(m+n))$ cost of the call to \cref{alg:low-rank-approx}.
When $s \gg k$, we need to invoke a randomized SVD algorithm~\cite{HMT11:Finding-Structure}
to achieve this arithmetic cost, but an ordinary dense SVD sometimes serves.

\begin{algorithm}[t] %
  \caption{\textsl{Initial Approximation.}  Implements~\cref{eqn:Ahat}. %
  \label{alg:low-rank-approx}}
  \begin{algorithmic}[1]
    \Ensure{Rank-$k$ approximation of sketched matrix in the form $\hat{\mtx{A}} = \mtx{QCP}^*$ with orthonormal $\mtx{Q} \in \F^{m \times k}$ and $\mtx{P} \in \F^{n \times k}$ and $\mtx{C} \in \F^{k \times k}$}
\vspace{0.5pc}

	\Function{Sketch.InitialApprox}{\,}

	\State	$(\mtx{Q}, \sim) \gets \texttt{qr}(\mtx{Y}, \texttt{0})$
		\Comment Compute orthogonal part of thin \textsf{QR}
	\State	$(\mtx{P}, \sim) \gets \texttt{qr}(\mtx{X}^*, \texttt{0})$
	\State	$\mtx{C} \gets ((\mtx{\Phi Q}) \backslash \mtx{Z}) / ((\mtx{\Psi P})^*)$
		\Comment	Solve two least-squares problems %
	\State	\Return $(\mtx{Q}, \mtx{C}, \mtx{P})$
	\EndFunction

	\vspace{0.25pc}

\end{algorithmic}
\end{algorithm}

\begin{algorithm}[t]
  \caption{\textsl{Truncated Approximation.}  Implements~\cref{eqn:Ahat-fixed}. %
  \label{alg:fixed-rank-approx}}
  \begin{algorithmic}[1]
    \Require{Final rank $r$ of the approximation}
    \Ensure{Rank-$r$ approximation of sketched matrix in the form $\hat{\mtx{A}}_r = \mtx{U\Sigma V}^*$
    with orthonormal $\mtx{U} \in \F^{m \times r}$ and $\mtx{V} \in \F^{n \times r}$
    and nonnegative diagonal $\mtx{\Sigma} \in \R^{r \times r}$}
\vspace{0.5pc}

	\Function{Sketch.TruncateApprox}{$r$}

	\State	$(\mtx{Q}, \mtx{C}, \mtx{P}) \gets \textsc{Sketch.InitialApprox}(\,)$
	\State	$(\mtx{U}, \mtx{\Sigma}, \mtx{V}) \gets \texttt{svd}(\mtx{C})$
		\Comment Dense or randomized SVD
	\State 	$\mtx{\Sigma} \gets \mtx{\Sigma}( \texttt{1:r}, \texttt{1:r})$
			\Comment Truncate SVD to rank $r$
	\State	$\mtx{U} \gets \mtx{U}(\texttt{:}, \texttt{1:r})$
	\State	$\mtx{V} \gets \mtx{V}(\texttt{:}, \texttt{1:r})$
	\State	$\mtx{U} \gets \mtx{Q} \mtx{U}$
		\Comment Consolidate unitary factors
	\State	$\mtx{V} \gets \mtx{P} \mtx{V}$
	\State	\Return $(\mtx{U}, \mtx{\Sigma}, \mtx{V})$
	\EndFunction

	\vspace{0.25pc}

\end{algorithmic}
\end{algorithm}

\section{\emph{A Priori} Error Bounds}
\label{sec:theory}

It is always important to characterize the behavior of numerical algorithms, but
the challenge is more acute for sketching methods.  Indeed, we cannot store the stream
of updates, so we cannot repeat the computation with new parameters if it is unsuccessful.
As a consequence, we must perform \emph{a priori} theoretical analysis to be able to implement
sketching algorithms with confidence.

In this section, we analyze the low-rank reconstruction algorithms
in the ideal case where all of the dimension reduction maps are standard normal.
These results allow us to make concrete recommendations for the sketch size parameters.
Empirically, other types of dimension reduction exhibit almost identical performance (\cref{sec:universality}),
so our analysis also supports more practical implementations based on SSRFTs
or sparse sign matrices.  The numerical work in \cref{sec:numerics} confirms
the value of this analysis.

\subsection{Notation}

For each integer $r \geq 0$, the \emph{tail energy} of the input matrix is
$$
\tau_{r+1}^2(\mtx{A}) := \min_{\rank( \mtx{B} ) \leq r} \fnormsq{ \mtx{A} - \mtx{B} }
	= \fnormsq{ \mtx{A} - \lowrank{\mtx{A}}{r}}
	= \sum\nolimits_{j > r} \sigma_j^2(\mtx{A}),
$$
where $\sigma_j$ returns the $j$th largest singular value of a matrix.  The second
identity follows from the Eckart--Young Theorem~\cite[Sec.~6]{Hig89:Matrix-Nearness}.

We also introduce parameters that reflect the field over which we are working:
\begin{equation} \label{eqn:alpha-beta}
\alpha := \alpha(\F) := \begin{cases} 1, & \F = \R \\ 0, & \F = \C \end{cases}
\quad\text{and}\quad
\beta := \beta(\F) := \begin{cases} 1, & \F = \R \\ 2, & \F = \C. \end{cases}
\end{equation}
These quantities let us present real and complex results in a single formula.

\subsection{Analysis of Initial Approximation}

The first result gives a bound for the expected error in the initial rank-$k$
approximation $\hat{\mtx{A}}$ of the input matrix $\mtx{A}$.

\begin{theorem}[Initial Approximation: Error Bound]
\label{thm:low-rank-error-bound}
Let $\mtx{A} \in \F^{m \times n}$ be an arbitrary input matrix.
Assume the sketch size parameters satisfy $s \geq 2k + \alpha$.
Draw independent Gaussian dimension reduction maps
$(\mtx{\Upsilon}, \mtx{\Omega}, \mtx{\Phi}, \mtx{\Psi})$,
as in~\cref{eqn:test-matrices}.  Extract a sketch~\cref{eqn:range-sketch,eqn:core-sketch}
of the input matrix.  Then
the rank-$k$ approximation $\hat{\mtx{A}}$,
constructed in~\cref{eqn:Ahat},
satisfies the error bound
\begin{equation} \label{eqn:low-rank-error-bound}
\Expect \fnormsq{ \mtx{A} - \hat{\mtx{A}} }
	\leq \frac{s - \alpha}{s - k - \alpha} \cdot \min_{\varrho < k - \alpha} \frac{k + \varrho - \alpha}{k - \varrho - \alpha} \cdot \tau_{\varrho+1}^2(\mtx{A}).
\end{equation}
\end{theorem}

\noindent
We postpone the proof to \cref{app:error-bound-proof}.  The analysis
is similar in spirit to the proof of~\cite[Thm.~4.3]{TYUC17:Practical-Sketching},
but it is somewhat more challenging.

\Cref{thm:low-rank-error-bound} contains explicit and reasonable constants,
so we can use it to design algorithms that achieve a specific error tolerance.
For example, suppose that $r_0$ is the target rank of the approximation.
Then the choice
\begin{equation} \label{eqn:my-params}
k = 4 r_0 + \alpha
\quad\text{and}\quad
s = 2k + \alpha
\end{equation}
ensures that the expected error in the rank-$k$ approximation $\hat{\mtx{A}}$
is within a constant factor $10/3$ of the optimal rank-$r_0$ approximation:
$$
\Expect \fnormsq{ \mtx{A} - \hat{\mtx{A}} }
	\leq \tfrac{10}{3} \cdot \tau_{r_0 + 1}^2( \mtx{A} ).
$$
In practice, we have found the parameter selection \cref{eqn:my-params}
can be effective for matrices with a rapidly decaying spectrum.
Note that, by taking $k/r_0 \to \infty$ and $s/k \to \infty$, we can drive the
leading constants in \cref{eqn:low-rank-error-bound} to one.

The true meaning of \cref{thm:low-rank-error-bound} is more subtle.
The minimum over $\varrho$ indicates that we can exploit decay
in the spectrum of the input matrix by increasing the parameter $k$.
This effect is more significant than the improvement we get
from adjusting the parameter $s$ to reduce the first constant.
In \cref{sec:sketch-size-theory}, we
use this insight to recommend sketch size parameters for a given storage budget.

\begin{remark}[Parameter Values]
In \cref{thm:low-rank-error-bound},
we have imposed the condition $s \geq 2k + \alpha$ because theoretical
analysis and empirical work both suggest that the restriction
is useful in practice.  The approximation~\cref{eqn:Ahat}
only requires that $k \leq s$.
\end{remark}

\begin{remark}[Failure Probability]
Because of measure concentration effects, there is a negligible
probability that the error in the initial approximation is
significantly larger than the bound~\cref{eqn:low-rank-error-bound}
on the expected error.  This claim can be established with techniques
from~\cite[Sec.~10]{HMT11:Finding-Structure}.  See
\cref{sec:apost-experiments} for numerical evidence.
\end{remark}

\begin{remark}[Singular Values and Vectors]
The error bound~\cref{eqn:low-rank-error-bound} indicates
that we can approximate singular values of $\mtx{A}$ by
singular values of $\hat{\mtx{A}}$.  In particular,
an application~\cite[Prob.~III.6.13]{Bha97:Matrix-Analysis}
of Lidskii's theorem implies that
$$
\sum\nolimits_{j=1}^{\min\{m,n\}} \big[ \sigma_j(\mtx{A}) - \sigma_j(\hat{\mtx{A}}) \big]^2
	\leq \fnormsq{ \mtx{A} - \hat{\mtx{A}} }.
$$
We can also approximate the leading singular vectors of $\mtx{A}$ by the leading singular vectors
of $\hat{\mtx{A}}$.  Precise statements are slightly complicated, so we
refer the reader to~\cite[Thm.~VII.5.9]{Bha97:Matrix-Analysis} for a
typical result on the perturbation theory of singular subspaces.
\end{remark}

\subsection{Analysis of Truncated Approximation}

Our second result provides a bound on the error in the truncated
approximation $\lowrank{\hat{\mtx{A}}}{r}$ of the input matrix $\mtx{A}$.

\begin{corollary}[Truncated Approximation: Error Bound]
\label{cor:fixed-rank-error-bound}
Instate the assumptions of \cref{thm:low-rank-error-bound}.
Then the rank-$r$ approximation $\lowrank{\hat{\mtx{A}}}{r}$
satisfies the error bound
$$
\Expect \fnorm{ \mtx{A} - \lowrank{\hat{\mtx{A}}}{r} }
	\leq \tau_{r + 1}(\mtx{A})
	+ 2 \left[ \frac{s - \alpha}{s - k - \alpha} \cdot \min_{\varrho < k - \alpha} \frac{k + \varrho - \alpha}{k - \varrho - \alpha} \cdot \tau_{\varrho+1}^2(\mtx{A}) \right]^{1/2}.
$$
\end{corollary}

\noindent
This statement is an immediate consequence of \cref{thm:low-rank-error-bound}
and a general bound~\cite[Prop.~6.1]{TYUC17:Practical-Sketching} for fixed-rank approximation.
We omit the details.

Let us elaborate on \cref{cor:fixed-rank-error-bound}.
If the initial approximation $\hat{\mtx{A}}$ is accurate, then the
truncated approximation $\lowrank{\hat{\mtx{A}}}{r}$ attains similar accuracy.
In particular, the rank-$r$ approximation can achieve a very small relative error when the input
matrix has a decaying spectrum.  The empirical work in \cref{sec:numerics} highlights
the practical importance of this phenomenon.

\subsection{Theoretical Guidance for Sketch Size Parameters}
\label{sec:sketch-size-theory}

If we allocate a fixed amount of storage, how can we select
the sketch size parameters $(k, s)$ to achieve superior approximations
of the input matrix?
Using the error bound from~\cref{thm:low-rank-error-bound} and prior knowledge
about the spectrum of the matrix, we can make some practical recommendations.
\Cref{sec:oracle-performance} offers numerical support for this analysis.

\subsubsection{The Storage Budget}

We have recommended using structured dimension reduction maps
$(\mtx{\Upsilon}, \mtx{\Omega}, \mtx{\Phi}, \mtx{\Psi})$
so the storage cost for the dimension reduction maps is a fixed cost
that does not increase with the sketch size parameters $(k, s)$.
Therefore, we may focus on the cost of maintaining
the sketch $(\mtx{X}, \mtx{Y}, \mtx{Z})$ itself.

Counting dimensions, via~\cref{eqn:range-sketch,eqn:core-sketch},
we see that the three approximation sketch matrices
require a total storage budget of
\begin{equation} \label{eqn:storage-budget}
T := k (m + n) + s^2
\end{equation}
floating-point numbers in the field $\F$.
How do we best expend this budget?

\subsubsection{General Spectrum}

The theoretical bound~\cref{thm:low-rank-error-bound} on the approximation error
suggests that, lacking further information, we should make the parameter $k$ as large as possible.  Indeed,
the approximation error reflects the decay in the spectrum up to the index $k$.
Meanwhile, the condition $s \geq 2k + \alpha$ in \cref{thm:low-rank-error-bound}
ensures that the first fraction in the error bound cannot exceed $2$.

Therefore, for fixed storage budget $T$, we pose the optimization problem
\begin{equation} \label{eqn:ks-natural-opt}
\text{maximize}\quad k
\quad\text{subject to}\quad
s \geq 2k + \alpha
\quad\text{and}\quad
k (m + n) + s^2 = T.
\end{equation}
Up to rounding, the solution is
\begin{equation} \label{eqn:ks-natural}
\begin{aligned}
k_{\natural} &:= \left\lfloor \frac{1}{8} \left(\sqrt{(m+n+4\alpha)^2 + 16 (T - \alpha^2)} - (m+n+4\alpha)\right) \right\rfloor; \\ %
s_{\natural} &:= \left\lfloor \sqrt{T - k_{\natural} (m+n)} \right\rfloor. %
\end{aligned}
\end{equation}
The parameter choice $(k_{\natural}, s_{\natural})$ is suitable
for a wide range of examples.

\subsubsection{Flat Spectrum}

Suppose we know that the spectrum of the input matrix does not decay
past a certain point: $\sigma_j(\mtx{A}) \approx \sigma_{\hat{\varrho}}(\mtx{A})$ for $j > \hat{\varrho}$.
In this case, the minimum value of the error~\cref{eqn:low-rank-error-bound} tends
to occur when $\varrho = \hat{\varrho}$.  %

In this case, we can obtain a theoretically supported parameter choice
$(k_{\flat}, s_{\flat})$ by numerical solution of the optimization problem
\begin{equation} \label{eqn:ks-flat}
\begin{aligned}
\text{minimize}\quad \frac{s - \alpha}{s-k-\alpha} \cdot \frac{k + \hat{\varrho} - \alpha}{k - \hat{\varrho} - \alpha}
\quad\text{subject to}\quad
&s \geq 2k + \alpha, \quad
k \geq \hat{\varrho} + \alpha + 1, \\
& %
k(m+n) + s^2 = T.
\end{aligned}
\end{equation}
This problem admits a messy closed-form solution, or it can be solved numerically.

\section{\emph{A Posteriori} Error Estimation}
\label{sec:aposterior}

The \emph{a priori} error bounds from
\cref{thm:low-rank-error-bound,cor:fixed-rank-error-bound}
are essential for setting the sketch size parameters
to make the reconstruction algorithm reliable.
To evaluate whether the approximation was actually successful,
we need \emph{a posteriori} error estimators.

For this purpose,
Martinsson~\cite[Sec.~14]{Mar19:Randomized-Methods} has proposed to extract
a very small Gaussian sketch of the input matrix, independent from the
approximation sketch.  Our deep understanding of the Gaussian
distribution allows for a refined analysis 
of error estimators computed from this sketch.

We adopt Martinsson's idea to compute a simple estimate
for the Frobenius norm of the approximation error.
\Cref{sec:diagnose-spectrum} explains how this estimator helps us select
the precise rank $r$ for the truncated approximation~\cref{eqn:Ahat-fixed}.

\subsection{The Error Sketch}

For a parameter $q$, draw and fix a standard Gaussian dimension reduction map:
\begin{equation} \label{eqn:error-test}
\mtx{\Theta} \in \F^{q \times m}.
\end{equation}
Along with the approximation sketch~\cref{eqn:range-sketch,eqn:core-sketch}, %
we also maintain an error sketch:
\begin{equation} \label{eqn:error-sketch}
\mtx{W} := \mtx{\Theta} \mtx{A} \in \F^{q \times n}.
\end{equation}
We can track the error sketch along a sequence~\cref{eqn:linear-update}
of linear updates:
\begin{equation} \label{eqn:error-sketch-update}
\mtx{W} \gets \eta \mtx{W} + \nu \mtx{\Theta} \mtx{H}.
\end{equation}
The cost of storing the test matrix and sketch is $q(m+n)$ floating-point numbers.

\begin{algorithm}[t]
  \caption{\textsl{Randomized Error Estimator.}  Implements~\cref{eqn:error-estimate}.
  \label{alg:err-est}}
  \begin{algorithmic}[1]
    \Require{Matrix approximation $\hat{\mtx{A}}_{\rm out}$}
    \Ensure{Randomized error estimate $\err_2^2(\hat{\mtx{A}}_{\rm out})$ that
    satisfies~\cref{eqn:apost-expect-2,eqn:apost-prob-lower,eqn:apost-prob-upper}}
\vspace{0.5pc}

	\Function{Sketch.ErrorEstimate}{$\hat{\mtx{A}}_{\rm out}$}

	\State	$\beta \gets 1$ for $\F = \R$ or $\beta \gets 2$ for $\F = \C$
	\State	$\err_2^2 \gets (\beta q)^{-1} \, \fnormsq{ \mtx{W} - \mtx{\Theta} \hat{\mtx{A}}_{\rm out} }$
	\State	\Return $\err_2^2$
	\EndFunction

	\vspace{0.25pc}

\end{algorithmic}
\end{algorithm}

\subsection{A Randomized Error Estimator}
\label{sec:error-estimator}

Suppose that we have computed an approximation $\hat{\mtx{A}}_{\rm out}$
of the input $\mtx{A}$ via any method.%
\footnote{We assume only that the approximation $\hat{\mtx{A}}_{\rm out}$
does not depend on the matrices $\mtx{\Theta}, \mtx{W}$.}
We can obtain a probabilistic estimate for the squared Schatten 2-norm error
in this approximation:
\begin{equation} \label{eqn:error-estimate}
\begin{aligned}
\err_{2}^2(\hat{\mtx{A}}_{\rm out})
	&:= \frac{1}{\beta q} \cdot \fnormsq{ \mtx{W} - \mtx{\Theta} \hat{\mtx{A}}_{\rm out} }
	= \frac{1}{\beta q} \cdot \fnormsq{ \mtx{\Theta}( \mtx{A} - \hat{\mtx{A}}_{\rm out}) }.
\end{aligned}
\end{equation}
Recall that $\beta = 1$ for $\F = \R$ and $\beta = 2$ for $\F = \C$.

The error estimator can be computed efficiently when the approximation
is presented in factored form.  To assess a rank-$r$ approximation $\hat{\mtx{A}}_{\rm out}$,
the cost is typically $\mathcal{O}( qr (m + n) )$ arithmetic operations.
See~\cref{alg:err-est} for pseudocode.

\begin{remark}[Prior Work]
The formula~\cref{eqn:error-estimate} is essentially a randomized trace estimator;
for example, see~\cite{Hut89:Stochastic-Estimator,AT11:Randomized-Algorithms,
RA15:Improved-Bounds,GT18:Improved-Bounds}.
Our analysis is similar to the work in these papers.
Methods for spectral norm estimation are discussed in~\cite[Sec.~3.4]{WLRT08:Fast-Randomized}
and in~\cite[Secs.~4.3--4.4]{HMT11:Finding-Structure}; these results trace their
lineage to an early paper of Dixon~\cite{Dix83:Estimating-Extremal}.
The paper~\cite{LWM18:Error-Estimation} discusses bootstrap methods
for randomized linear algebra applications.
\end{remark}

\subsection{The Error Estimator: Mean and Variance}

The error estimator delivers reliable information about
the squared Schatten 2-norm approximation error:
\begin{equation} \label{eqn:apost-expect-2}
\begin{aligned}
\Expect\big[ \err_2^2(\hat{\mtx{A}}_{\rm out}) \big]
	&= \fnormsq{ \mtx{A} - \hat{\mtx{A}}_{\rm out} }; \\
\Var\big[ \err_2^2( \hat{\mtx{A}}_{\rm out} ) \big]
	&= \frac{2}{\beta q} \, \norm{ \mtx{A} - \hat{\mtx{A}}_{\rm out} }_4^4.
\end{aligned}
\end{equation}
These results follow directly from the rotational invariance of the Schatten norms
and of the standard normal distribution.  See~\cref{app:err2-mean}.

\subsection{The Error Estimator, in Probability}

We can also obtain bounds on the probability that the error estimator
returns an extreme value.  These results justify setting the size $q$
of the error sketch to a constant.
They are also useful for placing
confidence bands on the approximation error.  
See \cref{app:aposteriori} for the proofs.

First, let us state a bound on the probability that the estimator
reports a value that is much too small.  We have
\begin{equation} \label{eqn:apost-prob-lower}
\mathbb{P}_{\mtx{\Theta}}\left\{ \err_{2}^2(\hat{\mtx{A}}_{\rm out})
	\leq (1 - \eps) \fnormsq{ \mtx{A} - \hat{\mtx{A}}_{\rm out} } \right\}
	\leq \left[ \econst^{\eps} (1 - \eps) \right]^{\beta q/2}
	\quad\text{for $\eps \in (0, 1)$.}
\end{equation}
For example, the error estimate is smaller than
$0.1 \times$ the true error value with probability less than $2^{-\beta q}$.

Next, we provide a bound on the probability that the estimator
reports a value that is much too large.  We have
\begin{equation} \label{eqn:apost-prob-upper}
\mathbb{P}_{\mtx{\Theta}}\left\{ \err_{2}^2(\hat{\mtx{A}}_{\rm out})
	\geq (1 + \eps) \fnormsq{ \mtx{A} - \hat{\mtx{A}}_{\rm out} } \right\}
	\leq \left[ \frac{\econst^{\eps}}{1 + \eps} \right]^{-\beta q/2}
	\quad\text{for $\eps > 0$.}
\end{equation}
For example, the error estimate exceeds $4 \times$
the true error value with probability less than $2^{-\beta q}$.

\begin{remark}[Estimating Normalized Errors]
We may wish to compute the error of an approximation $\hat{\mtx{A}}_{\rm out}$
on the scale of the energy $\fnormsq{\mtx{A}}$ in the input matrix.
To that end, observe that $\err_2^2( \mtx{0} )$ is an estimate for $\fnormsq{\mtx{A}}$.
Therefore, the ratio $\err_2^2(\hat{\mtx{A}}_{\rm out}) / \err_2^2(\mtx{0})$
gives a good estimate for the normalized error. %
\end{remark}

\subsection{Diagnosing Spectral Decay}
\label{sec:diagnose-spectrum}

In many applications, our goal is to estimate a rank-$r$
truncated SVD of the input matrix that captures most of its spectral energy.
It is rare, however, that we can prophesy the precise value $r$ of the rank.
A natural solution is to use the spectral characteristics of the
initial approximation $\hat{\mtx{A}}$, defined in~\cref{eqn:Ahat},
to decide where to truncate.  We can deploy the error estimator
$\err_2^2$ to implement this strategy in a principled way and to validate the results.
See~\cref{sec:apost-experiments,sec:sst-data} for numerics.

If we had access to the full input matrix $\mtx{A}$, we would compute
the proportion of tail energy remaining after a rank-$r$ approximation:
\begin{equation} \label{eqn:scree}
\mathrm{scree}(r) := \left[ \frac{\tau_{r+1}(\mtx{A})}{ \fnorm{\mtx{A}} } \right]^2
	= \left[ \frac{\fnorm{ \mtx{A} - \lowrank{\mtx{A}}{r}}}{ \fnorm{\mtx{A}} } \right]^2.
\end{equation}
A visualization of the function~\eqref{eqn:scree} is called a \emph{scree plot}.
A standard technique for rank selection is to identify a ``knee''
in the scree plot.  It is also possible to apply quantitative
model selection criteria to the function~\cref{eqn:scree}.
See~\cite[Chap.~6]{Jol02:Principal-Component} for an extensive discussion.

We cannot compute~\cref{eqn:scree} without access to the input matrix,
but we can use the initial approximation and the error estimator creatively.
For $r \ll k$, the tail energy $\tau_{r+1}(\hat{\mtx{A}})$
of the initial approximation is a proxy for the tail energy $\tau_{r+1}(\mtx{A})$
of the input matrix.  This observation suggests that we consider the (lower) estimate
\begin{equation} \label{eqn:scree-lower}
\underline{\mathrm{scree}}(r) := \left[ \frac{ \tau_{r+1}(\hat{\mtx{A}}) }{ \err_2(\mtx{0}) } \right]^2
	= \left[ \frac{ \fnorm{\hat{\mtx{A}} - \lowrank{\hat{\mtx{A}}}{r}} }{ \err_2(\mtx{0}) } \right]^2.
\end{equation}
This function tracks the actual scree curve~\cref{eqn:scree} when $r \ll k$.
It typically underestimates the scree curve, and the underestimate is severe for large $r$.

To design a more rigorous approach, notice that
$$
\abs{ \tau_{r+1}(\mtx{A}) - \tau_{r+1}(\hat{\mtx{A}}) }
	\leq \fnorm{ \mtx{A} - \hat{\mtx{A}} } \approx \err_2(\hat{\mtx{A}}).
$$
The inequality requires a short justification; see~\cref{app:diagnose-spectrum}.
This bound suggests that we consider the (upper) estimator
\begin{equation} \label{eqn:scree-upper}
\overline{\mathrm{scree}}(r) := \left[ \frac{ \tau_{r+1}(\hat{\mtx{A}}) + \err_2(\hat{\mtx{A}}) }{ \err_2(\mtx{0}) } \right]^2.
\end{equation}
This function also tracks the actual scree curve~\cref{eqn:scree} when $r \ll k$.
It reliably overestimates the scree curve by a modest amount.

\section{Numerical Experiments}
\label{sec:numerics}

This section presents computer experiments that are
designed to evaluate the performance of the proposed
sketching algorithms for low-rank matrix approximation.
We include comparisons with alternative methods from the literature
to argue that the proposed approach produces superior results.
We also explore some applications to scientific simulation
and data analysis.

\subsection{Alternative Sketching and Reconstruction Methods}
\label{sec:other-algs}

We compare the proposed method~\cref{eqn:Ahat-fixed}
with three other algorithms
that construct a fixed-rank approximation
of a matrix from a random linear sketch:

\vspace{0.5pc}

\begin{enumerate} \setlength{\itemsep}{0.5pc}

\item	The [HMT11] method~\cite[Sec.~5.5, Rem.~5.4]{HMT11:Finding-Structure}
is a simplification of the method from Woolfe et al.~\cite[Sec.~5.2]{WLRT08:Fast-Randomized},
and they perform similarly.  There are two sketches, and the sketch size depends on one parameter $k$.
The total storage cost $T = k(m+n)$.

\item	The [TYUC17] method~\cite[Alg.~7]{TYUC17:Practical-Sketching} is a numerically stable
and more fully realized implementation of a proposal due to Clarkson \& Woodruff~\cite[Thm.~4.9]{CW09:Numerical-Linear}.
It involves two sketches, controlled by two parameters $k, \ell$.
The total storage cost $T = km + \ell n$.

\item	The [Upa16] method~\cite[Sec.~3.3]{Upa16:Fast-Space-Optimal}
simplifies a complicated approach from Boutsidis et al.~\cite[Thm.~12]{BWZ16:Optimal-Principal-STOC}. %
This algorithm involves three sketches, controlled by two parameters $k, s$.
The total storage cost $T = k(m+n) + s^2$.

\item	Our new method~\cref{eqn:Ahat-fixed} simultaneously extends
[Upa16] and [TYUC17].  It uses three sketches,
controlled by two parameters $k, s$.  The total storage cost $T = k(m+n) + s^2$.
\end{enumerate}

\vspace{0.5pc}

See~\cref{app:alt-sketch} for a more detailed description of these methods.
In each case, the storage budget neglects the cost of storing the dimension reduction maps
because this cost has lower order than the sketch when we use structured dimension reduction maps.
These methods have similar arithmetic costs, so we will not make a comparison
of runtimes.  Storage is the more significant issue for sketching algorithms.
We do not include storage costs for an error estimator in the comparisons.

Our recent paper~\cite{TYUC17:Practical-Sketching} demonstrates that
several other methods (\cite[Thm.~4.3, display 2]{Woo14:Sketching-Tool}
and \cite[Sec.~10.1]{CEM+15:Dimensionality-Reduction}) are uncompetitive,
so we omit them.

\subsection{Experimental Setup}

Our experimental design is quite similar to our previous
papers~\cite{TYUC17:Practical-Sketching,TYUC17:Fixed-Rank-Approximation}
on sketching algorithms for low-rank matrix approximation.

\subsubsection{Procedure}

Fix an input matrix $\mtx{A} \in \F^{n \times n}$ and a truncation rank $r$.
Select sketch size parameters. %
For each trial, draw dimension reduction maps %
from a specified distribution and form the sketch %
of the input matrix.  Compute a rank-$r$ approximation $\hat{\mtx{A}}_{\mathrm{out}}$ using
a specified reconstruction algorithm.  The approximation error is calculated
relative to the best rank-$r$ approximation error in Schatten $p$-norm:
\begin{equation} \label{eqn:relative-error}
\text{$S_p$ relative error} \quad = \quad
\frac{\norm{ \mtx{A} - \hat{\mtx{A}}_{\mathrm{out}}}_{p}}{\norm{ \mtx{A} - \lowrank{\mtx{A}}{r}}_{p}} - 1.
\end{equation}
We perform 20 independent trials and report the average error.  Owing to measure concentration effects,
the average error is also the typical error; see~\cref{sec:apost-experiments}.

In all experiments, we work in double-precision arithmetic
(i.e., 8 bytes per real floating-point number).
The body of this paper presents a limited selection of results.
\Cref{app:numerics} contains additional numerical evidence.
The supplementary materials also include \textsc{Matlab} code
that can reproduce these experiments.

\subsubsection{The Oracle Error}
\label{sec:oracle-error}

To make fair comparisons among algorithms, we can fix the storage budget %
and identify the parameter choices that minimize the (average) relative error~\cref{eqn:relative-error}
incurred over the repeated trials.  We refer to the minimum as the \emph{oracle error} for an algorithm.
The oracle error is not attainable in practice.

\subsection{Classes of Input Matrices}
\label{sec:input-matrix-examples}

As in our previous papers~\cite{TYUC17:Fixed-Rank-Approximation,TYUC17:Practical-Sketching},
we consider several different types of synthetic and real input matrices.
See \cref{fig:spectra} for a plot of the spectra of these input matrices.

\subsubsection{Synthetic Examples}

We work over the complex field $\C$. %
The matrix dimensions $m = n = 10^3$, and we introduce an effective rank parameter $R \in \{5, 10, 20\}$.
In each case, we compute an approximation with truncation rank $r = 10$.

\vspace{0.5pc}

\begin{enumerate} \setlength{\itemsep}{0.5pc}
\item	\textbf{Low-rank + noise:}  Let $\xi \geq 0$ be a signal-to-noise parameter.
These matrices take the form
$$
\mtx{A} = \diag(\underbrace{1,\dots,1}_R, 0, \dots,0) %
	+ \xi n^{-1} \mtx{C} \in \C^{n \times n},
$$
where $\mtx{C} = \mtx{GG}^*$ for a standard normal matrix $\mtx{G} \in \F^{n \times n}$.
We consider several parameter values: \texttt{LowRankLowNoise} ($\xi = 10^{-4}$),
\texttt{LowRankMedNoise} ($\xi = 10^{-2}$), \texttt{LowRankHiNoise} ($\xi = 10^{-1}$).

\item	\textbf{Polynomial decay:}  For a decay parameter $p > 0$, consider matrices %
$$
\mtx{A} = \diag(\underbrace{1, \dots, 1}_R, 2^{-p}, 3^{-p}, \dots, (n-R+1)^{-p}) \in \C^{n \times n}.
$$
We study three examples: \texttt{PolyDecaySlow} ($p = 0.5$),
\texttt{PolyDecayMed} ($p = 1$), \texttt{PolyDecayFast} ($p = 2$).

\item	\textbf{Exponential decay:} For a decay parameter $q > 0$, consider matrices
$$
\mtx{A} = \diag(\underbrace{1, \dots, 1}_R, 10^{-q}, 10^{-2q}, \dots, 10^{-(n-R)q}) \in \C^{n \times n}.
$$
We consider the cases \texttt{ExpDecaySlow} ($q = 0.01$),
\texttt{ExpDecayMed} ($q = 0.1$), \texttt{ExpDecayFast} ($q = 0.5$).

\end{enumerate}

\vspace{0.5pc}

\begin{remark}[Non-Diagonal Matrices]
We have also performed experiments using non-diagonal matrices with the same spectra.
The results were essentially identical.
\end{remark}

\subsubsection{Application Examples}

Next, we present some low-rank data matrices that arise in applications.
The truncation rank $r$ varies, depending on the matrix.

\vspace{0.5pc}

\begin{enumerate} \setlength{\itemsep}{0.5pc}
\item	\textbf{Navier--Stokes:}  We test the hypothesis,
discussed in \cref{sec:vignette}, that sketching methods
can be used to perform on-the-fly compression of the output of a PDE simulation.
We have obtained a direct numerical simulation (DNS) on a coarse mesh
of the 2D Navier--Stokes equations for a low-Reynolds number flow around a cylinder.
The simulation is started impulsively from a rest state.
Transient dynamics emerge in the first third of the simulation, while the remaining
time steps capture the limit cycle.  Each of the velocity and pressure fields is centered around its temporal mean.
This data is courtesy of Beverley McKeon and Sean Symon.

\vspace{0.5pc}

The real $m \times n$ matrix \texttt{StreamVel} contains streamwise velocities
at $m = 10,738$ points for each of $n = 5,001$ time instants.  The first 20 singular
values of the matrix decay by two orders of magnitude, and the rest of the spectrum
exhibits slow exponential decay. %

\item	\textbf{Weather:}  We test the hypothesis that sketching methods
can be used to perform on-the-fly compression of temporal data as it is collected.
We have obtained a matrix that tabulates meteorological variables at weather
stations across the northeastern United States on days during the years 1981--2016.
This data is courtesy of William North.

\vspace{0.5pc}

The real $m \times n$ matrix \texttt{MinTemp} contains the minimum temperature
recorded at each of $m = 19,264$ stations on each of $n = 7,305$ days.
The first 10 singular values decay by two orders of magnitude,
while the rest of the spectrum has medium polynomial decay. %

\item	\textbf{Sketchy Decisions:}  We also consider matrices that arise
from an optimization algorithm for solving large-scale semidefinite programs~\cite{YUTC17:Sketchy-Decisions}.
In this application, the data matrices are presented as a long series of rank-one updates,
and sketching is a key element of the algorithm.

\vspace{0.5pc}

\begin{enumerate}
\item	\texttt{MaxCut}: This is a real psd matrix with $m = n = 2,000$
that gives a high-accuracy solution to the \textsc{MaxCut} SDP for a sparse graph~\cite{GW95:Improved-Approximation}.
This matrix is effectively rank deficient with $R = 14$, and the spectrum
has fast exponential decay after this point.

\vspace{0.5pc}

\item	\texttt{PhaseRetrieval}: This is a complex psd matrix with $m = n = 25,921$
that gives a low-accuracy solution to a phase retrieval SDP~\cite{HCO+15:Solving-Ptychography}.
This matrix is effectively rank deficient with $R = 5$, and the spectrum has fast exponential
decay after this point.
\end{enumerate}
\vspace{0.5pc}

\item	\textbf{Sea Surface Temperature Data}: Last, we use a moderately large
climate dataset to showcase our overall methodology.
This data is provided by the National Oceanic and Atmospheric Administration (NOAA);
see~\cite{RRS+02:Improved-In-Situ,RSL+07:Daily-High-Resolution}
for details about the data preparation methodology.

\vspace{0.5pc}

The real $m \times n$ matrix \texttt{SeaSurfaceTemp}
consists of daily temperature estimates at $m = 691,150$ regularly spaced points
in the ocean for each of $n = 13,670$ days between 1981 and 2018.
\end{enumerate}

\subsection{Insensitivity to Dimension Reduction Map}
\label{sec:universality}

The proposed reconstruction method~\cref{eqn:Ahat-fixed} is insensitive
to the choice of dimension reduction map at the oracle parameter values
(\cref{sec:oracle-error}).  As a consequence, we can transfer
theoretical and empirical results for Gaussians to SSRFT
and sparse dimension reduction maps.  See~\cref{app:universality}
for numerical evidence.

\subsection{Approaching the Oracle Performance}
\label{sec:oracle-performance}

We can almost achieve the oracle error by implementing the
reconstruction method~\cref{eqn:Ahat-fixed} with sketch
size parameters chosen using the theory in \cref{sec:sketch-size-theory}.
This observation justifies the use of the theoretical parameters
when we apply the algorithm.  See~\cref{app:oracle-performance}
for numerical evidence.

\begin{figure}[t]
\begin{center}
\begin{subfigure}{.325\textwidth}
\begin{center}
\includegraphics[height=1.5in]{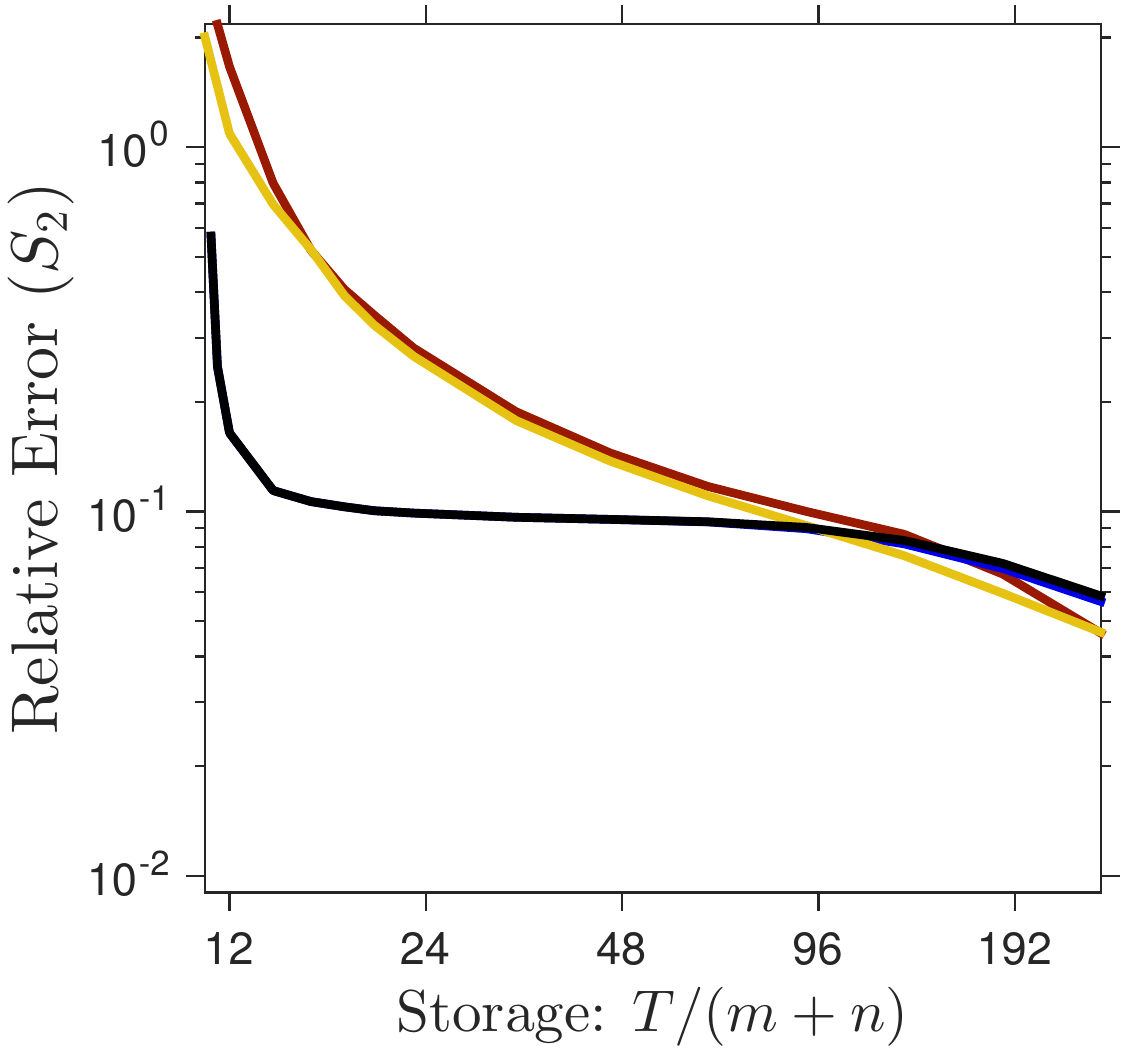}
\caption{\texttt{LowRankHiNoise}}
\end{center}
\end{subfigure}
\begin{subfigure}{.325\textwidth}
\begin{center}
\includegraphics[height=1.5in]{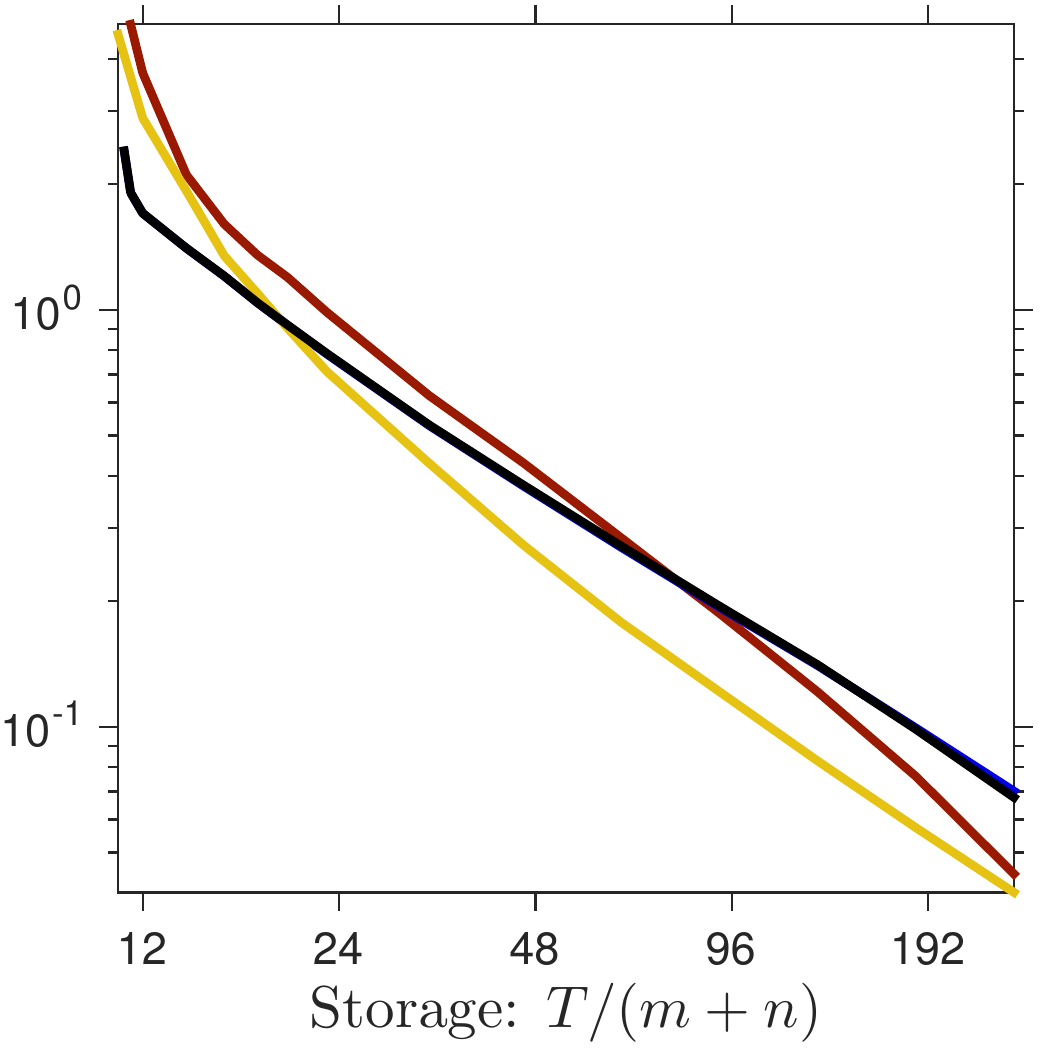}
\caption{\texttt{LowRankMedNoise}}
\end{center}
\end{subfigure}
\begin{subfigure}{.325\textwidth}
\begin{center}
\includegraphics[height=1.5in]{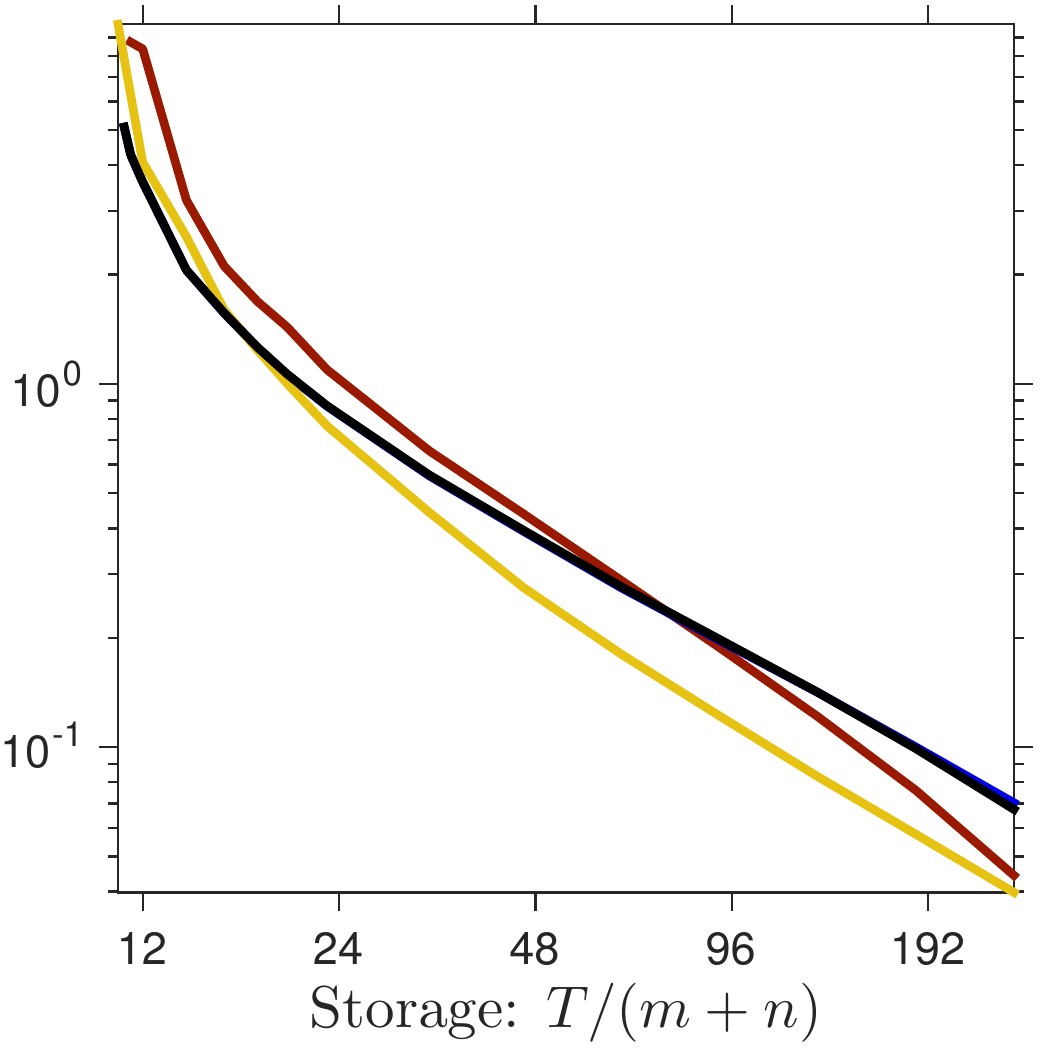}
\caption{\texttt{LowRankLowNoise}}
\end{center}
\end{subfigure}
\end{center}

\vspace{.5em}

\begin{center}
\begin{subfigure}{.325\textwidth}
\begin{center}
\includegraphics[height=1.5in]{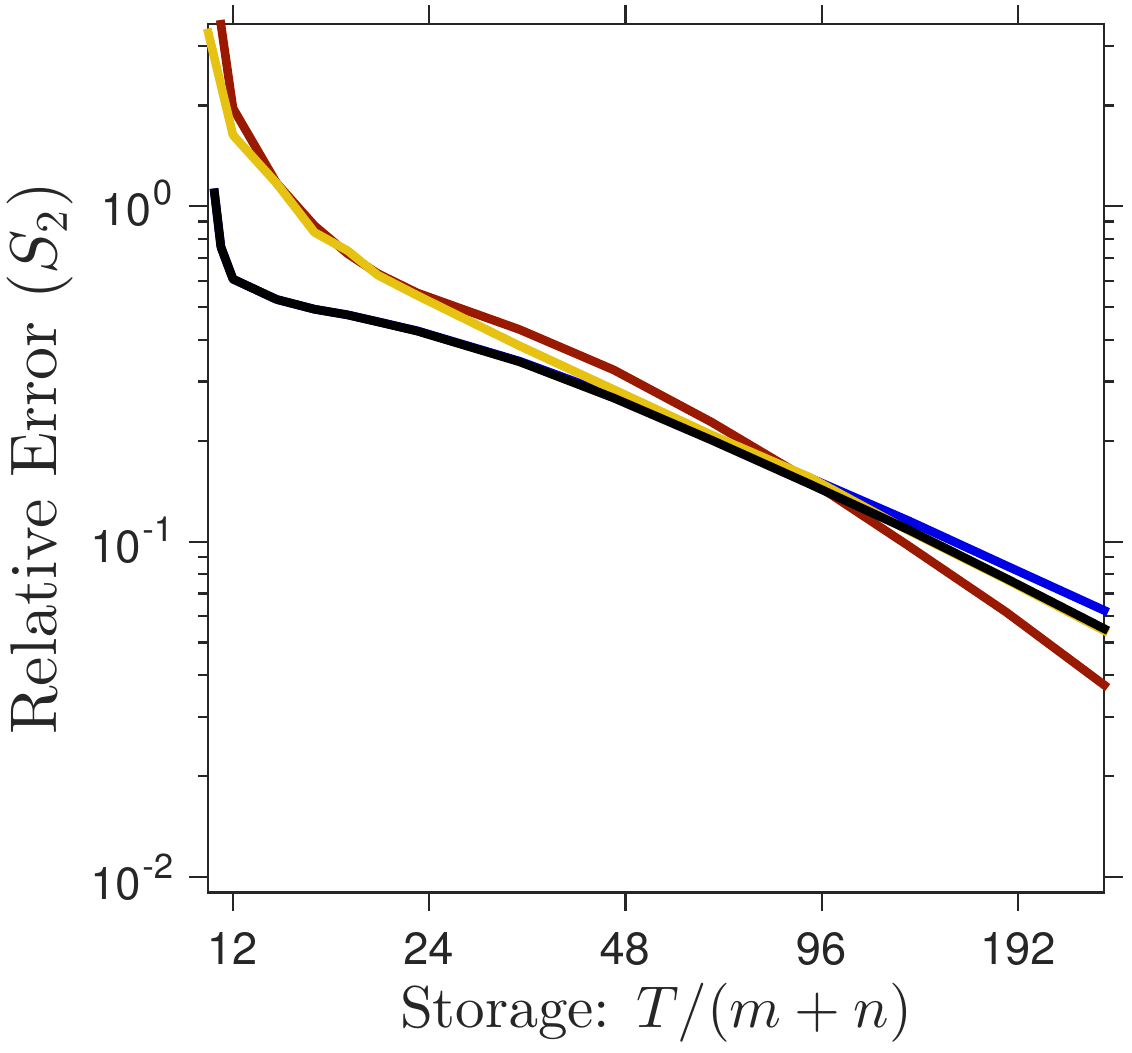}
\caption{\texttt{PolyDecaySlow}}
\end{center}
\end{subfigure}
\begin{subfigure}{.325\textwidth}
\begin{center}
\includegraphics[height=1.5in]{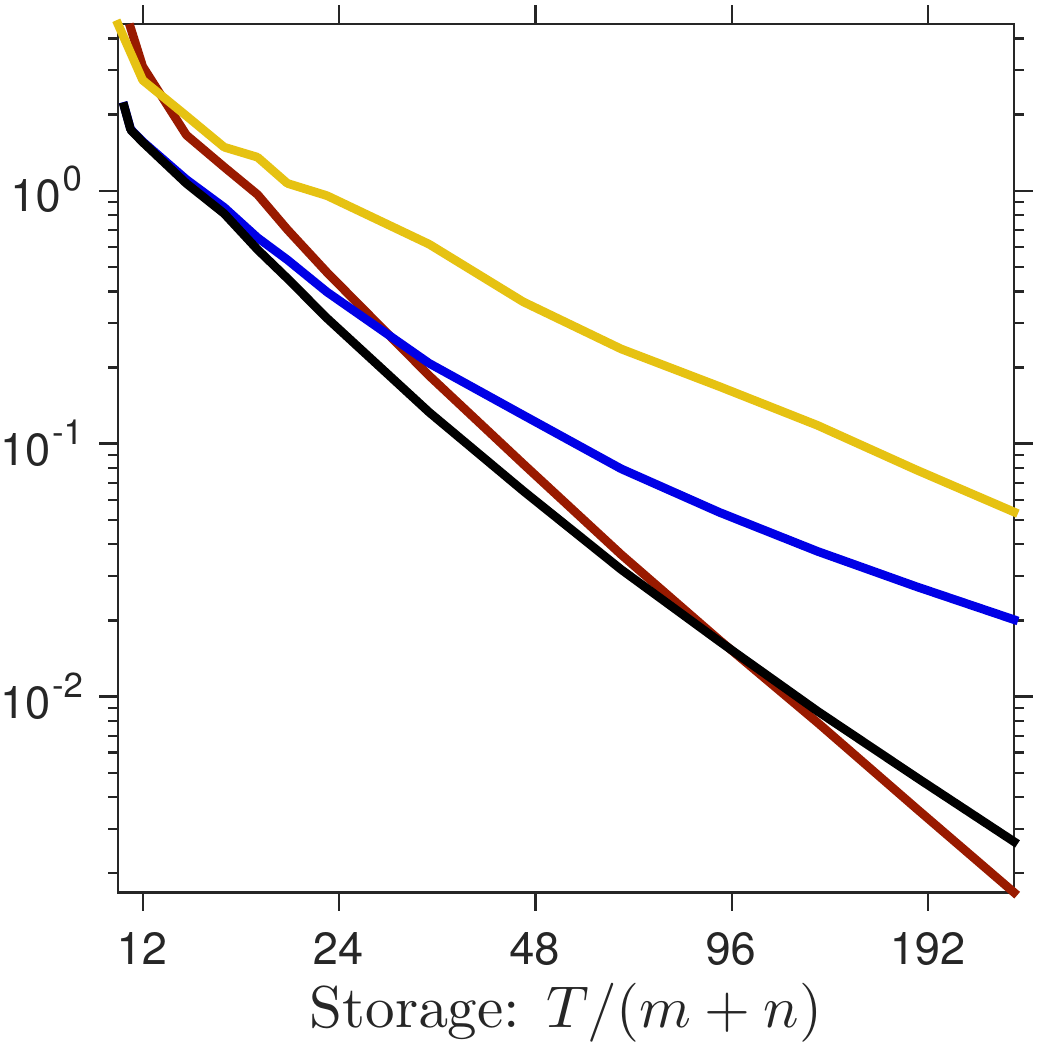}
\caption{\texttt{PolyDecayMed}}
\end{center}
\end{subfigure}
\begin{subfigure}{.325\textwidth}
\begin{center}
\includegraphics[height=1.5in]{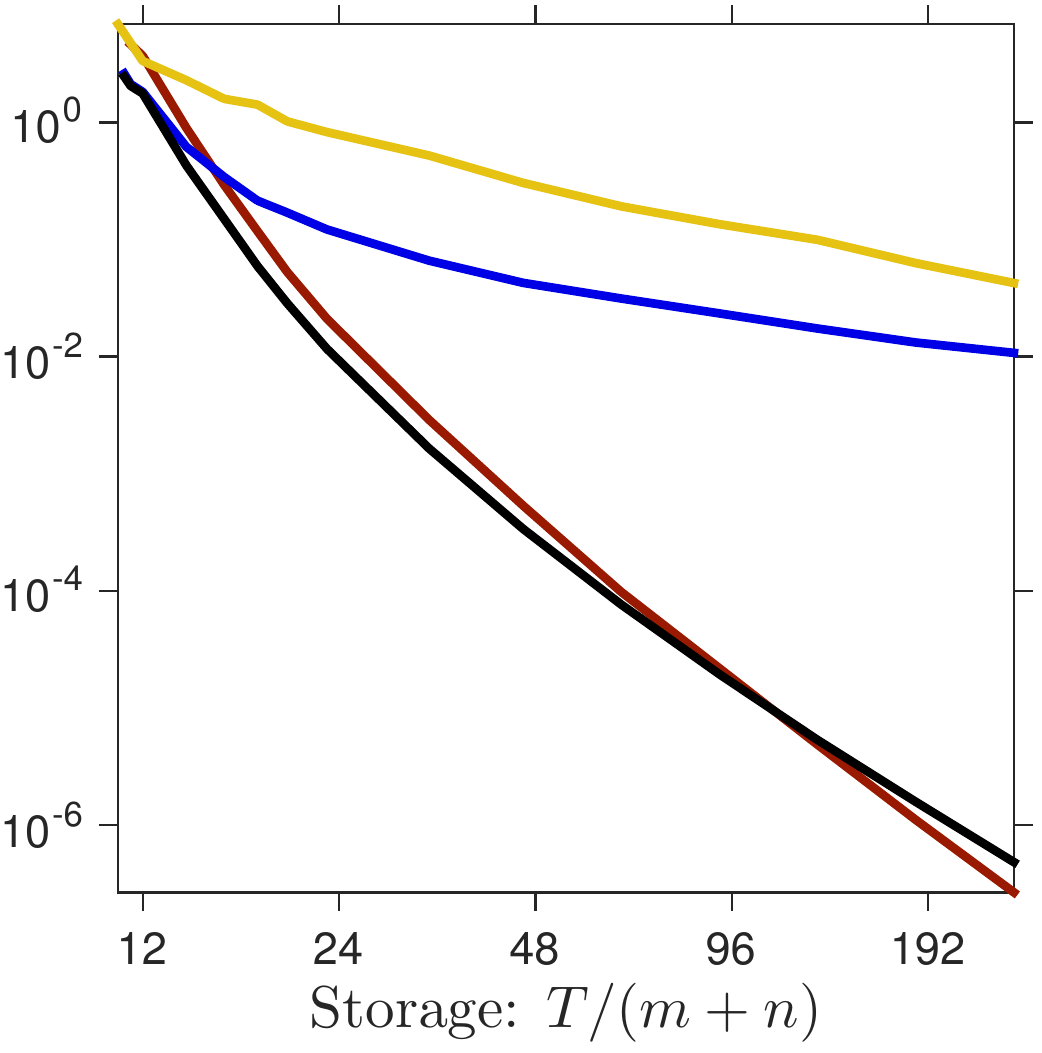}
\caption{\texttt{PolyDecayFast}}
\end{center}
\end{subfigure}
\end{center}

\vspace{0.5em}

\begin{center}
\begin{subfigure}{.325\textwidth}
\begin{center}
\includegraphics[height=1.5in]{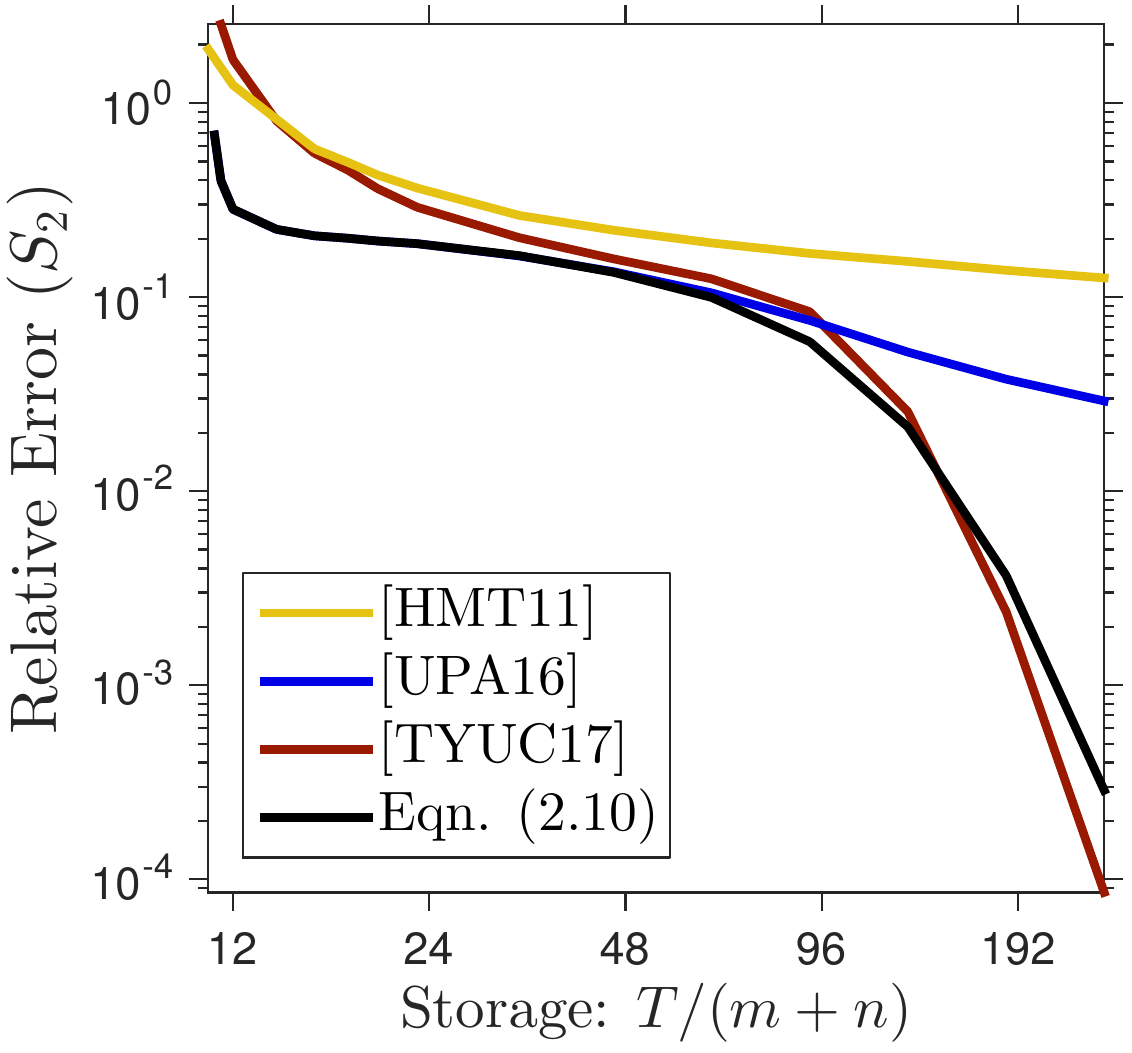}
\caption{\texttt{ExpDecaySlow}}
\end{center}
\end{subfigure}
\begin{subfigure}{.325\textwidth}
\begin{center}
\includegraphics[height=1.5in]{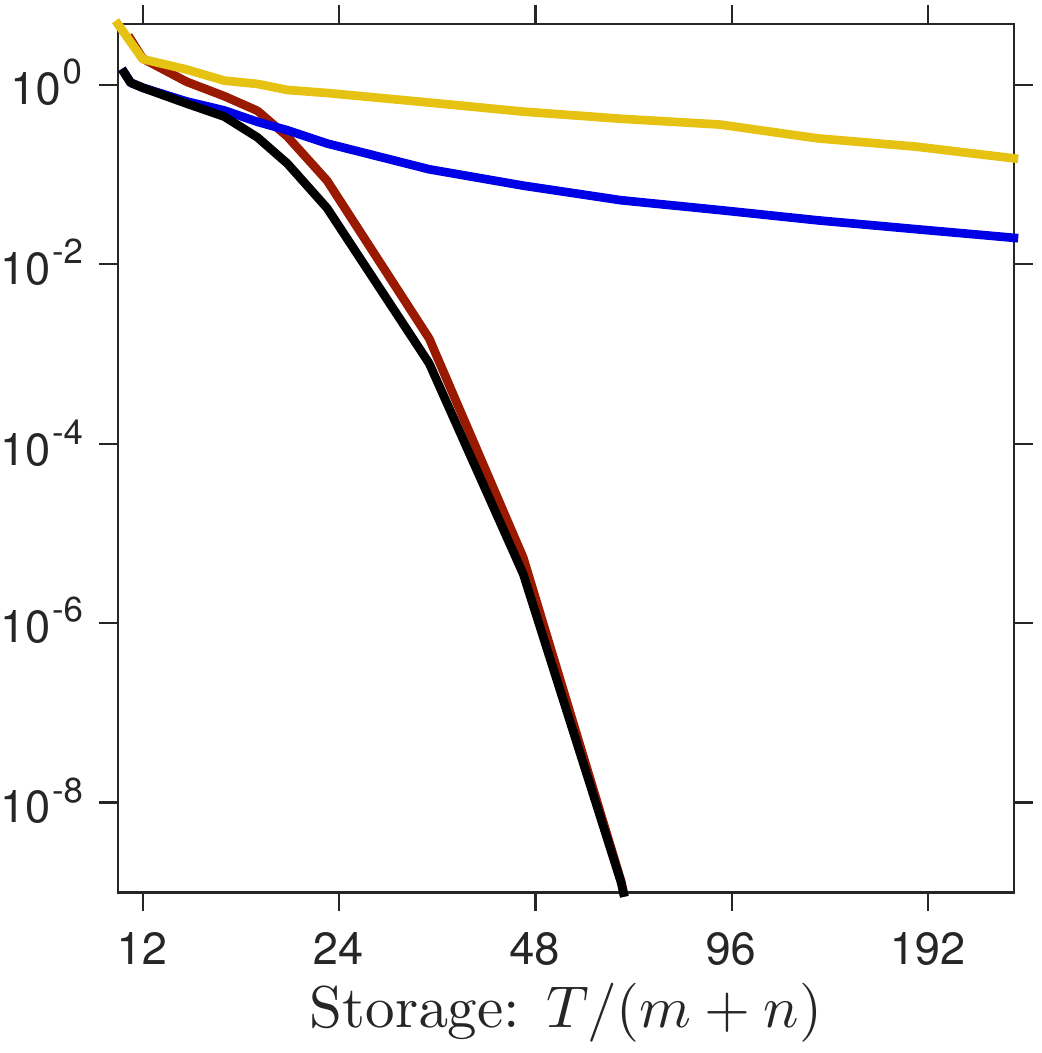}
\caption{\texttt{ExpDecayMed}}
\end{center}
\end{subfigure}
\begin{subfigure}{.325\textwidth}
\begin{center}
\includegraphics[height=1.5in]{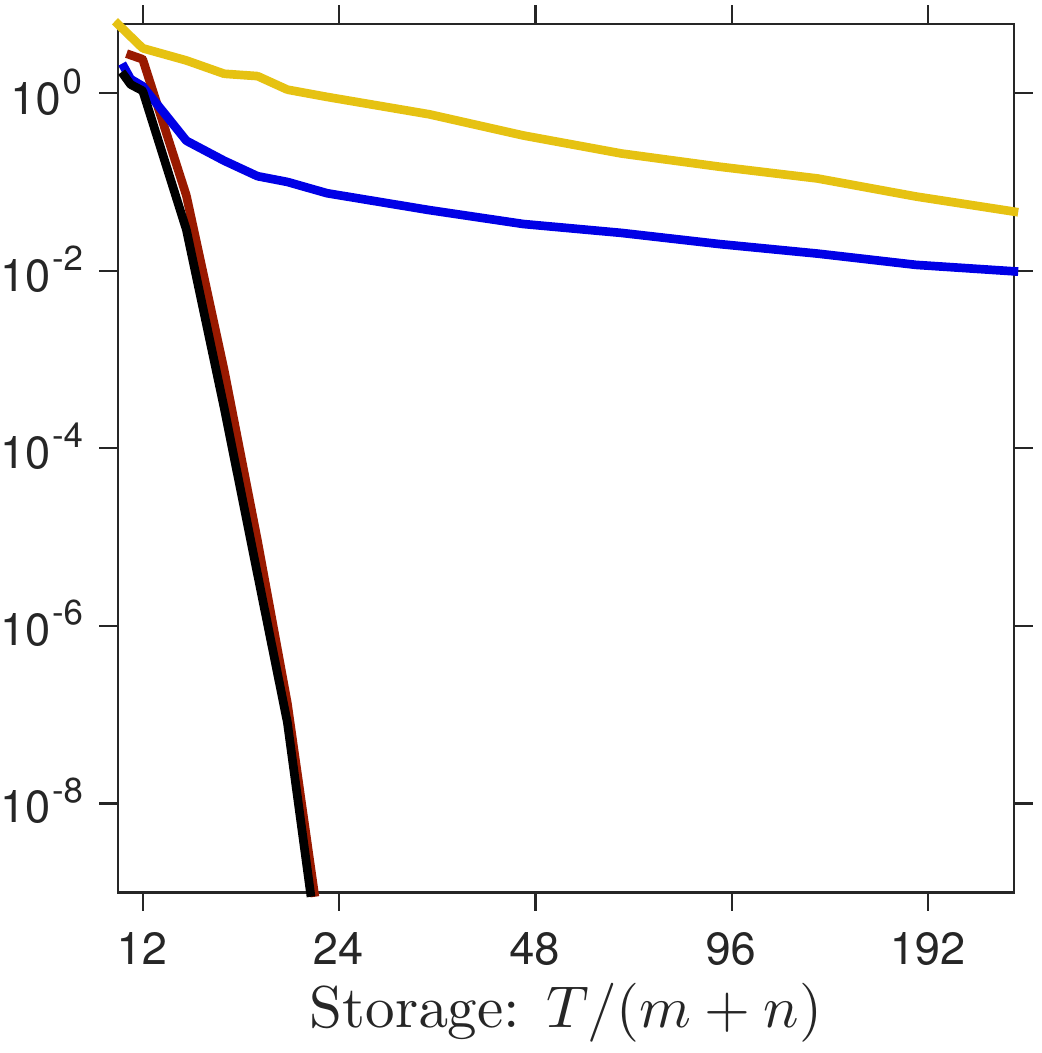}
\caption{\texttt{ExpDecayFast}}
\end{center}
\end{subfigure}
\end{center}

\vspace{0.5em}

\caption{\textbf{Comparison of reconstruction formulas: Synthetic examples.}
(Gaussian maps, effective rank $R = 10$, approximation rank $r = 10$, Schatten 2-norm.)
We compare the oracle error achieved by the proposed fixed-rank
approximation~\cref{eqn:Ahat-fixed} against methods [HMT11], [Upa16], and [TYUC17]
from the literature.
See \cref{sec:oracle-error} for details.}
\label{fig:oracle-comparison-R10-S2}
\end{figure}

\subsection{Comparison of Reconstruction Formulas: Synthetic Examples}
\label{sec:alg-comparison}

Let us now compare the proposed rank-$r$ reconstruction formula~\cref{eqn:Ahat-fixed}
with [HMT11], [Upa16], and [TYUC17] on synthetic data.

\Cref{fig:oracle-comparison-R10-S2} %
present the results of the following experiment.
For synthetic matrices with effective rank $R = 10$ and truncation rank $r = 10$,
we compare the relative error~\cref{eqn:relative-error} achieved by each of the four
algorithms as a function of storage. %
We use Gaussian dimension reduction maps in these experiments; similar results are
evident for other types of maps.  Results for effective rank $R \in \{ 5, 20 \}$
and Schatten $\infty$-norm appear in \cref{app:numerics-comparison-synth}.
Let us make some remarks:

\vspace{0.5pc}

\begin{itemize}
\item	This experiment demonstrates clearly that the proposed approximation~\cref{eqn:Ahat-fixed}
improves over the earlier methods for most of the synthetic input matrices, almost uniformly
and sometimes by orders of magnitude.

\item	For input matrices where the spectral tail decays slowly (\texttt{PolyDecaySlow},
\texttt{LowRankLowNoise}, \texttt{LowRankMedNoise}, \texttt{LowRankHiNoise}),
the newly proposed method~\cref{eqn:Ahat-fixed} has identical behavior to [Upa16].
The new method is slightly worse than [HMT11] in several of these cases.

\item	For input matrices whose spectral tail decays more quickly (\texttt{ExpDecaySlow},
\texttt{ExpDecayMed}, \texttt{ExpDecayFast}, \texttt{PolyDecayMed}, \texttt{PolyDecayFast}),
the proposed method improves dramatically over [HMT11] and [Upa16].

\item	The new method~\cref{eqn:Ahat-fixed} shows its strength over [TYUC17]
when the storage budget is small.  It also yields superior
performance in Schatten $\infty$-norm.  These differences are most evident for matrices
with slow spectral decay.

\end{itemize}

\vspace{0.5pc}

\begin{figure}[htp!]

\begin{center}
\begin{subfigure}{.45\textwidth}
\begin{center}
\includegraphics[height=2in]{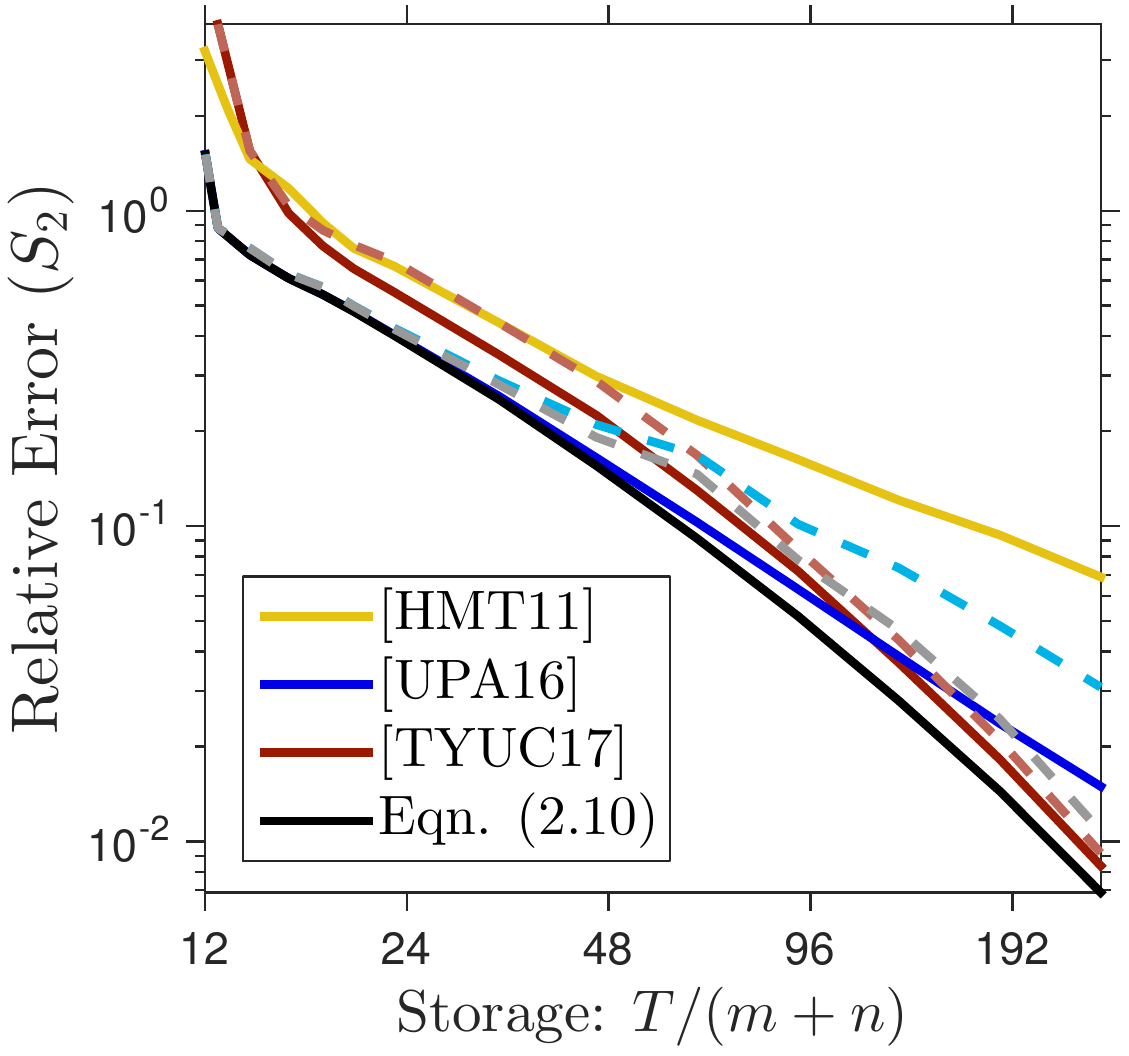}
\caption{\texttt{MinTemp} ($r = 10$)}
\end{center}
\end{subfigure}
\begin{subfigure}{.45\textwidth}
\begin{center}
\includegraphics[height=2in]{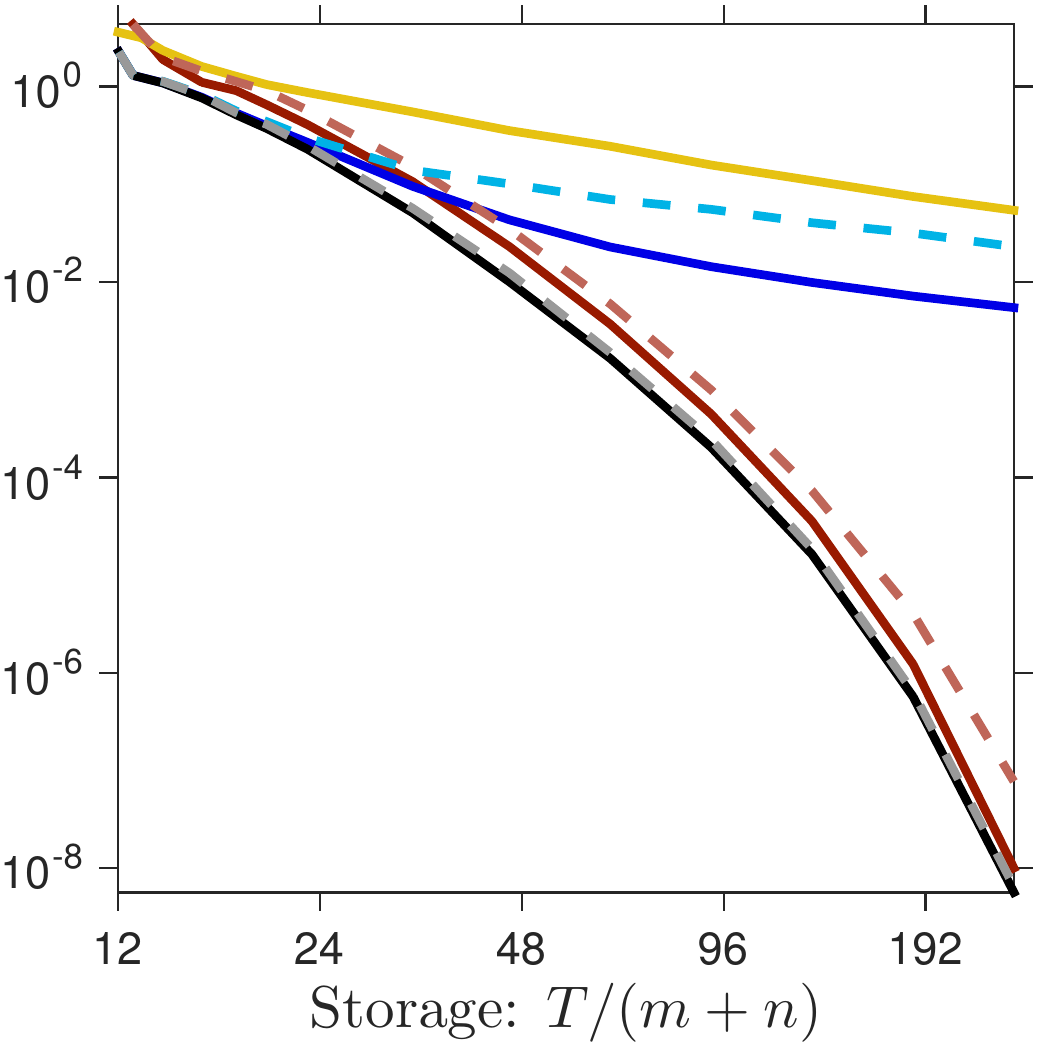}
\caption{\texttt{StreamVel} ($r = 10$)}
\end{center}
\end{subfigure}
\end{center}

\vspace{0.5em}

\begin{center}
\begin{subfigure}{.45\textwidth}
\begin{center}
\includegraphics[height=2in]{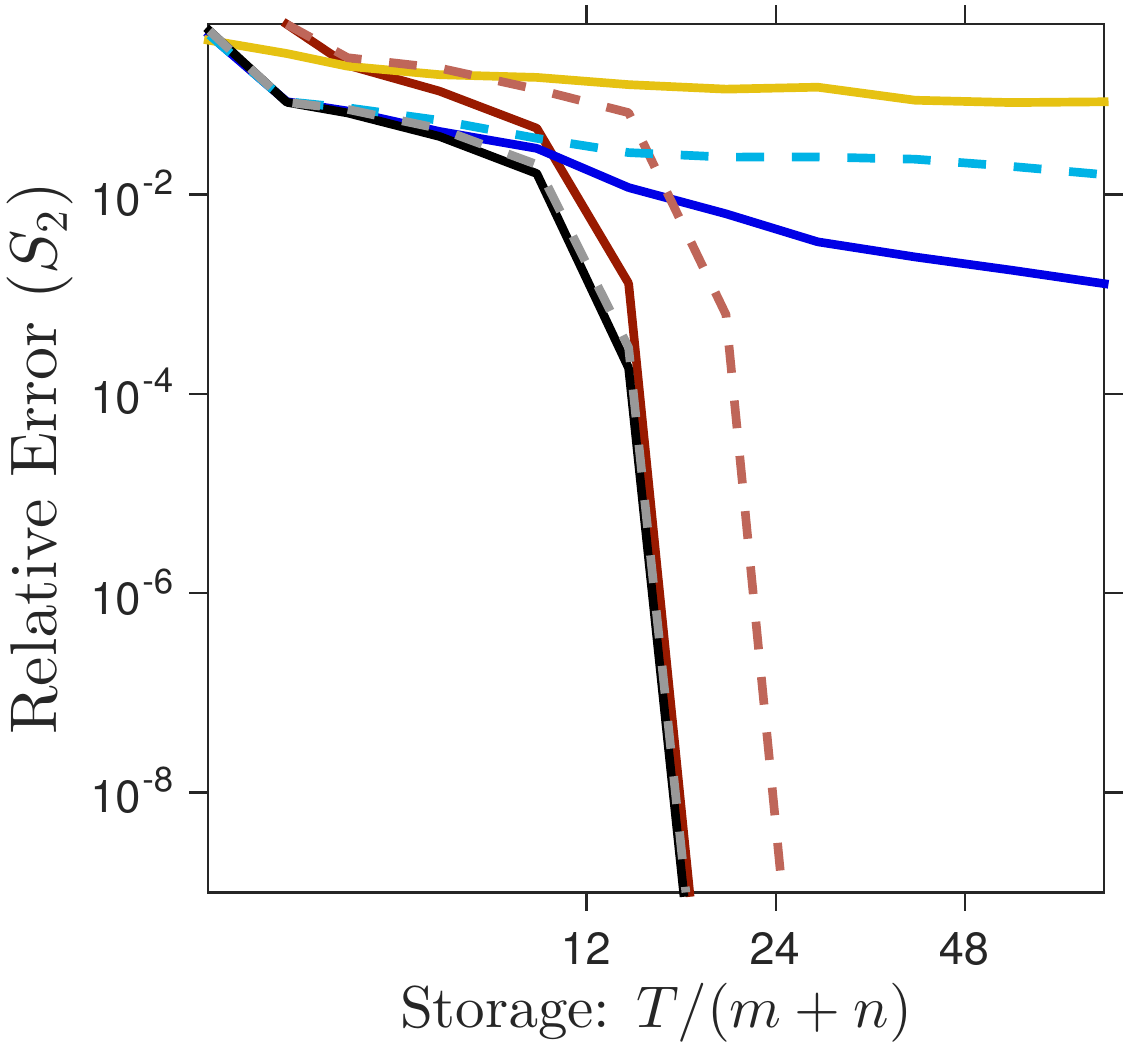}
\caption{\texttt{MaxCut} ($r = 1$)}
\end{center}
\end{subfigure}
\begin{subfigure}{.45\textwidth}
\begin{center}
\includegraphics[height=2in]{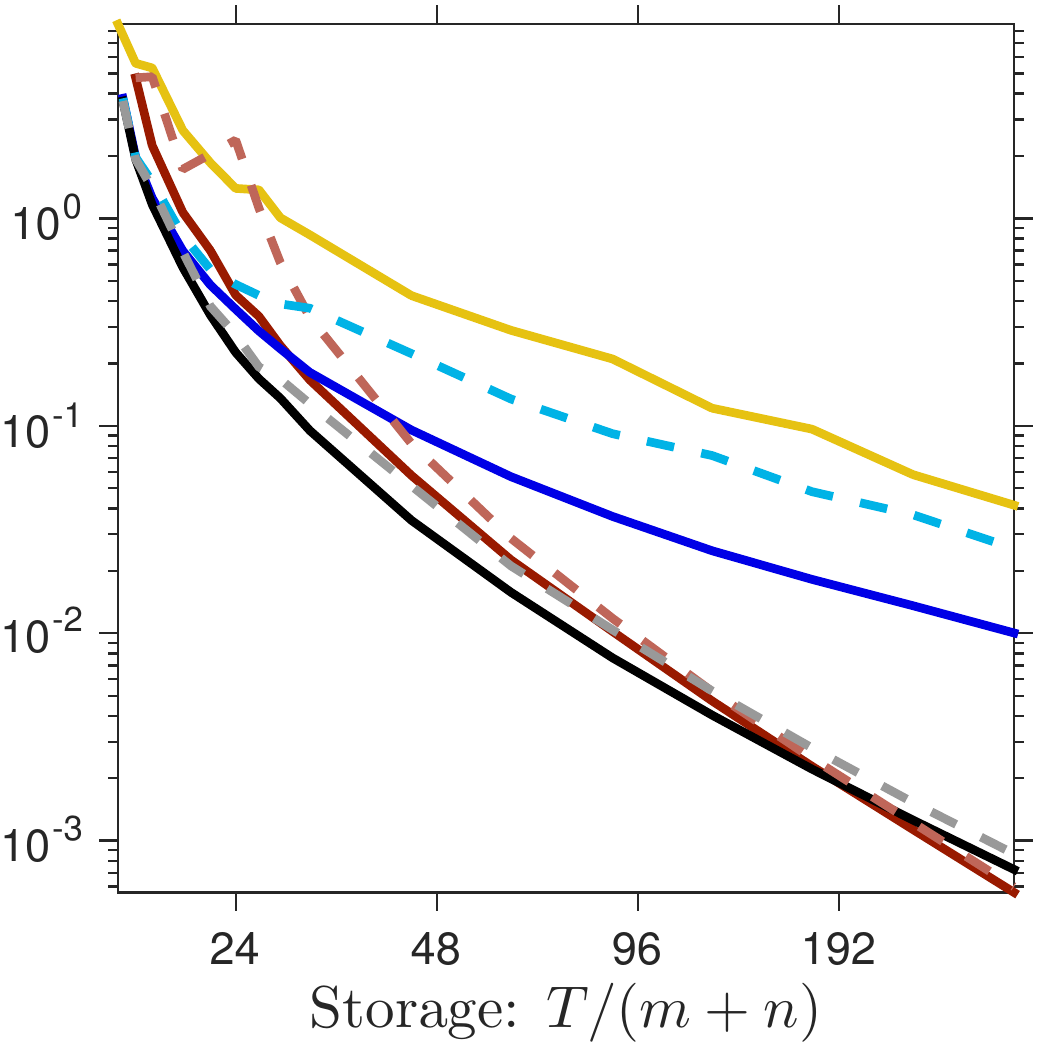}
\caption{\texttt{MaxCut} ($r = 14$)}
\end{center}
\end{subfigure}
\end{center}

\vspace{.5em}

\begin{center}
\begin{subfigure}{.45\textwidth}
\begin{center}
\includegraphics[height=2in]{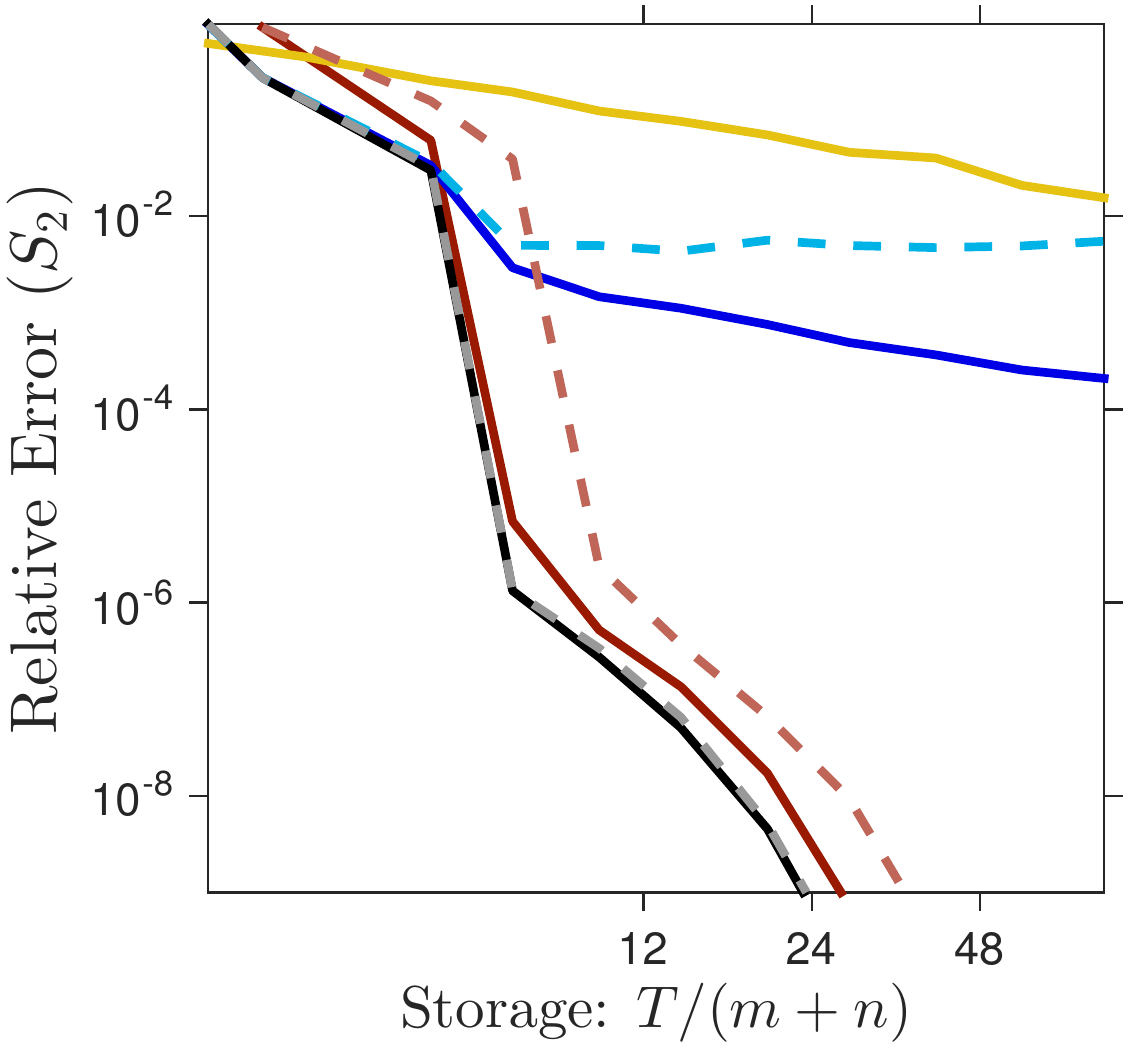}
\caption{\texttt{PhaseRetrieval} ($r = 1$)}
\end{center}
\end{subfigure}
\begin{subfigure}{.45\textwidth}
\begin{center}
\includegraphics[height=2in]{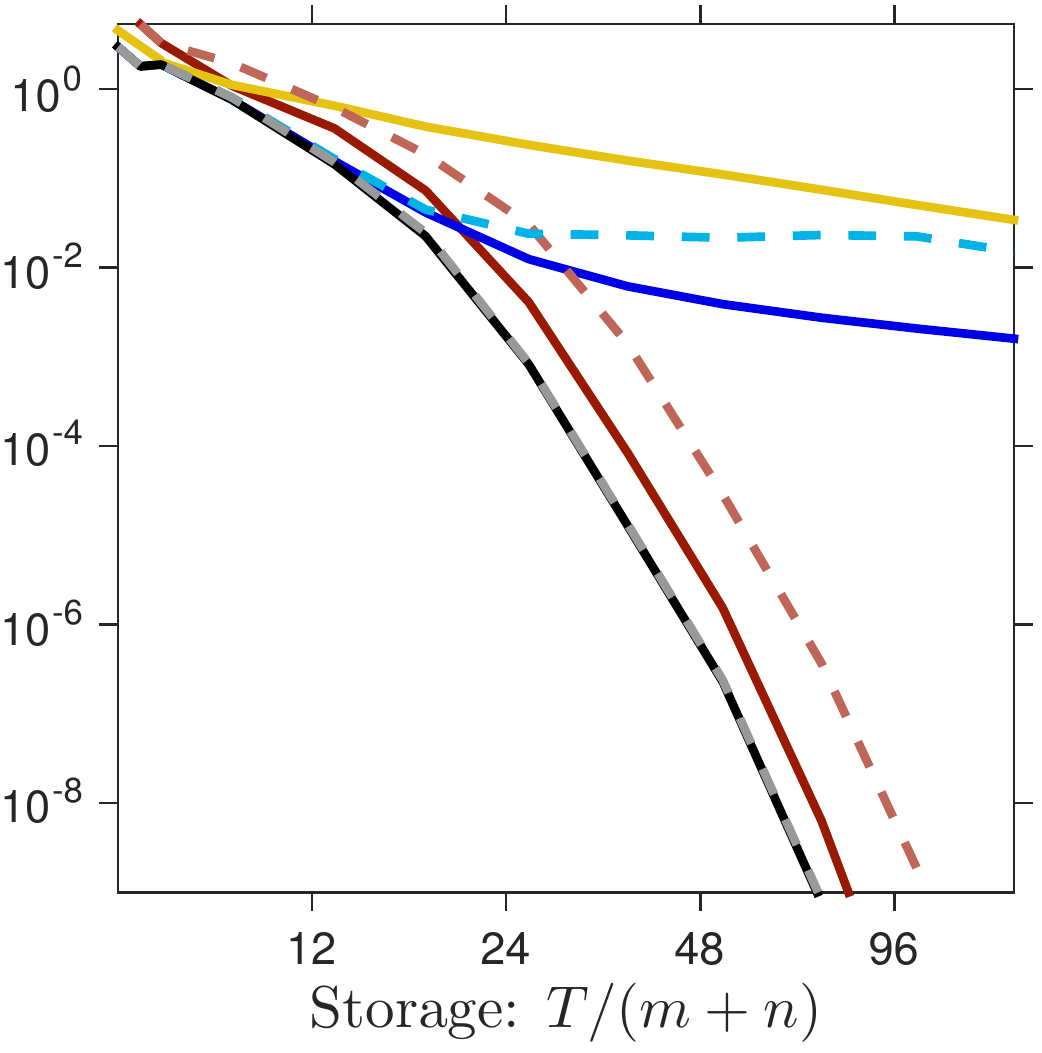}
\caption{\texttt{PhaseRetrieval} ($r = 5$)}
\end{center}
\end{subfigure}
\end{center}

\vspace{0.5em}

\caption{\textbf{Comparison of reconstruction formulas: Real data examples.}
(Sparse maps, Schatten $2$-norm.)
We compare the relative error achieved by the proposed fixed-rank
approximation~\cref{eqn:Ahat-fixed} against methods [HMT11], [Upa16], and [TYUC17]
from the literature.
\textbf{Solid lines} are oracle errors; \textbf{dashed lines} are errors with ``natural''
parameter choices.  (There is no dashed line for [HMT11].)  See \cref{sec:real-data} for details.}
\label{fig:data-comparison-S2}
\end{figure}

\begin{figure}[htp!]

\begin{center}
\begin{subfigure}{.48\textwidth}
\begin{center}
\includegraphics[width=2.4in]{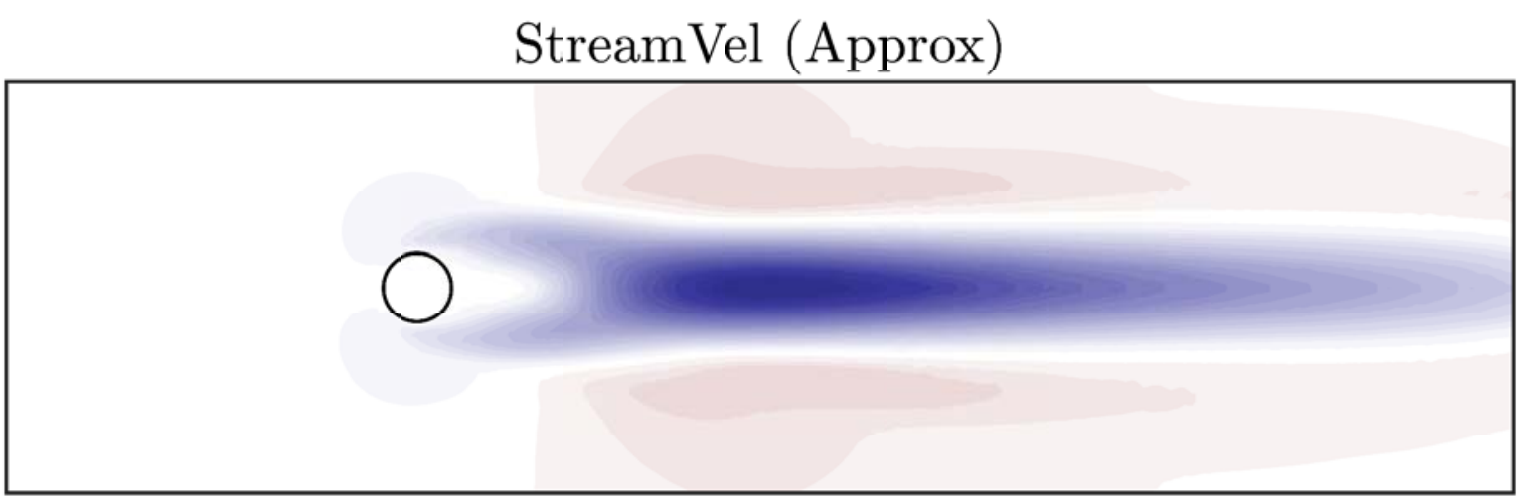}
\end{center}
\end{subfigure}
\begin{subfigure}{.48\textwidth}
\begin{center}
\includegraphics[width=2.4in]{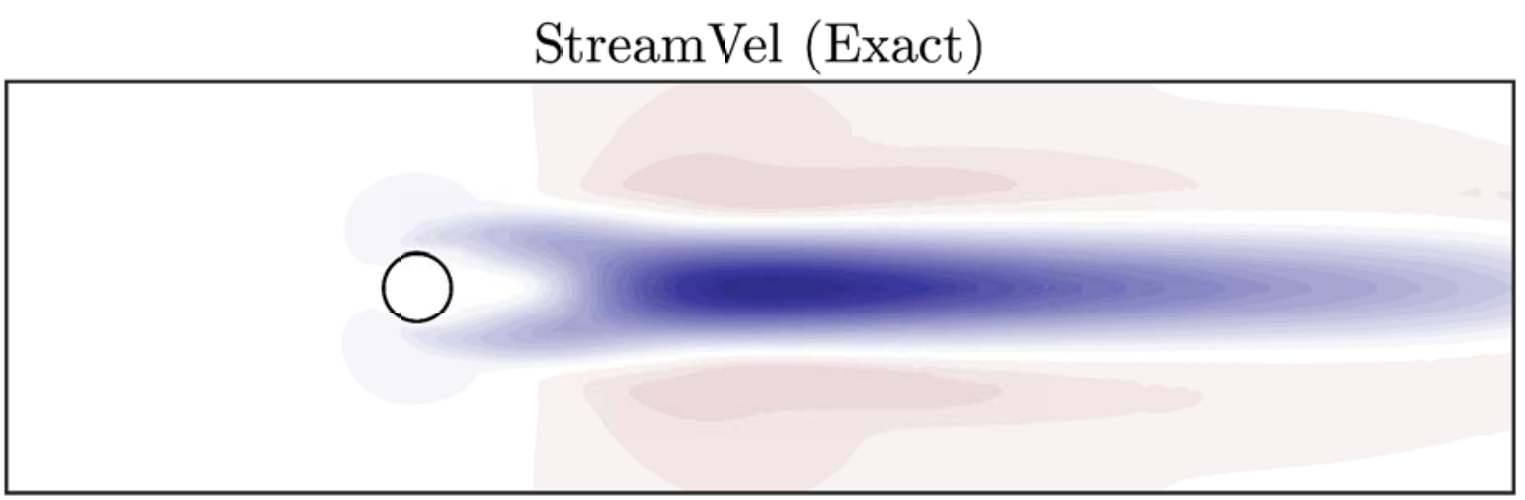}
\end{center}
\end{subfigure}

\vspace{1mm}

\begin{subfigure}{.48\textwidth}
\begin{center}
\includegraphics[width=2.4in]{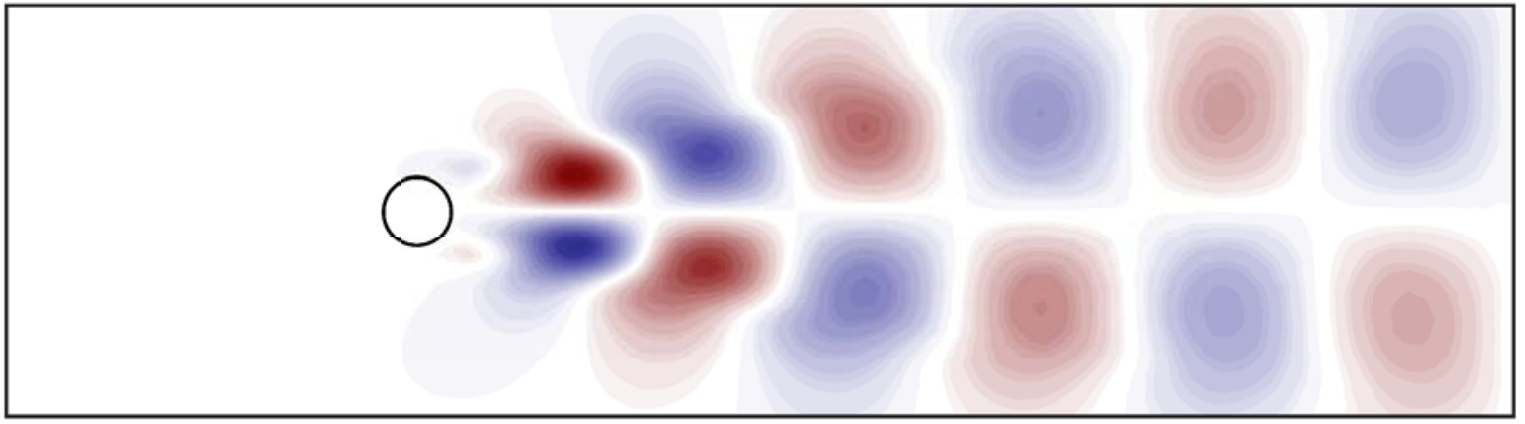}
\end{center}
\end{subfigure}
\begin{subfigure}{.48\textwidth}
\begin{center}
\includegraphics[width=2.4in]{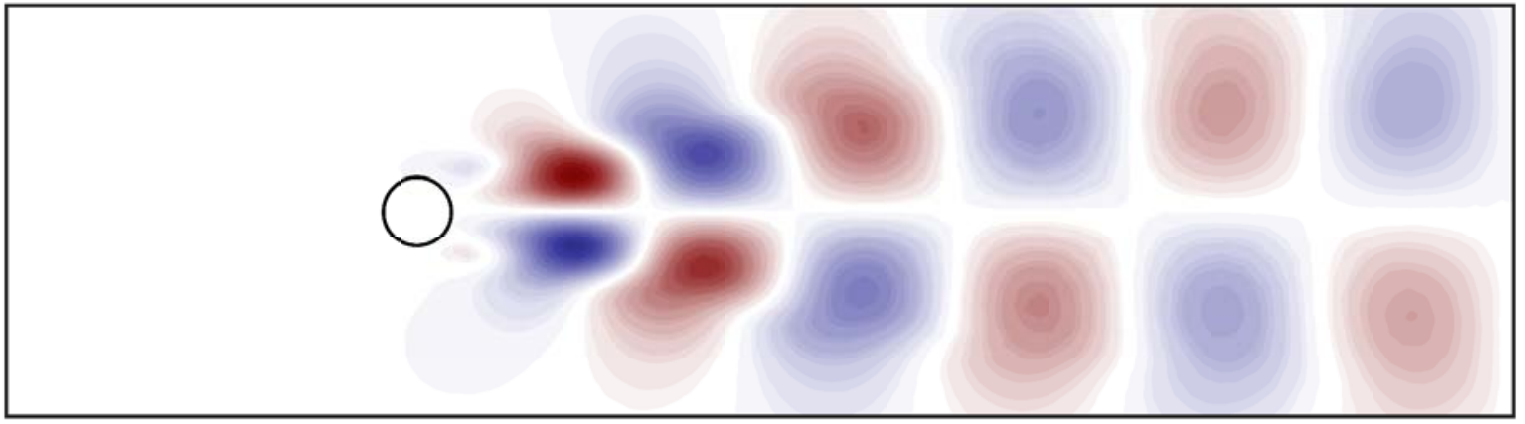}
\end{center}
\end{subfigure}

\vspace{1mm}

\begin{subfigure}{.48\textwidth}
\begin{center}
\includegraphics[width=2.4in]{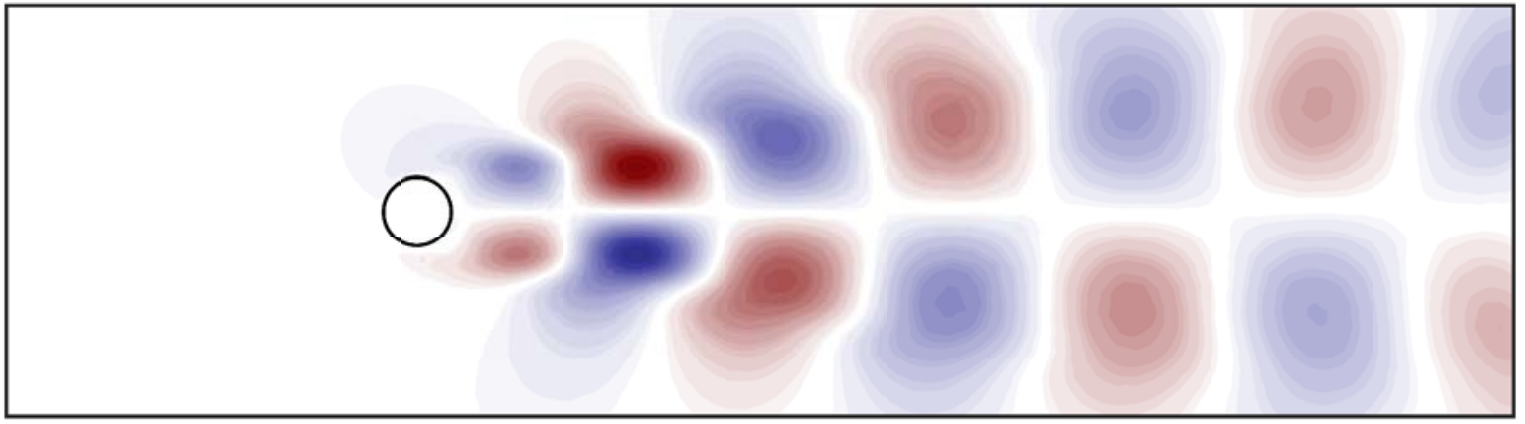}
\end{center}
\end{subfigure}
\begin{subfigure}{.48\textwidth}
\begin{center}
\includegraphics[width=2.4in]{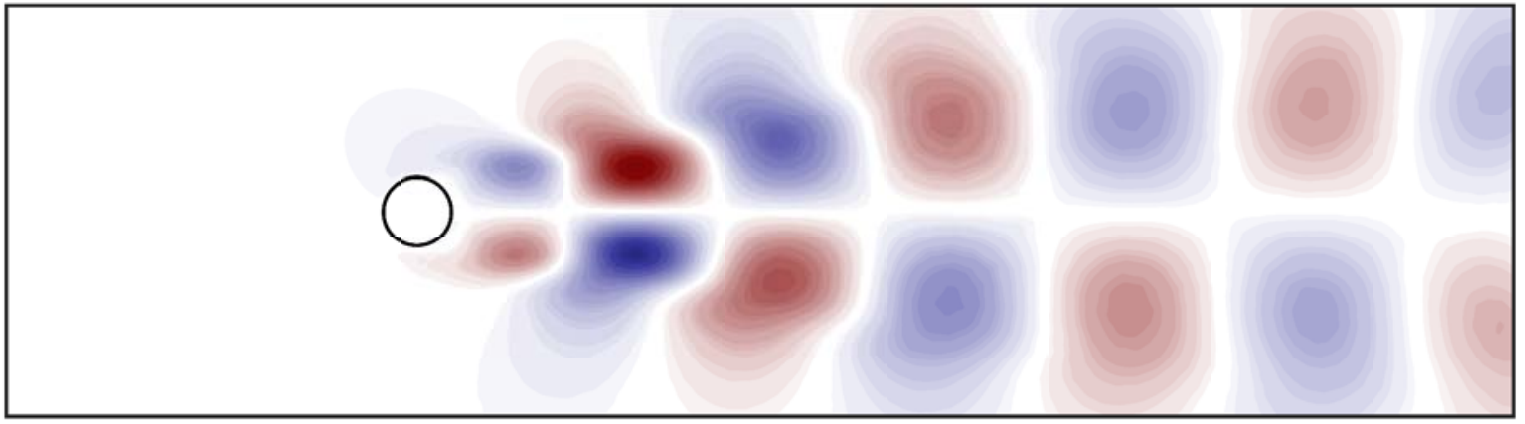}
\end{center}
\end{subfigure}

\vspace{1mm}

\begin{subfigure}{.48\textwidth}
\begin{center}
\includegraphics[width=2.4in]{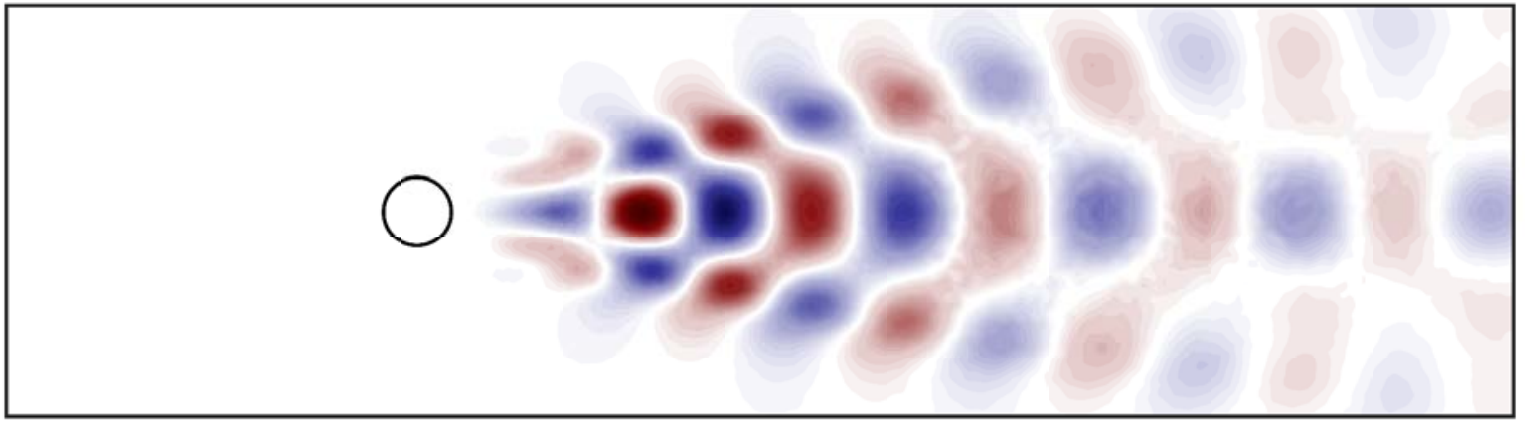}
\end{center}
\end{subfigure}
\begin{subfigure}{.48\textwidth}
\begin{center}
\includegraphics[width=2.4in]{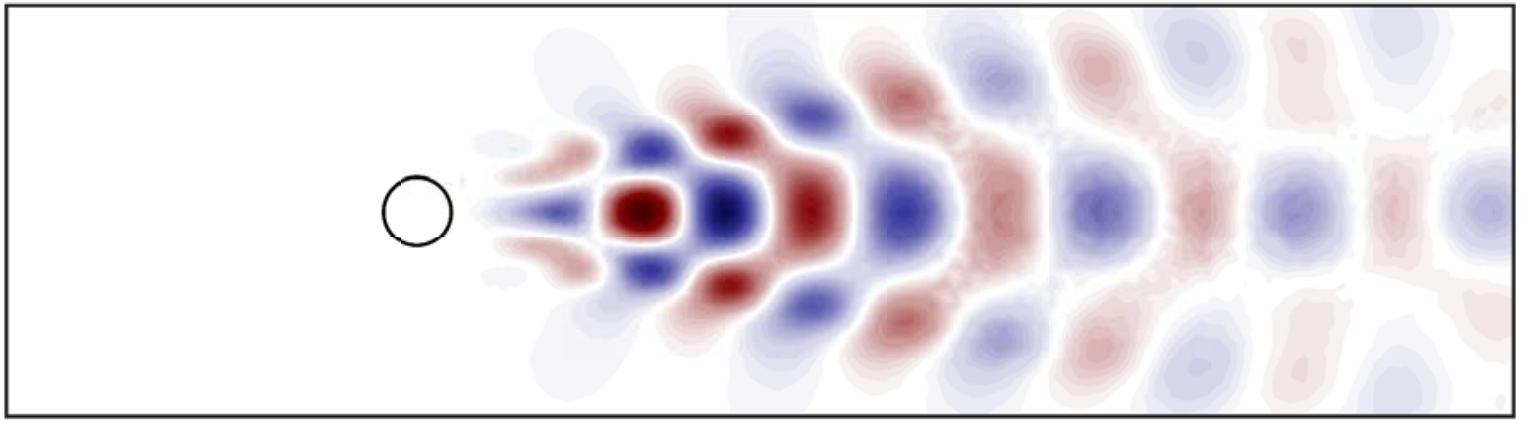}
\end{center}
\end{subfigure}

\vspace{1mm}

\begin{subfigure}{.48\textwidth}
\begin{center}
\includegraphics[width=2.4in]{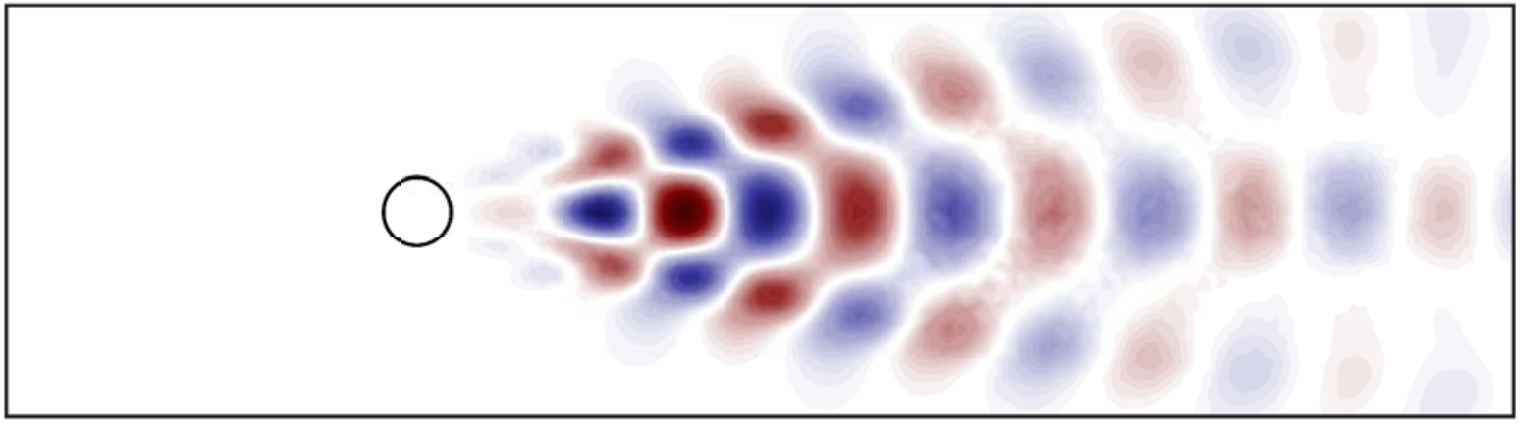}
\end{center}
\end{subfigure}
\begin{subfigure}{.48\textwidth}
\begin{center}
\includegraphics[width=2.4in]{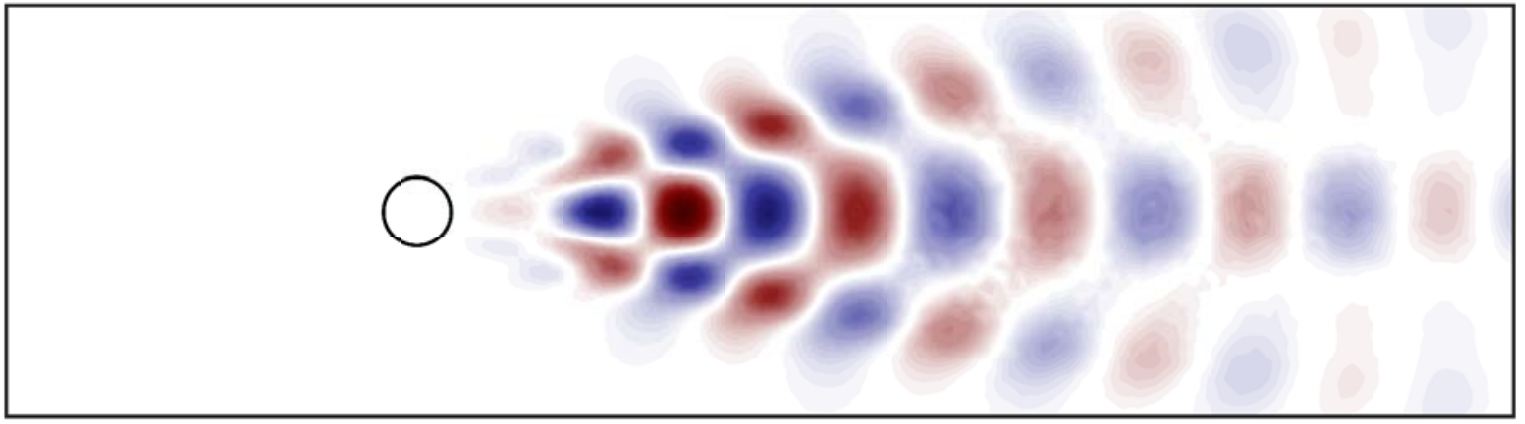}
\end{center}
\end{subfigure}

\vspace{1mm}

\begin{subfigure}{.48\textwidth}
\begin{center}
\includegraphics[width=2.4in]{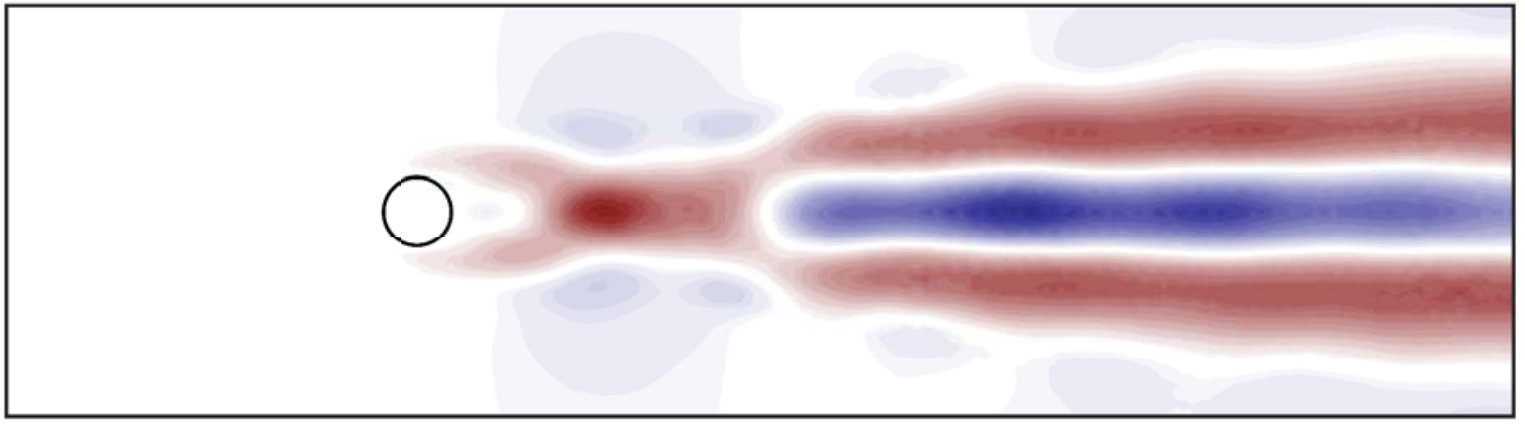}
\end{center}
\end{subfigure}
\begin{subfigure}{.48\textwidth}
\begin{center}
\includegraphics[width=2.4in]{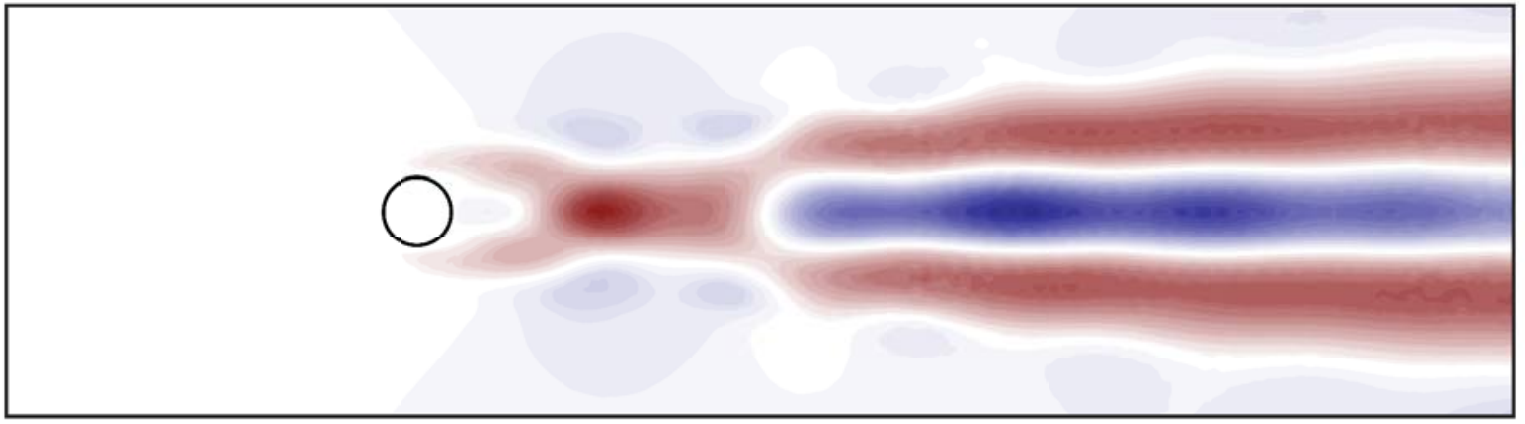}
\end{center}
\end{subfigure}

\vspace{1mm}

\begin{subfigure}{.48\textwidth}
\begin{center}
\includegraphics[width=2.4in]{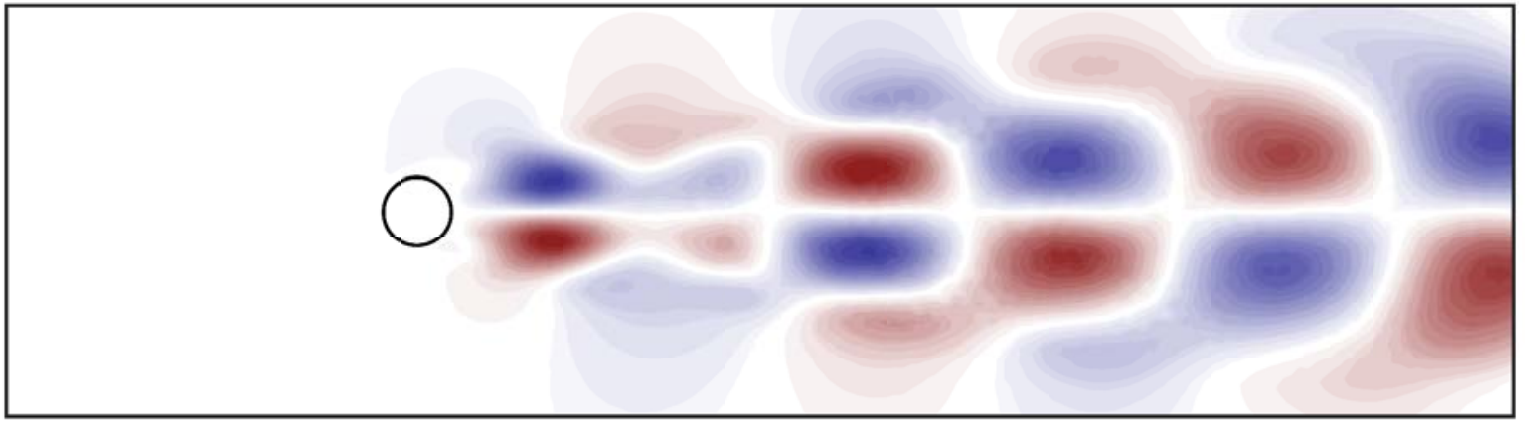}
\end{center}
\end{subfigure}
\begin{subfigure}{.48\textwidth}
\begin{center}
\includegraphics[width=2.4in]{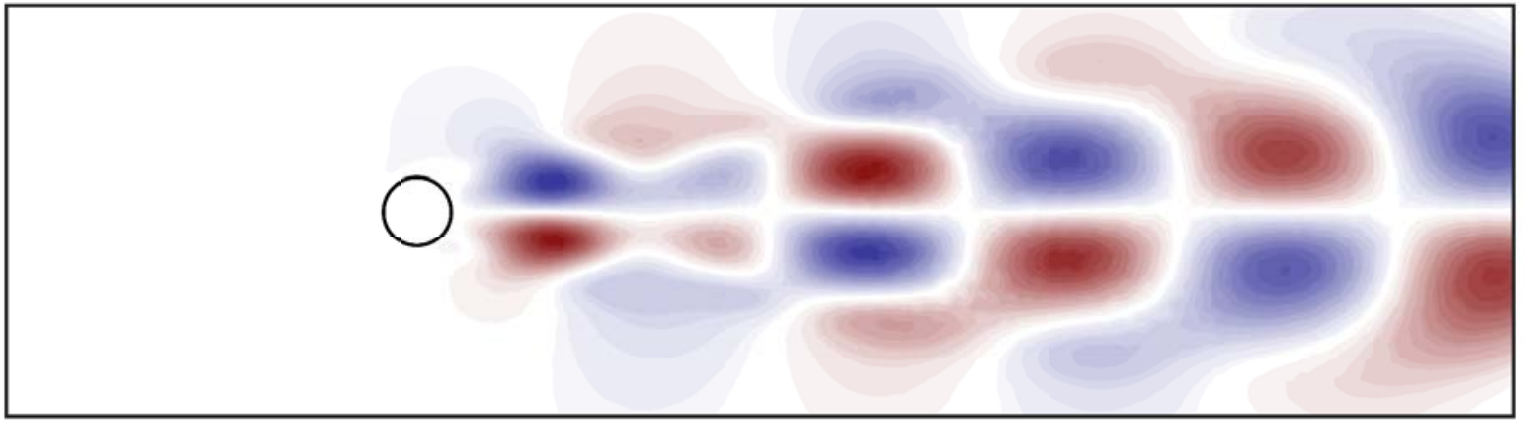}
\end{center}
\end{subfigure}

\vspace{1mm}

\begin{subfigure}{.48\textwidth}
\begin{center}
\includegraphics[width=2.4in]{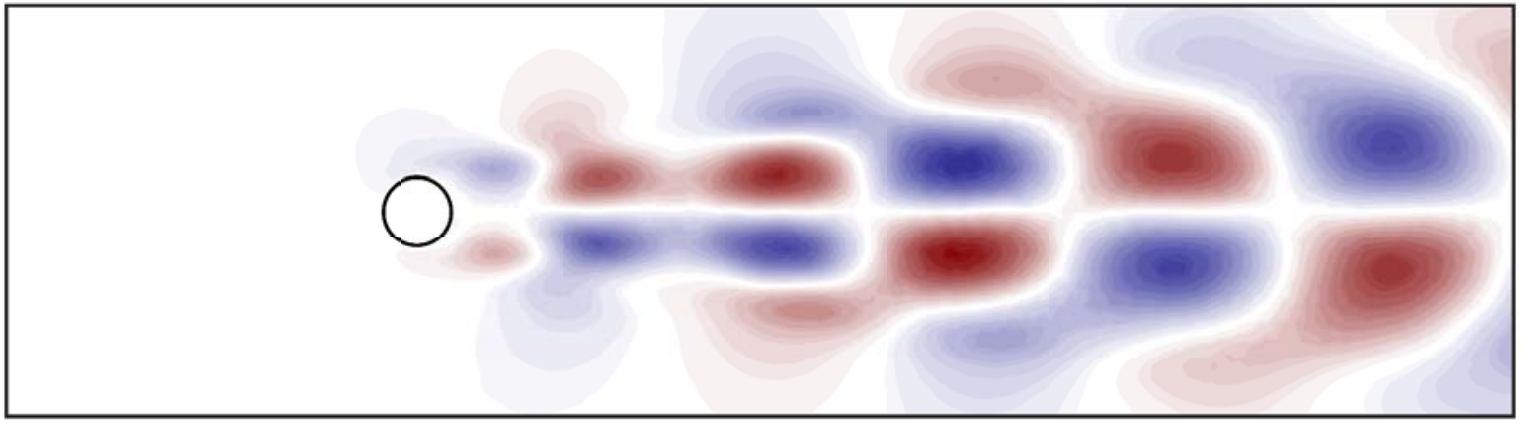}
\end{center}
\end{subfigure}
\begin{subfigure}{.48\textwidth}
\begin{center}
\includegraphics[width=2.4in]{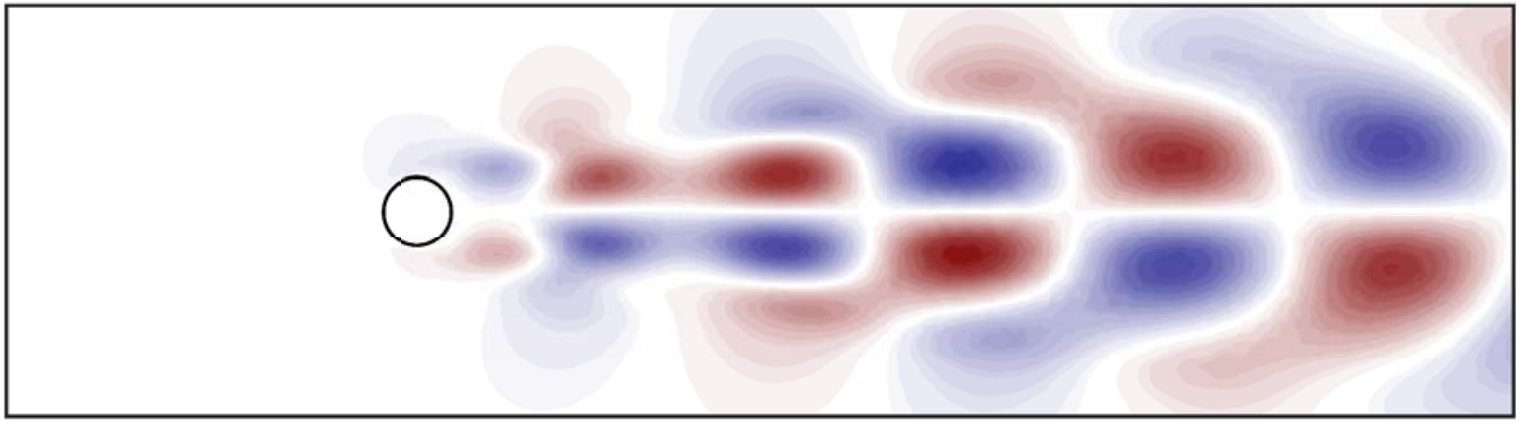}
\end{center}
\end{subfigure}

\vspace{1mm}

\begin{subfigure}{.48\textwidth}
\begin{center}
\includegraphics[width=2.4in]{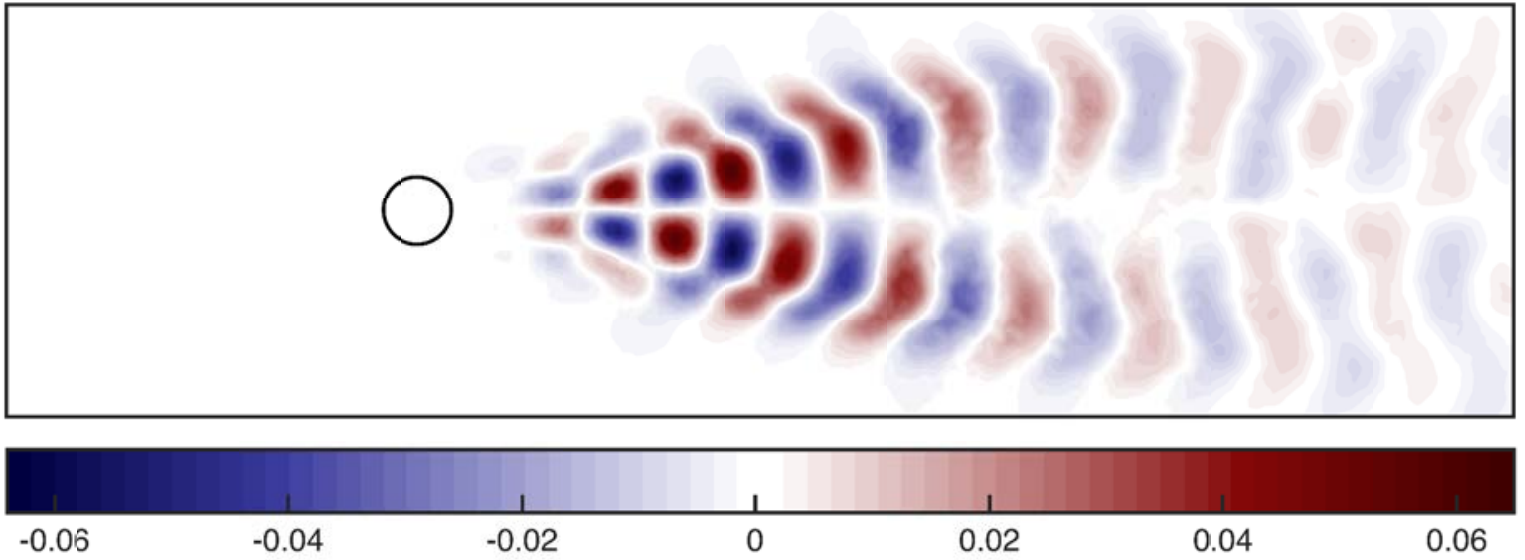}
\end{center}
\end{subfigure}
\begin{subfigure}{.48\textwidth}
\begin{center}
\includegraphics[width=2.4in]{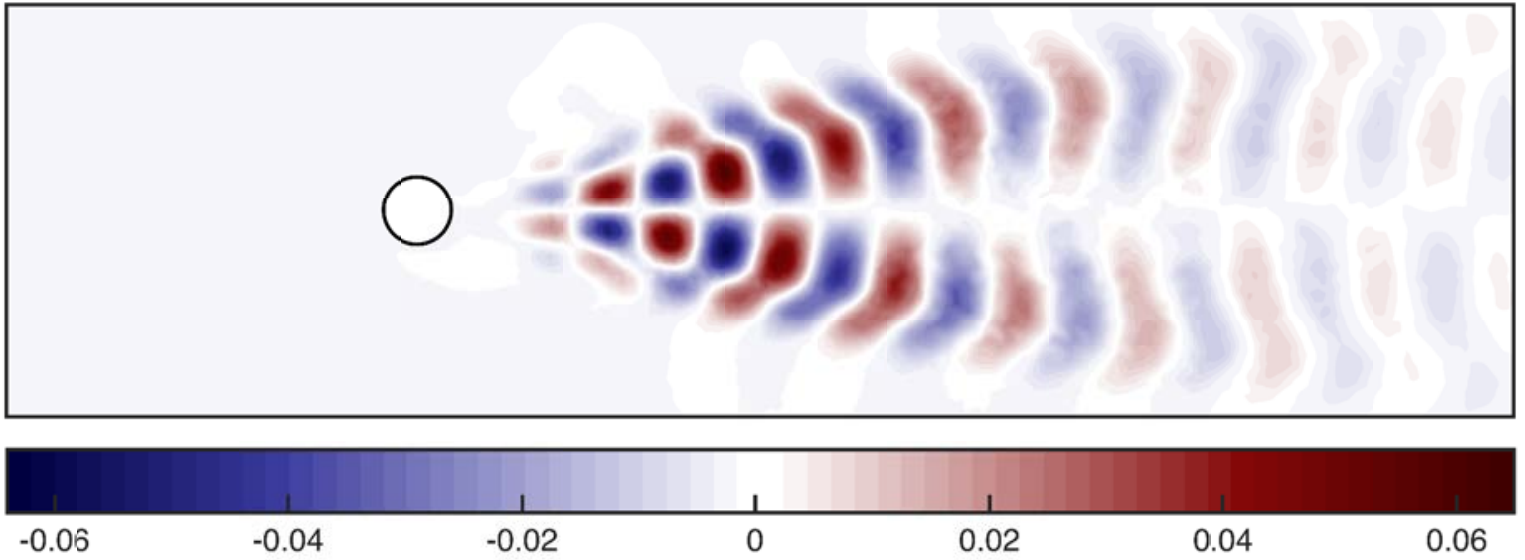}
\end{center}
\end{subfigure}

\vspace{1pc}

\caption{\textbf{Left singular vectors of \texttt{StreamVel}.}
(Sparse maps, approximation rank $r = 10$, storage budget $T = 48 (m+n)$.)
The columns of the matrix \texttt{StreamVel} describe the fluctuations
of the streamwise velocity field about its mean value as a function of time.
From top to bottom, the panels show the first nine computed left singular
vectors of the matrix.
The \textbf{left-hand side} is computed from the sketch,
while the \textbf{right-hand side} is computed from the exact flow field.
The heatmap indicates the magnitude of the fluctuation.
See \cref{sec:flow-field} for details.}
\label{fig:flow-field-svec}
\end{center}
\end{figure}

\subsection{Comparison of Reconstruction Formulas: Real Data Examples}
\label{sec:real-data}

The next set of experiments compares the behavior of the algorithms
for matrices drawn from applications.

\Cref{fig:data-comparison-S2} %
contains the results of the following
experiment.  For each of the four algorithms,
we display the relative error~\cref{eqn:relative-error} as a function of storage.
We use sparse dimension reduction maps, which is justified by
the experiments in \cref{sec:universality}.

We plot the oracle error (\cref{sec:oracle-error}) attained by each method.
Since the oracle error is not achievable in practice,
we also chart the performance of each method
at an \emph{a priori} parameter selection;
see~\cref{app:numerics-comparison-data} for details.

As with the synthetic examples, the proposed method~\cref{eqn:Ahat-fixed}
improves over the competing methods for all the examples we considered.  This
is true when we compare oracle errors or when we compare the errors using
theoretical parameter choices.  The benefits of the new method are least
pronounced for the matrix \texttt{MinTemp}, %
whose spectrum has medium polynomial decay.  The benefits of the new method
are quite clear for the matrix \texttt{StreamVel}, which %
has an exponentially decaying spectrum.  The advantages are
even more striking for the two matrices \texttt{MaxCut} and \texttt{PhaseRetrieval},
which are effectively rank deficient.

In summary, we believe that the numerical work here supports the use of
our new method~\cref{eqn:Ahat-fixed}.  The methods [HMT11] and [Upa16]
cannot achieve a small relative error~\cref{eqn:relative-error}, even with a large amount
of storage.  The method [TYUC17] can yield small relative
error, but it often requires more storage to achieve this
goal---especially at the \emph{a priori} parameter choices.

\subsection{Example: Flow-Field Reconstruction}
\label{sec:flow-field}

Next, we elaborate on using sketching to compress the Navier--Stokes data matrix
\texttt{StreamVel}.   We compute the best rank-$10$ approximation
of the matrix via~\cref{eqn:Ahat-fixed} using storage $T = 48 (m+n)$
and the ``natural'' parameter choices~\cref{eqn:ks-natural}.  For this example,
we can use plots of the flow field to make visual comparisons.

\cref{fig:flow-field-svec} illustrates the leading left singular vectors
of the streamwise velocity field \texttt{StreamVel}, as computed from the
sketch and the full matrix.  We see that the approximate left singular vectors
closely match the actual left singular vectors, although some small discrepancies
appear in the higher singular vectors.
See~\cref{app:flow-field-reconstruction} for additional numerics.
In particular, we find that the output from the algorithms
[HMT11] and [Upa16] changes violently when we adjust
the truncation rank $r$.

We see that our sketching method leads to an excellent rank-$10$ approximation of the matrix.
In fact, the relative error~\eqref{eqn:relative-error} in Frobenius norm is under $9.2 \cdot 10^{-3}$.
While the sketch uses $5.8 \, \mathrm{MB}$ of storage in double precision,
the full matrix requires $409.7 \, \mathrm{MB}$.  The compression rate is $70.6 \times$.
Therefore, it is possible to compress the output of the Navier--Stokes simulation automatically
using sketching.

\subsection{Rank Truncation and \emph{A Posteriori} Error Estimation}
\label{sec:apost-experiments}

\begin{figure}[t!]
\begin{center}
\begin{subfigure}{.45\textwidth}
\begin{center}
\includegraphics[height=1.75in]{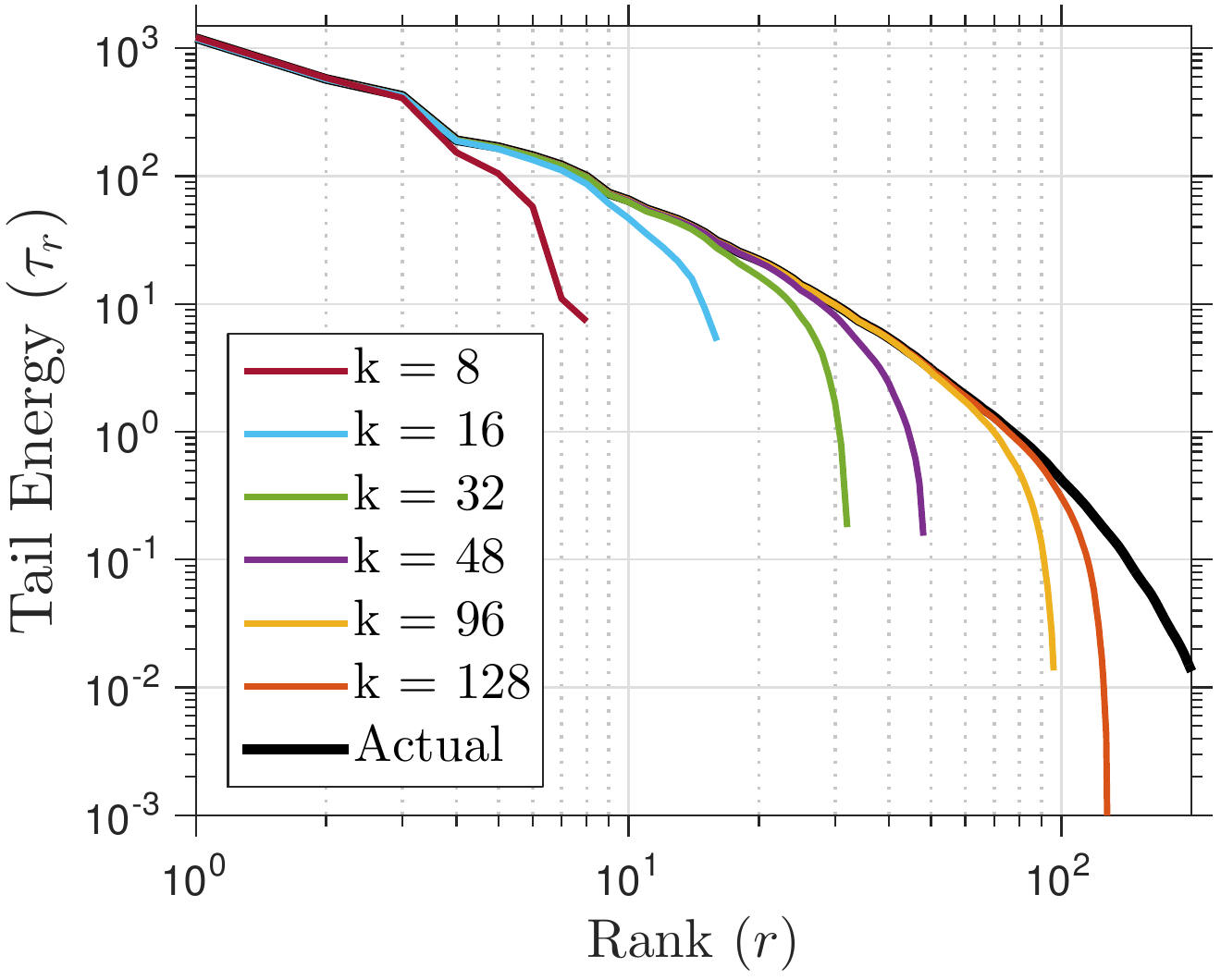}
\caption{Tail Energy $\tau_{r}(\hat{\mtx{A}})$} \label{subfig:tailenergy-Ahat}
\end{center}
\end{subfigure}
\begin{subfigure}{.45\textwidth}
\begin{center}
\includegraphics[height=1.75in]{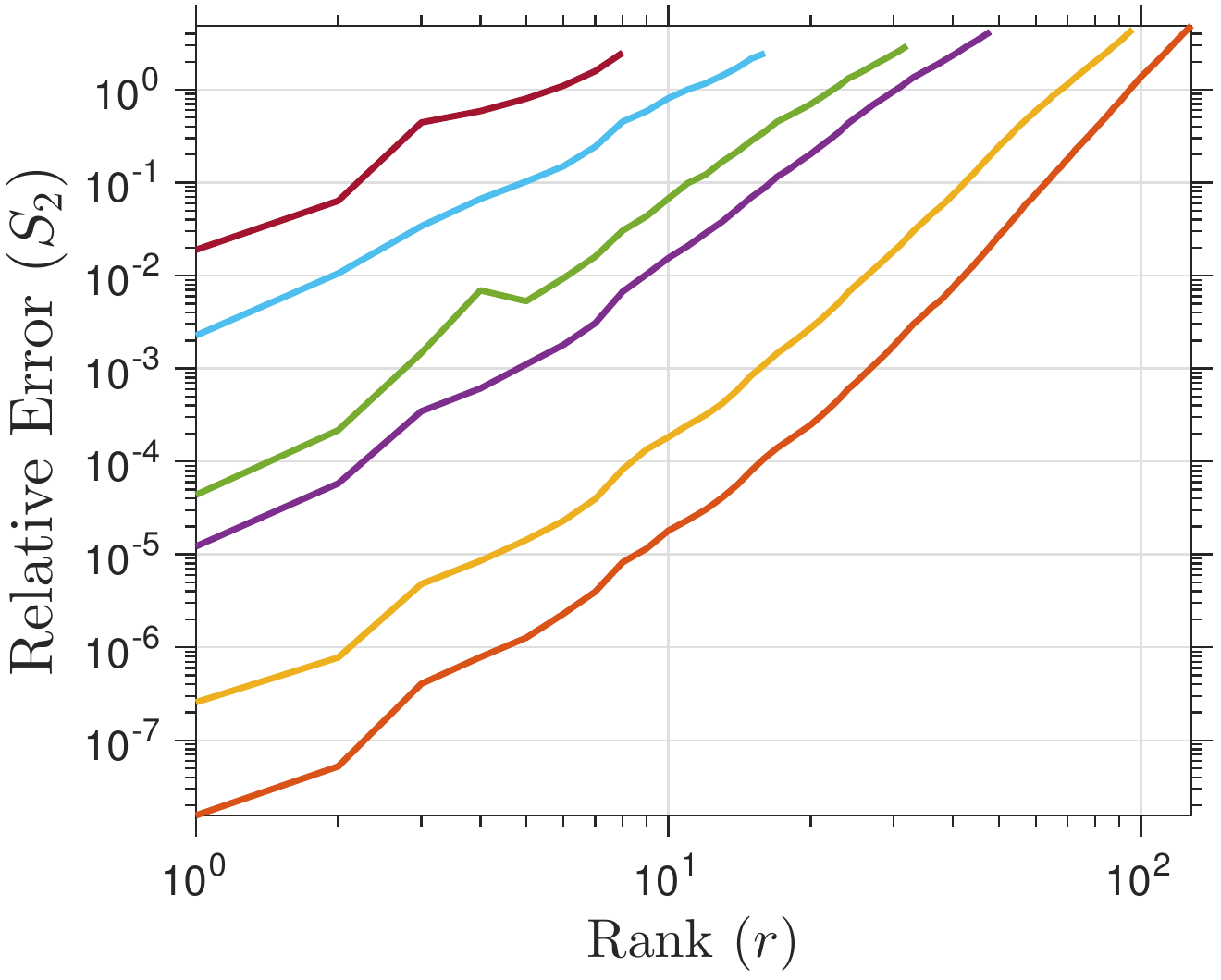}
\caption{Rank-$r$ Truncations} \label{subfig:error-Ahat-fixed}
\end{center}
\end{subfigure}

\caption{\textbf{Why Truncate?}  (\texttt{StreamVel}, sparse maps, $s = 2k+1$.)
\Cref{subfig:tailenergy-Ahat}
compares the tail energy $\tau_{r}(\hat{\mtx{A}})$ of the rank-$k$ approximation $\hat{\mtx{A}}$
against the tail energy $\tau_{r}(\mtx{A})$ of the actual matrix $\mtx{A}$.
\Cref{subfig:error-Ahat-fixed} shows the relative error~\cref{eqn:relative-error}
in the truncated approximation $\lowrank{\hat{\mtx{A}}}{r}$ as a function of rank $r$.
The relative error in the rank-$k$
approximation (\textbf{right endpoints} of series) \emph{increases} with $k$.
See~\cref{sec:apost-experiments}.}
\label{fig:why-truncate}
\end{center}
\end{figure}

\begin{figure}[t!]

\begin{center}
\begin{subfigure}{.45\textwidth}
\begin{center}
\includegraphics[height=1.75in]{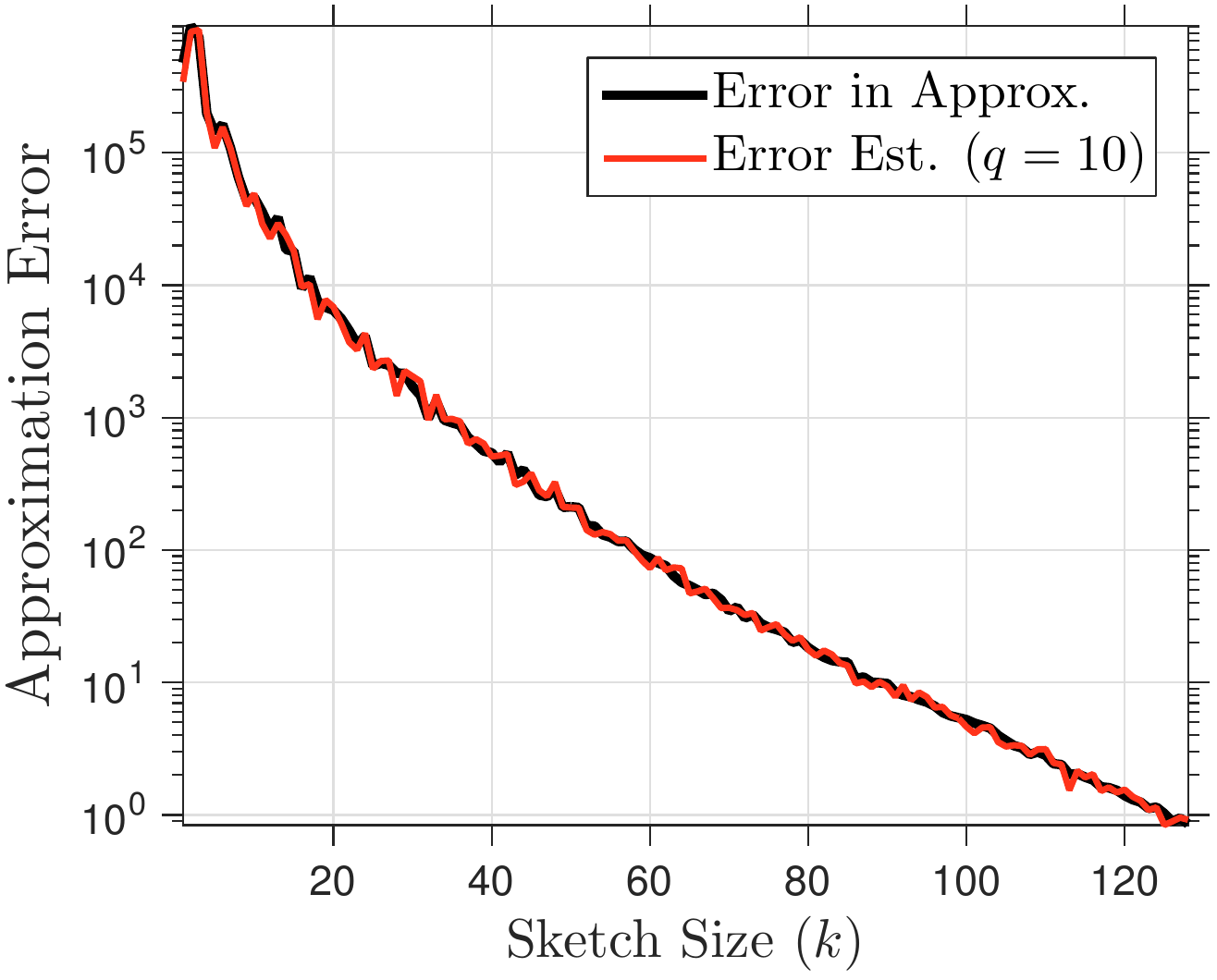}
\caption{Error Estimates for $\hat{\mtx{A}}$} \label{subfig:err2-Ahat}
\end{center}
\end{subfigure}
\begin{subfigure}{.45\textwidth}
\begin{center}
\includegraphics[height=1.75in]{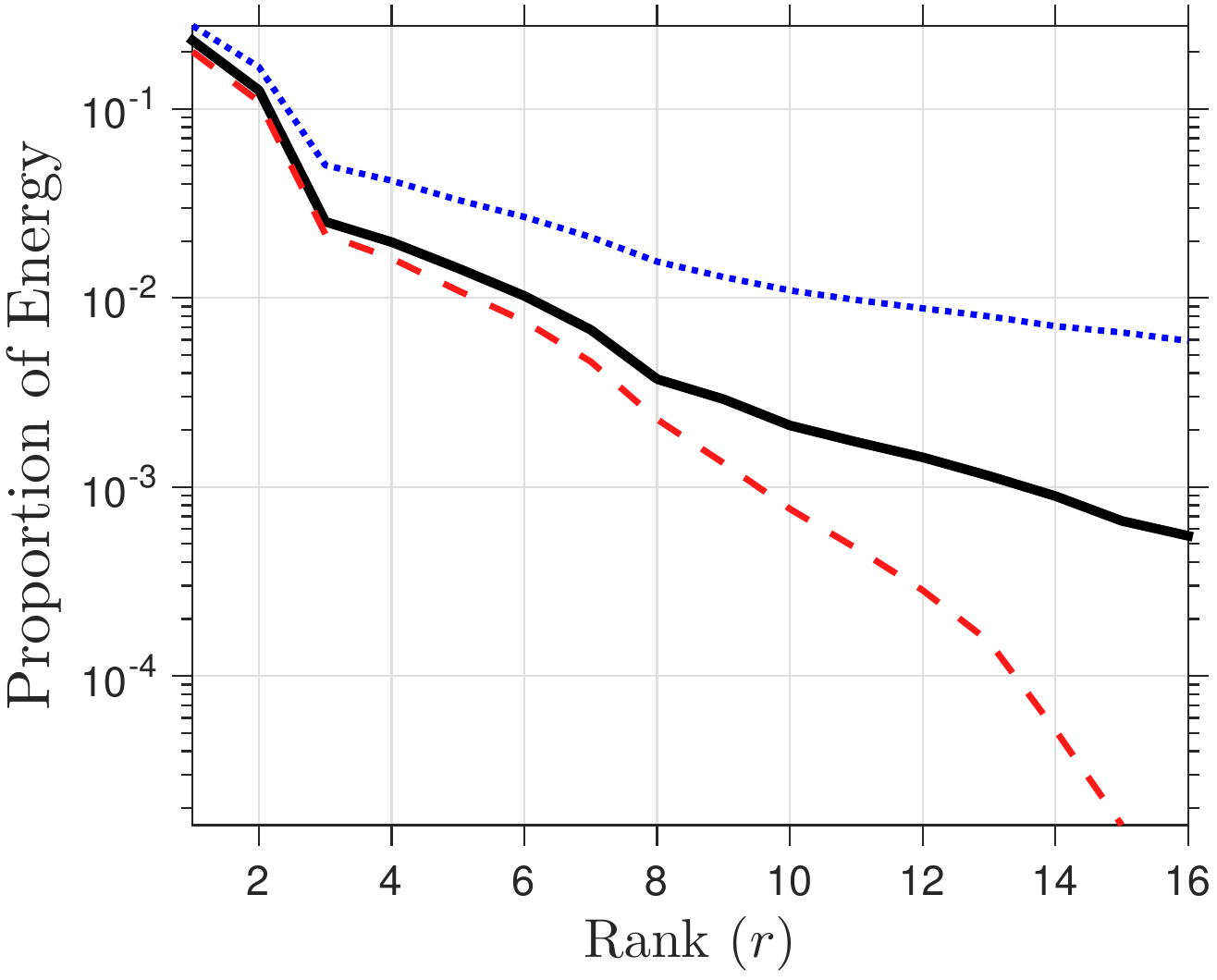}
\caption{Scree Plot ($k = 16$)}
\end{center}
\end{subfigure}
\end{center}

\vspace{0.5em}

\begin{center}
\begin{subfigure}{.45\textwidth}
\begin{center}
\includegraphics[height=1.75in]{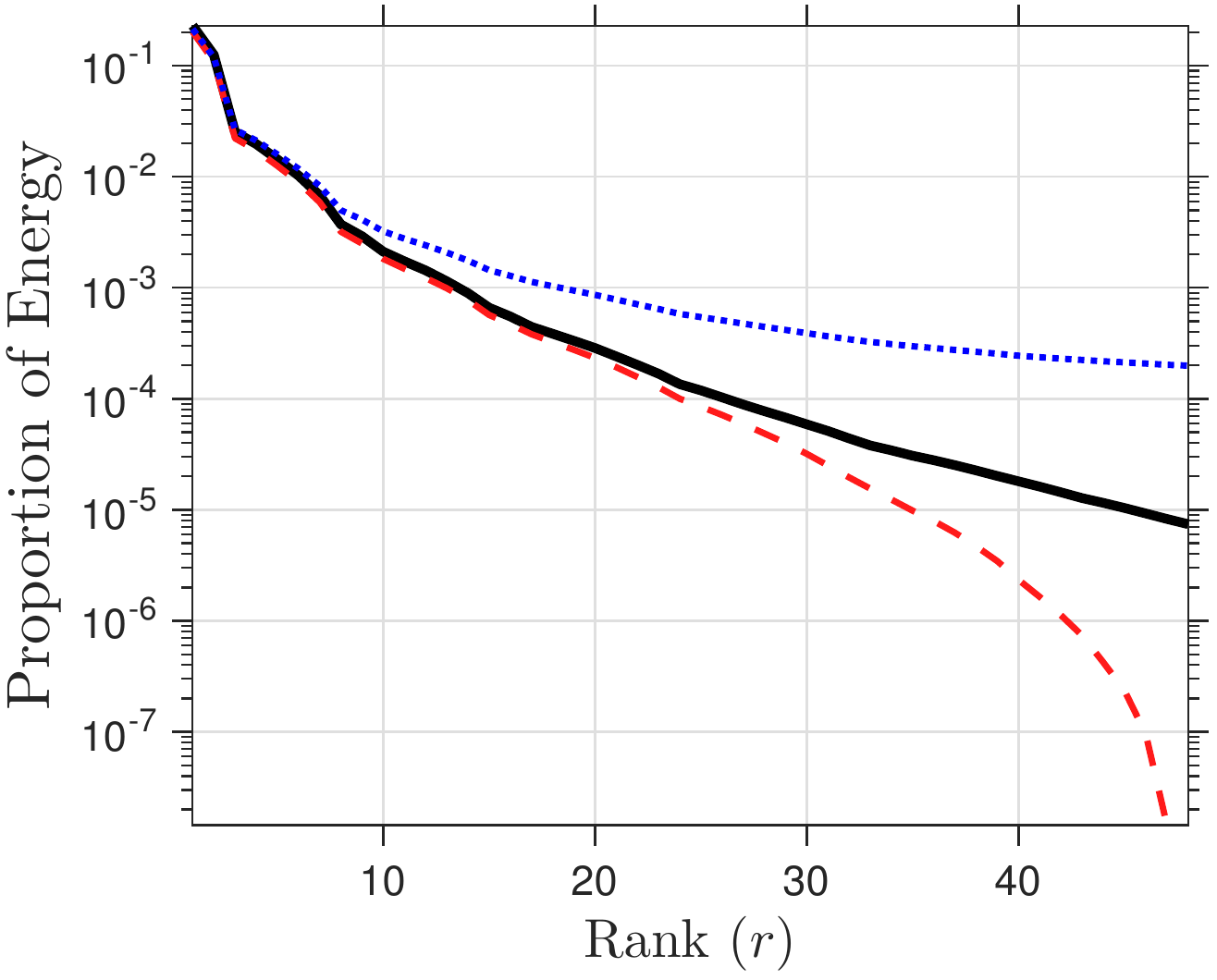}
\caption{Scree Plot ($k = 48$)}
\end{center}
\end{subfigure}
\begin{subfigure}{.45\textwidth}
\begin{center}
\includegraphics[height=1.75in]{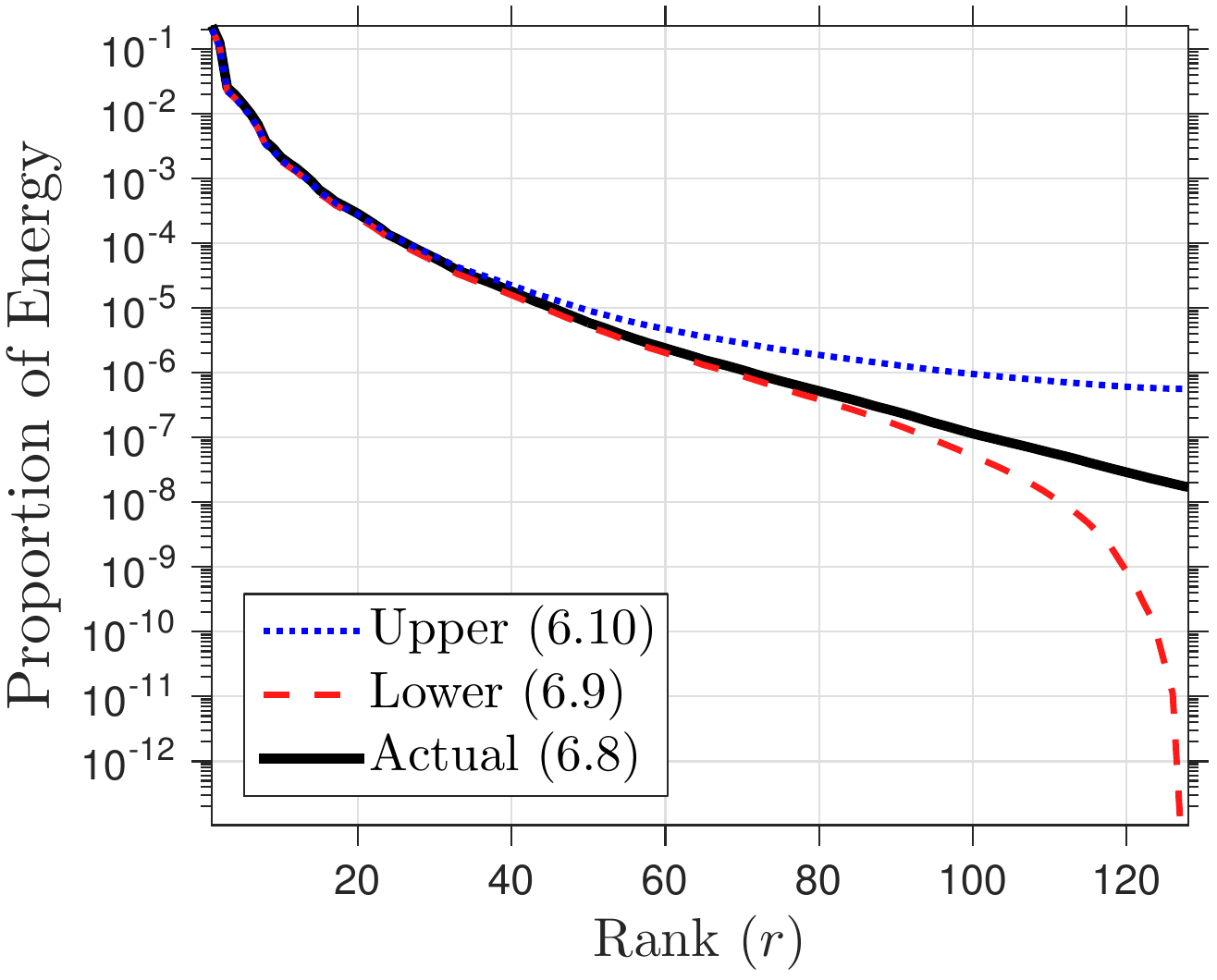}
\caption{Scree Plot ($k = 128$)}
\end{center}
\end{subfigure}
\end{center}

\vspace{0.5em}

\caption{\textbf{Error Estimation and Scree Plots.}
(\texttt{StreamVel}, sparse maps, $s = 2k+1$.)
For an error sketch with size $q = 10$, %
\cref{subfig:err2-Ahat} compares the absolute error
$\fnormsq{ \mtx{A} - \hat{\mtx{A}} }$
in the rank-$k$ approximation versus
the estimate $\err_2^2(\hat{\mtx{A}})$.
The other panels are scree plots of
the actual proportion of energy remaining~\cref{eqn:scree} %
versus a computable lower estimate~\cref{eqn:scree-lower} %
and upper estimate~\cref{eqn:scree-upper}. %
See \cref{sec:apost-experiments}.}
\label{fig:scree-plots}
\end{figure}

\begin{figure}[htp!]

\begin{center}
\begin{subfigure}{.45\textwidth}
\begin{center}
\includegraphics[height=2in]{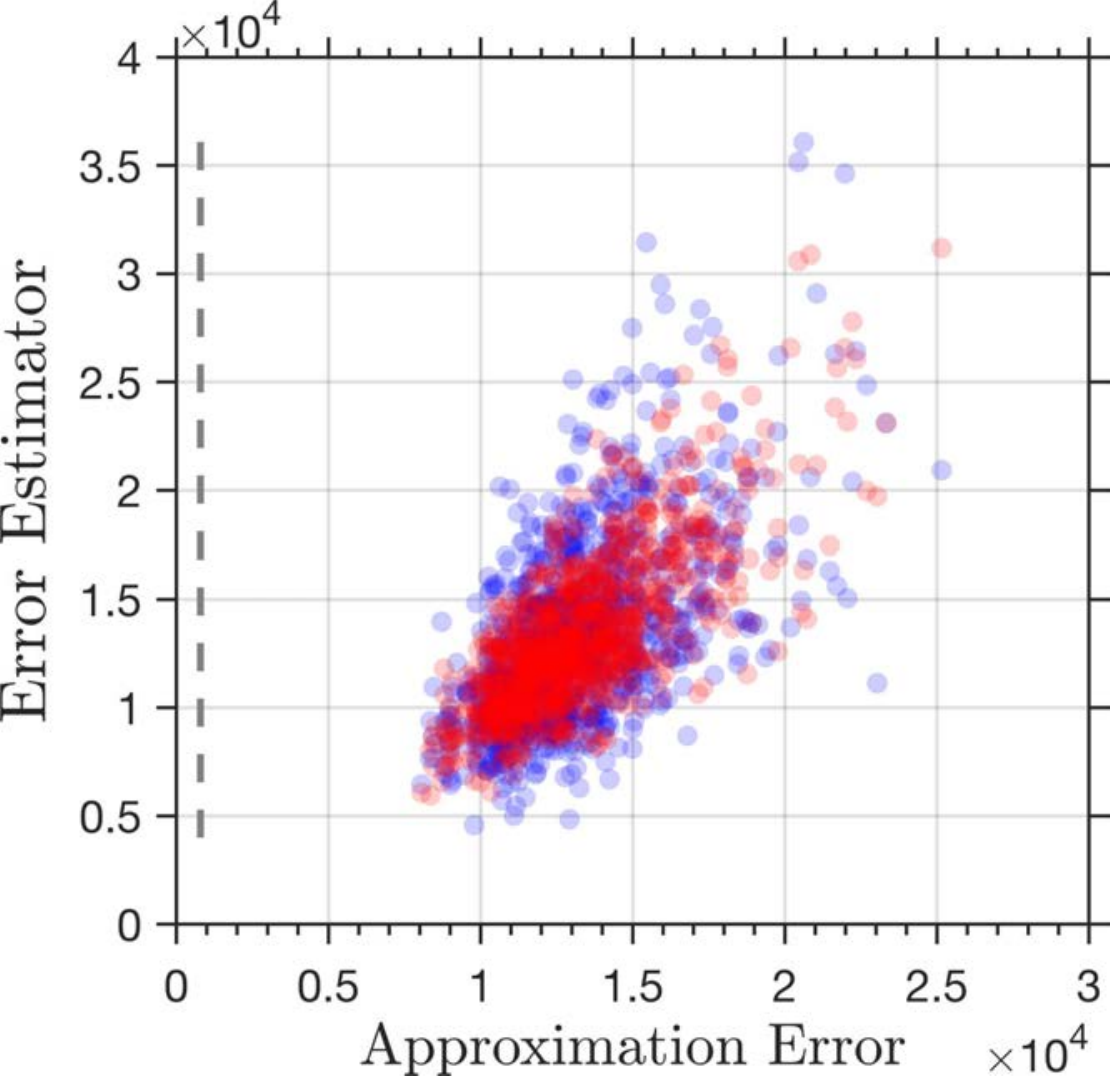}
\caption{Rank-$k$ Approximation ($k = 16$)}
\end{center}
\end{subfigure}
\begin{subfigure}{.45\textwidth}
\begin{center}
\includegraphics[height=2in]{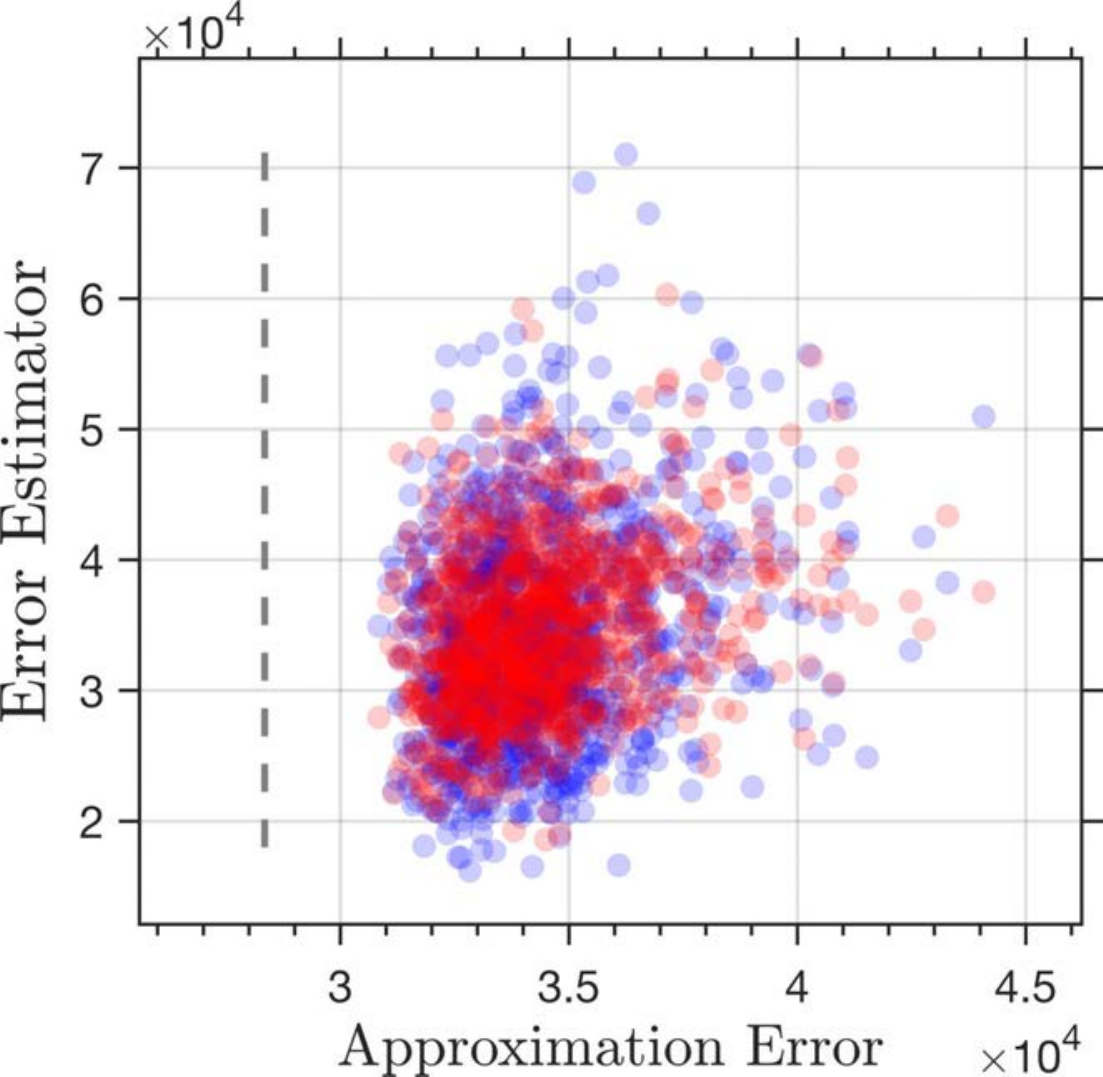}
\caption{Rank-$r$ Truncation ($k = 16$, $r = 4$)}
\end{center}
\end{subfigure}
\end{center}

\vspace{0.5em}

\begin{center}
\begin{subfigure}{.45\textwidth}
\begin{center}
\includegraphics[height=2in]{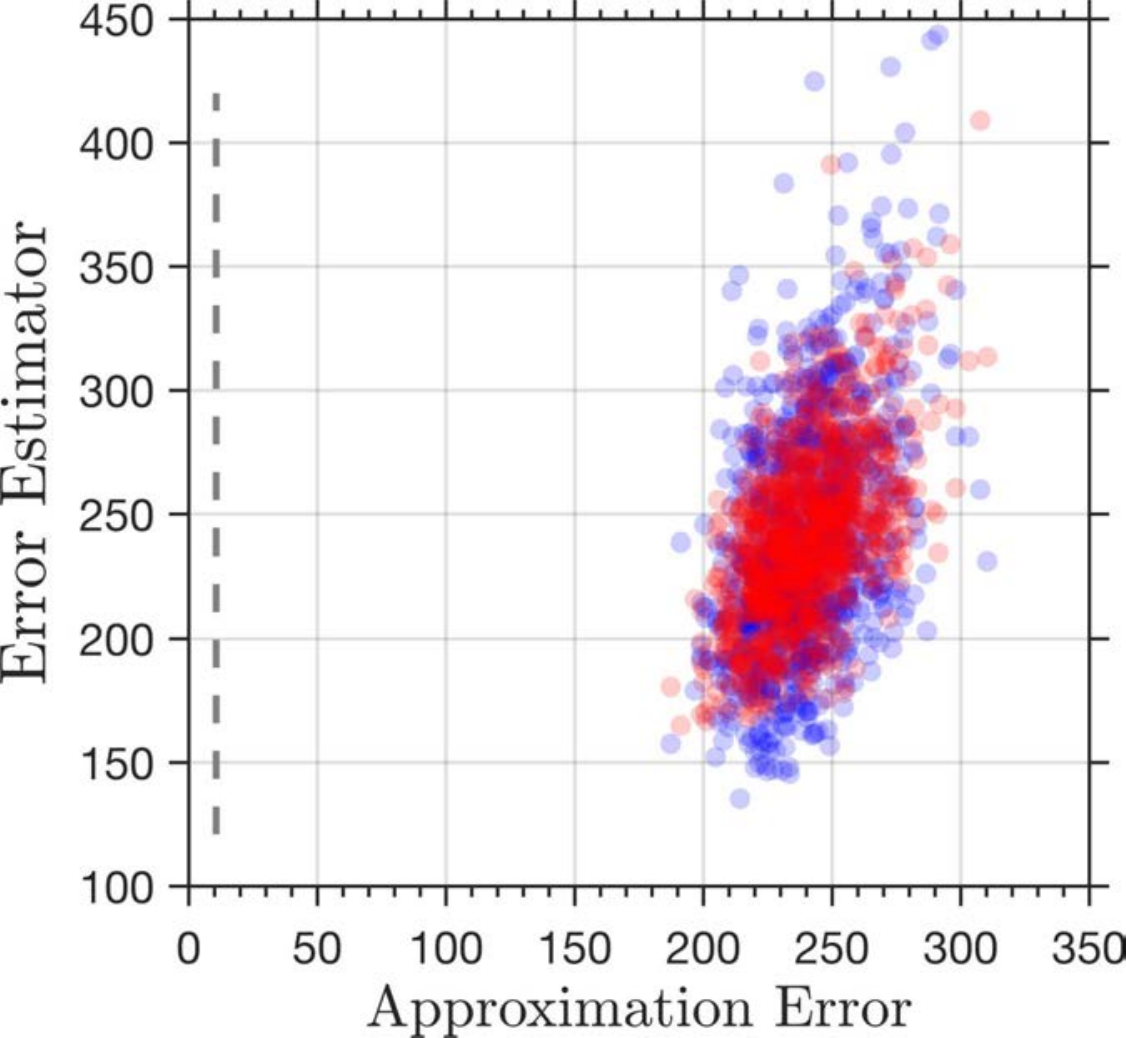}
\caption{Rank-$k$ Approximation ($k = 48$)}
\end{center}
\end{subfigure}
\begin{subfigure}{.45\textwidth}
\begin{center}
\includegraphics[height=2in]{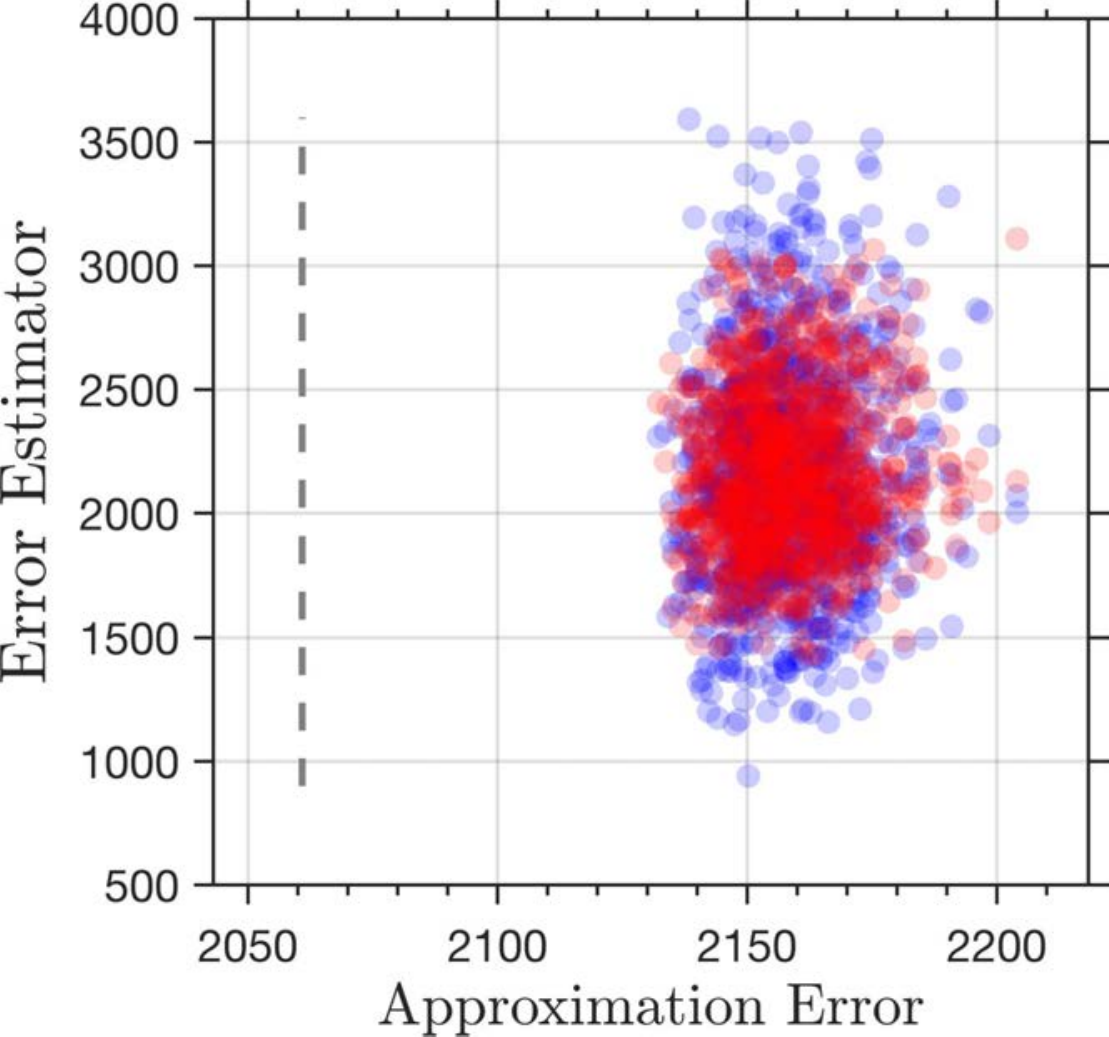}
\caption{Rank-$r$ Truncation ($k = 48$, $r = 12$)}
\end{center}
\end{subfigure}
\end{center}

\vspace{.5em}

\begin{center}
\begin{subfigure}{.45\textwidth}
\begin{center}
\includegraphics[height=2in]{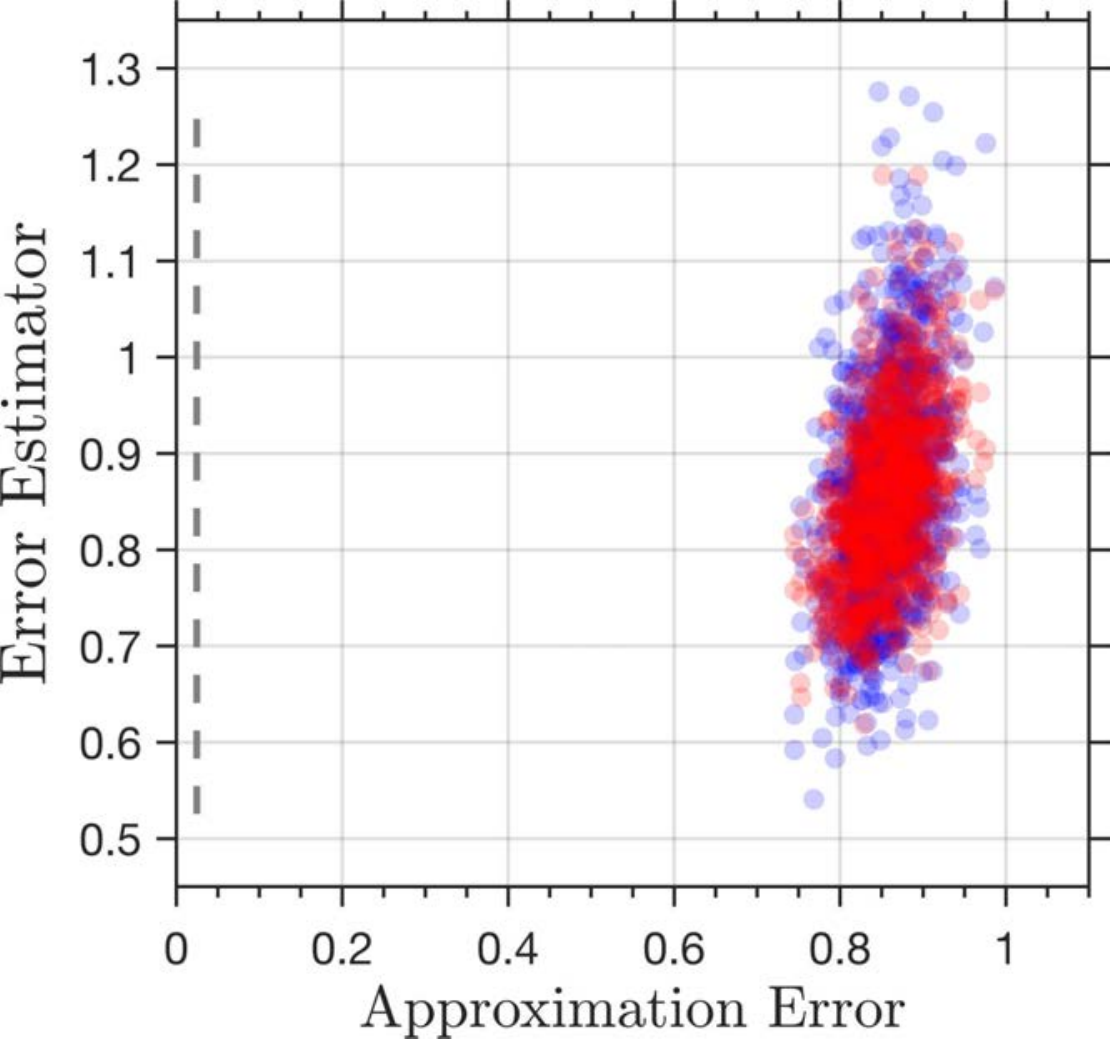}
\caption{Rank-$k$ Approximation ($k = 128$)}
\end{center}
\end{subfigure}
\begin{subfigure}{.45\textwidth}
\begin{center}
\includegraphics[height=2in]{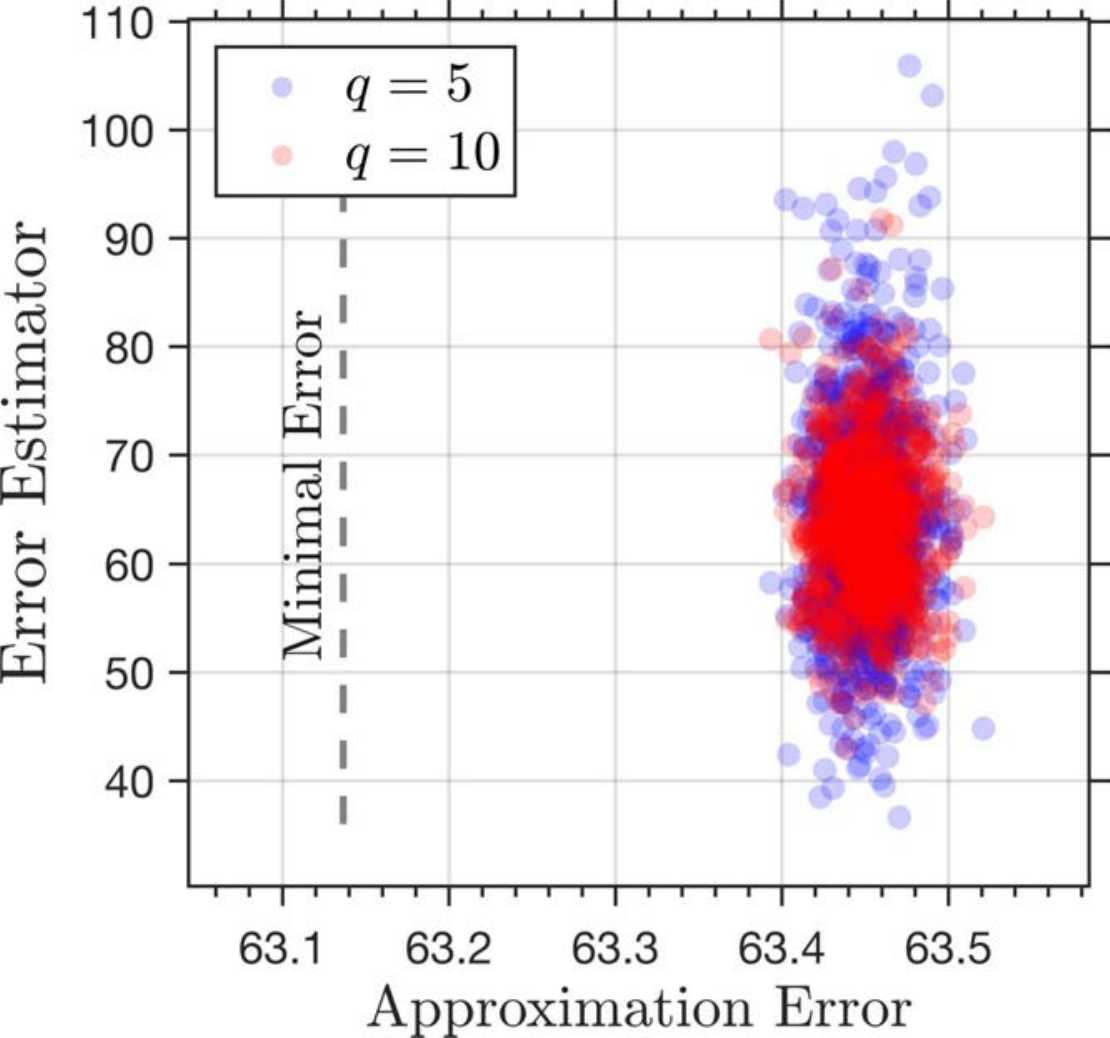}
\caption{Rank-$r$ Truncation ($k = 128$, $r = 32$)}
\end{center}
\end{subfigure}
\end{center}

\vspace{0.5em}

\caption{\textbf{Sampling Distributions of the Approximation Error and the Error Estimator.}
(\texttt{StreamVel}, sparse maps, $s = 2k+1$, Schatten $2$-norm.)
For error sketches with size $q \in \{5, 10\}$,
the \textbf{left-hand side} shows the sampling distribution of the error $\fnormsq{\mtx{A} - \hat{\mtx{A}}}$
in the rank-$k$ approximation versus the sampling distribution of the error
estimator $\err_2^2(\hat{\mtx{A}})$ for several values of $k$.  The \textbf{dashed line}
marks the error in the best rank-$k$ approximation of $\mtx{A}$.
The \textbf{right-hand side} contains similar plots with $\hat{\mtx{A}}$ replaced by
the rank-$r$ truncation $\lowrank{\hat{\mtx{A}}}{r}$.
See \cref{sec:apost-experiments}.}
\label{fig:sampling-distribution}
\end{figure}

This section uses the Navier--Stokes data
to explore the behavior of the error estimator \cref{sec:error-estimator}.
We also demonstrate that it is important to truncate the rank of
the approximation, and we show that the error estimator can assist us.

Let us undertake a single trial of the following experiment with the matrix \texttt{StreamVel}.
For each sketch size parameter $k \in \{ 1, 2, \dots, 128 \}$,
set the other sketch size parameter $s = 2k + 1$.
Extract an error sketch with size $q = 10$.
In each instance, we use the formula~\cref{eqn:Ahat}
to construct an initial rank-$k$ approximation $\hat{\mtx{A}}$ of the data matrix $\mtx{A}$
and the formula~\cref{eqn:Ahat-fixed} to construct a truncated rank-$r$ approximation
$\lowrank{\hat{\mtx{A}}}{r}$.
The plots will be indexed with the sketch size parameter $k$
or the rank truncation parameter $r$, rather than the storage budget.

\Cref{fig:why-truncate} illustrates the need to truncate the rank
of the approximation.  Observe that the tail energy $\tau_r(\hat{\mtx{A}})$
of the rank-$k$ approximation significantly underestimates the tail energy
$\tau_r(\mtx{A})$ of the matrix when $r \approx k$.
As a consequence, the error~\cref{eqn:relative-error}
in the rank-$k$ approximation $\hat{\mtx{A}}$, relative
to the best rank-$k$ approximation of $\mtx{A}$,
actually \emph{increases} with $k$.
In contrast, when $r \ll k$,
the rank-$r$ truncation $\lowrank{\hat{\mtx{A}}}{r}$
can attain very small error, relative to the best
rank-$r$ approximation of $\mtx{A}$.
In this instance, %
we achieve relative error below $10^{-2}$ across
a range of parameters $k$ by selecting $r \leq k/4$.
Therefore, we can be confident about the quality
of our estimates for the first $r$ singular vectors of $\mtx{A}$.

Next, let us study the behavior of the error estimator~\cref{eqn:error-estimate}. %
\Cref{fig:scree-plots} compares the actual approximation error
$\fnormsq{ \mtx{A} - \hat{\mtx{A}} }$ and the empirical error estimate
$\err_2^2( \hat{\mtx{A}} )$ as a function of the sketch size $k$.
The other panels are scree plots.  The baseline is the actual
scree function~\cref{eqn:scree} computed from the input matrix.
The remaining curves are the lower~\cref{eqn:scree-lower} and
upper~\cref{eqn:scree-upper} %
estimates for this curve.  %
We see that the scree estimators give good lower and upper bounds
for the energy missed, while tracking the shape of the baseline curve.
As a consequence, we can use these empirical estimates
to select the truncation rank.

\begin{figure}[t]
\begin{center}
\includegraphics[width=0.7\textwidth]{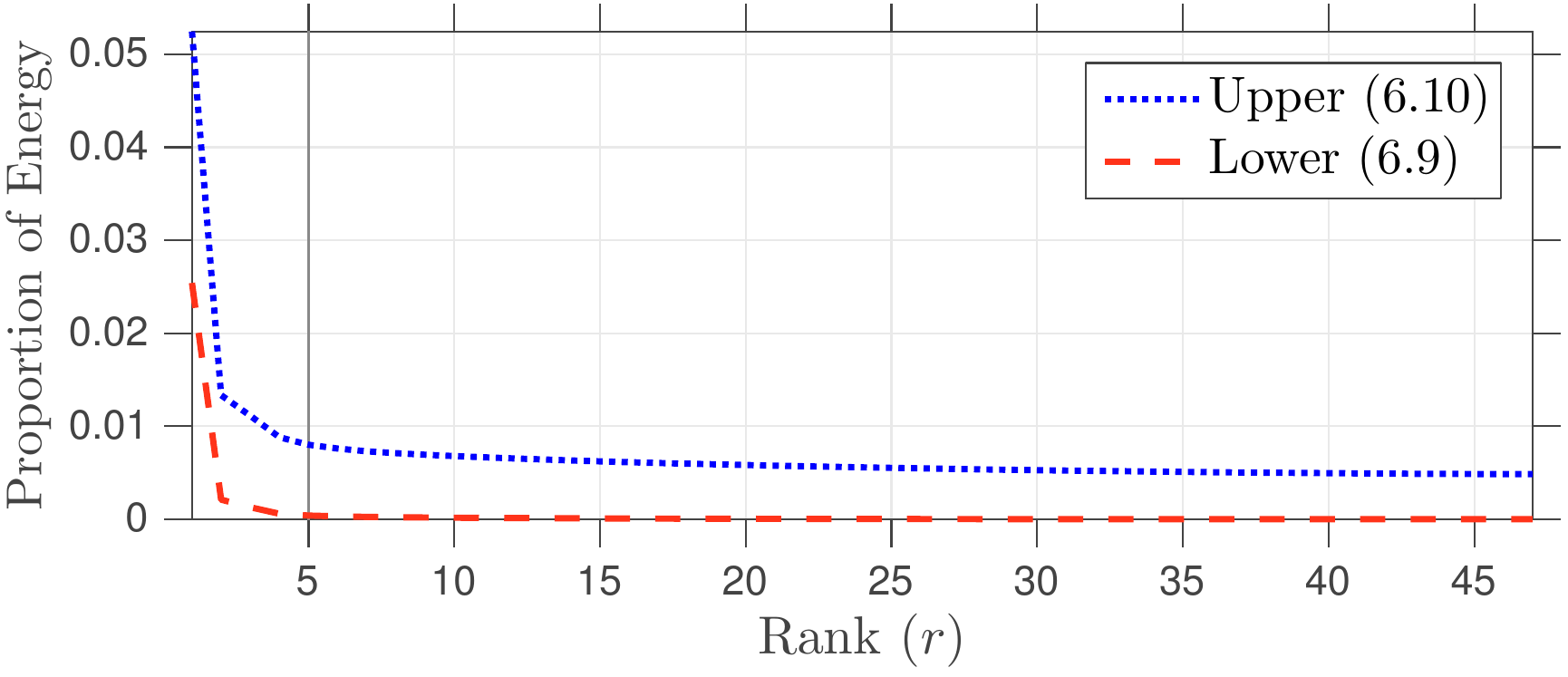}
\end{center}
\caption{\textbf{Empirical Scree Plot for \texttt{SeaSurfaceTemp} Approximation.} (Sparse maps, $k = 48$, $s = 839$, $q = 10$.)
The lower~\cref{eqn:scree-lower} and upper~\cref{eqn:scree-upper} approximations
of the scree curve~\cref{eqn:scree}.
The \textbf{vertical line} marks the truncation rank $r = 5$.
See~\cref{sec:sst-data,tab:sst-scree} for details.}
\label{fig:sst-scree}
\end{figure}

\begin{table}[t]
\begin{center}
\caption{\textbf{\emph{A Posteriori} Error Evaluation of Sea Surface Temperature Approximations.}
This table lists the lower and upper estimates for the true scree curve~\cref{eqn:scree} of the matrix
\texttt{SeaSurfaceTemp}.  We truncate at rank $r = 5$ (\textbf{blue}).}
\label{tab:sst-scree}
\begin{tabular}{|r|r|r|}
\hline
\textbf{Rank} & \textbf{Lower Estimate~\cref{eqn:scree-lower}} & \textbf{Upper Estimate~\cref{eqn:scree-upper}} \\
($r$) & $\underline{\mathrm{scree}}(r)$ & $\overline{\mathrm{scree}}(r)$ \\
\hline \hline
 1 &   $2.5415 \cdot 10^{-2}$   &	$5.2454 \cdot 10^{-2}$ \\
 2 &   $2.1068 \cdot 10^{-3}$   &	$1.3342 \cdot 10^{-2}$ \\
 3 &   $1.2867 \cdot 10^{-3}$	&   $1.1126 \cdot 10^{-2}$ \\
 4 &   $5.8939 \cdot 10^{-4}$	& 	$8.8143 \cdot 10^{-3}$ \\
\rowcolor{paleblue}
 5 &   $3.9590 \cdot 10^{-4}$	& 	$8.0110 \cdot 10^{-3}$ \\
 6 &   $3.0878 \cdot 10^{-4}$	& 	$7.6002 \cdot 10^{-3}$ \\
 7 &   $2.5140 \cdot 10^{-4}$	& 	$7.3039 \cdot 10^{-3}$ \\
 8 &   $2.1541 \cdot 10^{-4}$	& 	$7.1038 \cdot 10^{-3}$ \\
 9 &   $1.8673 \cdot 10^{-4}$	& 	$6.9342 \cdot 10^{-3}$ \\
10 &   $1.6410 \cdot 10^{-4}$	& 	$6.7926 \cdot 10^{-3}$ \\
\hline
\end{tabular}
\end{center}
\end{table}

Last, we investigate the sampling distribution of the error in
the randomized matrix approximation and the sampling distribution of
the error estimator.  To do so, we perform 1000 independent
trials of the same experiment for select values of $k$
and with error sketch sizes $q \in \{5, 10\}$.

\Cref{fig:sampling-distribution} contains scatter plots
of the actual approximation error $\fnormsq{ \mtx{A} - \hat{\mtx{A}}_{\rm out} }$
versus the estimated approximation error $\err_2^2(\hat{\mtx{A}}_{\rm out})$
for $\hat{\mtx{A}}_{\rm out} = \hat{\mtx{A}}$ and
$\hat{\mtx{A}}_{\rm out} = \lowrank{\hat{\mtx{A}}}{r}$.
The error estimators are unbiased, but they exhibit
a lot of variability.  Doubling the error sketch
size $q$ reduces the spread of the error estimate by a factor of two.
The approximation errors cluster tightly, as we
expect from concentration of measure.
The plots also highlight that the initial rank-$k$ approximations
are far from attaining the minimal rank-$k$ error,
while the truncated rank-$r$ approximations are more successful.

\begin{figure}[t!]
\vspace{3pc}
\begin{center}
\begin{subfigure}{.48\textwidth}
\begin{center}
\includegraphics[width=0.9\textwidth]{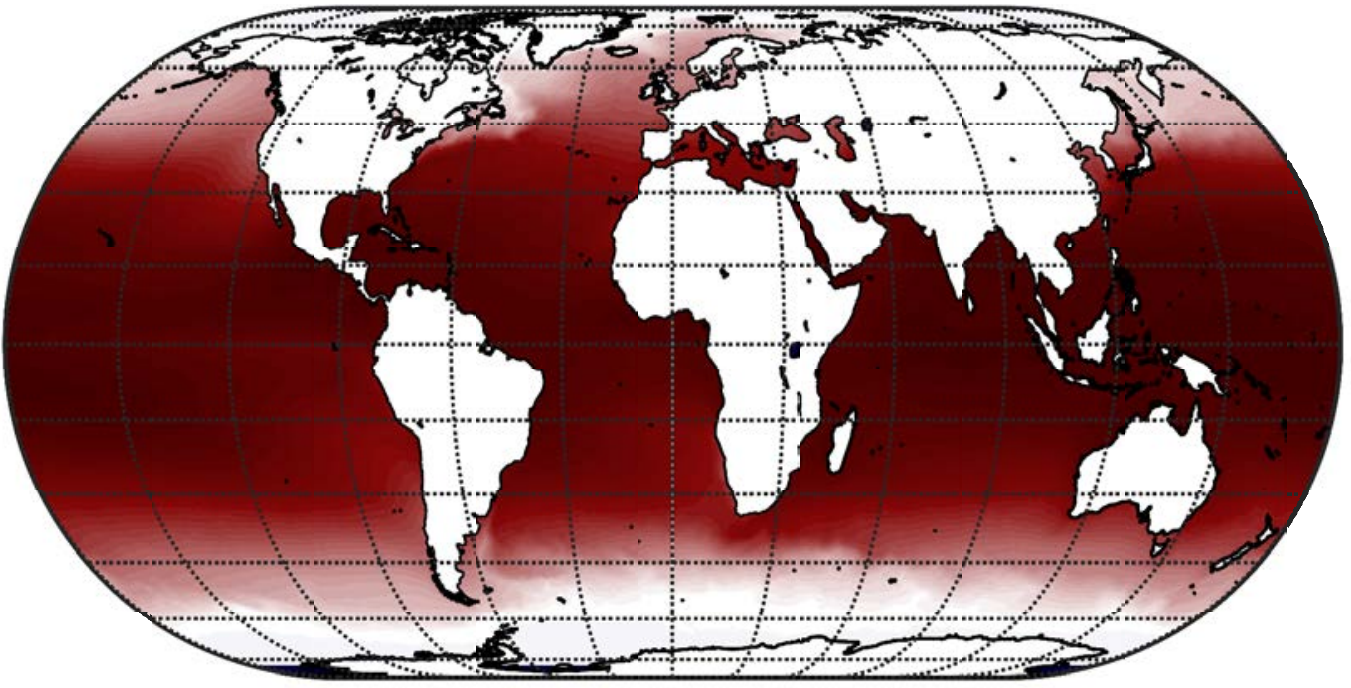}
\end{center}
\end{subfigure}
\begin{subfigure}{.48\textwidth}
\begin{center}
\includegraphics[width=\textwidth]{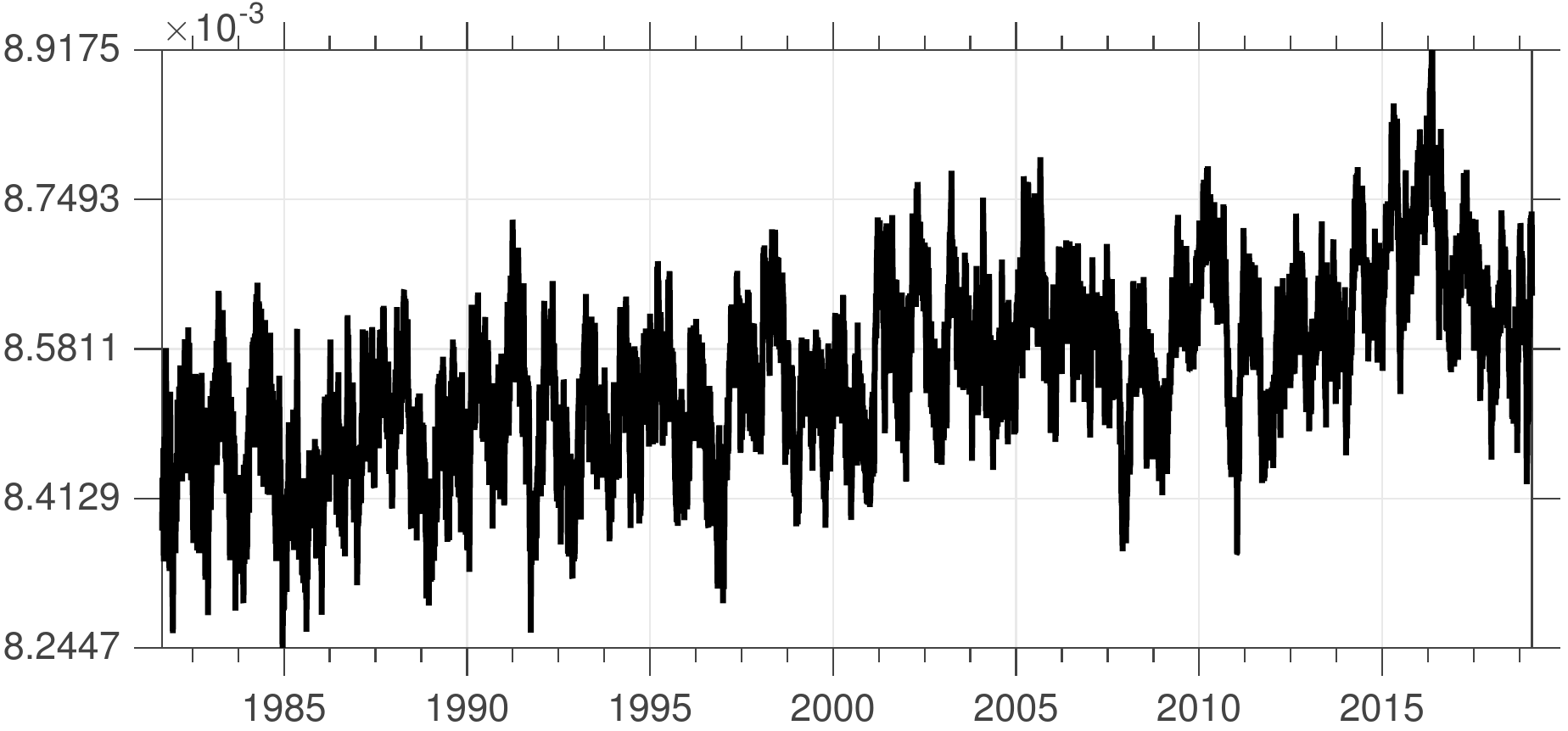}
\end{center}
\end{subfigure}
\end{center}

\vspace{1em}

\begin{center}
\begin{subfigure}{.48\textwidth}
\begin{center}
\includegraphics[width=0.9\textwidth]{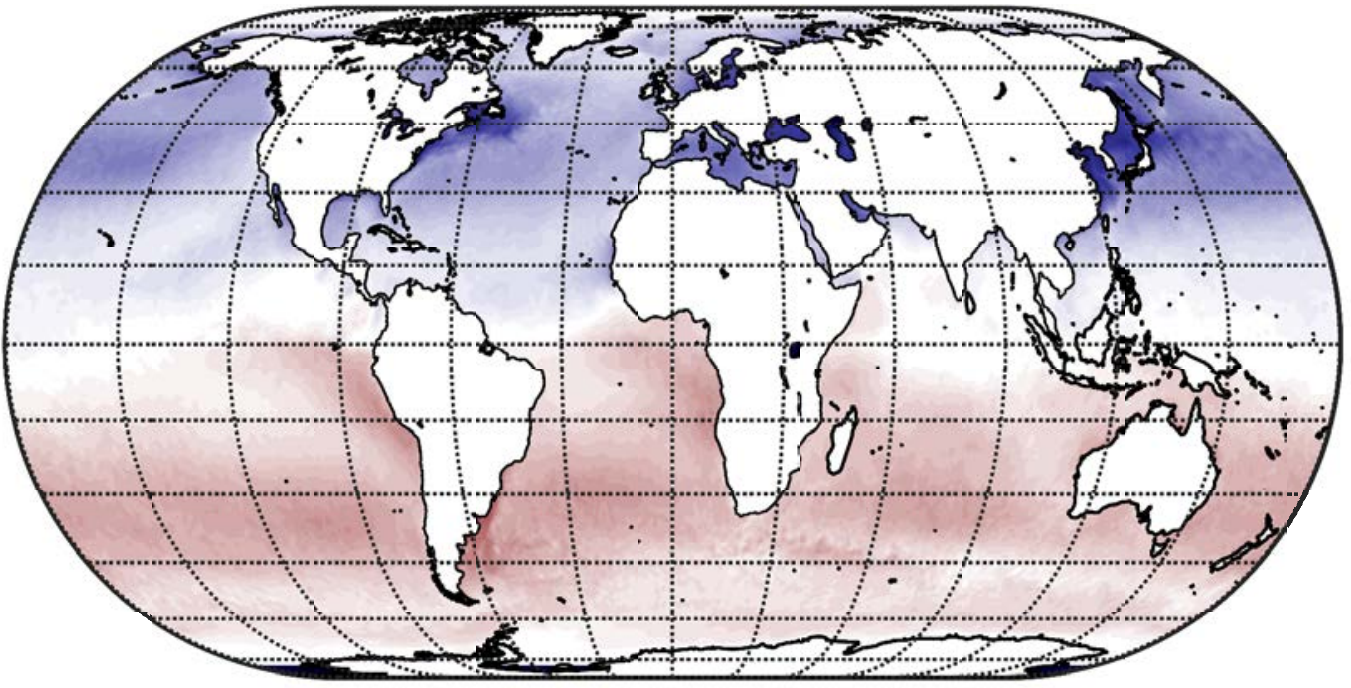}
\end{center}
\end{subfigure}
\begin{subfigure}{.48\textwidth}
\begin{center}
\includegraphics[width=\textwidth]{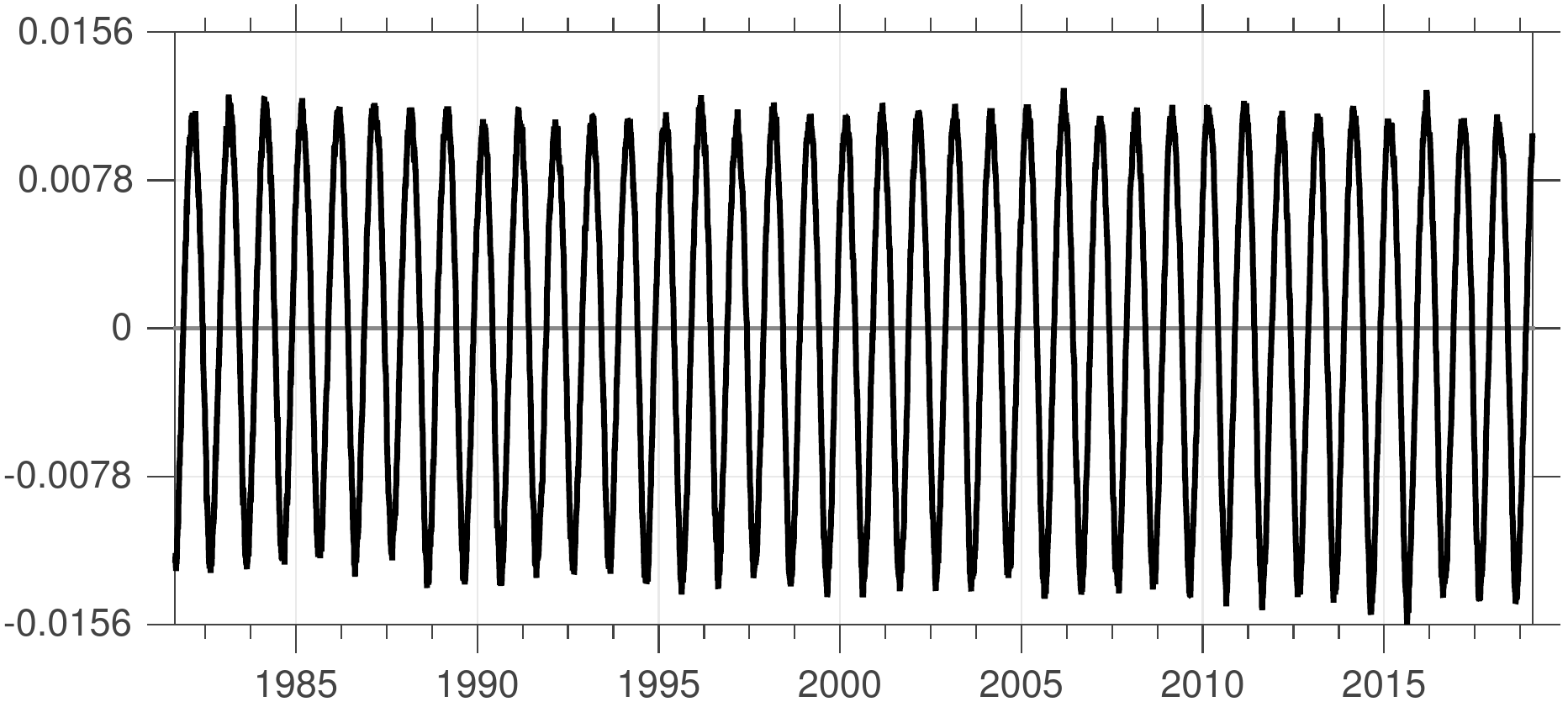}
\end{center}
\end{subfigure}
\end{center}

\vspace{1em}

\begin{center}
\begin{subfigure}{.48\textwidth}
\begin{center}
\includegraphics[width=0.9\textwidth]{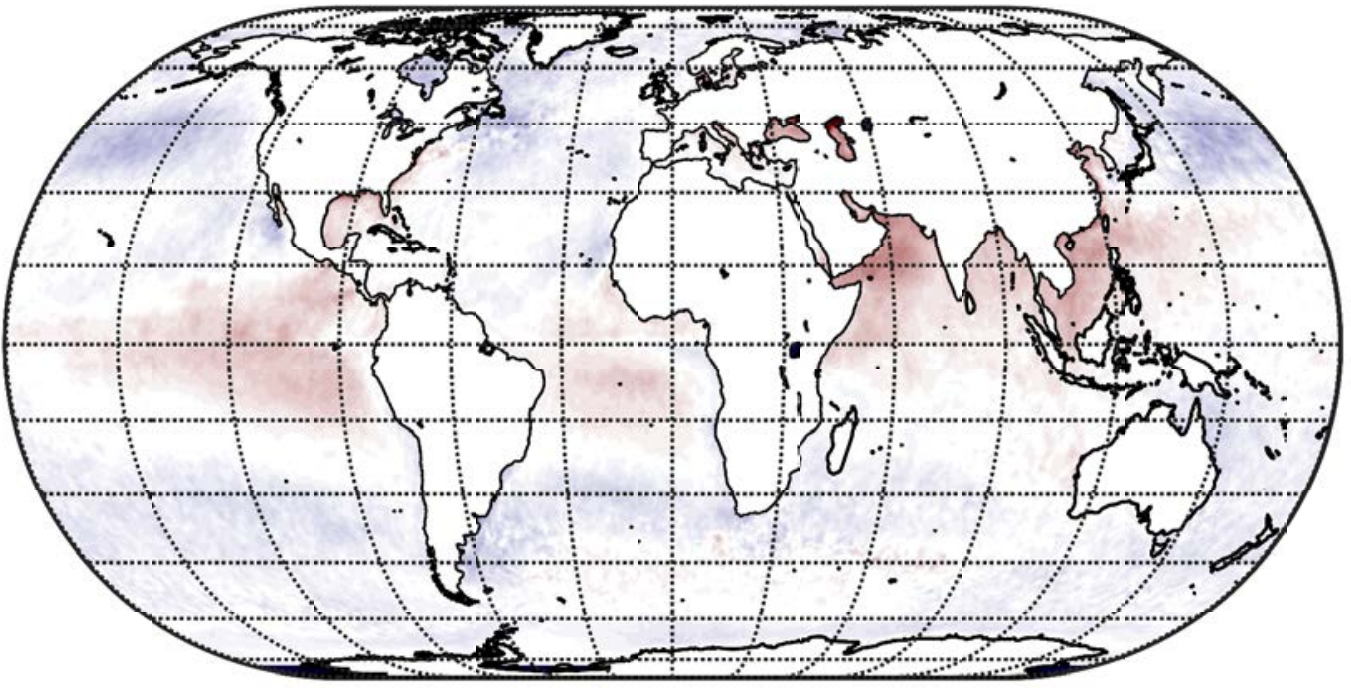}
\end{center}
\end{subfigure}
\begin{subfigure}{.48\textwidth}
\begin{center}
\includegraphics[width=\textwidth]{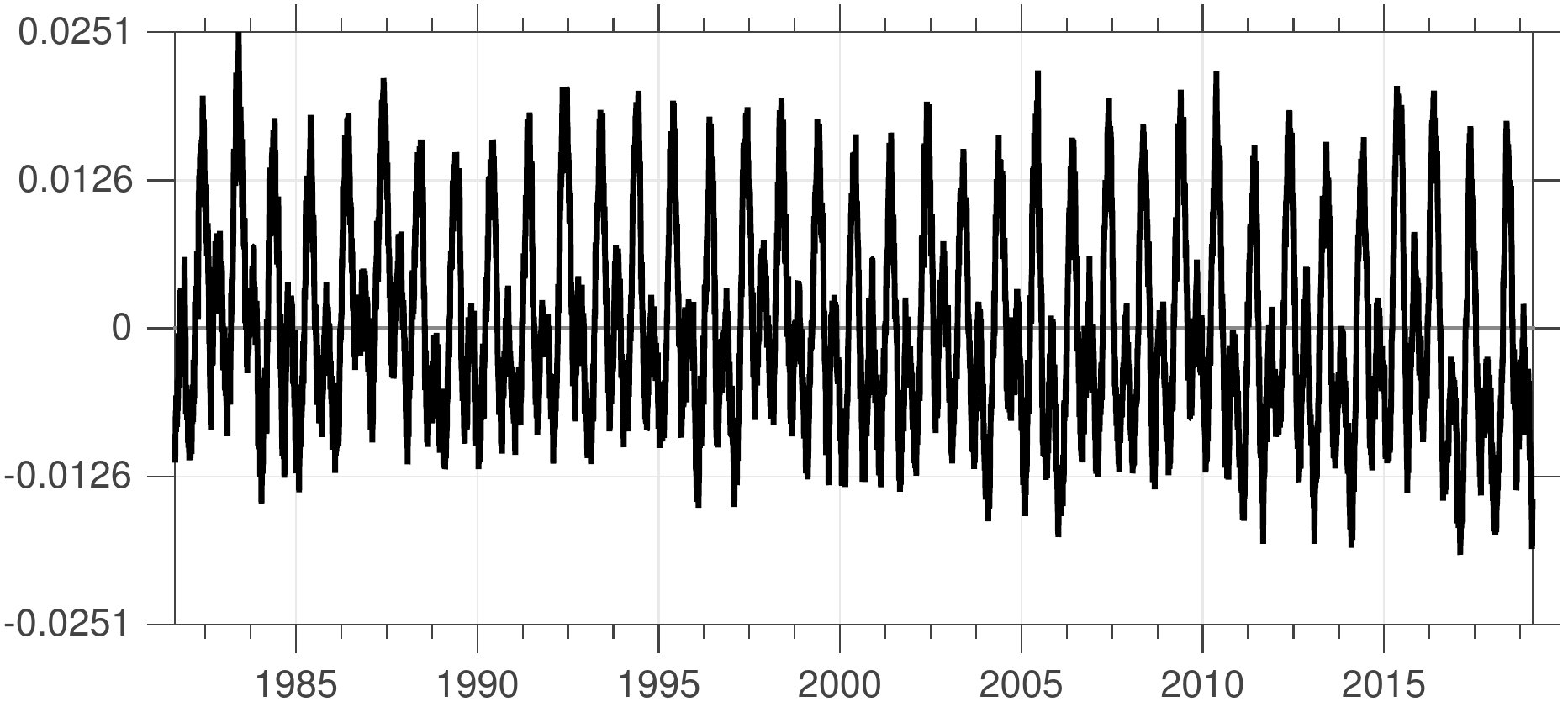}
\end{center}
\end{subfigure}
\end{center}

\vspace{1em}

\begin{center}
\begin{subfigure}{.45\textwidth}
\begin{center}
\includegraphics[width=0.9\textwidth]{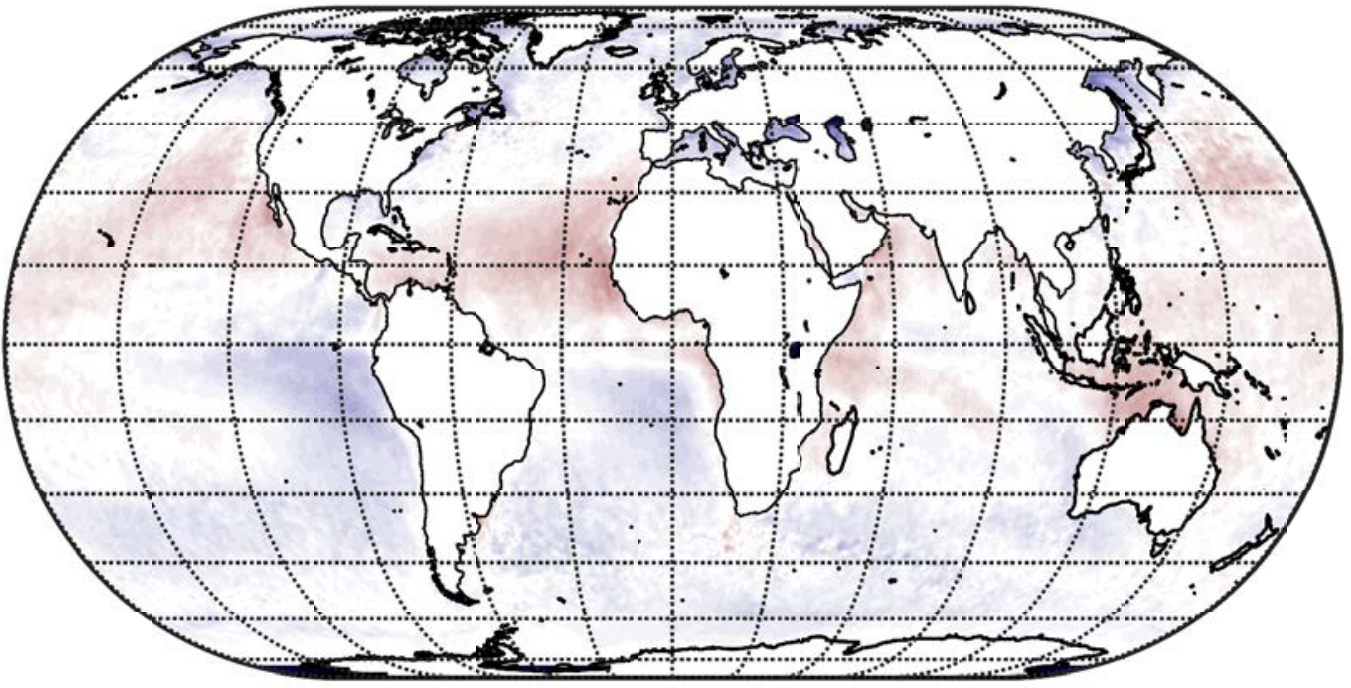}
\end{center}
\end{subfigure}
\begin{subfigure}{.48\textwidth}
\begin{center}
\includegraphics[width=\textwidth]{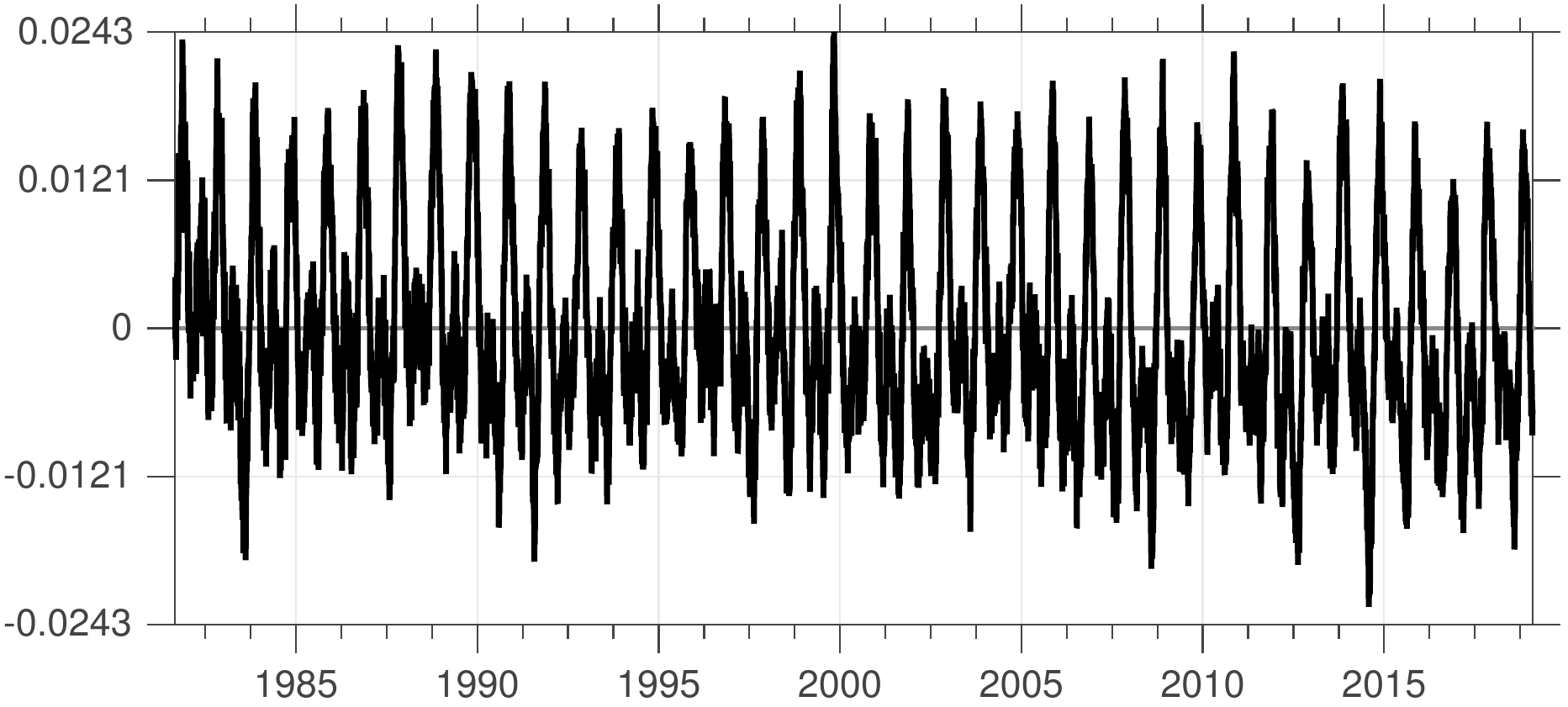}
\end{center}
\end{subfigure}
\end{center}

\vspace{1em}

\begin{center}
\begin{subfigure}{.48\textwidth}
\begin{center}
\includegraphics[width=0.9\textwidth]{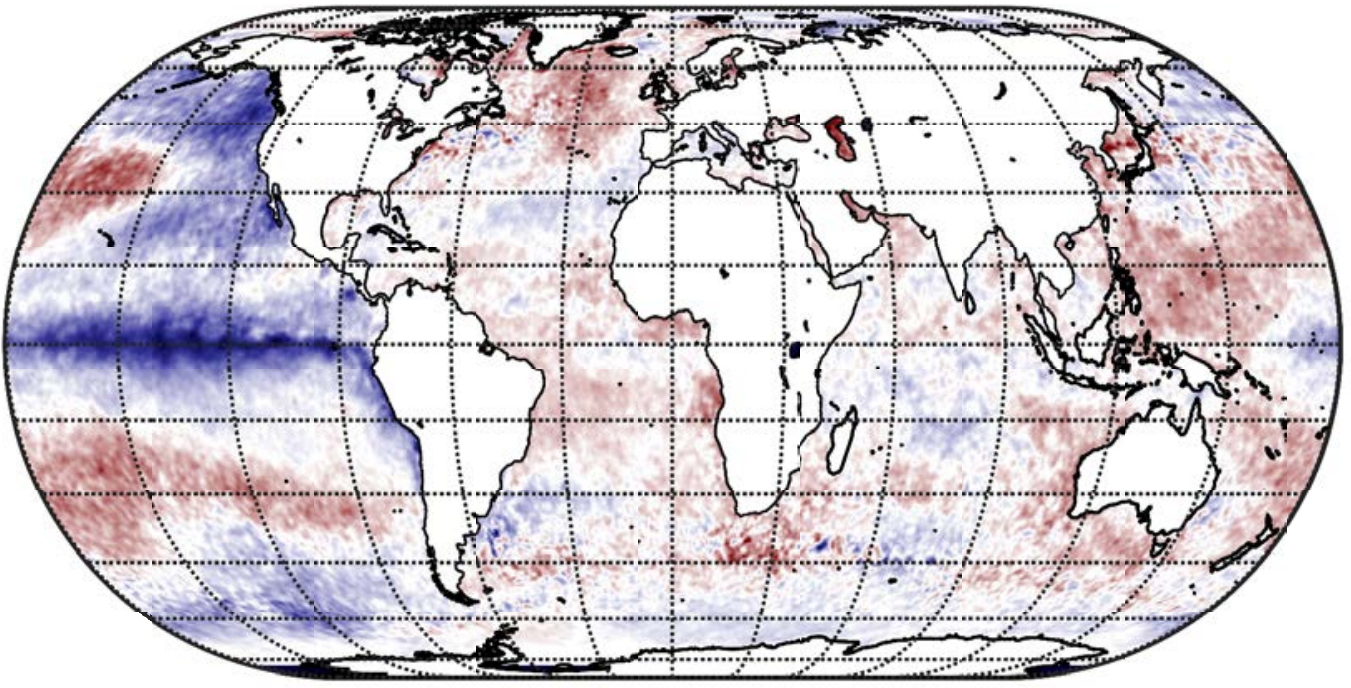}
\end{center}
\end{subfigure}
\begin{subfigure}{.48\textwidth}
\begin{center}
\includegraphics[width=\textwidth]{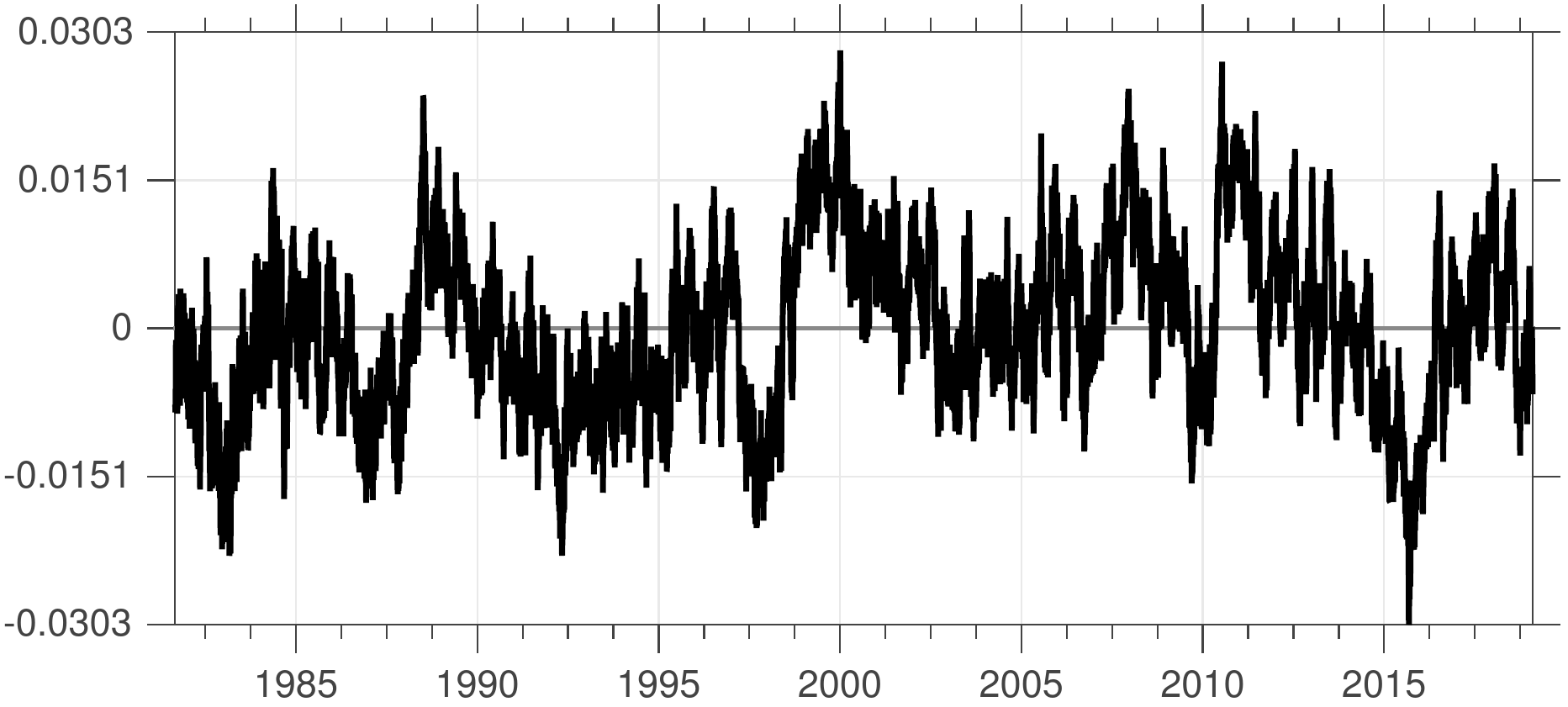}
\end{center}
\end{subfigure}
\end{center}

\vspace{1em}

\caption{\textbf{Singular Vectors of \texttt{SeaSurfaceTemp}.} (Sparse maps, $k = 48$, $s = 839$, $q = 10$.)
The \textbf{left column} displays the first five left singular vectors.
The heatmaps use \textbf{white} to represent zero Celsius degrees; each image is scaled independently.
The \textbf{right column} displays the first five right singular vectors.
The horizontal axis marks the year (common era); the vertical axis is unitless.
See \cref{sec:sst-data}.}
\label{fig:sst-svecs}
\end{figure}

\subsection{Example: Sea Surface Temperature Data}
\label{sec:sst-data}

Finally, we give a complete demonstration of the overall methodology
for the matrix \texttt{SeaSurfaceTemp}.  Like the matrix \texttt{MinTemp},
we expect that the sea surface temperature data has medium polynomial
decay, so it should be well approximated by a low-rank matrix.

\vspace{0.5pc}

\begin{enumerate} \setlength{\itemsep}{0.5pc}
\item	\textbf{Parameter selection.}
We fix the storage budget $T = 48 (m + n)$. The natural
parameter selection~\cref{eqn:ks-natural} yields
$k = 47$ and $s = 839$.  We use sparse dimension
reduction maps.  The error sketch has size $q = 10$.

\item	\textbf{Data collection.}
We ``stream'' the data one year at a time to construct the
approximation and error sketches.

\item	\textbf{Error estimates and rank truncation.}
Once the data is collected, we compute the rank-$k$ approximation
$\hat{\mtx{A}}$ using the formula~\cref{eqn:Ahat}.
We present the empirical scree estimates~\cref{eqn:scree-lower,eqn:scree-upper}
in \cref{fig:sst-scree,tab:sst-scree}.  These values should bracket
the unknown scree curve~\cref{eqn:scree}, while mimicking its shape.
By visual inspection, we set the truncation rank $r = 5$.
We expect that the rank-$5$ approximation captures all but $0.04\%$ to $0.8\%$
of the energy. %

\item	\textbf{Visualization.}
\Cref{fig:sst-svecs} illustrates the first five singular vector pairs of the rank-$r$
approximation of the matrix \texttt{SeaSurfaceTemp}.  The first left singular vector
can be interpreted as the mean temperature profile; a warming trend is visible
in the first right singular vector.
The second pair reflects the austral/boreal divide.
The remaining singular vectors capture long-term climatological features.
\end{enumerate}

\vspace{0.5pc}

The total storage required for the approximation sketch and the error sketch
is $4.09 \cdot 10^7$ numbers.  
This stands in contrast to the $mn = 9.09 \cdot 10^9$ numbers %
appearing in the matrix itself.  The compression ratio is $222 \times$. 
Moreover, the computational time required to obtain the approximation is modest
because we are working with substantially smaller matrices.

\section{Conclusions}

This paper exhibits a sketching method and a new reconstruction algorithm
for low-rank approximation of matrices that are presented as a sequence
of linear updates (\cref{sec:sketching}).  The algorithm is accompanied
by \emph{a priori} error bounds that allow us to set algorithm parameters
reliably (\cref{sec:theory}), as well as an \emph{a posteriori} error estimator
that allows us to validate its performance and to select the final rank of
the approximation (\cref{sec:aposterior}).  We discuss implementation
issues (\cref{sec:dim-red-maps,sec:implementation}), and we present
numerical experiments to show that the new method improves
over existing techniques (\cref{sec:alg-comparison,sec:real-data}).

A potential application of these techniques is for on-the-fly-compression
of large-scale scientific simulations and data collection.
Our experiments with a Navier--Stokes simulation (\cref{sec:flow-field})
and with sea surface temperature data (\cref{sec:sst-data}) both support this
hypothesis.  We hope that this work motivates researchers to investigate
the use of sketching in new applications.

\newpage
\appendix

\section{Analysis of the Low-Rank Approximation}
\label{app:error-bound-proof}
This section contains the proof of \cref{thm:low-rank-error-bound},
the theoretical result on the behavior of the basic low-rank
approximation~\cref{eqn:Ahat}.  We maintain the notation
from \cref{sec:sketching}.

\subsection{Facts about Random Matrices}

First, let us state a useful formula that allows us to compute
some expectations involving a Gaussian random matrix.  This
identity is drawn from \cite[Prop.~A.1 and A.6]{HMT11:Finding-Structure}.
See also~\cite[Fact~A.1]{TYUC17:Practical-Sketching}.

\begin{fact} \label{fact:twisted-expectation}
Assume that $t > q + \alpha$.
Let $\mtx{G}_1 \in \F^{t \times q}$ and $\mtx{G}_2 \in \F^{t \times p}$
be independent standard normal matrices.  For any matrix $\mtx{B}$
with conforming dimensions,
$$
\Expect \fnormsq{ \mtx{G}_1^\dagger \mtx{G}_2 \mtx{B}  }
	= \frac{q}{t - q - \alpha} \cdot \fnormsq{\mtx{B}}.
$$
The number $\alpha = 1$ when $\F = \R$, while $\alpha = 0$ when $\F = \C$. 
\end{fact}

\subsection{Results from Randomized Linear Algebra}

Our argument also depends on the analysis of randomized
low-rank approximation developed in~\cite[Sec.~10]{HMT11:Finding-Structure}.

\begin{fact}[Halko et al.~2011] \label{fact:hmt-err}
Fix $\mtx{A} \in \F^{m \times n}$.  Let $\varrho$ be a natural number
such that $\varrho < k - \alpha$.
Draw the random test matrix $\mtx{\Omega} \in \F^{k \times n}$
from the standard normal distribution.
Then the matrix $\mtx{Q} \in \F^{m \times k}$ computed by~\cref{eqn:range-corange} satisfies
$$
\Expect_{\mtx{\Omega}} \fnormsq{ \mtx{A} - \mtx{QQ}^* \mtx{A} }
	\leq \left( 1 + \frac{\varrho}{k - \varrho - \alpha} \right) \cdot \tau_{\varrho+1}^2(\mtx{A}).
$$
An analogous result holds for the matrix $\mtx{P} \in \F^{n \times k}$ computed by~\cref{eqn:range-corange}:
$$
\Expect_{\mtx{\Upsilon}} \fnormsq{ \mtx{A} - \mtx{A}\mtx{PP}^*  }
	\leq \left( 1 + \frac{\varrho}{k - \varrho - \alpha} \right) \cdot \tau_{\varrho+1}^2(\mtx{A}).
$$
The number $\alpha = 1$ when $\F = \R$, while $\alpha = 0$ when $\F = \C$. 
\end{fact}

\noindent
This result follows immediately from the
proof of~\cite[Thm.~10.5]{HMT11:Finding-Structure}
using \cref{fact:twisted-expectation} to handle both
the real and complex case simultaneously.
See also~\cite[Sec.~8.2]{TYUC17:Randomized-Single-View-TR}.

\subsection{Decomposition of the Core Matrix Approximation Error}

The first step in the argument is to obtain a formula
for the error in the approximation
$\mtx{C} - \mtx{Q}^* \mtx{A} \mtx{P}$.
The core matrix $\mtx{C} \in \F^{s \times s}$
is defined in~\eqref{eqn:linkage-matrix}.
We constructed the orthonormal matrices
$\mtx{P} \in \F^{n \times k}$ and $\mtx{Q} \in \F^{m \times k}$
in~\eqref{eqn:range-corange}.

Let us introduce matrices whose ranges are complementary
to those of $\mtx{P}$ and $\mtx{Q}$:
$$
\begin{aligned}
\mtx{P}_\perp \mtx{P}_{\perp}^* := \Id - \mtx{PP}^*
&\quad\text{where}\quad \mtx{P}_{\perp} \in \F^{n \times (n - k)}; \\
\mtx{Q}_\perp \mtx{Q}_{\perp}^* := \Id - \mtx{QQ}^*
&\quad\text{where}\quad \mtx{Q}_{\perp} \in \F^{m \times (m - k)}.
\end{aligned}
$$
The columns of $\mtx{P}_{\perp}$ are orthonormal, and the columns of $\mtx{Q}_{\perp}$ are orthonormal.
Next, introduce the submatrices
\begin{equation} \label{eqn:test-decomp}
\begin{aligned}
\mtx{\Phi}_1 = \mtx{\Phi} \mtx{Q} \in \F^{s \times k}
&\quad\text{and}\quad
\mtx{\Phi}_2 = \mtx{\Phi} \mtx{Q}_\perp \in \F^{s \times (m - k)}; \\
\mtx{\Psi}_1^* = \mtx{P}^* \mtx{\Psi}^* \in \F^{k \times s}
&\quad\text{and}\quad
\mtx{\Psi}_2^* = \mtx{P}_\perp^* \mtx{\Psi}^* \in \F^{(n - k) \times s}.
\end{aligned}
\end{equation}
With this notation at hand, we can state and prove the first result.

\begin{lemma}[Decomposition of the Core Matrix Approximation]
\label{lem:link-approx-decomp}
Assume that the matrices $\mtx{\Phi}_1$ and $\mtx{\Psi}_1$ have full column rank.  Then
\begin{align*}
\mtx{C} - \mtx{Q}^* \mtx{A} \mtx{P} &=
	\mtx{\Phi}_1^\dagger \mtx{\Phi}_2 (\mtx{Q}_{\perp}^* \mtx{A} \mtx{P})
		+ (\mtx{Q}^* \mtx{A} \mtx{P}_{\perp}) \mtx{\Psi}_2^* (\mtx{\Psi}_1^\dagger)^* \\
		&+ \mtx{\Phi}_1^\dagger \mtx{\Phi}_2 (\mtx{Q}_{\perp}^* \mtx{A} \mtx{P}_{\perp}) \mtx{\Psi}_2^* (\mtx{\Psi}_1^\dagger)^*.
\end{align*}
\end{lemma}

\begin{proof}
Adding and subtracting terms, we write the core sketch $\mtx{Z}$ as
$$
\mtx{Z} = \mtx{\Phi} \mtx{A} \mtx{\Psi}^*
	= \mtx{\Phi}( \mtx{A} - \mtx{QQ}^* \mtx{A} \mtx{PP}^* ) \mtx{\Psi}^*
	+ (\mtx{\Phi Q})(\mtx{Q}^* \mtx{A} \mtx{P}) (\mtx{P}^* \mtx{\Psi}^*).
$$
Using~\eqref{eqn:test-decomp}, we identify the matrices $\mtx{\Phi}_1$ and $\mtx{\Psi}_1$.
Then left-multiply by $\mtx{\Phi}_1^\dagger$ and right-multiply by $(\mtx{\Psi}_1^\dagger)^*$
to arrive at
$$
\mtx{C} = \mtx{\Phi}_1^\dagger \mtx{Z} (\mtx{\Psi}_1^\dagger)^*
	= \mtx{\Phi}_1^\dagger \mtx{\Phi} ( \mtx{A} - \mtx{QQ}^* \mtx{A} \mtx{PP}^* ) \mtx{\Psi}^* (\mtx{\Psi}_1^\dagger)^*
	+ \mtx{Q}^* \mtx{A} \mtx{P}.
$$
We have identified the core matrix $\mtx{C}$, defined in \cref{eqn:linkage-matrix}.
Move the term $\mtx{Q}^* \mtx{A} \mtx{P}$ to the left-hand side to isolate
the approximation error.

To continue, notice that
$$
\mtx{\Phi}_1^\dagger \mtx{\Phi}
	= \mtx{\Phi}_1^\dagger \mtx{\Phi} \mtx{QQ}^* + \mtx{\Phi}_1^\dagger \mtx{\Phi} \mtx{Q}_{\perp} \mtx{Q}_{\perp}^*
	= \mtx{Q}^* + \mtx{\Phi}_1^\dagger \mtx{\Phi}_2 \mtx{Q}_{\perp}^*.
$$
Likewise,
$$
\mtx{\Psi}^* (\mtx{\Psi}_1^\dagger)^*
	= \mtx{PP}^* \mtx{\Psi}^* (\mtx{\Psi}_1^\dagger)^* + \mtx{P}_{\perp} \mtx{P}_{\perp}^* \mtx{\Psi}^* (\mtx{\Psi}_1^\dagger)^*
	= \mtx{P} + \mtx{P}_{\perp} \mtx{\Psi}_2^* (\mtx{\Psi}_1^\dagger)^*.
$$
Combine the last three displays to arrive at
$$
\mtx{C} - \mtx{Q}^* \mtx{A} \mtx{P}
	= (\mtx{Q}^* + \mtx{\Phi}_1^\dagger \mtx{\Phi}_2 \mtx{Q}_{\perp}^*)( \mtx{A} - \mtx{QQ}^* \mtx{A} \mtx{PP}^* )
	( \mtx{P} + \mtx{P}_{\perp} \mtx{\Psi}_2^* (\mtx{\Psi}_1^\dagger)^* ).
$$
Expand the expression and use the orthogonality relations $\mtx{Q}^*\mtx{Q} = \Id$ and $\mtx{Q}_{\perp}^* \mtx{Q} = \mtx{0}$
and $\mtx{P}^*\mtx{P} = \Id$ and $\mtx{P}^* \mtx{P}_{\perp} = \mtx{0}$ to arrive at the desired representation.
\end{proof}

\subsection{Probabilistic Analysis of the Core Matrix}

Next, we make distributional assumptions on the dimension reduction
maps $\mtx{\Phi}$ and $\mtx{\Psi}$.  We can then study the probabilistic
behavior of the error $\mtx{C} - \mtx{Q}^* \mtx{A} \mtx{P}$,
conditional on $\mtx{Q}$ and $\mtx{P}$.

\begin{lemma}[Probabilistic Analysis of the Core Matrix]
\label{lem:link-approx-stat}
Assume that the dimension reduction matrices $\mtx{\Phi}$ and $\mtx{\Psi}$
are drawn independently from the standard normal distribution.
When $s \geq k$, it holds that
\begin{equation} \label{eqn:W-unbiased}
\Expect_{\mtx{\Phi}, \mtx{\Psi}}[ \mtx{C} - \mtx{Q}^*\mtx{A} \mtx{P} ] = \mtx{0}.
\end{equation}
When $s > k + \alpha$, we can express the error as
$$
\begin{aligned}
\Expect_{\mtx{\Phi}, \mtx{\Psi}} \fnormsq{ \mtx{C} - \mtx{Q}^* \mtx{A} \mtx{P} }
	&= \frac{k}{s - k - \alpha} \cdot \fnormsq{ \mtx{A} - \mtx{QQ}^* \mtx{A} \mtx{PP}^* } \\
	&\qquad + \frac{k(2k + \alpha - s)}{(s - k - \alpha)^2} \cdot  \fnormsq{\mtx{Q}_\perp^* \mtx{A} \mtx{P}_{\perp}}.
\end{aligned}
$$
When $s < 2k + \alpha$, the last term is nonnegative;
when $s \geq 2k + \alpha$, the last term is nonpositive.
\end{lemma}

\begin{proof}
Since $\mtx{\Phi}$ is standard normal, the orthogonal submatrices
$\mtx{\Phi}_1$ and $\mtx{\Phi}_2$ are statistically independent
standard normal matrices because of the marginal property of the
normal distribution.  Likewise, $\mtx{\Psi}_1$ and $\mtx{\Psi}_2$
are statistically independent standard normal matrices.  Provided
that $s \geq k$, both matrices have full column rank with probability
one.

To establish the formula~\eqref{eqn:W-unbiased}, notice that
\begin{align*}
\Expect_{\mtx{\Phi}, \mtx{\Psi}} [ \mtx{C} - \mtx{Q}^* \mtx{A} \mtx{P} ]
	&= \Expect_{\mtx{\Phi}_1}\Expect_{\mtx{\Phi}_2} [ \mtx{\Phi}_1^\dagger \mtx{\Phi}_2 (\mtx{Q}_{\perp}^* \mtx{A} \mtx{P}) ]
		+ \Expect_{\mtx{\Psi}_1}\Expect_{\mtx{\Psi}_2} [(\mtx{Q}^* \mtx{A} \mtx{P}_{\perp}) \mtx{\Psi}_2^* (\mtx{\Psi}_1^\dagger)^* ] \\
		&+ \Expect \Expect_{\mtx{\Phi}_2} [ \mtx{\Phi}_1^\dagger \mtx{\Phi}_2 (\mtx{Q}_{\perp}^* \mtx{A} \mtx{P}_{\perp}) \mtx{\Psi}_2^* (\mtx{\Psi}_1^\dagger)^* ].
\end{align*}
We have used the decomposition of the approximation error from
\cref{lem:link-approx-decomp}.  Then we invoke independence to write the expectations
as iterated expectations.  Since $\mtx{\Phi}_2$ and $\mtx{\Psi}_2$ have mean zero,
this formula makes it clear that the approximation error has mean zero.

To study the fluctuations, apply the independence and zero-mean property of $\mtx{\Phi}_2$
and $\mtx{\Psi}_2$ to decompose
\begin{align*}
\Expect_{\mtx{\Phi}, \mtx{\Psi}} \fnormsq{ \mtx{C} - \mtx{Q}^* \mtx{A} \mtx{P} }
	&= \Expect_{\mtx{\Phi}} \fnormsq{ \mtx{\Phi}_1^\dagger \mtx{\Phi}_2 (\mtx{Q}_{\perp}^* \mtx{A} \mtx{P}) }
		+ \Expect_{\mtx{\Psi}} \fnormsq{ (\mtx{Q}^* \mtx{A} \mtx{P}_{\perp}) \mtx{\Psi}_2^* (\mtx{\Psi}_1^\dagger)^* } \\
		&+ \Expect_{\mtx{\Phi}} \Expect_{\mtx{\Psi}} \fnormsq{ \mtx{\Phi}_1^\dagger \mtx{\Phi}_2 (\mtx{Q}_{\perp}^* \mtx{A} \mtx{P}_{\perp}) \mtx{\Psi}_2^* (\mtx{\Psi}_1^\dagger)^* }.
\end{align*}
Continuing, we invoke \cref{fact:twisted-expectation} four times to see that
\begin{multline*}
\Expect_{\mtx{\Phi}, \mtx{\Psi}} \fnormsq{ \mtx{C} - \mtx{Q}^* \mtx{A} \mtx{P} } \\
	= \frac{k}{s - k - \alpha} \cdot \left[ \fnormsq{ \mtx{Q}_\perp^* \mtx{A} \mtx{P} }
	+ \fnormsq{ \mtx{Q}^* \mtx{A} \mtx{P}_{\perp} }
	+ \frac{k}{s - k - \alpha} \cdot \fnormsq{ \mtx{Q}_{\perp}^* \mtx{A} \mtx{P}_{\perp} } \right].
\end{multline*}
Add and subtract $\fnormsq{ \mtx{Q}_{\perp}^* \mtx{A} \mtx{P}_{\perp} }$ in the bracket to arrive at
\begin{align*}
\Expect \fnormsq{ \mtx{C} - \mtx{Q}^* \mtx{A} \mtx{P} }
	&= \frac{k}{s - k - \alpha} \cdot \bigg[ \fnormsq{ \mtx{Q}_\perp^* \mtx{A} \mtx{P} }
	+ \fnormsq{ \mtx{Q}^* \mtx{A} \mtx{P}_{\perp} }
	+ \fnormsq{ \mtx{Q}_{\perp}^* \mtx{A} \mtx{P}_{\perp} } \\
	&\qquad+ \frac{2k + \alpha - s}{s - k - \alpha} \cdot \fnormsq{ \mtx{Q}_{\perp}^* \mtx{A} \mtx{P}_{\perp} } \bigg].
\end{align*}
Use the Pythagorean Theorem to combine the terms on the first line.
\end{proof}

\subsection{Probabilistic Analysis of the Compression Error}

Next, we establish a bound for the expected error in the compression
of the matrix $\mtx{A}$ onto the range of the orthonormal matrices $\mtx{Q}$ and $\mtx{P}$,
computed in~\cref{eqn:range-corange}.
This result is similar in spirit to the analysis in~\cite{HMT11:Finding-Structure},
so we pass lightly over the details.

\begin{lemma}[Probabilistic Analysis of the Compression Error]
\label{lem:compression-error}
For any natural number $\varrho < k - \alpha$, it holds that
$$
\Expect \fnormsq{ \mtx{A} - \mtx{QQ}^* \mtx{A} \mtx{PP}^* }
	\leq \left(1 + \frac{2\varrho}{k - \varrho - \alpha}\right)
	\cdot \tau_{\varrho+1}^2(\mtx{A}).
$$
\end{lemma}

\begin{proof}[Proof Sketch]
Introduce the partitioned SVD of the matrix $\mtx{A}$:
$$
\mtx{A} = \mtx{U\Sigma V}^*
	= \begin{bmatrix} \mtx{U}_1 & \mtx{U}_2 \end{bmatrix}
	\begin{bmatrix} \mtx{\Sigma}_1 & \\ & \mtx{\Sigma}_2 \end{bmatrix}
	\begin{bmatrix} \mtx{V}_1^* \\ \mtx{V}_2^* \end{bmatrix}
	\quad\text{where}\quad \mtx{\Sigma}_1 \in \F^{\varrho \times \varrho}.
$$
Define the matrices
\begin{gather*}
\mtx{\Upsilon}_1 := \mtx{\Upsilon} \mtx{U}_1 \in \F^{s \times \varrho}
	\quad\text{and}\quad \mtx{\Upsilon}_2 := \mtx{\Upsilon} \mtx{U}_2 \in F^{s \times (m - \varrho)}; \\
\mtx{\Omega}_1^* := \mtx{V}_1^* \mtx{\Omega}^* \in \F^{\varrho \times s}
	\quad\text{and}\quad \mtx{\Omega}_2^* := \mtx{V}_2^* \mtx{\Omega}^* \in \F^{(n - \varrho) \times s}; \\
\mtx{P}_1 := \mtx{V}_1^* \mtx{P} \in \F^{\varrho \times k}
	\quad\text{and}\quad \mtx{P}_2 := \mtx{V}_2^* \mtx{P} \in \F^{(n - \varrho) \times k}.
\end{gather*}
With this notation, we proceed to the proof.

First, add and subtract terms and apply the Pythagorean Theorem to obtain
$$
\fnormsq{ \mtx{A} - \mtx{QQ}^* \mtx{A} \mtx{PP}^* }
	= \fnormsq{ \mtx{A} (\Id - \mtx{PP}^*) } + \fnormsq{ (\Id - \mtx{QQ}^*) \mtx{A} \mtx{PP}^* }.
$$
Use the SVD to decompose the matrix $\mtx{A}$ in the first term,
and apply the Pythagorean Theorem again:
\begin{multline*}
\fnormsq{ \mtx{A} - \mtx{QQ}^* \mtx{A} \mtx{PP}^* }
	= \fnormsq{ (\mtx{U}_2 \mtx{\Sigma}_2 \mtx{V}_2^*) (\Id - \mtx{PP}^*) } \\
	+ \fnormsq{ (\mtx{U}_1 \mtx{\Sigma}_1 \mtx{V}_1^*) (\Id - \mtx{PP}^*) }
	+ \fnormsq{ (\Id - \mtx{QQ}^*) \mtx{A} \mtx{P} }.
\end{multline*}
The result~\cite[Prop.~9.2]{TYUC17:Randomized-Single-View-TR}
implies that the second term satisfies
$$
\fnormsq{ (\mtx{U}_1 \mtx{\Sigma}_1 \mtx{V}_1^*) (\Id - \mtx{PP}^*) }
	\leq \fnormsq{ \mtx{\Upsilon}_1^\dagger \mtx{\Upsilon}_2 \mtx{\Sigma}_2 }.
$$
We can obtain a bound for the third term using the formula~\cite[p.~270, disp.~1]{HMT11:Finding-Structure}.
After a short computation, this result yields
$$
\begin{aligned}
\fnormsq{ (\Id - \mtx{QQ}^*) \mtx{A P} }
	&\leq \fnormsq{ \mtx{\Sigma}_2 \mtx{P}_2 }
	+ \fnormsq{ \mtx{\Sigma}_2 \mtx{\Omega}_2^* (\mtx{\Omega}_1^*)^\dagger \mtx{P}_1 } \\
	&\leq \fnormsq{ \mtx{\Sigma}_2 }
	+\fnormsq{ \mtx{\Sigma}_2 \mtx{\Omega}_2^* (\mtx{\Omega}_1^*)^\dagger }.
\end{aligned}
$$
We can remove $\mtx{P}_1$ and $\mtx{P}_2$ because their spectral norms are bounded by one,
being submatrices of the orthonormal matrix $\mtx{P}$.
Combine the last three displays to obtain
$$
\fnormsq{ \mtx{A} - \mtx{QQ}^* \mtx{A} \mtx{PP}^* }
	\leq \fnormsq{ \mtx{\Sigma}_2 } %
	+ \fnormsq{ \mtx{\Upsilon}_1^\dagger \mtx{\Upsilon}_2 \mtx{\Sigma}_2 }
	+ \fnormsq{ \mtx{\Sigma}_2 \mtx{\Omega}_2^* (\mtx{\Omega}_1^*)^\dagger }.
$$
We have used the Pythagorean Theorem again.

Take the expectation with respect to $\mtx{\Upsilon}$ and $\mtx{\Omega}$
to arrive at
$$
\begin{aligned}
\Expect \fnormsq{ \mtx{A} - \mtx{QQ}^* \mtx{A} \mtx{PP}^* }
	&\leq \fnormsq{ \mtx{\Sigma}_2 }
	+ \Expect \fnormsq{ \mtx{\Upsilon}_1^\dagger \mtx{\Upsilon}_2 \mtx{\Sigma}_2 }
	+ \Expect \fnormsq{ \mtx{\Sigma}_2 \mtx{\Omega}_2^* (\mtx{\Omega}_1^*)^\dagger } \\
	&= \fnormsq{ \mtx{\Sigma}_2 }
	+ \frac{2 \varrho}{k - \varrho - \alpha} \cdot \fnormsq{ \mtx{\Sigma}_2 }.
\end{aligned}
$$
Finally, note that $\fnormsq{\mtx{\Sigma}_2} = \tau_{\varrho+1}^2(\mtx{A})$.
\end{proof}

\subsection{The Endgame}

At last, we are prepared to finish the proof of \cref{thm:low-rank-error-bound}.
Fix a natural number $\varrho < k - \alpha$.
Using the formula~\cref{eqn:Ahat} for the approximation $\hat{\mtx{A}}$,
we see that
$$
\begin{aligned}
\fnormsq{ \mtx{A} - \hat{\mtx{A}} }
	&= \fnormsq{ \mtx{A} - \mtx{Q} \mtx{C} \mtx{P}^* } \\
	&= \fnormsq{ \mtx{A} - \mtx{QQ}^* \mtx{A} \mtx{PP}^* + \mtx{Q} (\mtx{Q}^* \mtx{A} \mtx{P} - \mtx{C} ) \mtx{P}^* } \\
	&= \fnormsq{ \mtx{A} - \mtx{QQ}^* \mtx{A} \mtx{PP}^* }
	+ \fnormsq{ \mtx{Q} (\mtx{Q}^* \mtx{A} \mtx{P} - \mtx{C} ) \mtx{P}^* }.
\end{aligned}
$$
The last identity is the Pythagorean theorem.  Drop the orthonormal matrices in the last term.
Then take the expectation with respect to $\mtx{\Phi}$ and $\mtx{\Psi}$:
$$
\Expect_{\mtx{\Phi}, \mtx{\Psi}} \fnormsq{\mtx{A} - \hat{\mtx{A}}}
	= \fnormsq{ \mtx{A} - \mtx{QQ}^* \mtx{A} \mtx{PP}^* }
	+ \Expect_{\mtx{\Phi}, \mtx{\Psi}} \fnormsq{ \mtx{Q}^* \mtx{A} \mtx{P} - \mtx{C} } \\
$$
We treat the two terms sequentially.

To continue, invoke the expression \cref{lem:link-approx-stat} for the expected error in the core matrix $\mtx{C}$:
$$
\begin{aligned}
\Expect_{\mtx{\Phi}, \mtx{\Psi}} \fnormsq{\mtx{A} - \hat{\mtx{A}}}
	&\leq \left( 1 + \frac{k}{s - k - \alpha} \right) \cdot \fnormsq{ \mtx{A} - \mtx{QQ}^* \mtx{A} \mtx{PP}^* } \\
	&+ \frac{k(2k + \alpha - s)}{(s - k - \alpha)^2} \cdot \fnormsq{ \mtx{Q}_\perp^* \mtx{A} \mtx{P}_{\perp} }.
\end{aligned}
$$
Now, take the expectation with respect to $\mtx{\Upsilon}$ and $\mtx{\Omega}$ to arrive at
\begin{equation} \label{eqn:almost-there}
\begin{aligned}
\Expect \fnormsq{\mtx{A} - \hat{\mtx{A}}}
	&\leq \left( 1 + \frac{k}{s - k - \alpha} \right) \cdot %
	\left( 1 + \frac{2 \varrho}{k - \varrho - \alpha} \right) \cdot \tau_{\varrho+1}^2(\mtx{A}) \\
	&\qquad+ \frac{k(2k + \alpha - s)}{(s - k - \alpha)^2} \cdot \Expect \fnormsq{ \mtx{Q}_\perp^* \mtx{A} \mtx{P}_{\perp} }.
\end{aligned}
\end{equation}
We have invoked \cref{lem:compression-error}.  The last term is nonpositive
because we require $s \geq 2k + \alpha$, so we may drop it from consideration.
Finally, we optimize over eligible choices $\varrho < k - \alpha$
to complete the argument.  The result stated in \cref{thm:low-rank-error-bound}
is algebraically equivalent.

\section{A Posteriori Error Estimation}
\label{app:aposteriori}

This section contains proofs of the bounds on the \emph{a posteriori}
error estimator $\err_2$ computed using a Gaussian error sketch.
It also establishes the linear algebra results that we need to
diagnose spectral decay in the input matrix.

\subsection{The Frobenius Norm Estimator}
\label{app:fnorm-est}

Fix an arbitrary matrix $\mtx{M} \in \F^{m \times n}$,
which plays the role of the discrepancy $\mtx{A} - \hat{\mtx{A}}_{\rm out}$.
For a parameter $q$, draw a standard normal dimension
reduction map $\mtx{\Theta} \in \F^{q \times m}$.
Define the random variable
$$
\varphi_2^2 := \frac{1}{\beta q} \cdot \fnormsq{ \mtx{\Theta} \mtx{M} }.
$$
The field parameter $\beta = 1$ for $\F = \R$ and $\beta = 2$ for $\F = \C$.
This random variable can be regarded as a randomized estimator for the Schatten 2-norm
of the matrix $\mtx{M}$.  The goal of this section is to develop
probabilistic results to support this claim.

\begin{remark}[Prior Work]
The analysis here is similar in spirit to recent
papers on randomized trace
estimators~\cite{AT11:Randomized-Algorithms,RA15:Improved-Bounds,GT18:Improved-Bounds}.
The details here are slightly different, but
we claim no novelty of insight.
\end{remark}

\subsubsection{An Alternative Representation}

By the unitary invariance of the Schatten norm and the standard normal matrix,
we can and will assume that $\mtx{M} = \diag(\sigma_1, \dots, \sigma_m)
\in \R^{m \times m}$ is a real diagonal matrix with (weakly) decreasing entries.

Since $\mtx{M}$ is real and diagonal, the estimator can be written as
\begin{equation} \label{eqn:phi2-chisq}
\varphi_2^2 = \frac{1}{\beta q} \cdot \fnormsq{ \mtx{\Theta} \mtx{M} }
	= \frac{1}{\beta q} \cdot \sum\nolimits_{i=1}^m \sigma_i^2 \cdot \fnormsq{ \vct{\theta}_{:i} }
	\sim \frac{1}{\beta q} \cdot \sum\nolimits_{i=1}^m \sigma_i^2 \chi_i^2.
\end{equation}
Here, $\vct{\theta}_{:i}$ is the $i$th column of $\mtx{\Theta}$.
We have also introduced an independent family $\{ \chi_i^2 : i = 1, \dots, n \}$
of chi-squared random variables, each with $\beta q$ degrees of freedom.
The symbol $\sim$ denotes equality of distribution.

\subsubsection{The Mean and Variance}
\label{app:err2-mean}

Using the representation~\cref{eqn:phi2-chisq}, we quickly compute the mean and variance
of the estimator.  By linearity of expectation,
$$
\Expect \varphi_2^2 = \frac{1}{\beta q} \cdot \sum\nolimits_{i=1}^m \sigma_i^2 \Expect \chi_i^2
	= \sum\nolimits_{i=1}^m \sigma_i^2 = \fnormsq{ \mtx{M} }.
$$
For the second relation, we introduce the mean of a chi-squared variable with $\beta q$
degrees of freedom.  Since the chi-squared variables are independent,
$$
\Var[ \varphi_2^2 ] = \frac{1}{(\beta q)^2} \cdot \sum\nolimits_{i=1}^m \sigma_i^4 \Var[ \chi_i^2 ]
	= \frac{2}{\beta q} \sum\nolimits_{i=1}^m \sigma_i^4
	= \frac{2}{\beta q} \norm{ \mtx{M} }_4^4.
$$
We have also used the fact that the variance is 2-homogeneous,
and we introduced the variance of a chi-squared variable with $\beta q$ degrees of freedom.

\subsubsection{Upper Tail Probabilities}
\label{app:err2-upper}

Our goal is to develop bounds on the probability that the
estimator takes an extreme value.  We begin with the upper tail.

We can use the Laplace transform method.  For $\eps \geq 0$,
by Markov's inequality,
$$
\log \Prob{ \varphi_2^2 \geq (1+\eps) \cdot \fnormsq{\mtx{M}} }
	\leq \inf_{\eta > 0} \left( - \eta (1+\eps) \fnormsq{\mtx{M}} + \log \Expect \econst^{\eta \varphi_2^2} \right).
$$
To compute the moment generating function, we exploit independence of the chi-squared variates
in the representation~\eqref{eqn:phi2-chisq}:
$$
\log \Expect \econst^{\eta \varphi_2^2}
	= \prod\nolimits_{i=1}^m \log \Expect \econst^{(\eta \sigma_i^2/(\beta q)) \cdot \chi_i^2 }
	= \frac{-\beta q}{2} \sum\nolimits_{i=1}^m \log \left[ 1 - \frac{2\eta \sigma_i^2}{\beta q} \right].
$$
The last relation follows when we introduce the moment generating function of a chi-squared
variable with $\beta q$ degrees of freedom.  We tacitly assume that $\eta$ is sufficiently small.
We have the bound
$$
\log \Expect \econst^{\eta \varphi_2^2}
	\leq \frac{-\beta q}{2} \log \left[ 1 - \frac{2\eta \sum_{i=1}^m \sigma_i^2}{\beta q} \right]
	= \frac{-\beta q}{2} \log \left[ 1 - \frac{2\eta \fnormsq{\mtx{M}}}{\beta q} \right].
$$
This point follows by repeated application of the numerical inequality
$(1-a)(1-b) \geq 1 - a - b$, valid when $ab \geq 0$.  In summary,
$$
\begin{aligned}
\log \Prob{ \varphi_2^2 \geq (1+\eps) \cdot \fnormsq{\mtx{M}} }
	&\leq \inf_{\eta > 0} \left( - \eta(1+ \eps) \fnormsq{\mtx{M}} - \frac{\beta q}{2} \log\left[1 - \frac{2\eta \fnormsq{\mtx{M}}}{\beta q} \right] \right) \\
	&= \frac{-\beta q}{2} \left[ \eps - \log(1 + \eps) \right].
\end{aligned}
$$
Exponentiate this expression to reach the required bound.

\begin{remark}[Improvements]
Sharper estimates are possible in the case where
the stable rank of the matrix $\mtx{M}$ is large.
For results of this type, see~\cite{GT18:Improved-Bounds}.
\end{remark}

\subsubsection{Lower Tail Probabilities}
\label{app:err2-lower}

For the lower tail, we use essentially the same argument.
Therefore, we gloss over most of the details.

For $\eps \in (0, 1)$, the Laplace transform method gives
$$
\log \Prob{ \varphi_2^2 \leq (1 - \eps) \cdot \fnormsq{\mtx{M}} }
	\leq \inf_{\eta > 0} \left( \eta (1-\eps) \fnormsq{\mtx{M}} + \log \Expect \econst^{-\eta \varphi_2^2} \right).
$$
We bound the moment generating function as
$$
\log \Expect \econst^{-\eta \varphi_2^2}
	\leq \frac{-\beta q}{2} \log \left[ 1 + \frac{2\eta \fnormsq{\mtx{M}}}{\beta q} \right].
$$
Combine the last two displays:
$$
\begin{aligned}
\log \Prob{ \varphi_2^2 \leq \eps \cdot \fnormsq{\mtx{M}} }
	&\leq \inf_{\eta > 0} \left( \eta (1- \eps) \fnormsq{\mtx{M}} - \frac{\beta q}{2} \log\left[1 + \frac{2\eta \fnormsq{\mtx{M}}}{\beta q} \right] \right) \\
	&= \frac{\beta q}{2} \left[ \eps + \log (1- \eps) \right].
\end{aligned}
$$
Exponentiate this expression to reach the desired bound.

\subsection{Diagnosing Spectral Decay}
\label{app:diagnose-spectrum}

In this section, we explain why the square root of the 
tail energy is a Lipschitz function.
For a matrix $\mtx{A} \in \F^{m \times n}$ and an integer $r \geq 0$,
recall that
$$
\tau_{r+1}^2(\mtx{A}) = \sum\nolimits_{j > r} \sigma_j^2(\mtx{A})
	= \sum\nolimits_{j > r} \lambda_j( \mtx{A}^* \mtx{A} ).
$$
As usual, $\lambda_j$ returns the $j$th largest eigenvalue of an Hermitian matrix.
Ky Fan's minimum principle~\cite[Prob.~I.6.15]{Bha97:Matrix-Analysis} gives a variational
representation for this quantity:
$$
\tau_{r+1}^2(\mtx{A}) = \min_{\mtx{U} \in \F^{n \times (n-r)}} \trace[ \mtx{U}^* (\mtx{A}^* \mtx{A}) \mtx{U} ]
	= \min_{\mtx{U} \in \F^{n \times (n-r)}} \fnormsq{ \mtx{AU} }
$$
where $\mtx{U}$ ranges over matrices with orthonormal columns.
As a consequence, for conformal matrices $\mtx{A}$ and $\mtx{B}$,
we have
$$
\begin{aligned}
\tau_{r+1}(\mtx{A}) - \tau_{r+1}(\mtx{B})
	&= \min\nolimits_{\mtx{U}} \fnorm{ \mtx{AU} } - \min\nolimits_{\mtx{U}} \fnorm{ \mtx{BU} } \\
	&\leq \fnorm{ \mtx{A} \mtx{U}_{\mtx{B}} } - \fnorm{ \mtx{B} \mtx{U}_{\mtx{B}} } \\
	&= \fnorm{ (\mtx{A} - \mtx{B}) \mtx{U}_{\mtx{B}} }
	\leq \fnorm{ \mtx{A} - \mtx{B} }.
\end{aligned}
$$
We have written $\mtx{U}_{\mtx{B}}$ for the orthonormal matrix in $\F^{n \times (n-r)}$
that minimizes
the functional $\mtx{U} \mapsto \fnorm{ \mtx{BU} }$.  The last inequality
follows because $\mtx{U}_{\mtx{B}}$ has spectral norm one.  Reverse the
roles of the two matrices to conclude that
$$
\abs{ \tau_{r+1}(\mtx{A}) - \tau_{r+1}(\mtx{B}) }
	\leq \fnorm{ \mtx{A} - \mtx{B} }.
$$
This is the advertised result.

\section{Code \& Pseudocode}
\label{app:pseudocode}

This section contains pseudocode for the dimension reduction maps
described in \cref{sec:dim-red-maps}.  We use the same mathematical
notation as the rest of the paper.  We also rely on \textsc{Matlab R2018b} commands,
which appear in typewriter font.  The electronic materials include a \textsc{Matlab}
implementation of these methods.

\vspace{0.5pc}

\begin{itemize}

\item	The template for the \textsc{DimRedux} class appears in the body of
the paper as~\cref{alg:dim-redux}.

\item	\cref{alg:dim-redux-gauss} defines a Gaussian dimension reduction class (\textsc{GaussDR}),
which is a subclass of \textsc{DimRedux}.  It describes the constructor and the left and right action
of this dimension reduction map.  See~\cref{sec:gauss} for the explanation.

\item	\cref{alg:dim-redux-ssrft} defines a SSRFT dimension reduction class (\textsc{SSRFT}).
It is a subclass of \textsc{DimRedux}.  It describes the constructor and the left and right action
of this dimension reduction map.  See~\cref{sec:ssrft} for the explanation.

\item	\cref{alg:dim-redux-sparse} defines a sparse dimension reduction class (\textsc{SparseDR}),
which is a subclass of \textsc{DimRedux}.  It describes the constructor and the left and right action
of this dimension reduction map.  See~\cref{sec:sparse} for the explanation.

\end{itemize}

\vspace{1pc}

\begin{algorithm}[t]
  \caption{\textsl{Gaussian Dimension Reduction Map.}
  (\cref{sec:gauss})  %
  \label{alg:dim-redux-gauss}}
  \begin{algorithmic}[1]
\vspace{0.25pc}

	\State \textbf{class} \textsc{GaussDR} (\textsc{DimRedux})
		\Comment Subclass of \textsc{DimRedux}
  \Indent
	\State \textbf{local variable} $\mtx{\Xi}$ (dense matrix)

	\Function{randn}{$d, N; \F$}
		\Comment{Gaussian matrix over field $\F$}
		\If{$\F = \R$} \Return $\texttt{randn}(d, N)$
		\EndIf
		\If{$\F = \C$} \Return $\texttt{randn}(d, N) + \texttt{1i * } \texttt{randn}(d,N)$
		\EndIf
	\EndFunction

	\Function{GaussDR}{$k, N$}
		\Comment{Constructor}
		\State $\mtx{\Xi} \gets \textsc{randn}(d, N; \F)$
			\Comment{Gaussian over $\F$}
	\EndFunction

	\Function{GaussDR.mtimes}{\texttt{DRmap}, $\mtx{M}$}
	\State \Return \texttt{mtimes}($\mtx{\Xi}$, $\mtx{M}$)
	\EndFunction

\EndIndent

	\vspace{0.25pc}

\end{algorithmic}
\end{algorithm}

\begin{algorithm}[t]
  \caption{\textsl{SSRFT Dimension Reduction Map.}
  (\cref{sec:ssrft})  %
  \label{alg:dim-redux-ssrft}}
  \begin{algorithmic}[1]
\vspace{0.25pc}

	\State \textbf{class} \textsc{SSRFT} (\textsc{DimRedux})
		\Comment Subclass of \textsc{DimRedux}

	\State \textbf{local variables}	\texttt{coords}, $\texttt{perm}_j$, $\vct{\eps}_j$ for $j = 1, 2$
	\Function{SSRFT}{$d, N$}
		\Comment{Constructor}
		\State $\texttt{coords} \gets \texttt{randperm}(N, d)$
		\State $\texttt{perm}_j \gets \texttt{randperm}(N)$ for $j = 1, 2$
		\State $\vct{\eps}_j \gets \texttt{sign}(\textsc{randn}(N,1; \F))$ for $j = 1, 2$
	\EndFunction

	\Function{SSRFT.mtimes}{\texttt{DRmap}, $\mtx{M}$}
		\If{$\F = \R$}
		\State	$\mtx{M} \gets \texttt{dct}(\diag(\vct{\eps}_1) \mtx{M}(\texttt{perm}_1, \texttt{:}))$
		\State	$\mtx{M} \gets \texttt{dct}(\diag(\vct{\eps}_2) \mtx{M}(\texttt{perm}_2, \texttt{:}))$
		\EndIf
		\If{$\F = \C$}
		\State	$\mtx{M} \gets \texttt{dft}(\diag(\vct{\eps}_1) \mtx{M}(\texttt{perm}_1, \texttt{:}))$
		\State	$\mtx{M} \gets \texttt{dft}(\diag(\vct{\eps}_2) \mtx{M}(\texttt{perm}_2, \texttt{:}))$
		\EndIf
		\State \Return $\mtx{M}(\texttt{coords}, \texttt{:})$
	\EndFunction

	\vspace{0.25pc}

\end{algorithmic}
\end{algorithm}

\begin{algorithm}[t]
  \caption{\textsl{Sparse Dimension Reduction Map.}
  (\cref{sec:sparse})  %
  \label{alg:dim-redux-sparse}}
  \begin{algorithmic}[1]
\vspace{0.25pc}

	\State \textbf{class} \textsc{SparseDR} (\textsc{DimRedux})
		\Comment Subclass of \textsc{DimRedux}
	\State \textbf{local variable} $\mtx{\Xi}$ (sparse matrix)

	\Function{SparseDR}{$d, N$}
		\Comment{Constructor}
		\State $\zeta \gets \min\{ d, 8 \}$
			\Comment{Sparsity of each column}
			\For{$j = 1, \dots, N$}
			\State	$\mtx{\Xi}(\texttt{randperm}(d,\zeta), j) \gets \texttt{sign}( \textsc{randn}(\zeta,1; \F) )$
		\EndFor
	\EndFunction

	\Function{SparseDR.mtimes}{\texttt{DRmap}, $\mtx{M}$}
	\State \Return \texttt{mtimes}($\mtx{\Xi}$, $\mtx{M}$)
	\EndFunction

	\vspace{0.25pc}

\end{algorithmic}
\end{algorithm}

\section{Supplemental Numerical Results}
\label{app:numerics}

This section summarizes the additional numerical results
that are presented in this supplement.  The \textsc{Matlab}
code in the electronic materials can reproduce these experiments.

\subsection{Alternative Sketching and Reconstruction Methods}
\label{app:alt-sketch}

In this section, we give full mathematical descriptions of
other sketching and reconstruction methods from the literature.
We compare our approach against these algorithms.

\subsubsection{The [HMT11] Method}
\label{app:hmt}

The paper~\cite[Sec.~5.4]{HMT11:Finding-Structure} describes
a one-pass SVD algorithm, which can be reinterpreted as
a sketching algorithm for low-rank matrix approximation.
This method simplifies a more involved approach~\cite[Sec.~5.2]{WLRT08:Fast-Randomized}
due to Woolfe et al.
The two approaches have similar performance in practice.

This method uses two dimension reduction maps,
controlled by one parameter $k$:
$$
\mtx{\Upsilon} \in \F^{k \times m}
\quad\text{and}\quad
\mtx{\Omega} \in \R^{k \times n}.
$$
The sketch takes the form
$$
\mtx{X} = \mtx{\Upsilon} \mtx{A}
\quad\text{and}\quad
\mtx{Y} = \mtx{A} \mtx{\Omega}.
$$
To obtain a rank-$r$ approximation from the sketch,
we first compute $r$ leading singular vectors of
the sketch matrices:
$$
\begin{aligned}
(\mtx{P}, \sim, \sim) &= \texttt{svd}(\mtx{X}^*, \texttt{'econ'})
&&\quad\text{and}\quad &
\mtx{P} &= \mtx{P}( \texttt{:}, \texttt{1:r} ); \\
(\mtx{Q}, \sim, \sim) &= \texttt{svd}(\mtx{Y}, \texttt{'econ'})
&&\quad\text{and}\quad &
\mtx{Q} &= \mtx{Q}( \texttt{:}, \texttt{1:r} ).
\end{aligned}
$$
Next, we compute two separate estimates for the core matrix
by solving two families of least-squares problems:
$$
\mtx{C}_1 = (\mtx{Q}^* \mtx{Y}) (\mtx{P}^* \mtx{\Omega})^\dagger \in \F^{r \times r}
\quad\text{and}\quad
\mtx{C}_2^* = (\mtx{P}^* \mtx{X}) (\mtx{Q}^* \mtx{\Upsilon})^\dagger \in \F^{r \times r}.
$$
Combine these two estimates and compute the SVD:
$$
(\mtx{U}, \mtx{\Sigma}, \mtx{V}) = \texttt{svd}( (\mtx{C}_1 + \mtx{C}_2)/2 ).
$$
Last, we obtain the rank-$r$ approximation in factored form:
$$
\hat{\mtx{A}}_{\mathrm{hmt}}
	:= (\mtx{QU}) \mtx{\Sigma} (\mtx{PV})^*.
$$
This approach is not competitive with more modern techniques.
Some of the deficiencies stem from truncating the singular vectors to rank $r$ at the first
step of the procedure; see~\cref{fig:flow-field-svec-hmt5,fig:flow-field-svec-hmt}.

\subsubsection{The [TYUC17] Method}

In our previous paper, we developed and analyzed a
sketching algorithm~\cite[Alg.~7]{TYUC17:Practical-Sketching}
for low-rank matrix approximation.  Our work contains
a detailed theoretical analysis, prescriptions for choosing
algorithm parameters, and an extensive numerical evaluation.
We later discovered that this method is algebraically (but not numerically)
equivalent to a proposal of Clarkson \& Woodruff~\cite[Thm.~4.9]{CW09:Numerical-Linear}.
The paper~\cite{CW09:Numerical-Linear} also lacks reliable instructions for implementation.

This approach uses two dimension reduction maps that are indexed
by two parameters $k, \ell$:
$$
\mtx{\Upsilon} \in \F^{\ell \times m}
\quad\text{and}\quad
\mtx{\Omega} \in \F^{k \times n}
\quad\text{where $k \leq \ell$.}
$$
The sketch takes the form
$$
\mtx{X} = \mtx{\Upsilon A}
\quad\text{and}\quad
\mtx{Y} = \mtx{A \Omega}^*.
$$
To obtain a rank-$r$ approximation from the sketch, we compute
a thin orthogonal--triangular decomposition:
\begin{equation*} %
\mtx{Y} =: \mtx{QR} \quad\text{where $\mtx{Q} \in \F^{m \times k}$.}
\end{equation*}
Then we form the approximation:
\begin{equation} \label{eqn:tyuc2017}
\hat{\mtx{A}}_{\mathrm{tyuc}} := \mtx{Q} \lowrank{ (\mtx{\Upsilon} \mtx{Q})^{\dagger} \mtx{X} }{r}.
\end{equation}
Of course, we solve the least-squares problems, rather than computing
and applying the pseudoinverse.  We use a dense SVD or a randomized SVD~\cite{HMT11:Finding-Structure}
to calculate the best rank-$r$ approximation.

This method works well, but it uses more storage than necessary
because $\ell$ needs to be somewhat larger than $k$.  The algorithm
can also be sensitive to the relative size of the parameters $k, \ell$.

\subsubsection{The [Upa16] Method}
\label{app:upa}

In a paper on privacy-preserving matrix approximation,
Upadhyay~\cite[Sec.~3]{Upa16:Fast-Space-Optimal} developed
an algorithm that also serves for streaming low-rank matrix approximation.
This method simplifies a far more complicated approach due
to Boutsidis et al.~\cite[Sec.~6]{BWZ16:Optimal-Principal-STOC}.

Upadhyay proposed the sketch~\cref{eqn:test-matrices,eqn:range-sketch,eqn:core-sketch},
which depends on two parameters $k, s$.  We are building on his idea in this paper.
In contrast to our work, Upadhyay designs a rank-$r$ reconstruction algorithm
using the ``sketch-and-solve'' framework; see \cref{sec:provenance}.

His approach leads to the following algorithm.
First, compute orthonormal bases $\mtx{Q}$ and $\mtx{P}$
for the range and co-range:
$$
\begin{aligned}
\mtx{X}^* &=: \mtx{PR}_1 \quad\text{where}\quad \mtx{P} \in \F^{n \times k}; \\
\mtx{Y} &=: \mtx{QR}_2 \quad\text{where}\quad \mtx{Q} \in \F^{m \times k}.
\end{aligned}
$$
Next, form thin singular value decompositions:
$$
\mtx{\Phi Q} = \mtx{U}_1\mtx{\Sigma}_1 \mtx{V}_1^* \in \F^{s \times k}
\quad\text{and}\quad
\mtx{\Psi P} = \mtx{U}_2 \mtx{\Sigma}_2 \mtx{V}_2^* \in \F^{s \times k}.
$$
Construct the rank-$r$ approximation using the formula
\begin{equation} \label{eqn:upa}
\hat{\mtx{A}}_{\mathrm{upa}} := \mtx{Q} \mtx{V}_1 \mtx{\Sigma}_1^{\dagger} \,
	\lowrank{ \mtx{U}_1^* \mtx{Z} \mtx{U}_2 }{r} \,
	\mtx{\Sigma}_2^{\dagger} \mtx{V}_2^* \mtx{P}^*.
\end{equation}
We use a truncated SVD to perform the rank truncation of the central
matrix.  Of course, we should take care in applying the pseudoinverses.

Superficially, the approximation $\hat{\mtx{A}}_{\mathrm{upa}}$
may appear similar to the approximation we developed in~\cref{eqn:Ahat-fixed}.
Nevertheless, they are designed using different principles,
and their performance is quite different in practice.
The [Upa16] method cannot achieve high relative accuracy,
even for matrices with rapid spectral decay.  Furthermore,
it has the bizarre feature that decreasing the rank
parameter $r$ can actually make the approximation less reliable!
See \cref{fig:flow-field-svec-upa5,fig:flow-field-svec-upa}.

\subsection{Spectra of Input Matrices}

\Cref{fig:spectra} plots the spectrum of each of the synthetic
and application matrices that we use in our experiments.

\begin{figure}[t!]
\begin{center}
\begin{subfigure}{.45\textwidth}
\begin{center}
\includegraphics[height=1.5in]{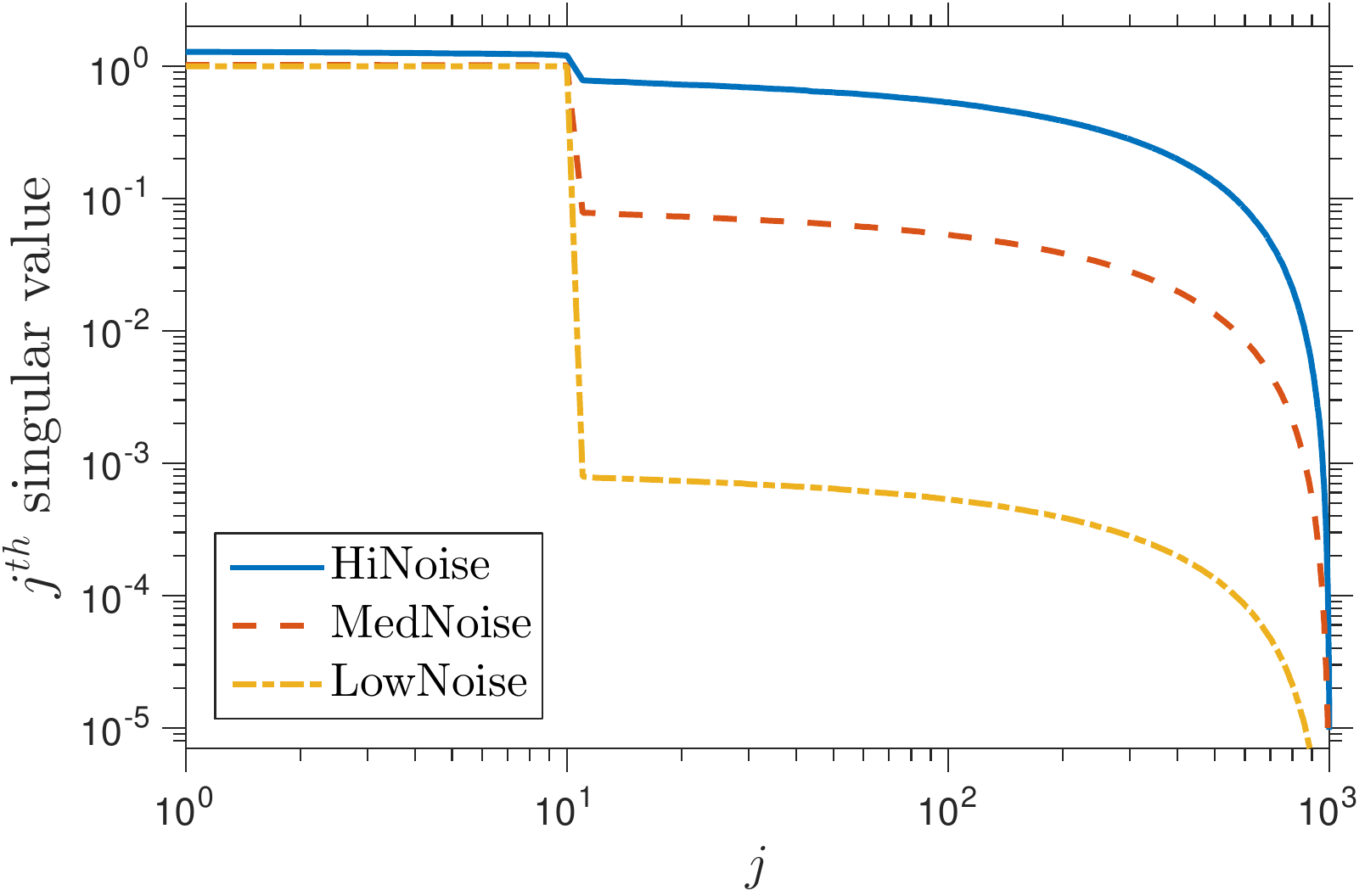}
\caption{Low-Rank + Noise}
\end{center}
\end{subfigure}
\begin{subfigure}{.45\textwidth}
\begin{center}
\includegraphics[height=1.5in]{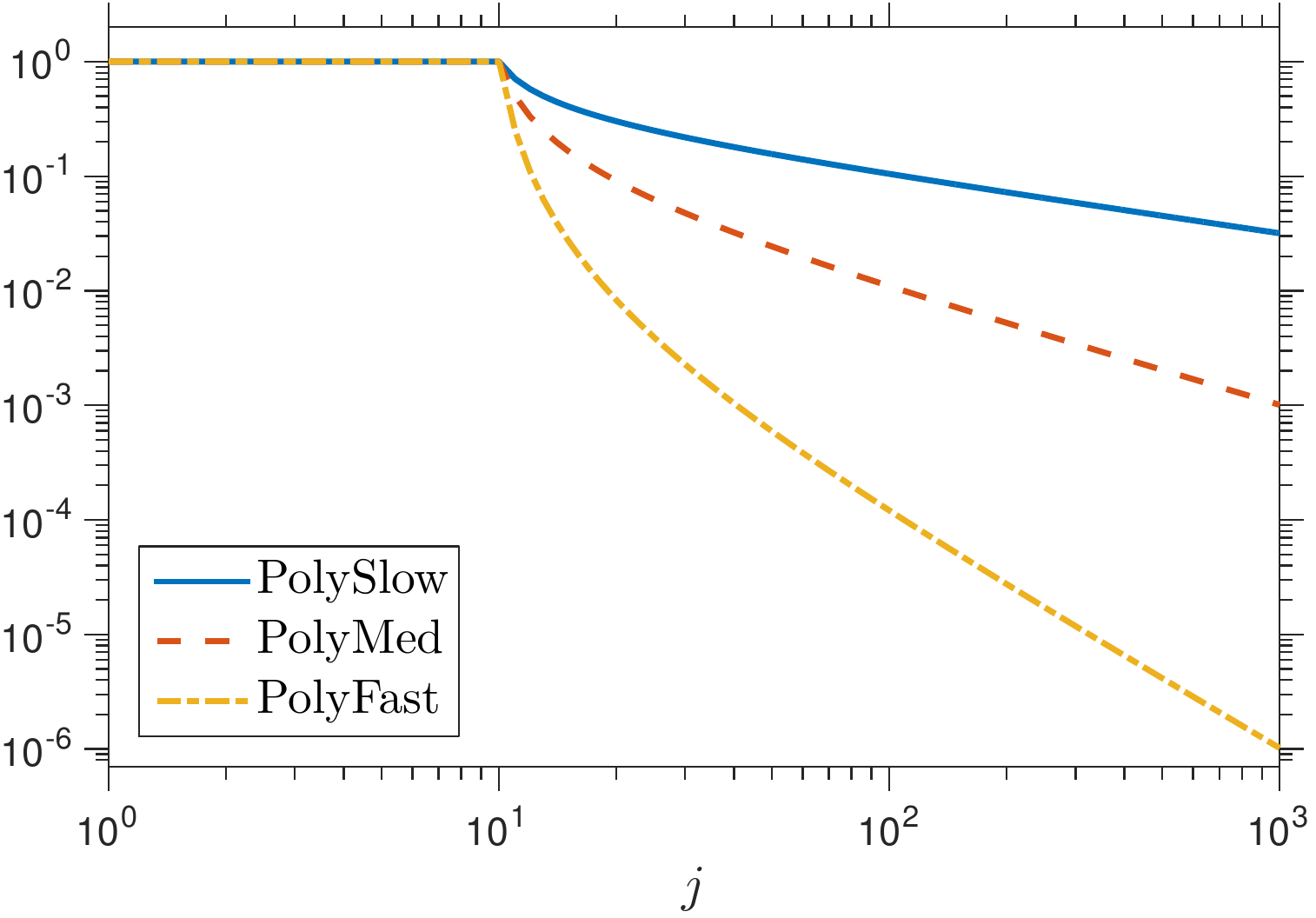}
\caption{Polynomial Decay}
\end{center}
\end{subfigure}

\begin{subfigure}{.45\textwidth}
\begin{center}
\includegraphics[height=1.5in]{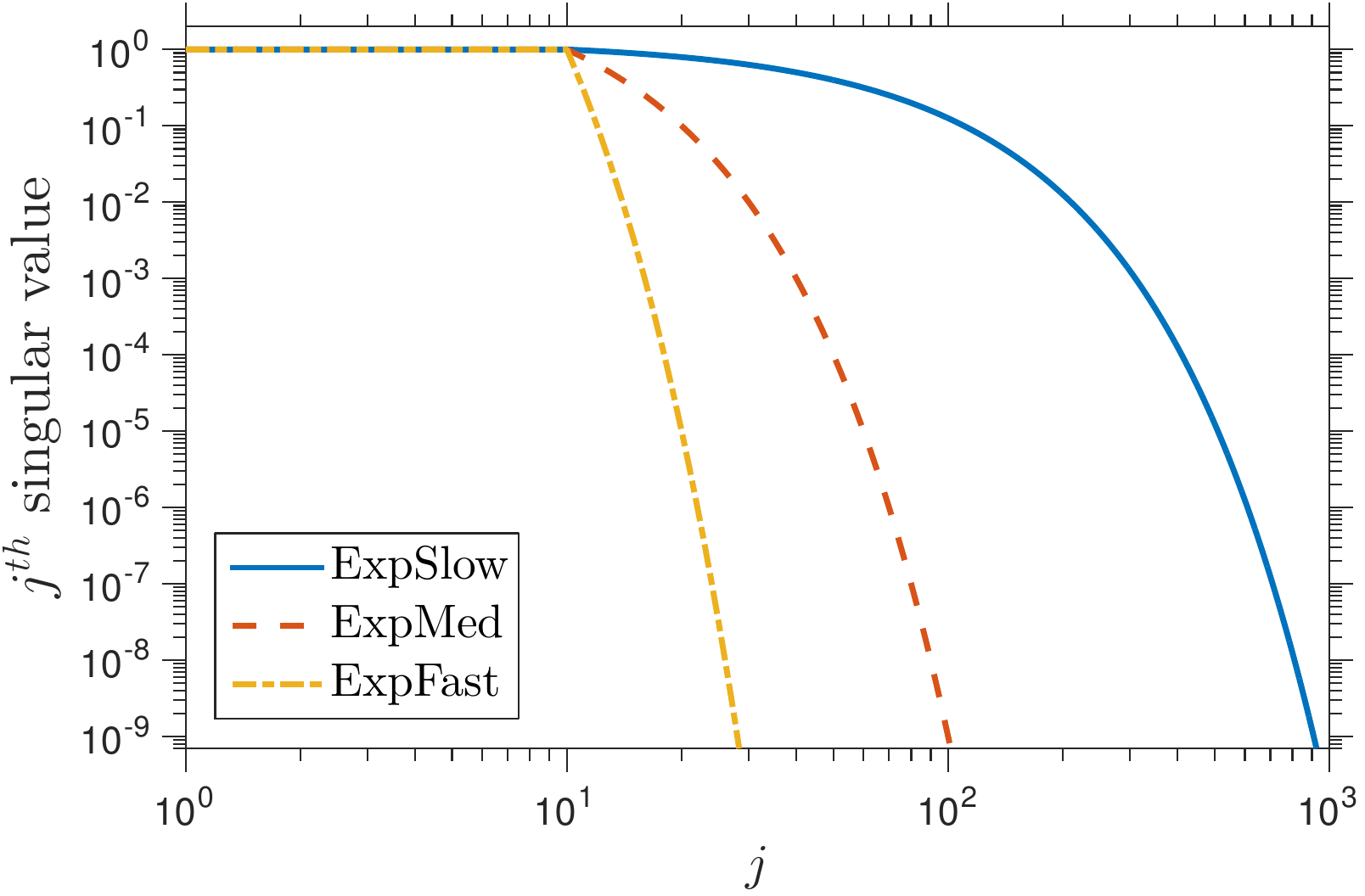}
\caption{Exponential Decay}
\end{center}
\end{subfigure}
\begin{subfigure}{.45\textwidth}
\begin{center}
\includegraphics[height=1.5in]{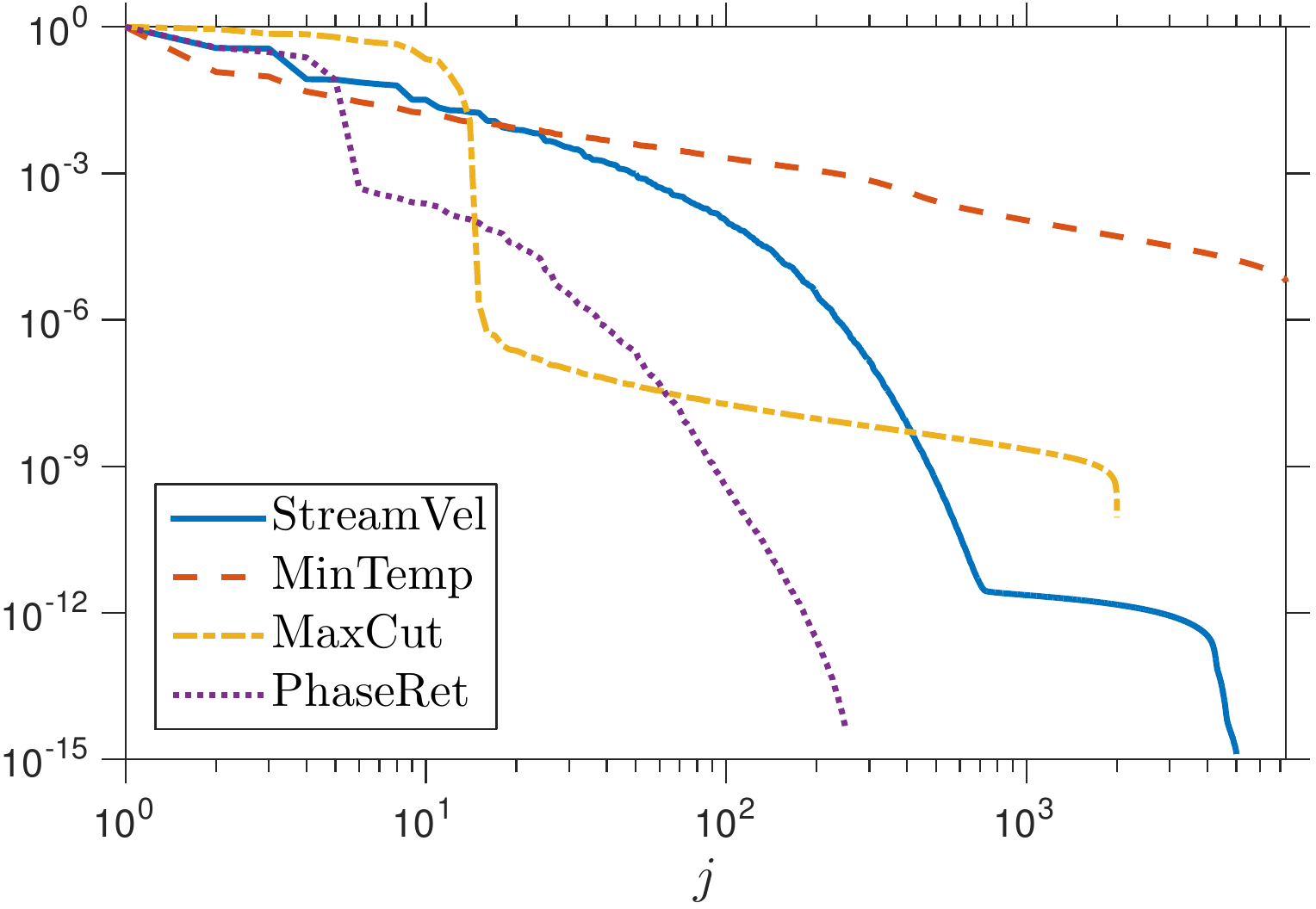}
\caption{Data Matrices}
\end{center}
\end{subfigure}
\caption{\textbf{Spectra of input matrices.} Plots of the singular value spectrum for an input
matrix from each of the synthetic classes (\texttt{LowRank}, \texttt{PolyDecay}, \texttt{ExpDecay} with effective rank $R = 10$) and from each of the real data classes (\texttt{MinTemp}, \texttt{StreamVel}, \texttt{MaxCut}, \texttt{PhaseRetrieval})
described in \cref{sec:input-matrix-examples}.}
\label{fig:spectra}
\end{center}
\end{figure}

\subsection{Insensitivity to the Dimension Reduction Map}
\label{app:universality}

Our first experiment is designed to show that the proposed rank-$r$ reconstruction
formula~\cref{eqn:Ahat-fixed} %
is insensitive to the distribution of the dimension reduction map at the oracle parameter values
(\cref{sec:oracle-error}) for synthetic input matrices.

We plot the oracle error for~\cref{eqn:Ahat-fixed}
as a function of storage budget $T$ for Gaussian, SSRFT,
and sparse dimension reduction maps.
See \cref{fig:universality-R5-S2,fig:universality-R5-Sinf,fig:universality-R10-S2,%
fig:universality-R10-Sinf,fig:universality-R20-S2,fig:universality-R20-Sinf}.
The curves are almost identical, except that the unitary SSRFT map performs slightly
\emph{better} than the others when the storage budget is very large.
Similar results hold for matrices drawn from real applications.

We have also found that the other reconstruction methods
[HMT11], [TYUC17], and [Upa16]
are insensitive to the choice of dimension reduction map.
These observations justify the transfer of theoretical and empirical results
for Gaussians to SSRFT and sparse dimension reduction maps.

\subsection{Achieving the Oracle Performance}
\label{app:oracle-performance}

Next, we show that we can almost achieve the oracle error
by implementing~\cref{eqn:Ahat-fixed} with sketch size
parameters chosen using our theory.

We perform the following experiment.
For synthetic input matrices, we compare the oracle performance (\cref{sec:oracle-error}) of our rank-$r$
approximation~\cref{eqn:Ahat-fixed} %
with its performance at the theoretical parameters proposed in~\cref{sec:sketch-size-theory}.
(In the formula~\cref{eqn:ks-flat} for a flat spectrum, we set the tail location $\hat{\varrho} = r$.)
We use Gaussian dimension reduction maps, but equivalent results hold for other types of dimension reduction maps.
Plots of the results appear in \cref{fig:theory-params-R5-S2,fig:theory-params-R5-Sinf,%
fig:theory-params-R10-S2,fig:theory-params-R10-Sinf,%
fig:theory-params-R20-S2,fig:theory-params-R20-Sinf}.

For most of the examples, the general parameter choice~\cref{eqn:ks-natural}
is able to deliver a relative error that tracks the oracle error closely.
The parameter choice~\cref{eqn:ks-flat} for a flat spectrum works somewhat
better for matrices whose spectral tail exhibits slow decay (\texttt{LowRankLowNoise}, \texttt{LowRankMedNoise}, \texttt{LowRankHiNoise}).  We also learn that the theoretical formulas are not entirely reliable
when the storage budget is very small.
Matrices with a lot of tail energy (\texttt{LowRankHiNoise}, \texttt{PolyDecaySlow})
are very hard to approximate accurately with a sketching algorithm,
so it is not surprising that our theoretical parameter choices fall short
of the oracle parameters in these cases.

\subsection{Algorithm Comparisons for Synthetic Instances}
\label{app:numerics-comparison-synth}

We compared all four of the reconstruction formulas
at the oracle parameters for a wide range of synthetic problem instances.
See~\cref{sec:alg-comparison} for details.

\Cref{fig:oracle-comparison-R5-S2,fig:oracle-comparison-R5-Sinf,%
fig:oracle-comparison-R20-S2,fig:oracle-comparison-R20-Sinf} contain the
results for matrices with effective rank $R = 5$ and $R = 20$
with relative error measured in Schatten $2$-norm and Schatten $\infty$-norm.

\subsection{Algorithm Comparisons for Real Data Instances}
\label{app:numerics-comparison-data}

In this experiment, we compared all four of the reconstruction
formulas at the oracle parameters and at theoretically chosen parameters
for several application examples.

Here are the details of the
\emph{a priori} parameter selections for the several methods.
For the proposed method~\cref{eqn:Ahat-fixed}, we use the ``natural''
parameter choice~\cref{eqn:ks-natural} that follows from our theoretical analysis.
The [Upa16] algorithm uses the same sketch---but lacks a comparable theory---%
so we instantiate it with the parameters~\cref{eqn:ks-natural}.
For [TYUC17], we assume that the
input matrix $\mtx{A} \in \F^{m \times n}$
is tall ($m \geq n$), and we use the theoretically motivated parameter values
$$
k = \max\{ r + \alpha + 1, \lfloor (T - n\alpha) / (m + 2n) \rfloor \}
\quad\text{and}\quad
\ell = \lfloor (T - km) / n\rfloor.
$$
This choice adapts the arguments in~\cite[Sec.~4.5.2]{TYUC17:Practical-Sketching}
to use the current definition of the storage budget $T$.
The [HMT11] algorithm does not have any free parameters.

\subsection{Flow-Field Reconstruction}
\label{app:flow-field-reconstruction}

\cref{fig:flow-field-time} illustrates the streamwise velocity field \texttt{StreamVel}
and its rank-$10$ approximation via~\cref{eqn:Ahat-fixed} using storage budget $T = 48(m+n)$
and the parameter choices~\cref{eqn:ks-natural}.  We see that the approximation
captures the large-scale features of the flow, although there are small errors
visible for the higher-order singular vectors.

We also performed the same experiment with the algorithms [HMT11], [Upa16],
and [TYUC17].  We set the truncation rank $r = 5$ and $r = 10$ to see
whether this change affects the behavior of the methods.
We plot the leading left singular vectors of the flow fields
in \cref{fig:flow-field-svec-hmt5,fig:flow-field-svec-hmt,%
fig:flow-field-svec-upa5,fig:flow-field-svec-upa,fig:flow-field-svec-tyuc}.
For truncation $r = 10$, all of the algorithms produce reasonable results.
Nevertheless, with algorithms [HMT11], [Upa16], and [TYUC17],
the singular vector estimates for rank $6, 7, 8, 9$ start to deviate
from the singular vectors of the original matrix.

When we change the truncation rank to $r = 5$,
our methods [TYUC17] and~\cref{eqn:Ahat-fixed} give exactly the same singular
vector estimates as with $r = 10$, by construction of the algorithm.
On the other hand, the methods [HMT11] and [Upa16] behave far worse when $r = 5$ than when $r = 10$.
This feature is both strange and dissatisfying.  By itself, this lack of stability
is already enough to disqualify the algorithms [HMT11] and [Upa16]
from practical use.

\section*{Acknowledgments}

The authors wish to thank Beverley McKeon and Sean Symon for providing the Navier--Stokes
simulation data and visualization software.  William North contributed the weather data.
The \texttt{NOAA\_OI\_SST\_V2} high-resolution SST data
are provided by NOAA/OAR/ESRL PSD, Boulder, Colorado, USA.
We are also grateful to the anonymous reviewers for a careful reading and thoughtful comments
that helped us to improve the manuscript.

\bibliographystyle{siamplain}

\newpage

\begin{figure}[htp!]
\vspace{0.5in}
\begin{center}
\begin{subfigure}{.325\textwidth}
\begin{center}
\includegraphics[height=1.5in]{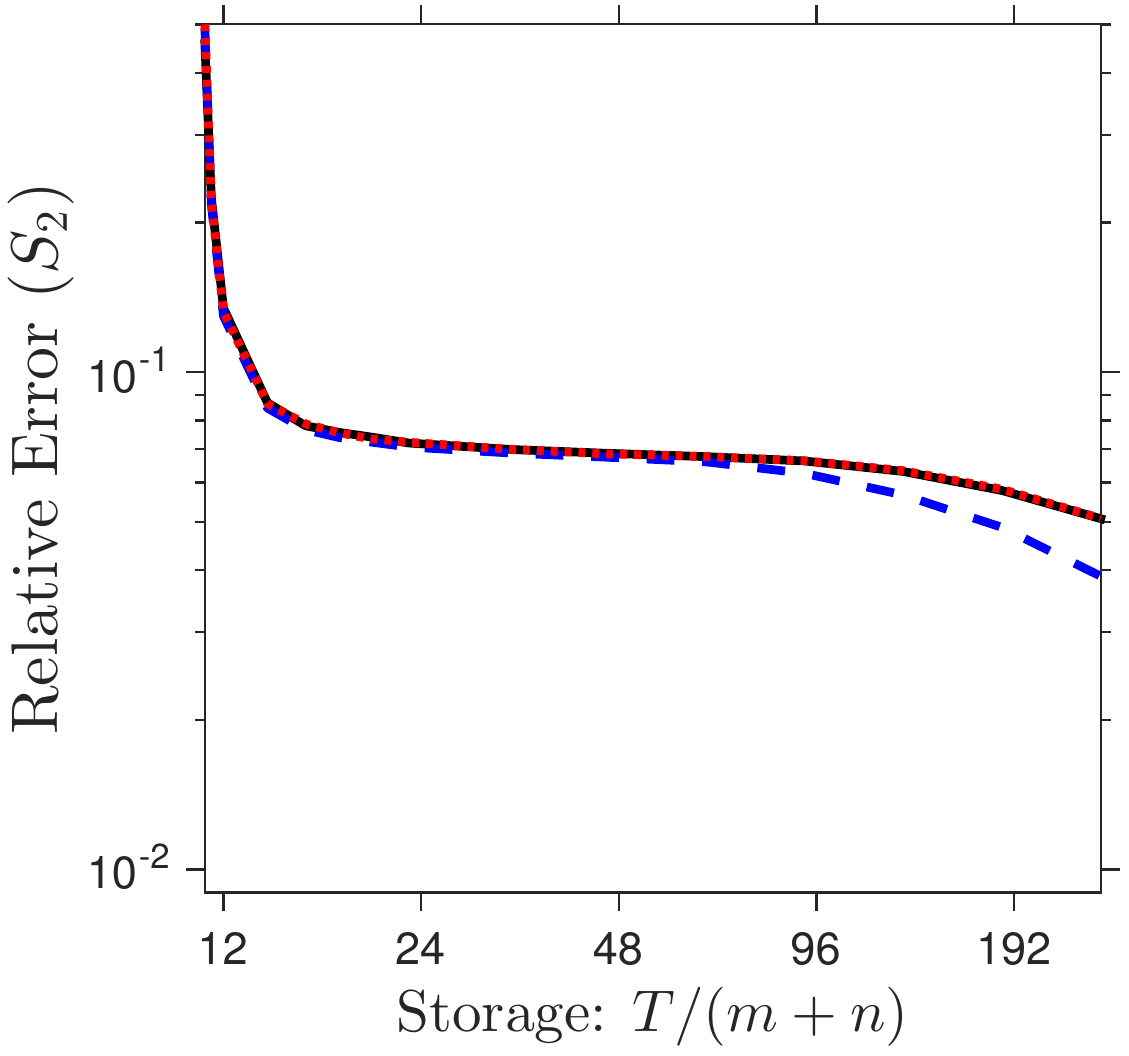}
\caption{\texttt{LowRankHiNoise}}
\end{center}
\end{subfigure}
\begin{subfigure}{.325\textwidth}
\begin{center}
\includegraphics[height=1.5in]{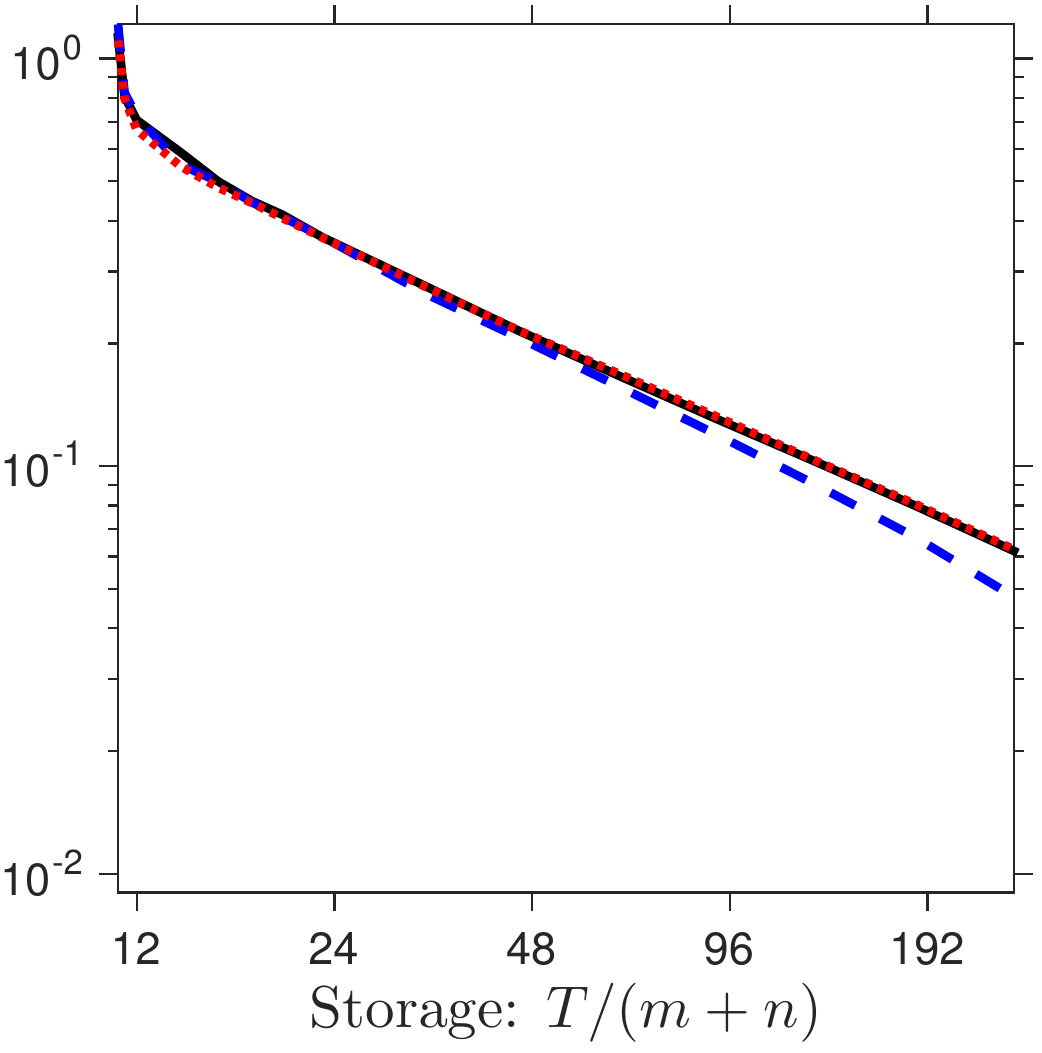}
\caption{\texttt{LowRankMedNoise}}
\end{center}
\end{subfigure}
\begin{subfigure}{.325\textwidth}
\begin{center}
\includegraphics[height=1.5in]{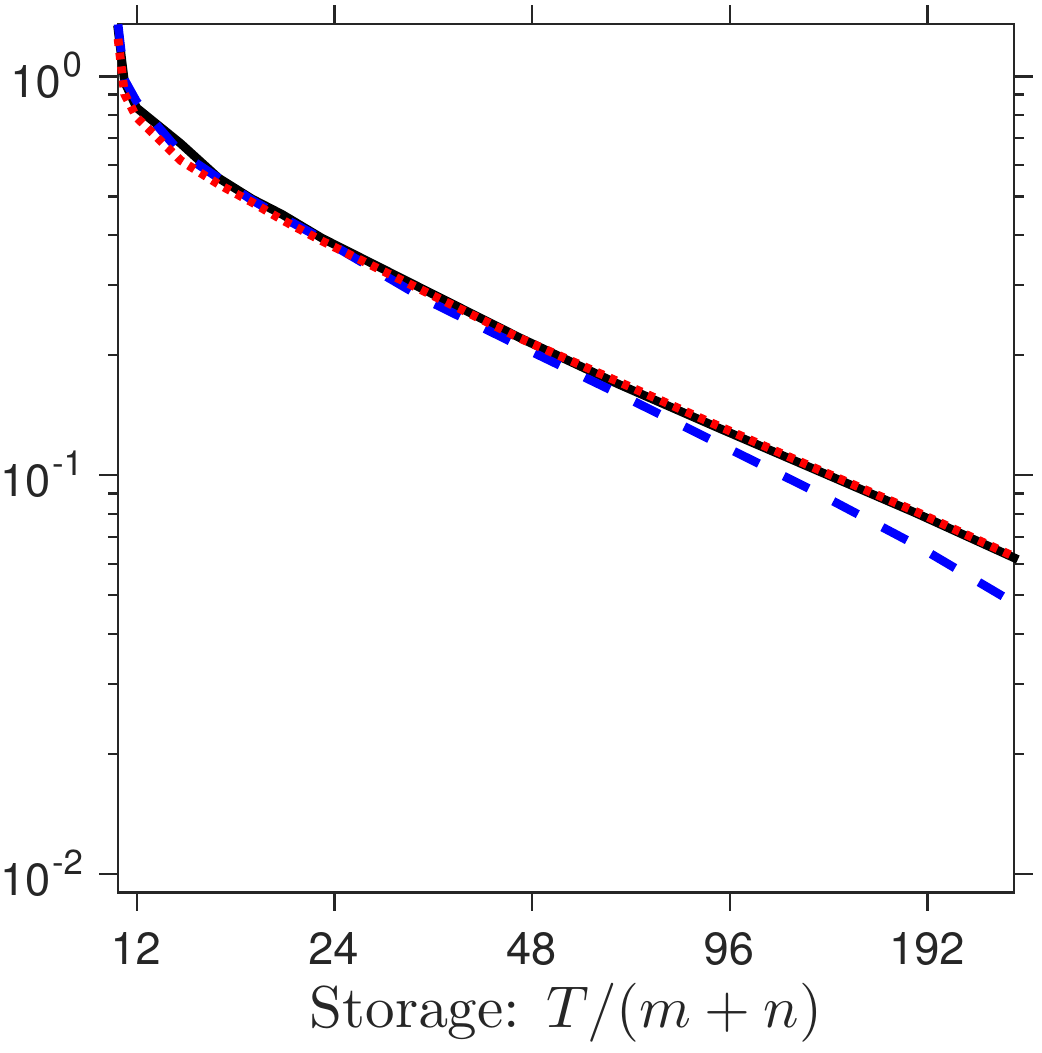}
\caption{\texttt{LowRankLowNoise}}
\end{center}
\end{subfigure}
\end{center}

\vspace{.5em}

\begin{center}
\begin{subfigure}{.325\textwidth}
\begin{center}
\includegraphics[height=1.5in]{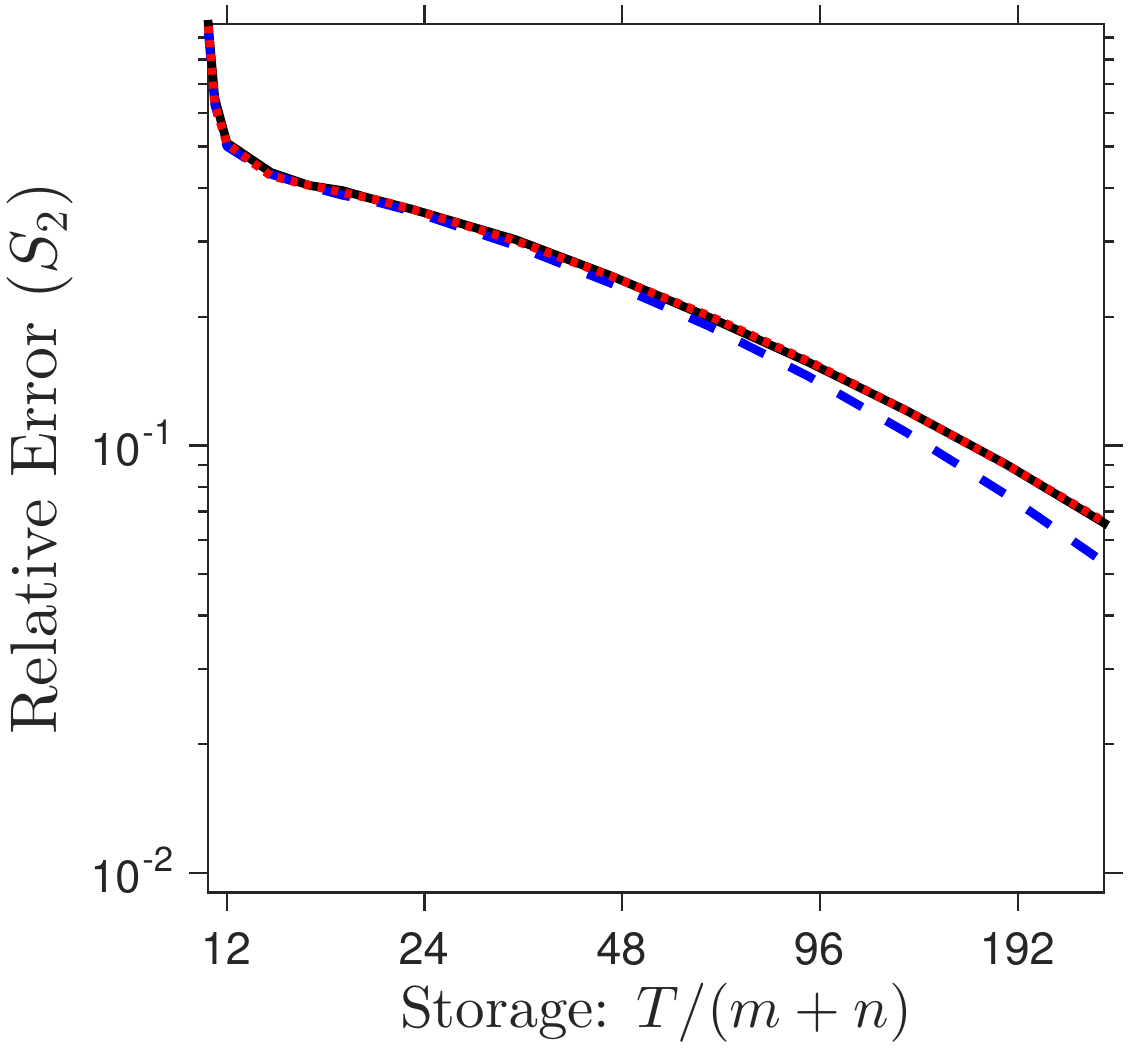}
\caption{\texttt{PolyDecaySlow}}
\end{center}
\end{subfigure}
\begin{subfigure}{.325\textwidth}
\begin{center}
\includegraphics[height=1.5in]{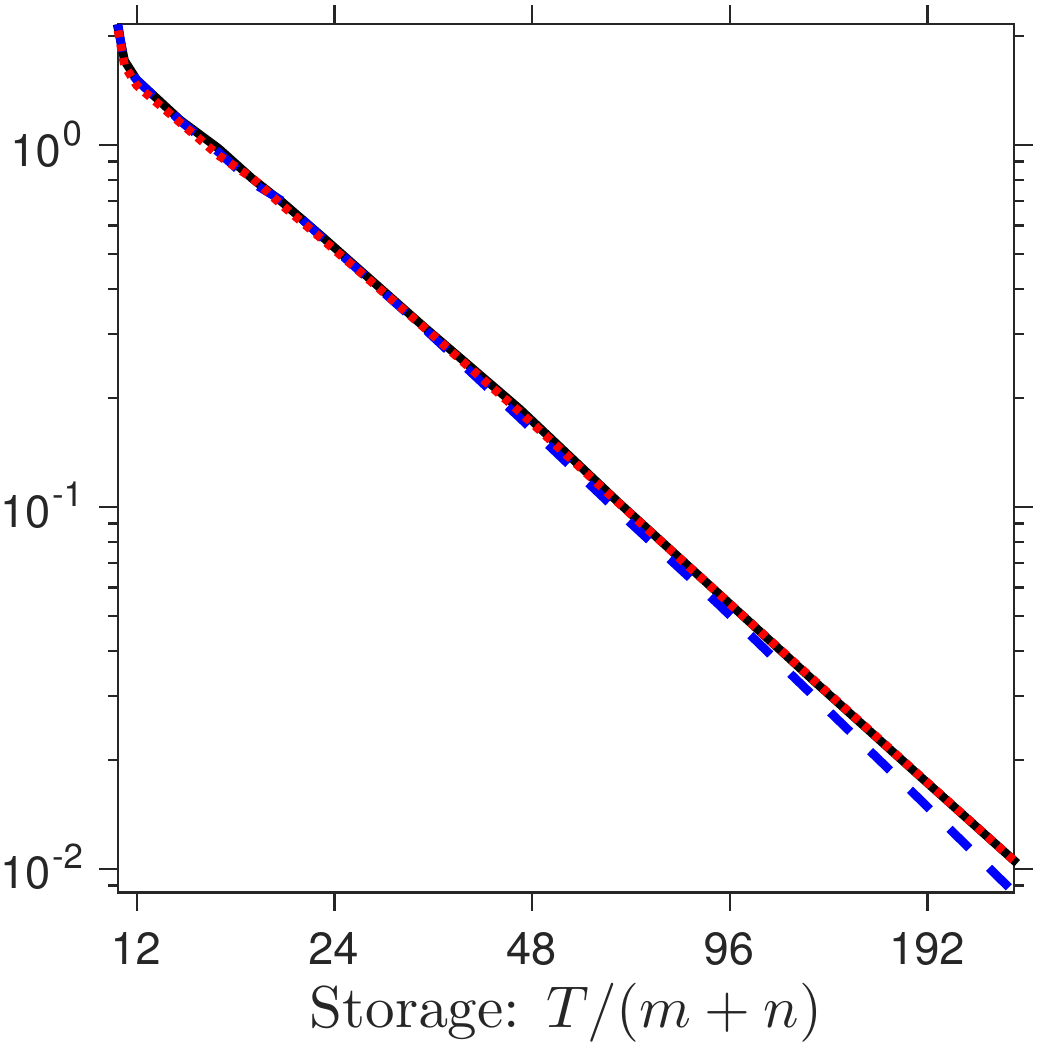}
\caption{\texttt{PolyDecayMed}}
\end{center}
\end{subfigure}
\begin{subfigure}{.325\textwidth}
\begin{center}
\includegraphics[height=1.5in]{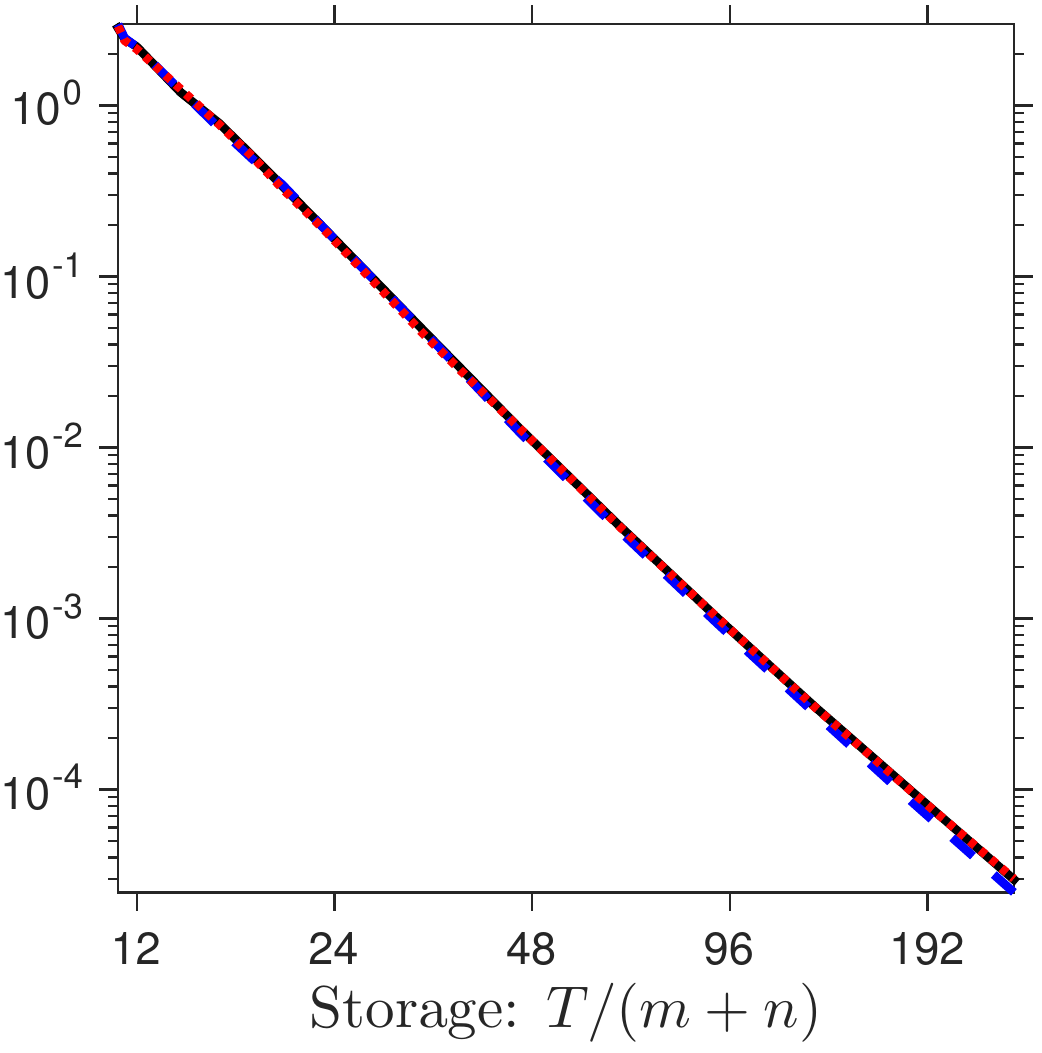}
\caption{\texttt{PolyDecayFast}}
\end{center}
\end{subfigure}
\end{center}

\vspace{0.5em}

\begin{center}
\begin{subfigure}{.325\textwidth}
\begin{center}
\includegraphics[height=1.5in]{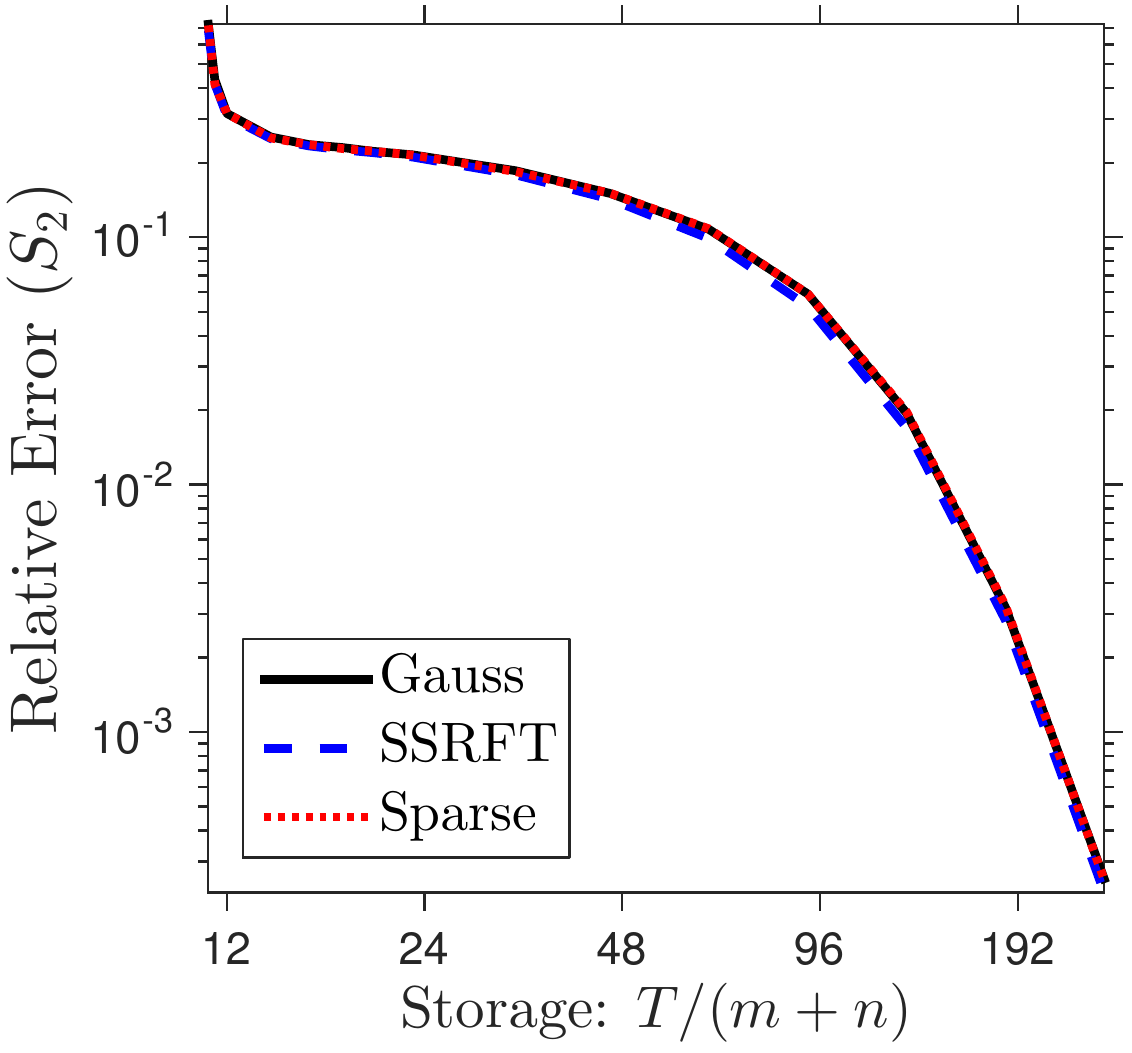}
\caption{\texttt{ExpDecaySlow}}
\end{center}
\end{subfigure}
\begin{subfigure}{.325\textwidth}
\begin{center}
\includegraphics[height=1.5in]{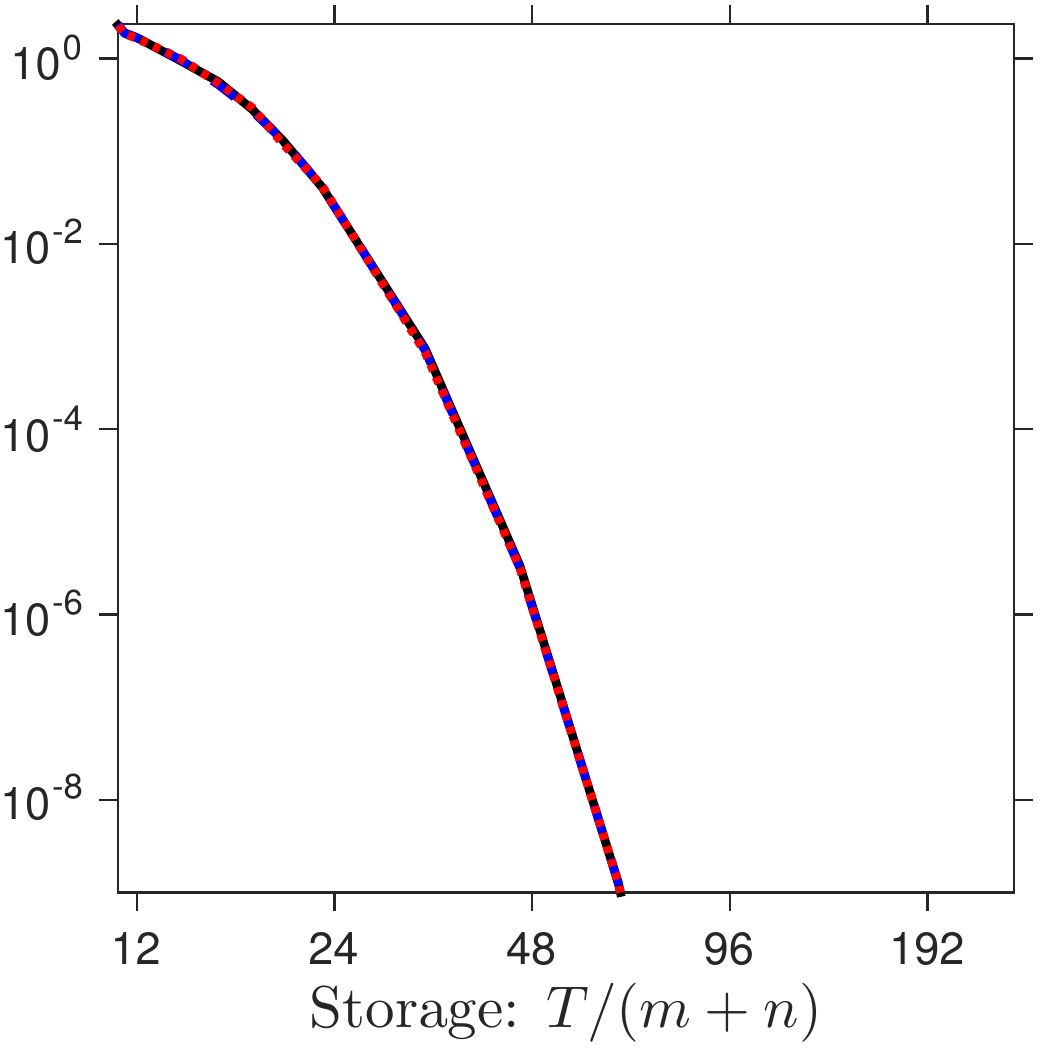}
\caption{\texttt{ExpDecayMed}}
\end{center}
\end{subfigure}
\begin{subfigure}{.325\textwidth}
\begin{center}
\includegraphics[height=1.5in]{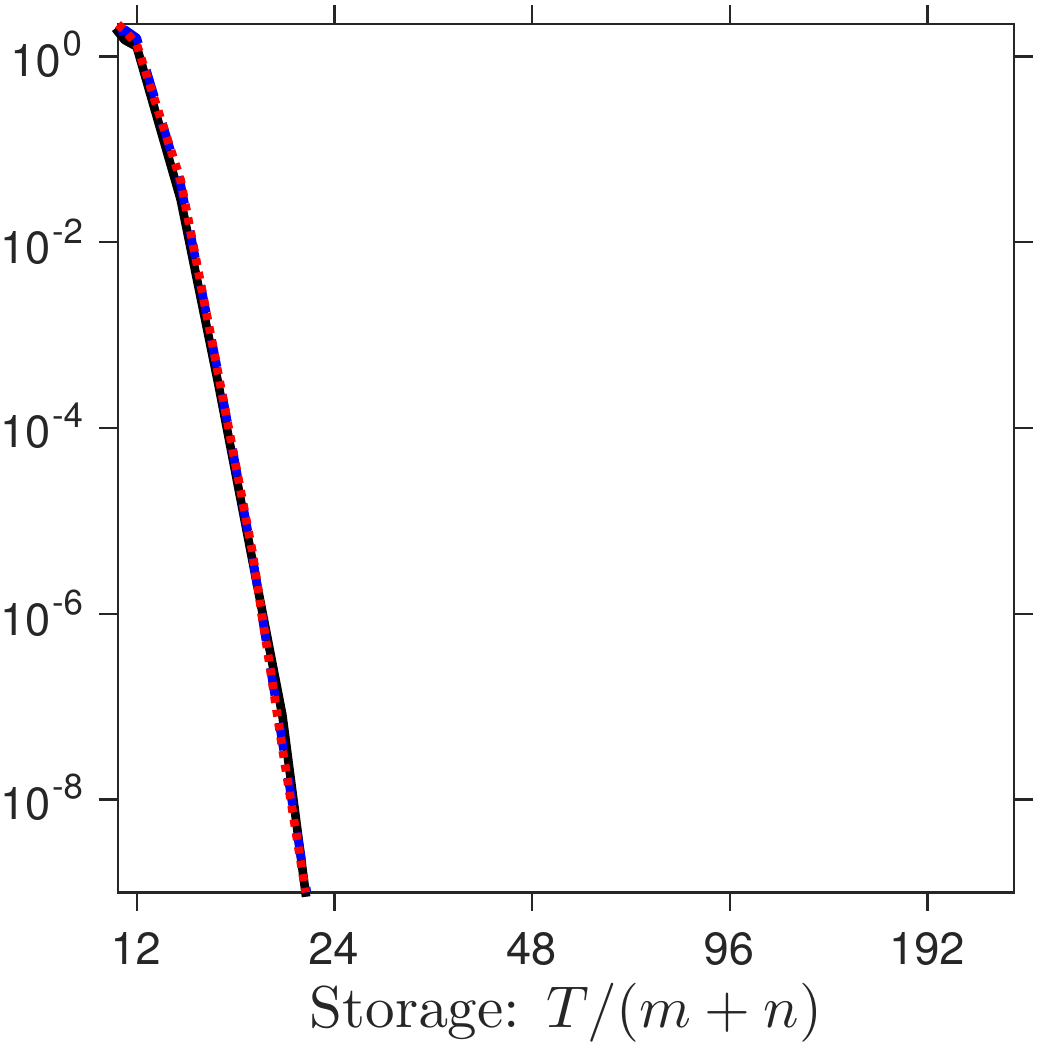}
\caption{\texttt{ExpDecayFast}}
\end{center}
\end{subfigure}
\end{center}

\vspace{0.5em}

\caption{\textbf{Insensitivity of proposed method to the dimension reduction map.}
(Effective rank $R = 5$, approximation rank $r = 10$, Schatten 2-norm.)
We compare the oracle performance of the proposed fixed-rank
approximation~\cref{eqn:Ahat-fixed} implemented with Gaussian, SSRFT, or sparse
dimension reduction maps.  See~\cref{app:universality}
for details.}
\label{fig:universality-R5-S2}
\end{figure}

\begin{figure}[htp!]
\begin{center}
\begin{subfigure}{.325\textwidth}
\begin{center}
\includegraphics[height=1.5in]{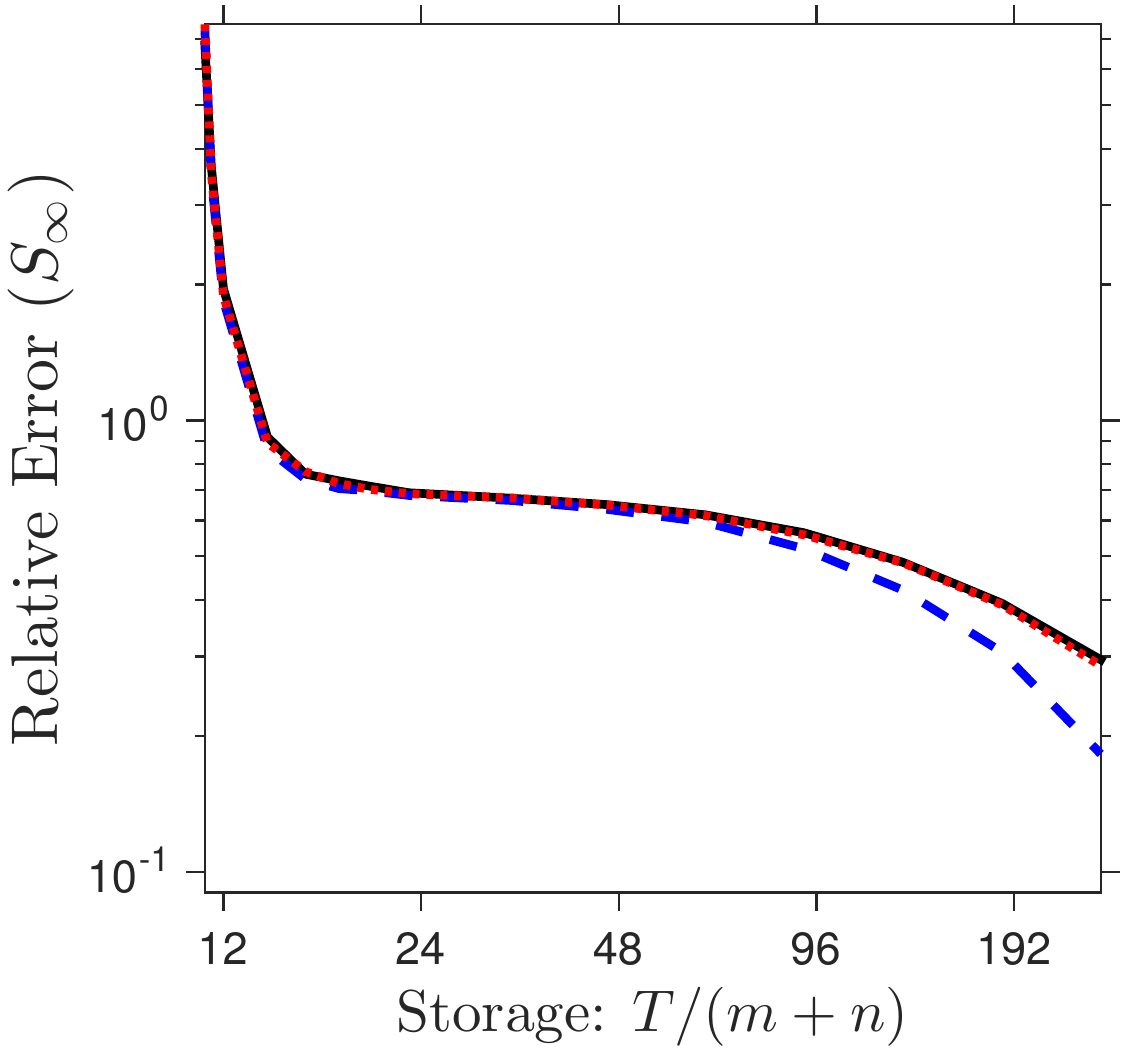}
\caption{\texttt{LowRankHiNoise}}
\end{center}
\end{subfigure}
\begin{subfigure}{.325\textwidth}
\begin{center}
\includegraphics[height=1.5in]{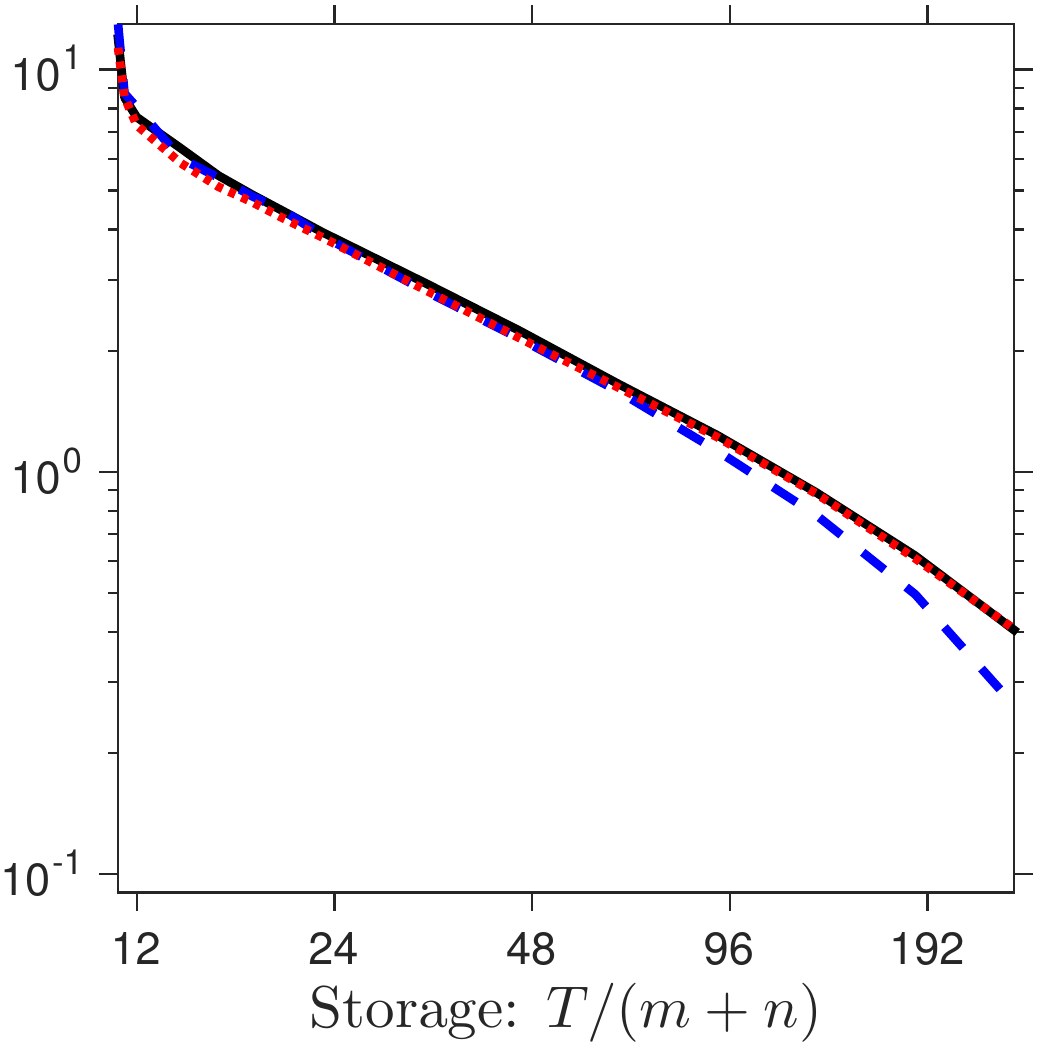}
\caption{\texttt{LowRankMedNoise}}
\end{center}
\end{subfigure}
\begin{subfigure}{.325\textwidth}
\begin{center}
\includegraphics[height=1.5in]{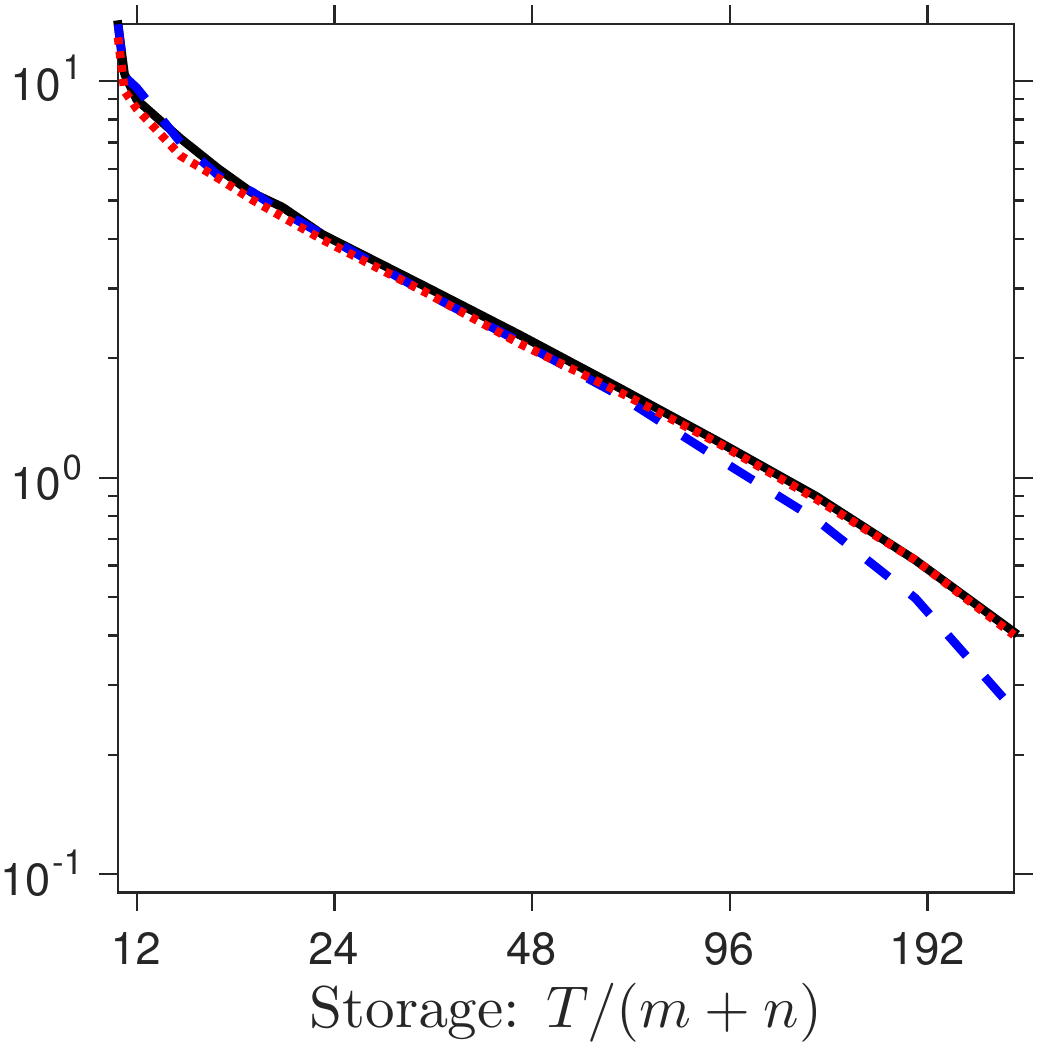}
\caption{\texttt{LowRankLowNoise}}
\end{center}
\end{subfigure}
\end{center}

\vspace{.5em}

\begin{center}
\begin{subfigure}{.325\textwidth}
\begin{center}
\includegraphics[height=1.5in]{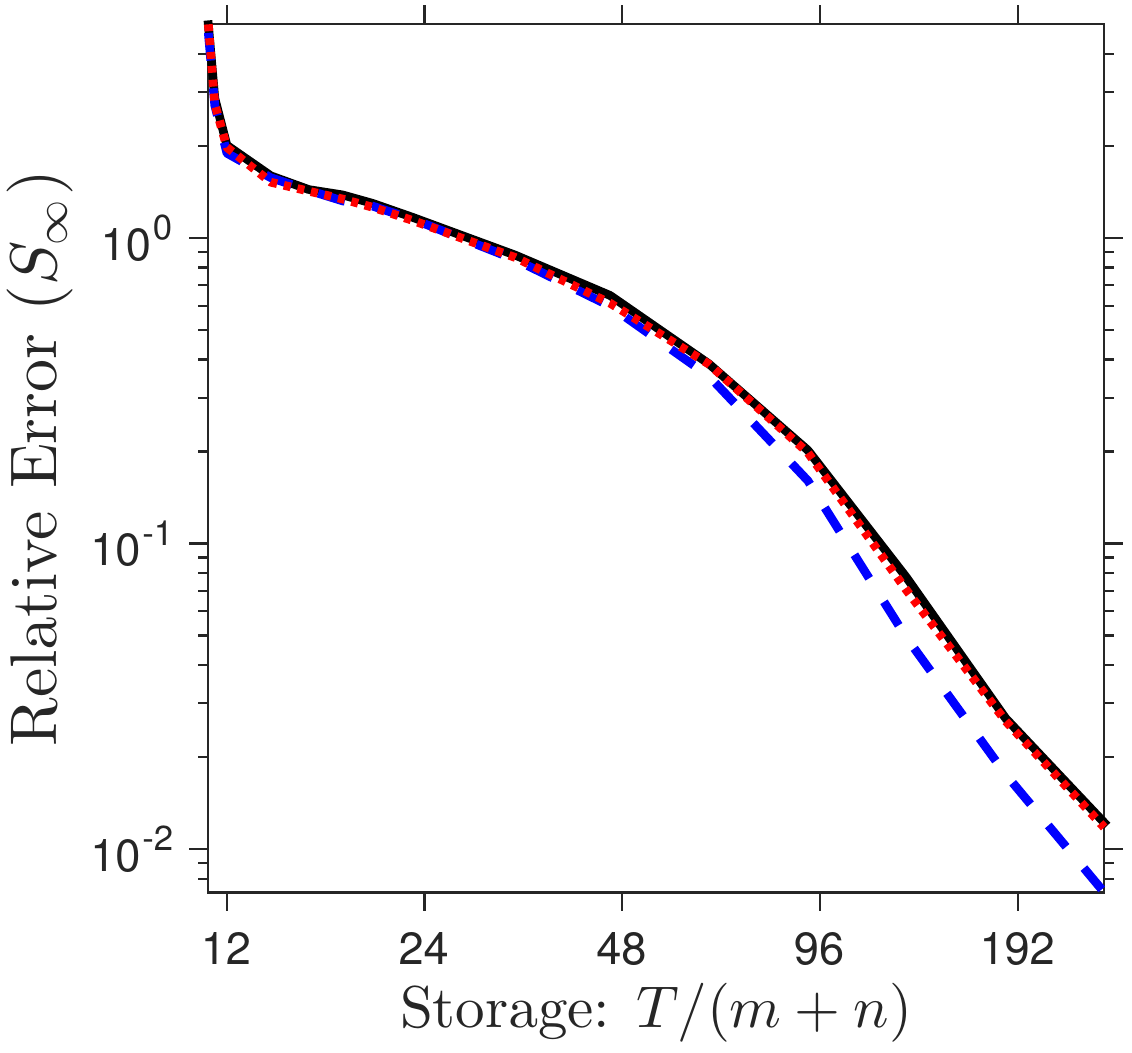}
\caption{\texttt{PolyDecaySlow}}
\end{center}
\end{subfigure}
\begin{subfigure}{.325\textwidth}
\begin{center}
\includegraphics[height=1.5in]{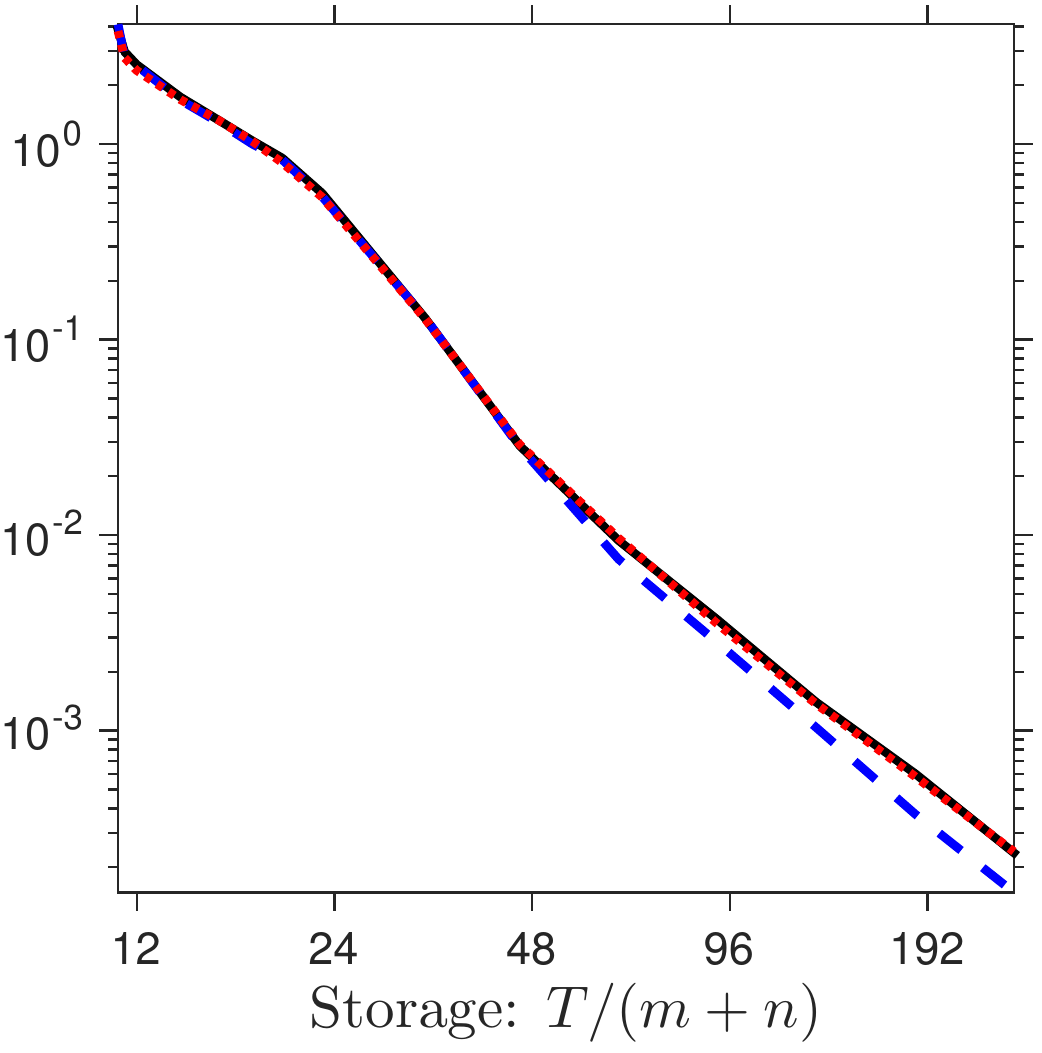}
\caption{\texttt{PolyDecayMed}}
\end{center}
\end{subfigure}
\begin{subfigure}{.325\textwidth}
\begin{center}
\includegraphics[height=1.5in]{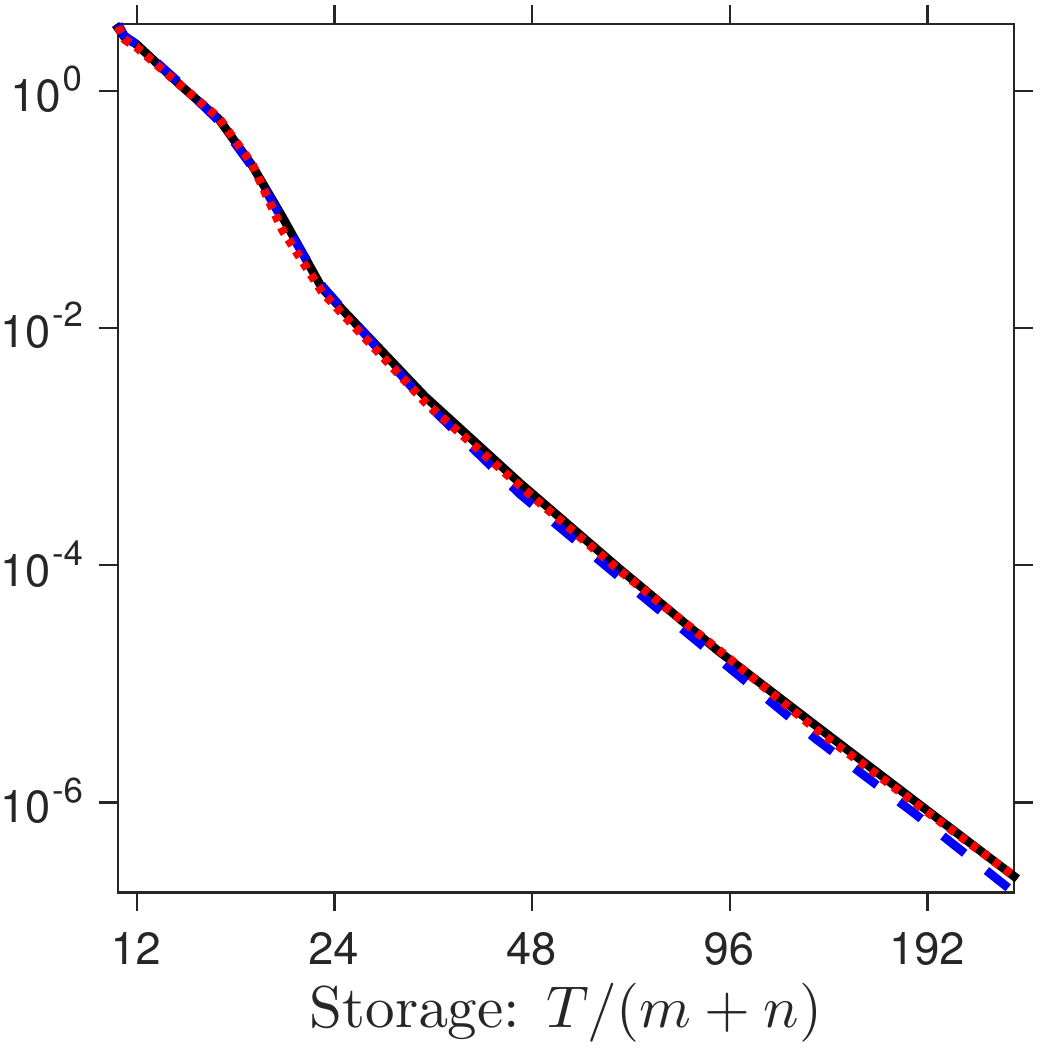}
\caption{\texttt{PolyDecayFast}}
\end{center}
\end{subfigure}
\end{center}

\vspace{0.5em}

\begin{center}
\begin{subfigure}{.325\textwidth}
\begin{center}
\includegraphics[height=1.5in]{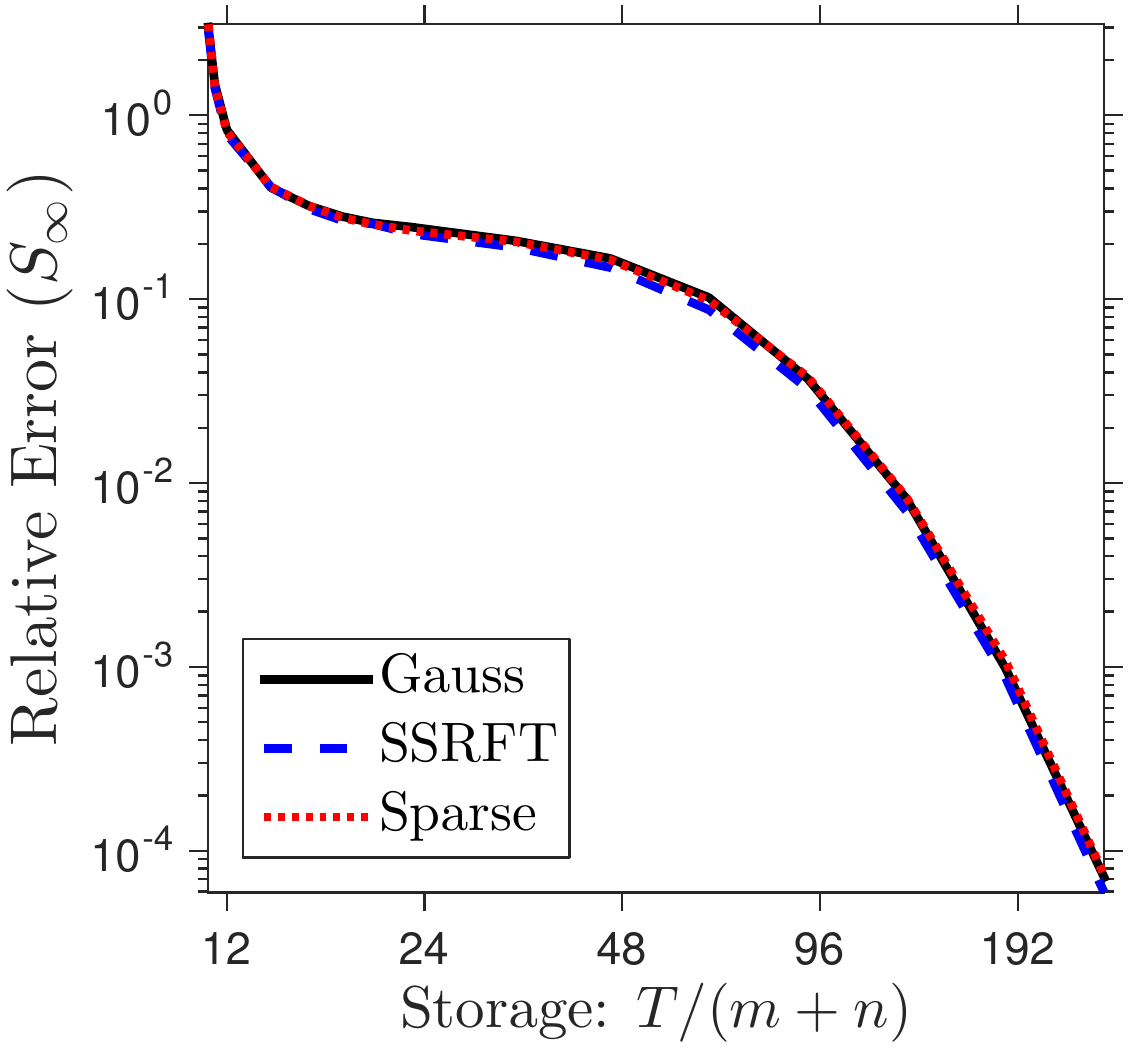}
\caption{\texttt{ExpDecaySlow}}
\end{center}
\end{subfigure}
\begin{subfigure}{.325\textwidth}
\begin{center}
\includegraphics[height=1.5in]{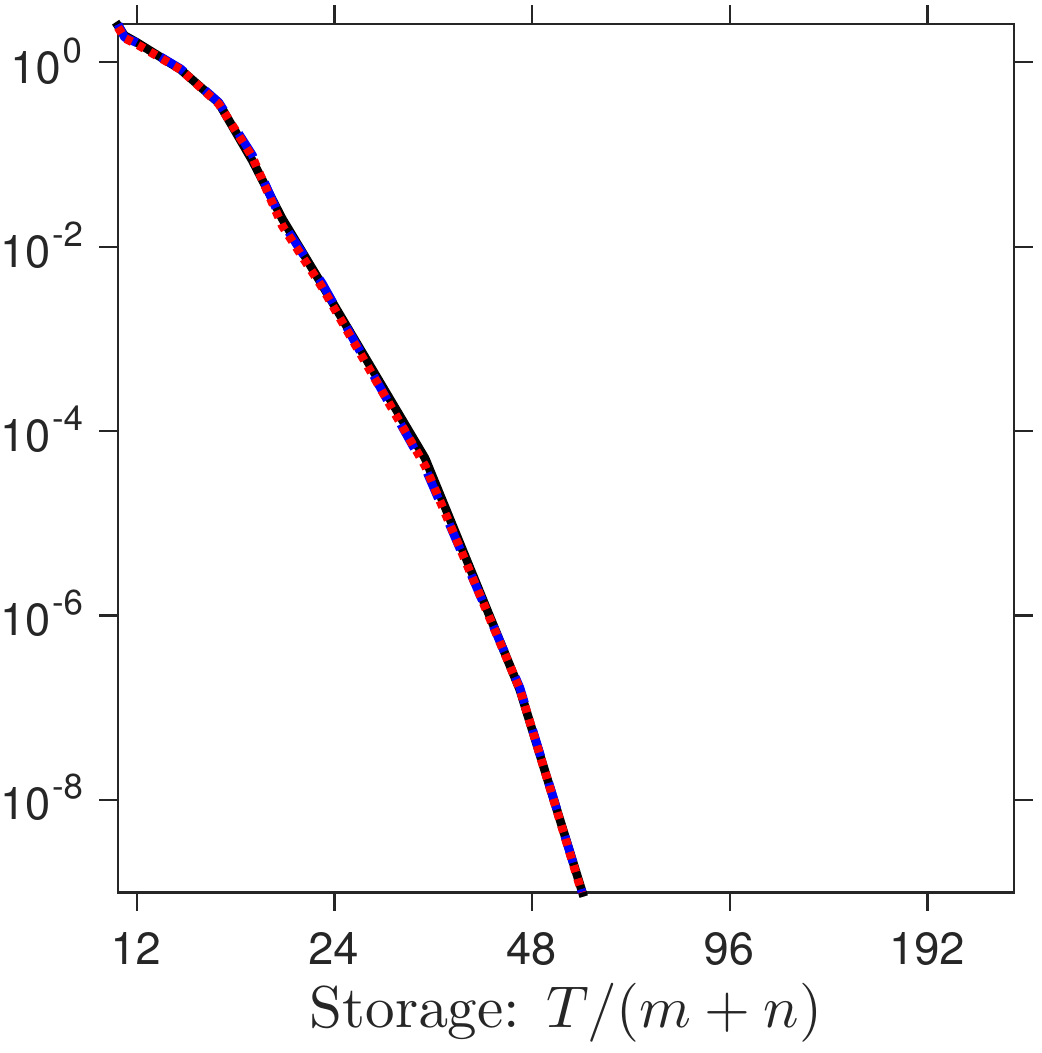}
\caption{\texttt{ExpDecayMed}}
\end{center}
\end{subfigure}
\begin{subfigure}{.325\textwidth}
\begin{center}
\includegraphics[height=1.5in]{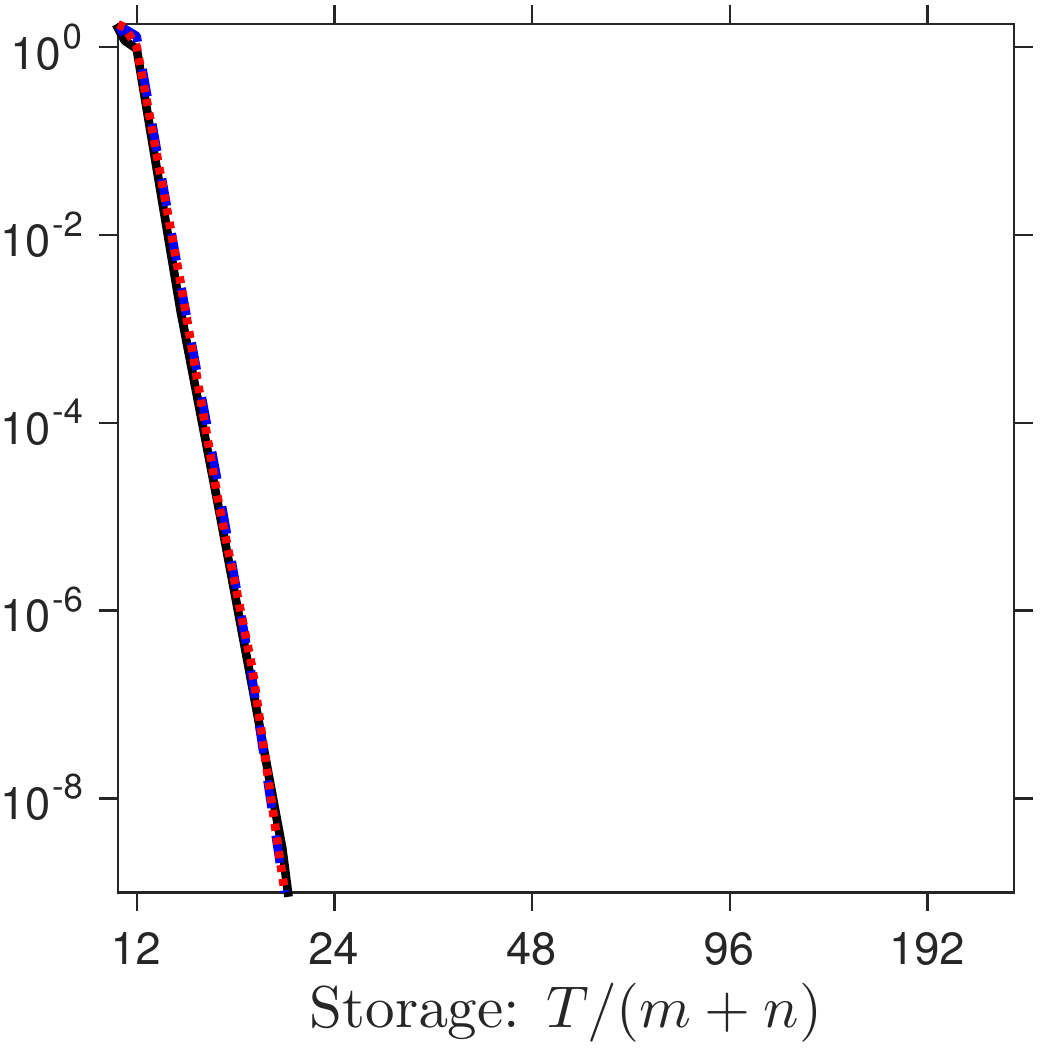}
\caption{\texttt{ExpDecayFast}}
\end{center}
\end{subfigure}
\end{center}

\vspace{0.5em}

\caption{\textbf{Insensitivity of proposed method to the dimension reduction map.}
(Effective rank $R = 5$, approximation rank $r = 10$, Schatten $\infty$-norm.)
We compare the oracle performance of the proposed fixed-rank
approximation~\cref{eqn:Ahat-fixed} implemented with Gaussian, SSRFT, or sparse
dimension reduction maps.  See~\cref{app:universality}
for details.}
\label{fig:universality-R5-Sinf}
\end{figure}

\begin{figure}[htp!]
\begin{center}
\begin{subfigure}{.325\textwidth}
\begin{center}
\includegraphics[height=1.5in]{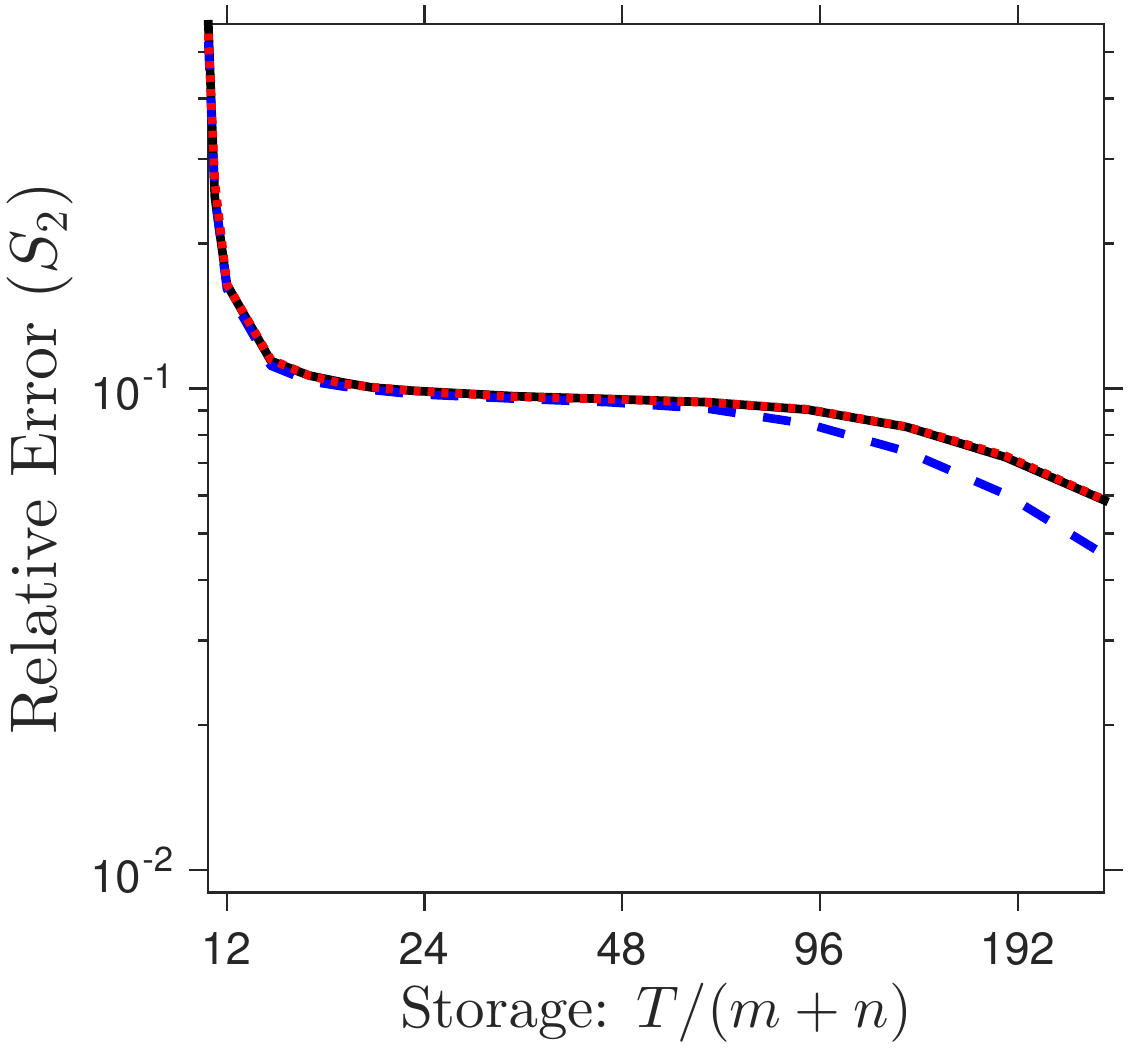}
\caption{\texttt{LowRankHiNoise}}
\end{center}
\end{subfigure}
\begin{subfigure}{.325\textwidth}
\begin{center}
\includegraphics[height=1.5in]{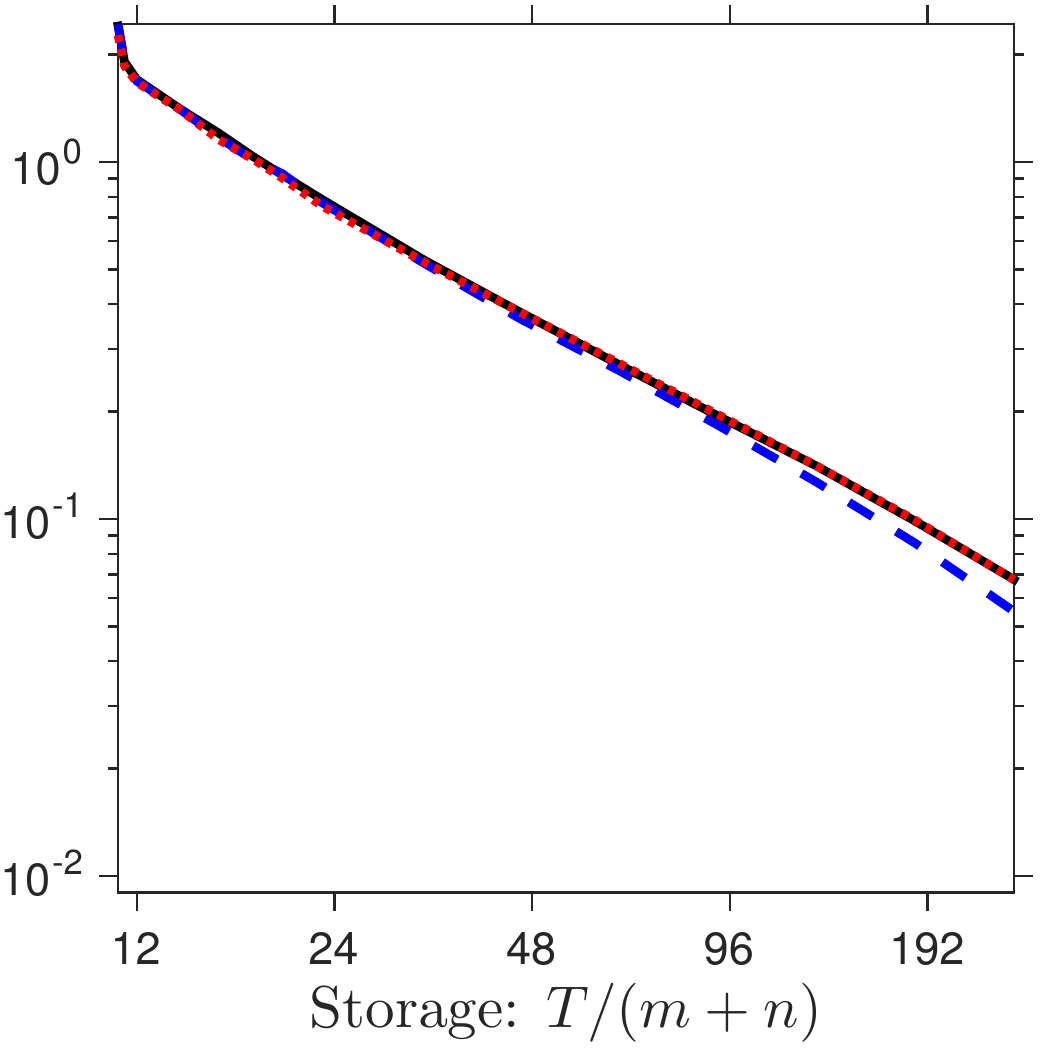}
\caption{\texttt{LowRankMedNoise}}
\end{center}
\end{subfigure}
\begin{subfigure}{.325\textwidth}
\begin{center}
\includegraphics[height=1.5in]{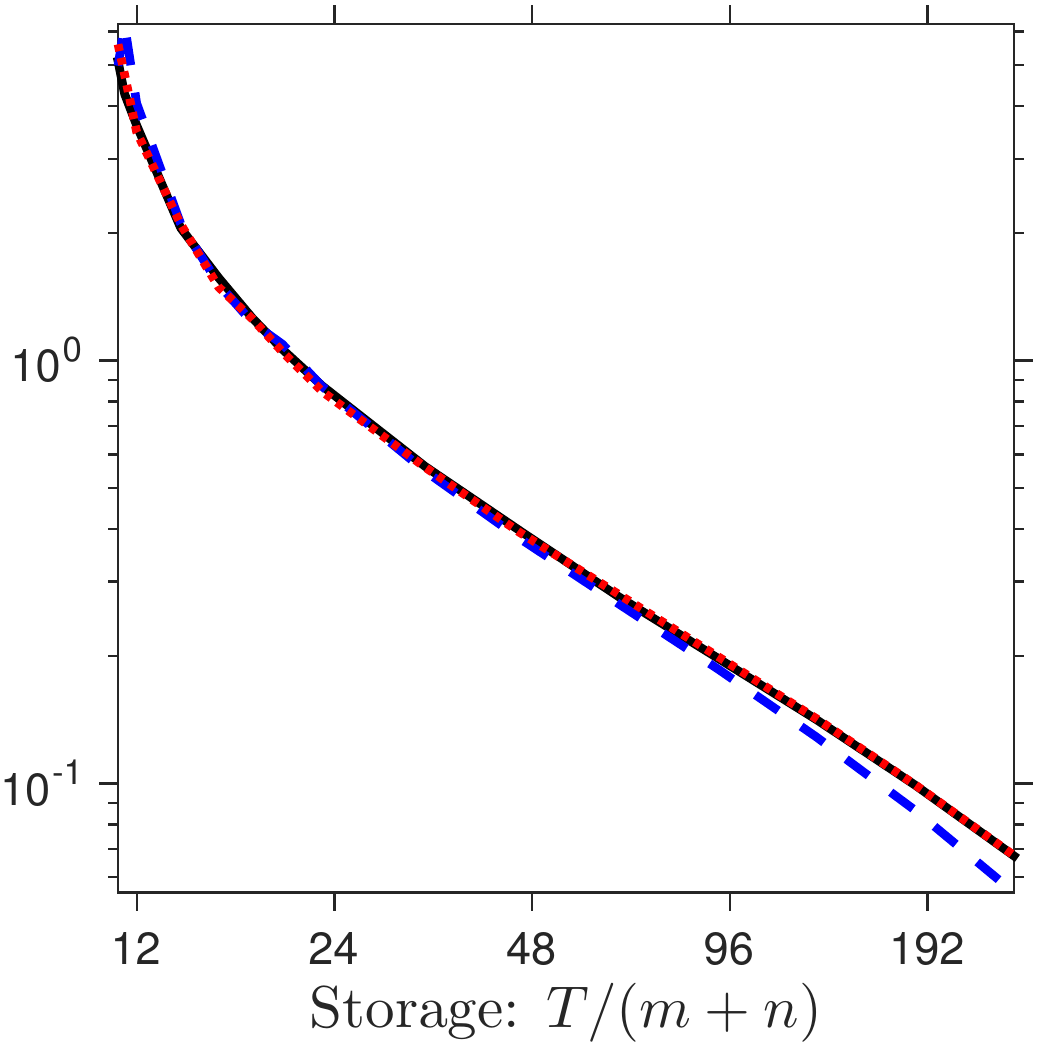}
\caption{\texttt{LowRankLowNoise}}
\end{center}
\end{subfigure}
\end{center}

\vspace{.5em}

\begin{center}
\begin{subfigure}{.325\textwidth}
\begin{center}
\includegraphics[height=1.5in]{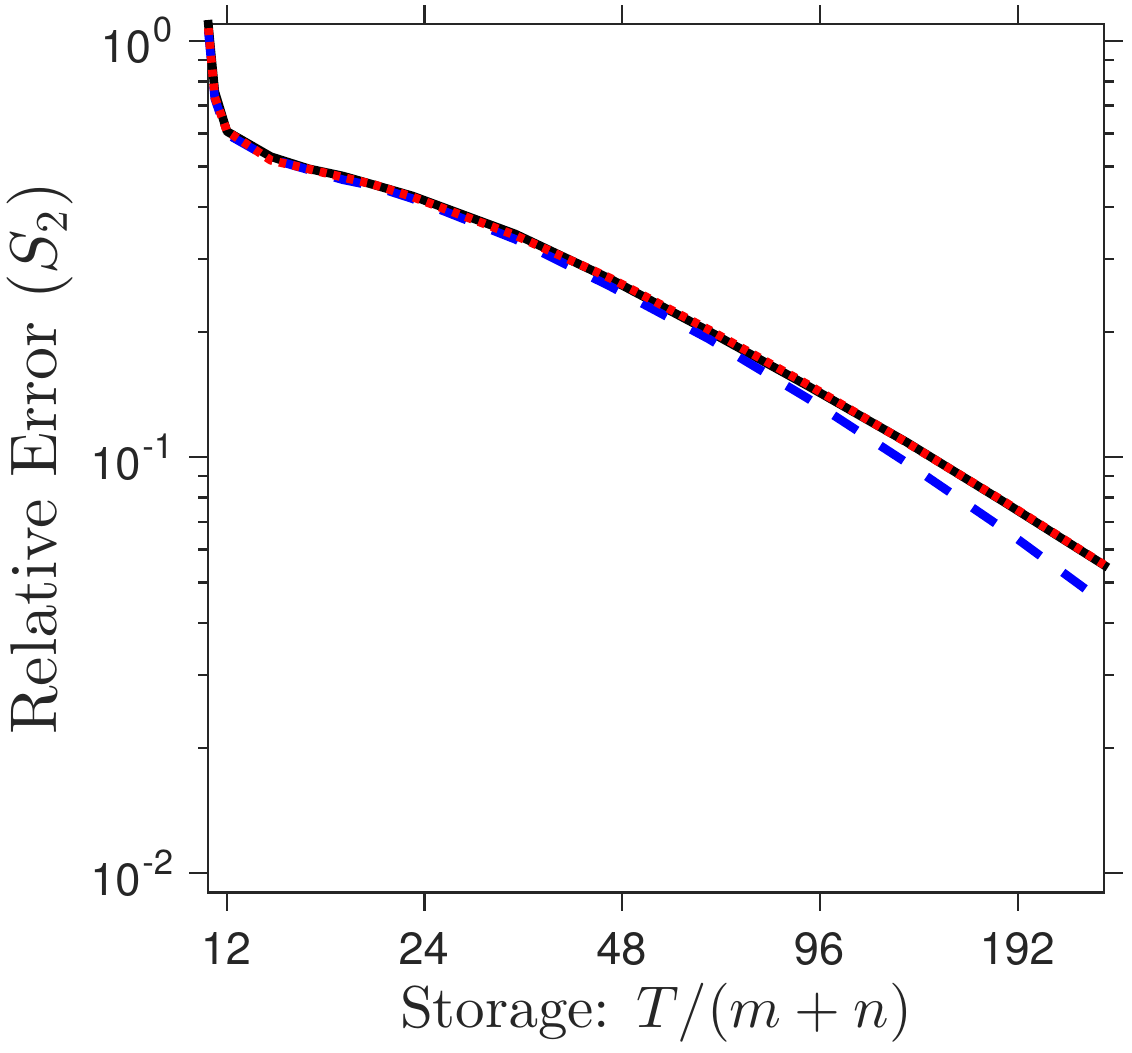}
\caption{\texttt{PolyDecaySlow}}
\end{center}
\end{subfigure}
\begin{subfigure}{.325\textwidth}
\begin{center}
\includegraphics[height=1.5in]{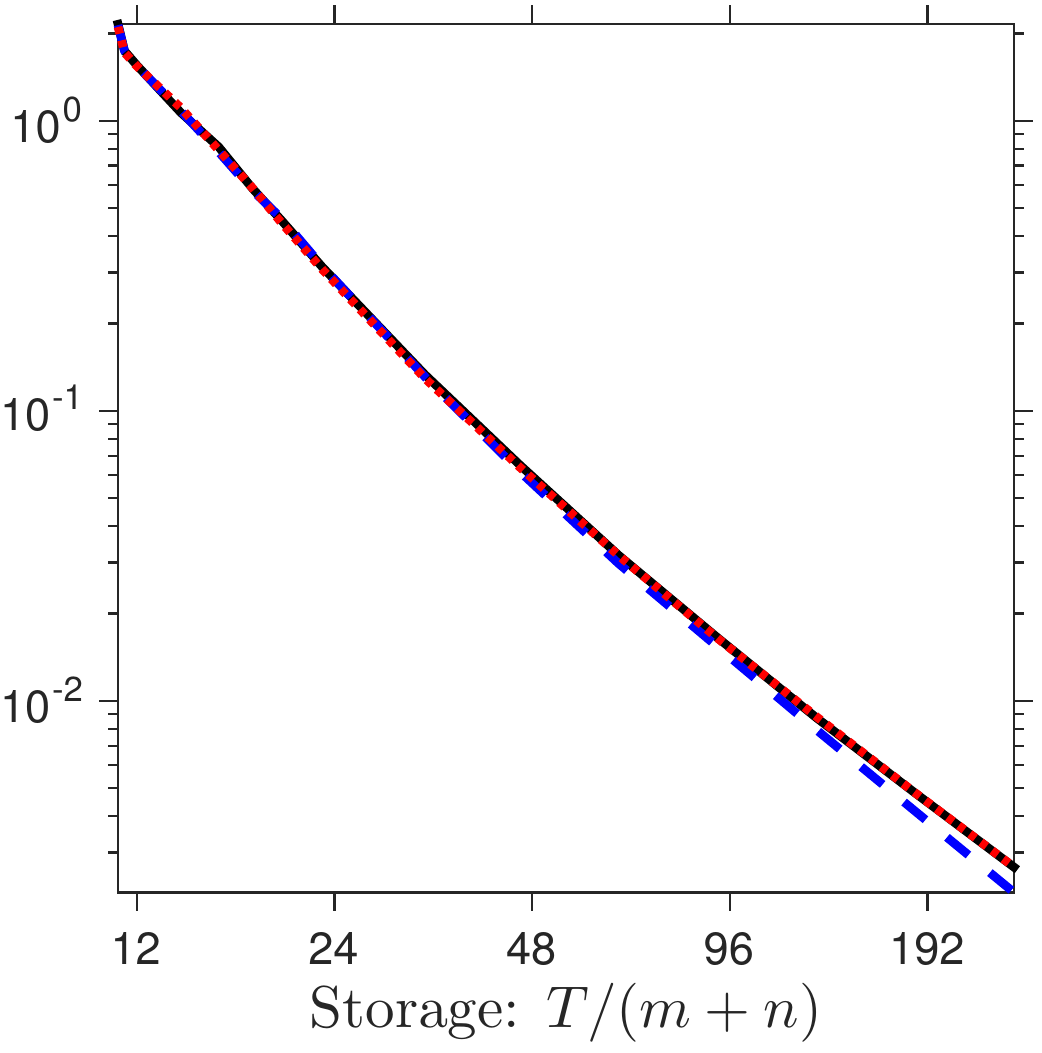}
\caption{\texttt{PolyDecayMed}}
\end{center}
\end{subfigure}
\begin{subfigure}{.325\textwidth}
\begin{center}
\includegraphics[height=1.5in]{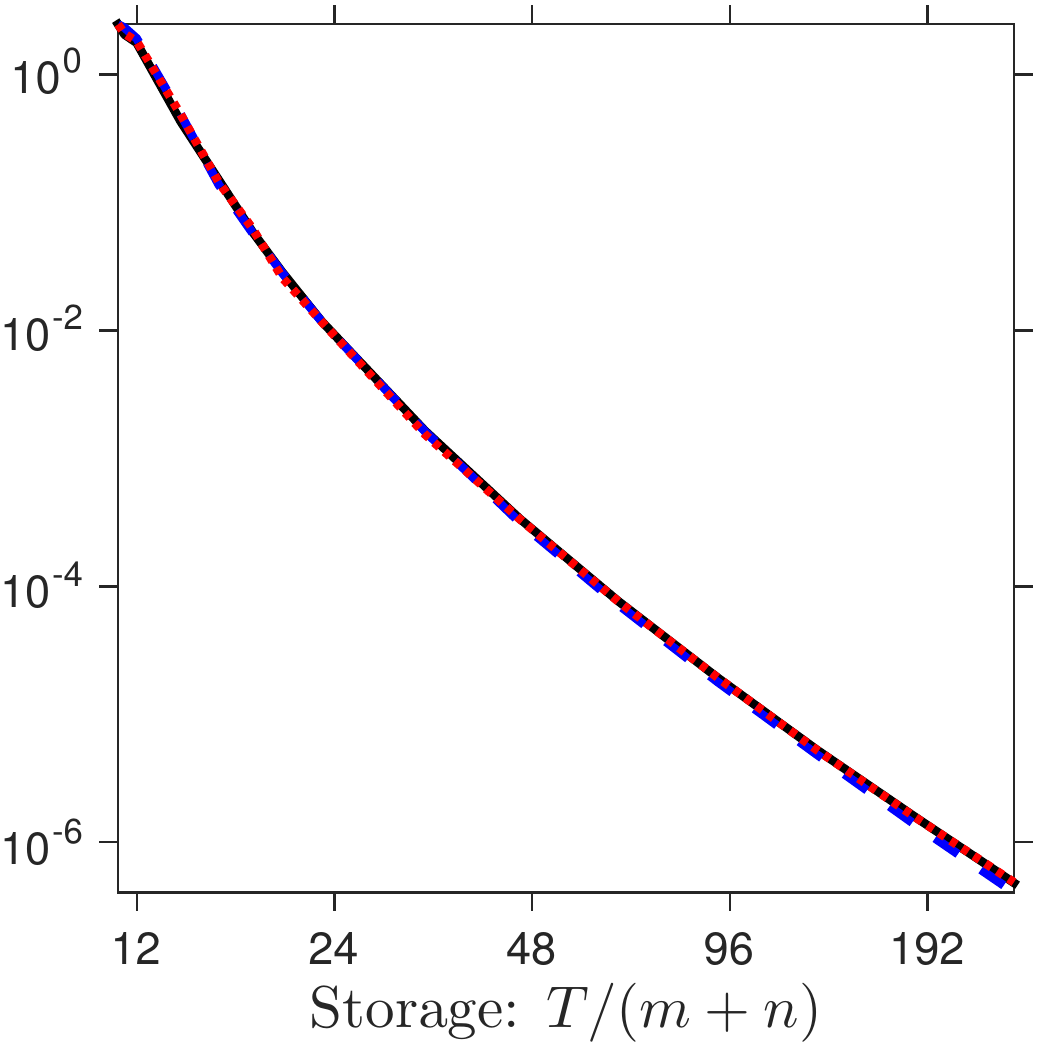}
\caption{\texttt{PolyDecayFast}}
\end{center}
\end{subfigure}
\end{center}

\vspace{0.5em}

\begin{center}
\begin{subfigure}{.325\textwidth}
\begin{center}
\includegraphics[height=1.5in]{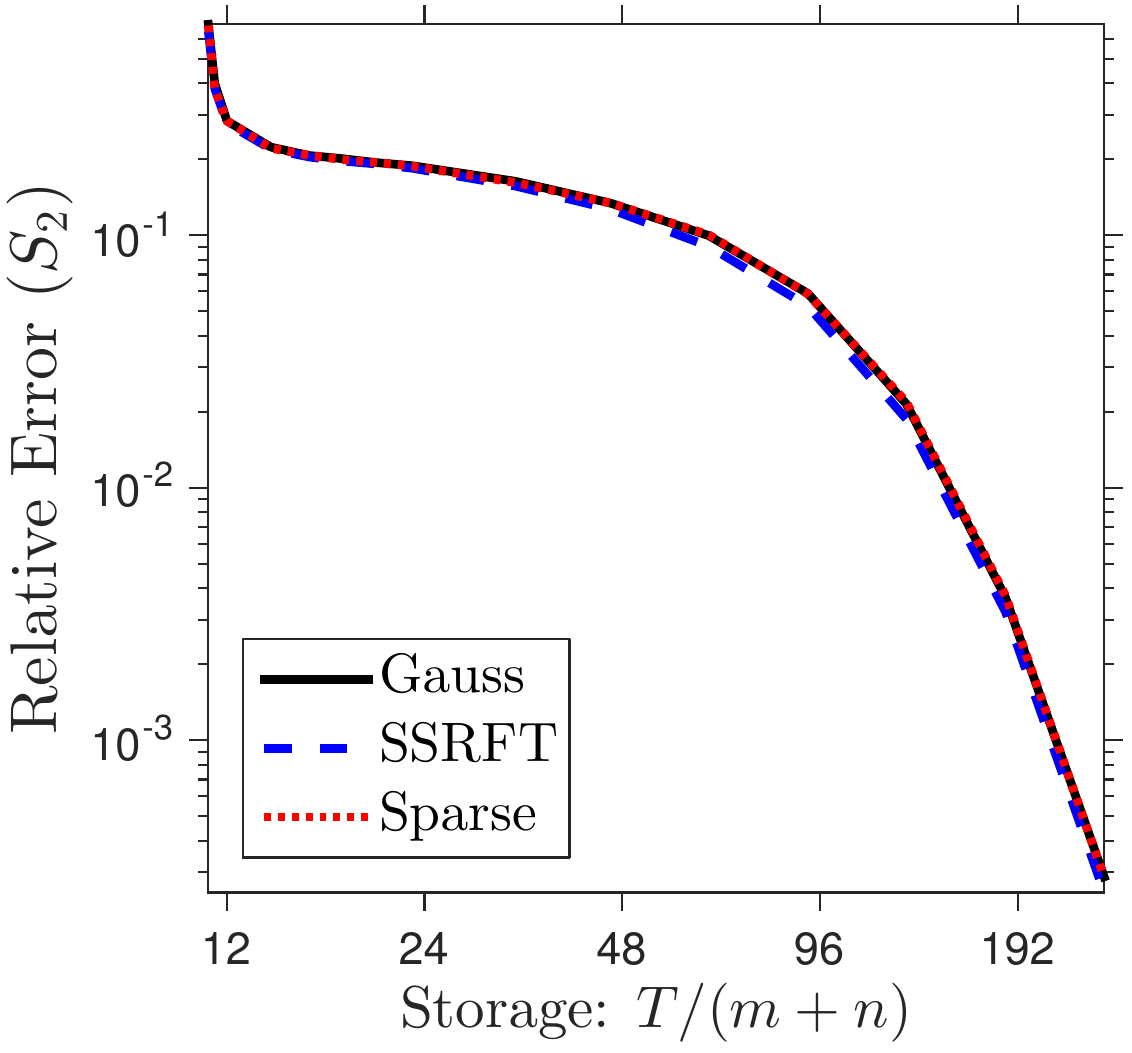}
\caption{\texttt{ExpDecaySlow}}
\end{center}
\end{subfigure}
\begin{subfigure}{.325\textwidth}
\begin{center}
\includegraphics[height=1.5in]{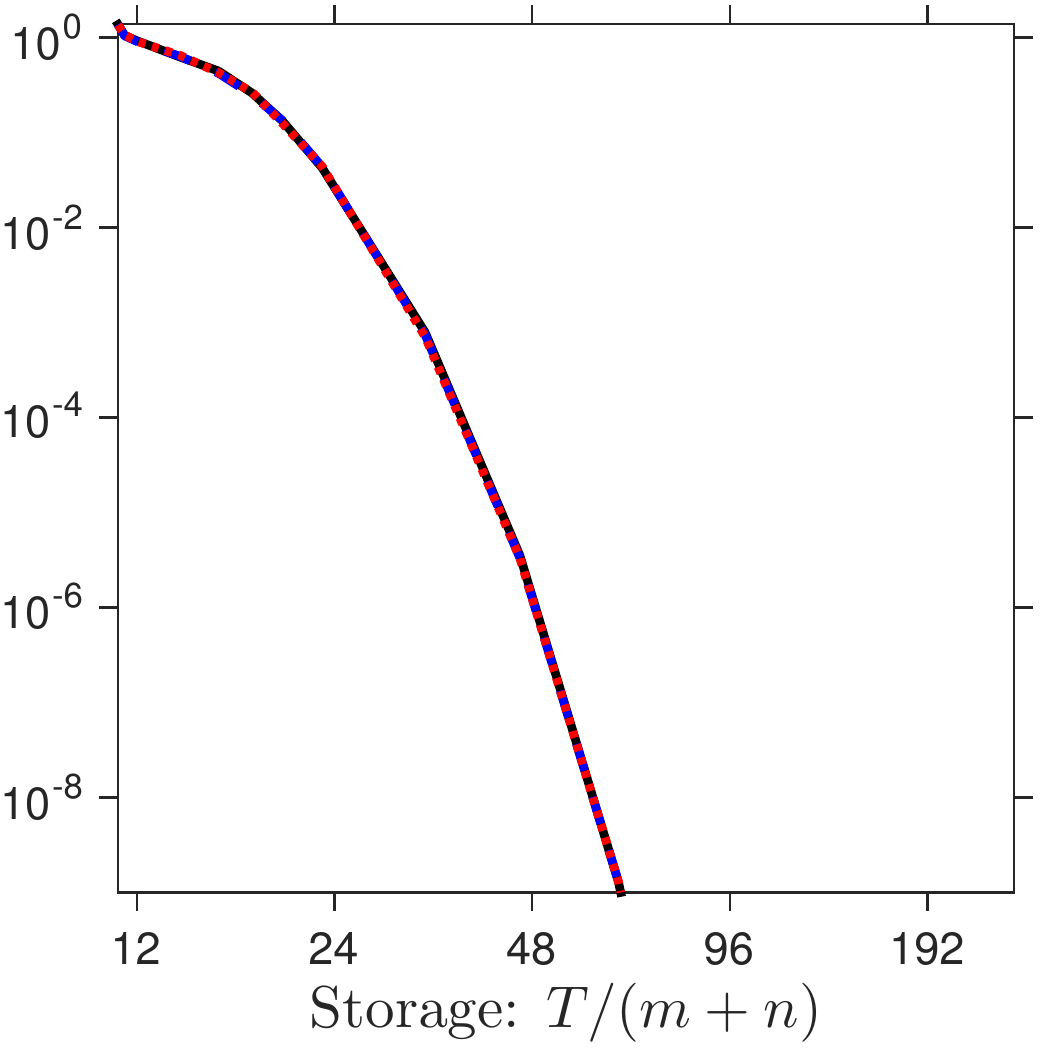}
\caption{\texttt{ExpDecayMed}}
\end{center}
\end{subfigure}
\begin{subfigure}{.325\textwidth}
\begin{center}
\includegraphics[height=1.5in]{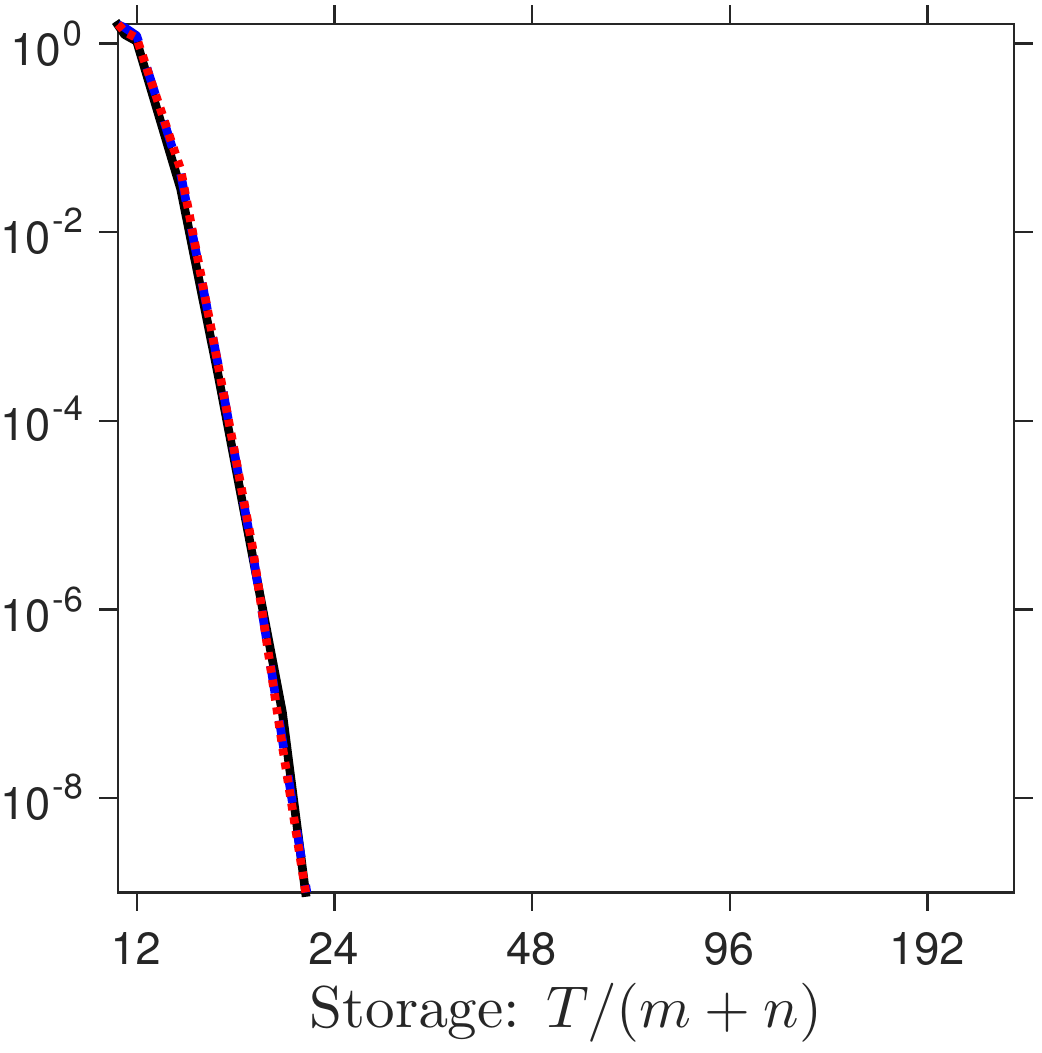}
\caption{\texttt{ExpDecayFast}}
\end{center}
\end{subfigure}
\end{center}

\vspace{0.5em}

\caption{\textbf{Insensitivity of proposed method to the dimension reduction map.}
(Effective rank $R = 10$, approximation rank $r = 10$, Schatten 2-norm.)
We compare the oracle performance of the proposed fixed-rank
approximation~\cref{eqn:Ahat-fixed} implemented with Gaussian, SSRFT, or sparse
dimension reduction maps.  See~\cref{app:universality}
for details.}
\label{fig:universality-R10-S2}
\end{figure}

\begin{figure}[htp!]
\begin{center}
\begin{subfigure}{.325\textwidth}
\begin{center}
\includegraphics[height=1.5in]{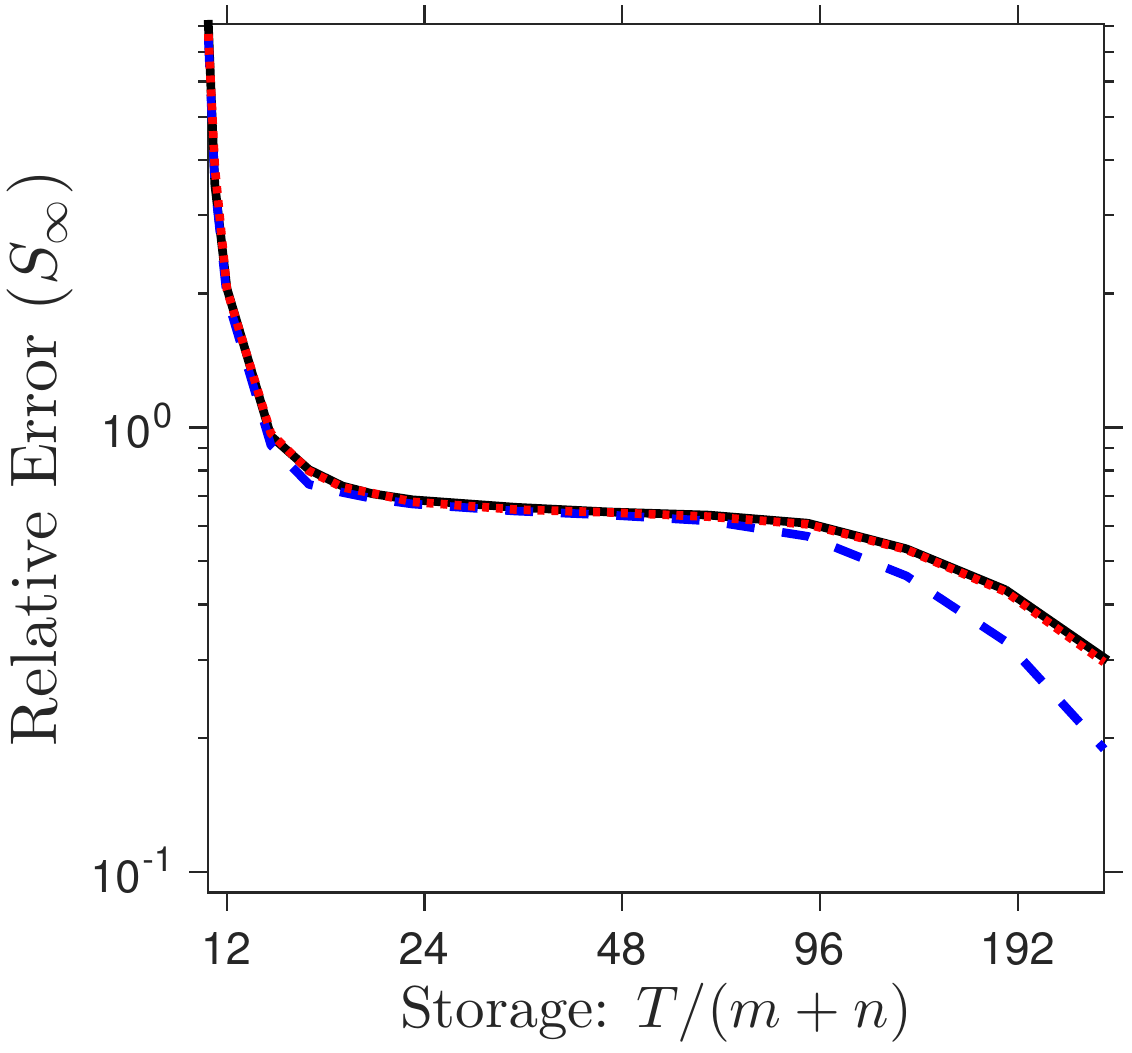}
\caption{\texttt{LowRankHiNoise}}
\end{center}
\end{subfigure}
\begin{subfigure}{.325\textwidth}
\begin{center}
\includegraphics[height=1.5in]{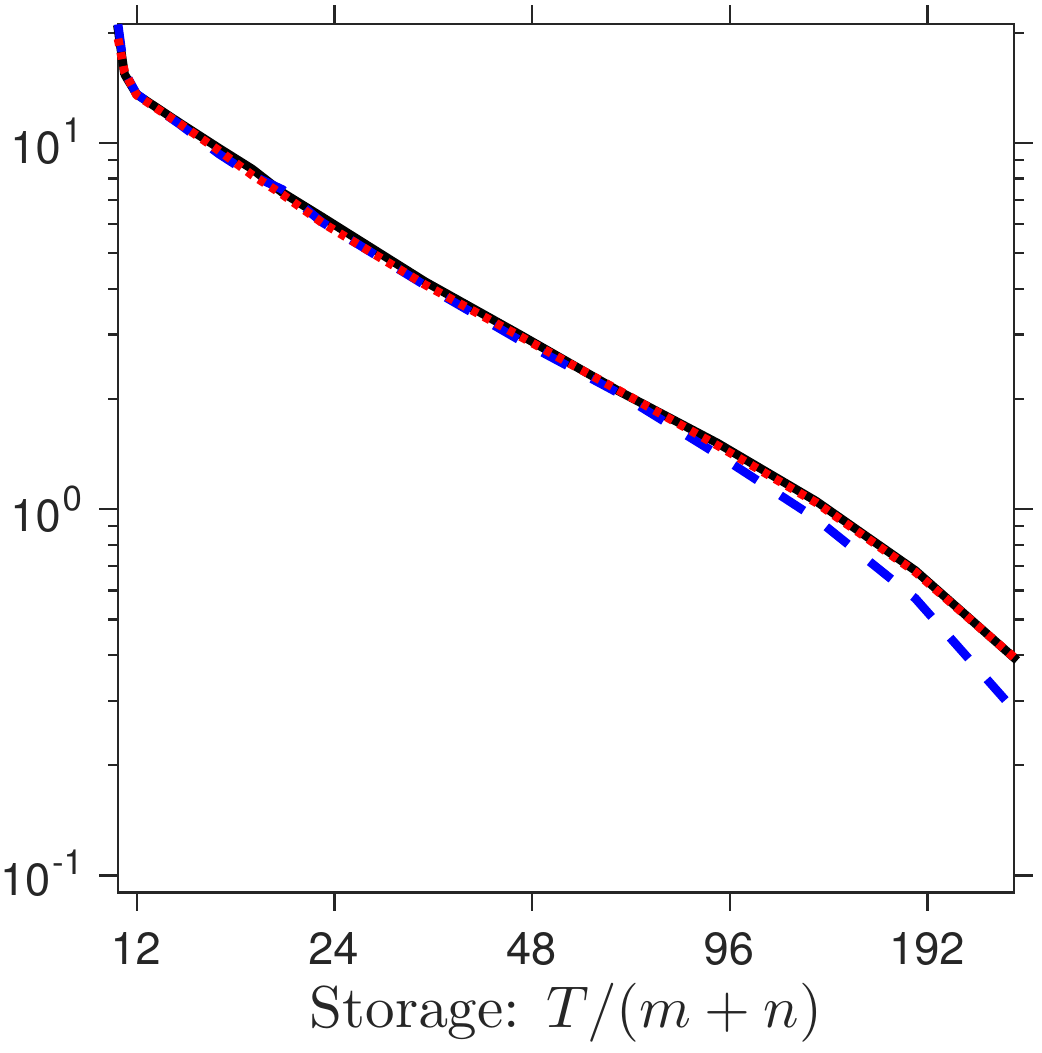}
\caption{\texttt{LowRankMedNoise}}
\end{center}
\end{subfigure}
\begin{subfigure}{.325\textwidth}
\begin{center}
\includegraphics[height=1.5in]{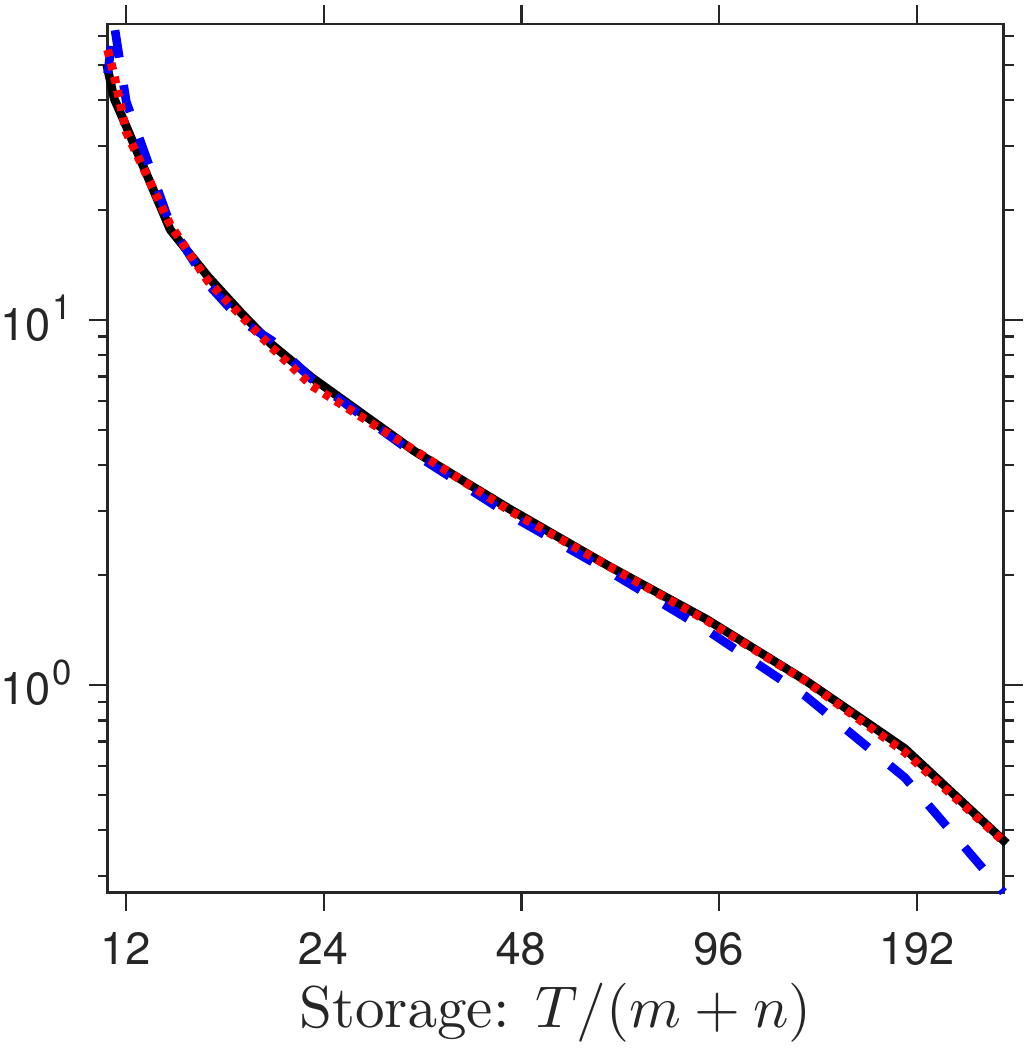}
\caption{\texttt{LowRankLowNoise}}
\end{center}
\end{subfigure}
\end{center}

\vspace{.5em}

\begin{center}
\begin{subfigure}{.325\textwidth}
\begin{center}
\includegraphics[height=1.5in]{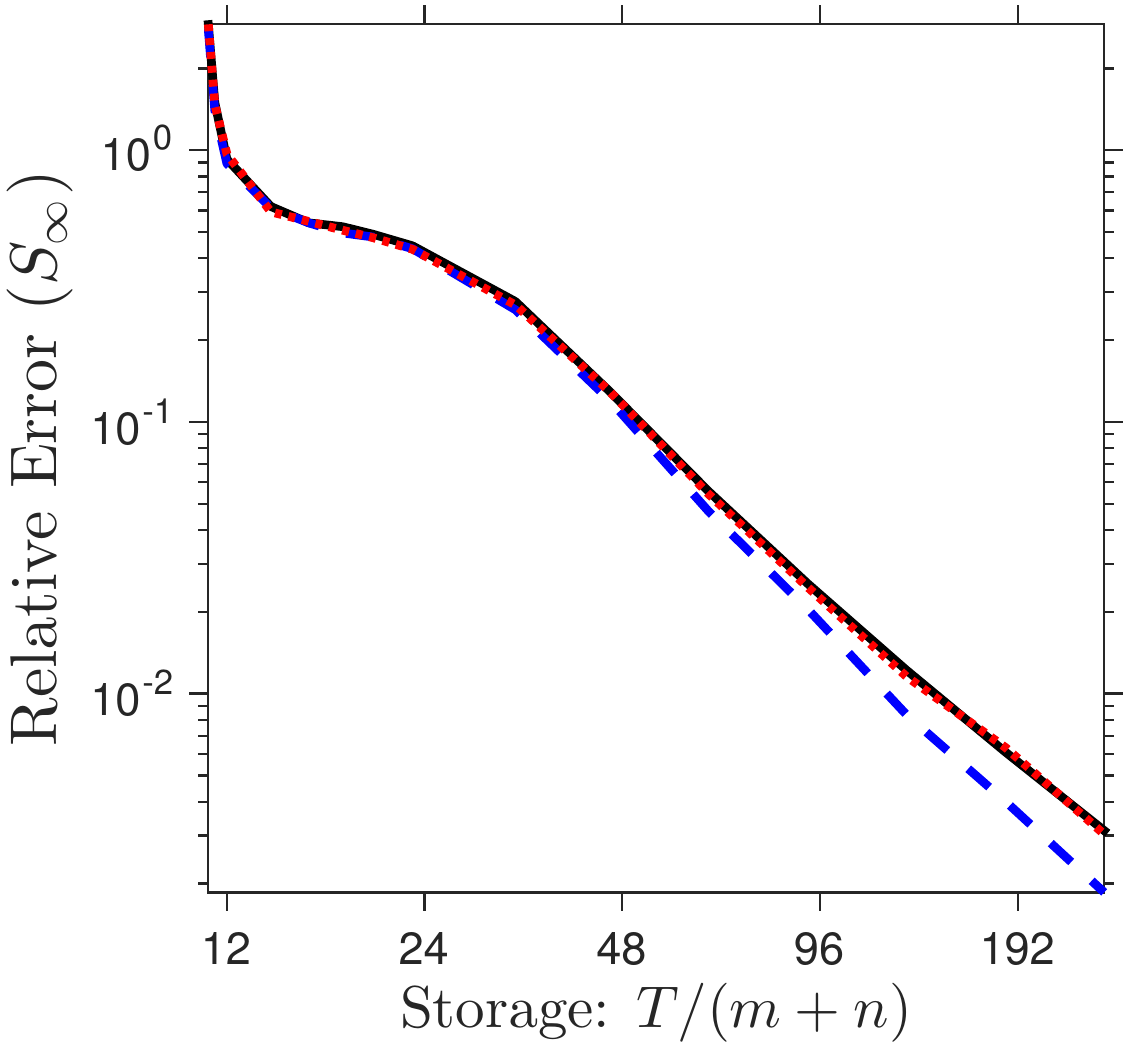}
\caption{\texttt{PolyDecaySlow}}
\end{center}
\end{subfigure}
\begin{subfigure}{.325\textwidth}
\begin{center}
\includegraphics[height=1.5in]{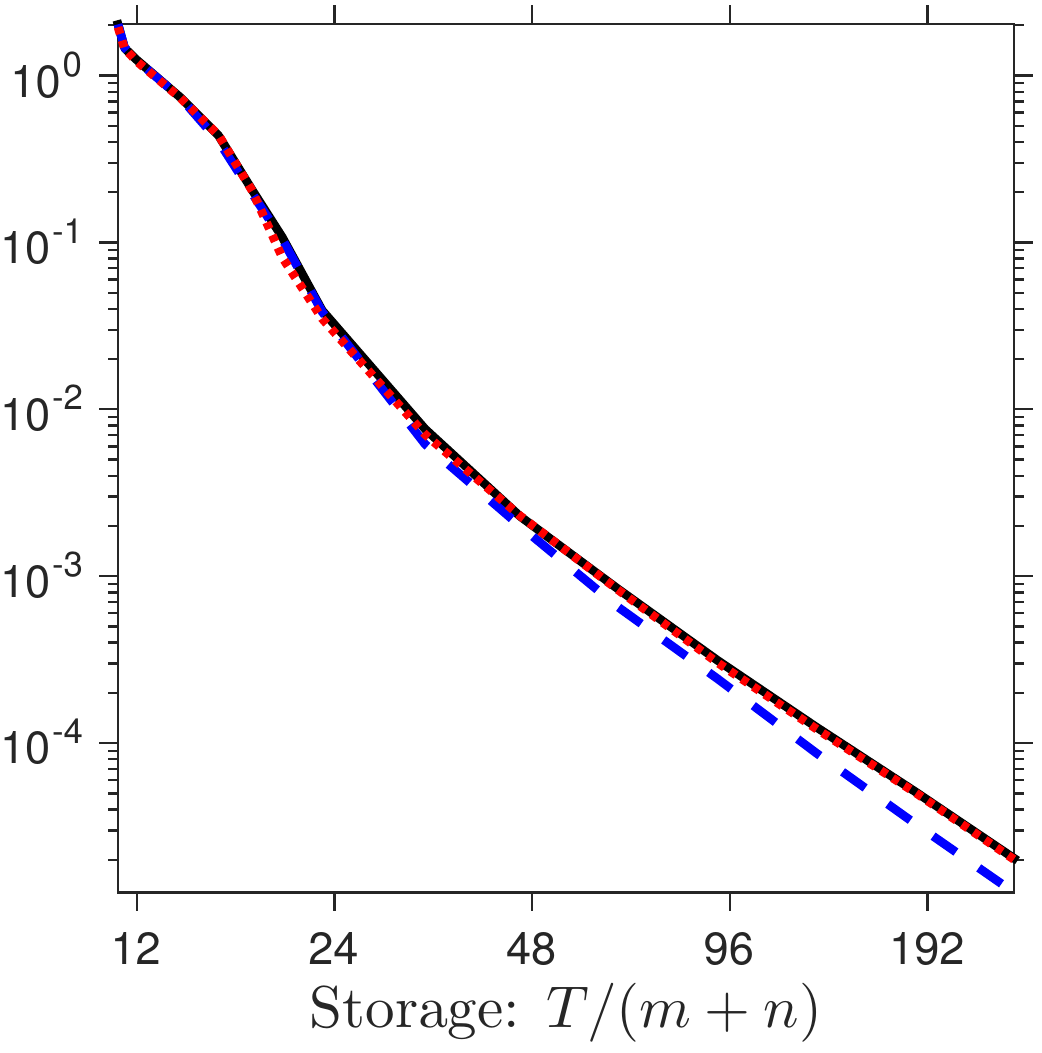}
\caption{\texttt{PolyDecayMed}}
\end{center}
\end{subfigure}
\begin{subfigure}{.325\textwidth}
\begin{center}
\includegraphics[height=1.5in]{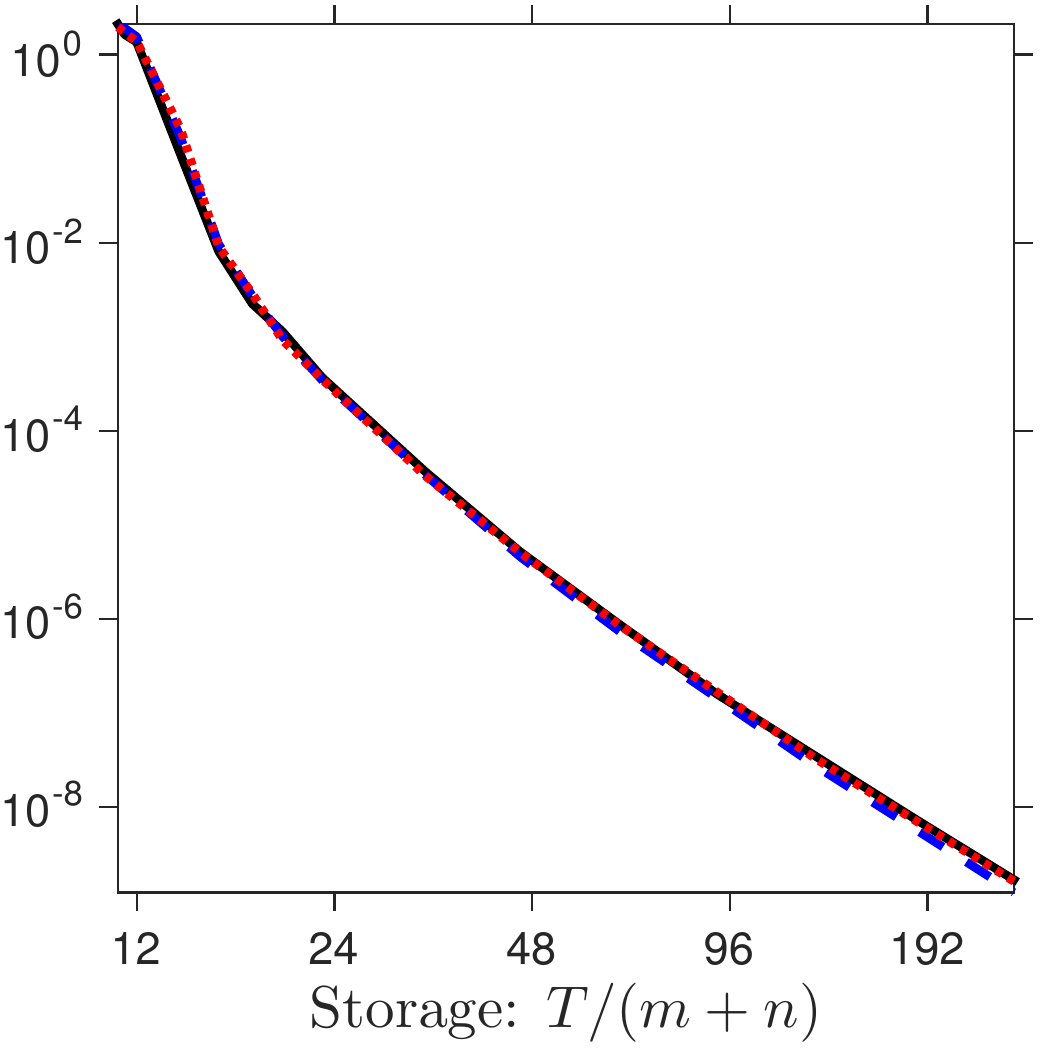}
\caption{\texttt{PolyDecayFast}}
\end{center}
\end{subfigure}
\end{center}

\vspace{0.5em}

\begin{center}
\begin{subfigure}{.325\textwidth}
\begin{center}
\includegraphics[height=1.5in]{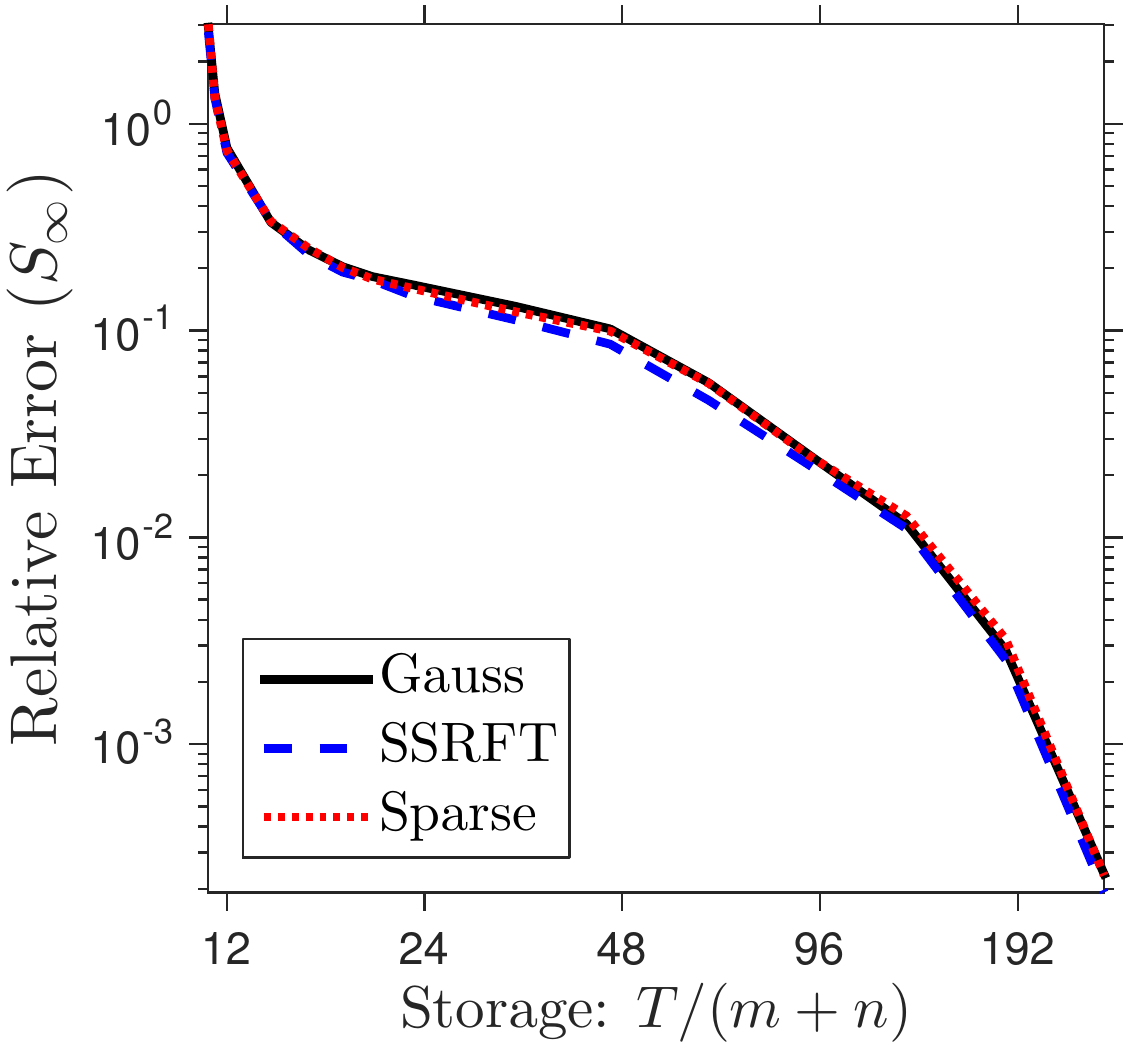}
\caption{\texttt{ExpDecaySlow}}
\end{center}
\end{subfigure}
\begin{subfigure}{.325\textwidth}
\begin{center}
\includegraphics[height=1.5in]{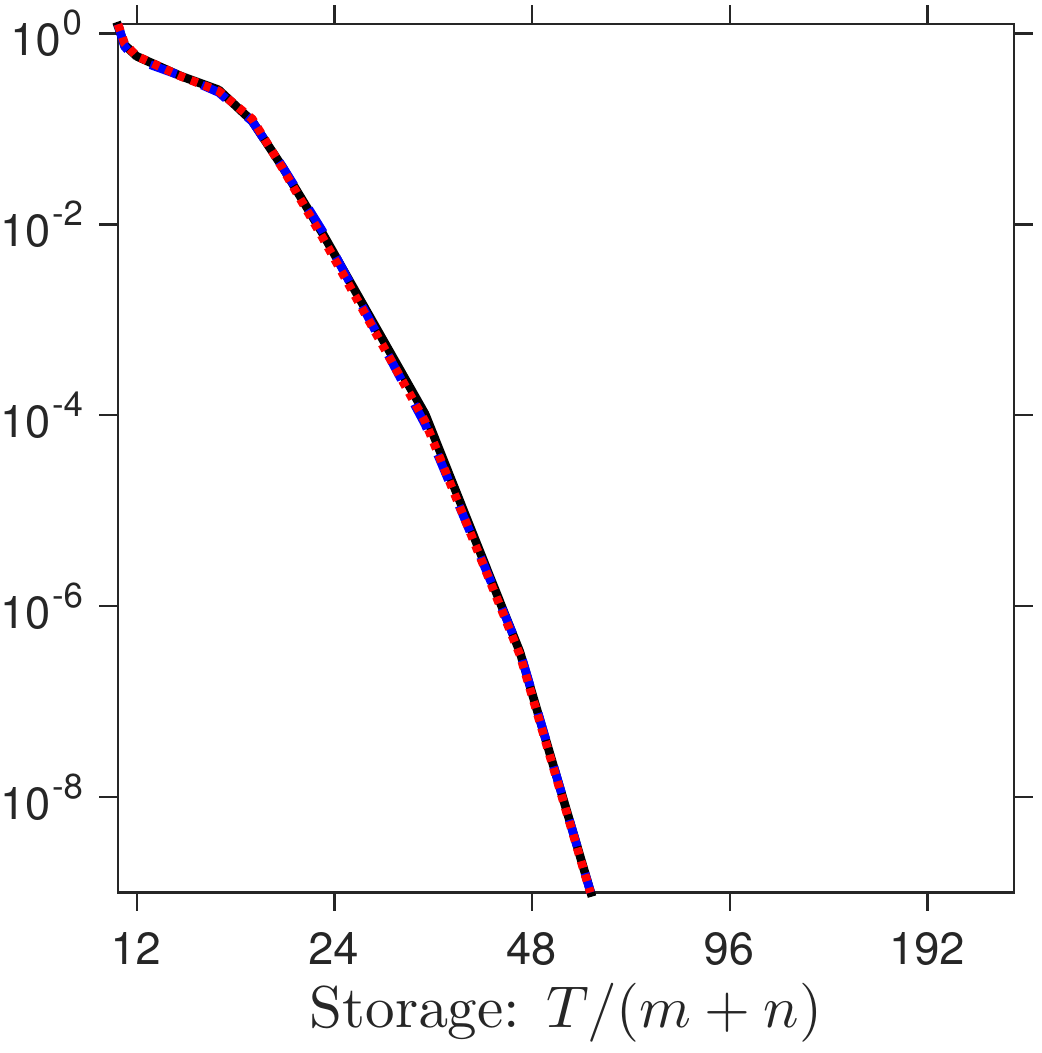}
\caption{\texttt{ExpDecayMed}}
\end{center}
\end{subfigure}
\begin{subfigure}{.325\textwidth}
\begin{center}
\includegraphics[height=1.5in]{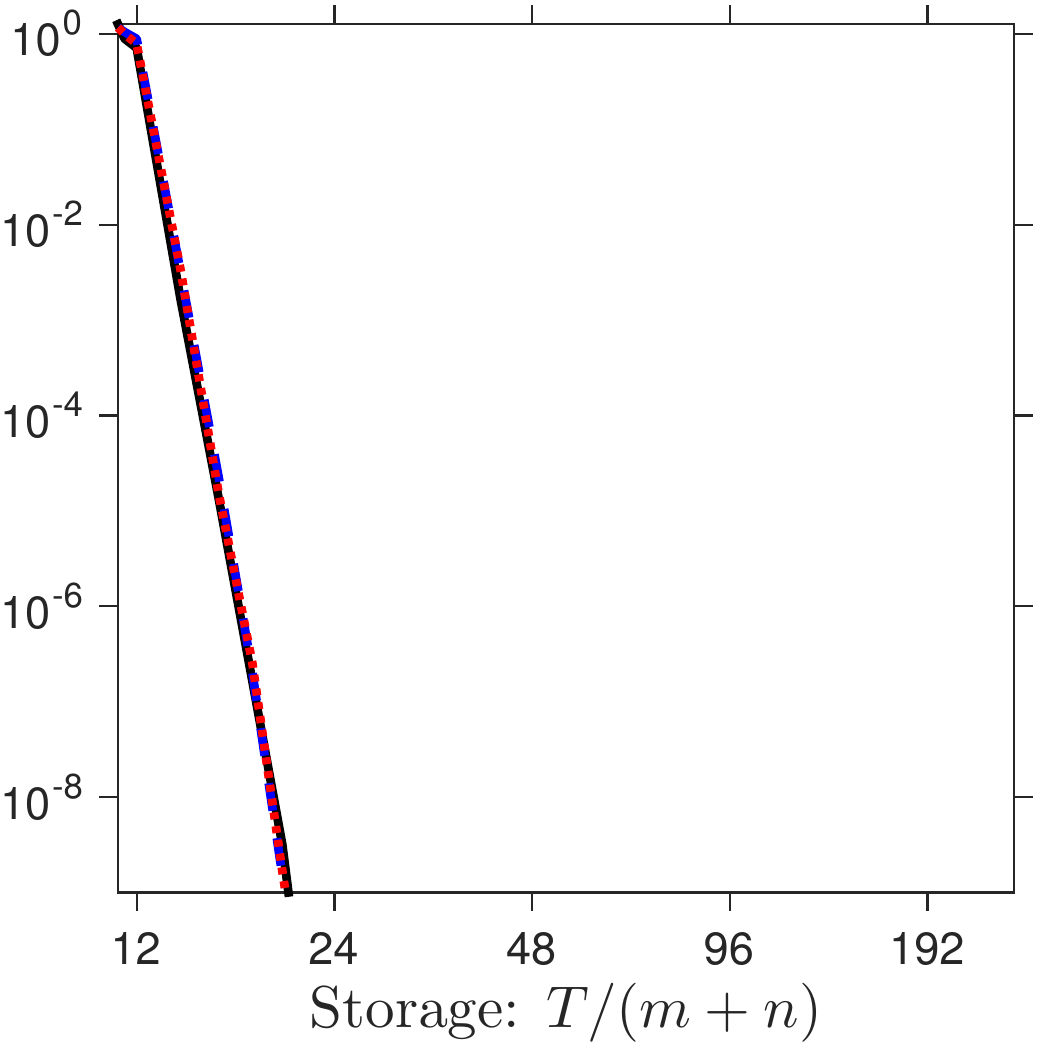}
\caption{\texttt{ExpDecayFast}}
\end{center}
\end{subfigure}
\end{center}

\vspace{0.5em}

\caption{\textbf{Insensitivity of proposed method to the dimension reduction map.}
(Effective rank $R = 10$, approximation rank $r = 10$, Schatten $\infty$-norm.)
We compare the oracle performance of the proposed fixed-rank
approximation~\cref{eqn:Ahat-fixed} implemented with Gaussian, SSRFT, or sparse
dimension reduction maps.  See~\cref{app:universality}
for details.}
\label{fig:universality-R10-Sinf}
\end{figure}

\begin{figure}[htp!]
\begin{center}
\begin{subfigure}{.325\textwidth}
\begin{center}
\includegraphics[height=1.5in]{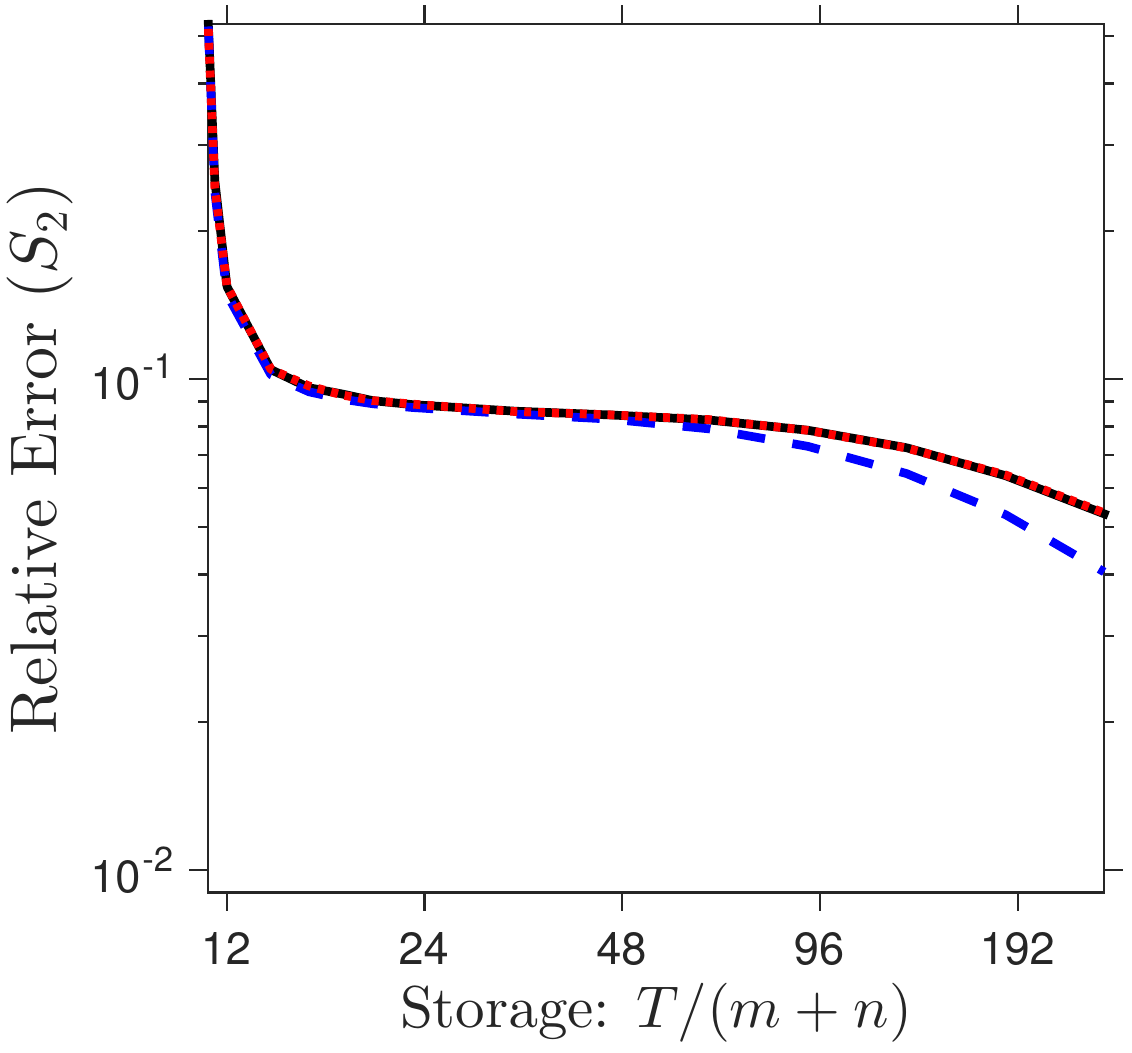}
\caption{\texttt{LowRankHiNoise}}
\end{center}
\end{subfigure}
\begin{subfigure}{.325\textwidth}
\begin{center}
\includegraphics[height=1.5in]{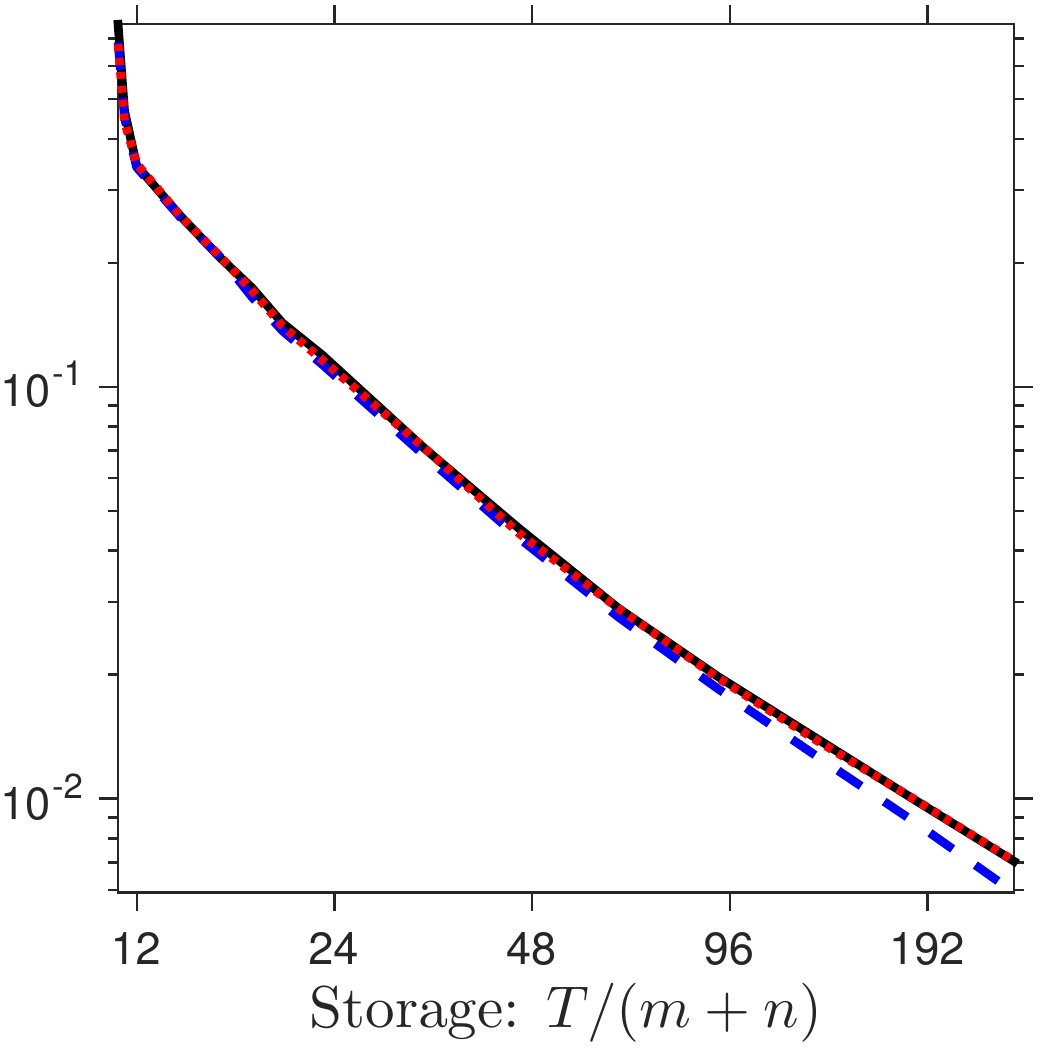}
\caption{\texttt{LowRankMedNoise}}
\end{center}
\end{subfigure}
\begin{subfigure}{.325\textwidth}
\begin{center}
\includegraphics[height=1.5in]{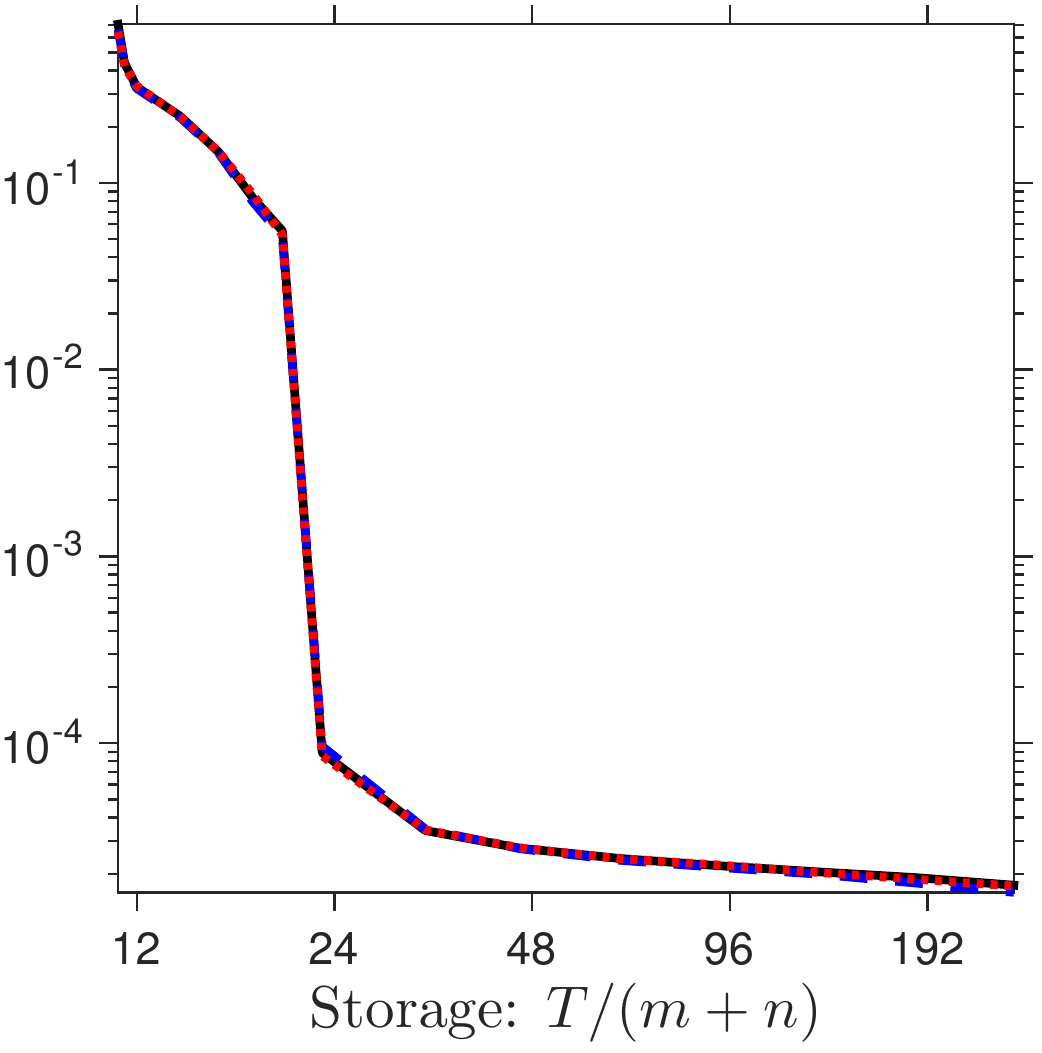}
\caption{\texttt{LowRankLowNoise}}
\end{center}
\end{subfigure}
\end{center}

\vspace{.5em}

\begin{center}
\begin{subfigure}{.325\textwidth}
\begin{center}
\includegraphics[height=1.5in]{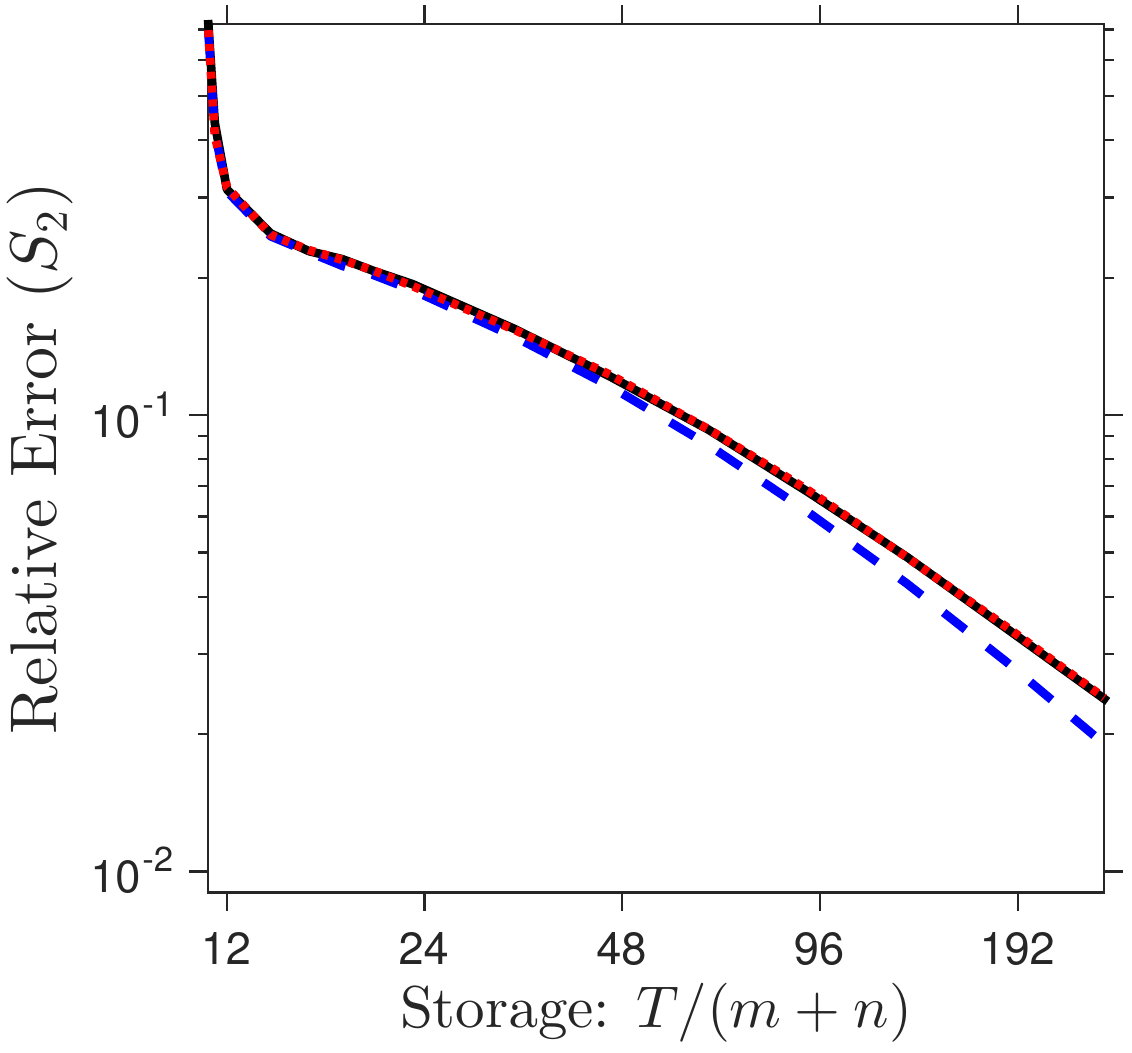}
\caption{\texttt{PolyDecaySlow}}
\end{center}
\end{subfigure}
\begin{subfigure}{.325\textwidth}
\begin{center}
\includegraphics[height=1.5in]{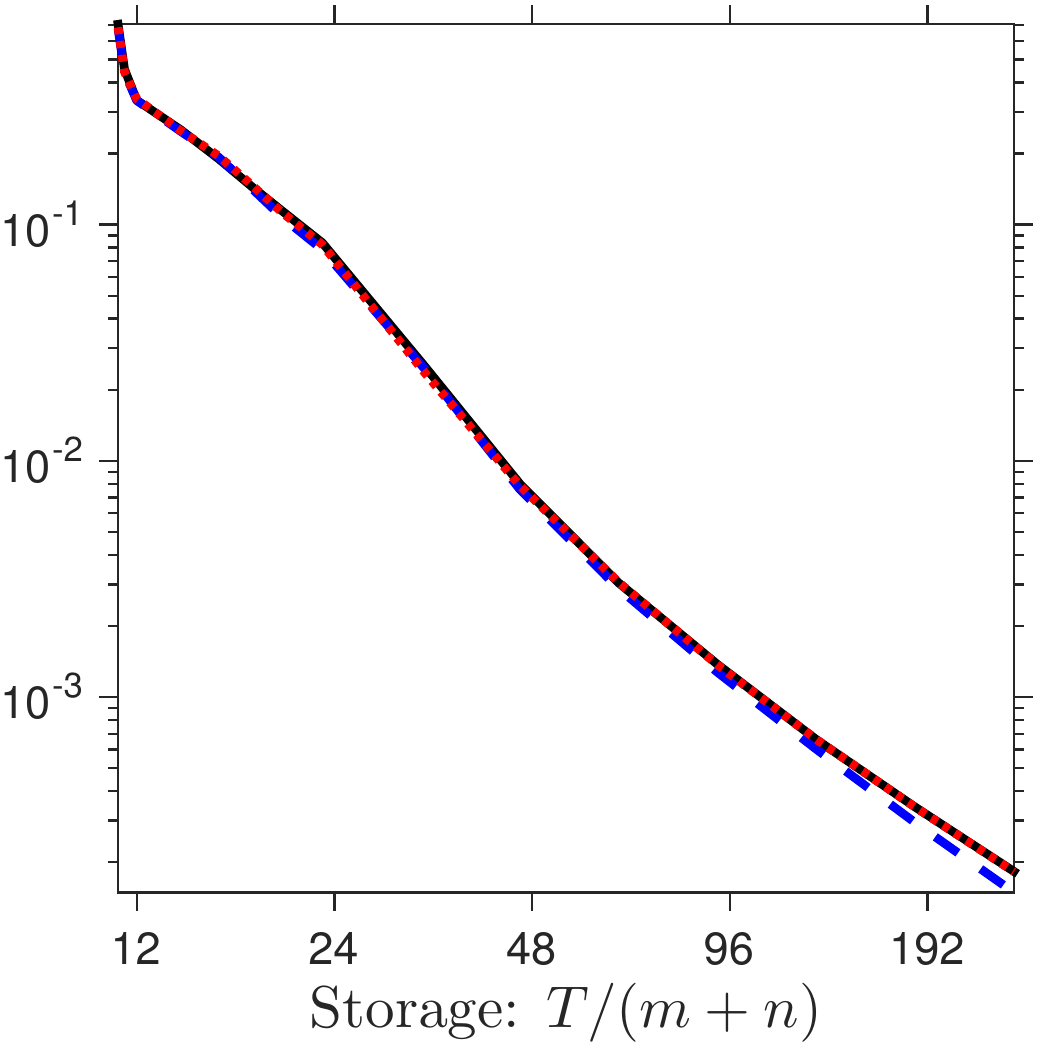}
\caption{\texttt{PolyDecayMed}}
\end{center}
\end{subfigure}
\begin{subfigure}{.325\textwidth}
\begin{center}
\includegraphics[height=1.5in]{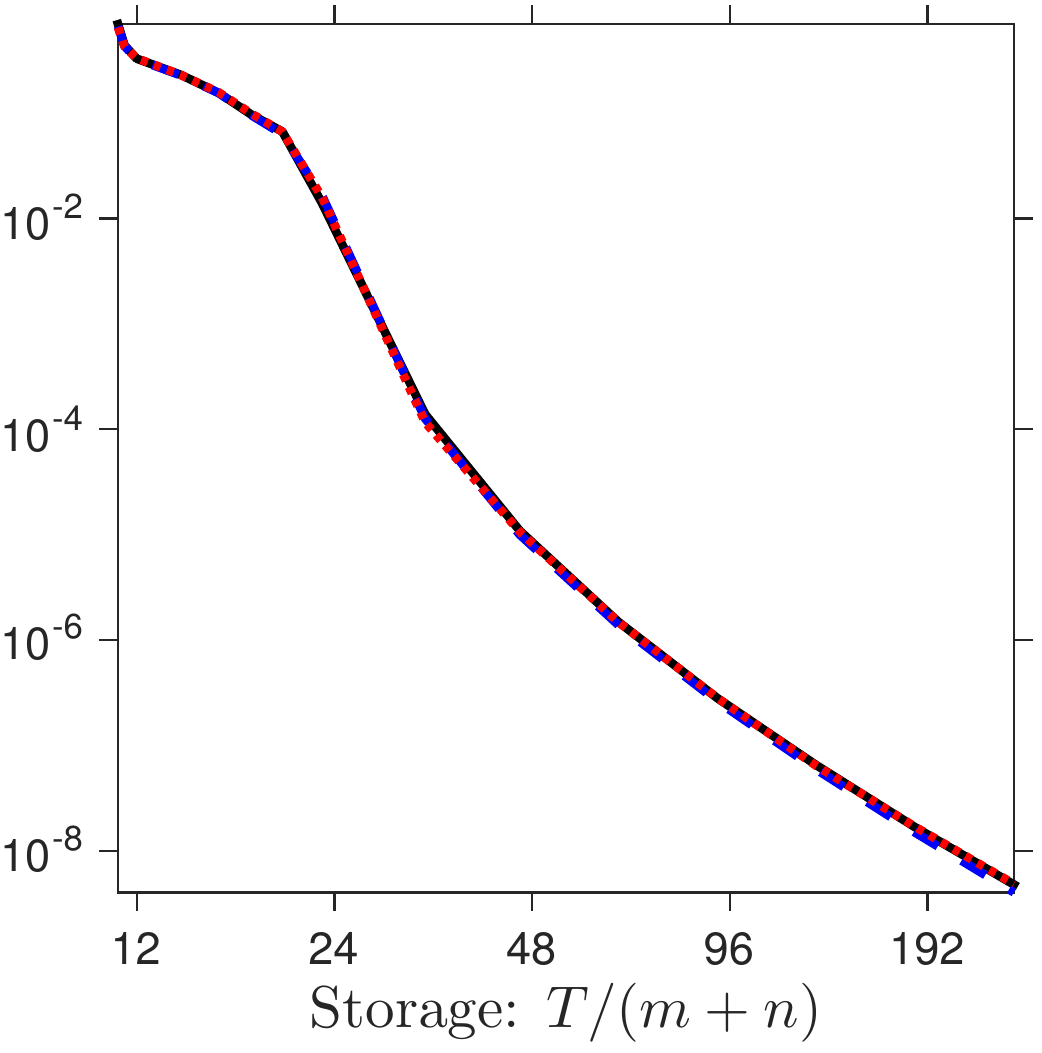}
\caption{\texttt{PolyDecayFast}}
\end{center}
\end{subfigure}
\end{center}

\vspace{0.5em}

\begin{center}
\begin{subfigure}{.325\textwidth}
\begin{center}
\includegraphics[height=1.5in]{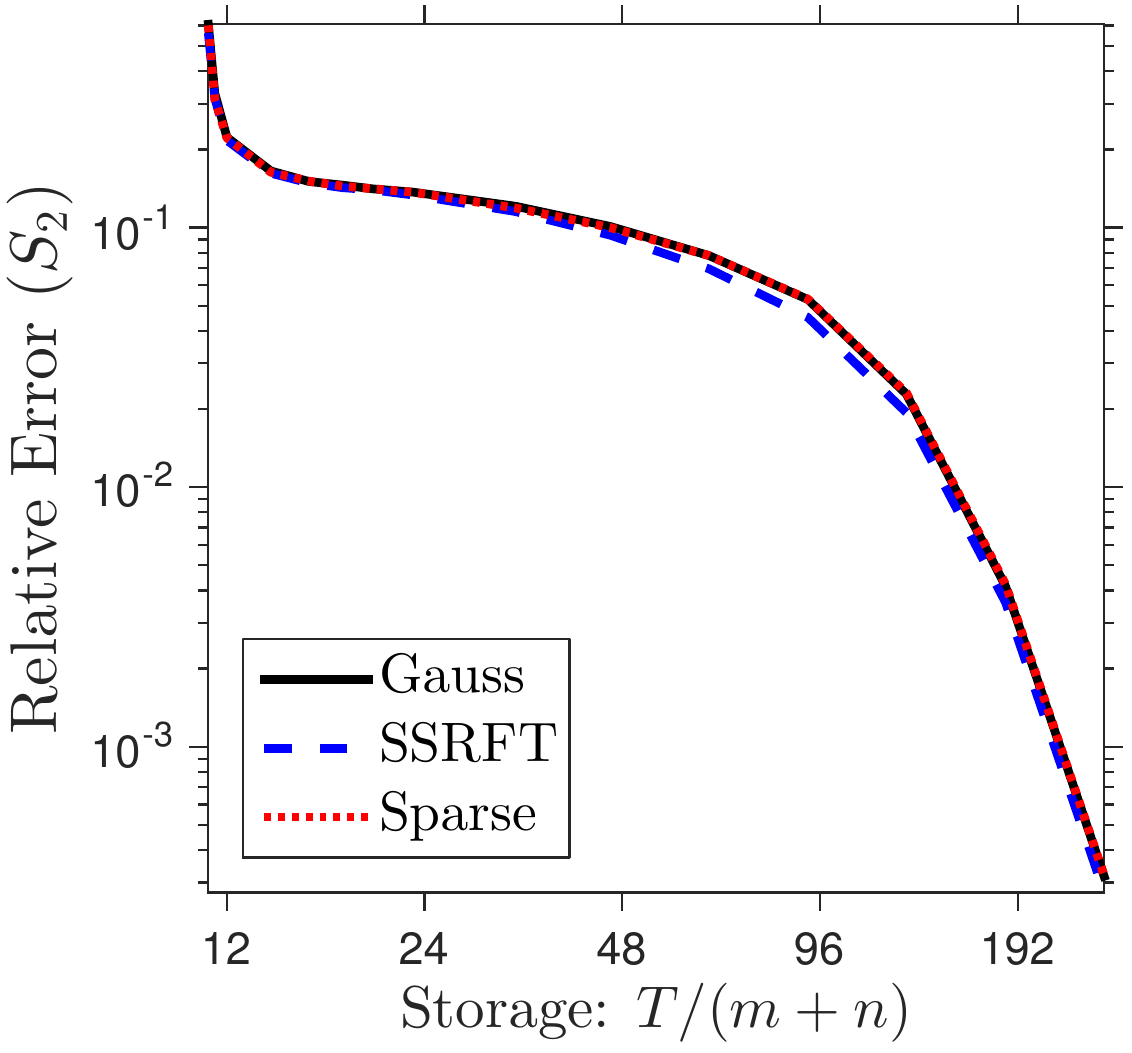}
\caption{\texttt{ExpDecaySlow}}
\end{center}
\end{subfigure}
\begin{subfigure}{.325\textwidth}
\begin{center}
\includegraphics[height=1.5in]{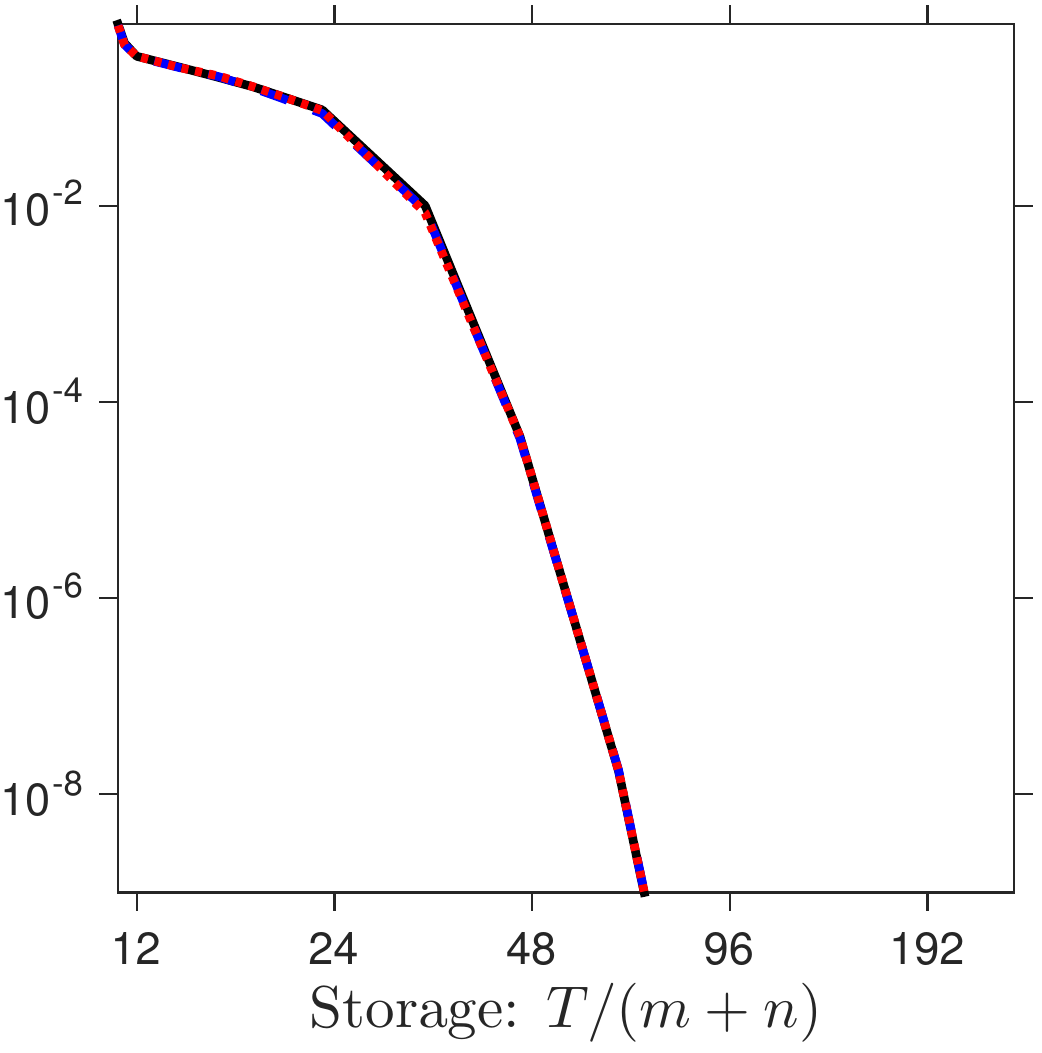}
\caption{\texttt{ExpDecayMed}}
\end{center}
\end{subfigure}
\begin{subfigure}{.325\textwidth}
\begin{center}
\includegraphics[height=1.5in]{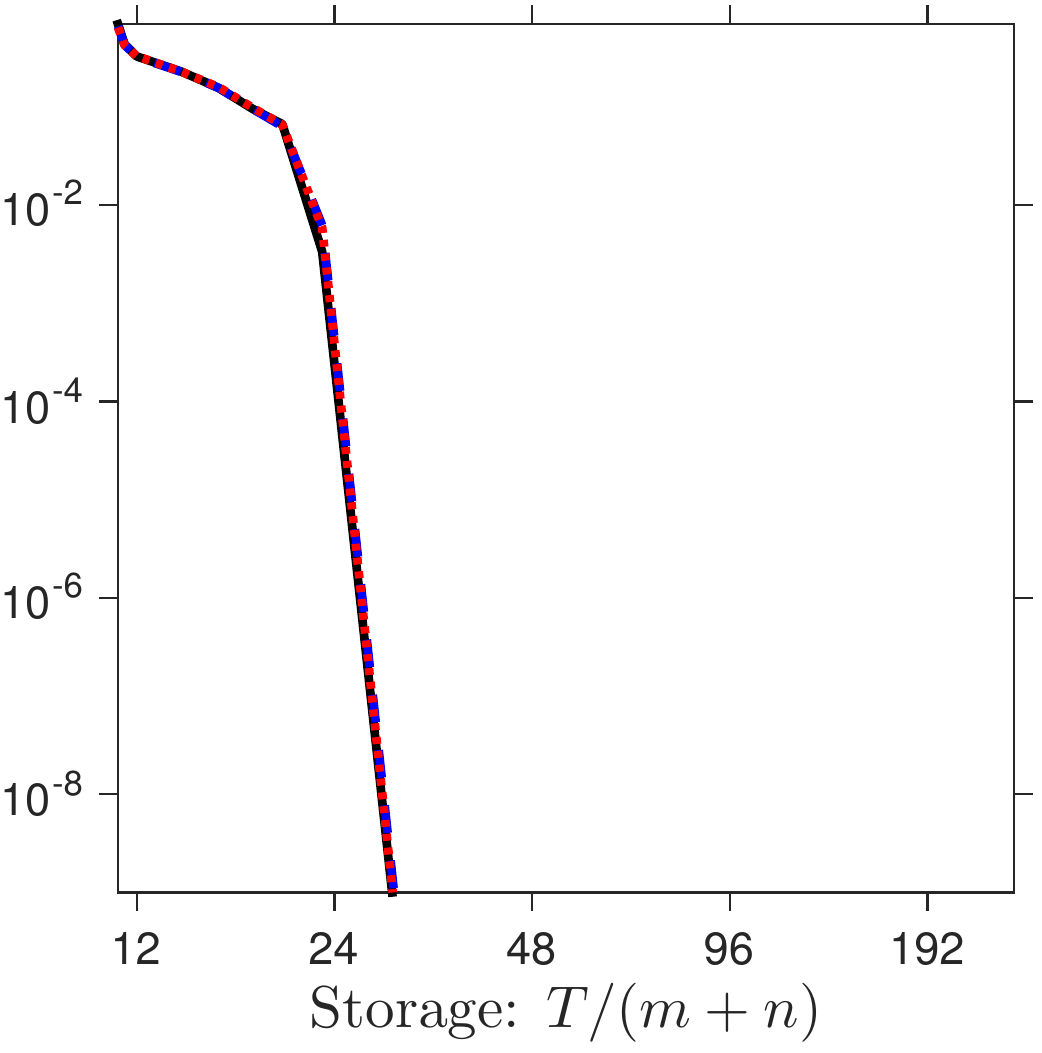}
\caption{\texttt{ExpDecayFast}}
\end{center}
\end{subfigure}
\end{center}

\vspace{0.5em}

\caption{\textbf{Insensitivity of proposed method to the dimension reduction map.}
(Effective rank $R = 20$, approximation rank $r = 10$, Schatten 2-norm.)
We compare the oracle performance of the proposed fixed-rank
approximation~\cref{eqn:Ahat-fixed} implemented with Gaussian, SSRFT, or sparse
dimension reduction maps.  See~\cref{app:universality}
for details.}
\label{fig:universality-R20-S2}
\end{figure}

\begin{figure}[htp!]
\begin{center}
\begin{subfigure}{.325\textwidth}
\begin{center}
\includegraphics[height=1.5in]{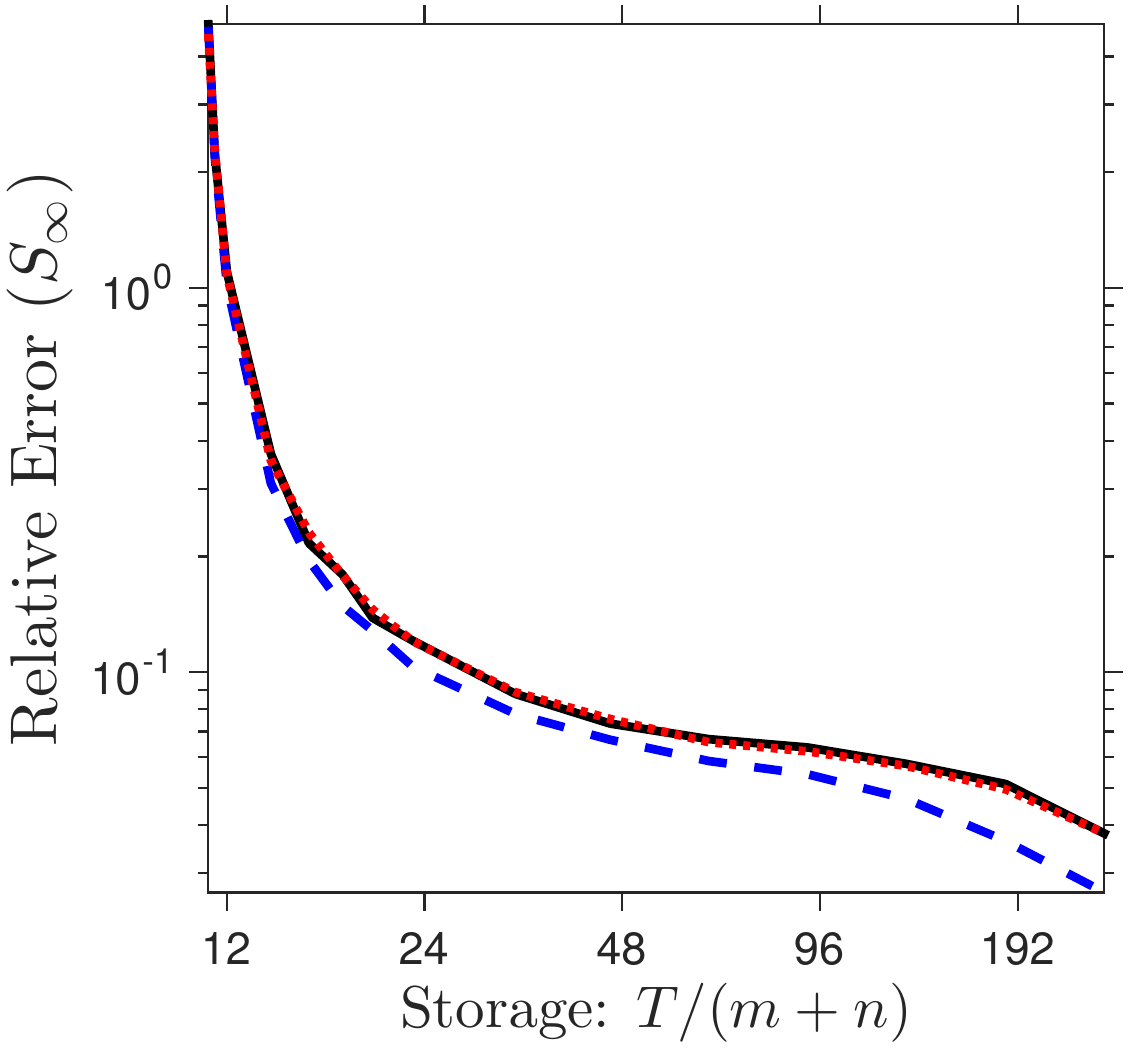}
\caption{\texttt{LowRankHiNoise}}
\end{center}
\end{subfigure}
\begin{subfigure}{.325\textwidth}
\begin{center}
\includegraphics[height=1.5in]{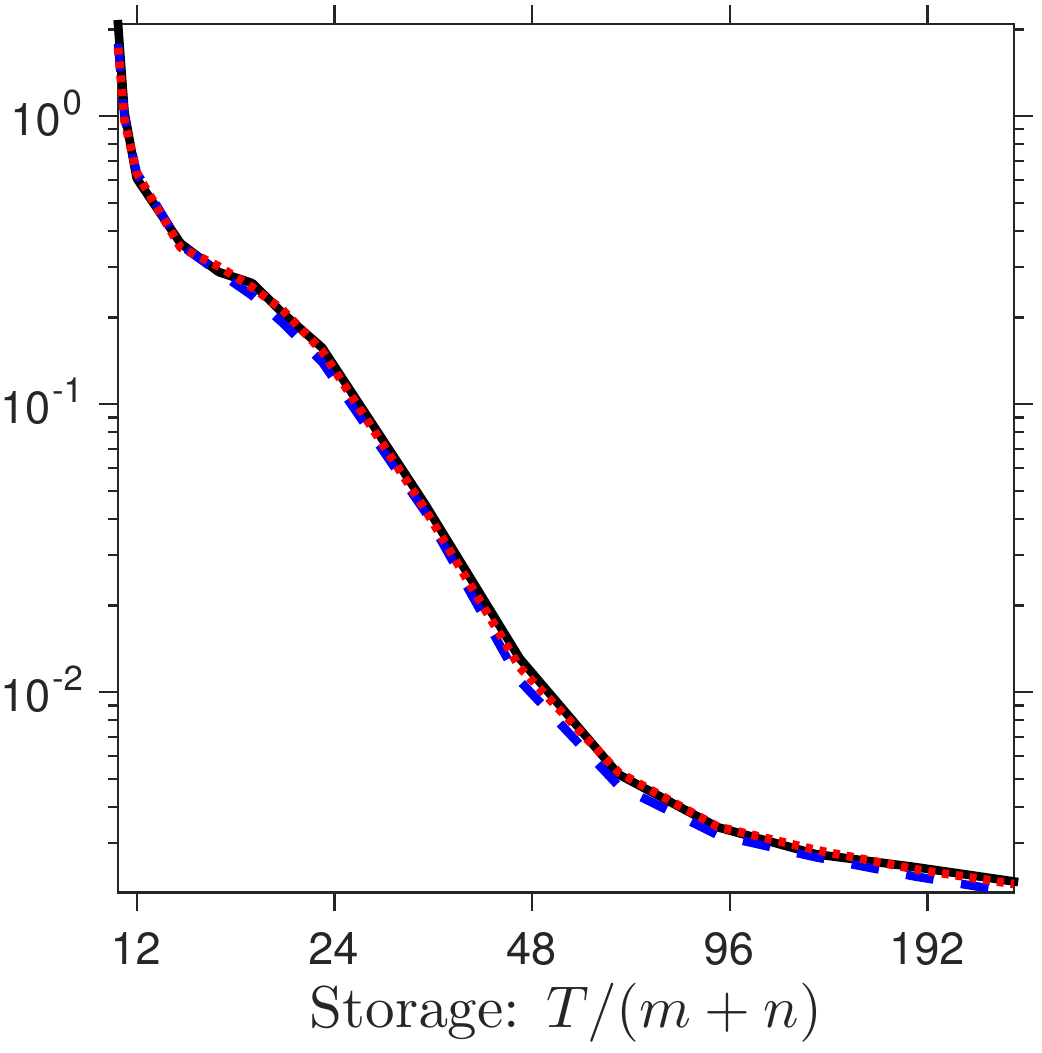}
\caption{\texttt{LowRankMedNoise}}
\end{center}
\end{subfigure}
\begin{subfigure}{.325\textwidth}
\begin{center}
\includegraphics[height=1.5in]{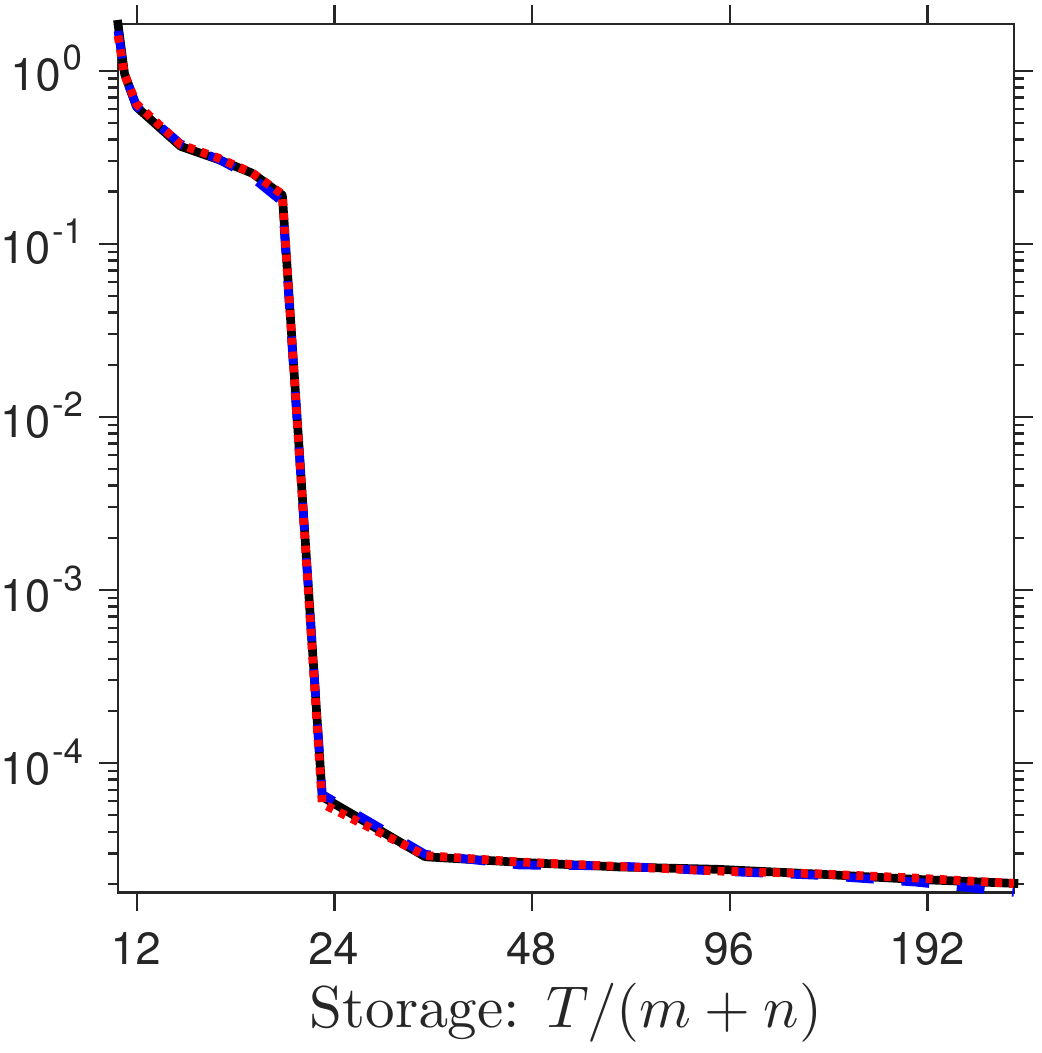}
\caption{\texttt{LowRankLowNoise}}
\end{center}
\end{subfigure}
\end{center}

\vspace{.5em}

\begin{center}
\begin{subfigure}{.325\textwidth}
\begin{center}
\includegraphics[height=1.5in]{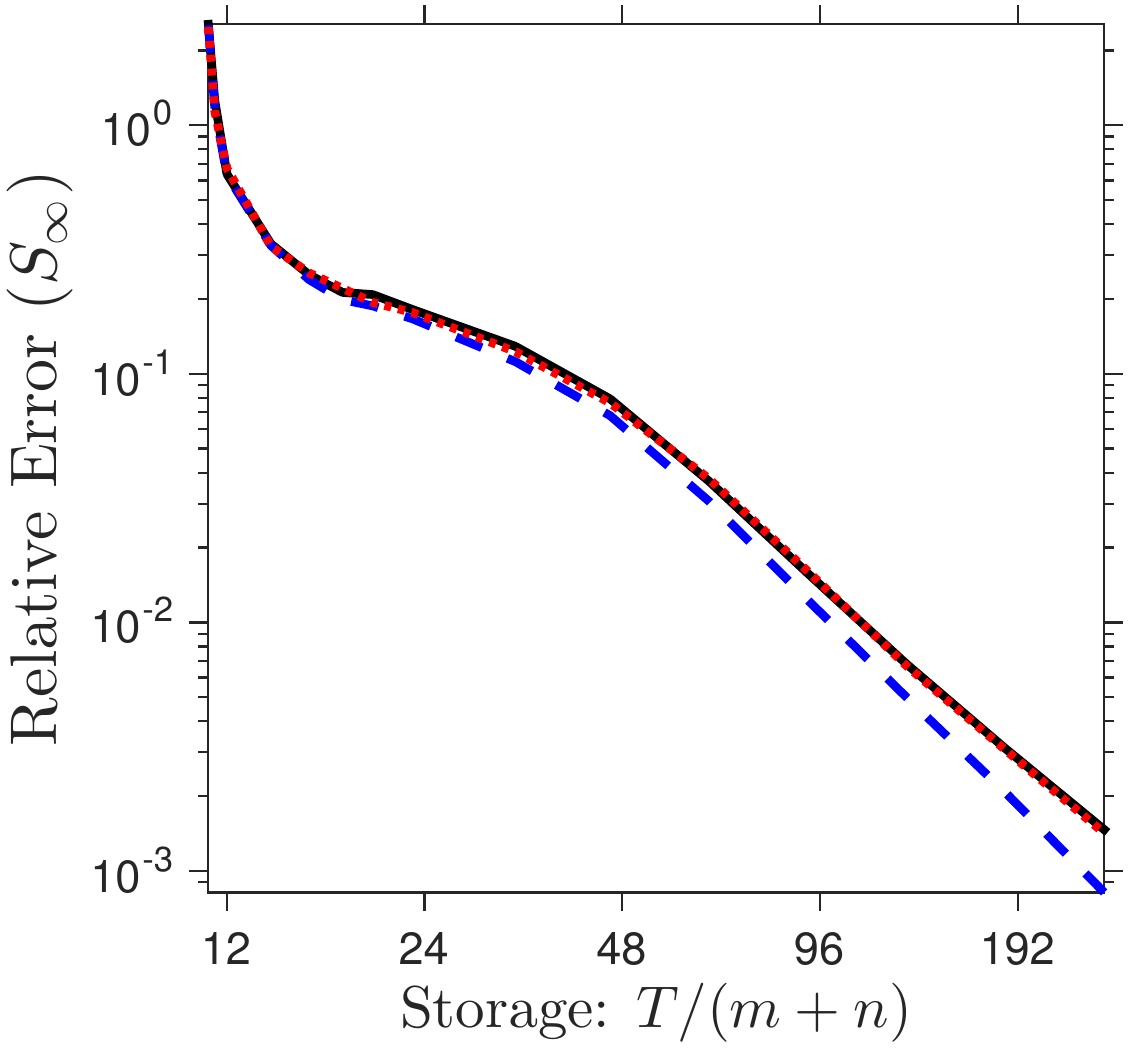}
\caption{\texttt{PolyDecaySlow}}
\end{center}
\end{subfigure}
\begin{subfigure}{.325\textwidth}
\begin{center}
\includegraphics[height=1.5in]{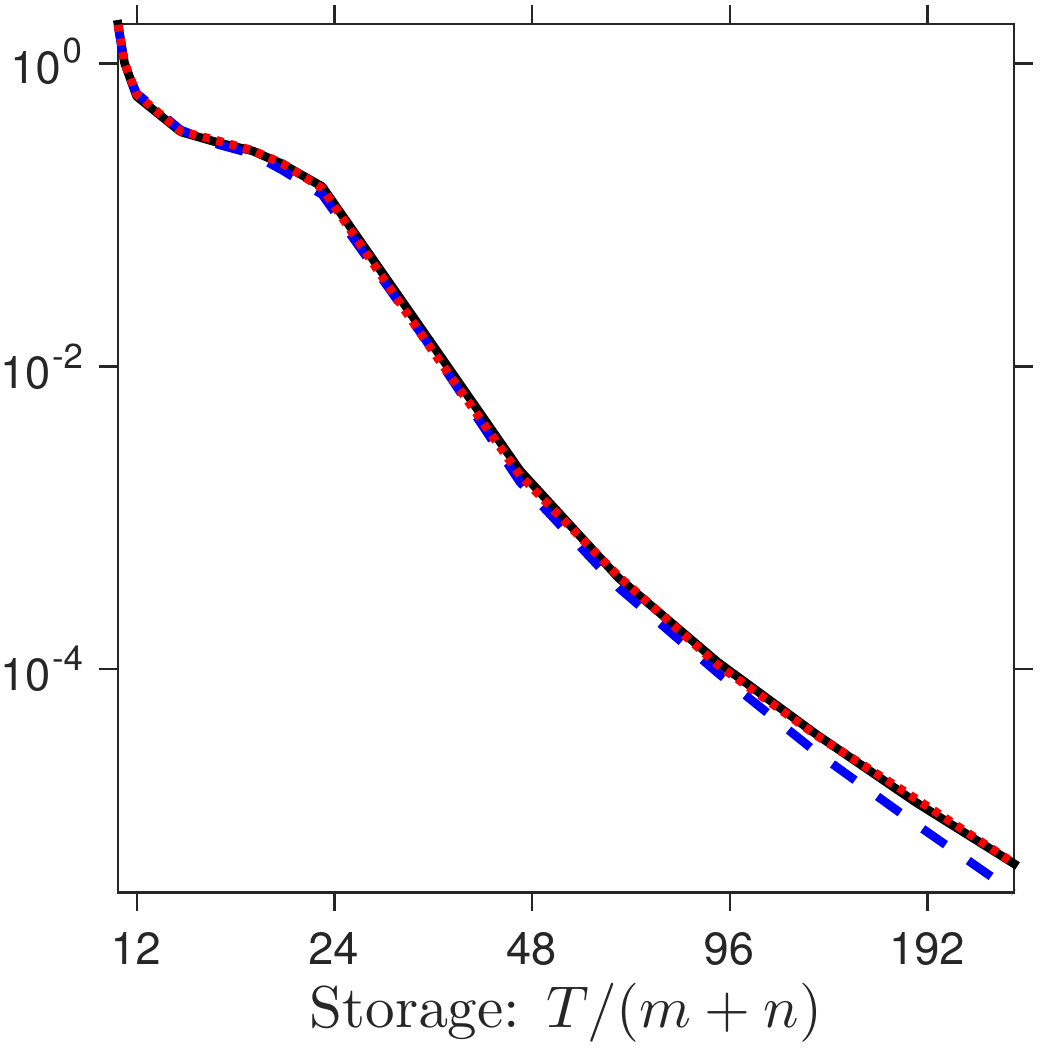}
\caption{\texttt{PolyDecayMed}}
\end{center}
\end{subfigure}
\begin{subfigure}{.325\textwidth}
\begin{center}
\includegraphics[height=1.5in]{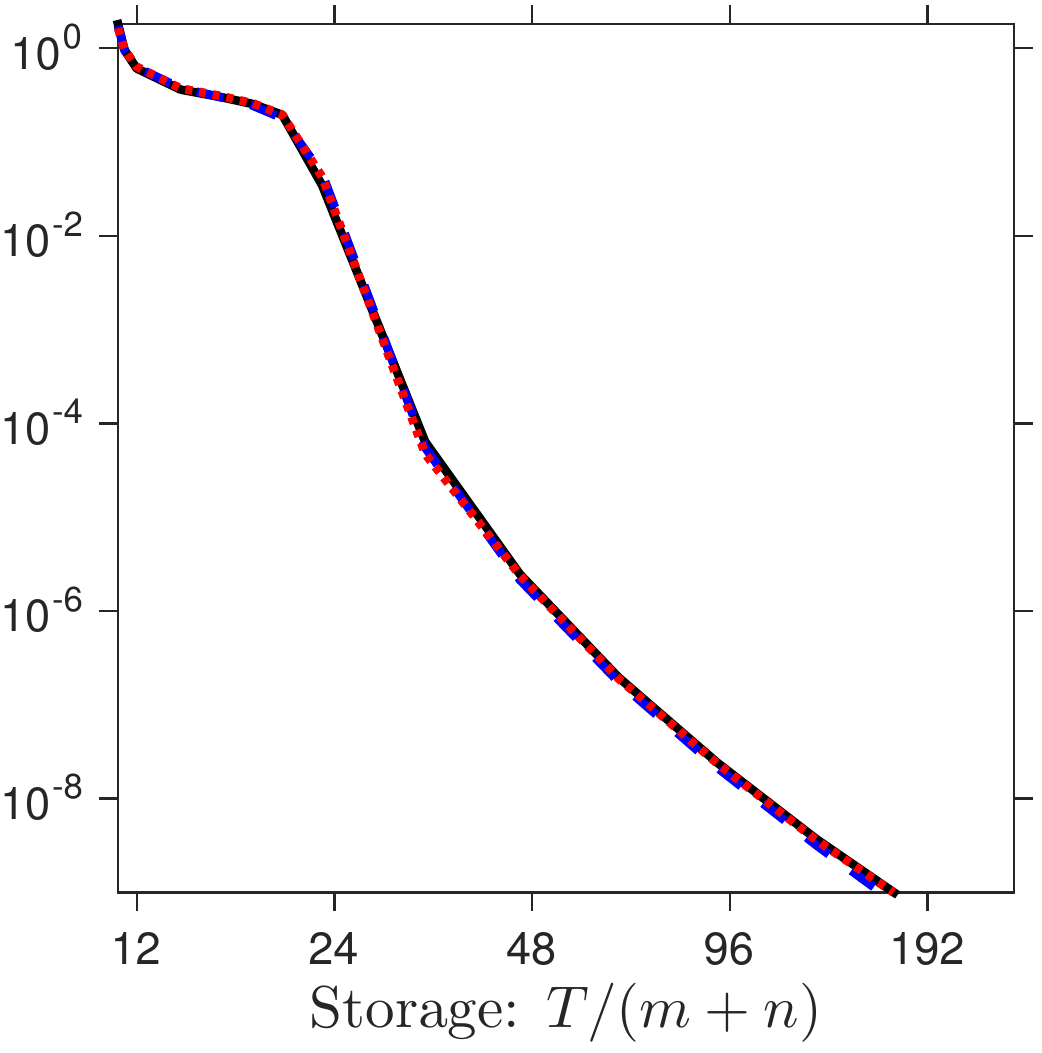}
\caption{\texttt{PolyDecayFast}}
\end{center}
\end{subfigure}
\end{center}

\vspace{0.5em}

\begin{center}
\begin{subfigure}{.325\textwidth}
\begin{center}
\includegraphics[height=1.5in]{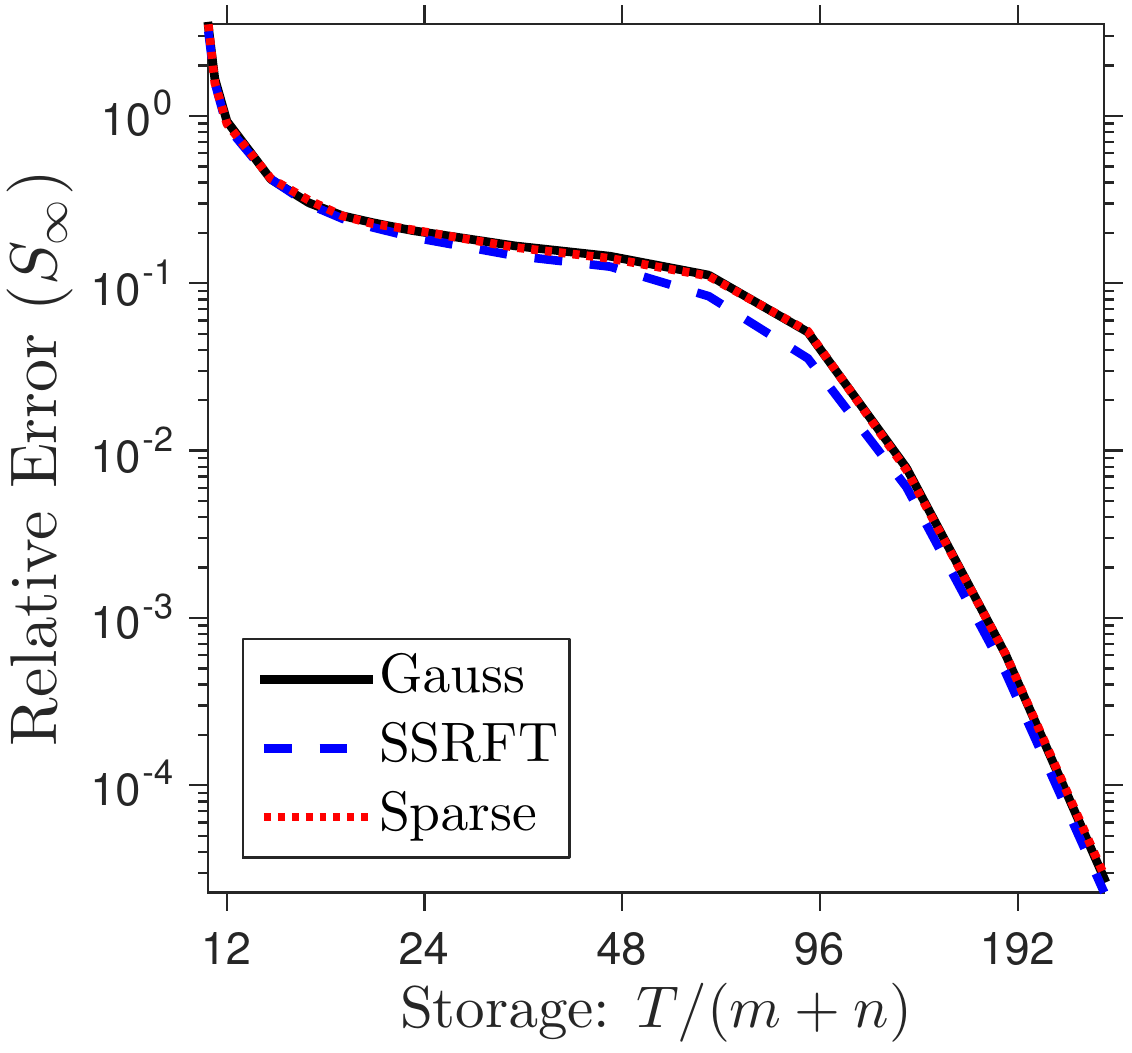}
\caption{\texttt{ExpDecaySlow}}
\end{center}
\end{subfigure}
\begin{subfigure}{.325\textwidth}
\begin{center}
\includegraphics[height=1.5in]{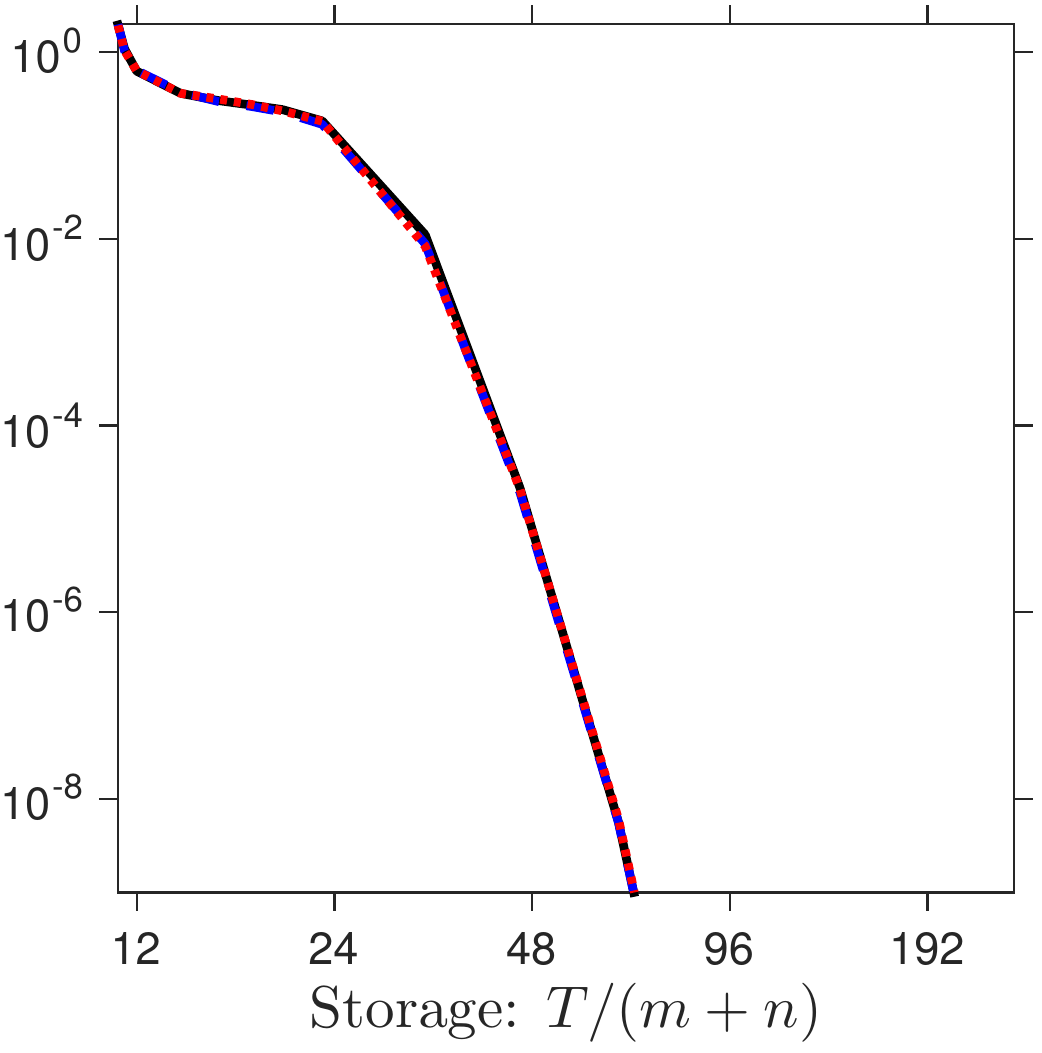}
\caption{\texttt{ExpDecayMed}}
\end{center}
\end{subfigure}
\begin{subfigure}{.325\textwidth}
\begin{center}
\includegraphics[height=1.5in]{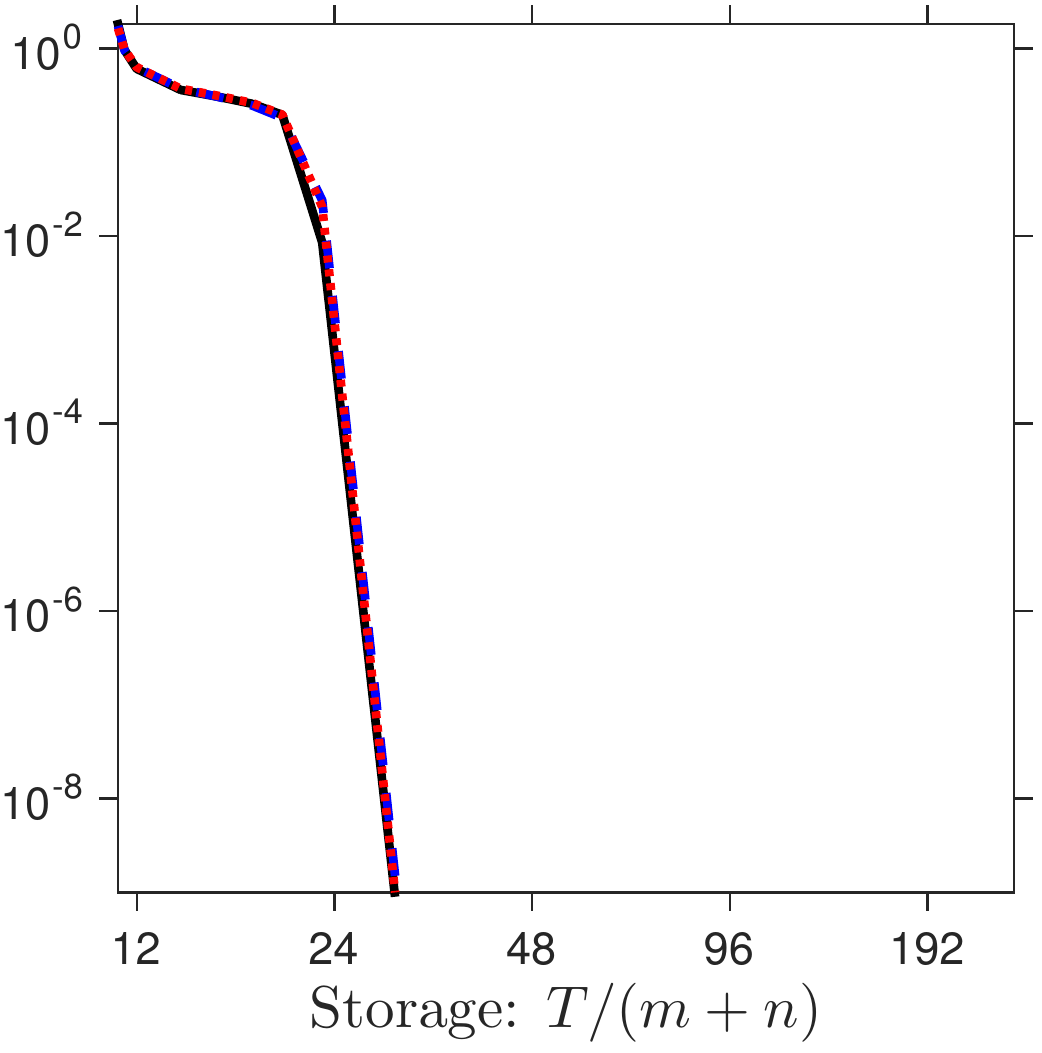}
\caption{\texttt{ExpDecayFast}}
\end{center}
\end{subfigure}
\end{center}

\vspace{0.5em}

\caption{\textbf{Insensitivity of proposed method to the dimension reduction map.}
(Effective rank $R = 20$, approximation rank $r = 10$, Schatten $\infty$-norm.)
We compare the oracle performance of the proposed fixed-rank
approximation~\cref{eqn:Ahat-fixed} implemented with Gaussian, SSRFT, or sparse
dimension reduction maps.  See~\cref{app:universality}
for details.}
\label{fig:universality-R20-Sinf}
\end{figure}

\begin{figure}[htp!]
\vspace{0.5in}
\begin{center}
\begin{subfigure}{.325\textwidth}
\begin{center}
\includegraphics[height=1.5in]{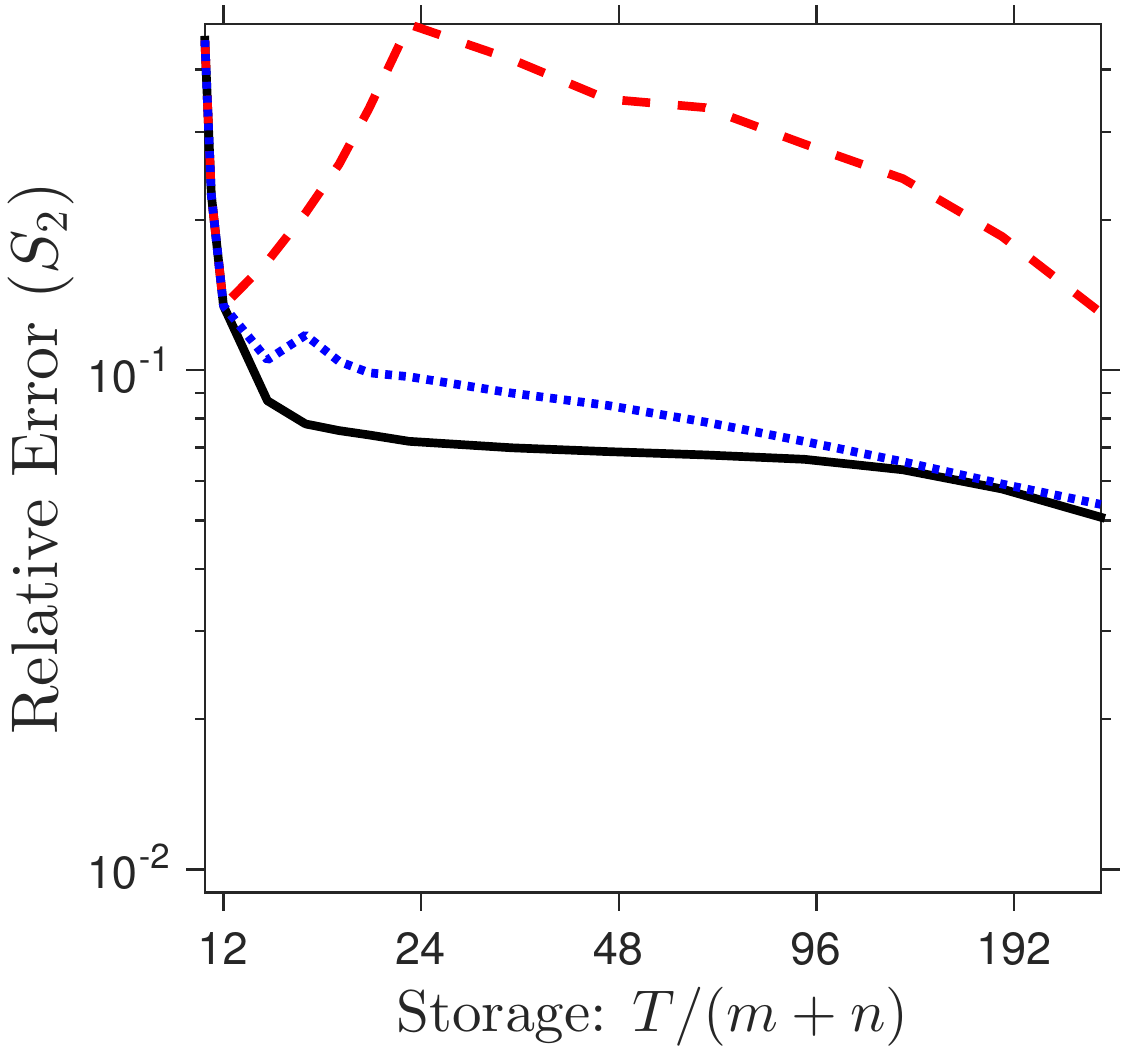}
\caption{\texttt{LowRankHiNoise}}
\end{center}
\end{subfigure}
\begin{subfigure}{.325\textwidth}
\begin{center}
\includegraphics[height=1.5in]{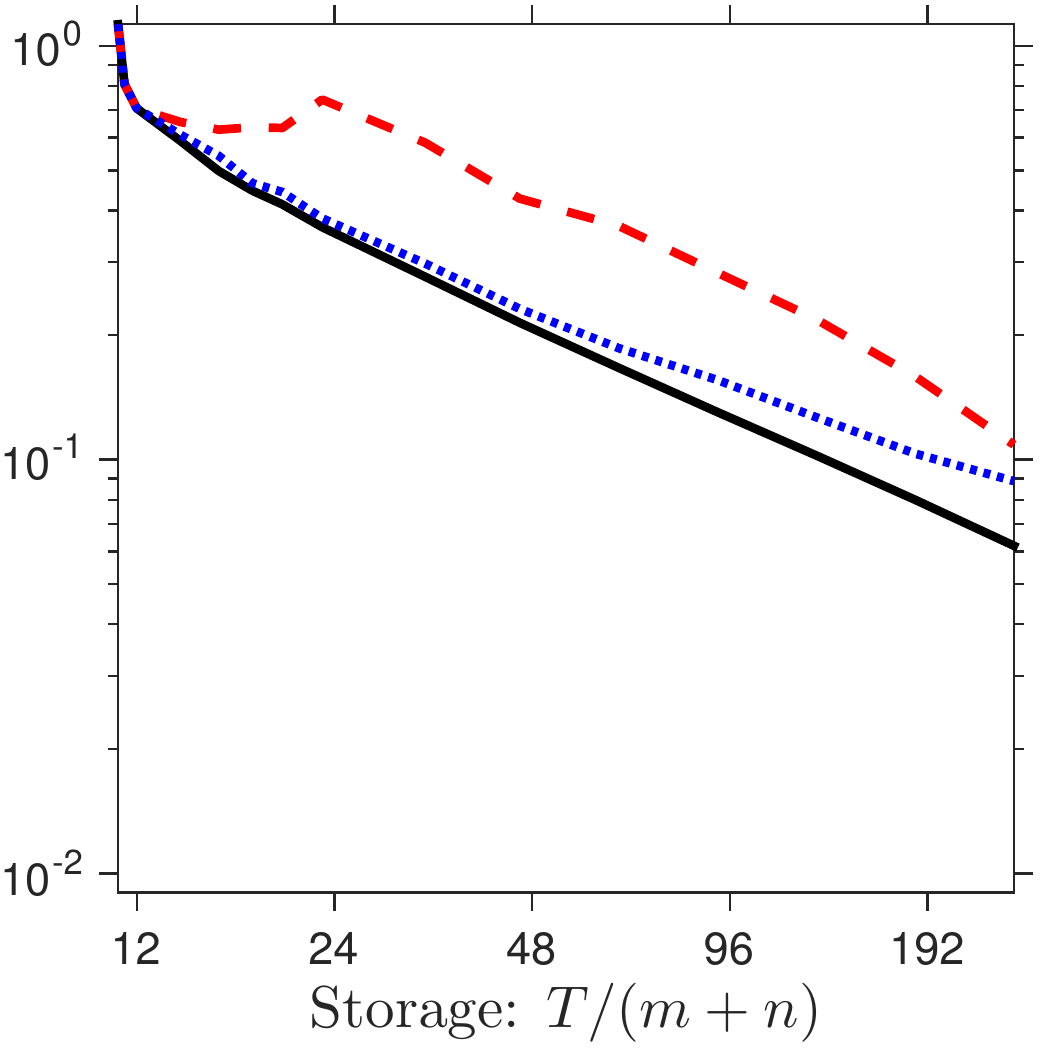}
\caption{\texttt{LowRankMedNoise}}
\end{center}
\end{subfigure}
\begin{subfigure}{.325\textwidth}
\begin{center}
\includegraphics[height=1.5in]{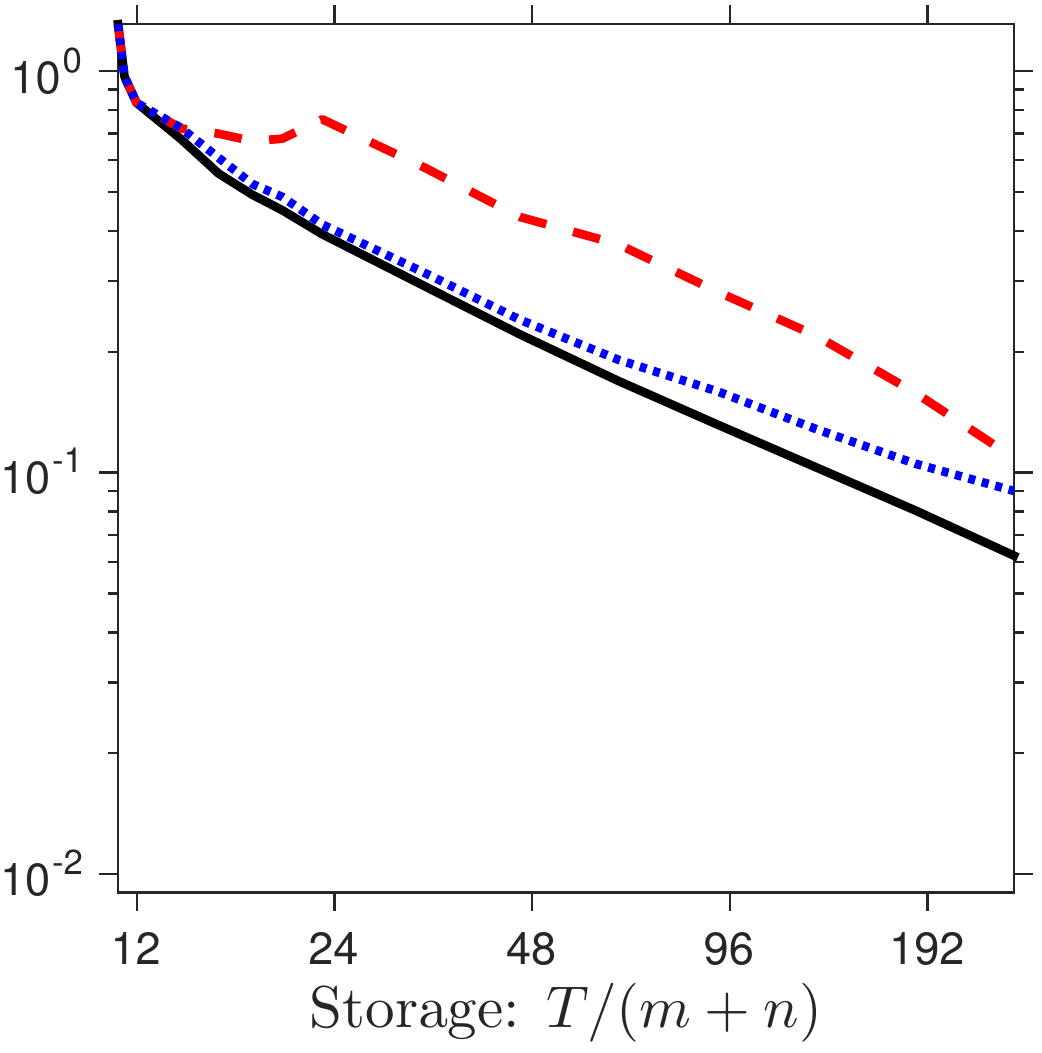}
\caption{\texttt{LowRankLowNoise}}
\end{center}
\end{subfigure}
\end{center}

\vspace{.5em}

\begin{center}
\begin{subfigure}{.325\textwidth}
\begin{center}
\includegraphics[height=1.5in]{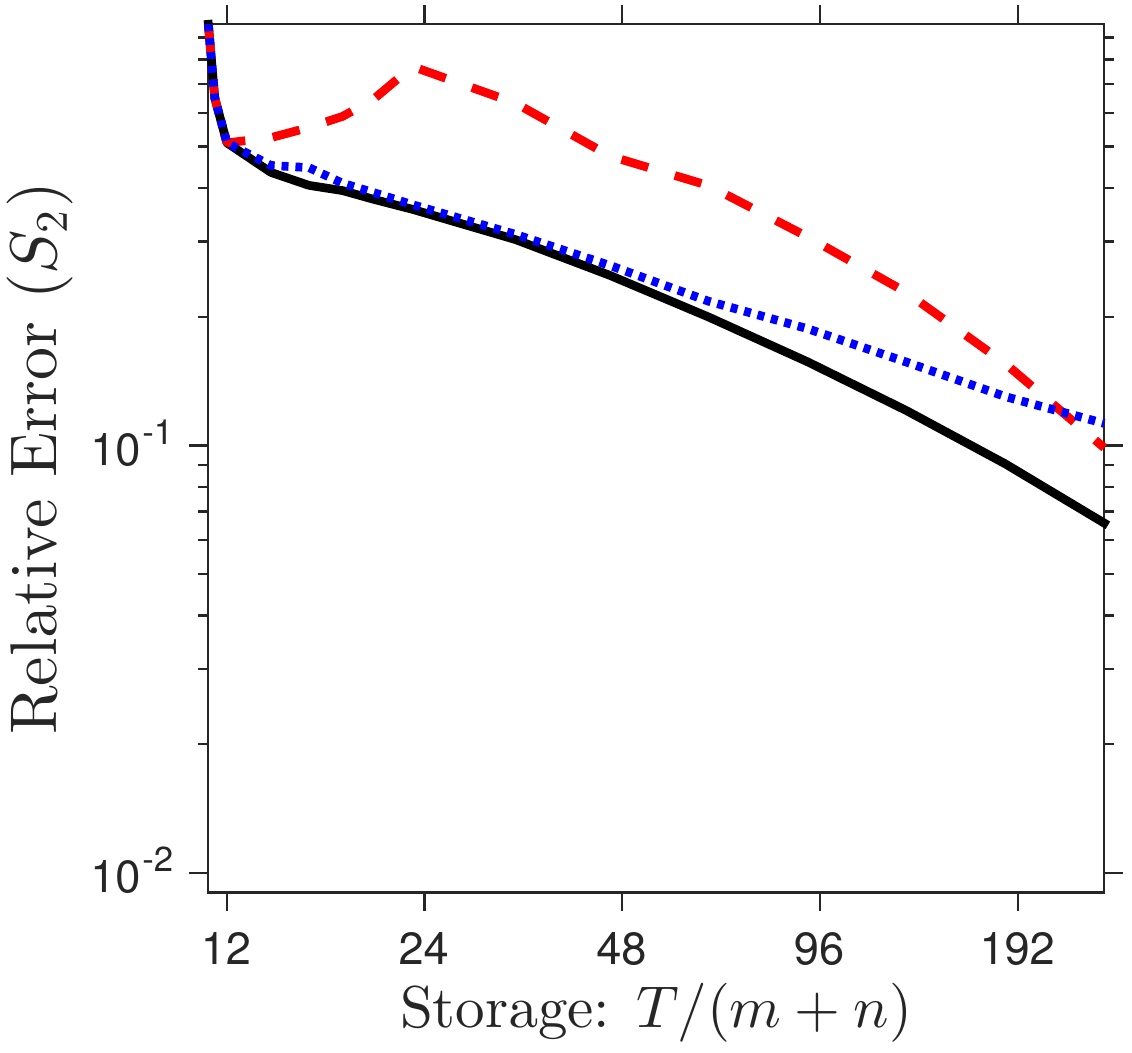}
\caption{\texttt{PolyDecaySlow}}
\end{center}
\end{subfigure}
\begin{subfigure}{.325\textwidth}
\begin{center}
\includegraphics[height=1.5in]{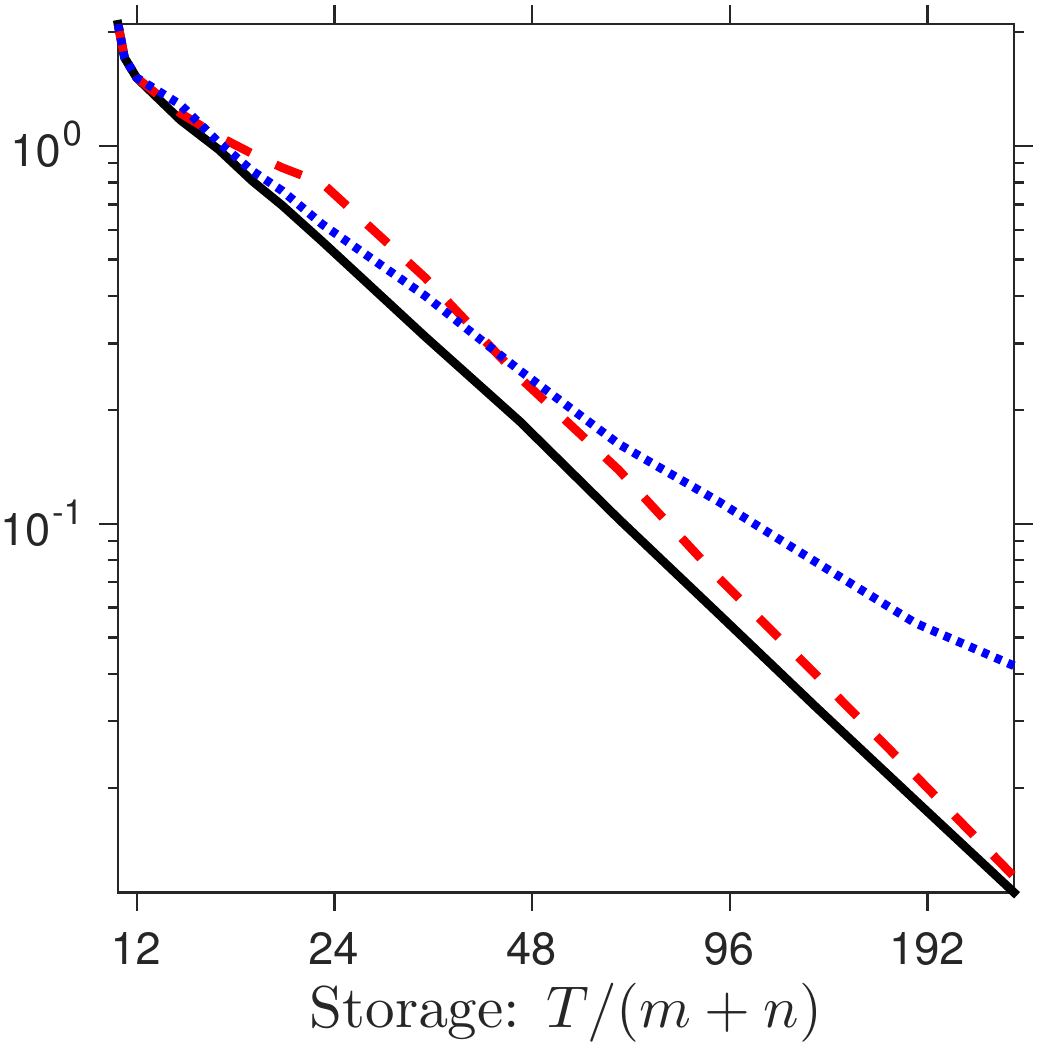}
\caption{\texttt{PolyDecayMed}}
\end{center}
\end{subfigure}
\begin{subfigure}{.325\textwidth}
\begin{center}
\includegraphics[height=1.5in]{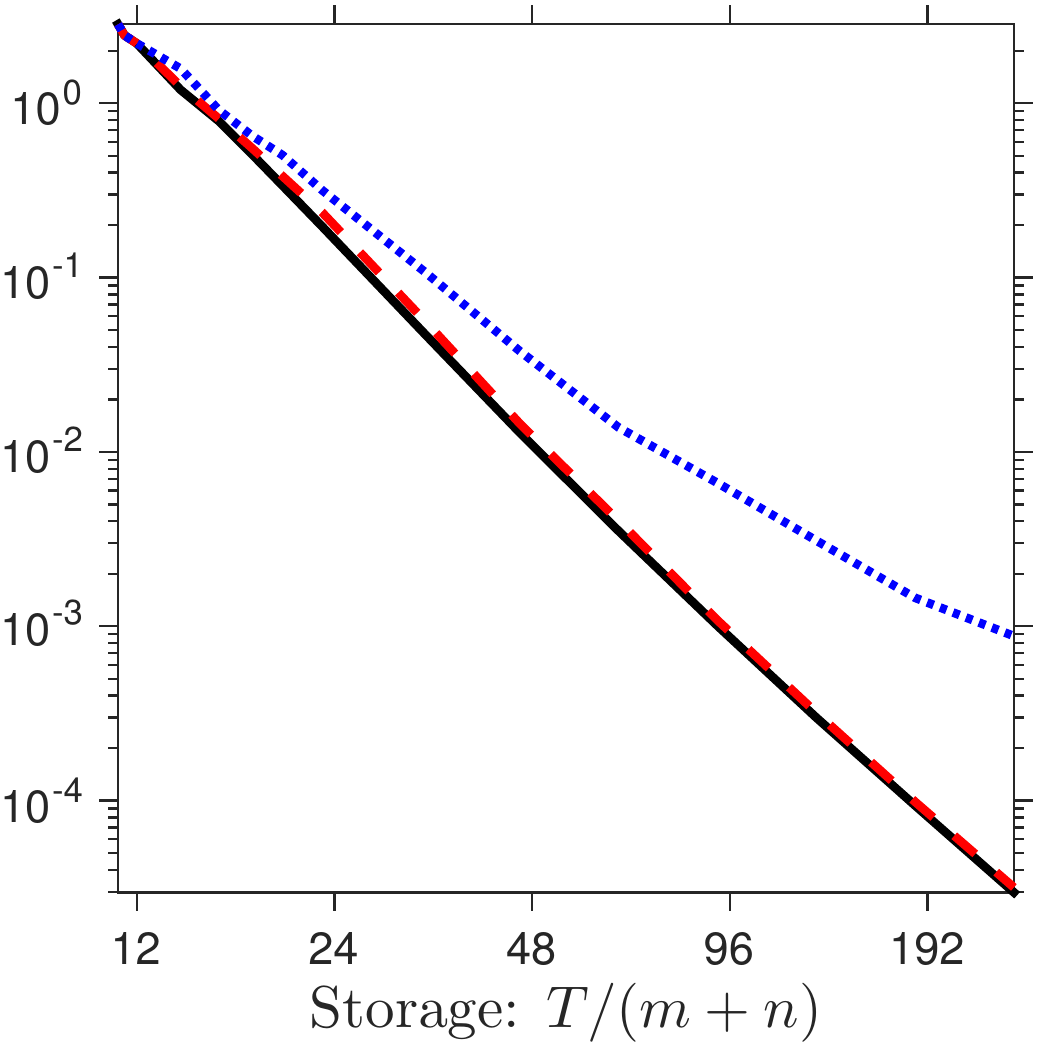}
\caption{\texttt{PolyDecayFast}}
\end{center}
\end{subfigure}
\end{center}

\vspace{0.5em}

\begin{center}
\begin{subfigure}{.325\textwidth}
\begin{center}
\includegraphics[height=1.5in]{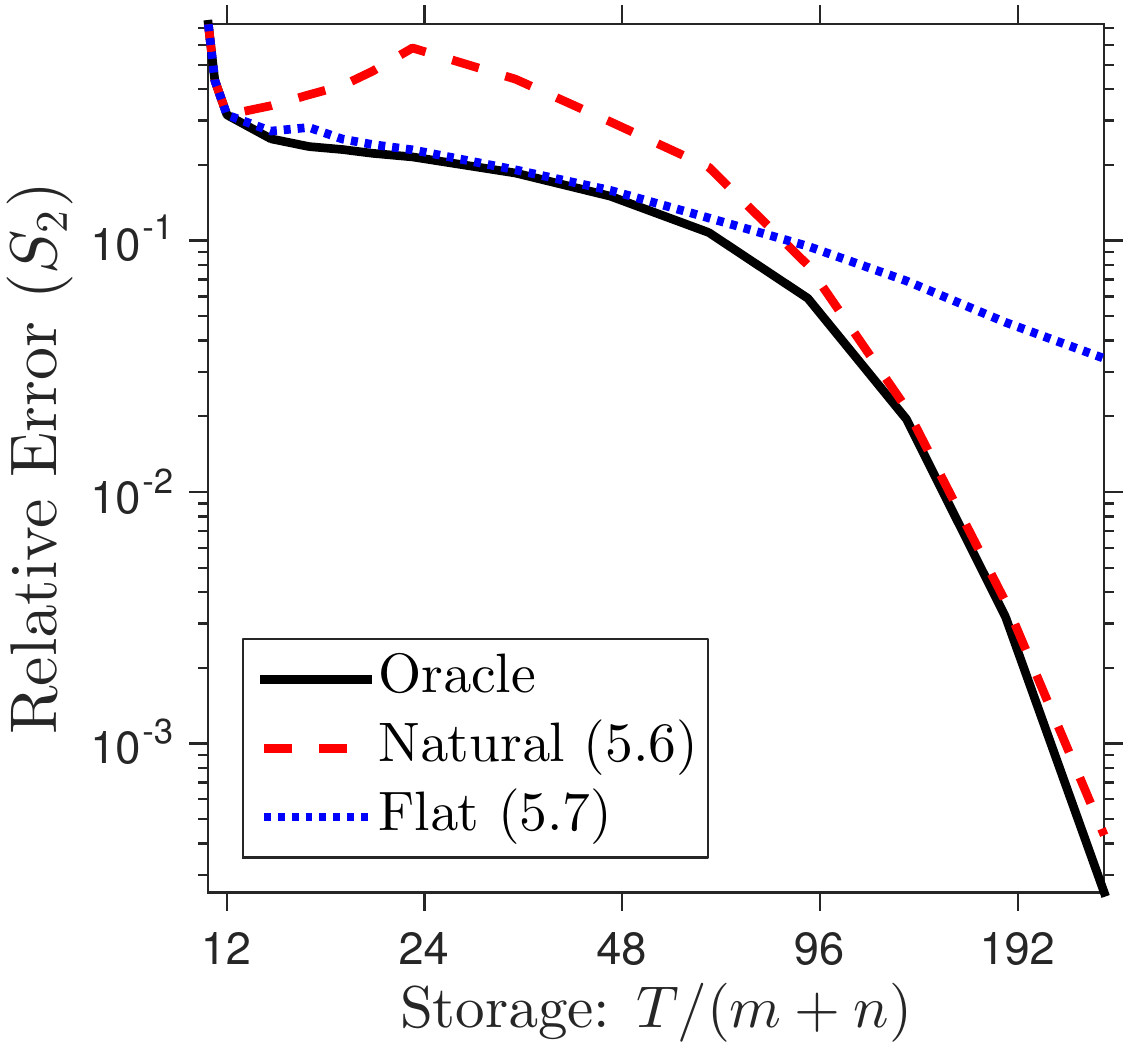}
\caption{\texttt{ExpDecaySlow}}
\end{center}
\end{subfigure}
\begin{subfigure}{.325\textwidth}
\begin{center}
\includegraphics[height=1.5in]{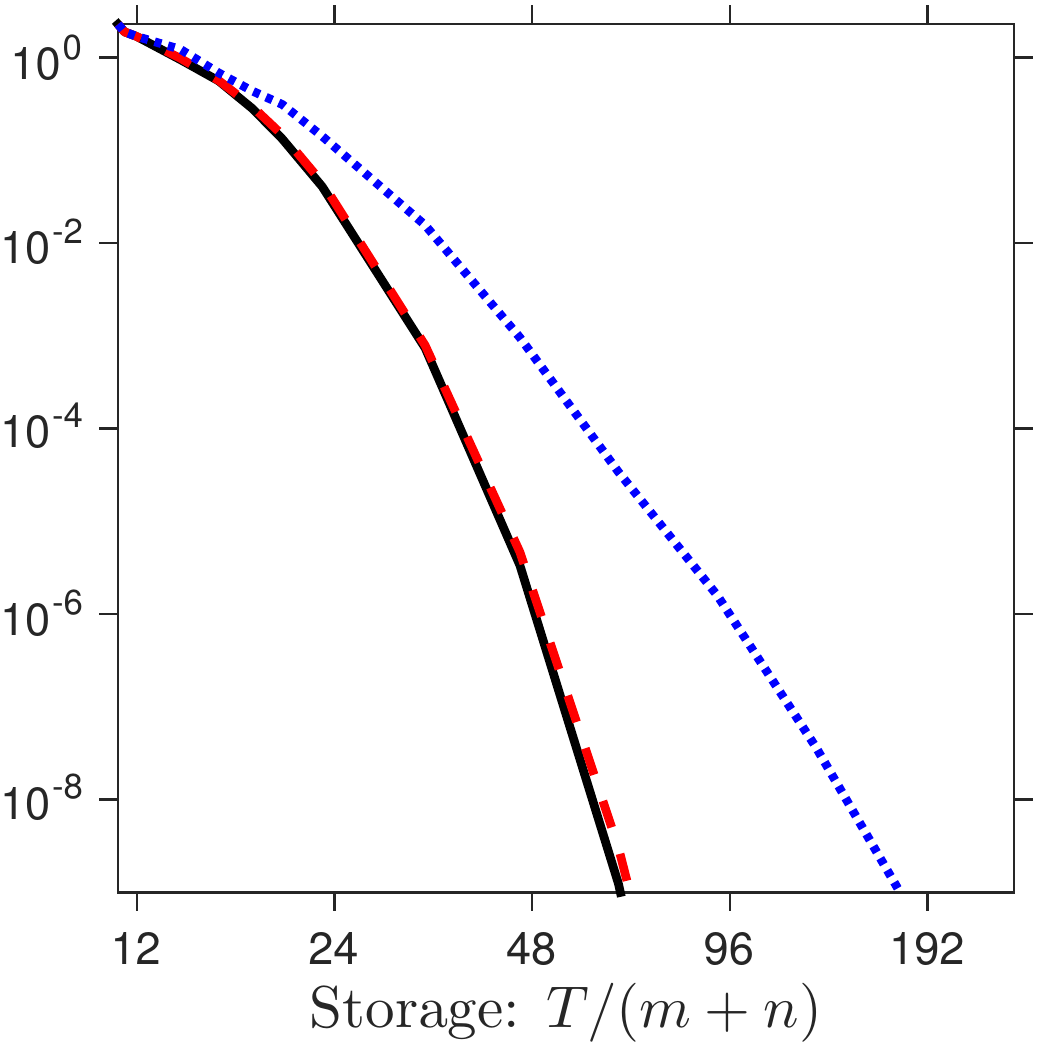}
\caption{\texttt{ExpDecayMed}}
\end{center}
\end{subfigure}
\begin{subfigure}{.325\textwidth}
\begin{center}
\includegraphics[height=1.5in]{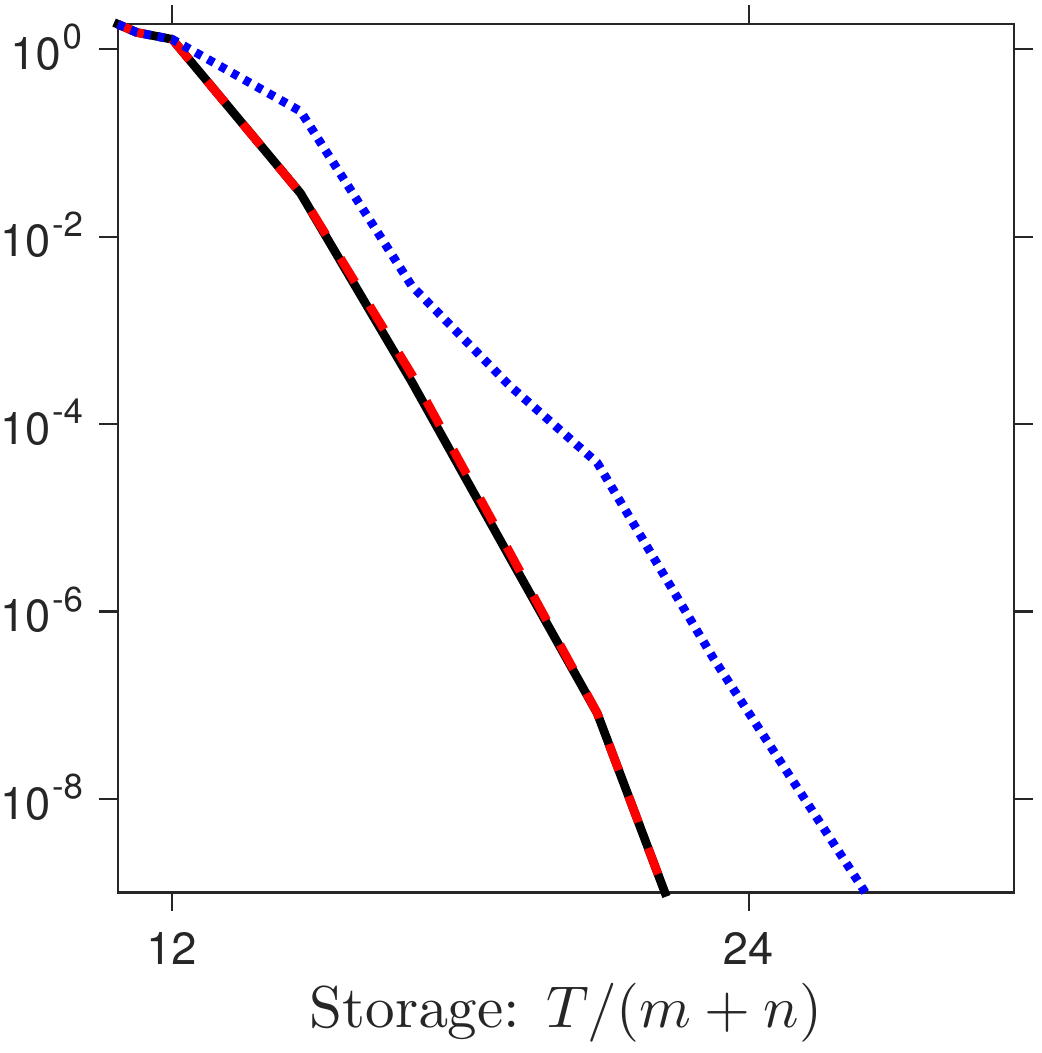}
\caption{\texttt{ExpDecayFast}}
\end{center}
\end{subfigure}
\end{center}

\vspace{0.5em}

\caption{\textbf{Relative error for proposed method with \emph{a priori} parameters.}
(Gaussian maps, effective rank $R = 5$, approximation rank $r = 10$,
Schatten $2$-norm.)
We compare the oracle performance of the proposed fixed-rank
approximation~\cref{eqn:Ahat-fixed} with its performance at theoretically justified
parameter values. See \cref{app:oracle-performance} for details.}
\label{fig:theory-params-R5-S2}
\end{figure}

\begin{figure}[htp!]
\vspace{0.5in}
\begin{center}
\begin{subfigure}{.325\textwidth}
\begin{center}
\includegraphics[height=1.5in]{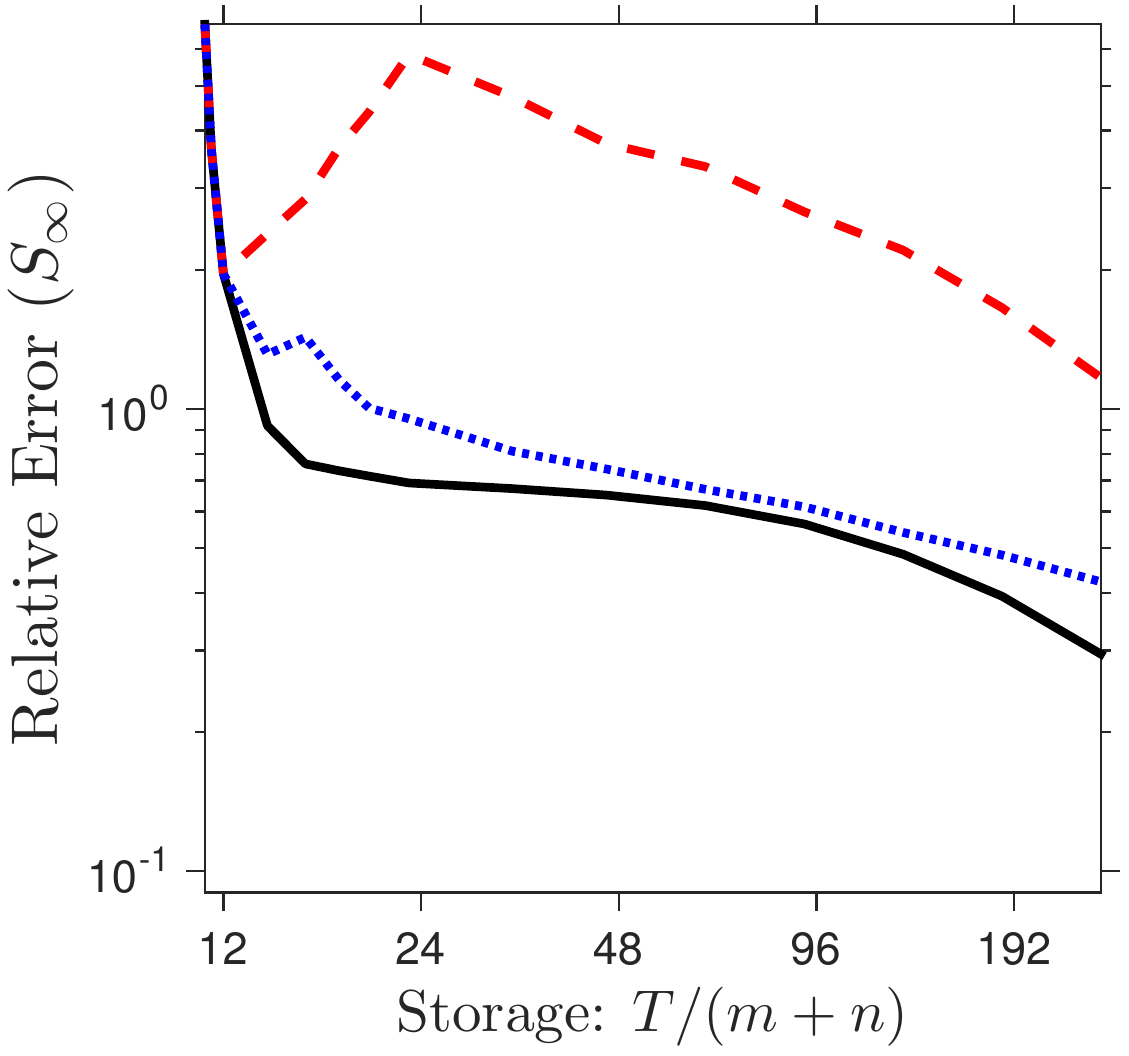}
\caption{\texttt{LowRankHiNoise}}
\end{center}
\end{subfigure}
\begin{subfigure}{.325\textwidth}
\begin{center}
\includegraphics[height=1.5in]{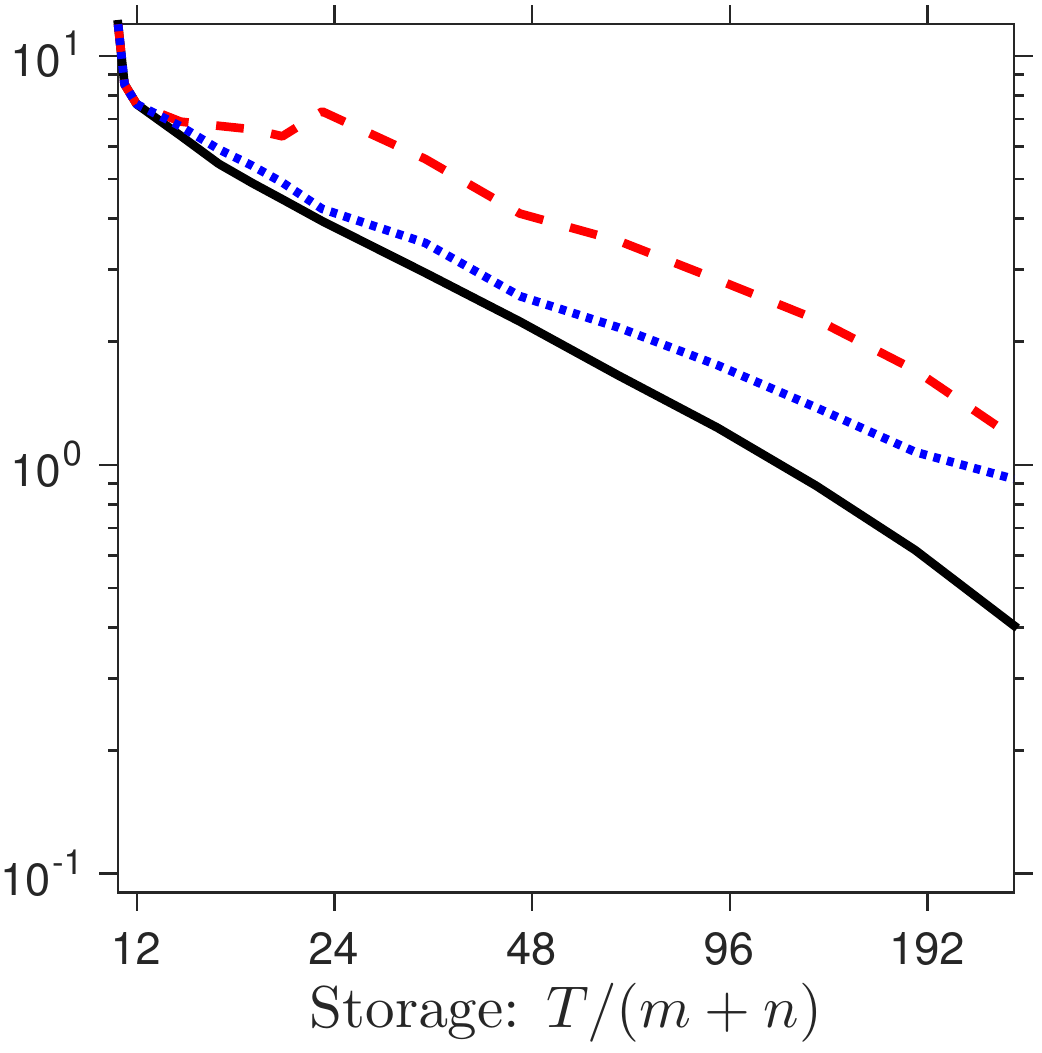}
\caption{\texttt{LowRankMedNoise}}
\end{center}
\end{subfigure}
\begin{subfigure}{.325\textwidth}
\begin{center}
\includegraphics[height=1.5in]{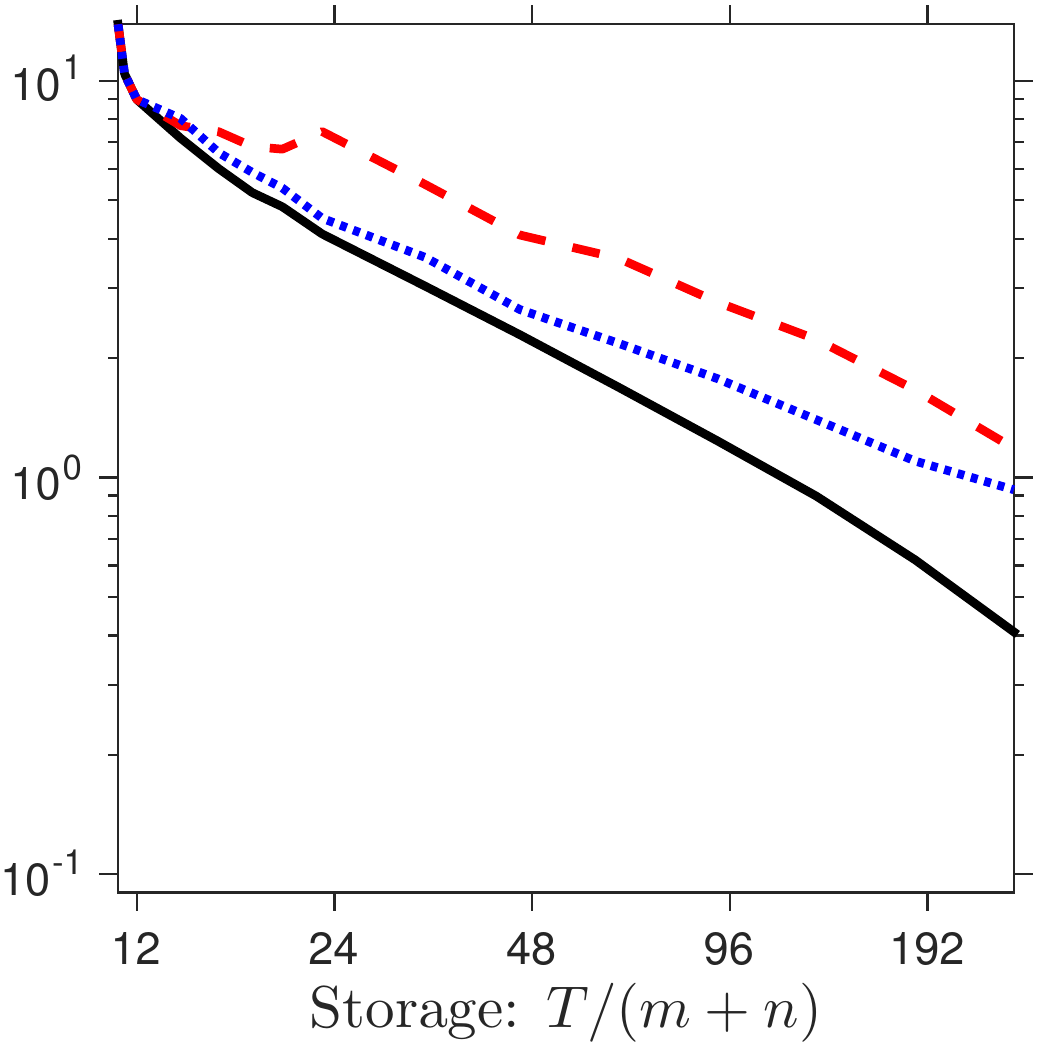}
\caption{\texttt{LowRankLowNoise}}
\end{center}
\end{subfigure}
\end{center}

\vspace{.5em}

\begin{center}
\begin{subfigure}{.325\textwidth}
\begin{center}
\includegraphics[height=1.5in]{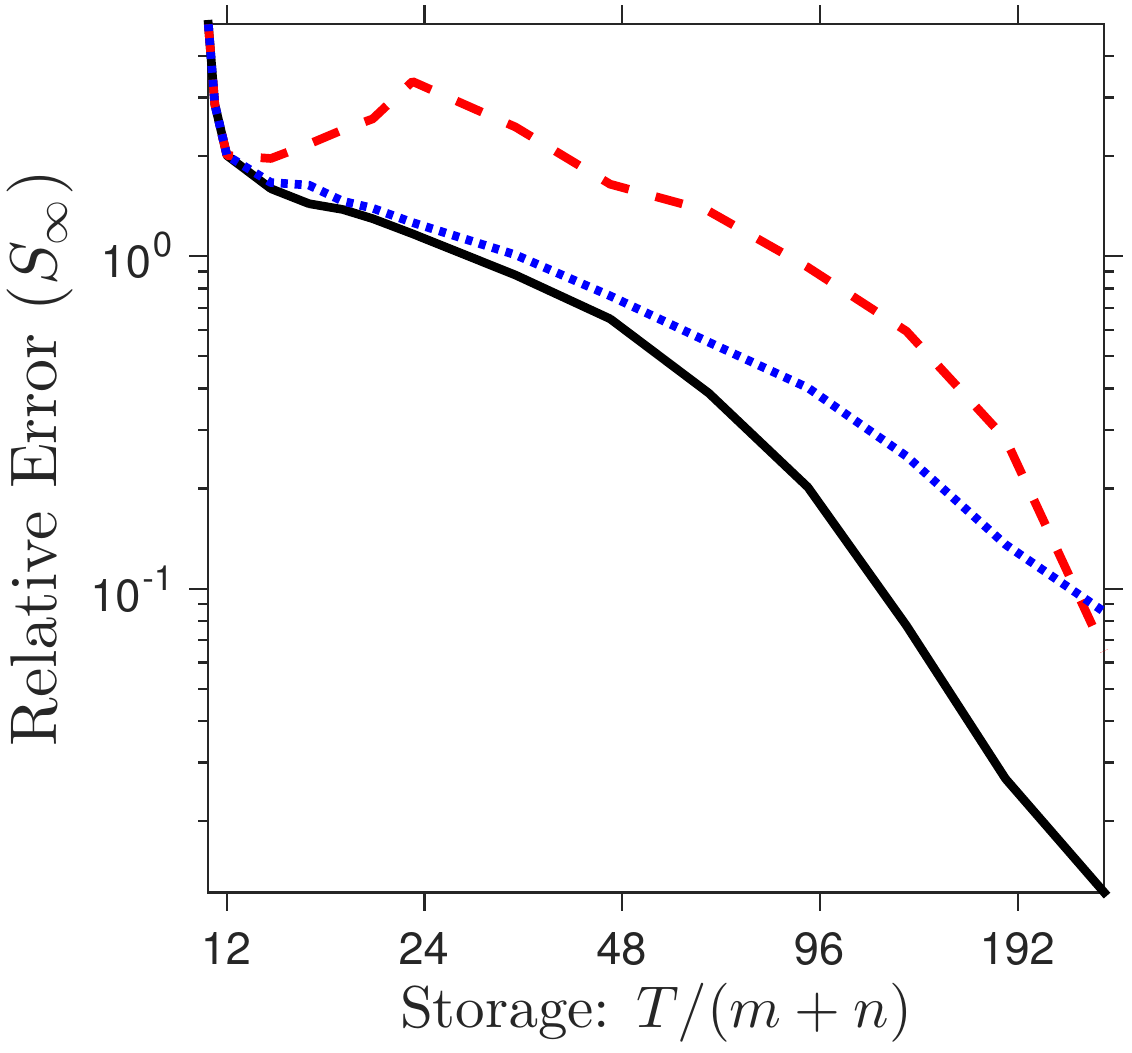}
\caption{\texttt{PolyDecaySlow}}
\end{center}
\end{subfigure}
\begin{subfigure}{.325\textwidth}
\begin{center}
\includegraphics[height=1.5in]{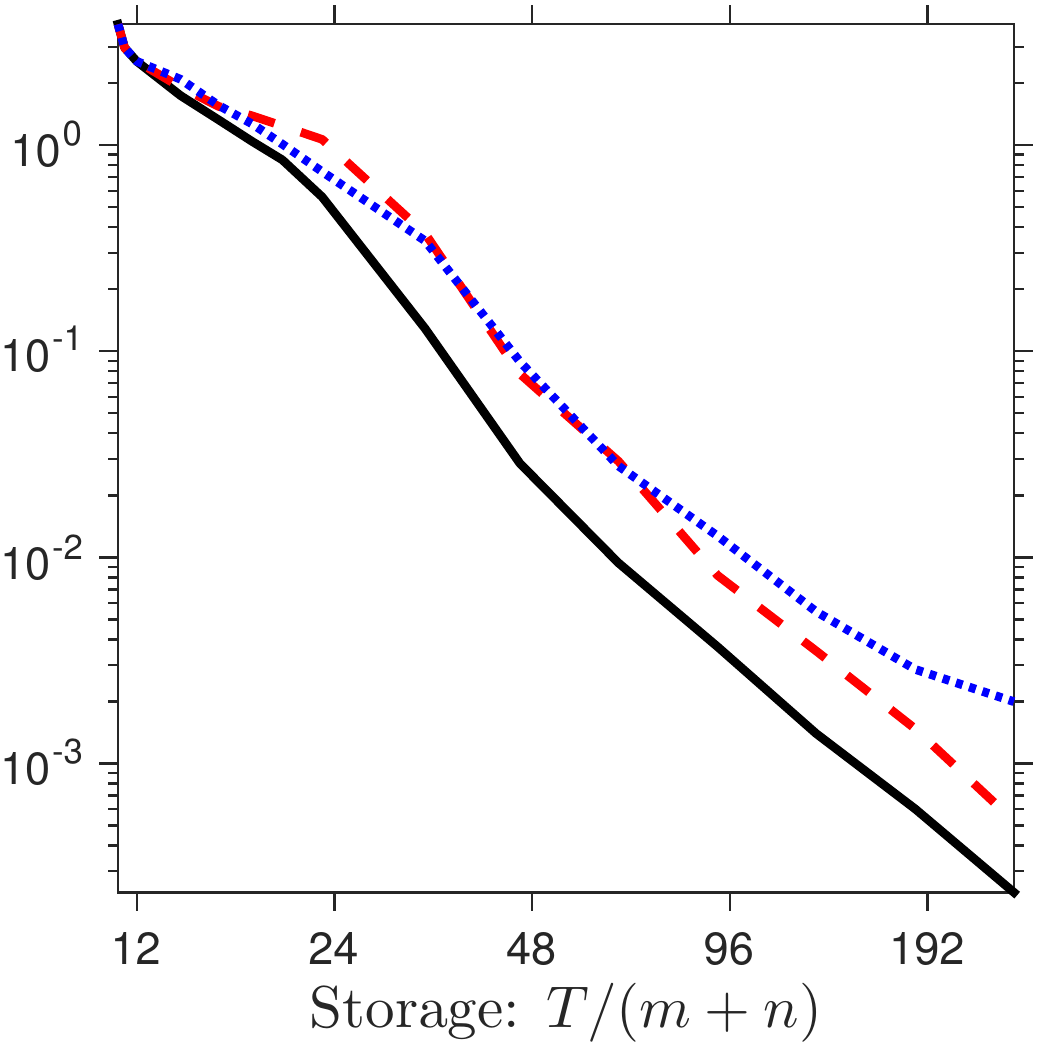}
\caption{\texttt{PolyDecayMed}}
\end{center}
\end{subfigure}
\begin{subfigure}{.325\textwidth}
\begin{center}
\includegraphics[height=1.5in]{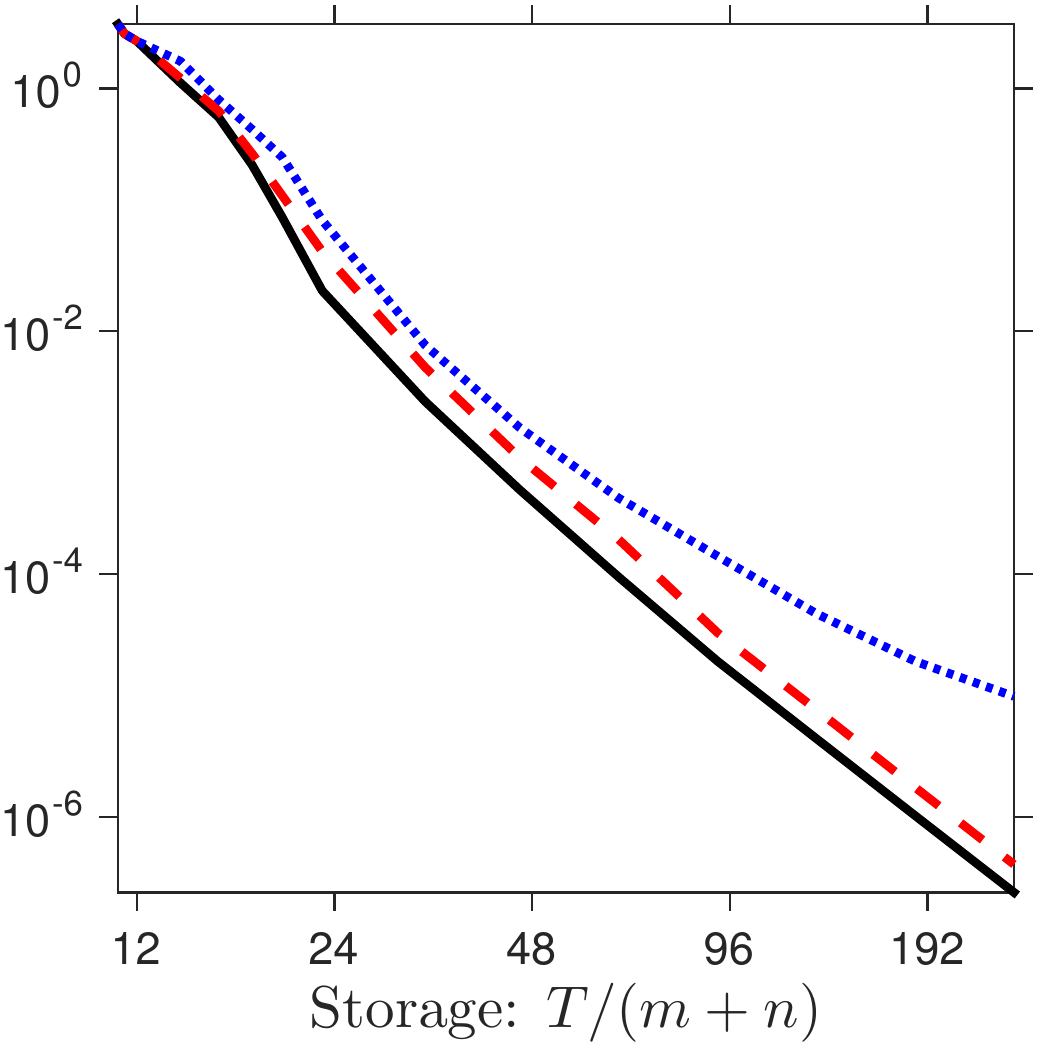}
\caption{\texttt{PolyDecayFast}}
\end{center}
\end{subfigure}
\end{center}

\vspace{0.5em}

\begin{center}
\begin{subfigure}{.325\textwidth}
\begin{center}
\includegraphics[height=1.5in]{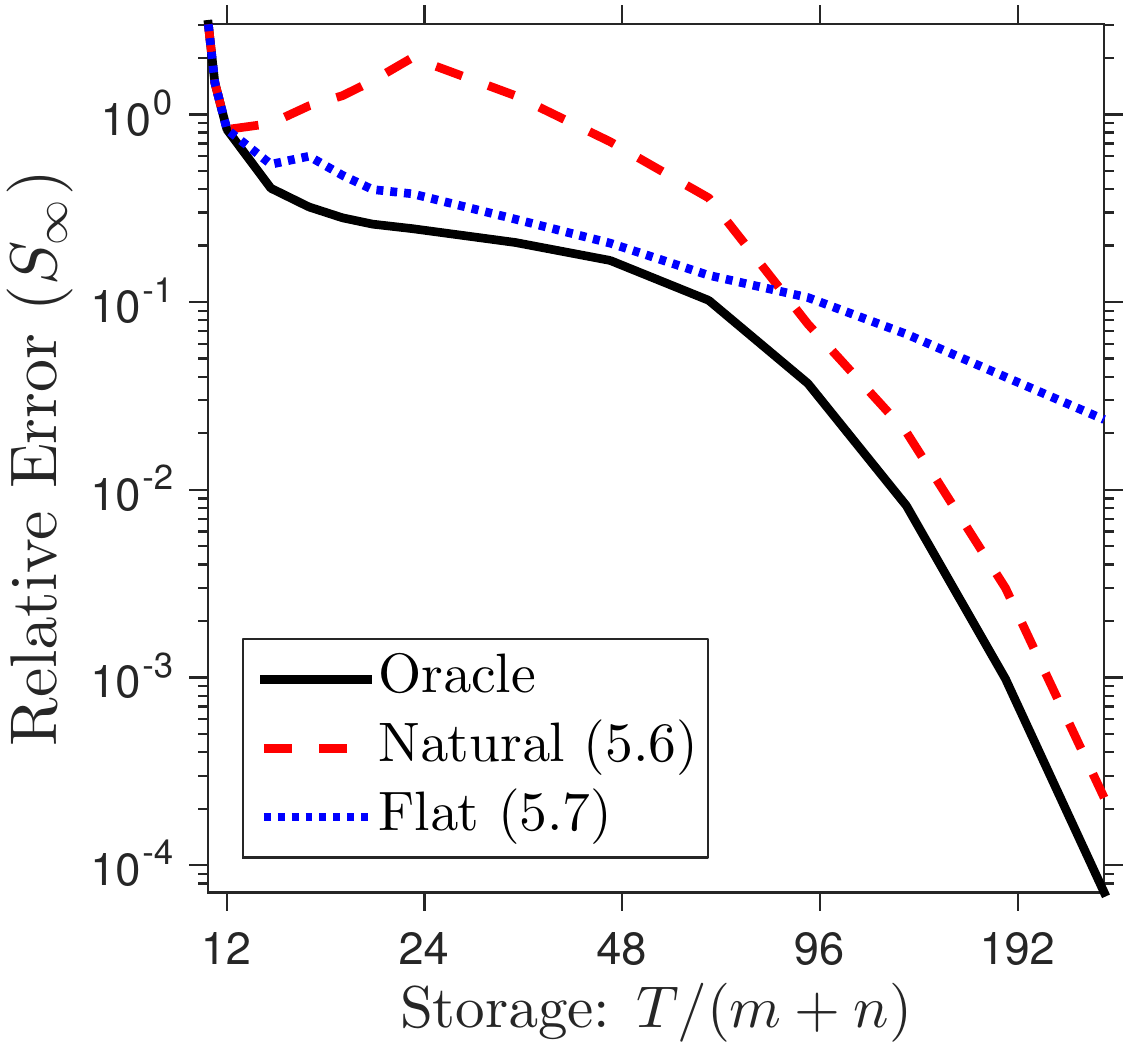}
\caption{\texttt{ExpDecaySlow}}
\end{center}
\end{subfigure}
\begin{subfigure}{.325\textwidth}
\begin{center}
\includegraphics[height=1.5in]{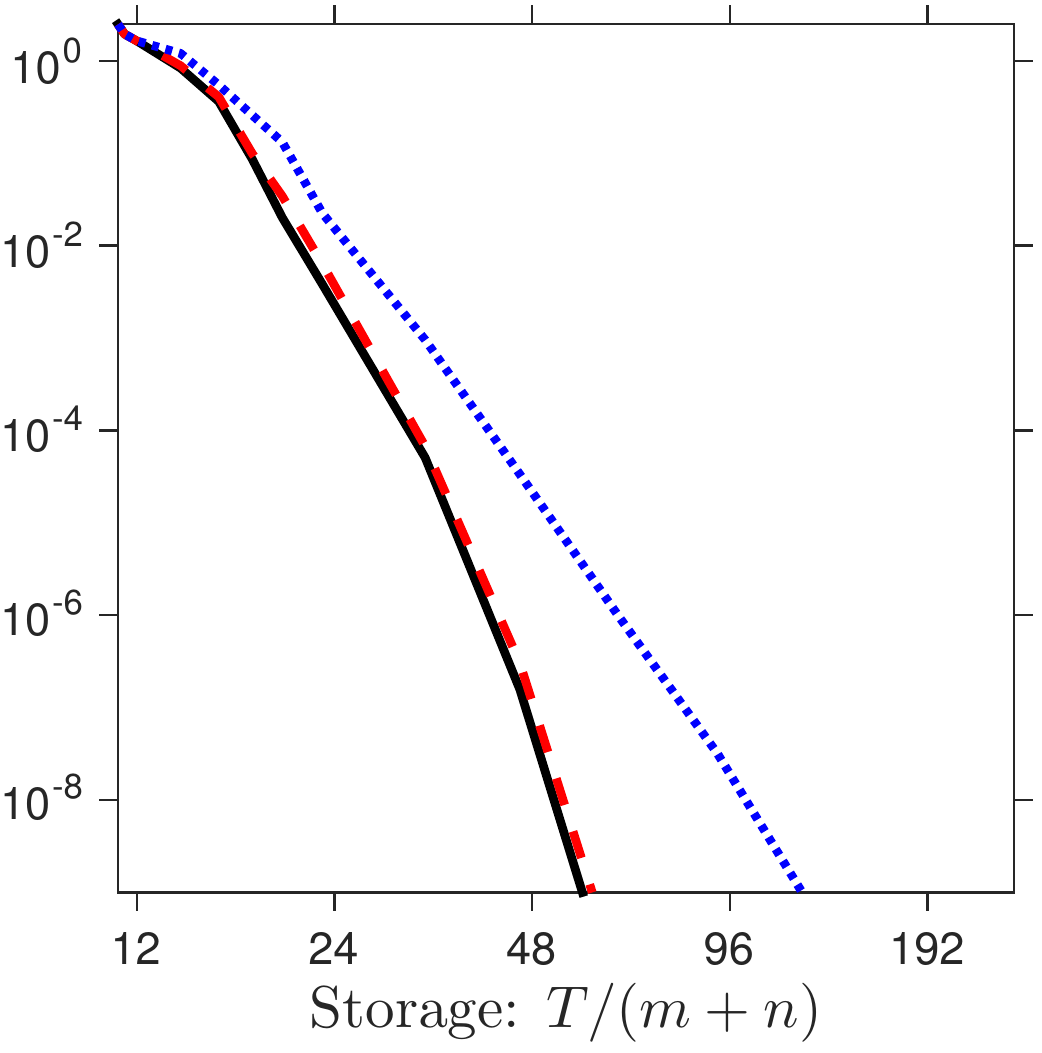}
\caption{\texttt{ExpDecayMed}}
\end{center}
\end{subfigure}
\begin{subfigure}{.325\textwidth}
\begin{center}
\includegraphics[height=1.5in]{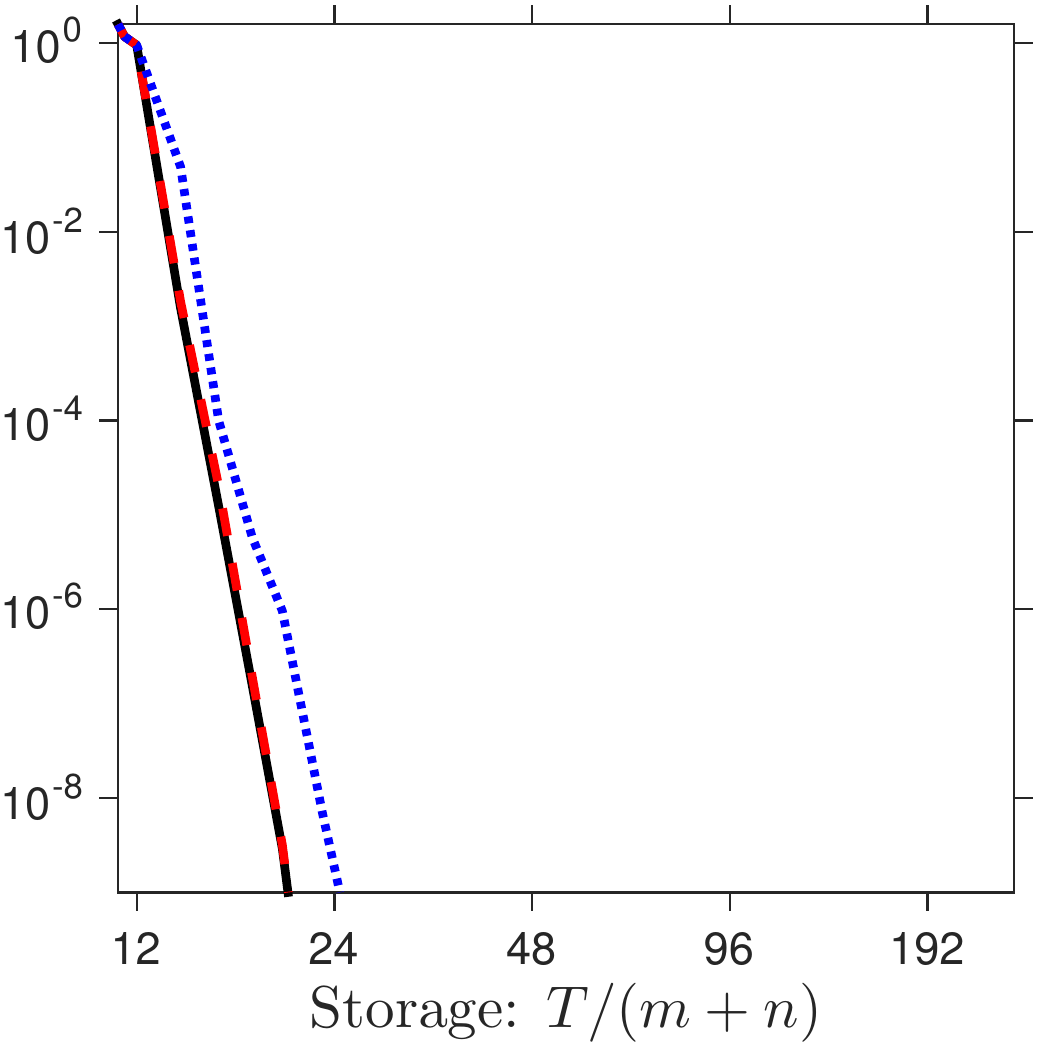}
\caption{\texttt{ExpDecayFast}}
\end{center}
\end{subfigure}
\end{center}

\vspace{0.5em}

\caption{\textbf{Relative error for proposed method with \emph{a priori} parameters.}
(Gaussian maps, effective rank $R = 5$, approximation rank $r = 10$,
Schatten $\infty$-norm.)
We compare the oracle performance of the proposed fixed-rank
approximation~\cref{eqn:Ahat-fixed} with its performance at theoretically justified
parameter values. See \cref{app:oracle-performance} for details.}
\label{fig:theory-params-R5-Sinf}
\end{figure}

\begin{figure}[htp!]
\vspace{0.5in}
\begin{center}
\begin{subfigure}{.325\textwidth}
\begin{center}
\includegraphics[height=1.5in]{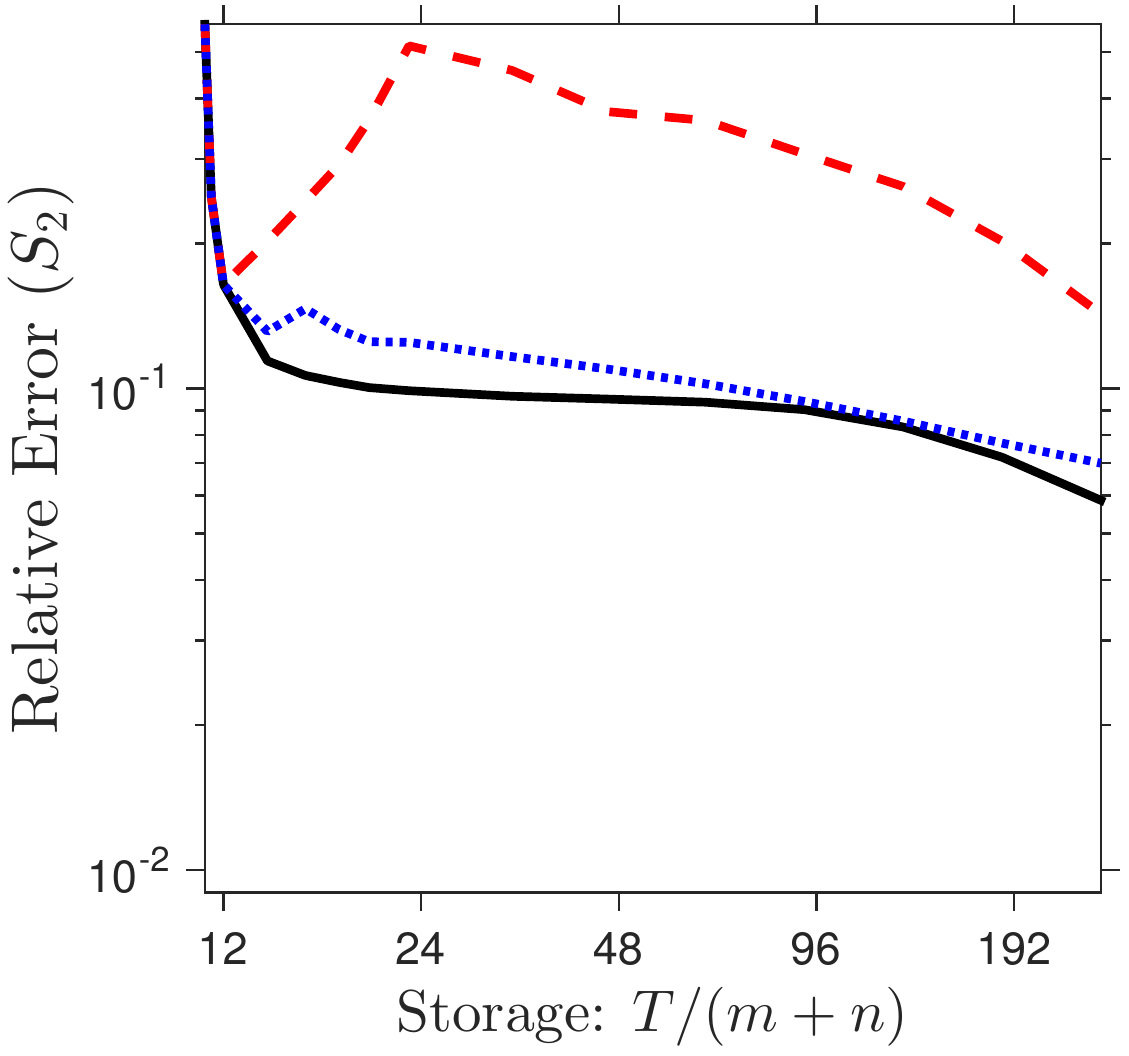}
\caption{\texttt{LowRankHiNoise}}
\end{center}
\end{subfigure}
\begin{subfigure}{.325\textwidth}
\begin{center}
\includegraphics[height=1.5in]{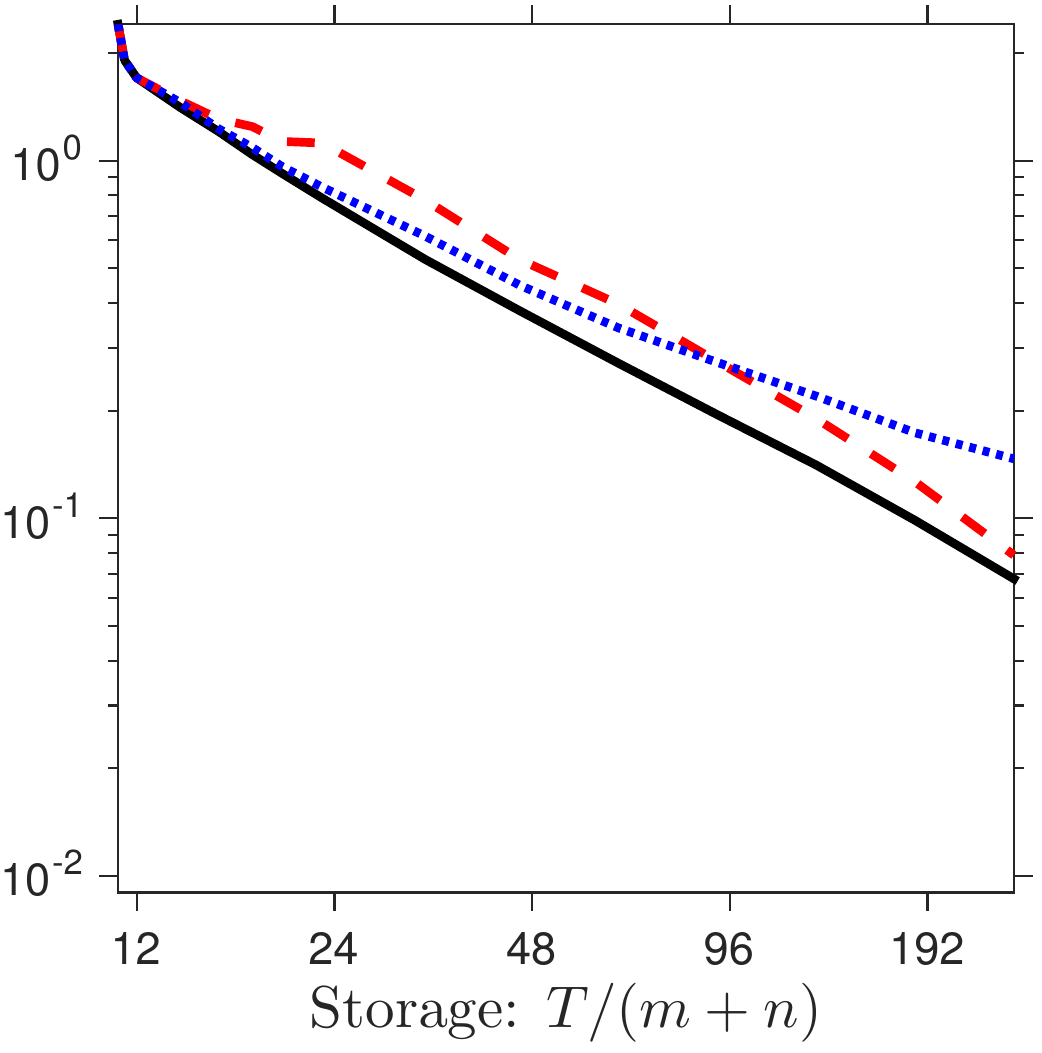}
\caption{\texttt{LowRankMedNoise}}
\end{center}
\end{subfigure}
\begin{subfigure}{.325\textwidth}
\begin{center}
\includegraphics[height=1.5in]{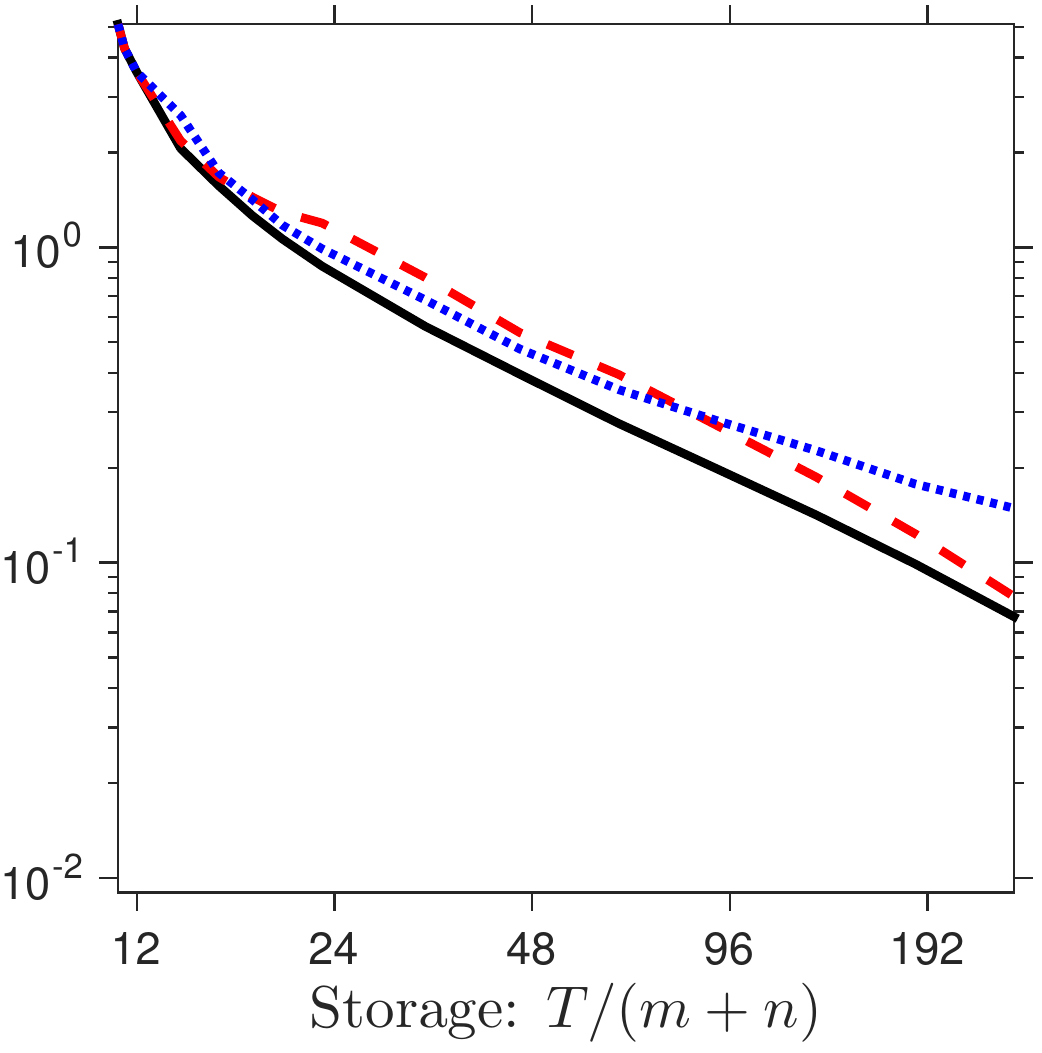}
\caption{\texttt{LowRankLowNoise}}
\end{center}
\end{subfigure}
\end{center}

\vspace{.5em}

\begin{center}
\begin{subfigure}{.325\textwidth}
\begin{center}
\includegraphics[height=1.5in]{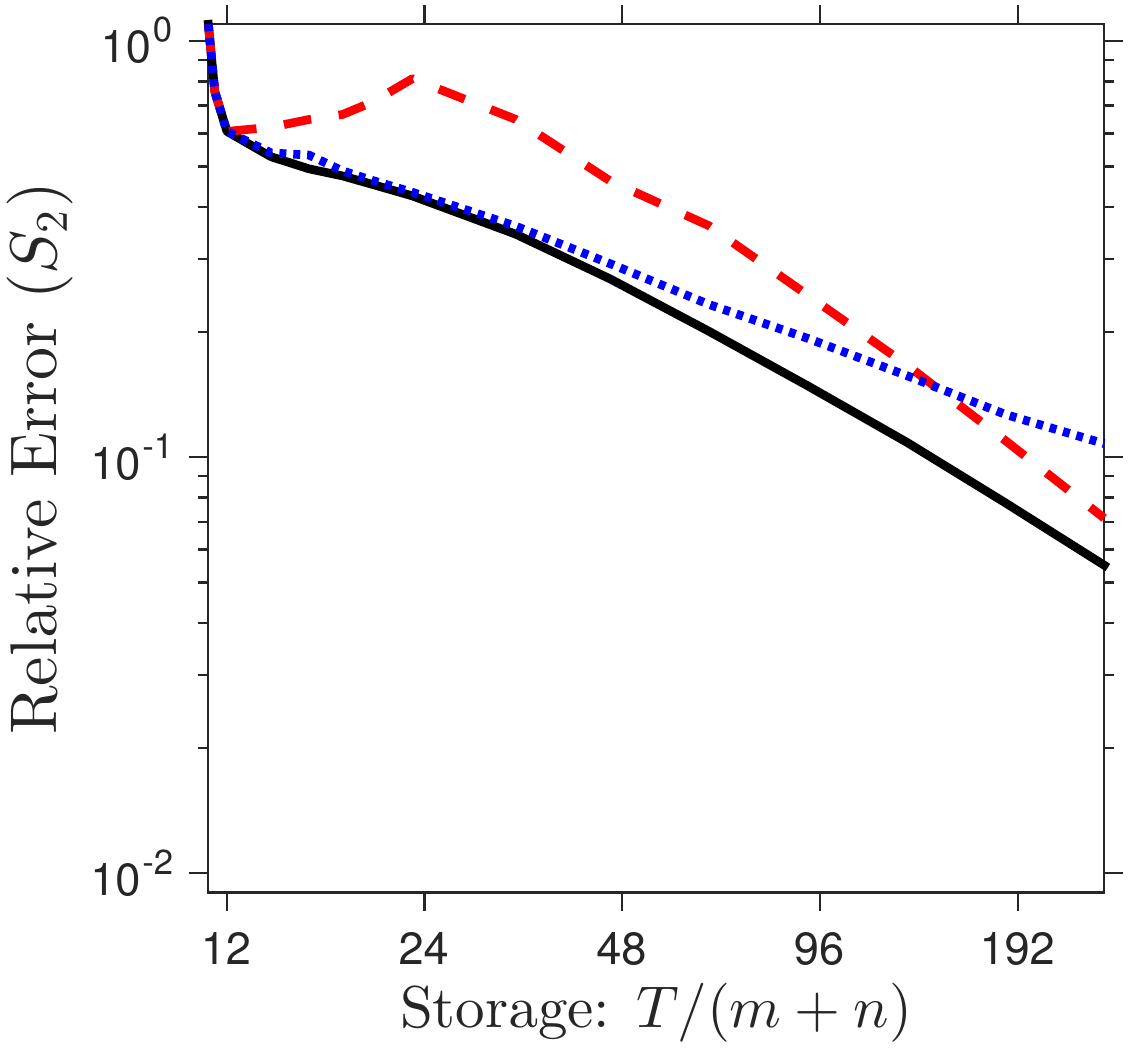}
\caption{\texttt{PolyDecaySlow}}
\end{center}
\end{subfigure}
\begin{subfigure}{.325\textwidth}
\begin{center}
\includegraphics[height=1.5in]{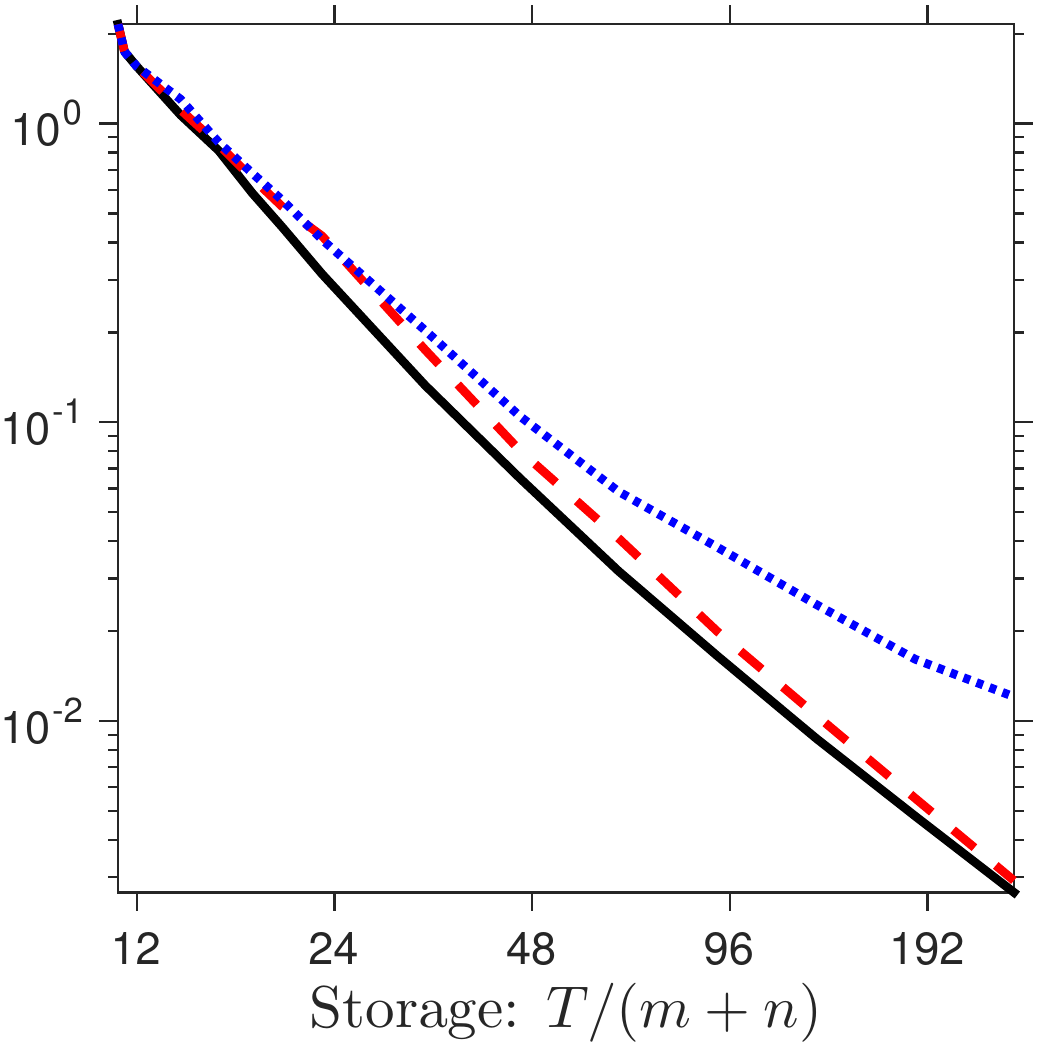}
\caption{\texttt{PolyDecayMed}}
\end{center}
\end{subfigure}
\begin{subfigure}{.325\textwidth}
\begin{center}
\includegraphics[height=1.5in]{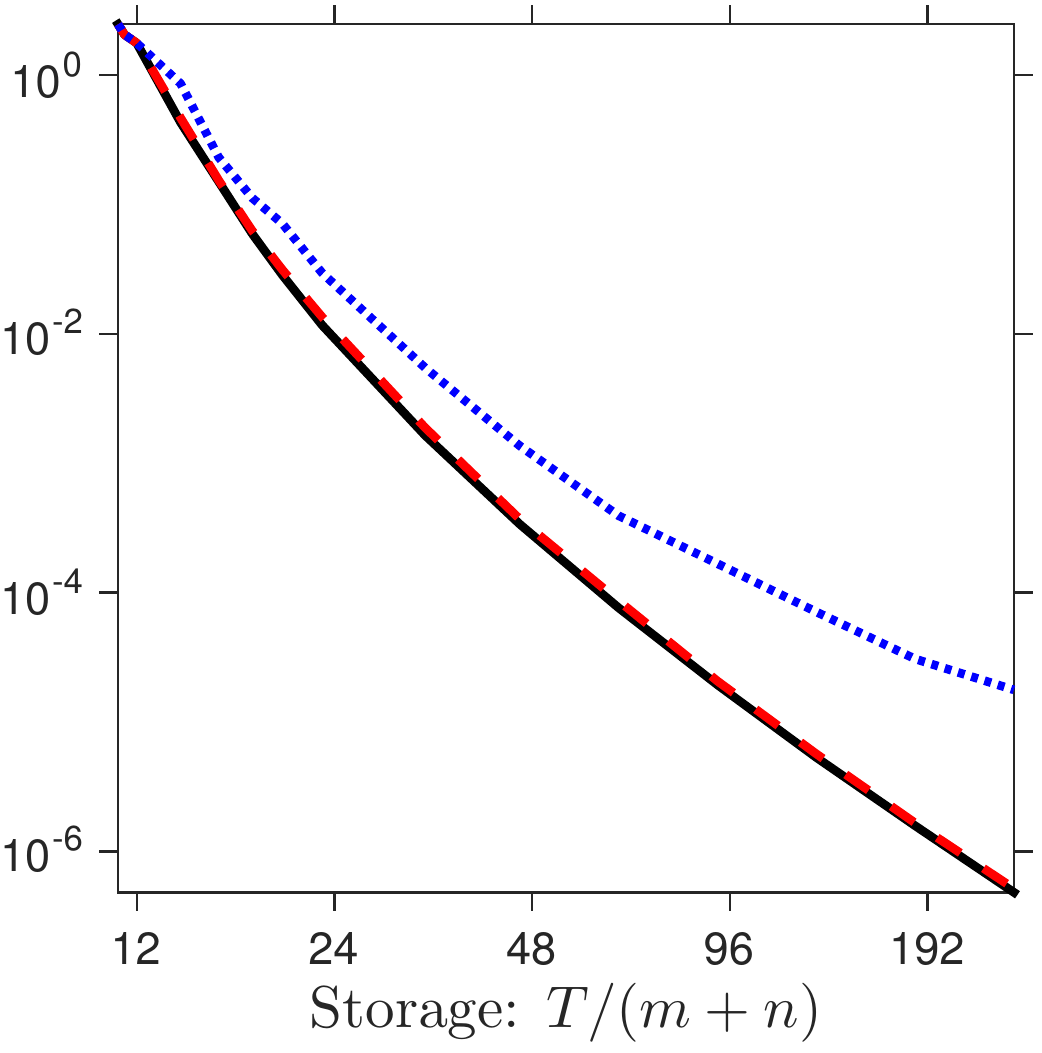}
\caption{\texttt{PolyDecayFast}}
\end{center}
\end{subfigure}
\end{center}

\vspace{0.5em}

\begin{center}
\begin{subfigure}{.325\textwidth}
\begin{center}
\includegraphics[height=1.5in]{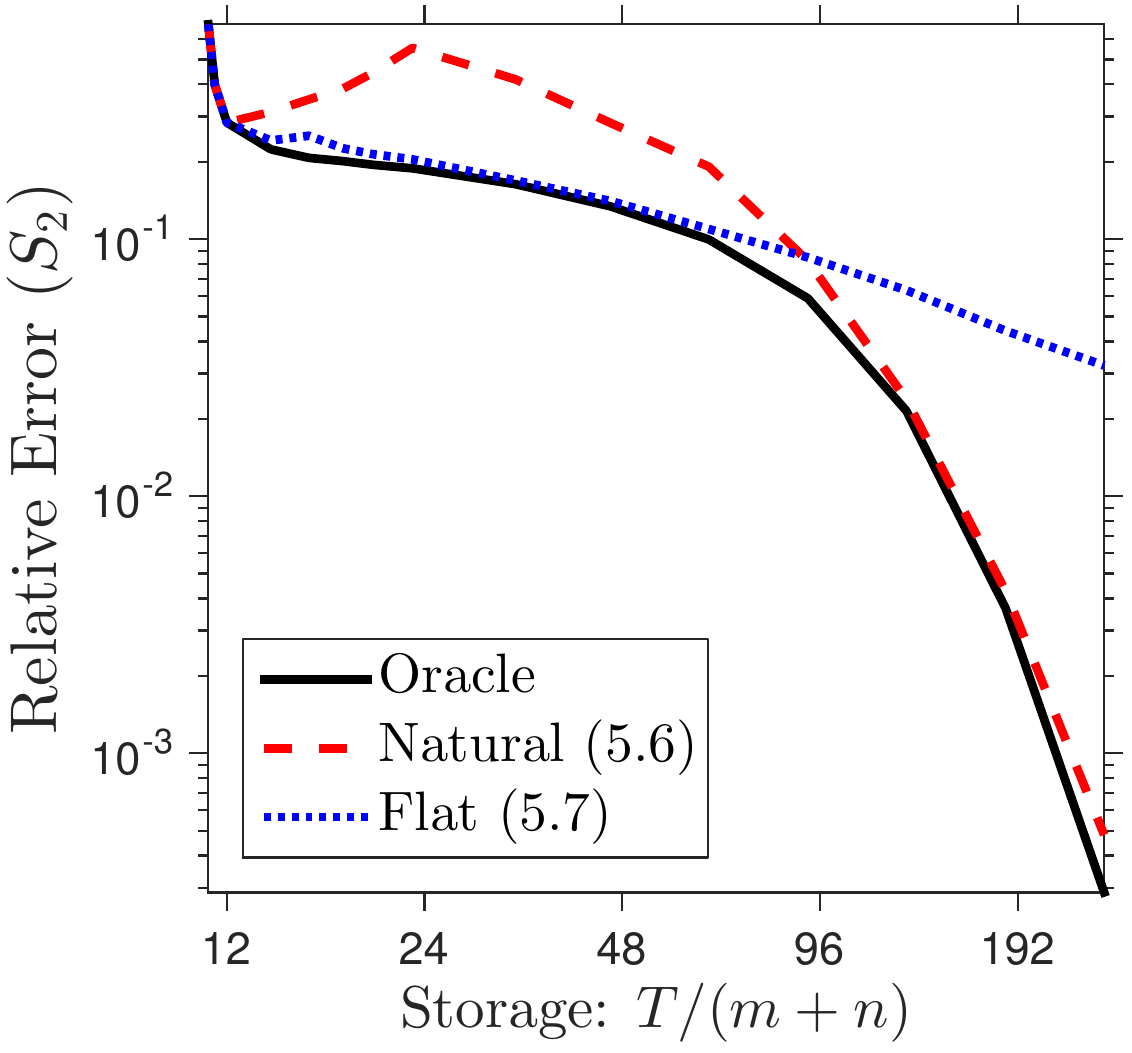}
\caption{\texttt{ExpDecaySlow}}
\end{center}
\end{subfigure}
\begin{subfigure}{.325\textwidth}
\begin{center}
\includegraphics[height=1.5in]{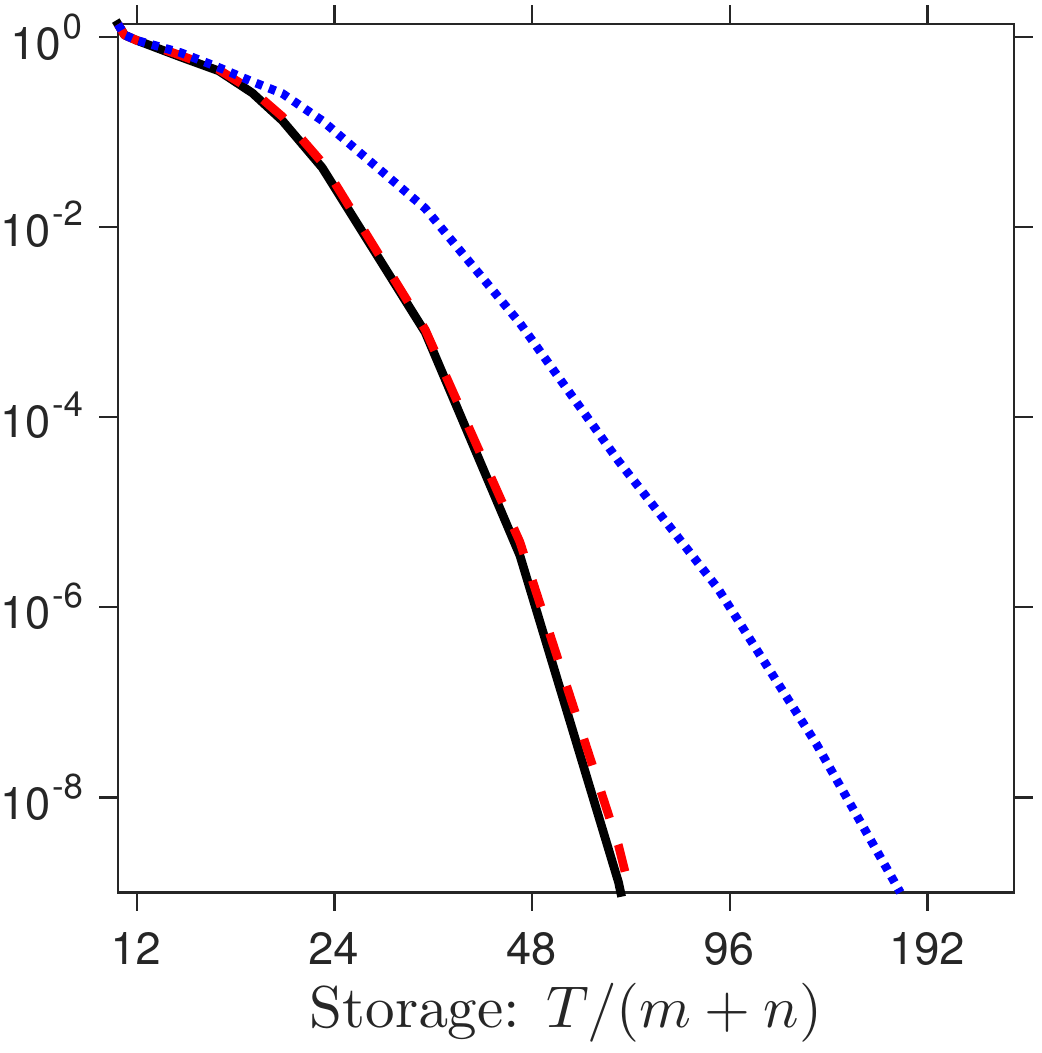}
\caption{\texttt{ExpDecayMed}}
\end{center}
\end{subfigure}
\begin{subfigure}{.325\textwidth}
\begin{center}
\includegraphics[height=1.5in]{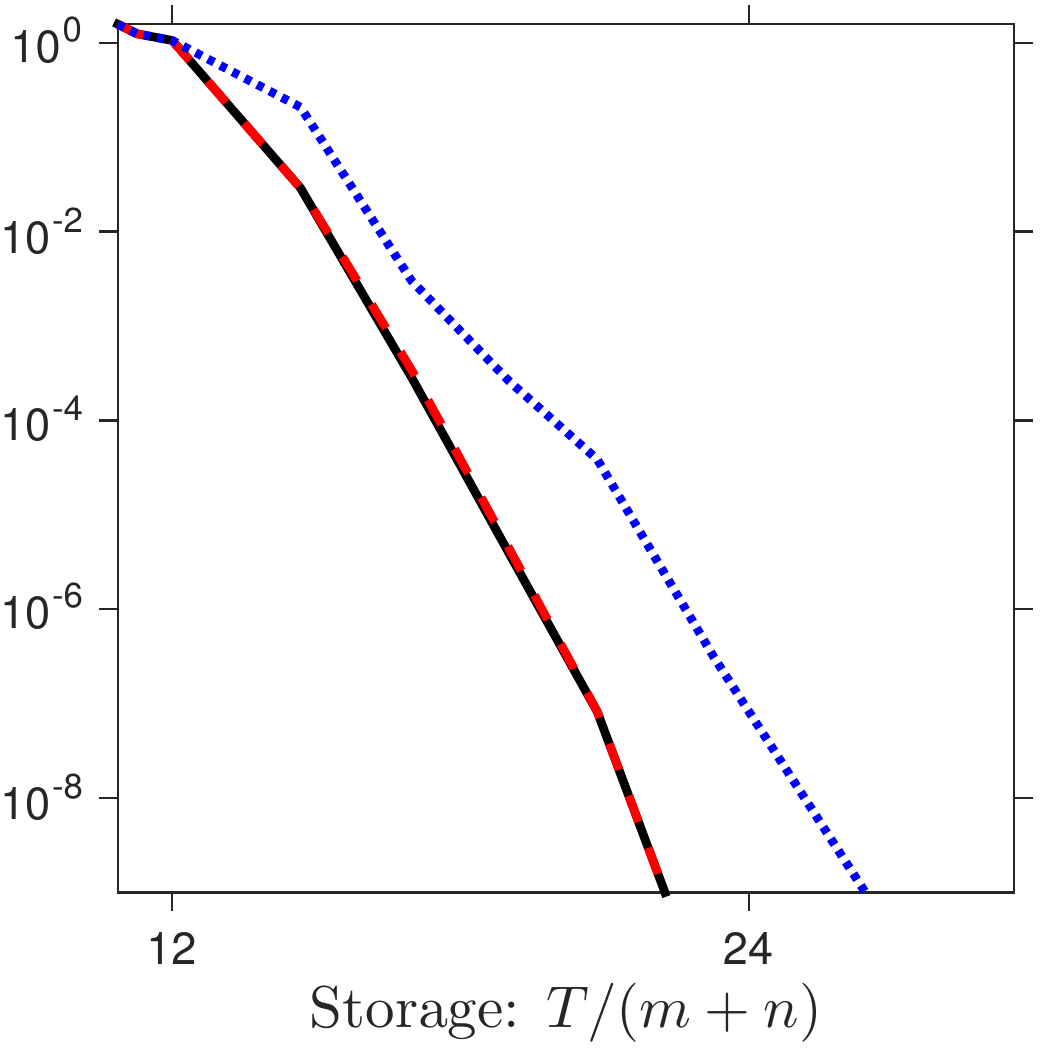}
\caption{\texttt{ExpDecayFast}}
\end{center}
\end{subfigure}
\end{center}

\vspace{0.5em}

\caption{\textbf{Relative error for proposed method with \emph{a priori} parameters.}
(Gaussian maps, effective rank $R = 10$, approximation rank $r = 10$,
Schatten $2$-norm.)
We compare the oracle performance of the proposed fixed-rank
approximation~\cref{eqn:Ahat-fixed} with its performance at theoretically justified
parameter values. See \cref{app:oracle-performance} for details.}
\label{fig:theory-params-R10-S2}
\end{figure}

\begin{figure}[htp!]
\vspace{0.5in}
\begin{center}
\begin{subfigure}{.325\textwidth}
\begin{center}
\includegraphics[height=1.5in]{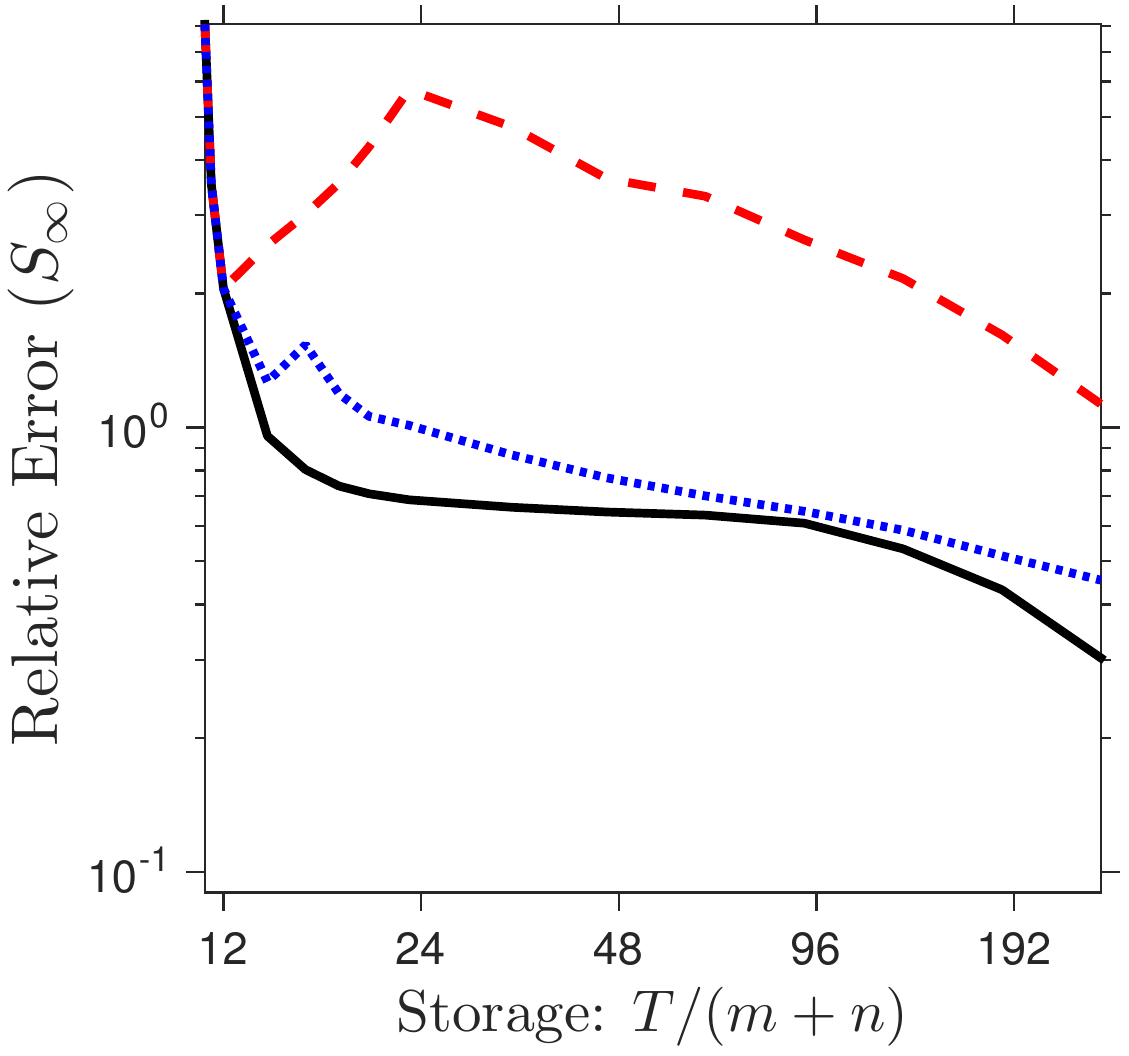}
\caption{\texttt{LowRankHiNoise}}
\end{center}
\end{subfigure}
\begin{subfigure}{.325\textwidth}
\begin{center}
\includegraphics[height=1.5in]{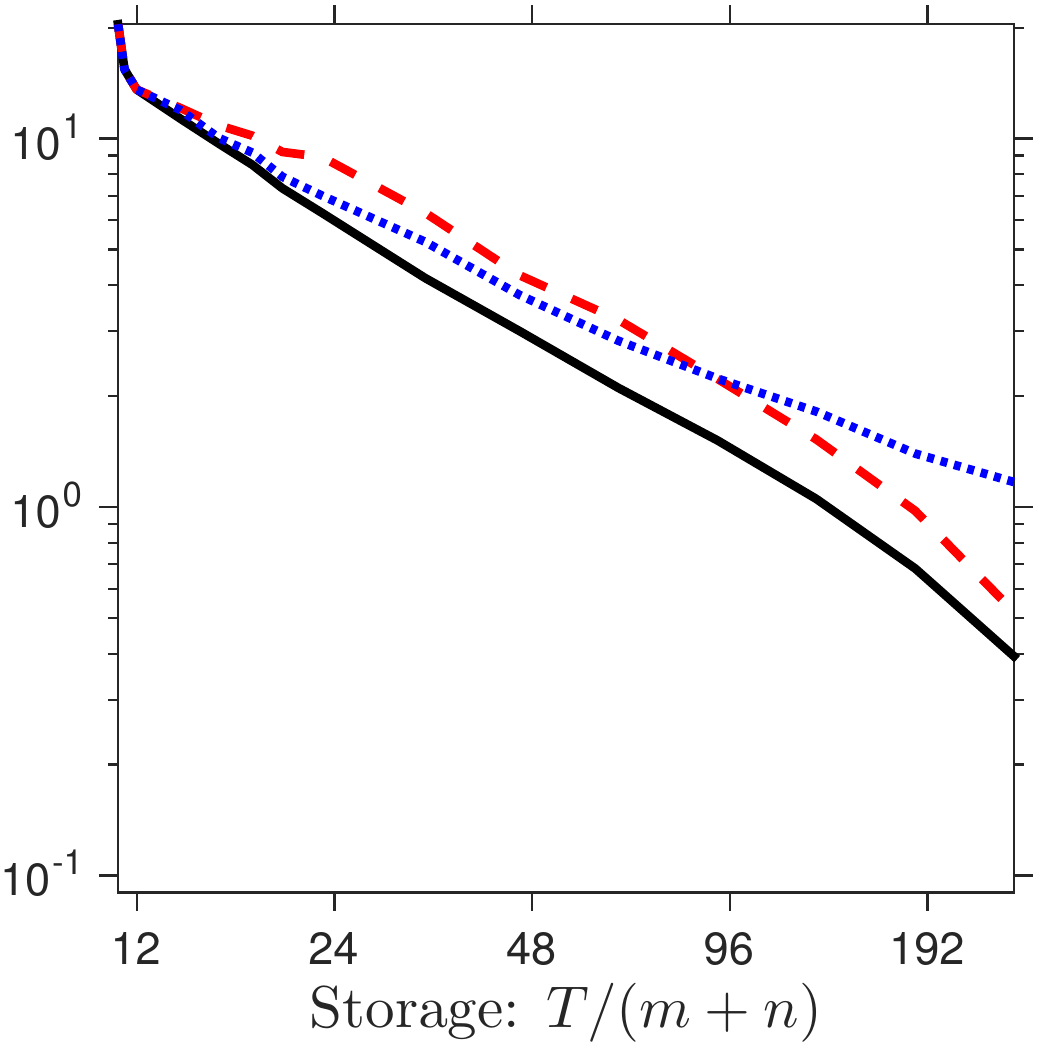}
\caption{\texttt{LowRankMedNoise}}
\end{center}
\end{subfigure}
\begin{subfigure}{.325\textwidth}
\begin{center}
\includegraphics[height=1.5in]{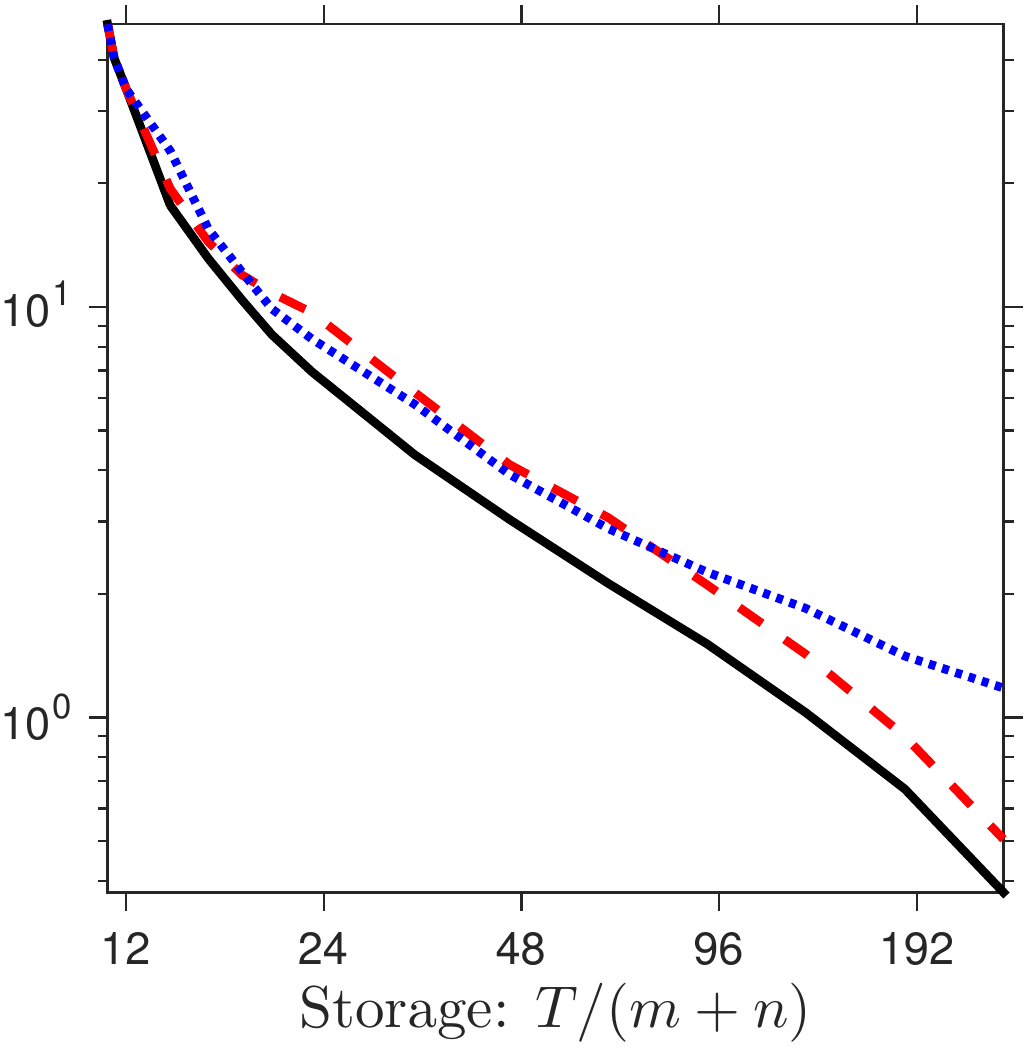}
\caption{\texttt{LowRankLowNoise}}
\end{center}
\end{subfigure}
\end{center}

\vspace{.5em}

\begin{center}
\begin{subfigure}{.325\textwidth}
\begin{center}
\includegraphics[height=1.5in]{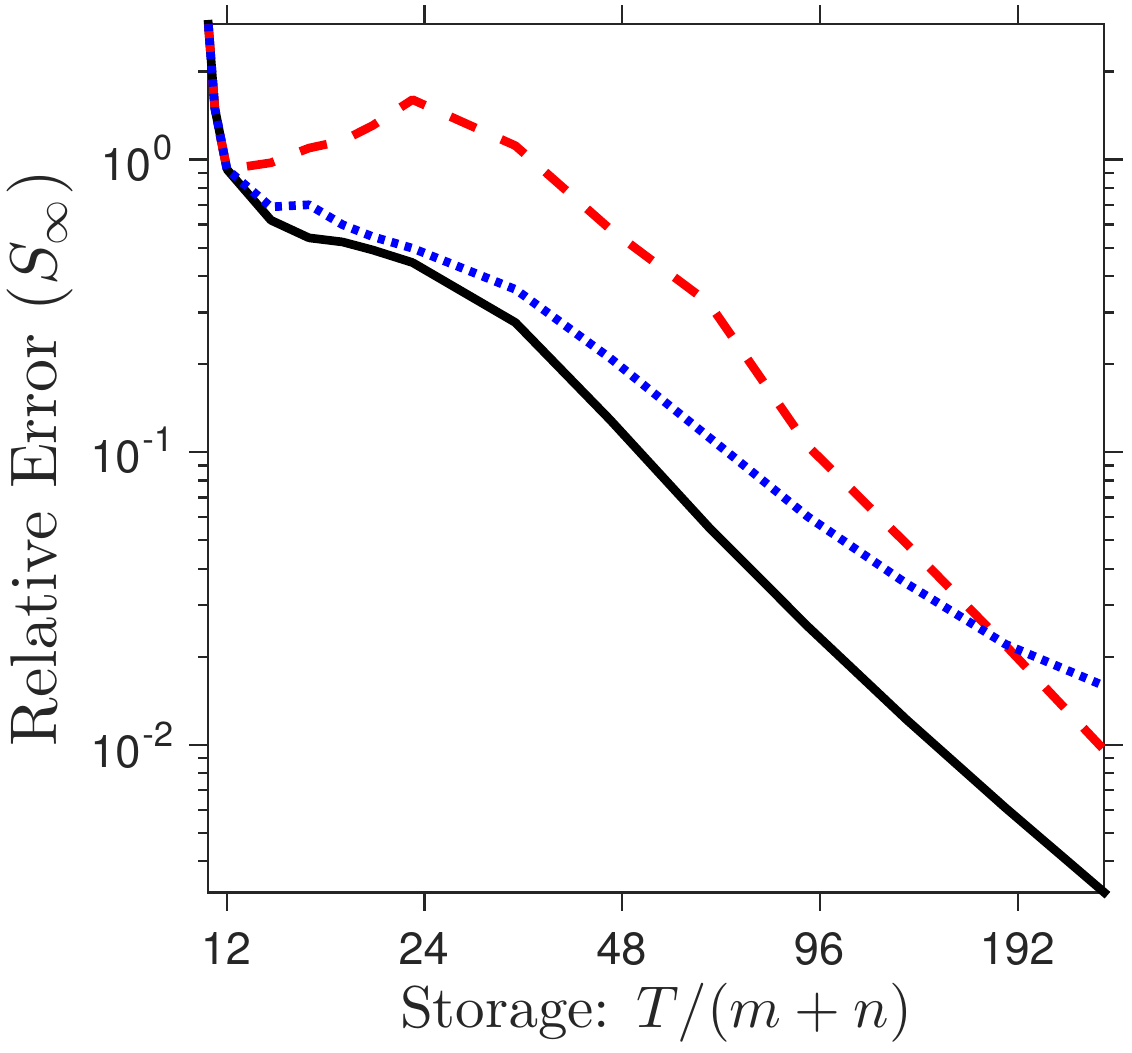}
\caption{\texttt{PolyDecaySlow}}
\end{center}
\end{subfigure}
\begin{subfigure}{.325\textwidth}
\begin{center}
\includegraphics[height=1.5in]{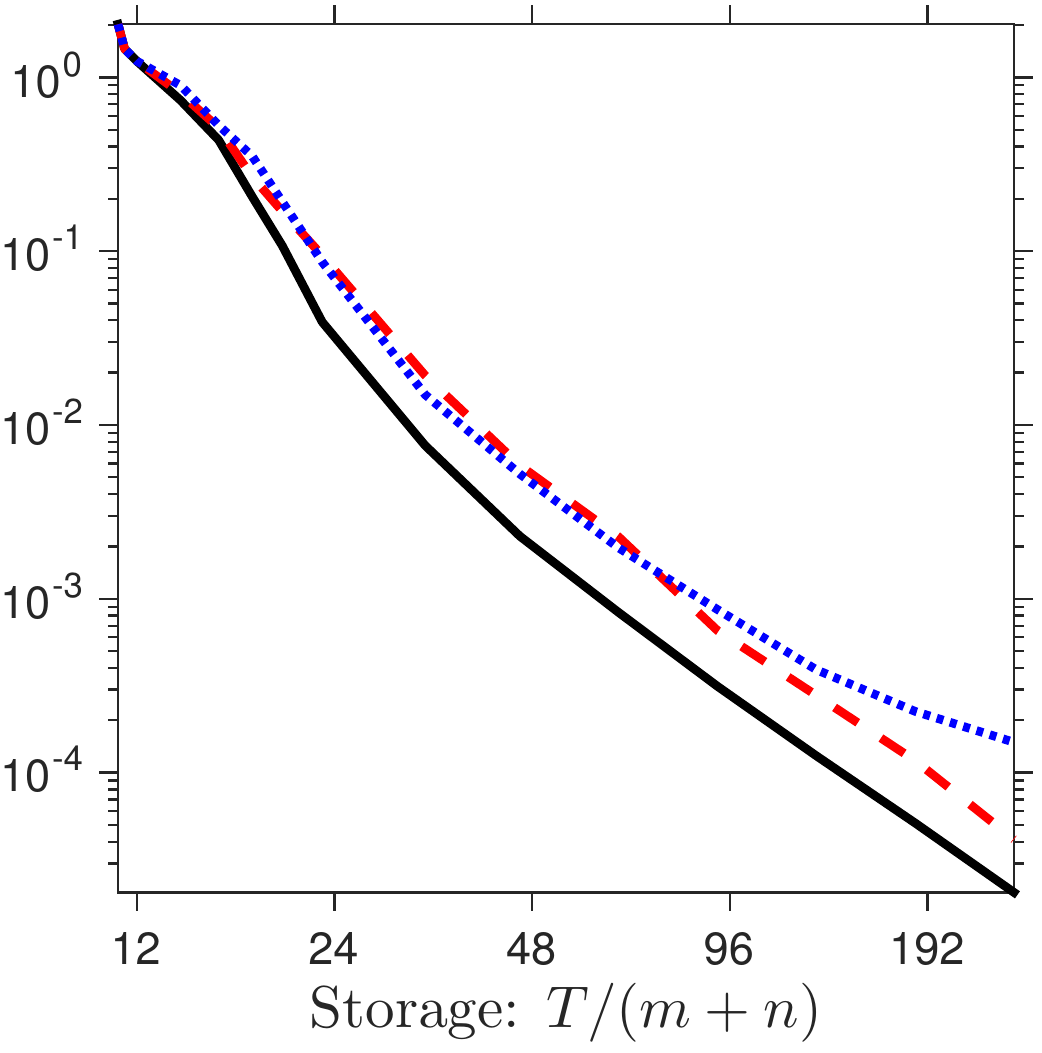}
\caption{\texttt{PolyDecayMed}}
\end{center}
\end{subfigure}
\begin{subfigure}{.325\textwidth}
\begin{center}
\includegraphics[height=1.5in]{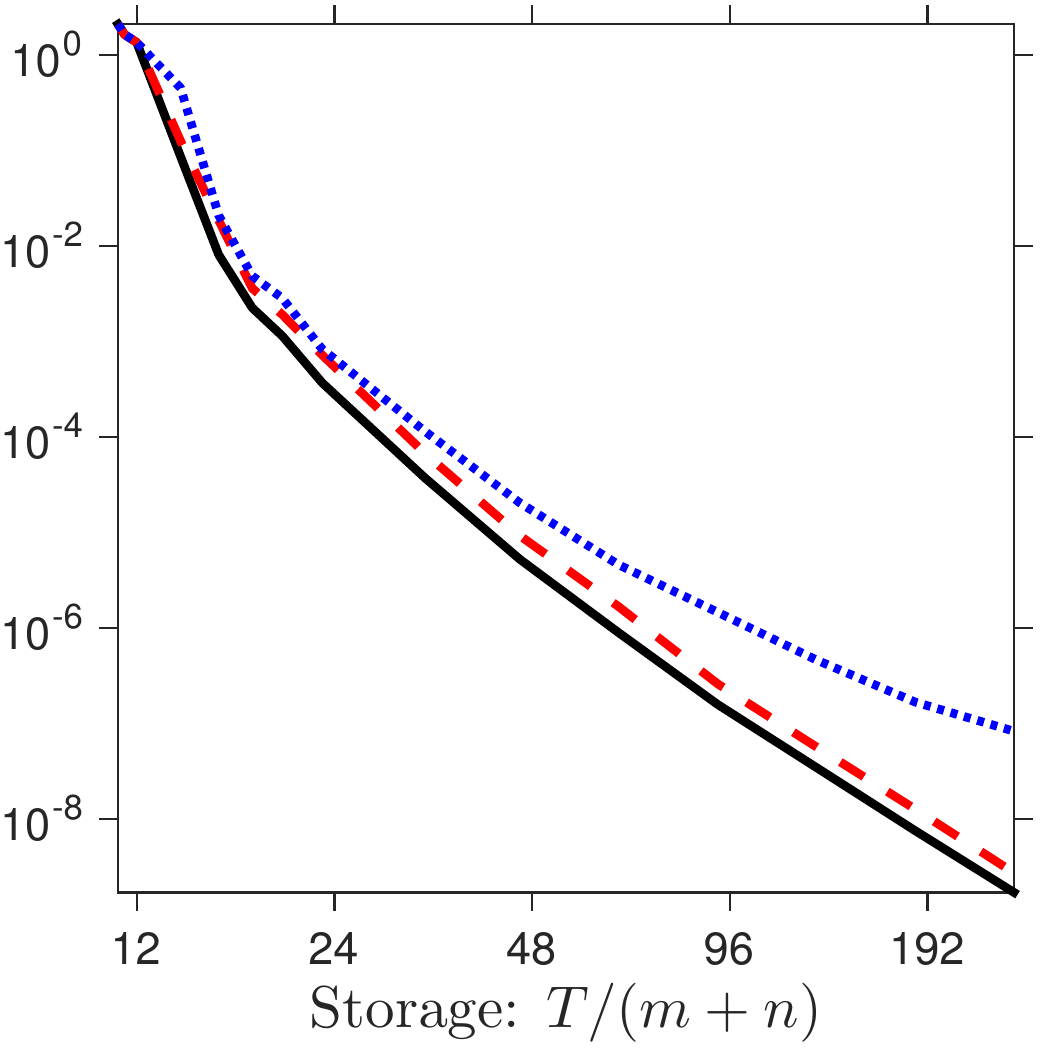}
\caption{\texttt{PolyDecayFast}}
\end{center}
\end{subfigure}
\end{center}

\vspace{0.5em}

\begin{center}
\begin{subfigure}{.325\textwidth}
\begin{center}
\includegraphics[height=1.5in]{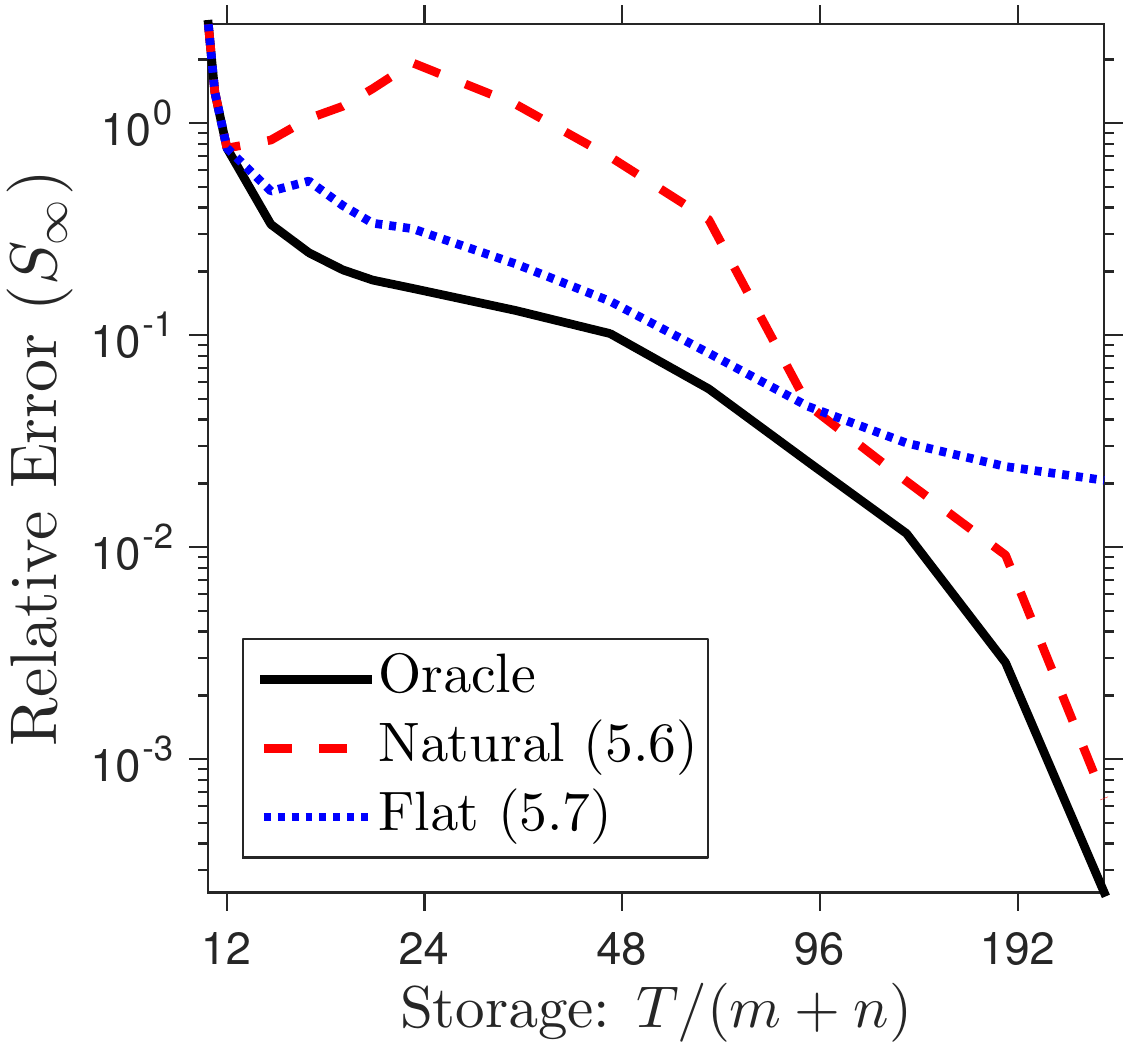}
\caption{\texttt{ExpDecaySlow}}
\end{center}
\end{subfigure}
\begin{subfigure}{.325\textwidth}
\begin{center}
\includegraphics[height=1.5in]{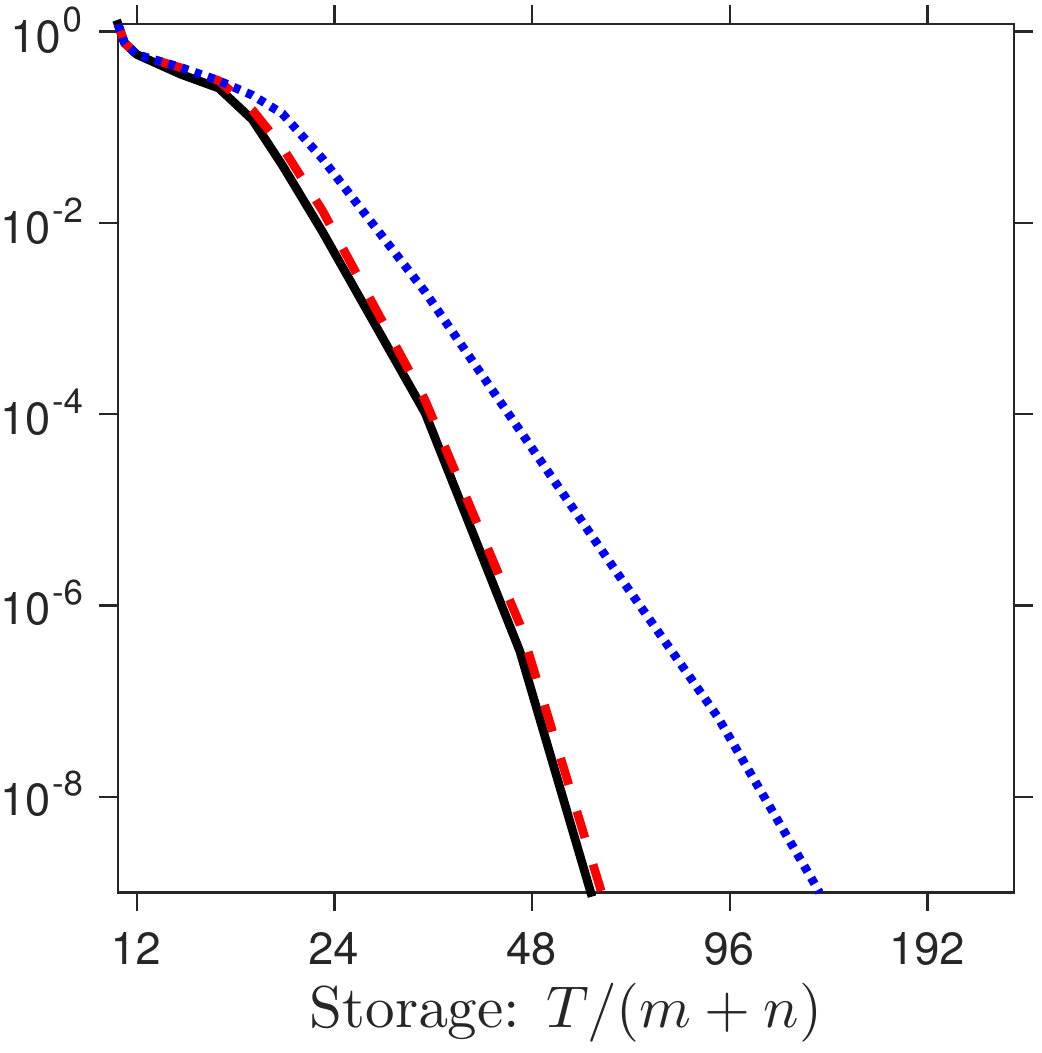}
\caption{\texttt{ExpDecayMed}}
\end{center}
\end{subfigure}
\begin{subfigure}{.325\textwidth}
\begin{center}
\includegraphics[height=1.5in]{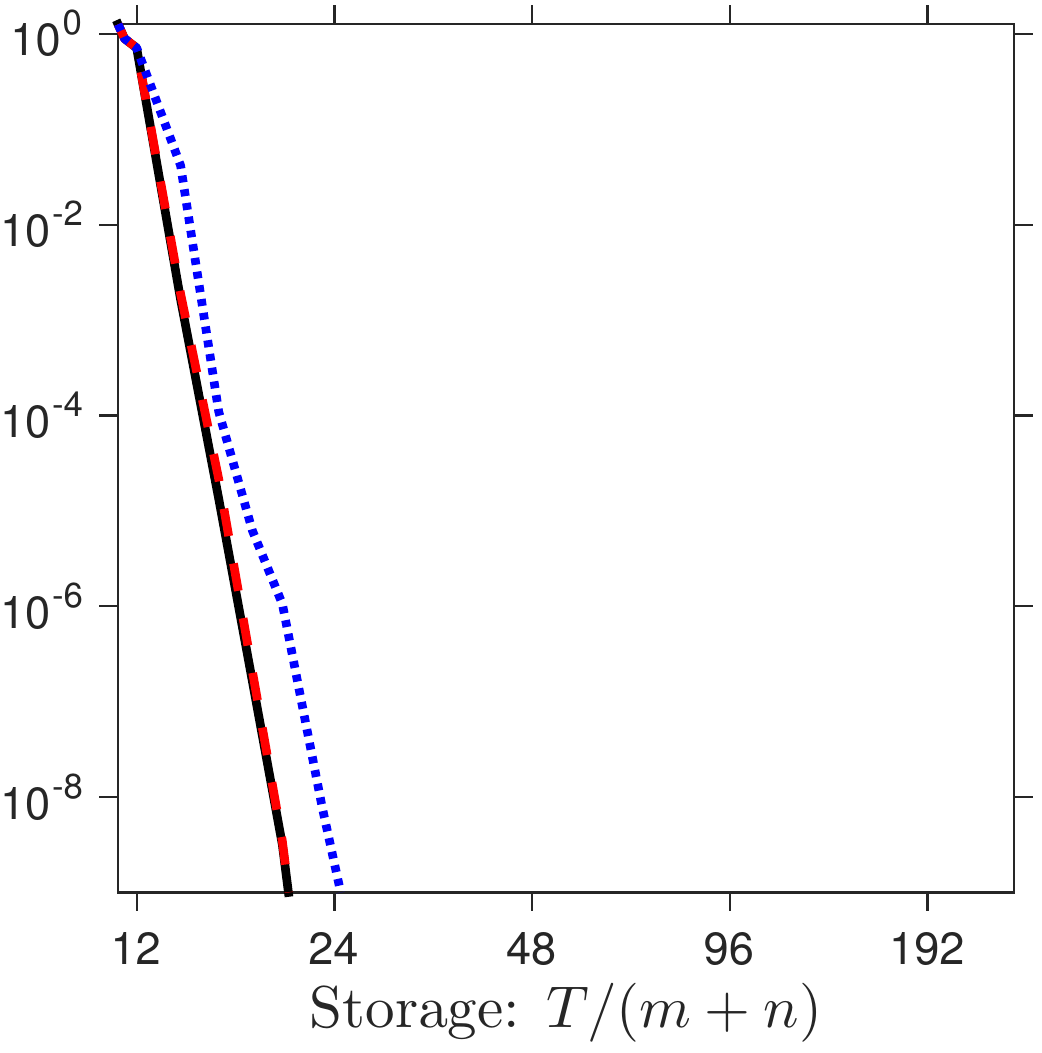}
\caption{\texttt{ExpDecayFast}}
\end{center}
\end{subfigure}
\end{center}

\vspace{0.5em}

\caption{\textbf{Relative error for proposed method with \emph{a priori} parameters.}
(Gaussian maps, effective rank $R = 10$, approximation rank $r = 10$,
Schatten $\infty$-norm.)
We compare the oracle performance of the proposed fixed-rank
approximation~\cref{eqn:Ahat-fixed} with its performance at theoretically justified
parameter values. See \cref{app:oracle-performance} for details.}
\label{fig:theory-params-R10-Sinf}
\end{figure}

\begin{figure}[htp!]
\vspace{0.5in}
\begin{center}
\begin{subfigure}{.325\textwidth}
\begin{center}
\includegraphics[height=1.5in]{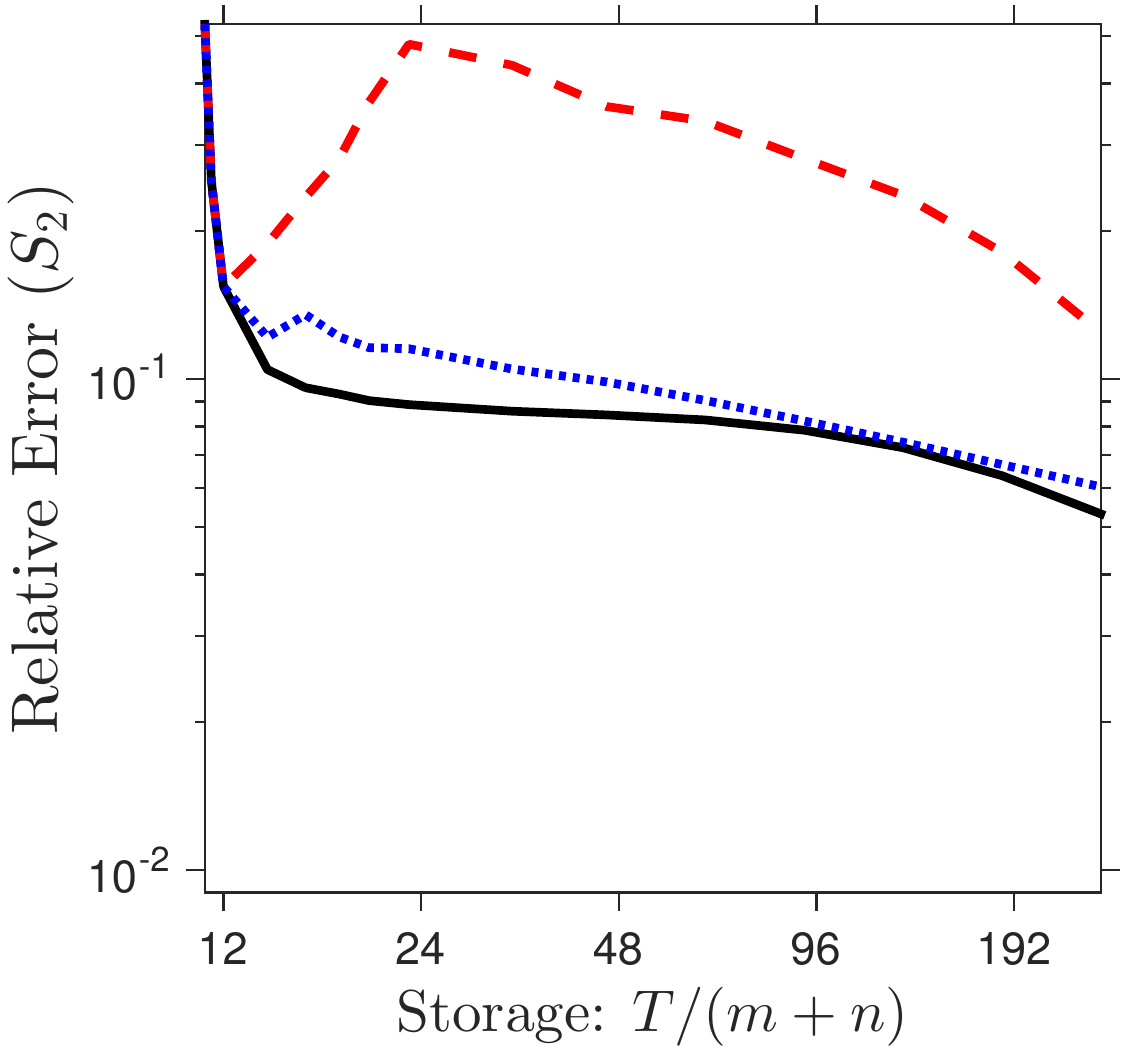}
\caption{\texttt{LowRankHiNoise}}
\end{center}
\end{subfigure}
\begin{subfigure}{.325\textwidth}
\begin{center}
\includegraphics[height=1.5in]{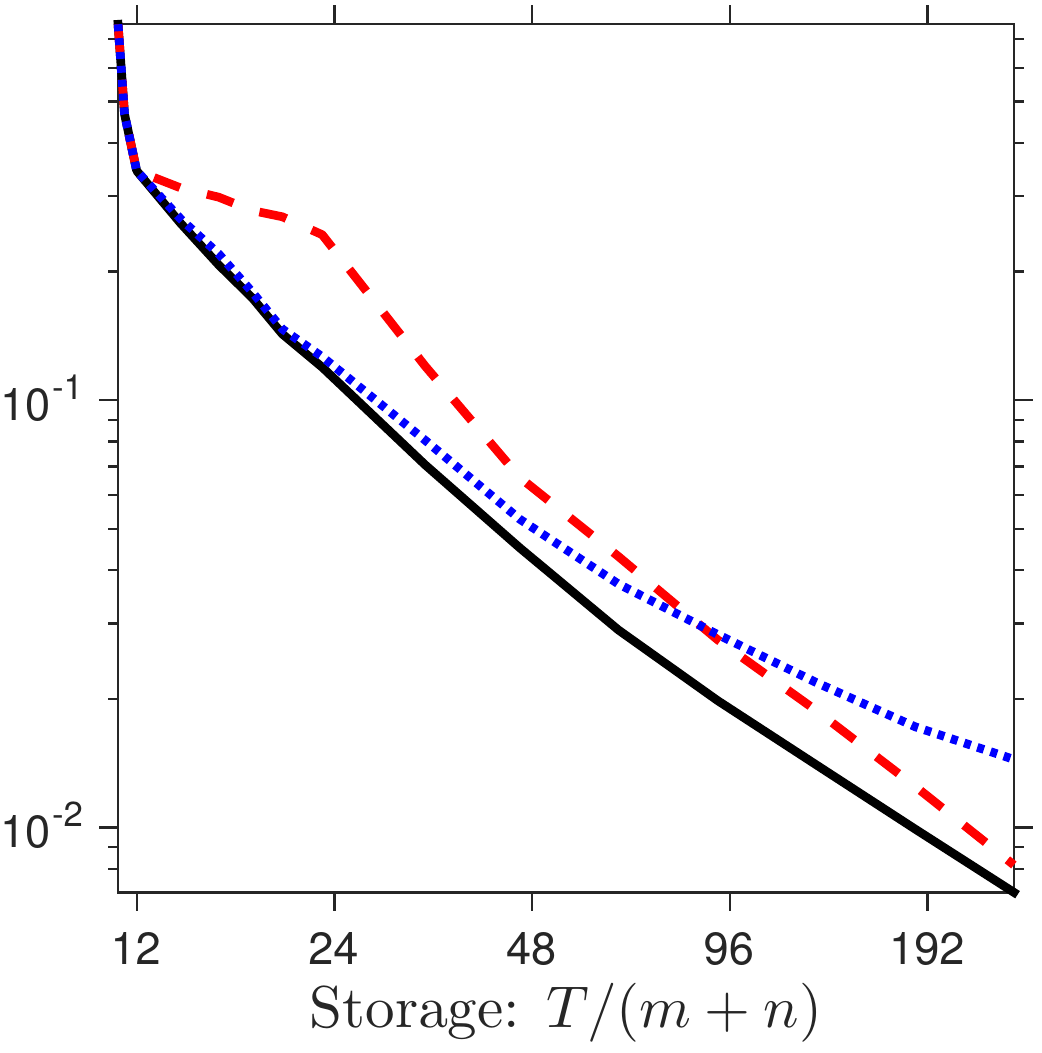}
\caption{\texttt{LowRankMedNoise}}
\end{center}
\end{subfigure}
\begin{subfigure}{.325\textwidth}
\begin{center}
\includegraphics[height=1.5in]{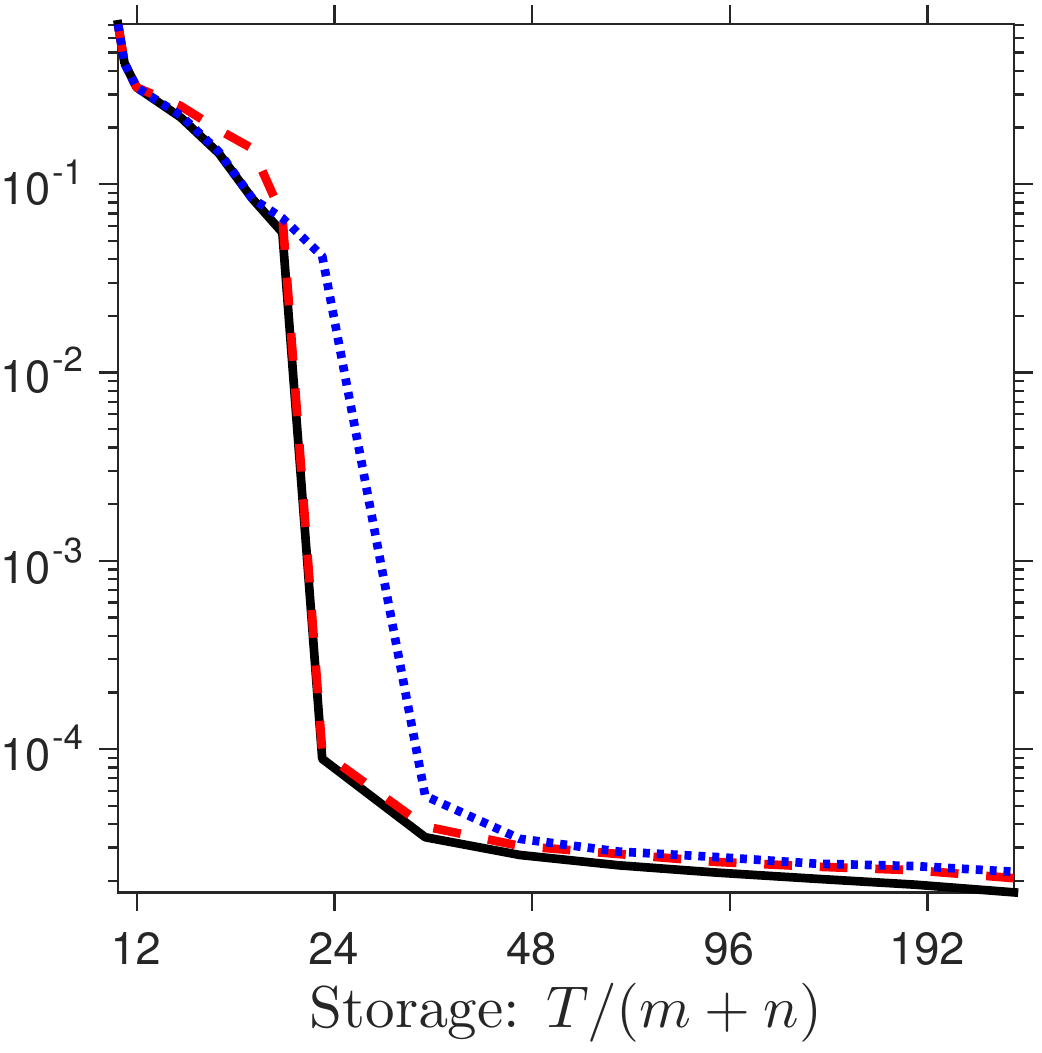}
\caption{\texttt{LowRankLowNoise}}
\end{center}
\end{subfigure}
\end{center}

\vspace{.5em}

\begin{center}
\begin{subfigure}{.325\textwidth}
\begin{center}
\includegraphics[height=1.5in]{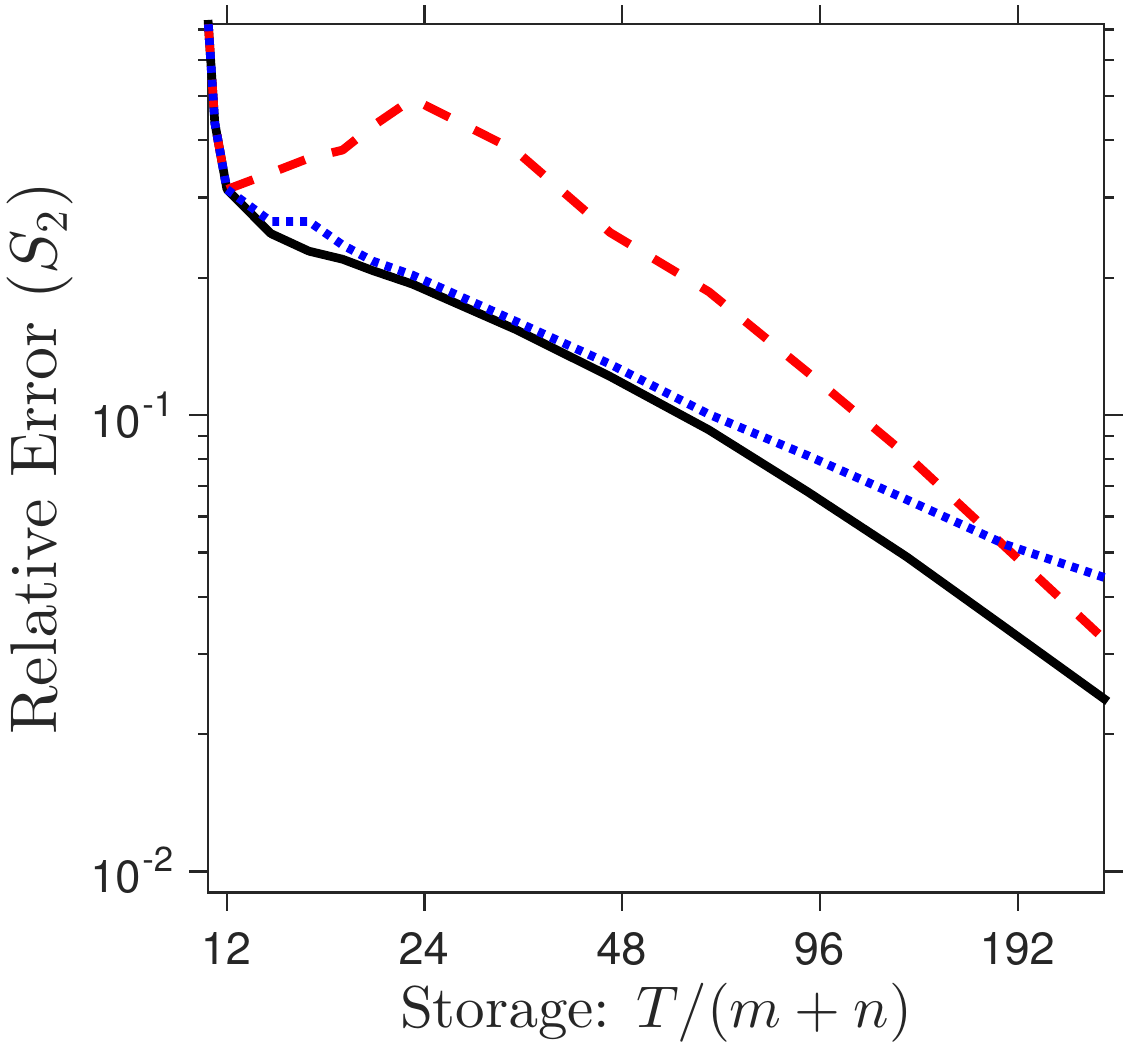}
\caption{\texttt{PolyDecaySlow}}
\end{center}
\end{subfigure}
\begin{subfigure}{.325\textwidth}
\begin{center}
\includegraphics[height=1.5in]{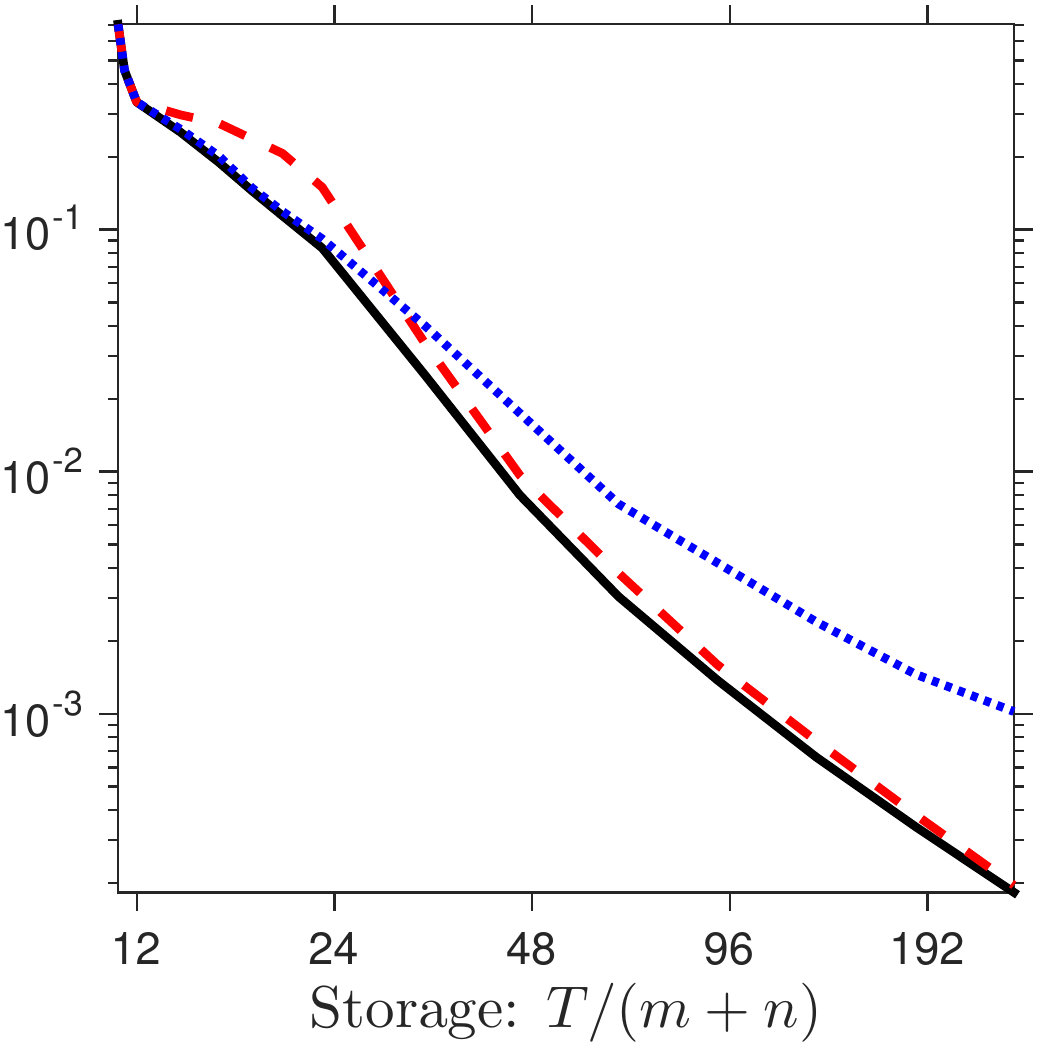}
\caption{\texttt{PolyDecayMed}}
\end{center}
\end{subfigure}
\begin{subfigure}{.325\textwidth}
\begin{center}
\includegraphics[height=1.5in]{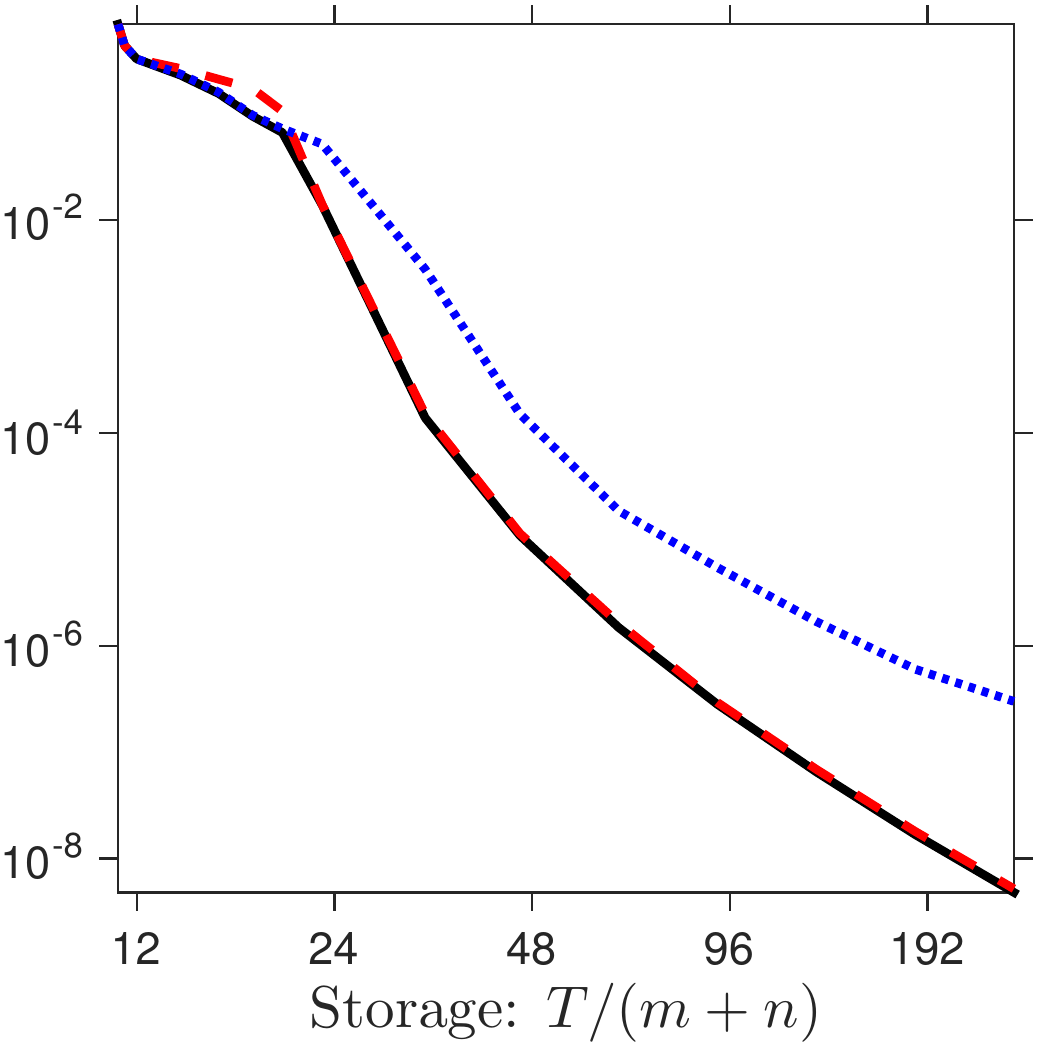}
\caption{\texttt{PolyDecayFast}}
\end{center}
\end{subfigure}
\end{center}

\vspace{0.5em}

\begin{center}
\begin{subfigure}{.325\textwidth}
\begin{center}
\includegraphics[height=1.5in]{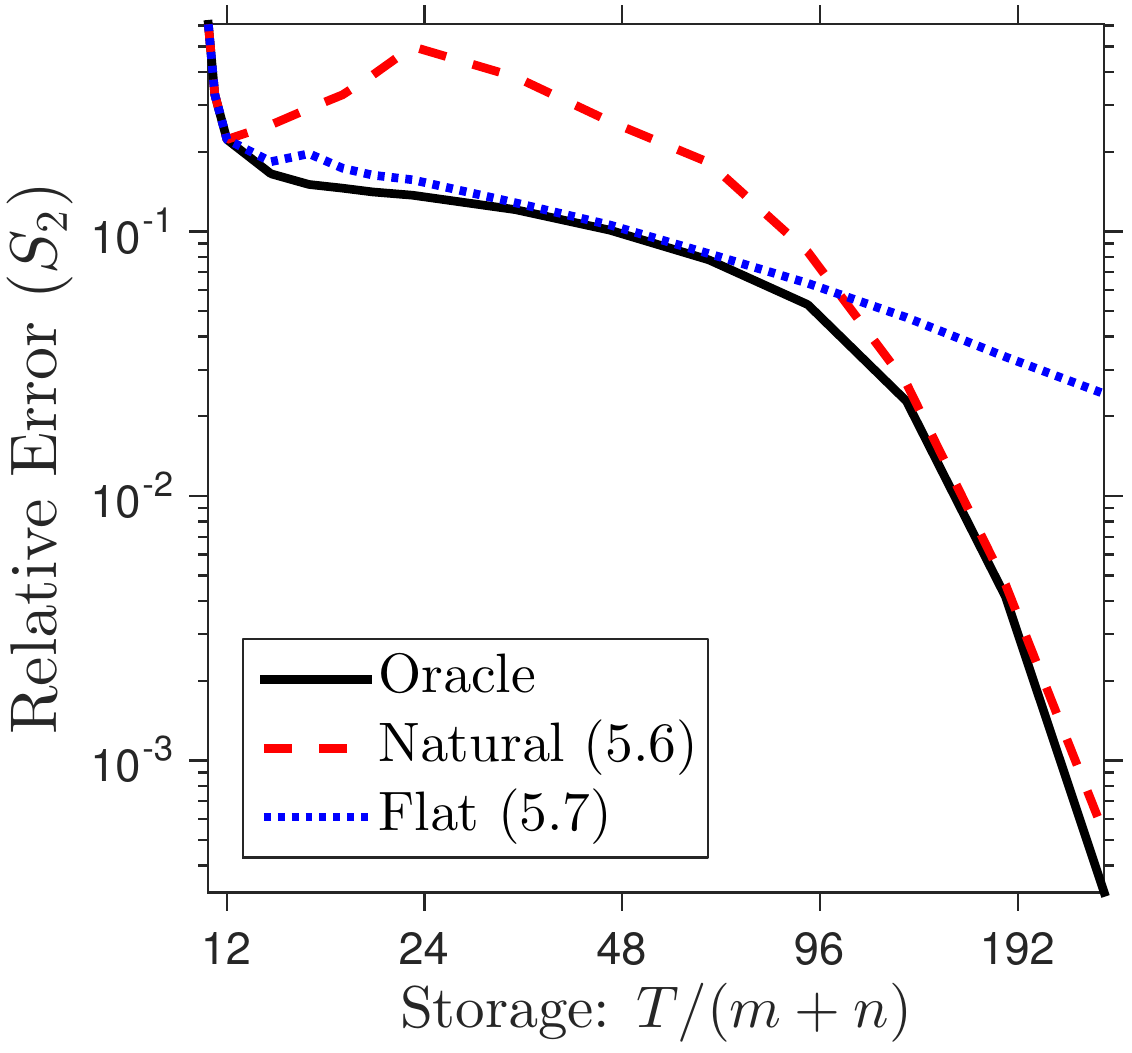}
\caption{\texttt{ExpDecaySlow}}
\end{center}
\end{subfigure}
\begin{subfigure}{.325\textwidth}
\begin{center}
\includegraphics[height=1.5in]{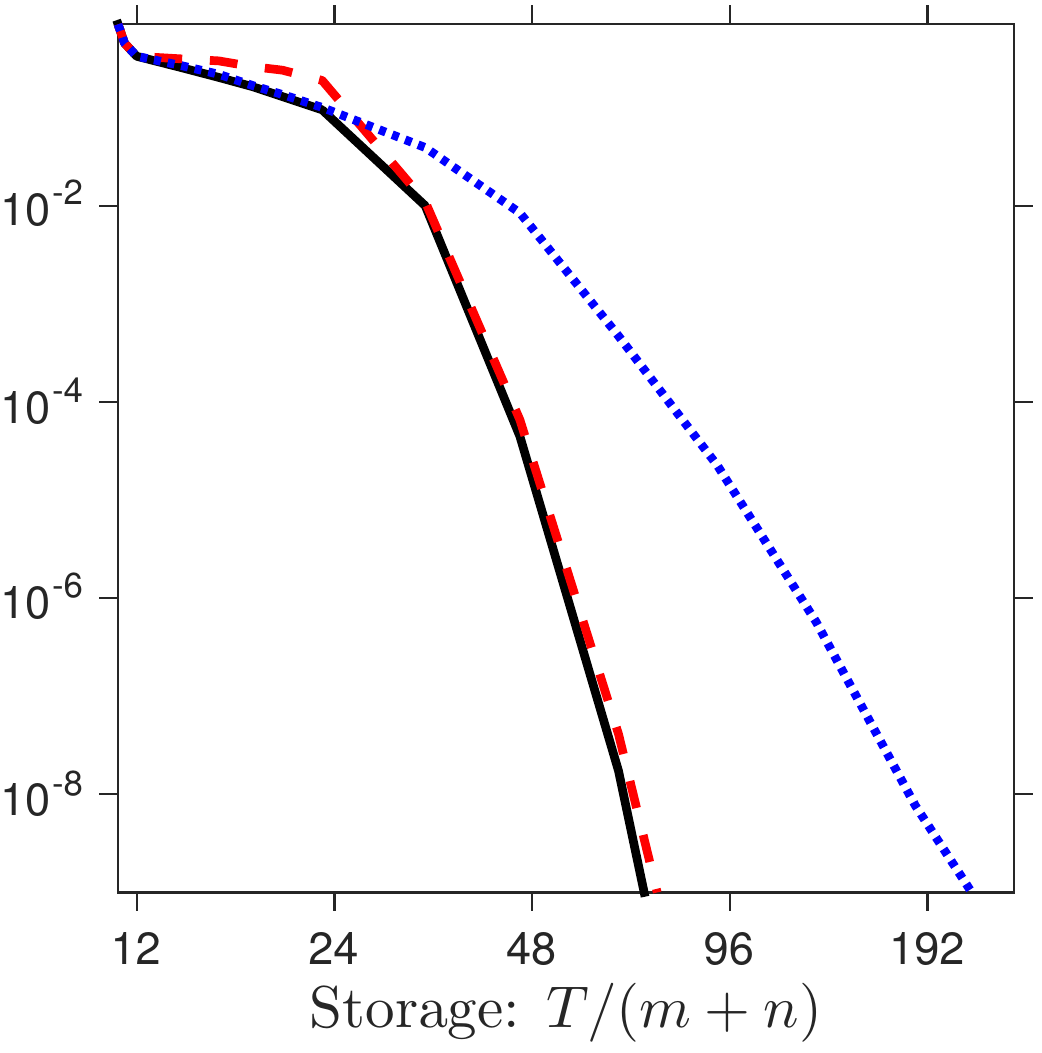}
\caption{\texttt{ExpDecayMed}}
\end{center}
\end{subfigure}
\begin{subfigure}{.325\textwidth}
\begin{center}
\includegraphics[height=1.5in]{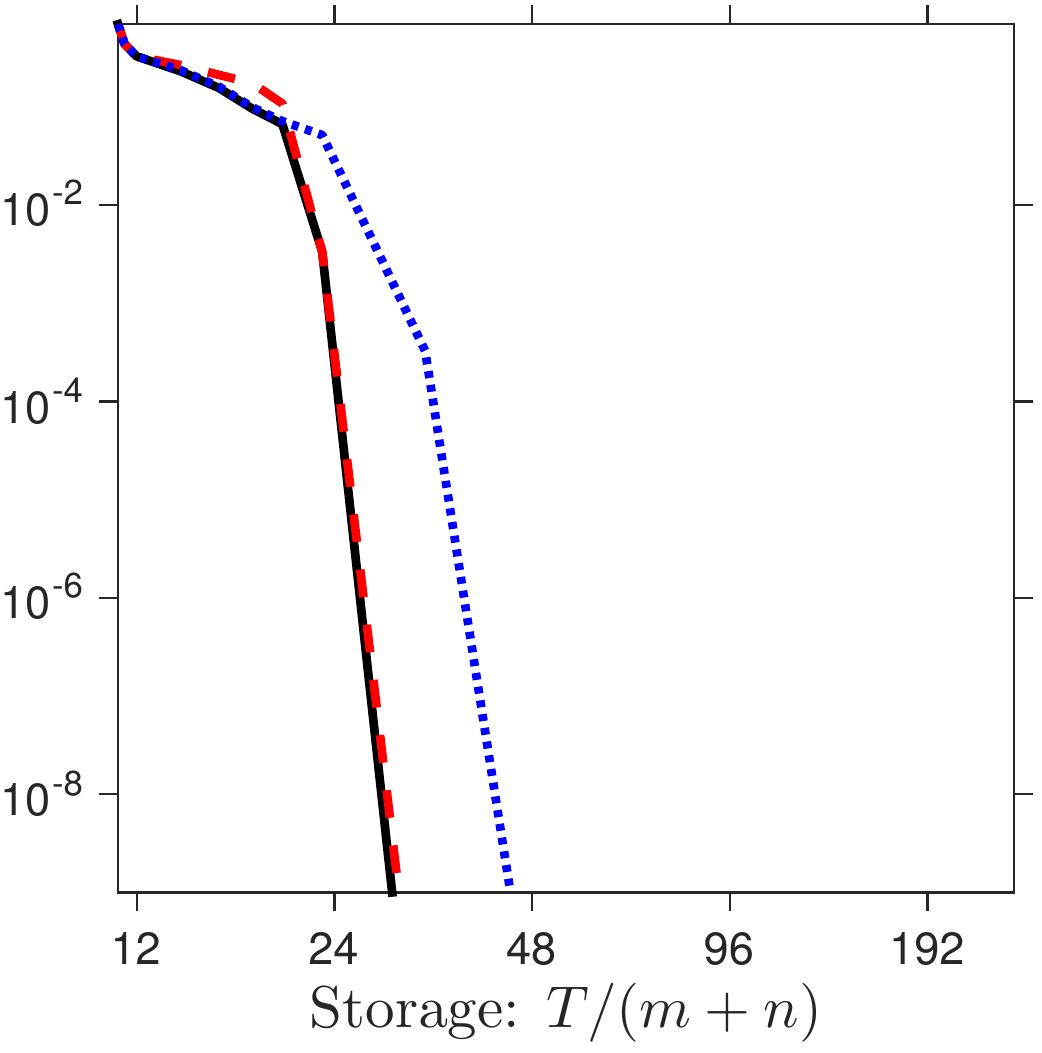}
\caption{\texttt{ExpDecayFast}}
\end{center}
\end{subfigure}
\end{center}

\vspace{0.5em}

\caption{\textbf{Relative error for proposed method with \emph{a priori} parameters.}
(Gaussian maps, effective rank $R = 20$, approximation rank $r = 10$,
Schatten $2$-norm.)
We compare the oracle performance of the proposed fixed-rank
approximation~\cref{eqn:Ahat-fixed} with its performance at theoretically justified
parameter values. See \cref{app:oracle-performance} for details.}
\label{fig:theory-params-R20-S2}
\end{figure}

\begin{figure}[htp!]
\vspace{0.5in}
\begin{center}
\begin{subfigure}{.325\textwidth}
\begin{center}
\includegraphics[height=1.5in]{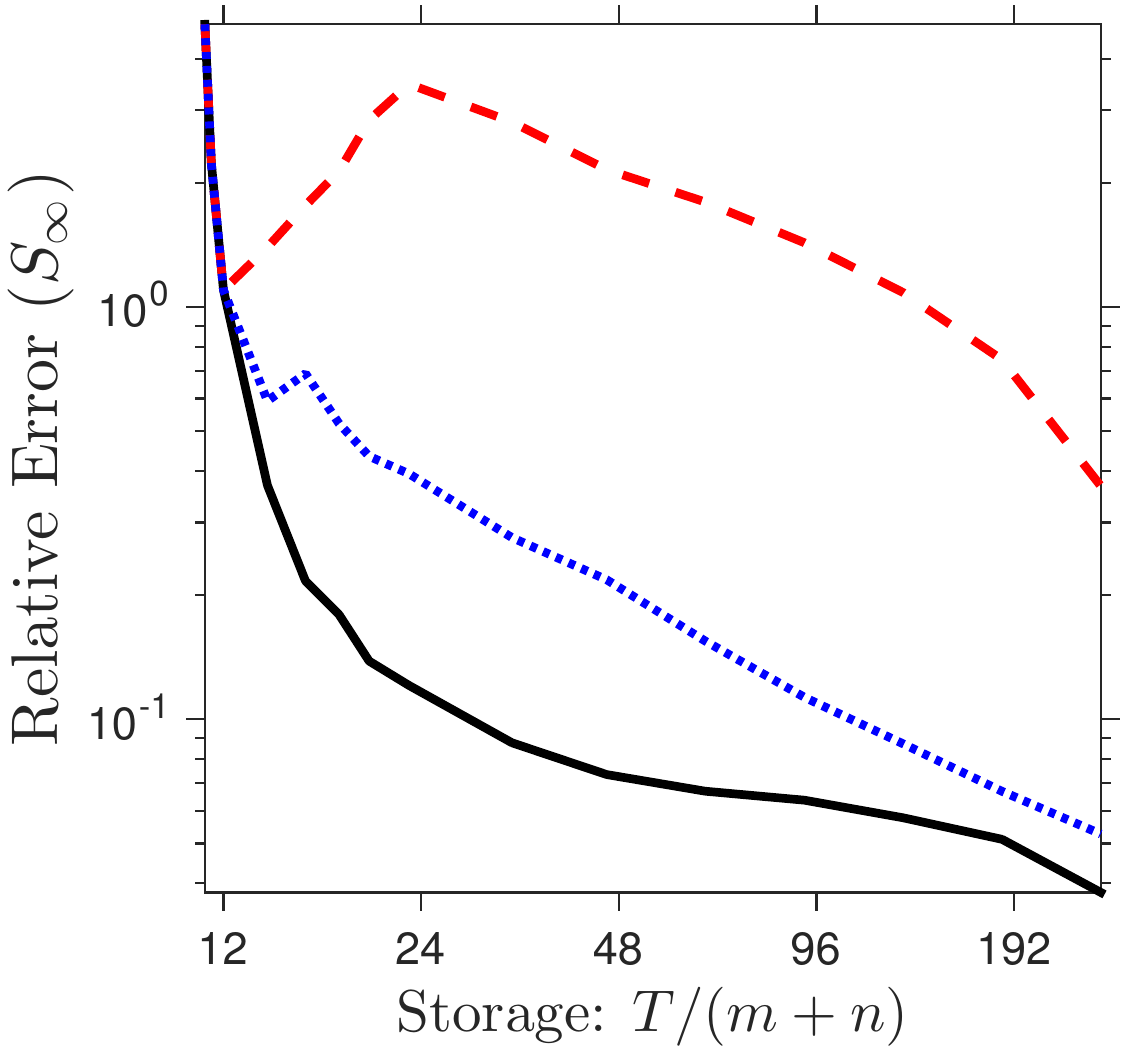}
\caption{\texttt{LowRankHiNoise}}
\end{center}
\end{subfigure}
\begin{subfigure}{.325\textwidth}
\begin{center}
\includegraphics[height=1.5in]{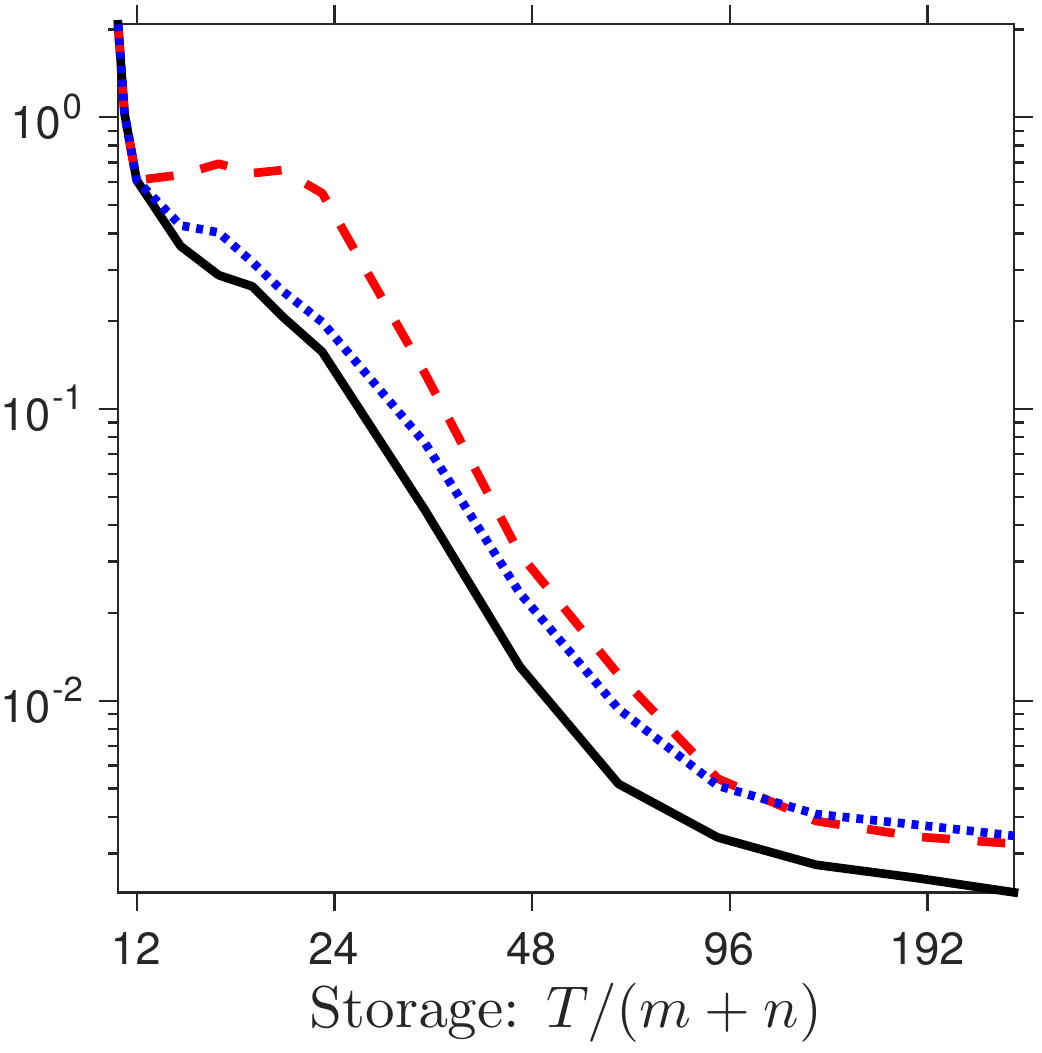}
\caption{\texttt{LowRankMedNoise}}
\end{center}
\end{subfigure}
\begin{subfigure}{.325\textwidth}
\begin{center}
\includegraphics[height=1.5in]{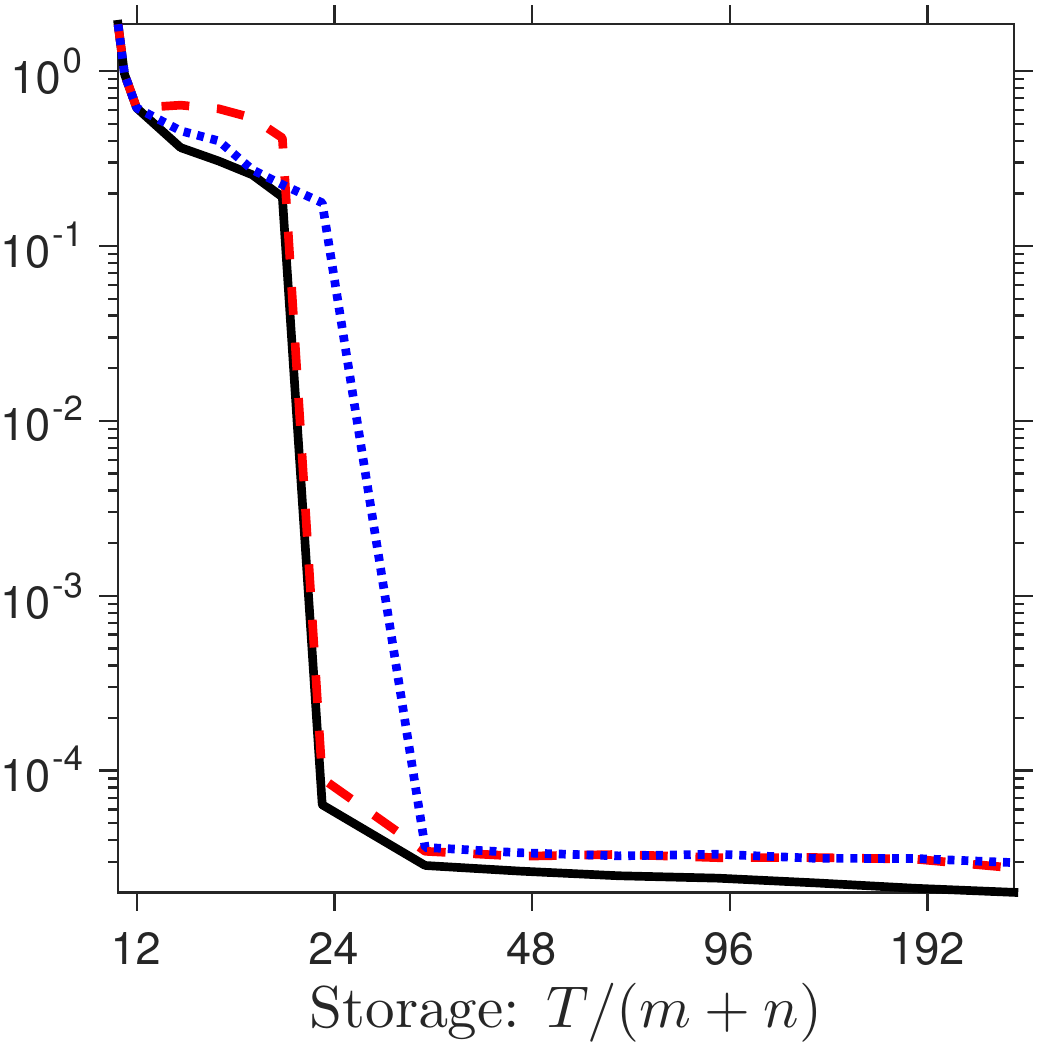}
\caption{\texttt{LowRankLowNoise}}
\end{center}
\end{subfigure}
\end{center}

\vspace{.5em}

\begin{center}
\begin{subfigure}{.325\textwidth}
\begin{center}
\includegraphics[height=1.5in]{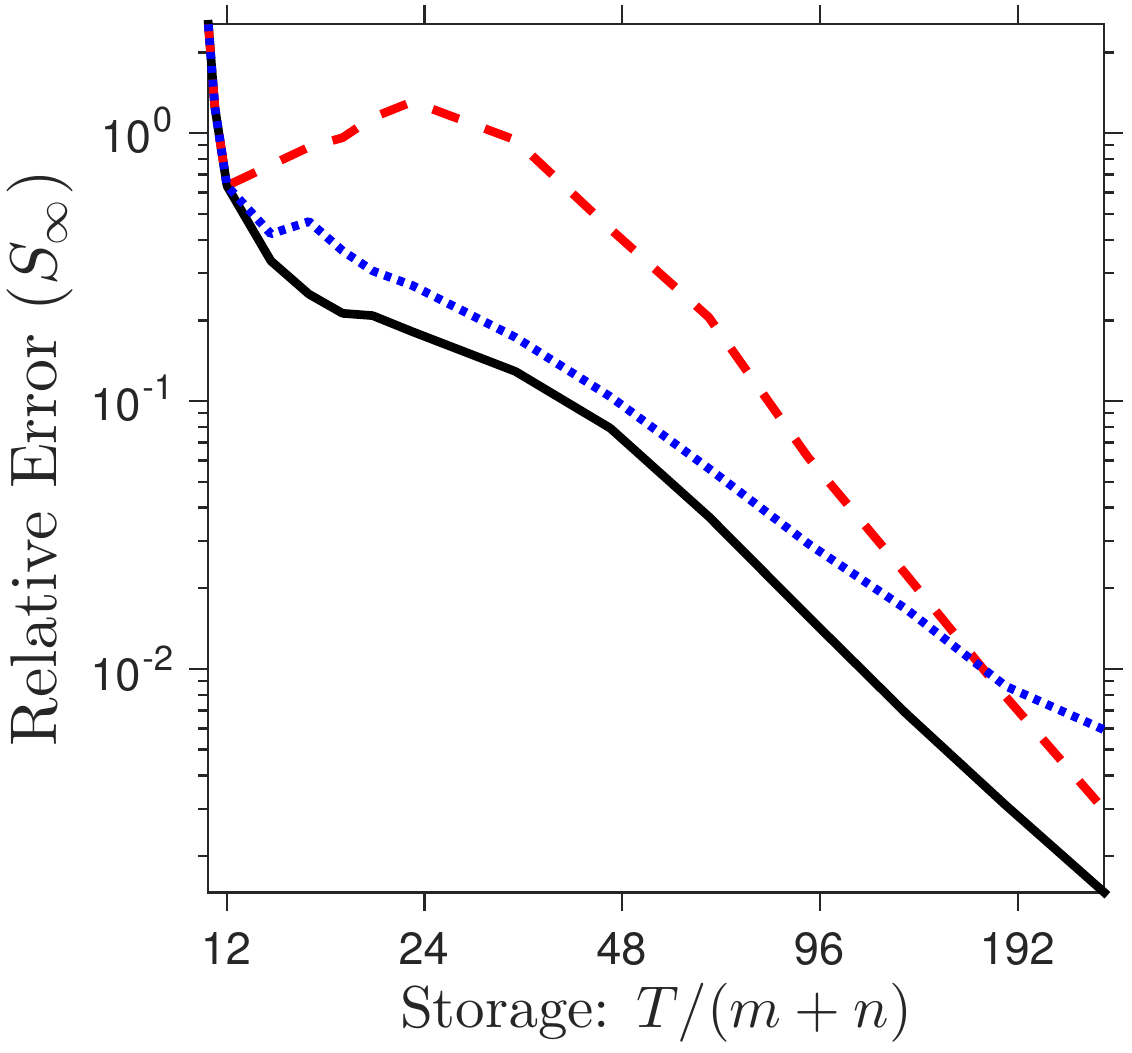}
\caption{\texttt{PolyDecaySlow}}
\end{center}
\end{subfigure}
\begin{subfigure}{.325\textwidth}
\begin{center}
\includegraphics[height=1.5in]{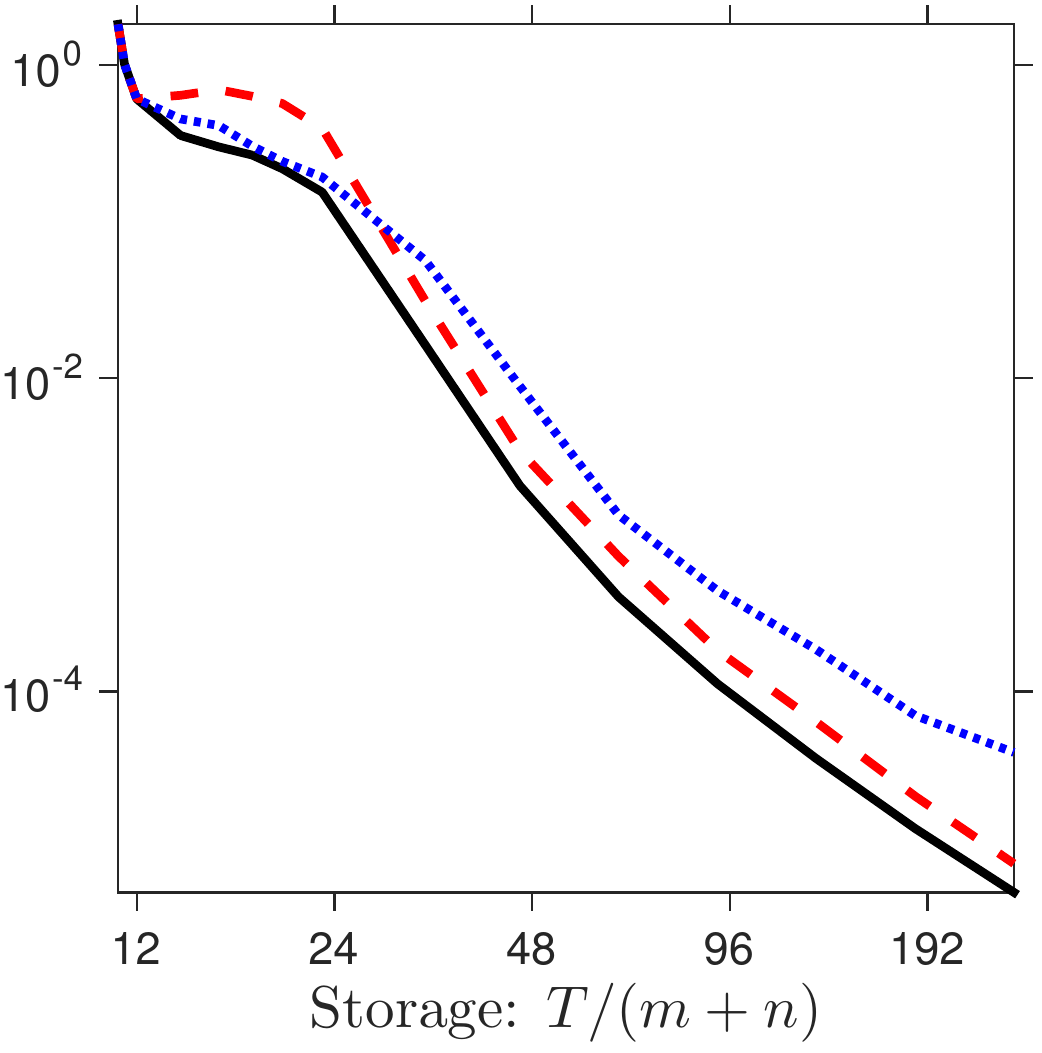}
\caption{\texttt{PolyDecayMed}}
\end{center}
\end{subfigure}
\begin{subfigure}{.325\textwidth}
\begin{center}
\includegraphics[height=1.5in]{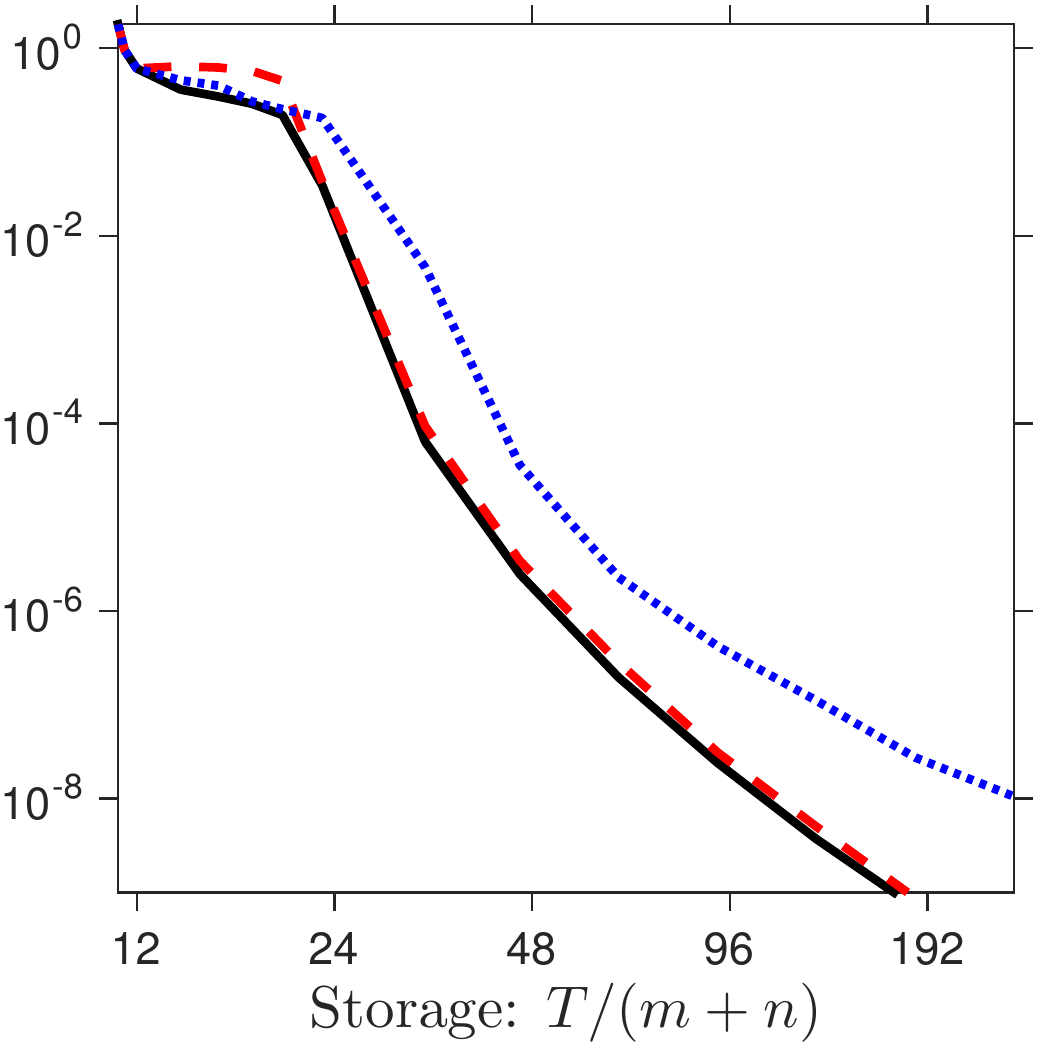}
\caption{\texttt{PolyDecayFast}}
\end{center}
\end{subfigure}
\end{center}

\vspace{0.5em}

\begin{center}
\begin{subfigure}{.325\textwidth}
\begin{center}
\includegraphics[height=1.5in]{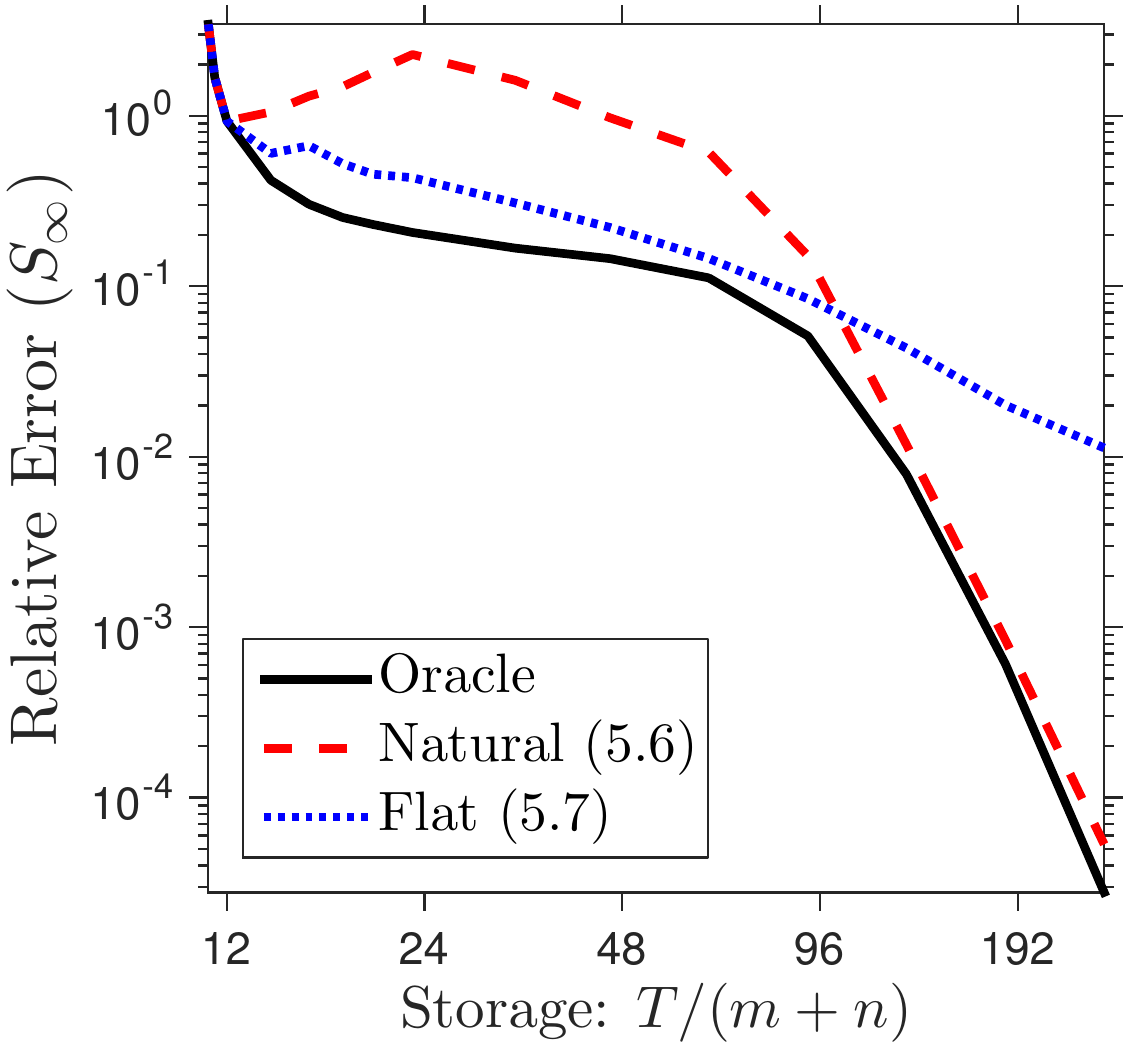}
\caption{\texttt{ExpDecaySlow}}
\end{center}
\end{subfigure}
\begin{subfigure}{.325\textwidth}
\begin{center}
\includegraphics[height=1.5in]{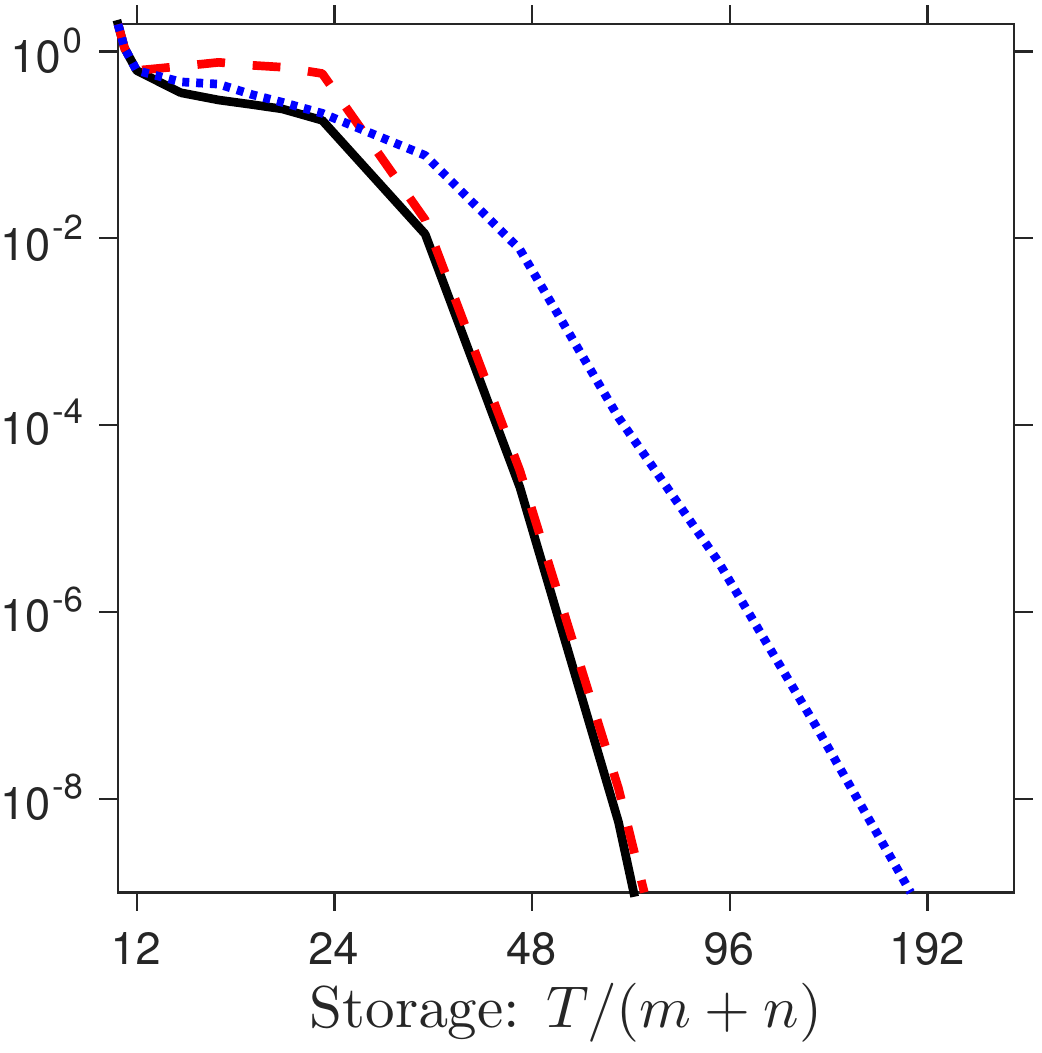}
\caption{\texttt{ExpDecayMed}}
\end{center}
\end{subfigure}
\begin{subfigure}{.325\textwidth}
\begin{center}
\includegraphics[height=1.5in]{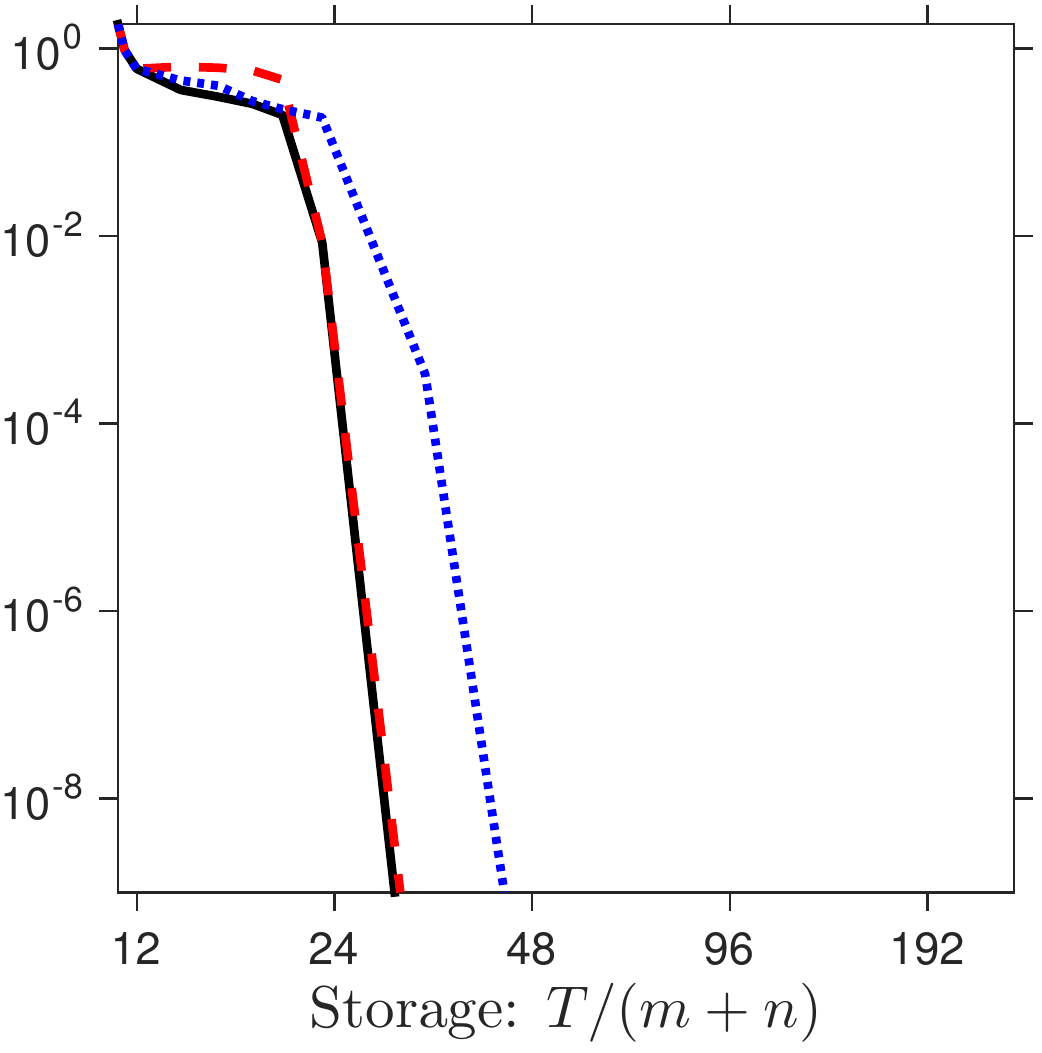}
\caption{\texttt{ExpDecayFast}}
\end{center}
\end{subfigure}
\end{center}

\vspace{0.5em}

\caption{\textbf{Relative error for proposed method with \emph{a priori} parameters.}
(Gaussian maps, effective rank $R = 20$, approximation rank $r = 10$,
Schatten $\infty$-norm.)
We compare the oracle performance of the proposed fixed-rank
approximation~\cref{eqn:Ahat-fixed} with its performance at theoretically justified
parameter values. See \cref{app:oracle-performance} for details.}
\label{fig:theory-params-R20-Sinf}
\end{figure}

\begin{figure}[htp!]
\begin{center}
\begin{subfigure}{.325\textwidth}
\begin{center}
\includegraphics[height=1.5in]{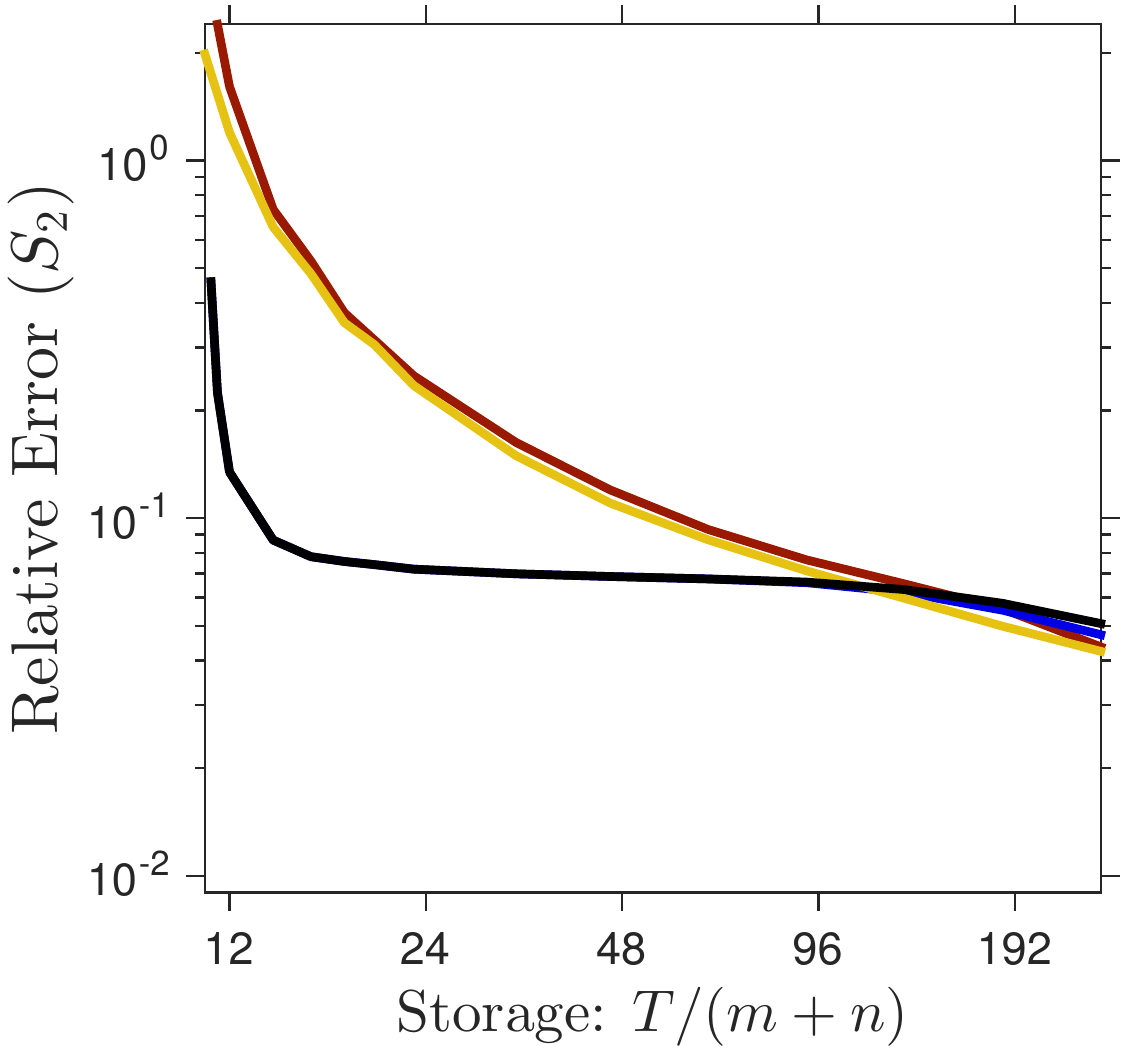}
\caption{\texttt{LowRankHiNoise}}
\end{center}
\end{subfigure}
\begin{subfigure}{.325\textwidth}
\begin{center}
\includegraphics[height=1.5in]{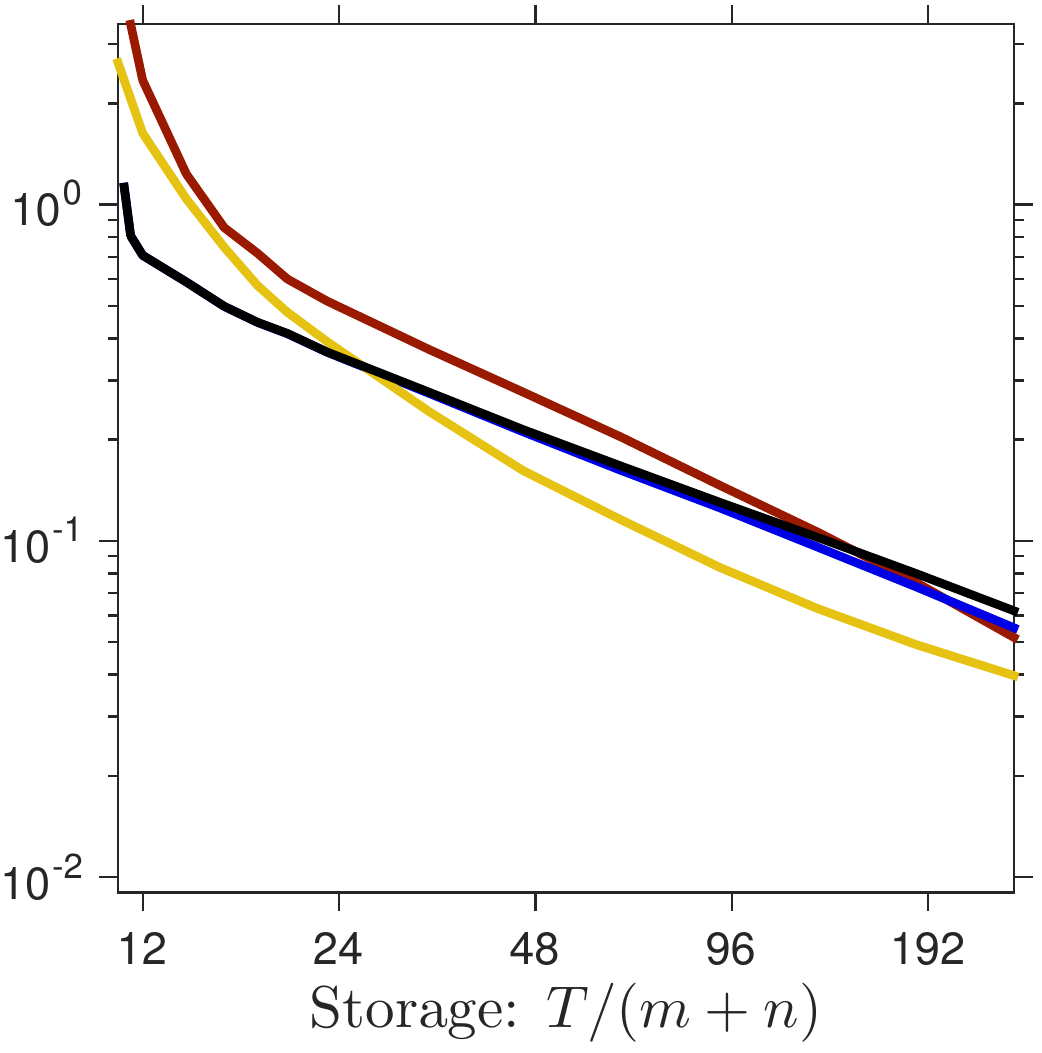}
\caption{\texttt{LowRankMedNoise}}
\end{center}
\end{subfigure}
\begin{subfigure}{.325\textwidth}
\begin{center}
\includegraphics[height=1.5in]{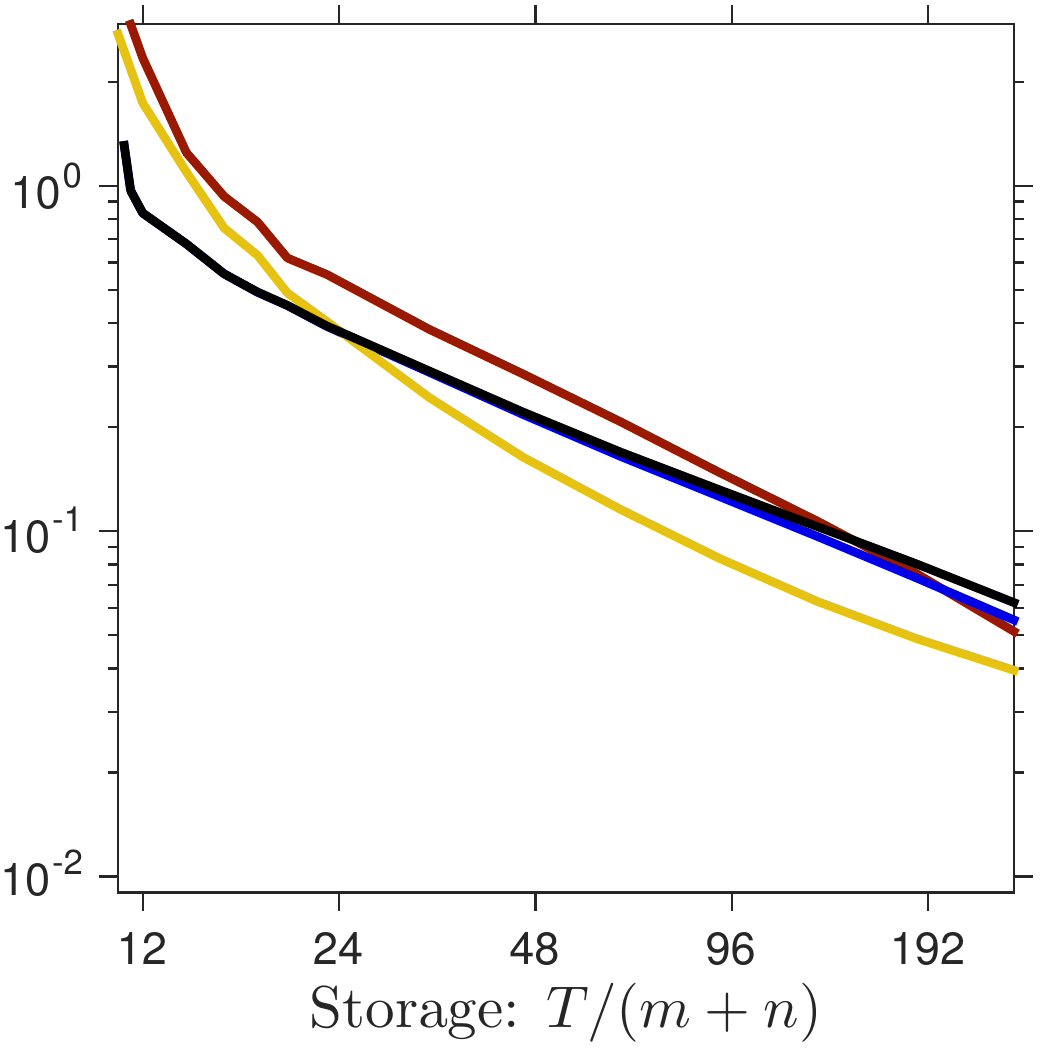}
\caption{\texttt{LowRankLowNoise}}
\end{center}
\end{subfigure}
\end{center}

\vspace{.5em}

\begin{center}
\begin{subfigure}{.325\textwidth}
\begin{center}
\includegraphics[height=1.5in]{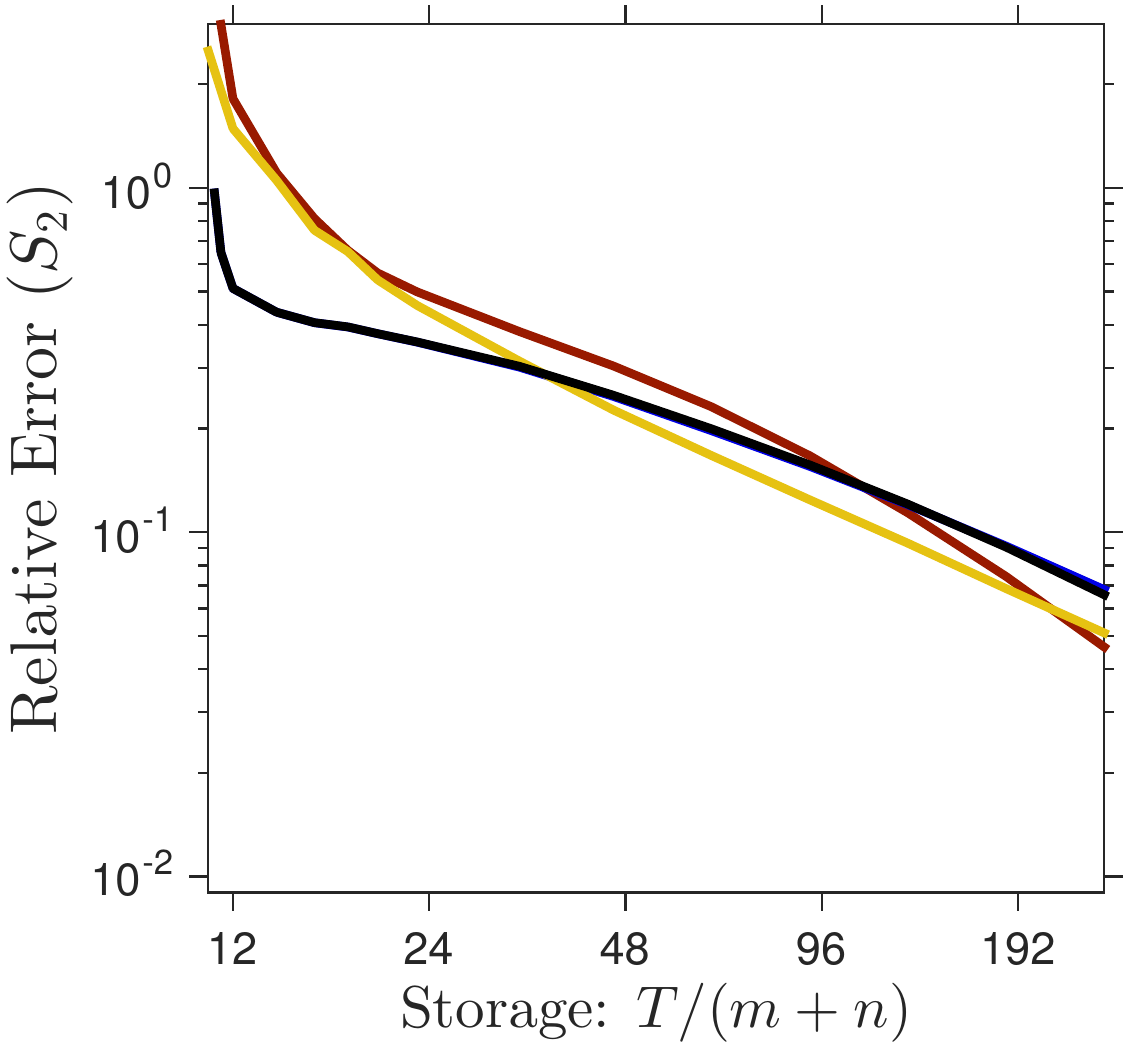}
\caption{\texttt{PolyDecaySlow}}
\end{center}
\end{subfigure}
\begin{subfigure}{.325\textwidth}
\begin{center}
\includegraphics[height=1.5in]{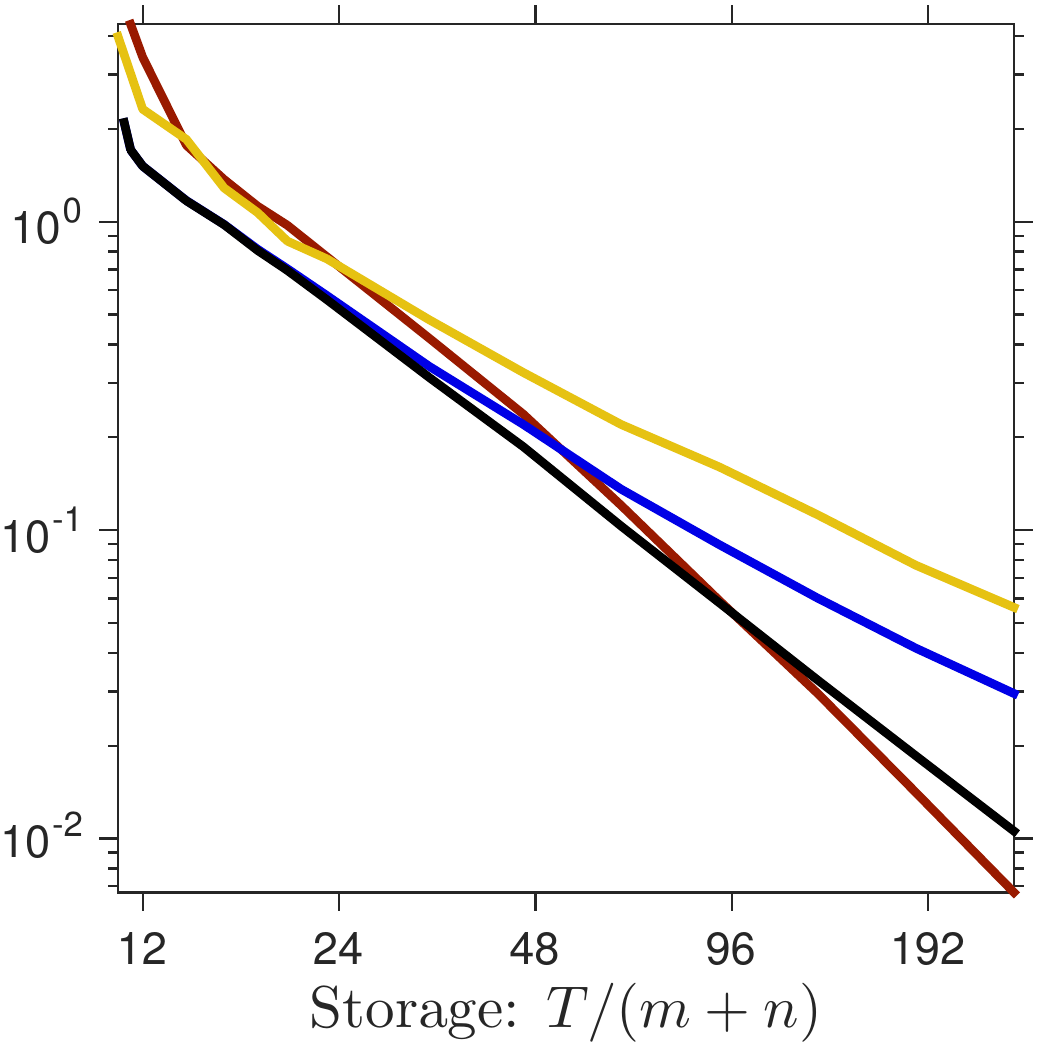}
\caption{\texttt{PolyDecayMed}}
\end{center}
\end{subfigure}
\begin{subfigure}{.325\textwidth}
\begin{center}
\includegraphics[height=1.5in]{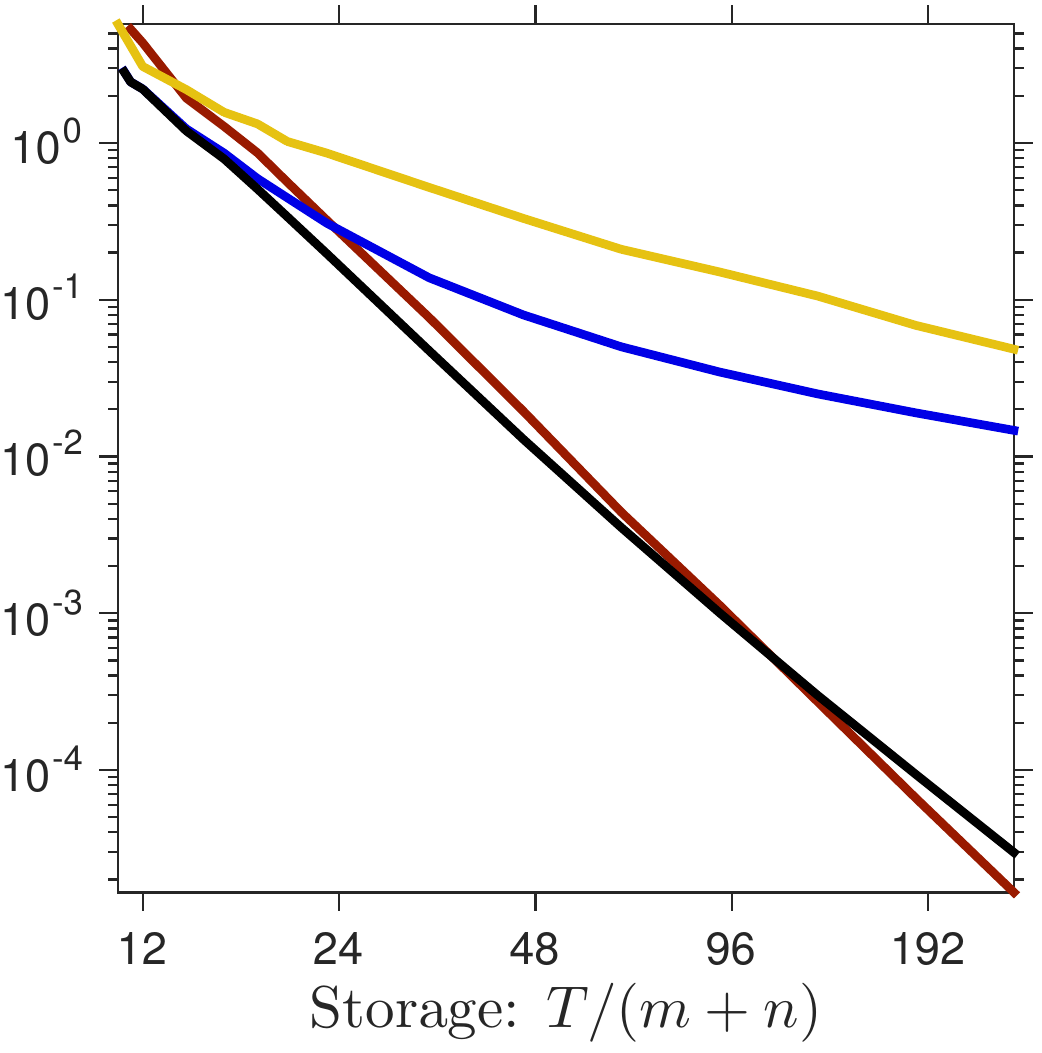}
\caption{\texttt{PolyDecayFast}}
\end{center}
\end{subfigure}
\end{center}

\vspace{0.5em}

\begin{center}
\begin{subfigure}{.325\textwidth}
\begin{center}
\includegraphics[height=1.5in]{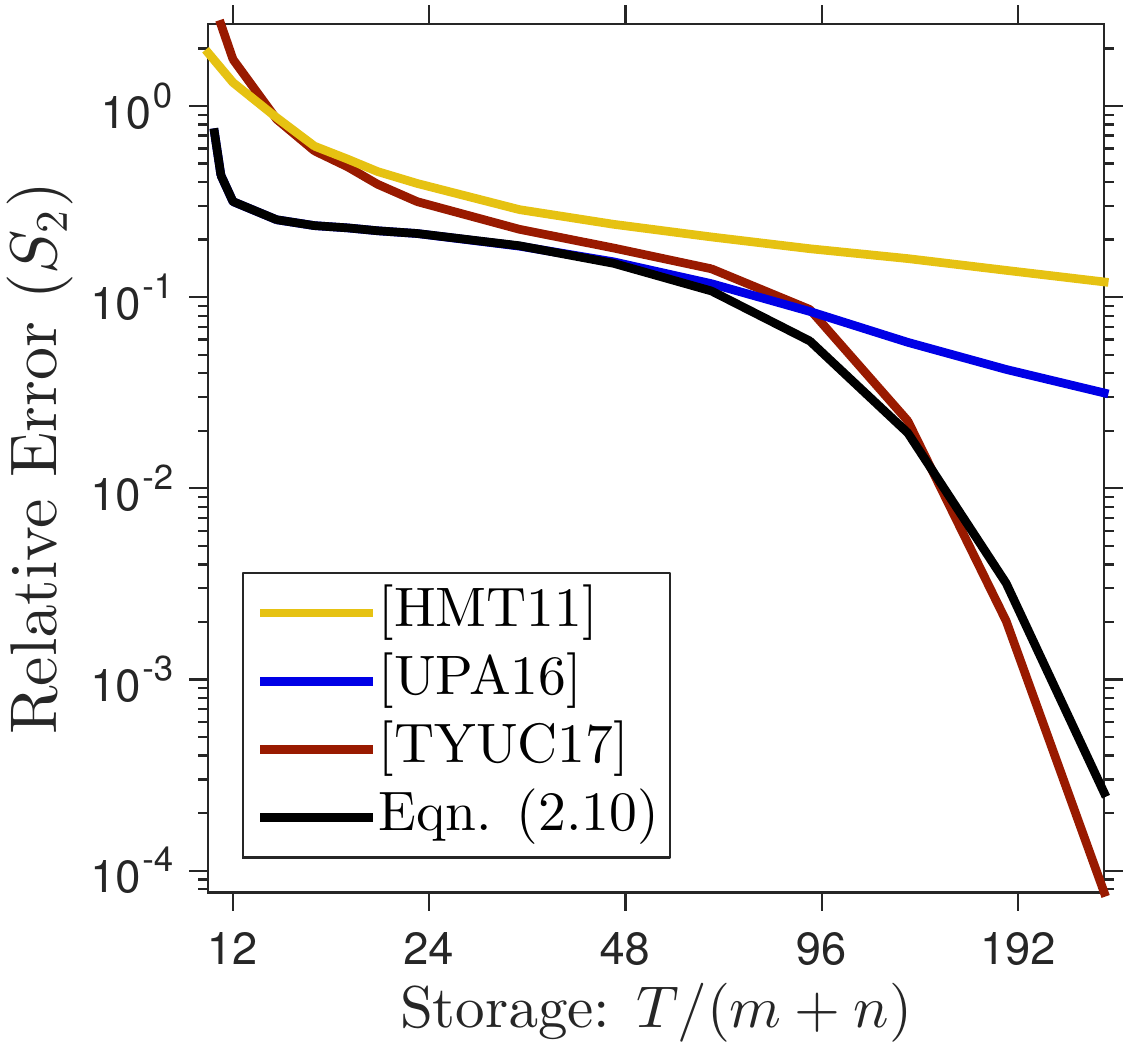}
\caption{\texttt{ExpDecaySlow}}
\end{center}
\end{subfigure}
\begin{subfigure}{.325\textwidth}
\begin{center}
\includegraphics[height=1.5in]{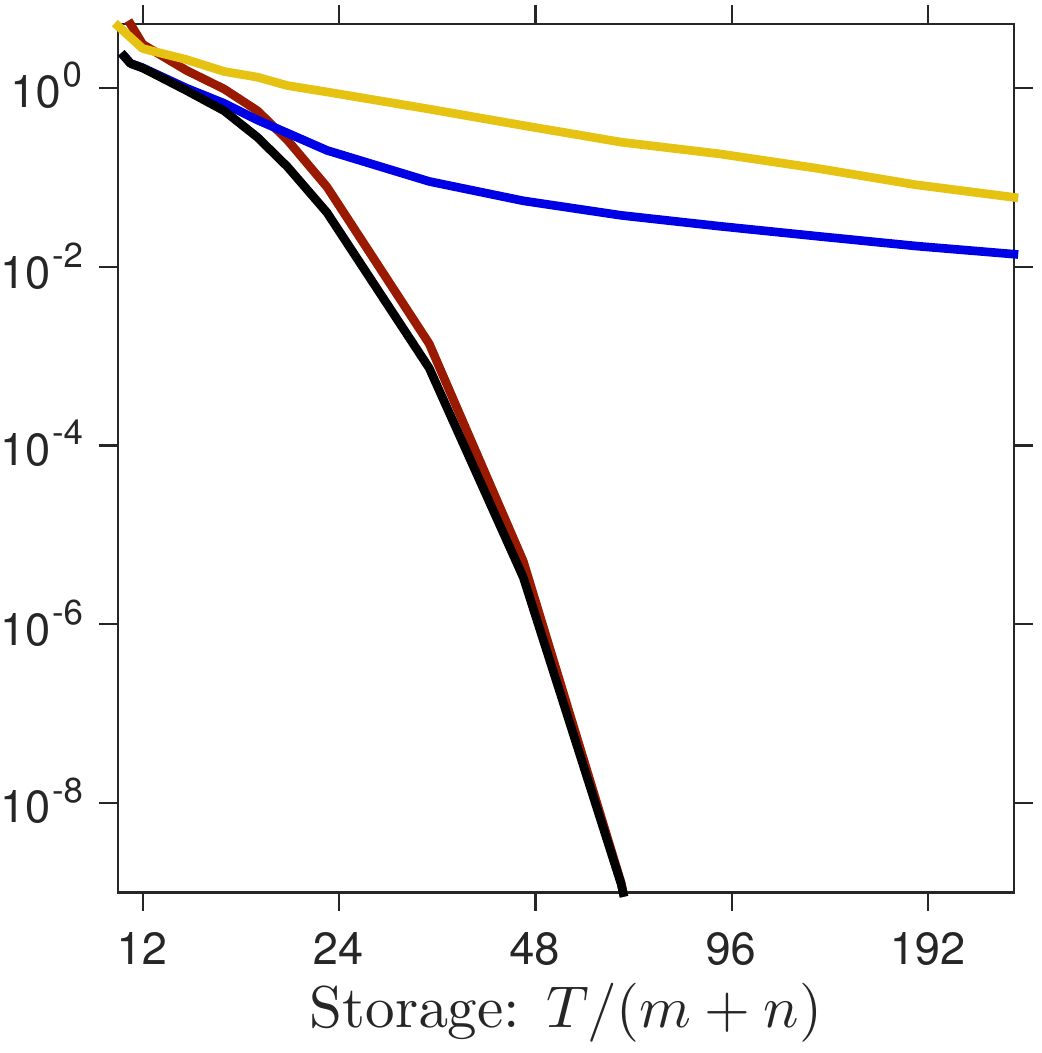}
\caption{\texttt{ExpDecayMed}}
\end{center}
\end{subfigure}
\begin{subfigure}{.325\textwidth}
\begin{center}
\includegraphics[height=1.5in]{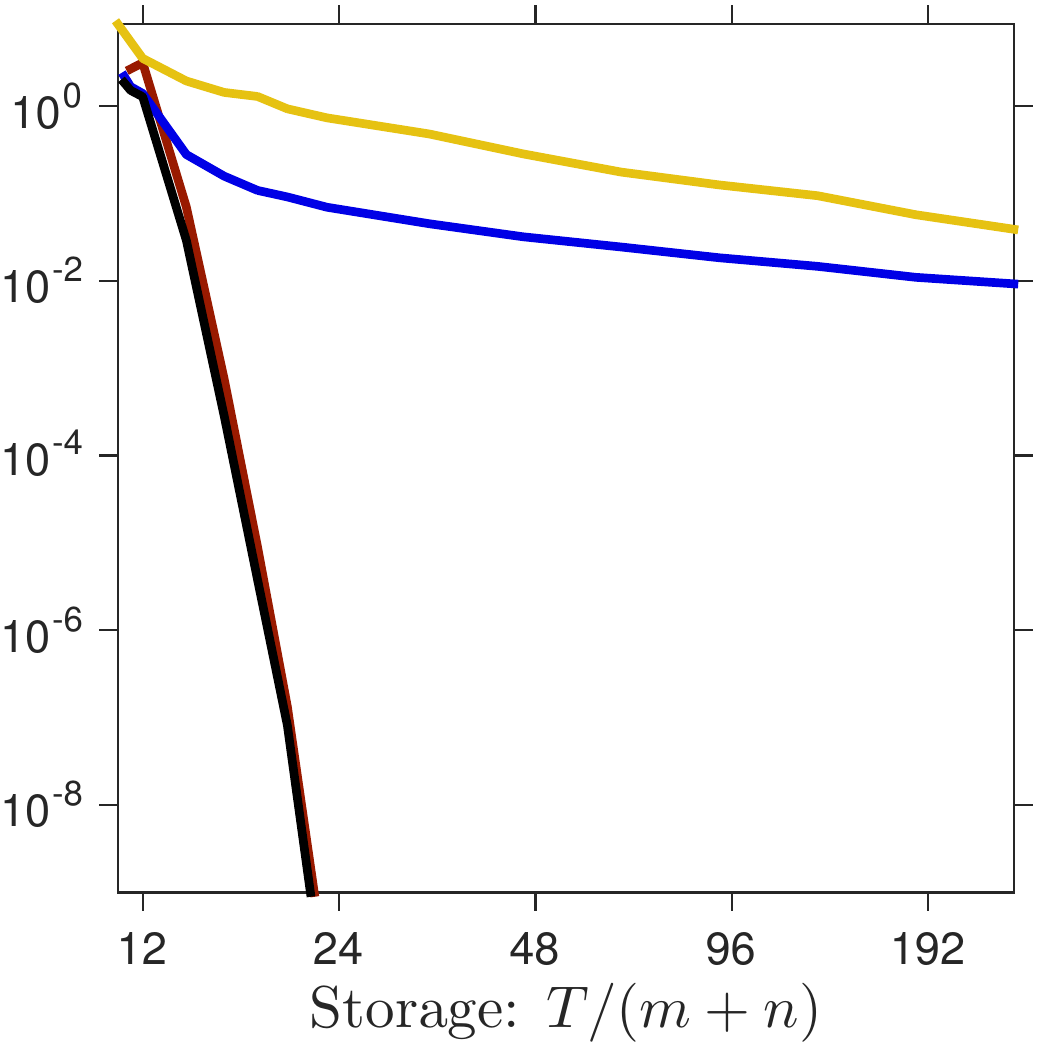}
\caption{\texttt{ExpDecayFast}}
\end{center}
\end{subfigure}
\end{center}

\vspace{0.5em}

\caption{\textbf{Comparison of reconstruction formulas: Synthetic examples.}
(Gaussian maps, effective rank $R = 5$, approximation rank $r = 10$, Schatten 2-norm.)
We compare the oracle error achieved by the proposed fixed-rank
approximation~\cref{eqn:Ahat-fixed} against methods~\cref{eqn:upa,eqn:tyuc2017} from the literature.
See \cref{sec:oracle-error} for details.}
\label{fig:oracle-comparison-R5-S2}
\end{figure}

\begin{figure}[htp!]
\begin{center}
\begin{subfigure}{.325\textwidth}
\begin{center}
\includegraphics[height=1.5in]{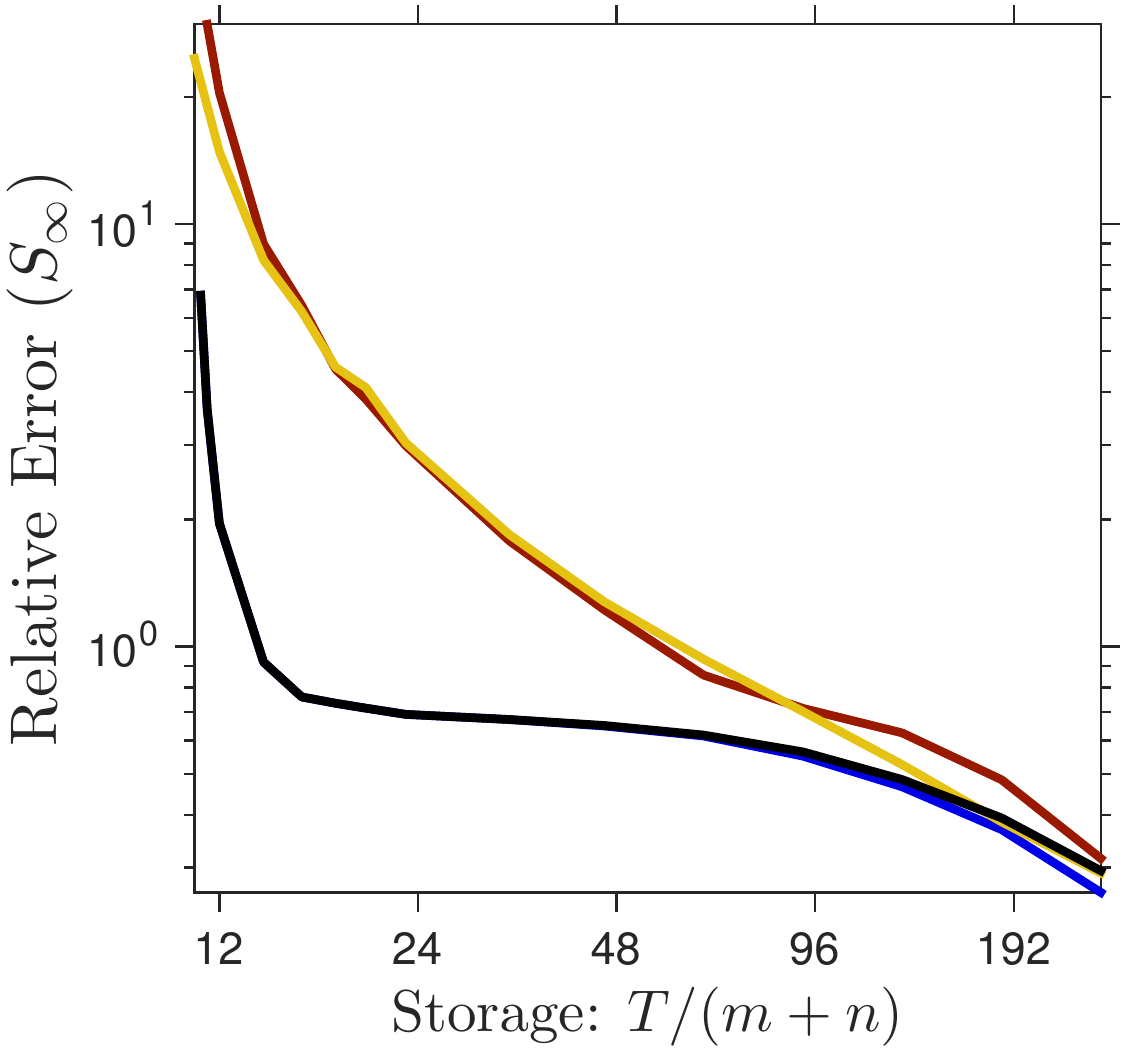}
\caption{\texttt{LowRankHiNoise}}
\end{center}
\end{subfigure}
\begin{subfigure}{.325\textwidth}
\begin{center}
\includegraphics[height=1.5in]{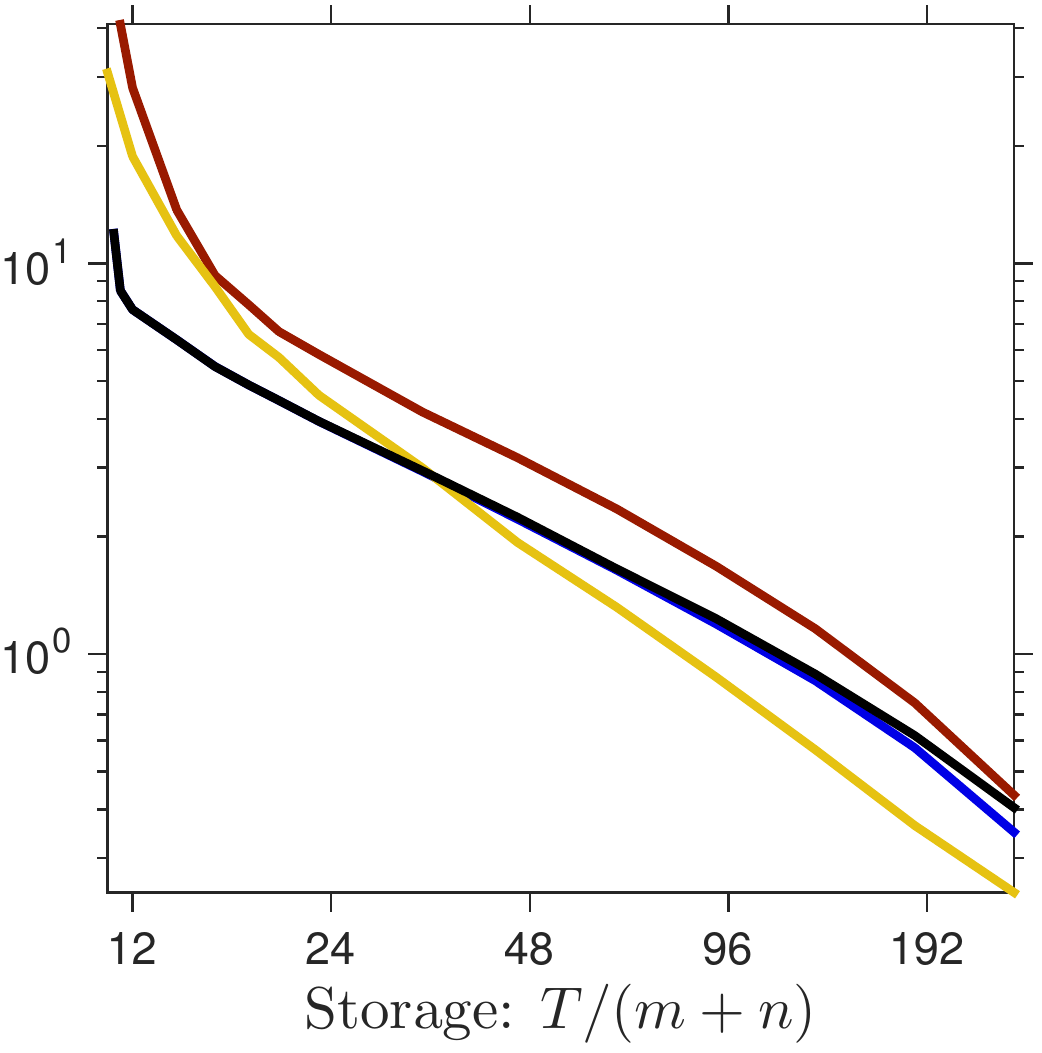}
\caption{\texttt{LowRankMedNoise}}
\end{center}
\end{subfigure}
\begin{subfigure}{.325\textwidth}
\begin{center}
\includegraphics[height=1.5in]{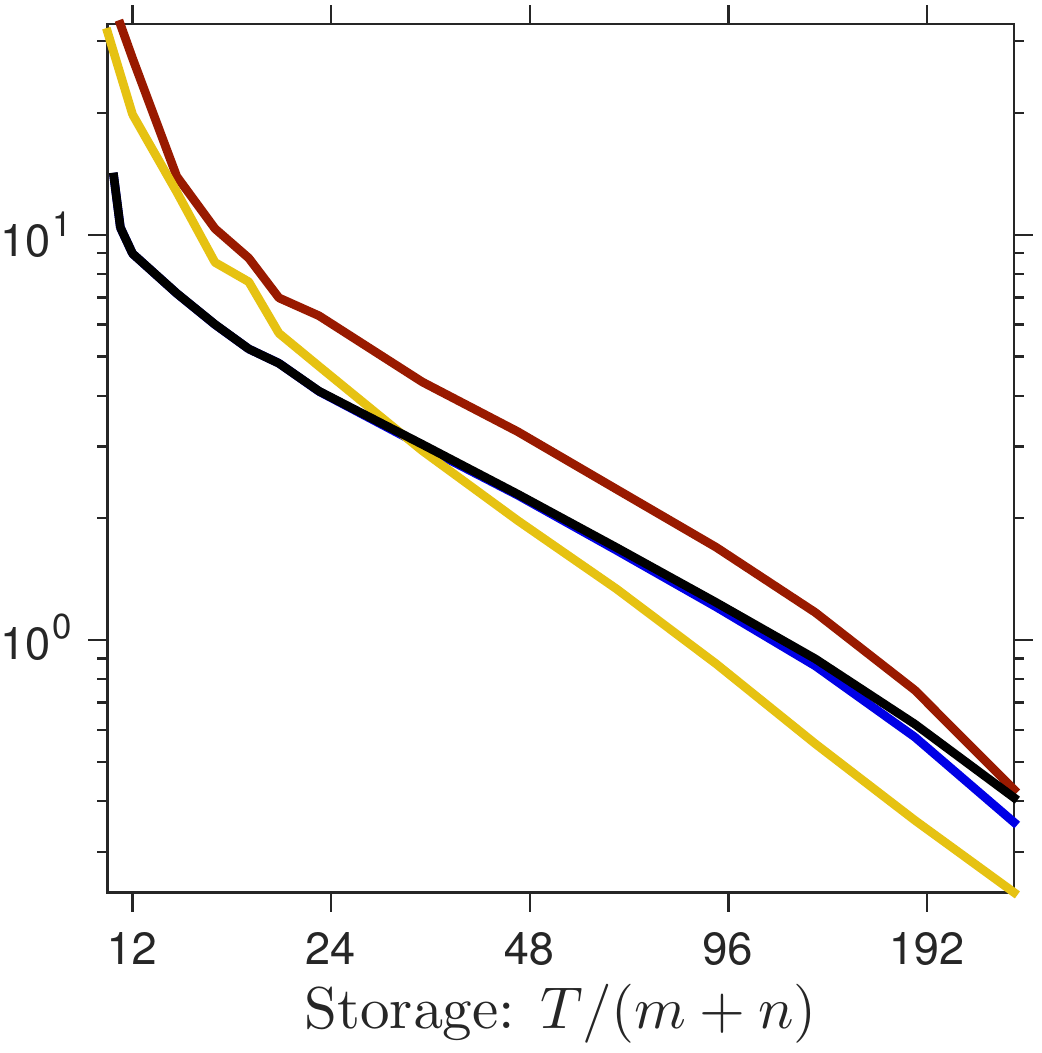}
\caption{\texttt{LowRankLowNoise}}
\end{center}
\end{subfigure}
\end{center}

\vspace{.5em}

\begin{center}
\begin{subfigure}{.325\textwidth}
\begin{center}
\includegraphics[height=1.5in]{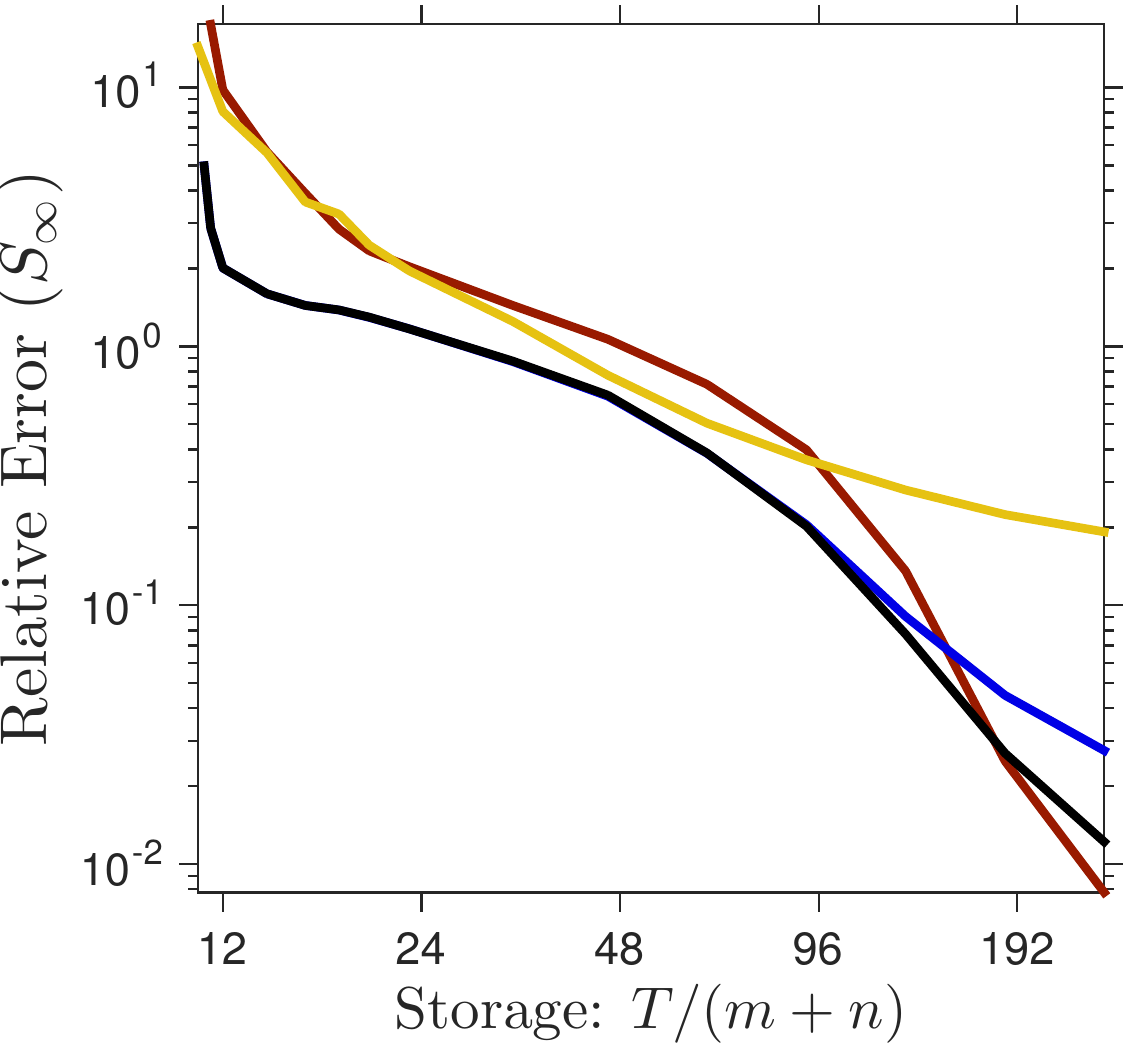}
\caption{\texttt{PolyDecaySlow}}
\end{center}
\end{subfigure}
\begin{subfigure}{.325\textwidth}
\begin{center}
\includegraphics[height=1.5in]{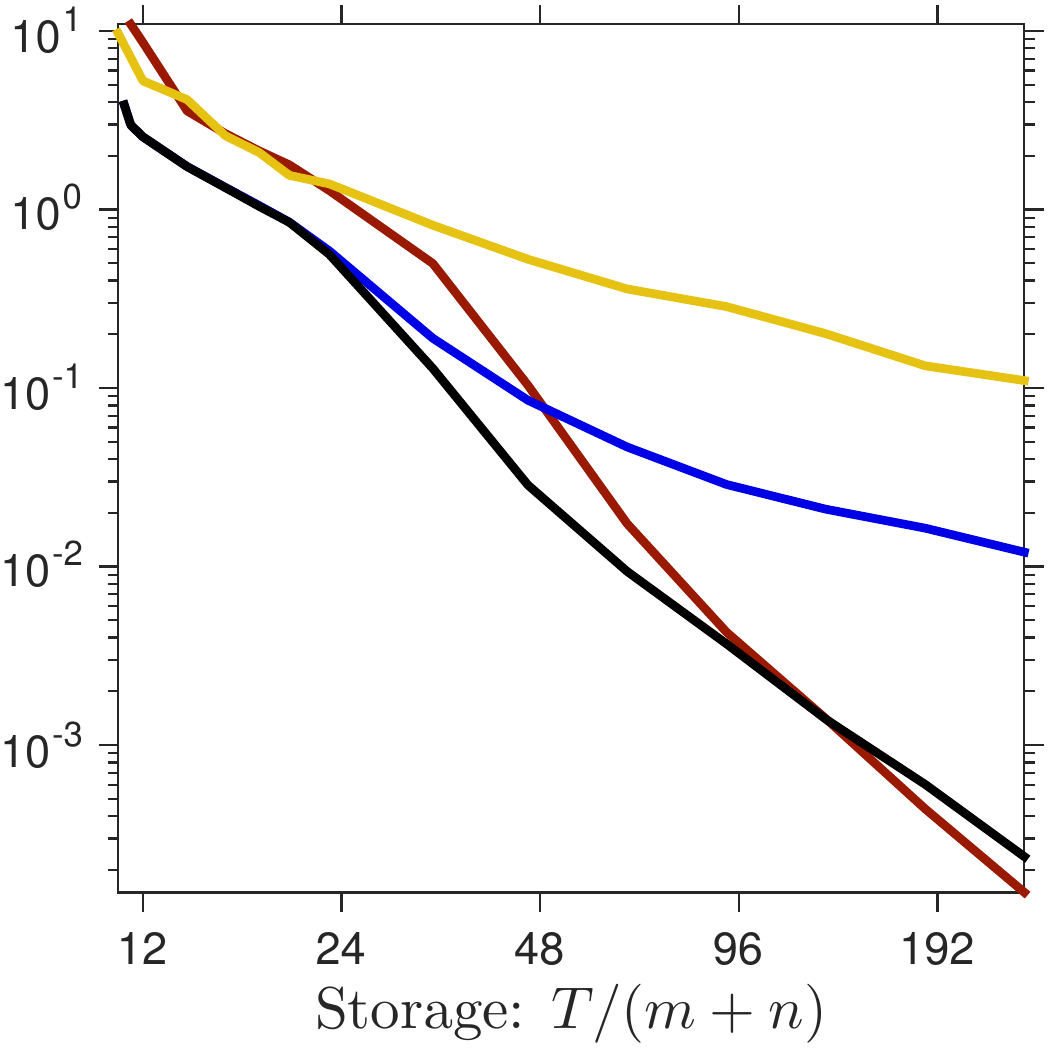}
\caption{\texttt{PolyDecayMed}}
\end{center}
\end{subfigure}
\begin{subfigure}{.325\textwidth}
\begin{center}
\includegraphics[height=1.5in]{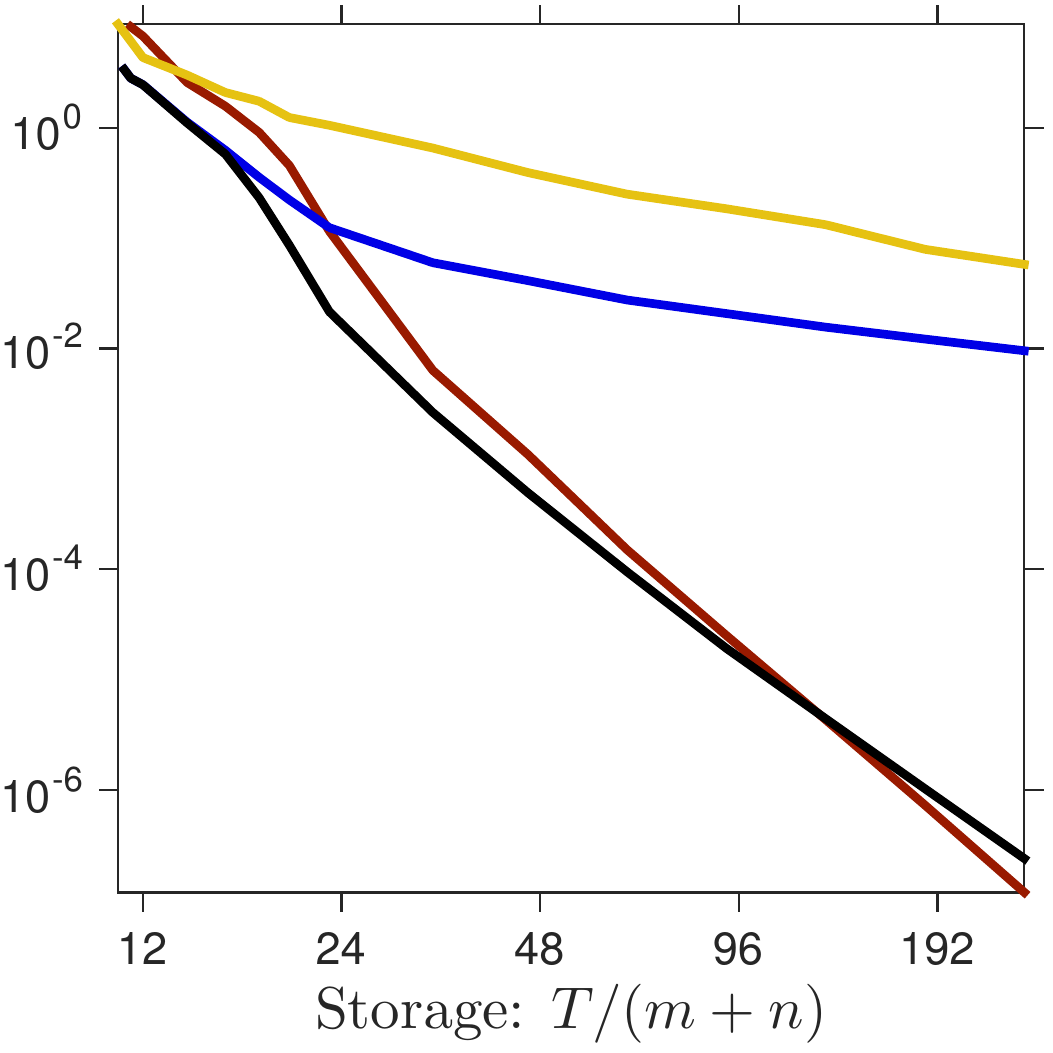}
\caption{\texttt{PolyDecayFast}}
\end{center}
\end{subfigure}
\end{center}

\vspace{0.5em}

\begin{center}
\begin{subfigure}{.325\textwidth}
\begin{center}
\includegraphics[height=1.5in]{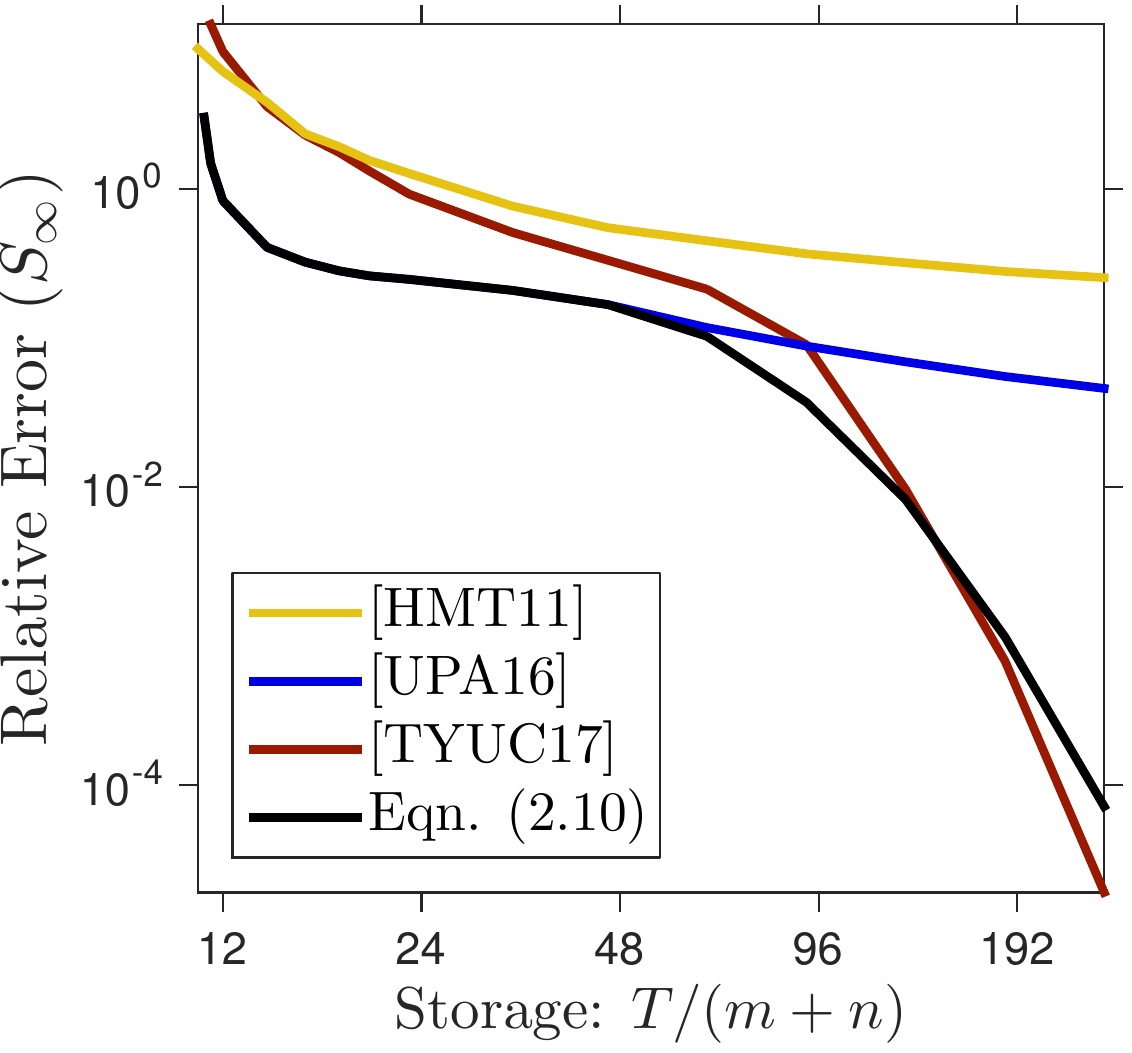}
\caption{\texttt{ExpDecaySlow}}
\end{center}
\end{subfigure}
\begin{subfigure}{.325\textwidth}
\begin{center}
\includegraphics[height=1.5in]{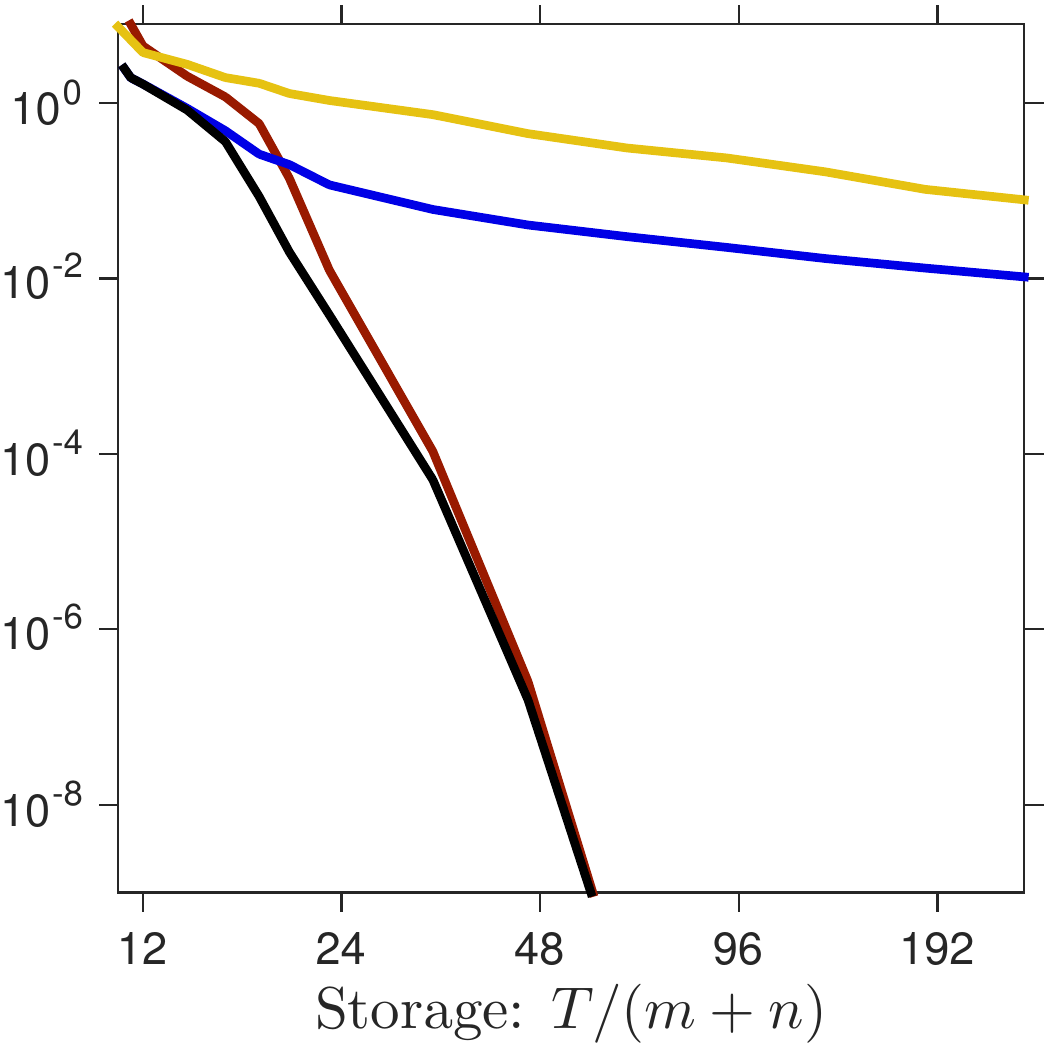}
\caption{\texttt{ExpDecayMed}}
\end{center}
\end{subfigure}
\begin{subfigure}{.325\textwidth}
\begin{center}
\includegraphics[height=1.5in]{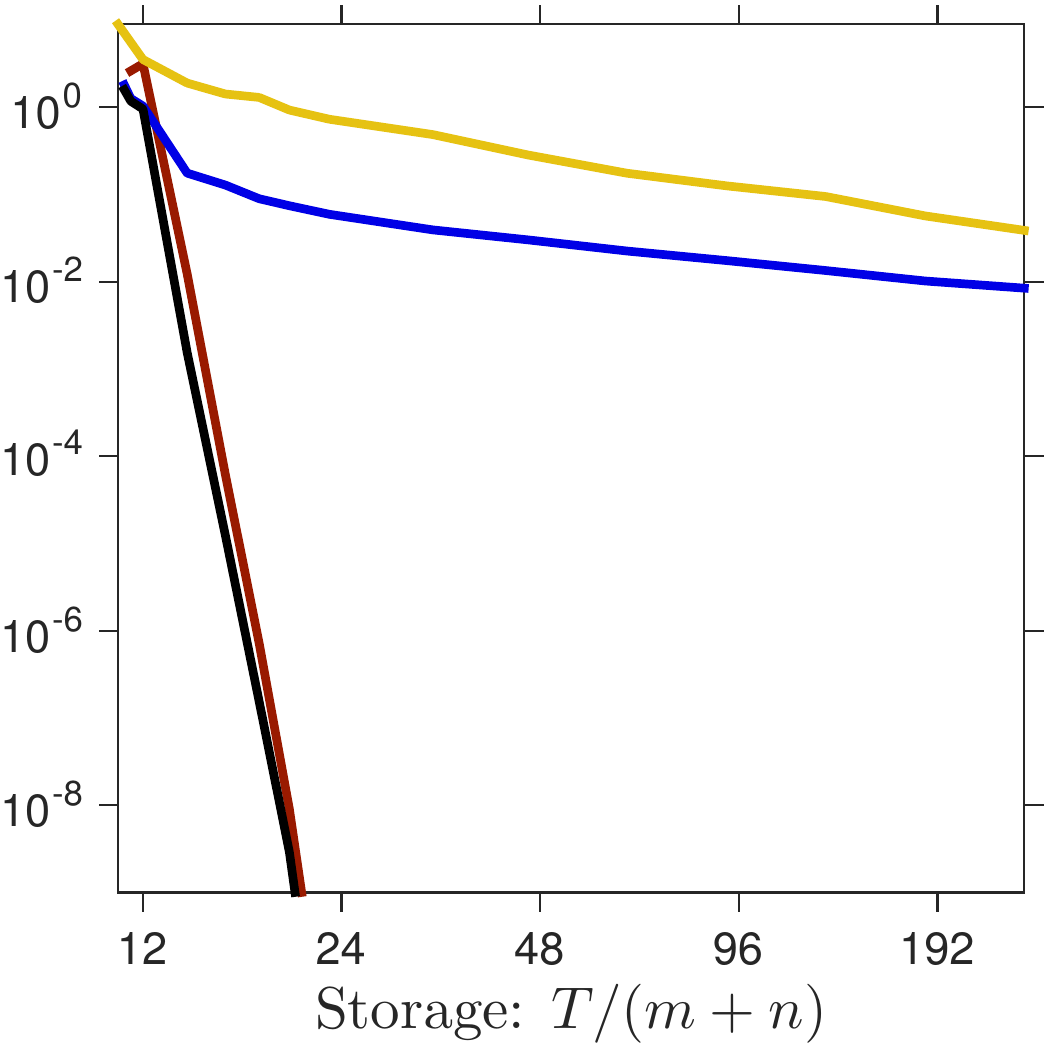}
\caption{\texttt{ExpDecayFast}}
\end{center}
\end{subfigure}
\end{center}

\vspace{0.5em}

\caption{\textbf{Comparison of reconstruction formulas: Synthetic examples.}
(Gaussian maps, effective rank $R = 5$, approximation rank $r = 10$, Schatten $\infty$-norm.)
We compare the oracle error achieved by the proposed fixed-rank
approximation~\cref{eqn:Ahat-fixed} against methods~\cref{eqn:upa,eqn:tyuc2017} from the literature.
See \cref{sec:oracle-error} for details.}
\label{fig:oracle-comparison-R5-Sinf}
\end{figure}

\begin{figure}[htp!]
\begin{center}
\begin{subfigure}{.325\textwidth}
\begin{center}
\includegraphics[height=1.5in]{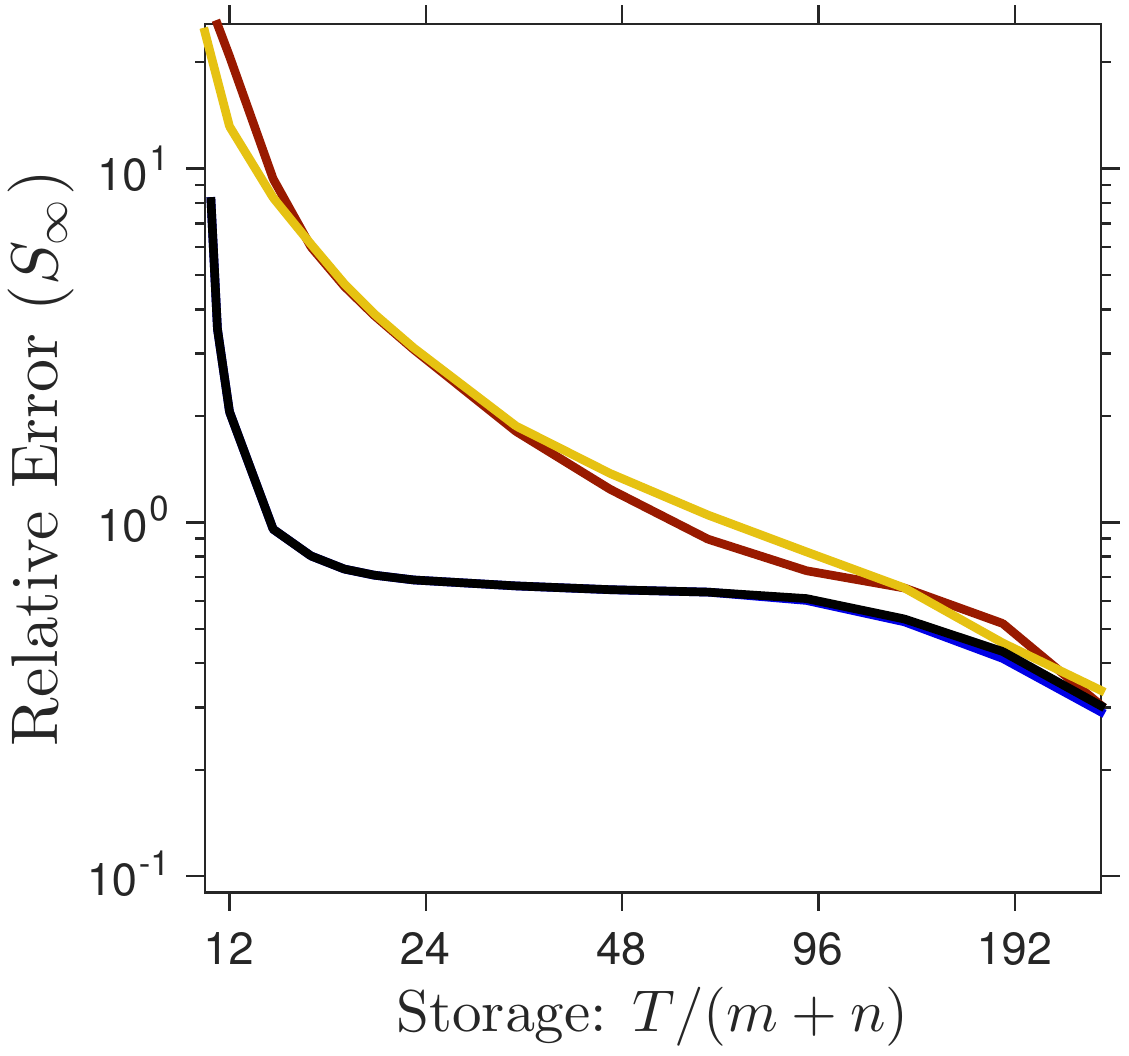}
\caption{\texttt{LowRankHiNoise}}
\end{center}
\end{subfigure}
\begin{subfigure}{.325\textwidth}
\begin{center}
\includegraphics[height=1.5in]{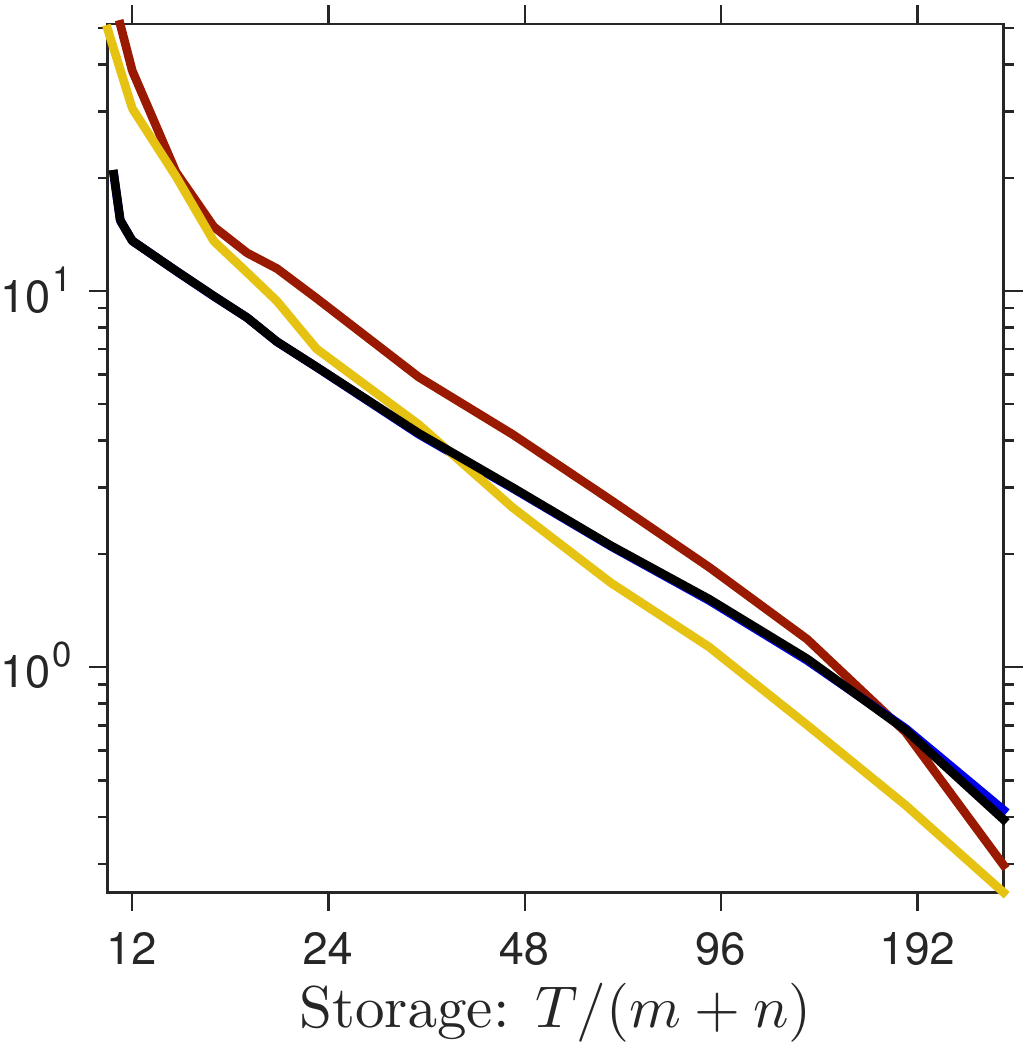}
\caption{\texttt{LowRankMedNoise}}
\end{center}
\end{subfigure}
\begin{subfigure}{.325\textwidth}
\begin{center}
\includegraphics[height=1.5in]{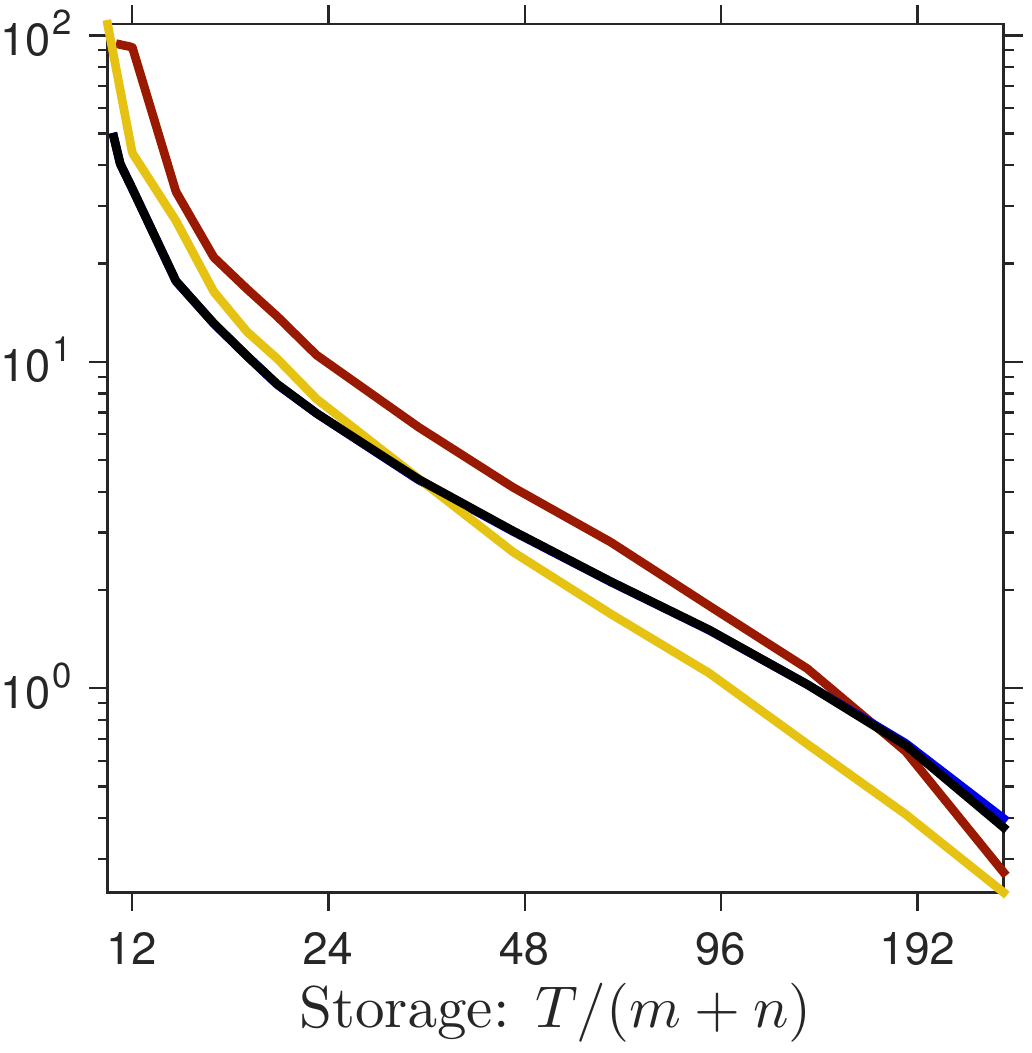}
\caption{\texttt{LowRankLowNoise}}
\end{center}
\end{subfigure}
\end{center}

\vspace{.5em}

\begin{center}
\begin{subfigure}{.325\textwidth}
\begin{center}
\includegraphics[height=1.5in]{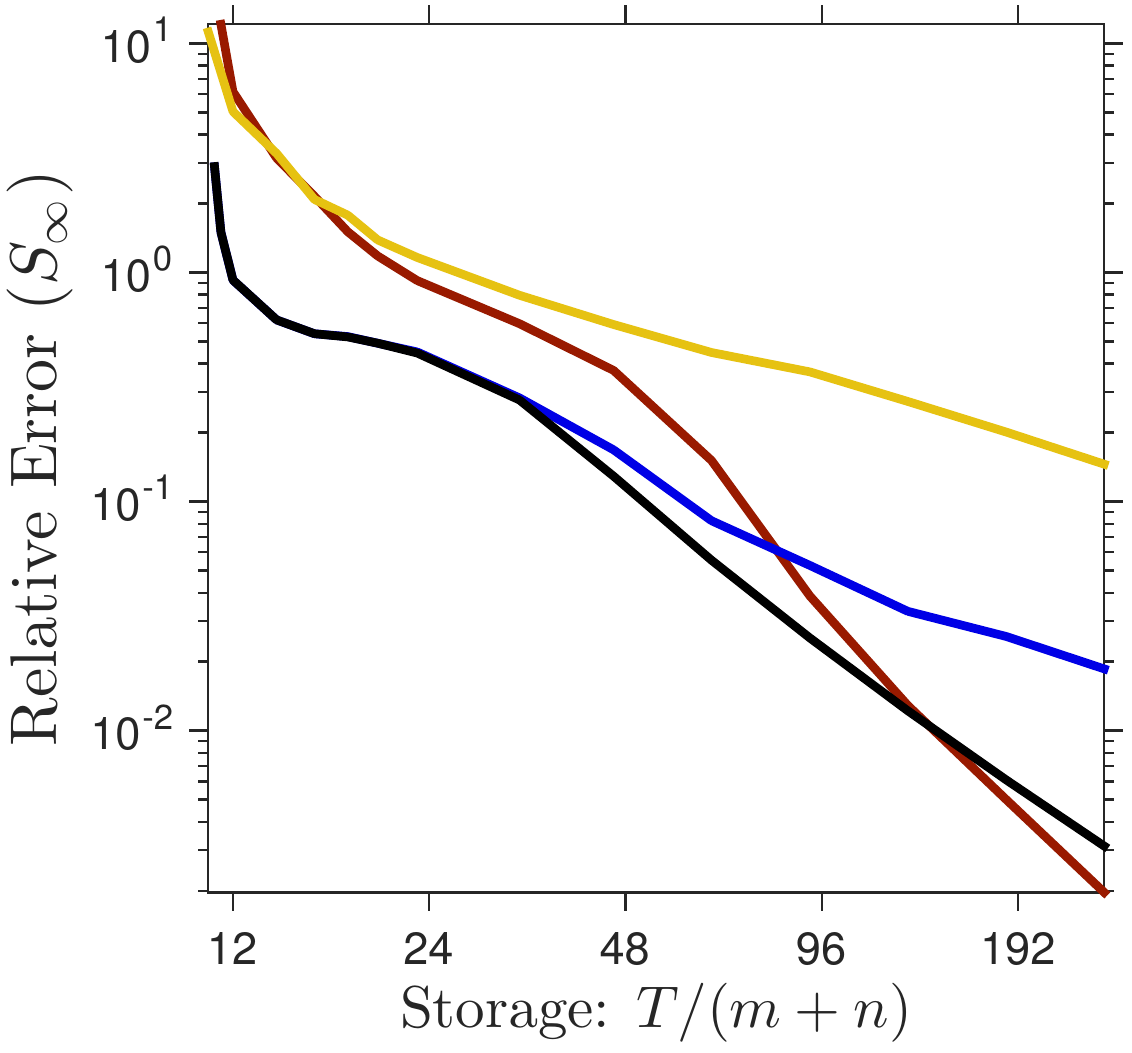}
\caption{\texttt{PolyDecaySlow}}
\end{center}
\end{subfigure}
\begin{subfigure}{.325\textwidth}
\begin{center}
\includegraphics[height=1.5in]{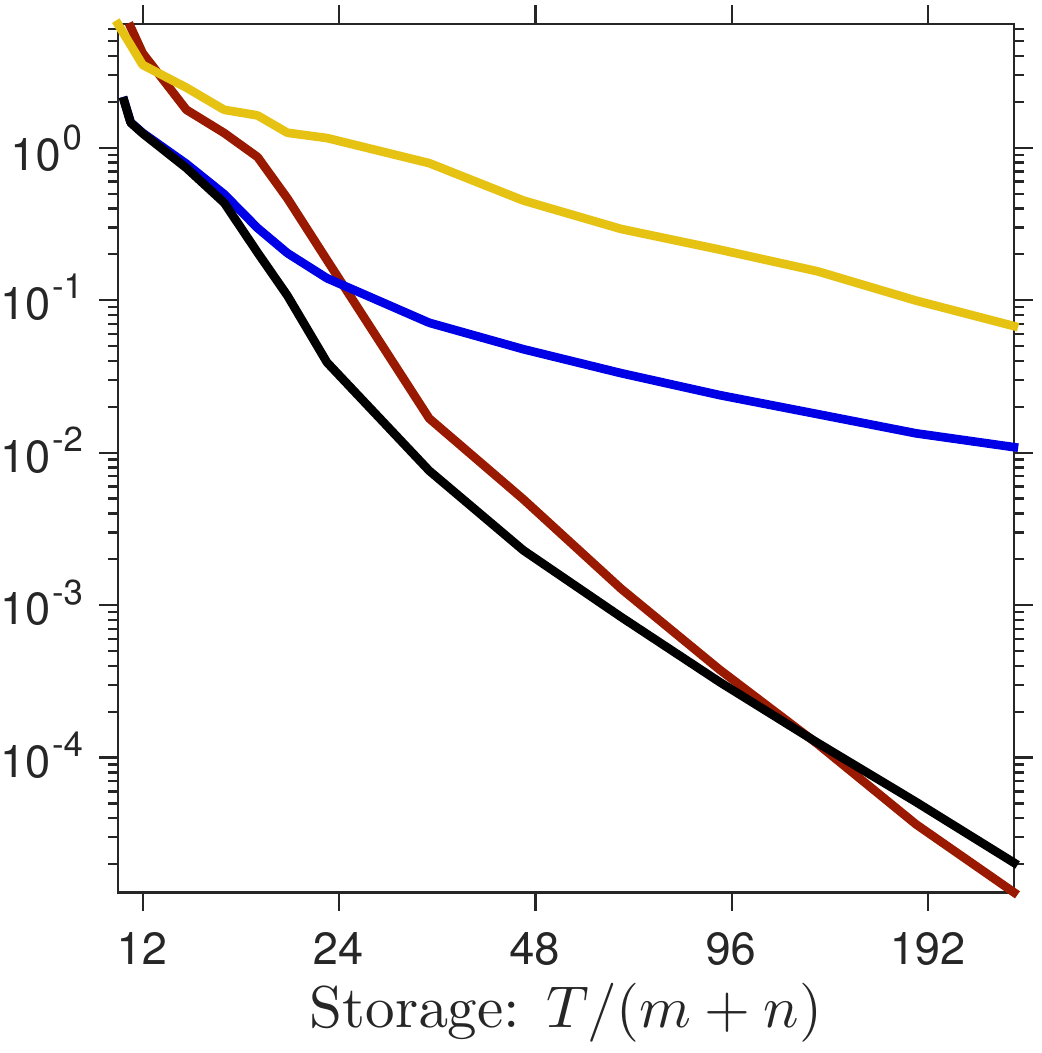}
\caption{\texttt{PolyDecayMed}}
\end{center}
\end{subfigure}
\begin{subfigure}{.325\textwidth}
\begin{center}
\includegraphics[height=1.5in]{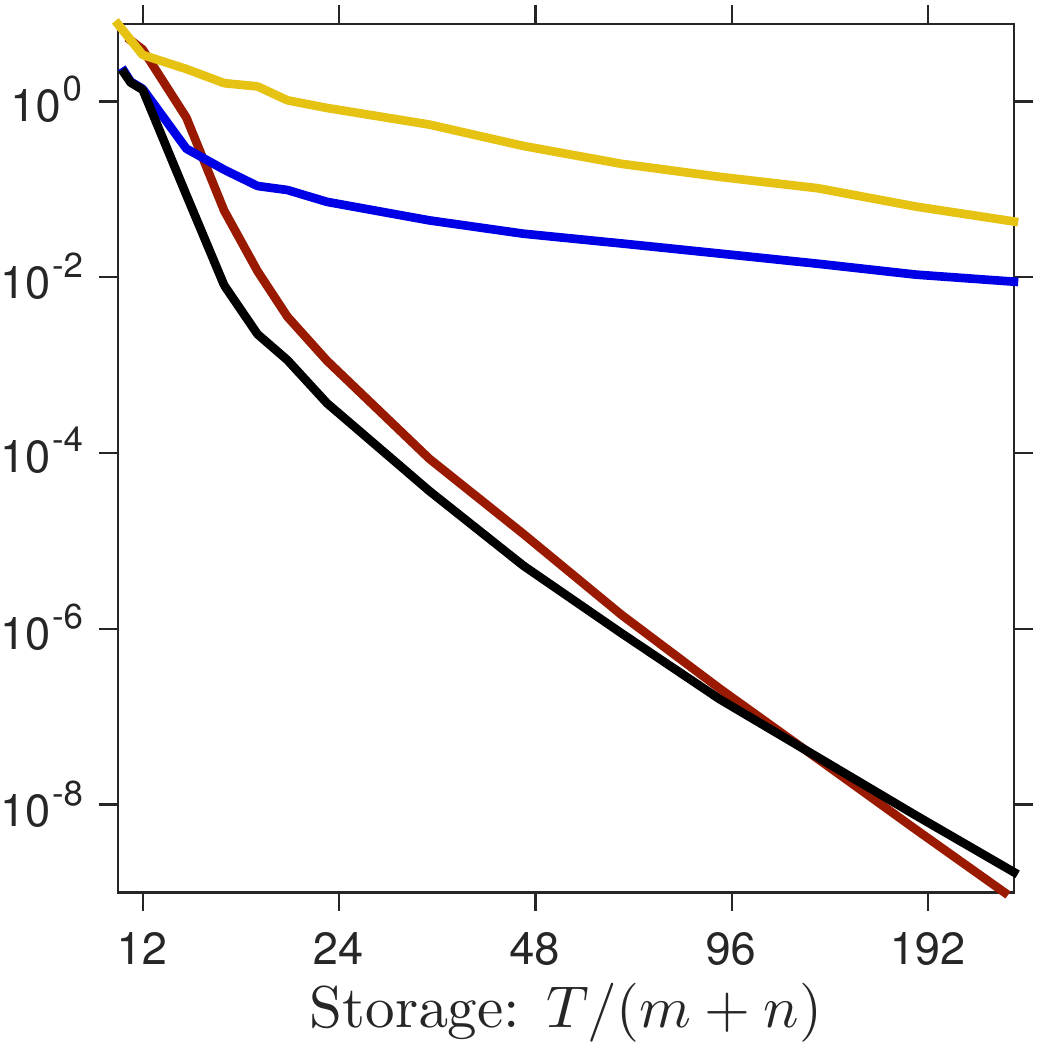}
\caption{\texttt{PolyDecayFast}}
\end{center}
\end{subfigure}
\end{center}

\vspace{0.5em}

\begin{center}
\begin{subfigure}{.325\textwidth}
\begin{center}
\includegraphics[height=1.5in]{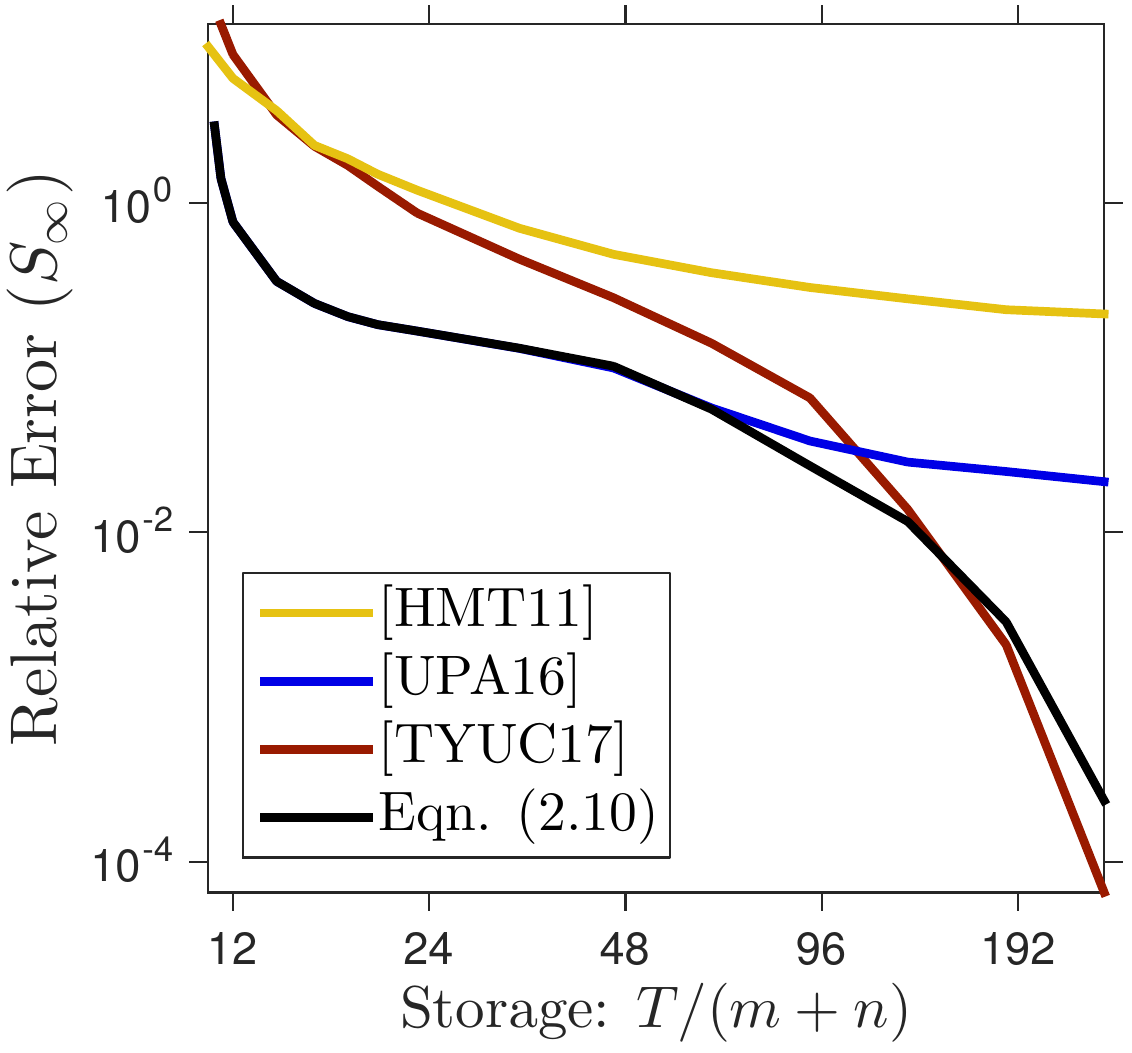}
\caption{\texttt{ExpDecaySlow}}
\end{center}
\end{subfigure}
\begin{subfigure}{.325\textwidth}
\begin{center}
\includegraphics[height=1.5in]{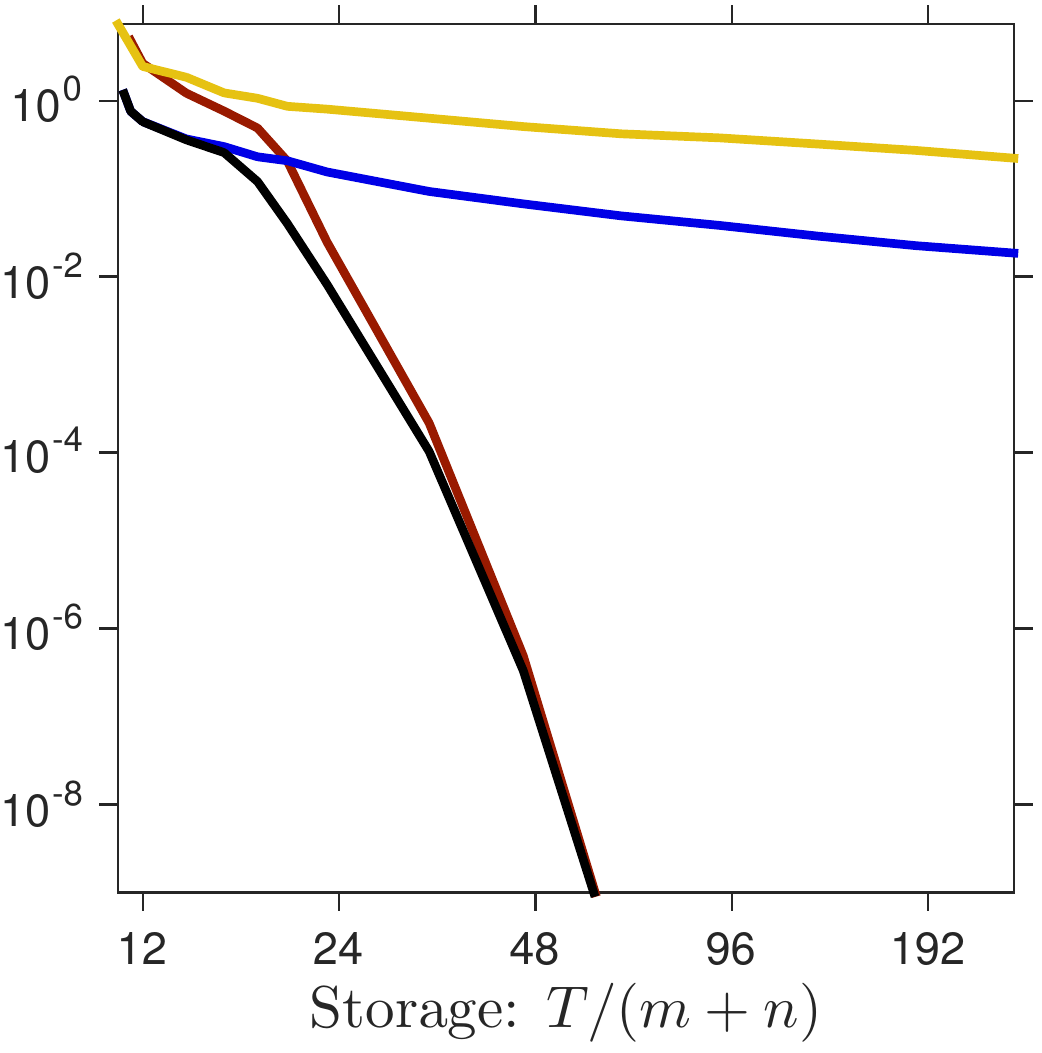}
\caption{\texttt{ExpDecayMed}}
\end{center}
\end{subfigure}
\begin{subfigure}{.325\textwidth}
\begin{center}
\includegraphics[height=1.5in]{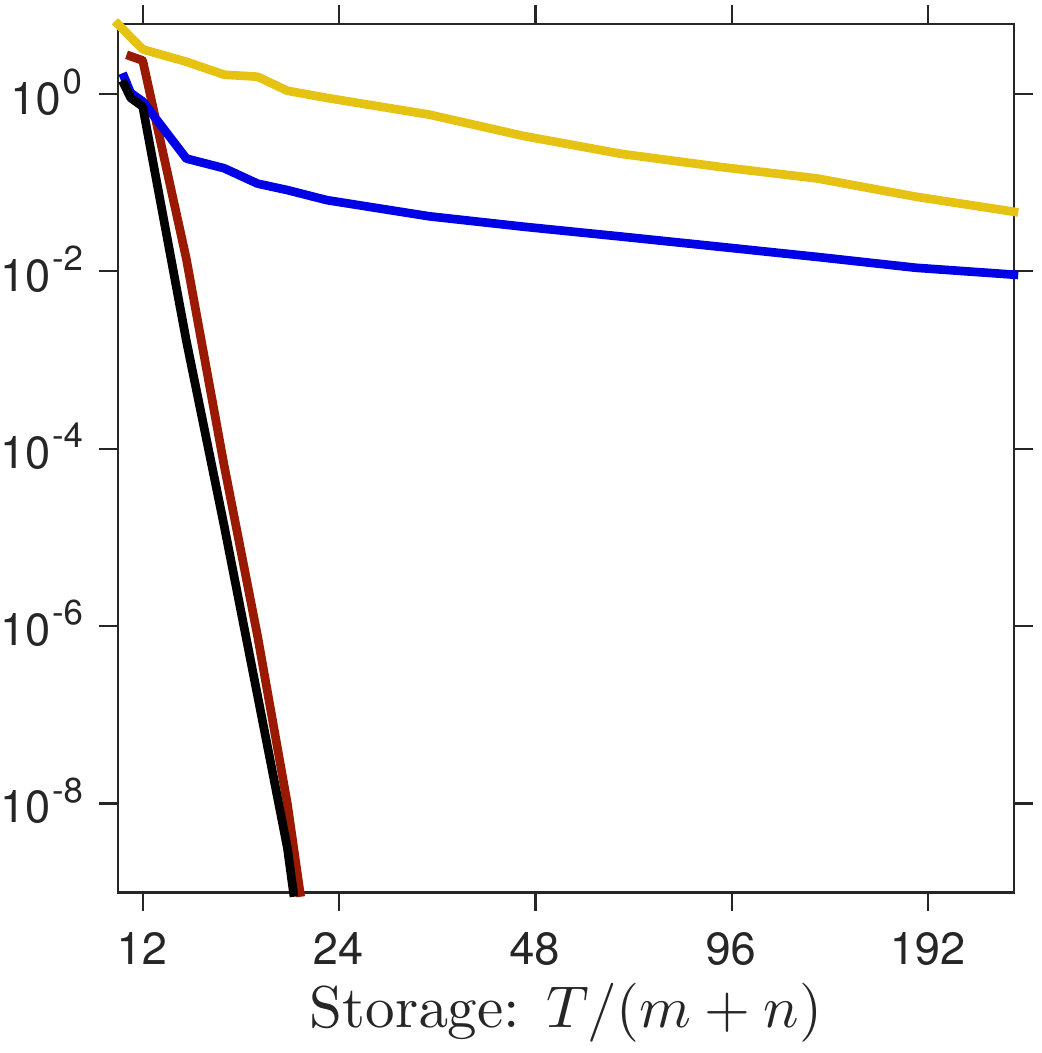}
\caption{\texttt{ExpDecayFast}}
\end{center}
\end{subfigure}
\end{center}

\vspace{0.5em}

\caption{\textbf{Comparison of reconstruction formulas: Synthetic examples.}
(Gaussian maps, effective rank $R = 10$, approximation rank $r = 10$, Schatten $\infty$-norm.)
We compare the oracle error achieved by the proposed fixed-rank
approximation~\cref{eqn:Ahat-fixed} against methods~\cref{eqn:upa,eqn:tyuc2017} from the literature.
See \cref{sec:oracle-error} for details.}
\label{fig:oracle-comparison-R10-Sinf}
\end{figure}

\begin{figure}[htp!]
\begin{center}
\begin{subfigure}{.325\textwidth}
\begin{center}
\includegraphics[height=1.5in]{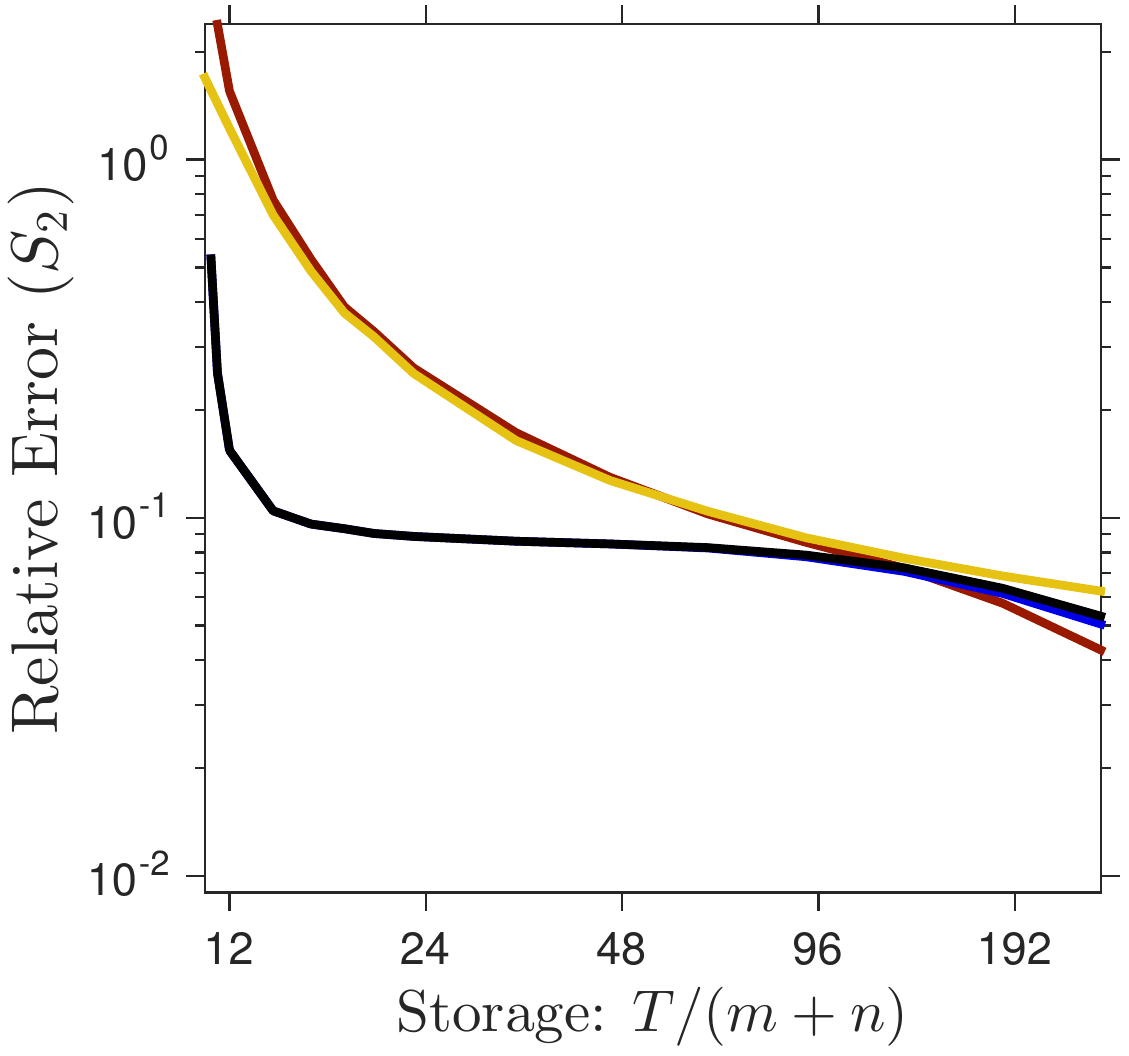}
\caption{\texttt{LowRankHiNoise}}
\end{center}
\end{subfigure}
\begin{subfigure}{.325\textwidth}
\begin{center}
\includegraphics[height=1.5in]{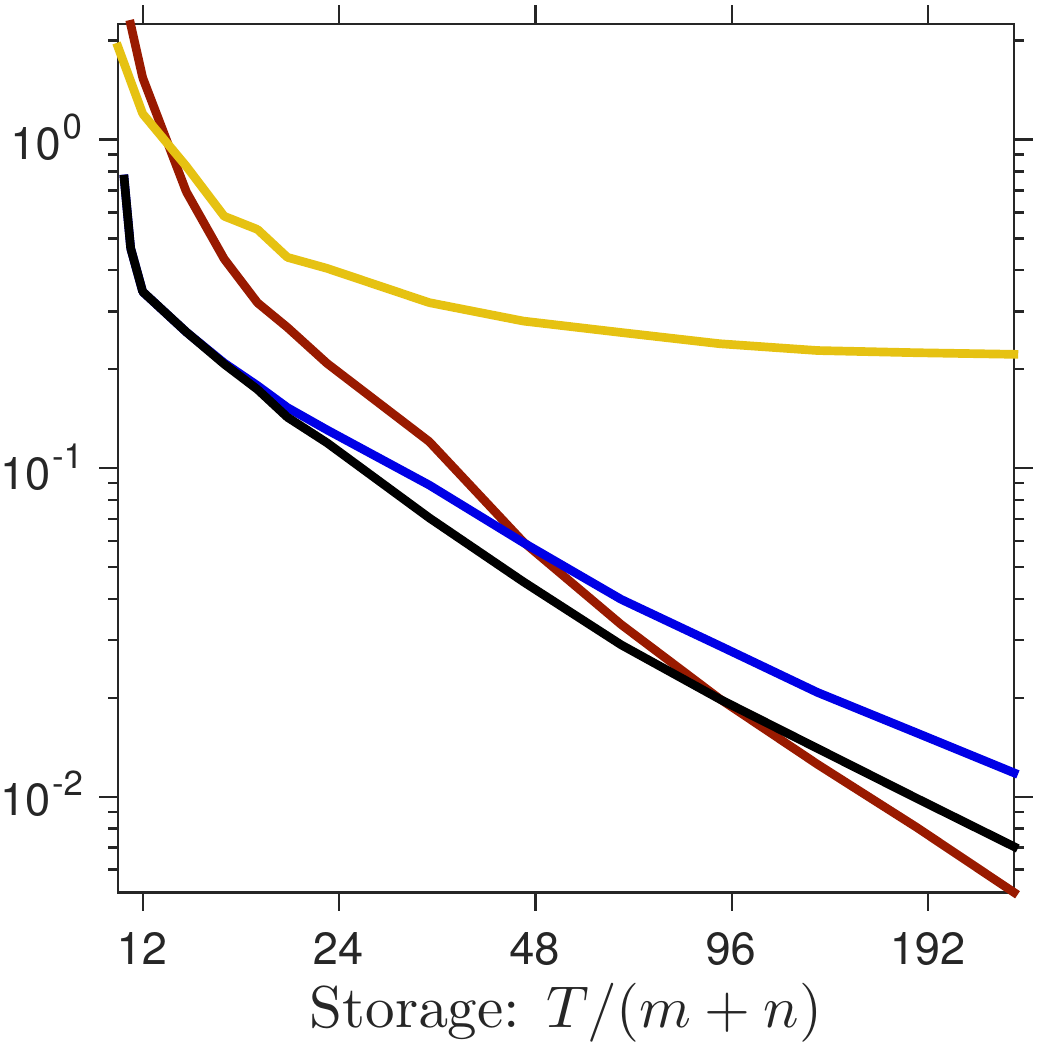}
\caption{\texttt{LowRankMedNoise}}
\end{center}
\end{subfigure}
\begin{subfigure}{.325\textwidth}
\begin{center}
\includegraphics[height=1.5in]{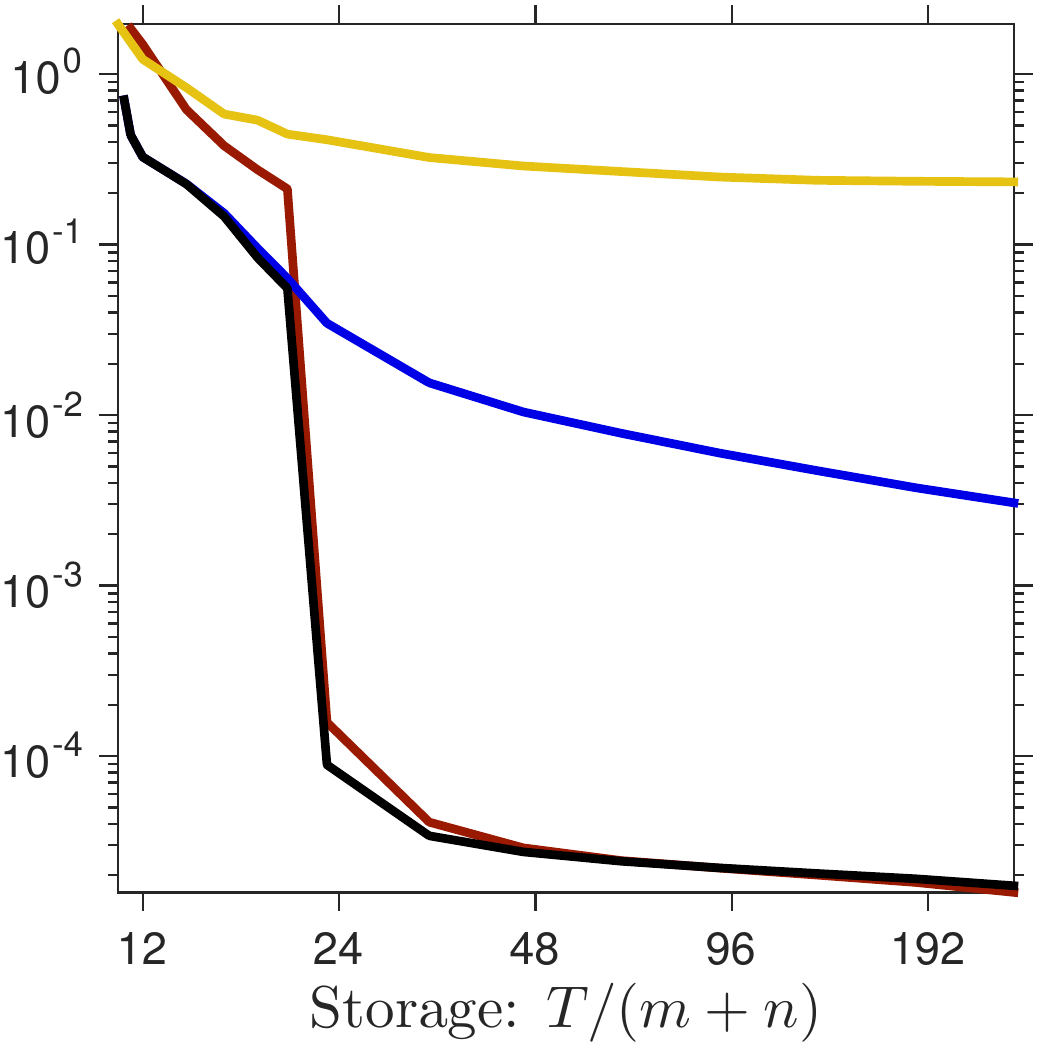}
\caption{\texttt{LowRankLowNoise}}
\end{center}
\end{subfigure}
\end{center}

\vspace{.5em}

\begin{center}
\begin{subfigure}{.325\textwidth}
\begin{center}
\includegraphics[height=1.5in]{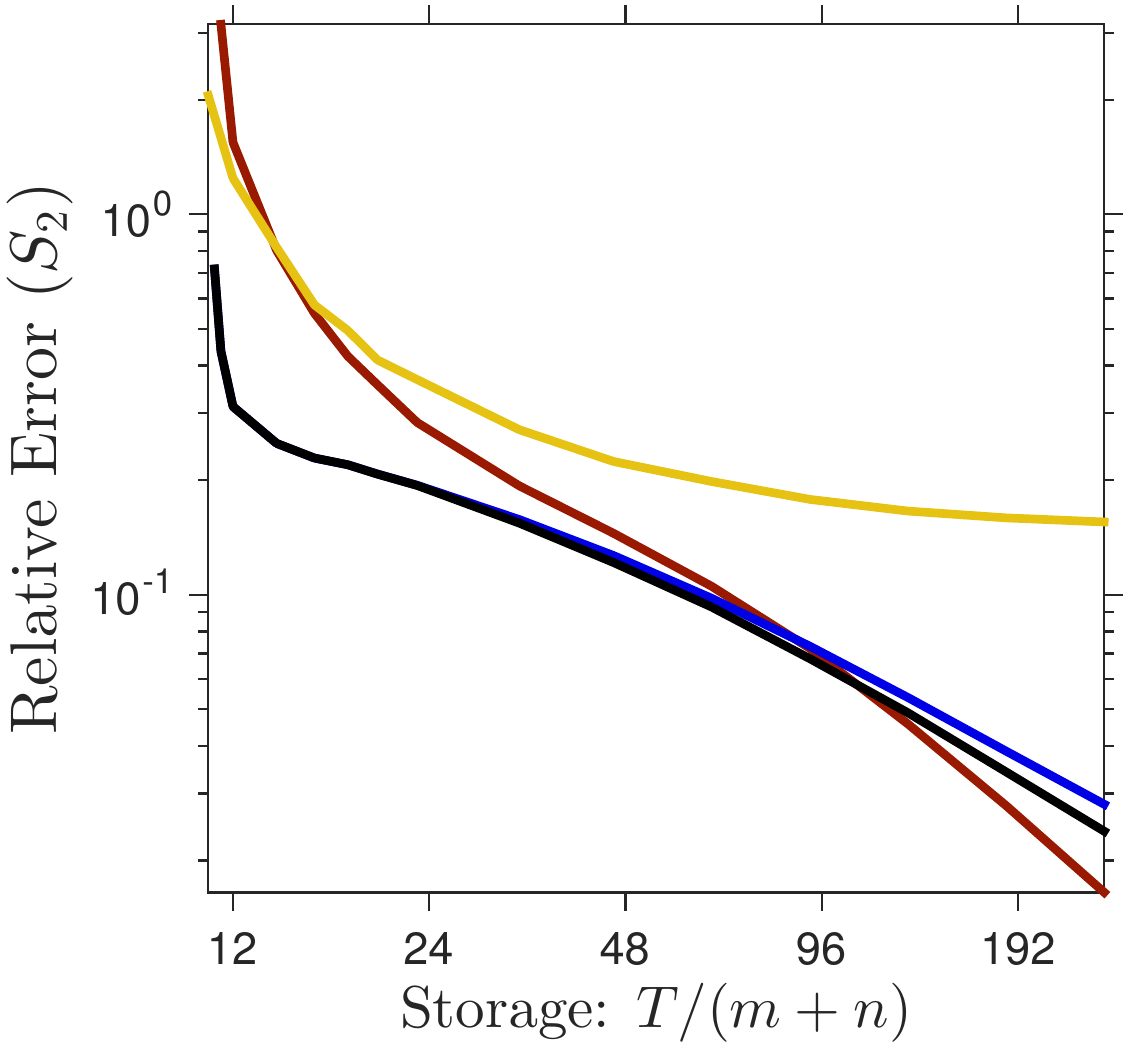}
\caption{\texttt{PolyDecaySlow}}
\end{center}
\end{subfigure}
\begin{subfigure}{.325\textwidth}
\begin{center}
\includegraphics[height=1.5in]{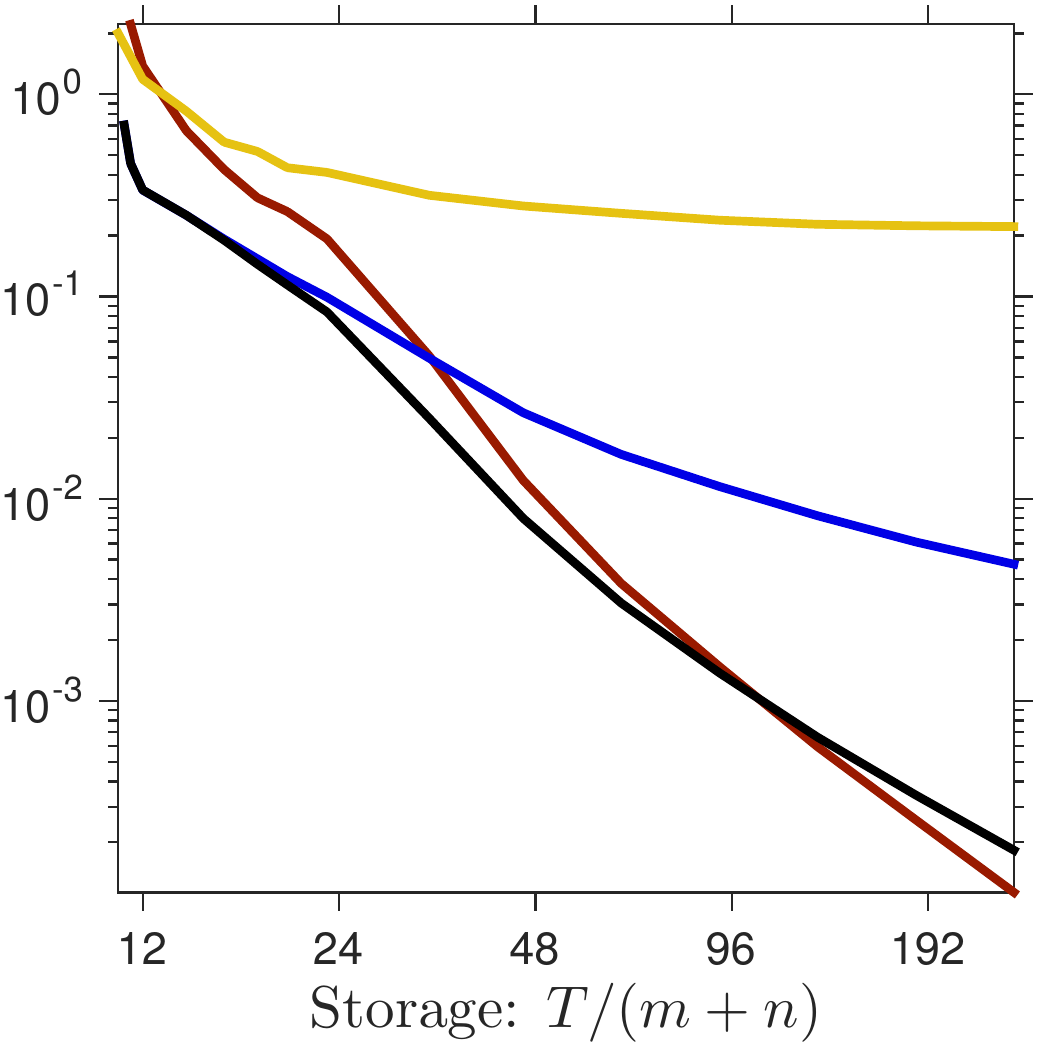}
\caption{\texttt{PolyDecayMed}}
\end{center}
\end{subfigure}
\begin{subfigure}{.325\textwidth}
\begin{center}
\includegraphics[height=1.5in]{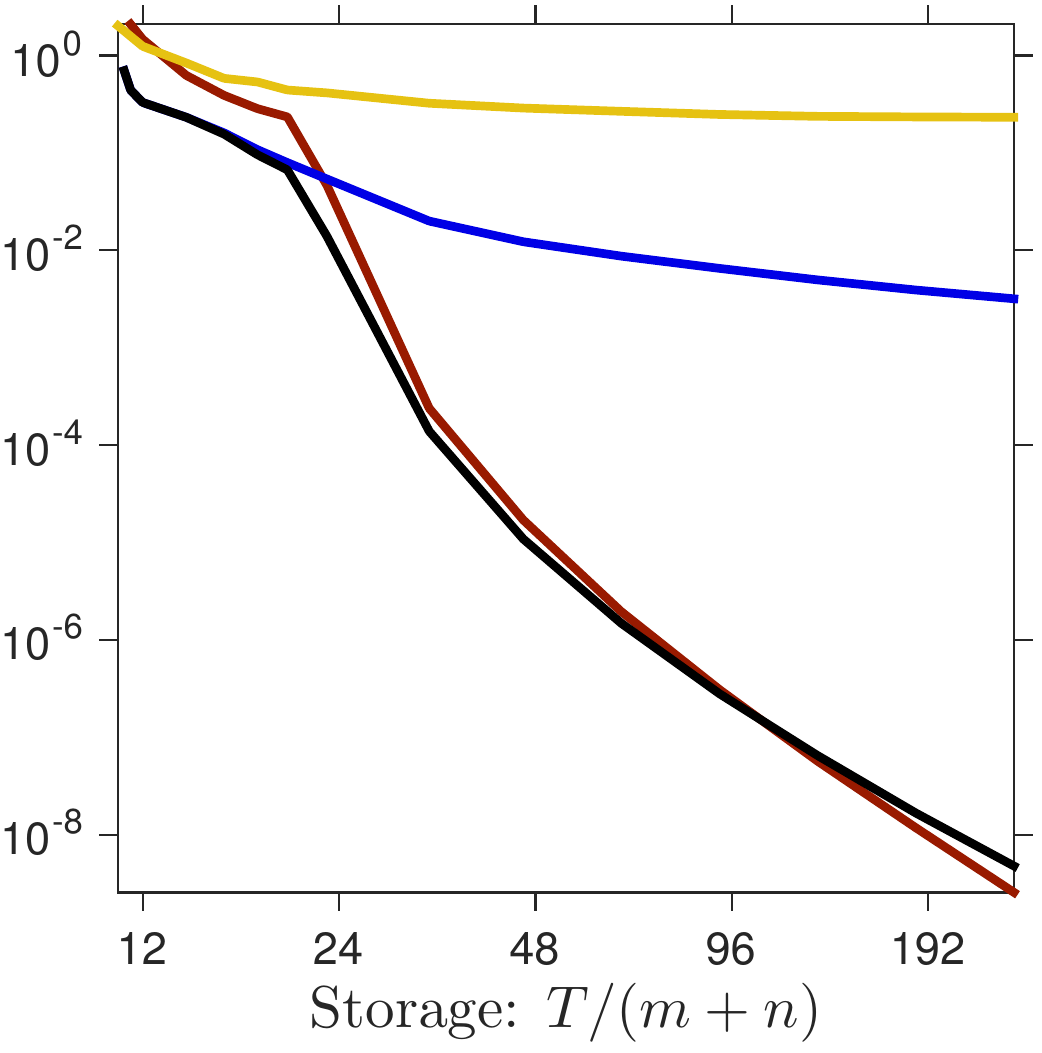}
\caption{\texttt{PolyDecayFast}}
\end{center}
\end{subfigure}
\end{center}

\vspace{0.5em}

\begin{center}
\begin{subfigure}{.325\textwidth}
\begin{center}
\includegraphics[height=1.5in]{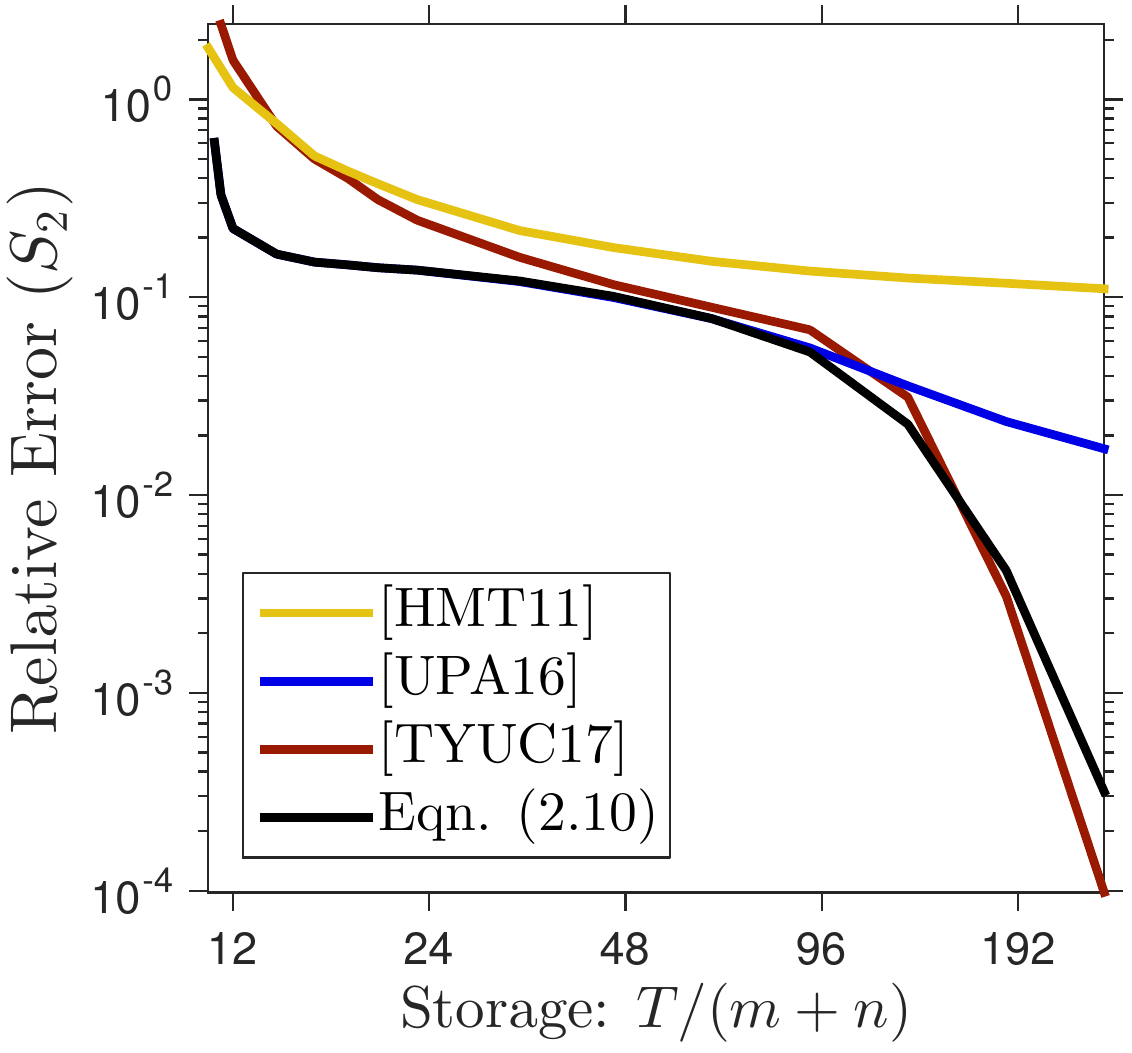}
\caption{\texttt{ExpDecaySlow}}
\end{center}
\end{subfigure}
\begin{subfigure}{.325\textwidth}
\begin{center}
\includegraphics[height=1.5in]{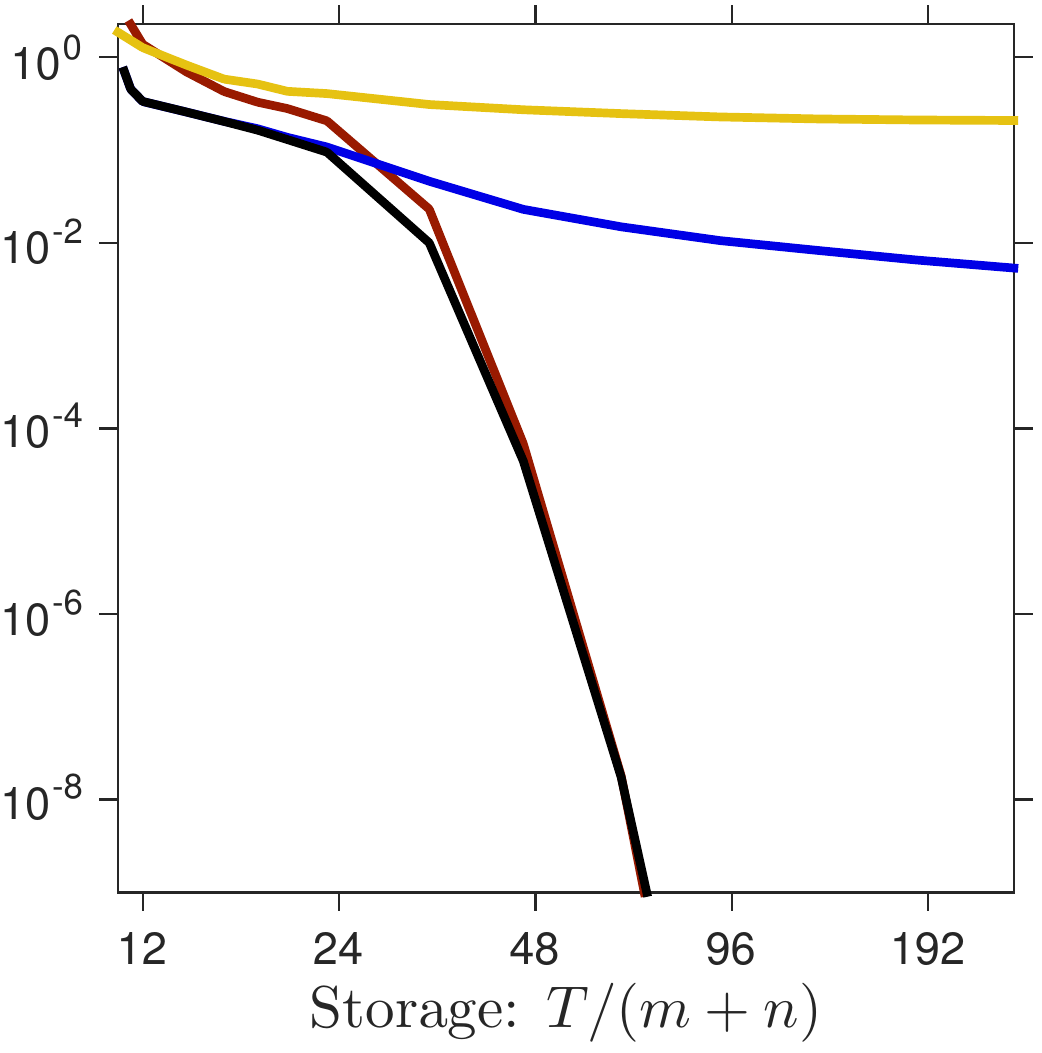}
\caption{\texttt{ExpDecayMed}}
\end{center}
\end{subfigure}
\begin{subfigure}{.325\textwidth}
\begin{center}
\includegraphics[height=1.5in]{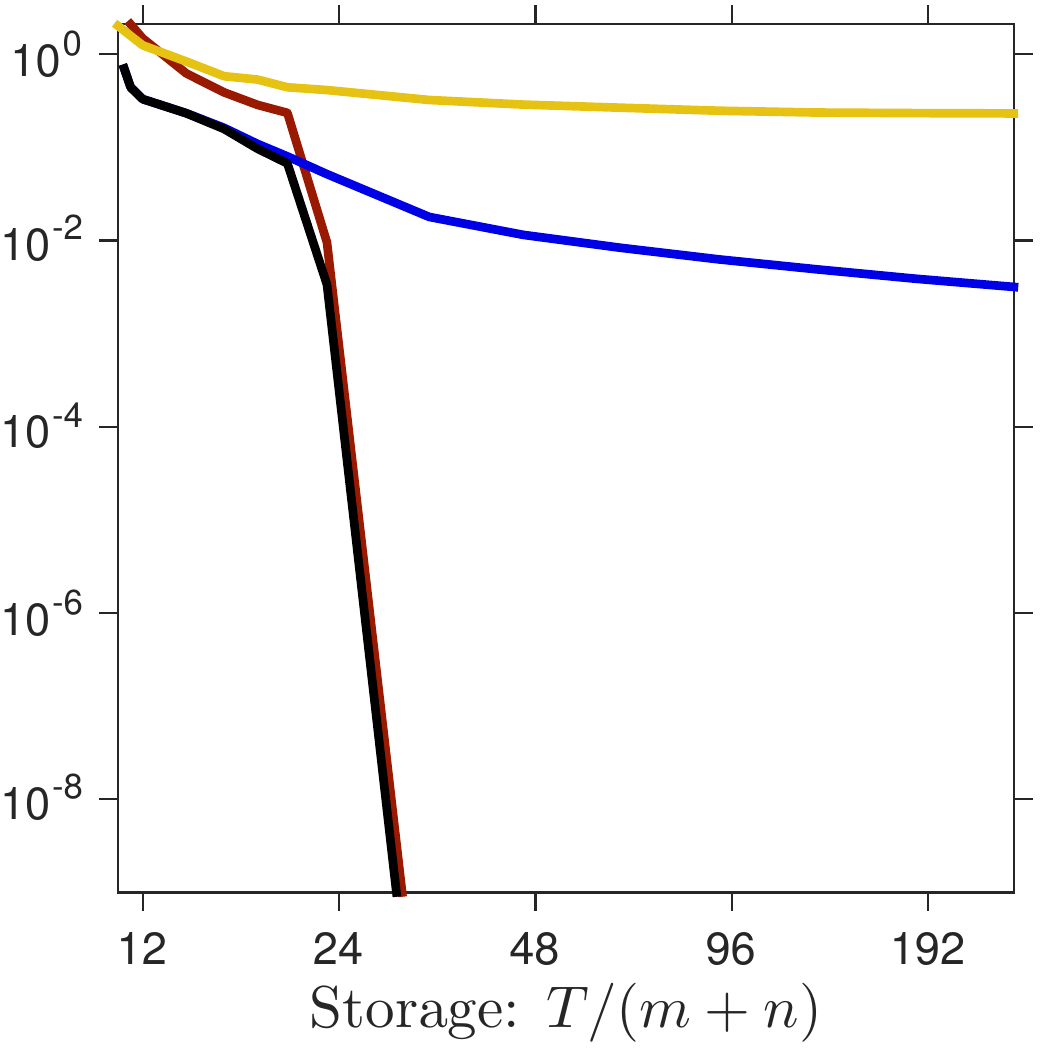}
\caption{\texttt{ExpDecayFast}}
\end{center}
\end{subfigure}
\end{center}

\vspace{0.5em}

\caption{\textbf{Comparison of reconstruction formulas: Synthetic examples.}
(Gaussian maps, effective rank $R = 20$, approximation rank $r = 10$, Schatten 2-norm.)
We compare the oracle error achieved by the proposed fixed-rank
approximation~\cref{eqn:Ahat-fixed} against methods~\cref{eqn:upa,eqn:tyuc2017} from the literature.
See \cref{sec:oracle-error} for details.}
\label{fig:oracle-comparison-R20-S2}
\end{figure}

\begin{figure}[htp!]
\begin{center}
\begin{subfigure}{.325\textwidth}
\begin{center}
\includegraphics[height=1.5in]{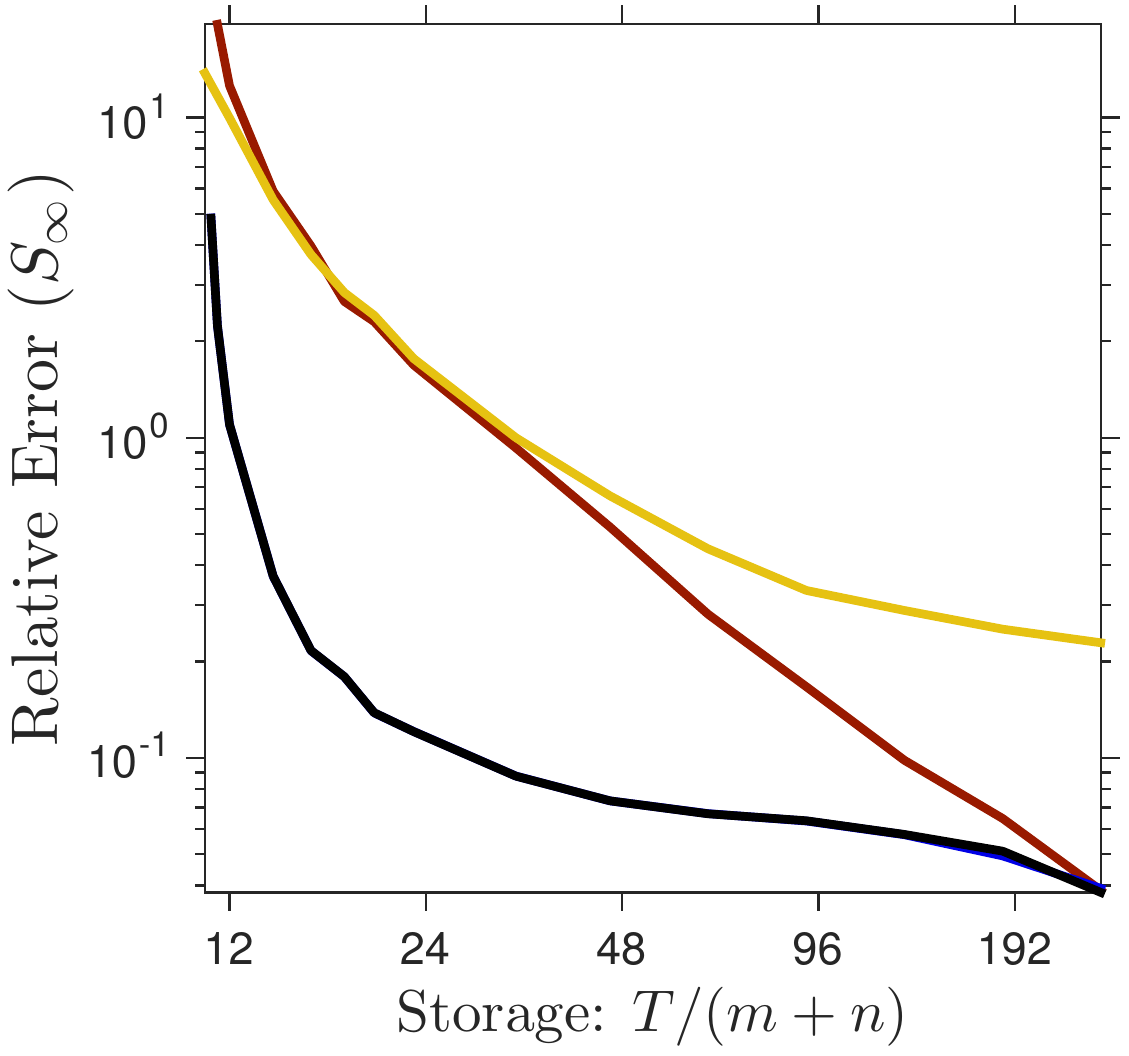}
\caption{\texttt{LowRankHiNoise}}
\end{center}
\end{subfigure}
\begin{subfigure}{.325\textwidth}
\begin{center}
\includegraphics[height=1.5in]{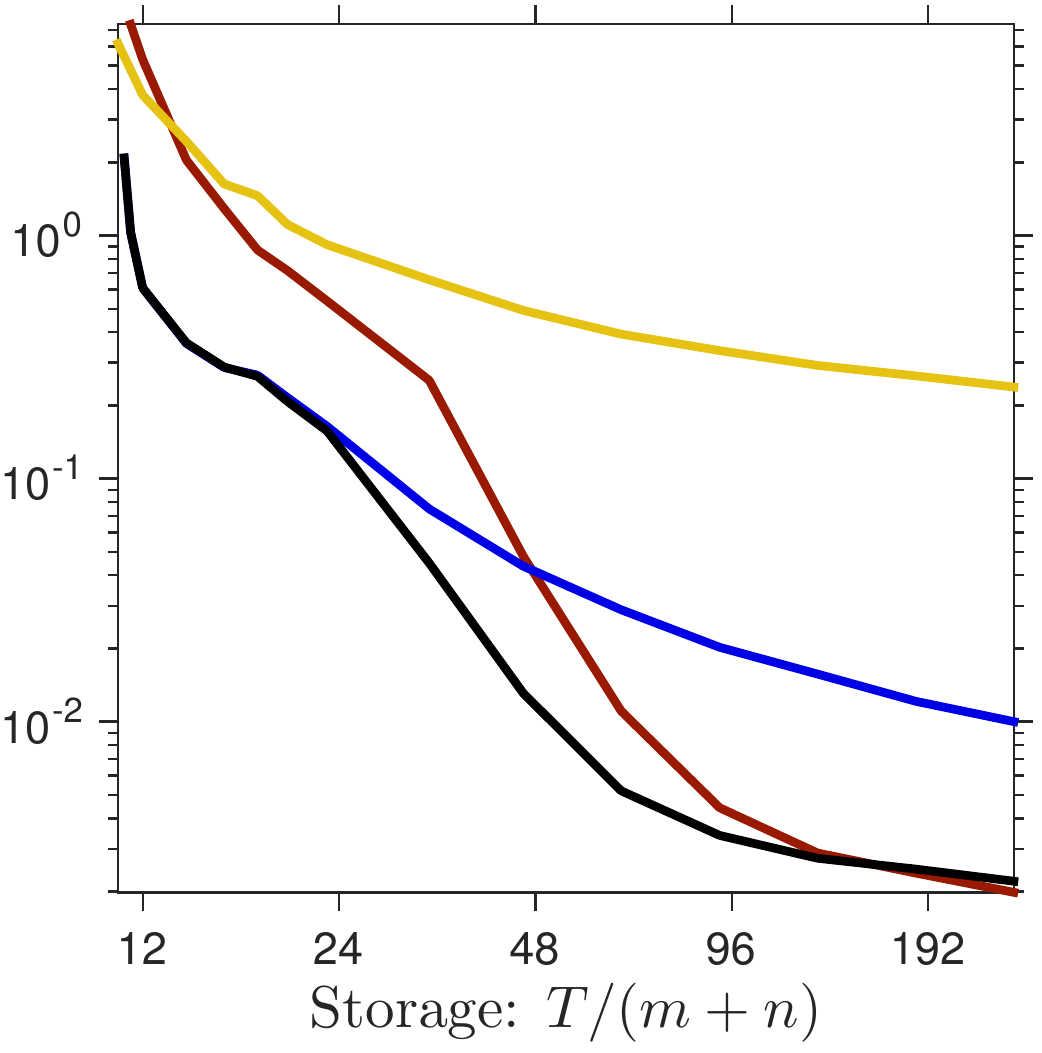}
\caption{\texttt{LowRankMedNoise}}
\end{center}
\end{subfigure}
\begin{subfigure}{.325\textwidth}
\begin{center}
\includegraphics[height=1.5in]{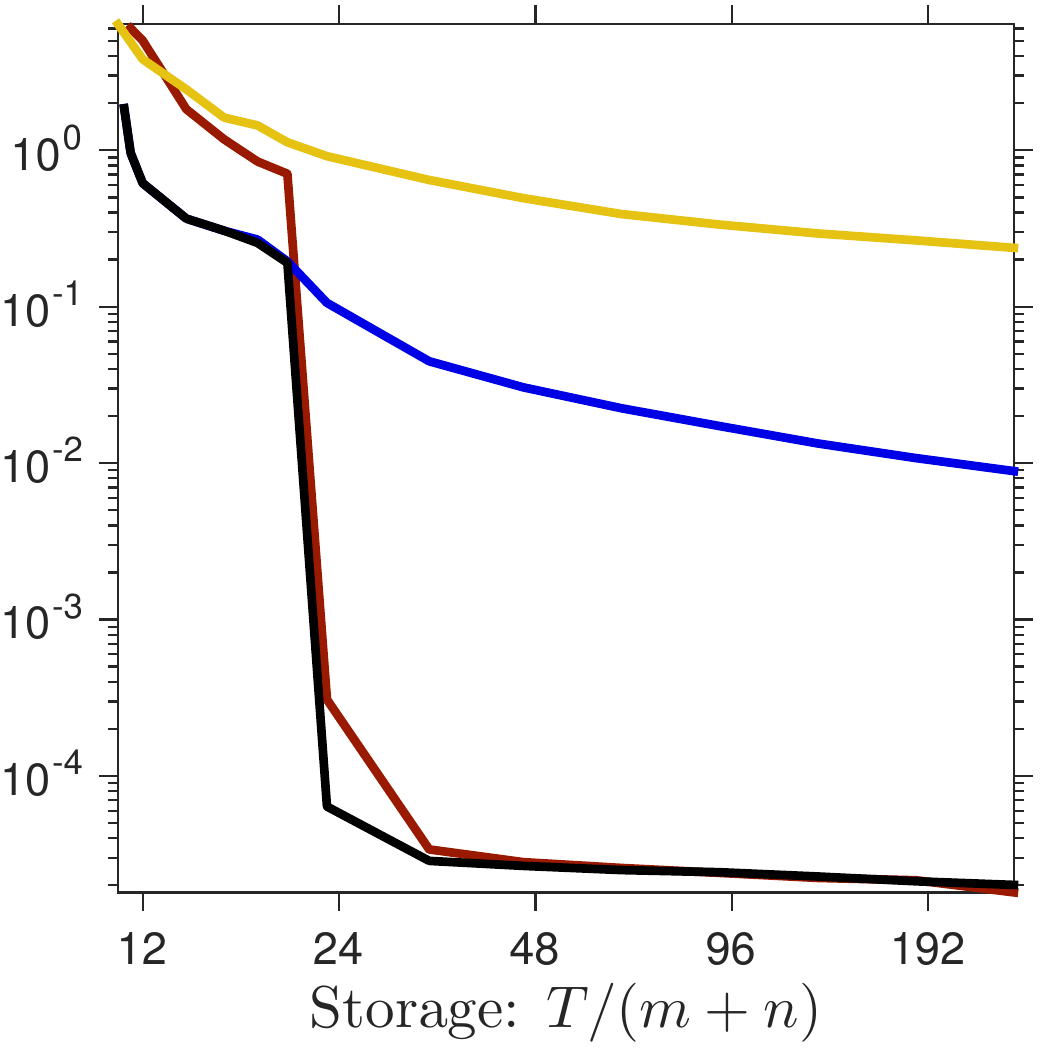}
\caption{\texttt{LowRankLowNoise}}
\end{center}
\end{subfigure}
\end{center}

\vspace{.5em}

\begin{center}
\begin{subfigure}{.325\textwidth}
\begin{center}
\includegraphics[height=1.5in]{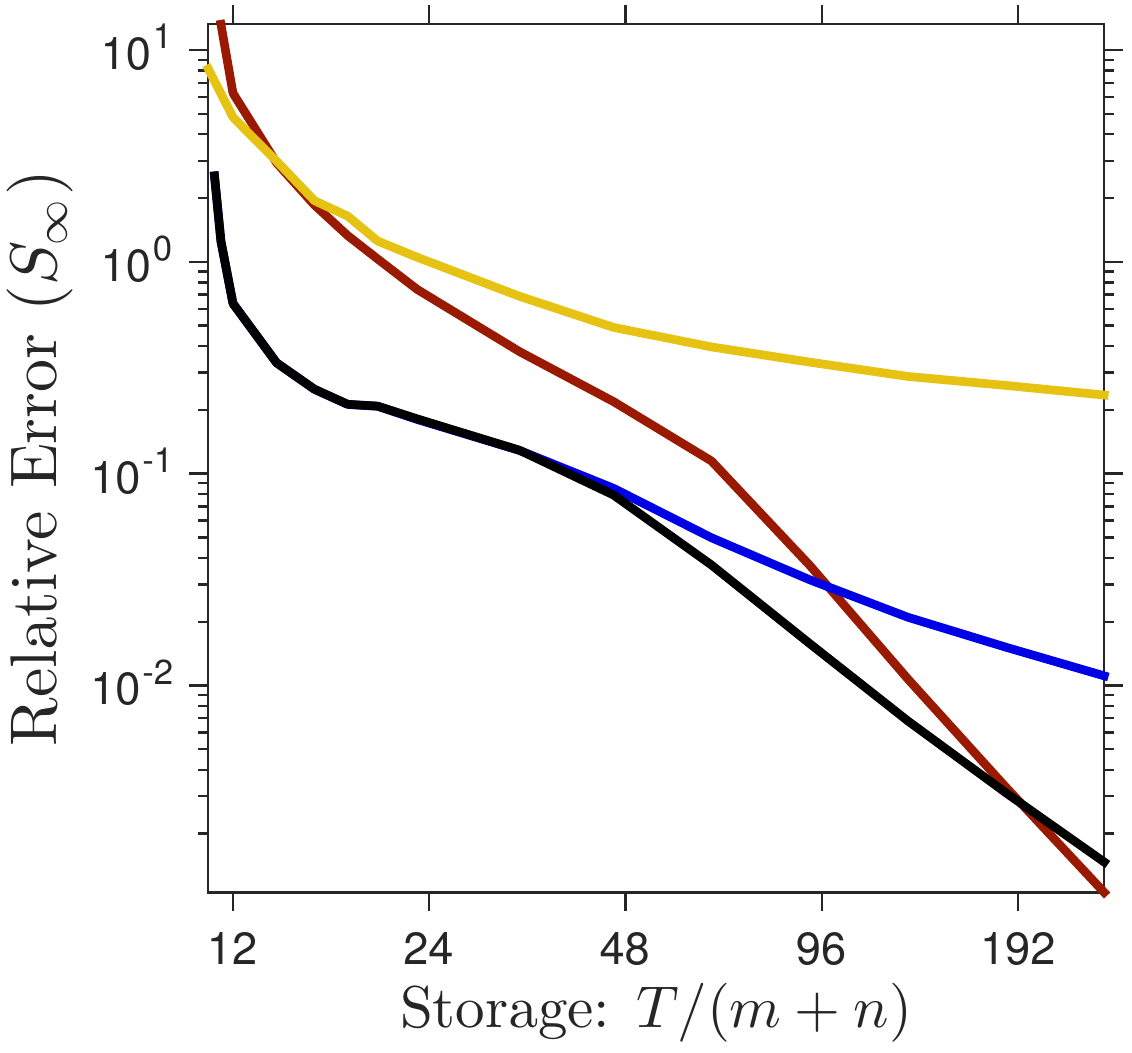}
\caption{\texttt{PolyDecaySlow}}
\end{center}
\end{subfigure}
\begin{subfigure}{.325\textwidth}
\begin{center}
\includegraphics[height=1.5in]{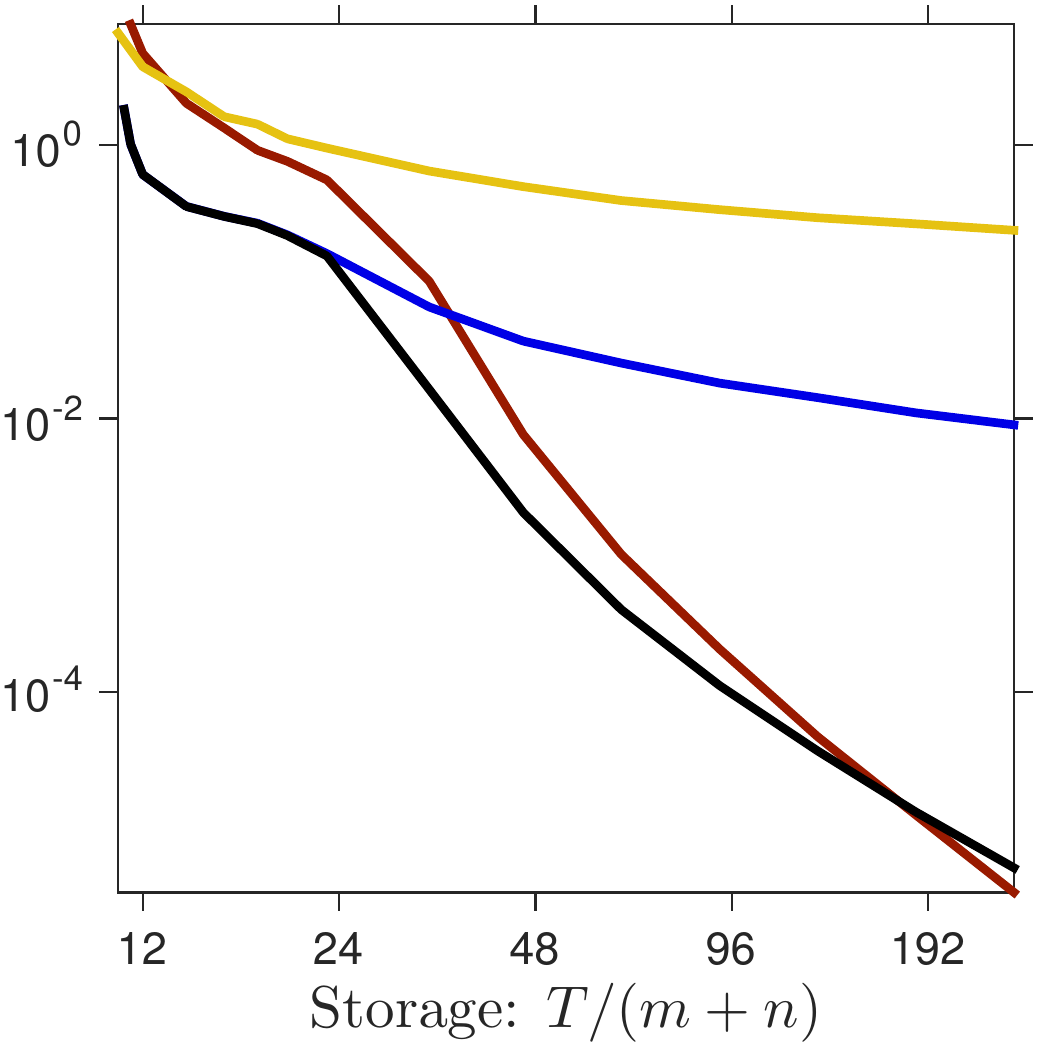}
\caption{\texttt{PolyDecayMed}}
\end{center}
\end{subfigure}
\begin{subfigure}{.325\textwidth}
\begin{center}
\includegraphics[height=1.5in]{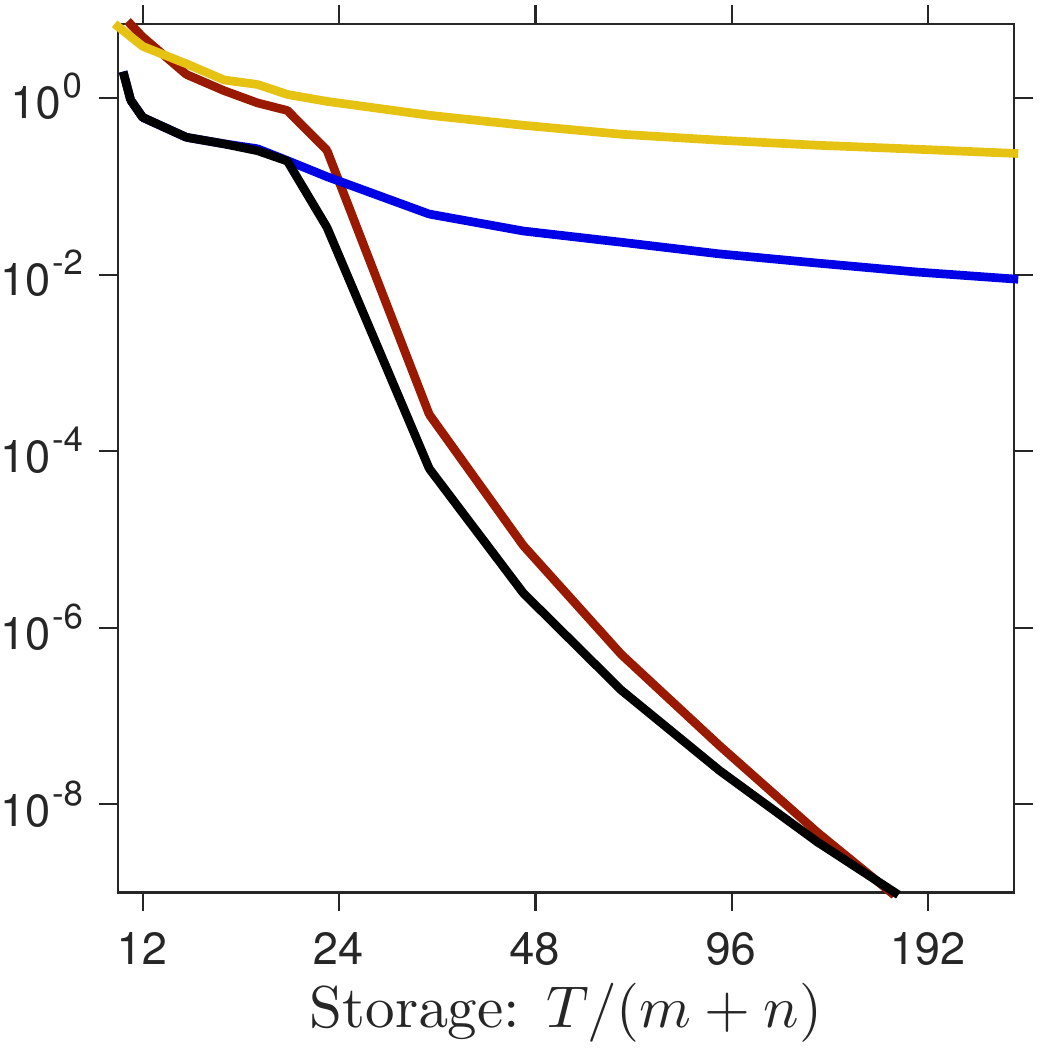}
\caption{\texttt{PolyDecayFast}}
\end{center}
\end{subfigure}
\end{center}

\vspace{0.5em}

\begin{center}
\begin{subfigure}{.325\textwidth}
\begin{center}
\includegraphics[height=1.5in]{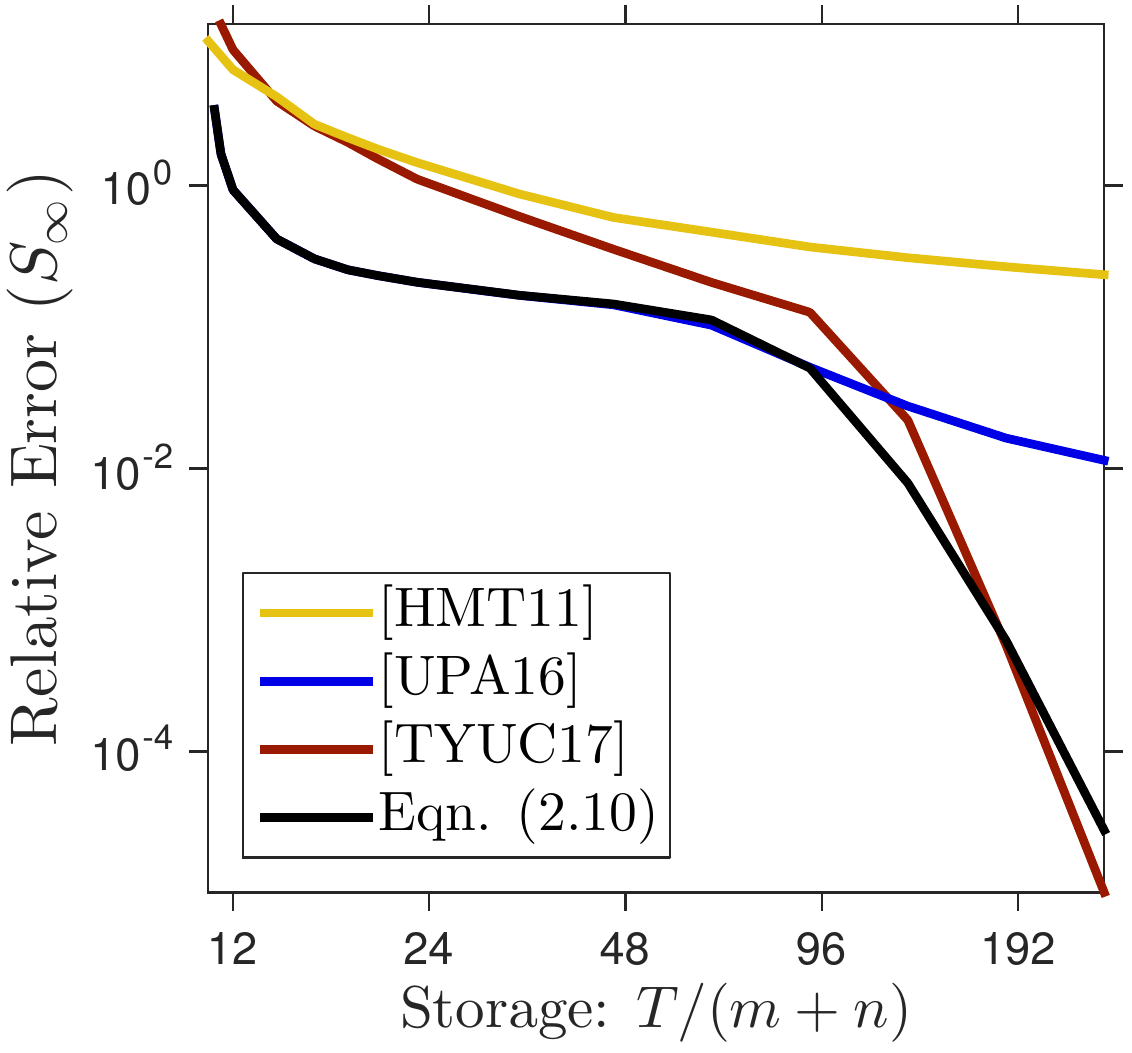}
\caption{\texttt{ExpDecaySlow}}
\end{center}
\end{subfigure}
\begin{subfigure}{.325\textwidth}
\begin{center}
\includegraphics[height=1.5in]{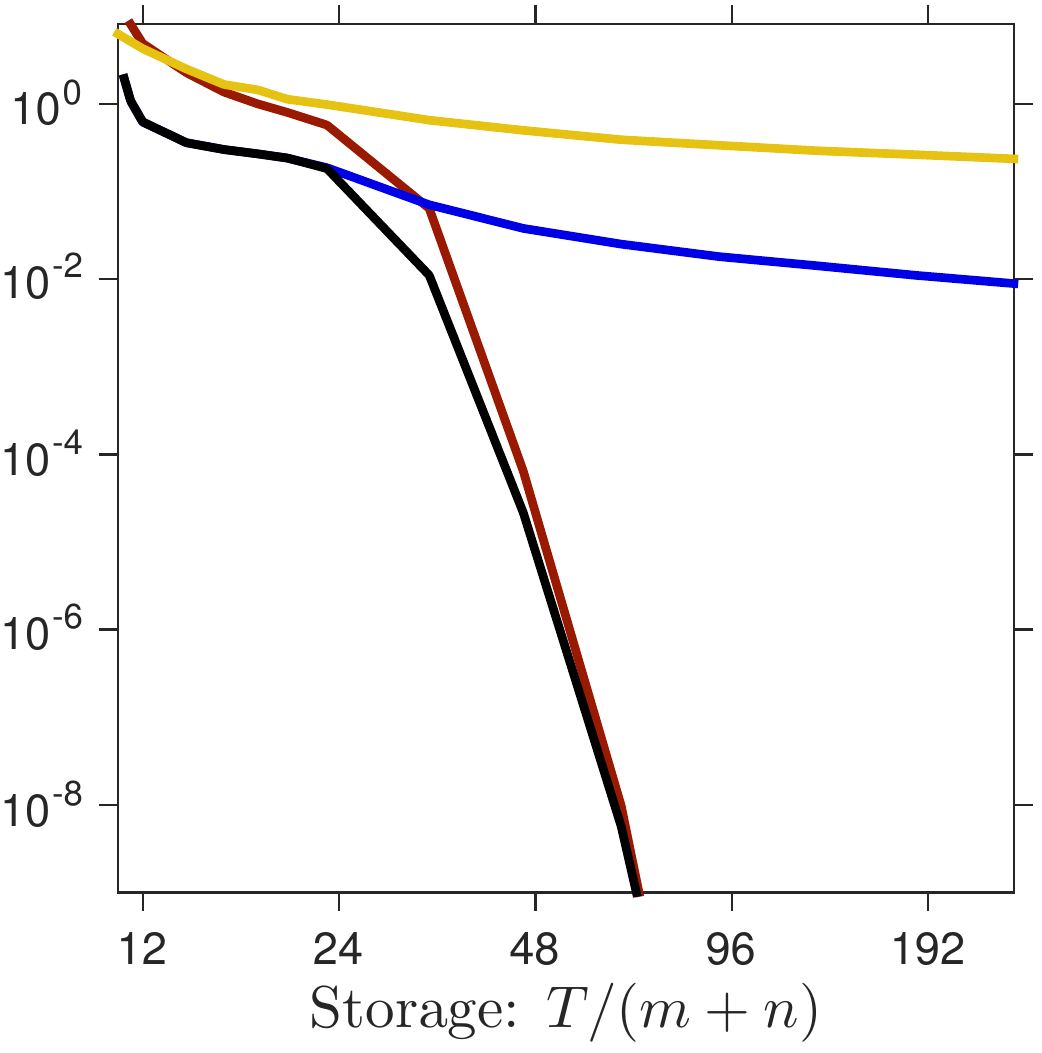}
\caption{\texttt{ExpDecayMed}}
\end{center}
\end{subfigure}
\begin{subfigure}{.325\textwidth}
\begin{center}
\includegraphics[height=1.5in]{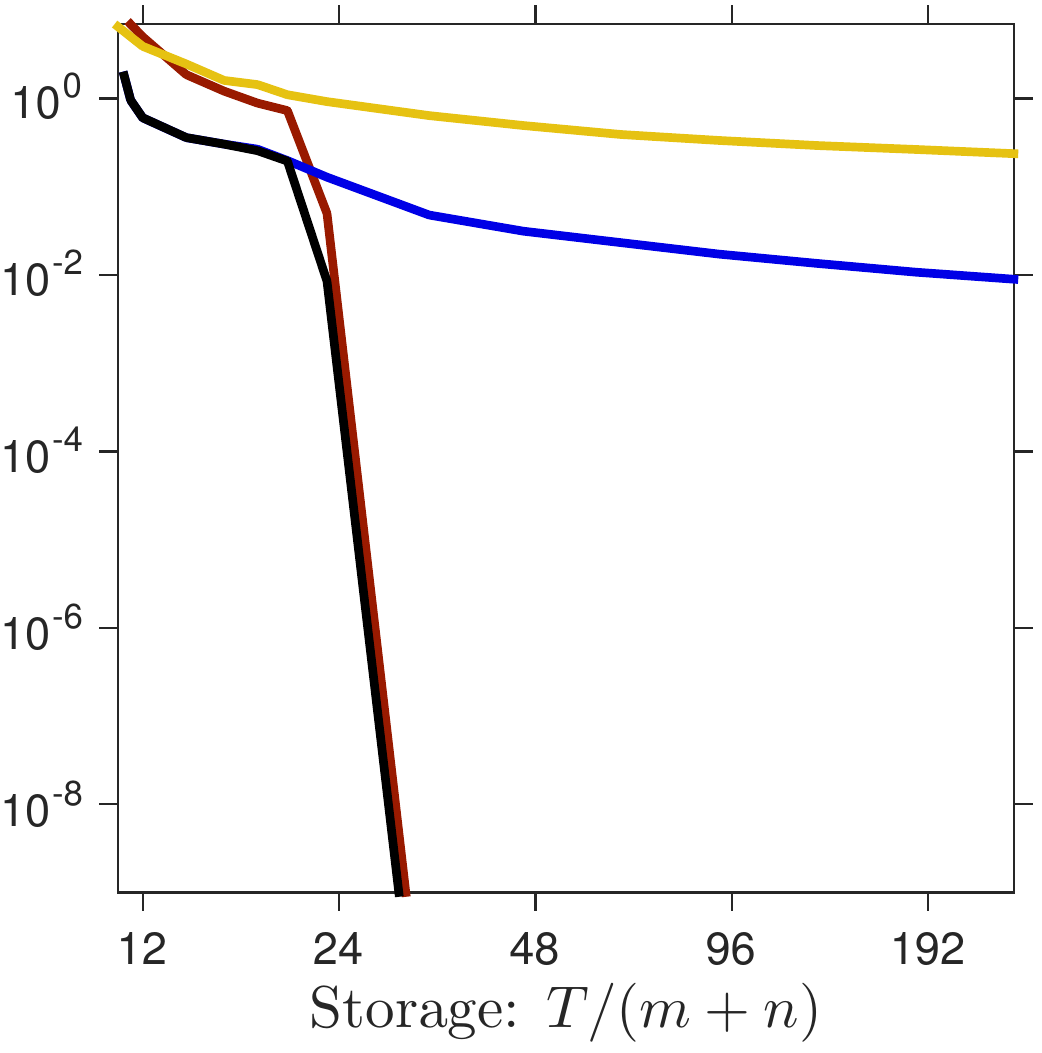}
\caption{\texttt{ExpDecayFast}}
\end{center}
\end{subfigure}
\end{center}

\vspace{0.5em}

\caption{\textbf{Comparison of reconstruction formulas: Synthetic examples.}
(Gaussian maps, effective rank $R = 20$, approximation rank $r = 10$, Schatten $\infty$-norm.)
We compare the oracle error achieved by the proposed fixed-rank
approximation~\cref{eqn:Ahat-fixed} against methods~\cref{eqn:upa,eqn:tyuc2017} from the literature.
See \cref{sec:oracle-error} for details.}
\label{fig:oracle-comparison-R20-Sinf}
\end{figure}

\begin{figure}[htp!]

\begin{center}
\begin{subfigure}{.45\textwidth}
\begin{center}
\includegraphics[height=2in]{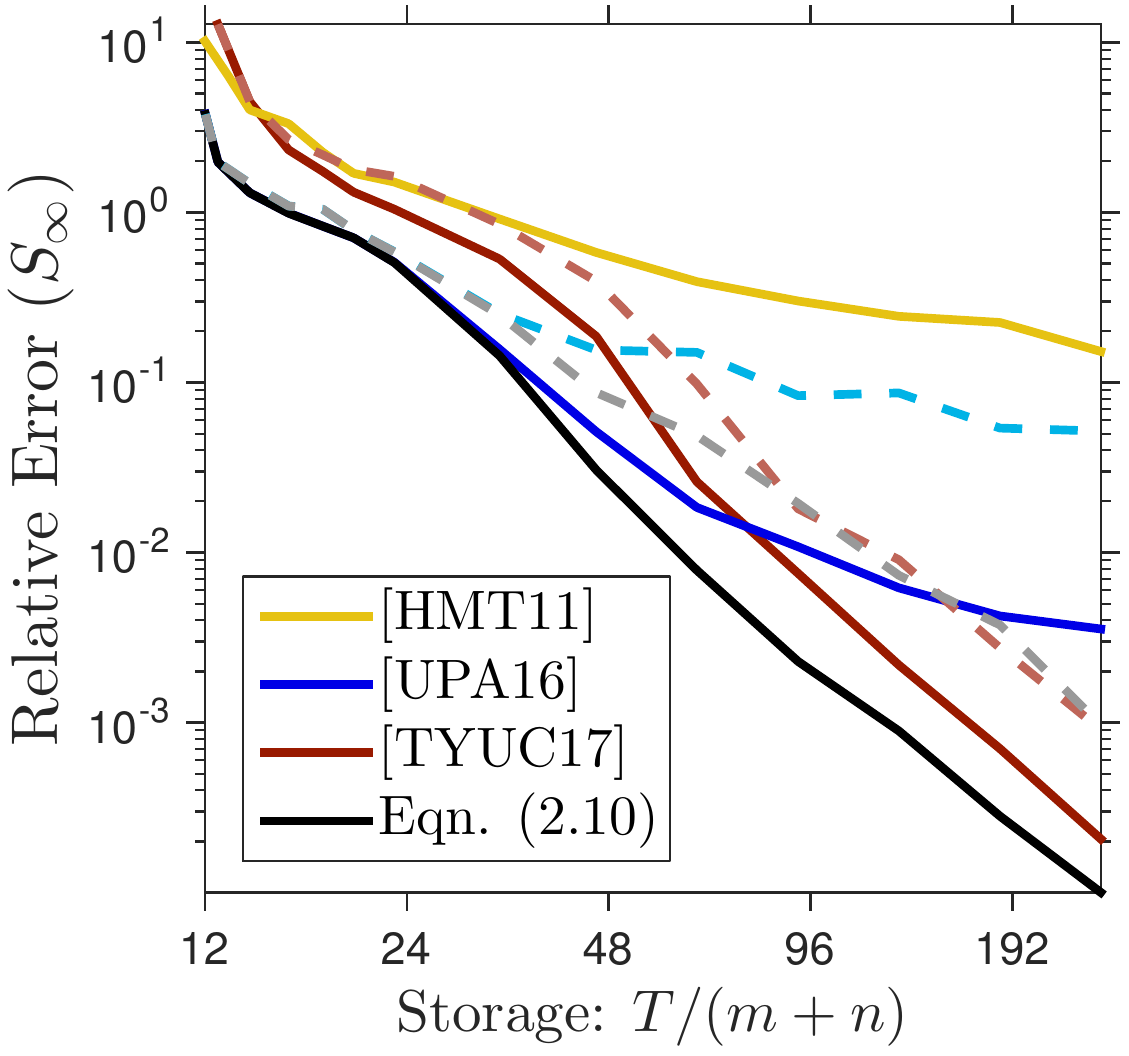}
\caption{\texttt{MinTemp} ($r = 10$)}
\end{center}
\end{subfigure}
\begin{subfigure}{.45\textwidth}
\begin{center}
\includegraphics[height=2in]{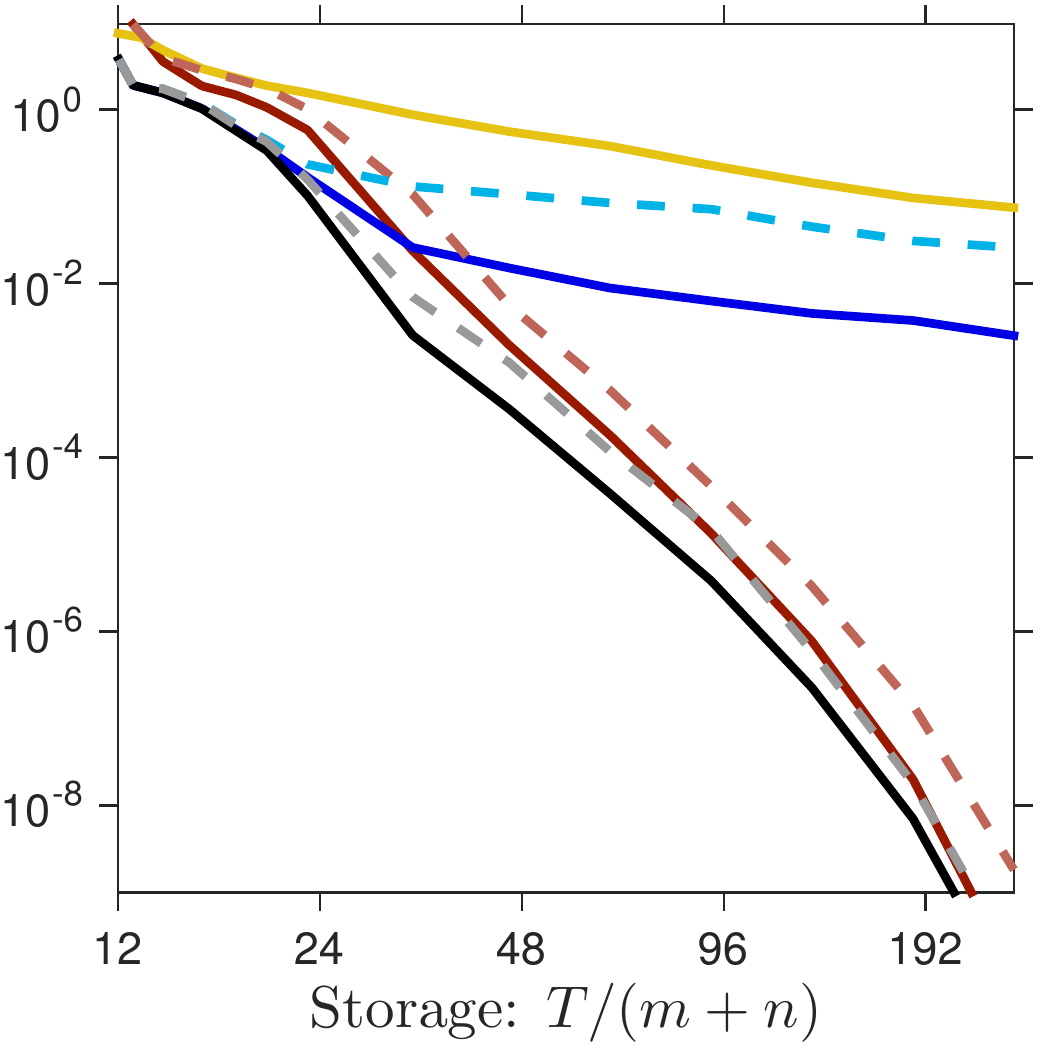}
\caption{\texttt{StreamVel} ($r = 10$)}
\end{center}
\end{subfigure}
\end{center}

\vspace{0.5em}

\begin{center}
\begin{subfigure}{.45\textwidth}
\begin{center}
\includegraphics[height=2in]{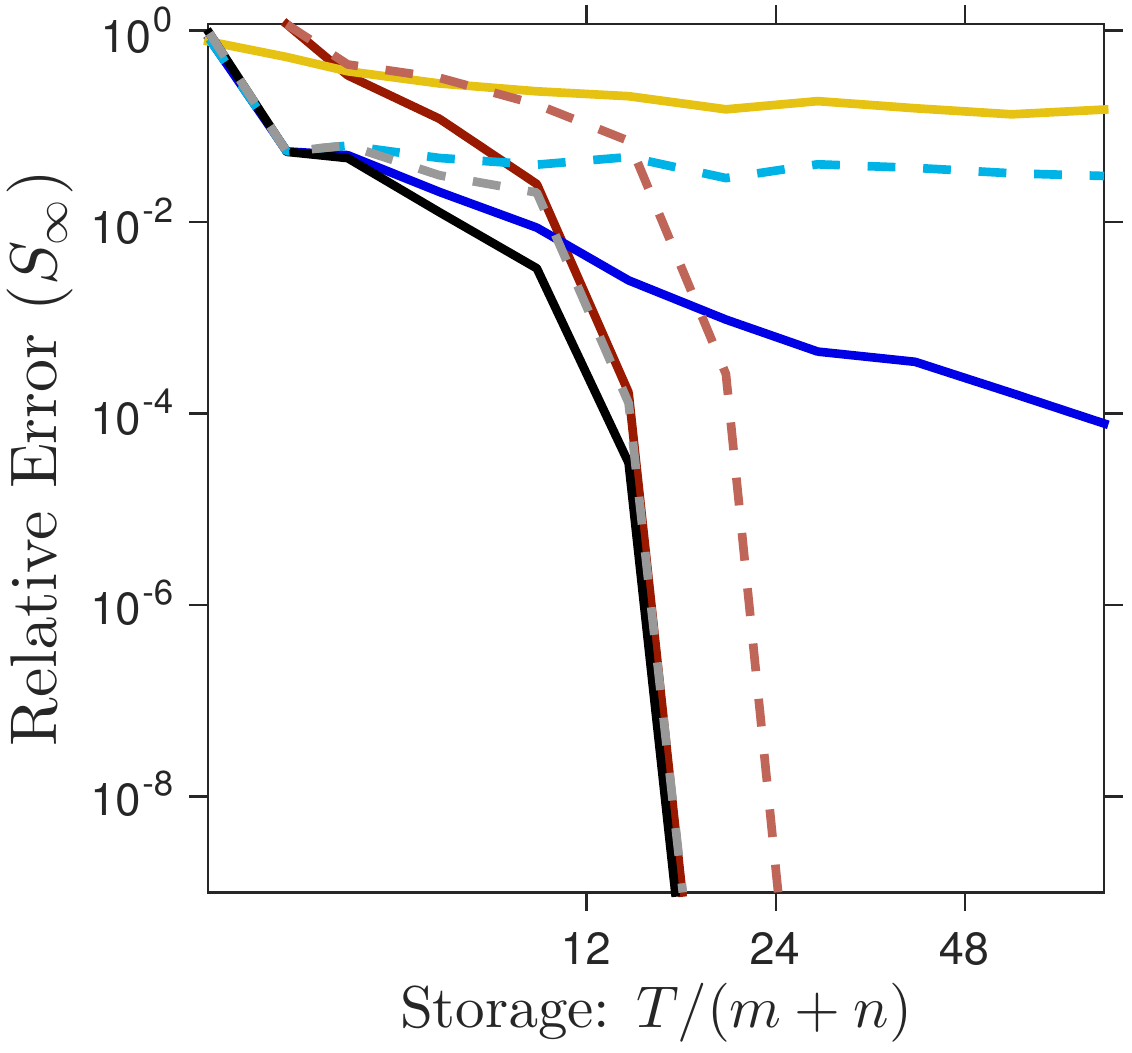}
\caption{\texttt{MaxCut} ($r = 1$)}
\end{center}
\end{subfigure}
\begin{subfigure}{.45\textwidth}
\begin{center}
\includegraphics[height=2in]{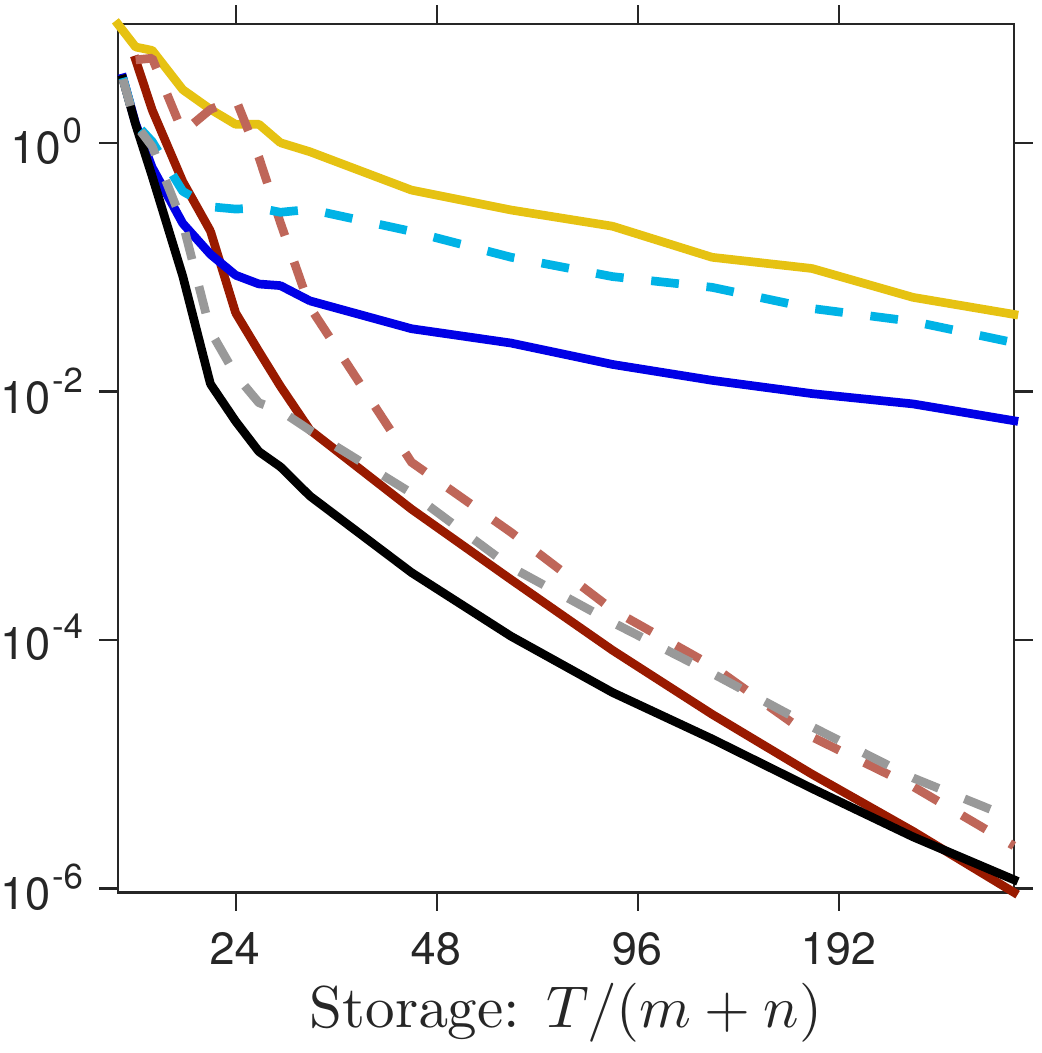}
\caption{\texttt{MaxCut} ($r = 14$)}
\end{center}
\end{subfigure}
\end{center}

\vspace{.5em}

\begin{center}
\begin{subfigure}{.45\textwidth}
\begin{center}
\includegraphics[height=2in]{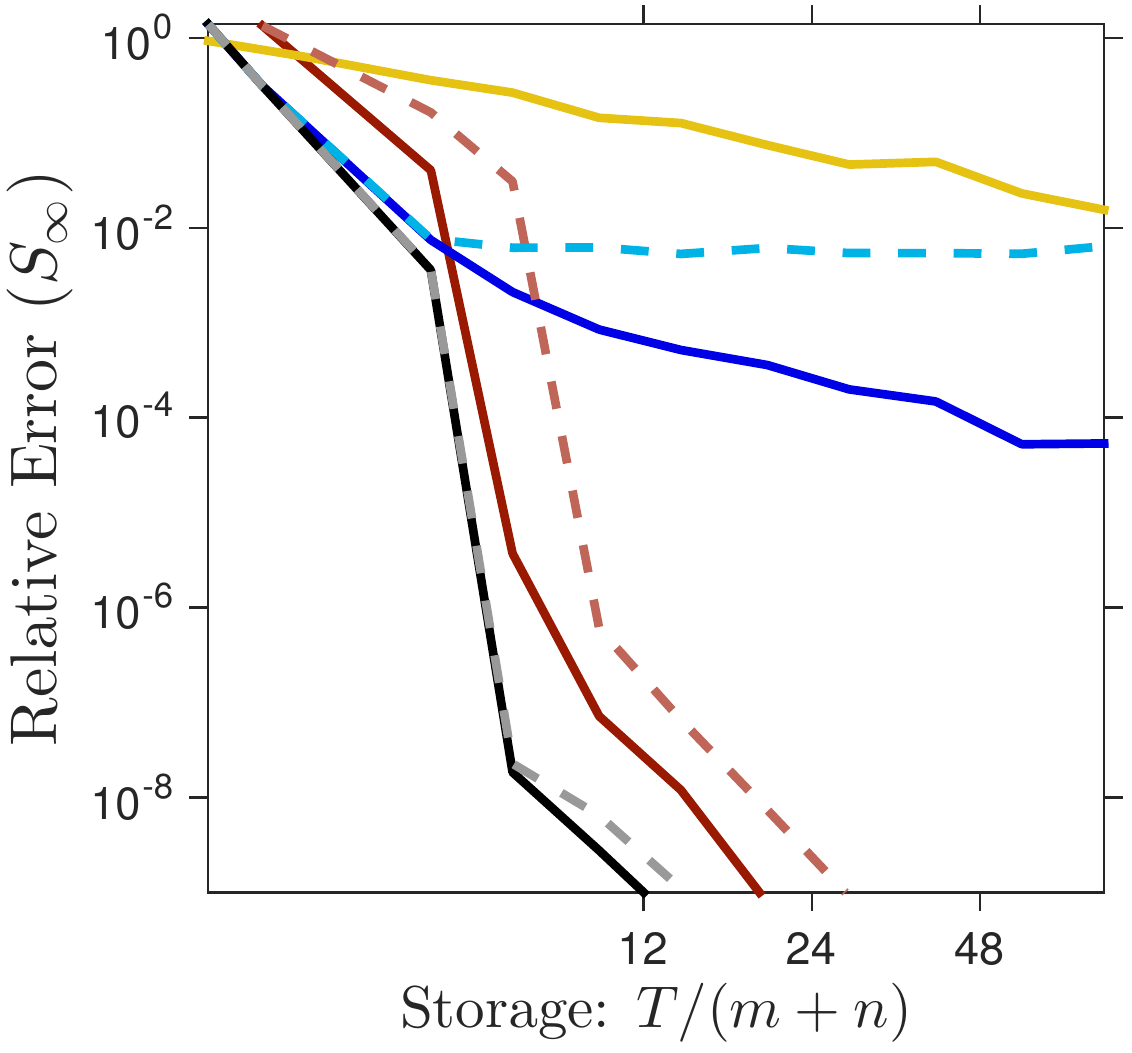}
\caption{\texttt{PhaseRetrieval} ($r = 1$)}
\end{center}
\end{subfigure}
\begin{subfigure}{.45\textwidth}
\begin{center}
\includegraphics[height=2in]{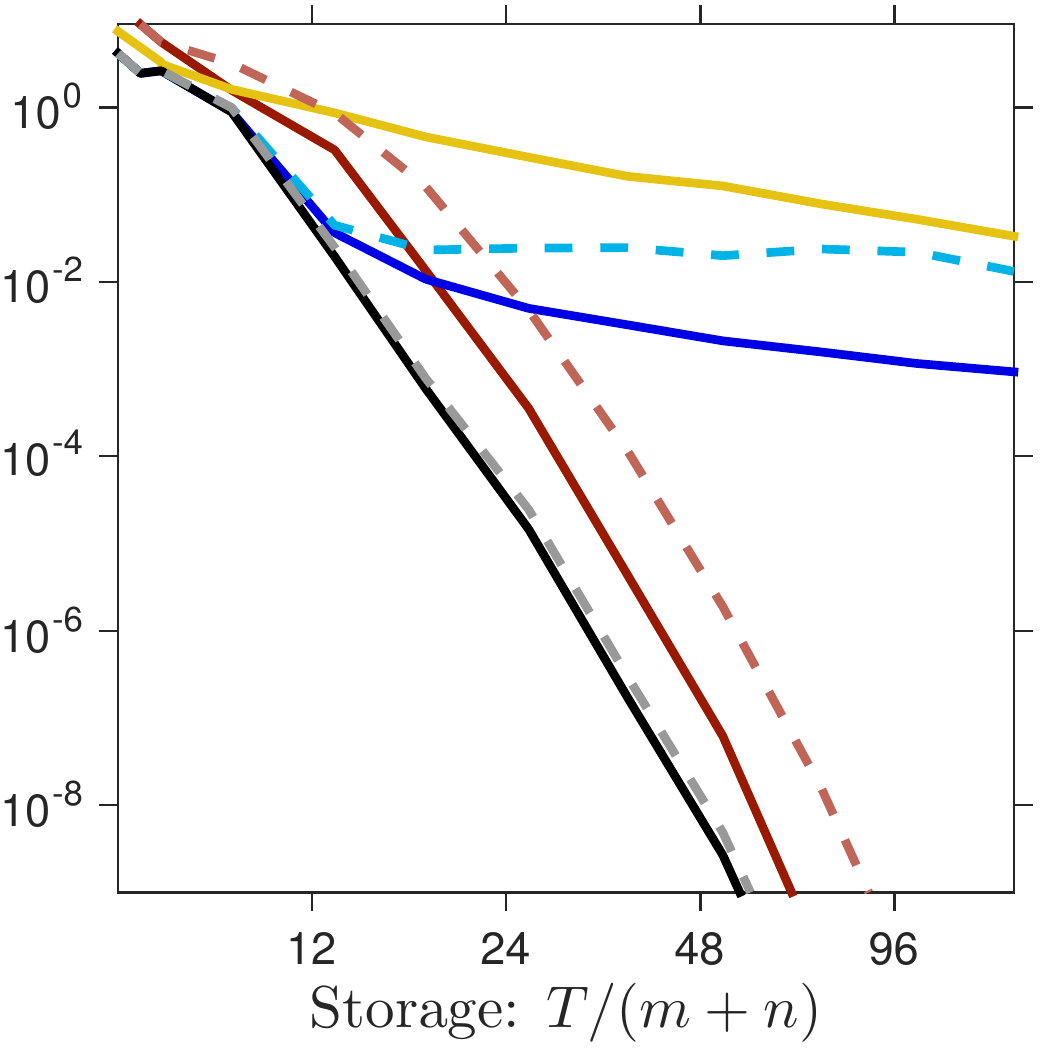}
\caption{\texttt{PhaseRetrieval} ($r = 5$)}
\end{center}
\end{subfigure}
\end{center}

\vspace{0.5em}

\caption{\textbf{Comparison of reconstruction formulas: Real data examples.}
(Sparse maps, Schatten $\infty$-norm.)
We compare the relative error achieved by the proposed fixed-rank
approximation~\cref{eqn:Ahat-fixed} against methods~\cref{eqn:upa,eqn:tyuc2017} from the literature.
\textbf{Solid lines} are oracle errors; \textbf{dashed lines} are errors with ``natural''
parameter choices.  See \cref{sec:real-data} for details.}
\label{fig:data-comparison-Sinf}
\end{figure}

\begin{figure}[htp!]

\begin{center}

\begin{subfigure}{.48\textwidth}
\begin{center}
\includegraphics[width=2.4in]{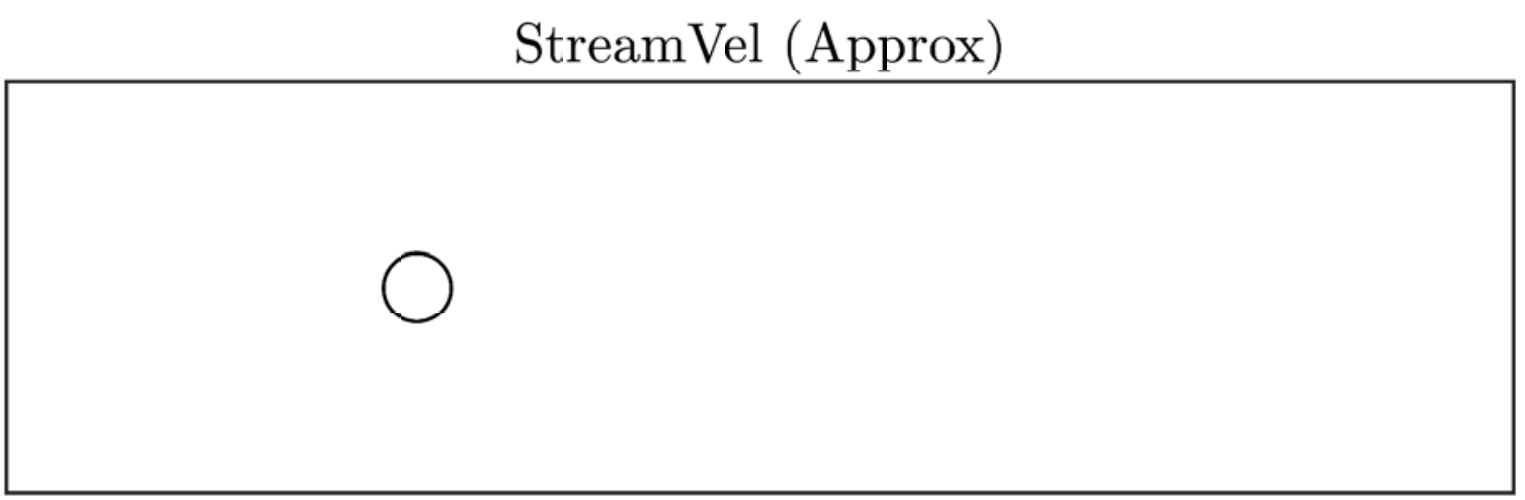}
\end{center}
\end{subfigure}
\begin{subfigure}{.48\textwidth}
\begin{center}
\includegraphics[width=2.4in]{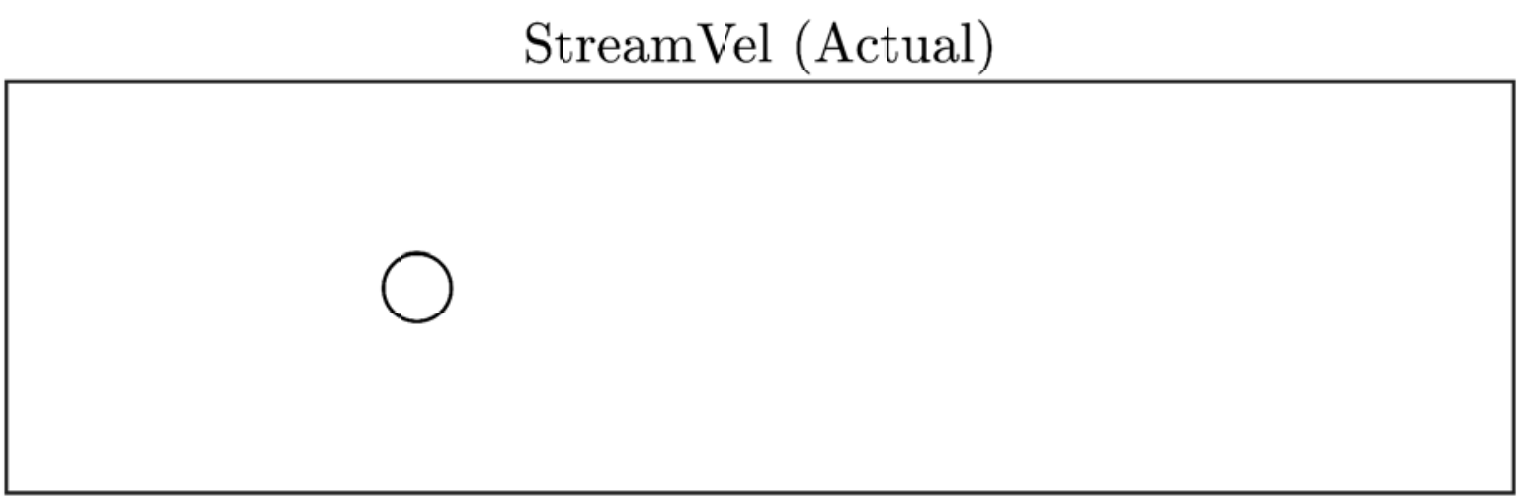}
\end{center}
\end{subfigure}

\vspace{1mm}

\begin{subfigure}{.48\textwidth}
\begin{center}
\includegraphics[width=2.4in]{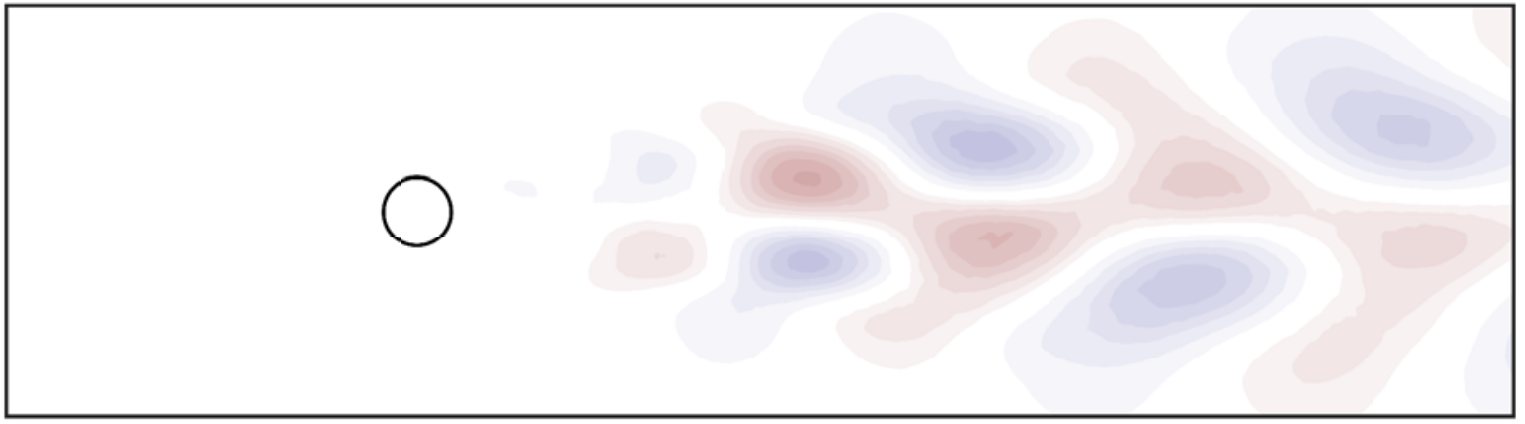}
\end{center}
\end{subfigure}
\begin{subfigure}{.48\textwidth}
\begin{center}
\includegraphics[width=2.4in]{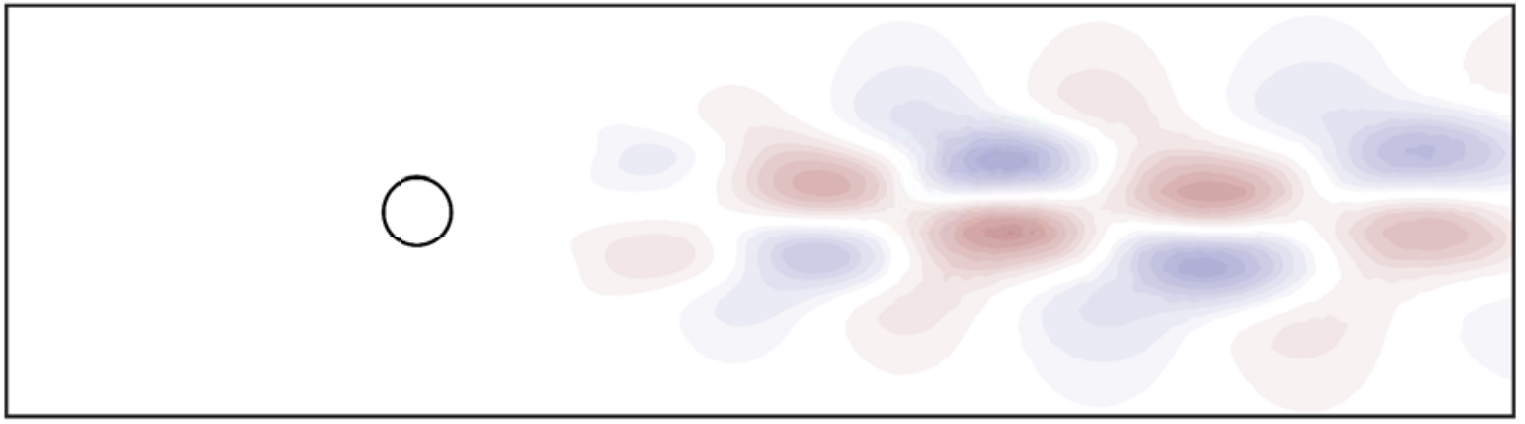}
\end{center}
\end{subfigure}

\vspace{1mm}

\begin{subfigure}{.48\textwidth}
\begin{center}
\includegraphics[width=2.4in]{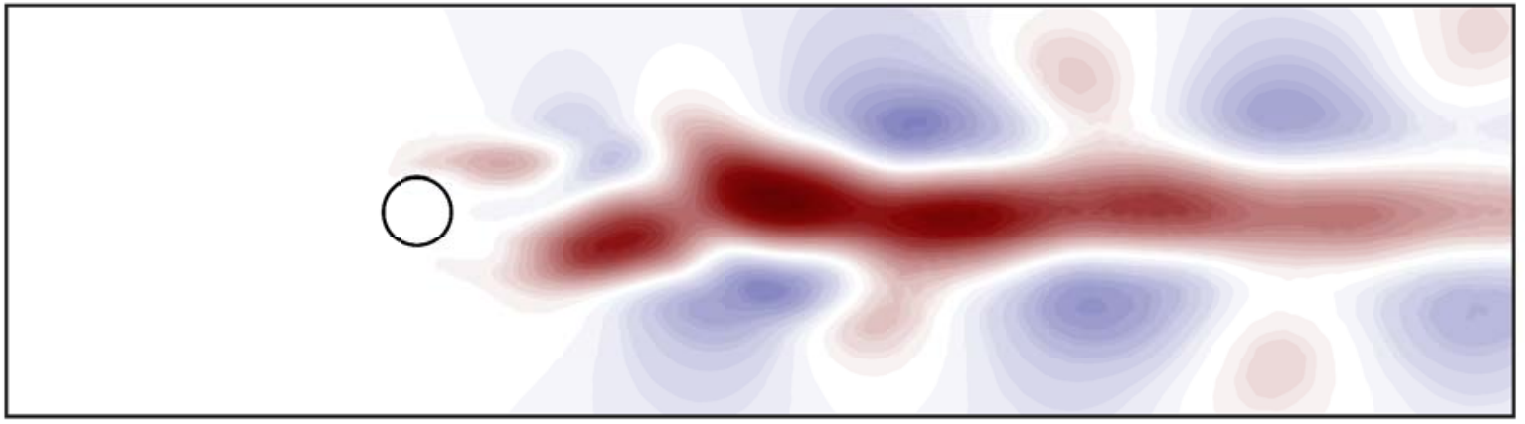}
\end{center}
\end{subfigure}
\begin{subfigure}{.48\textwidth}
\begin{center}
\includegraphics[width=2.4in]{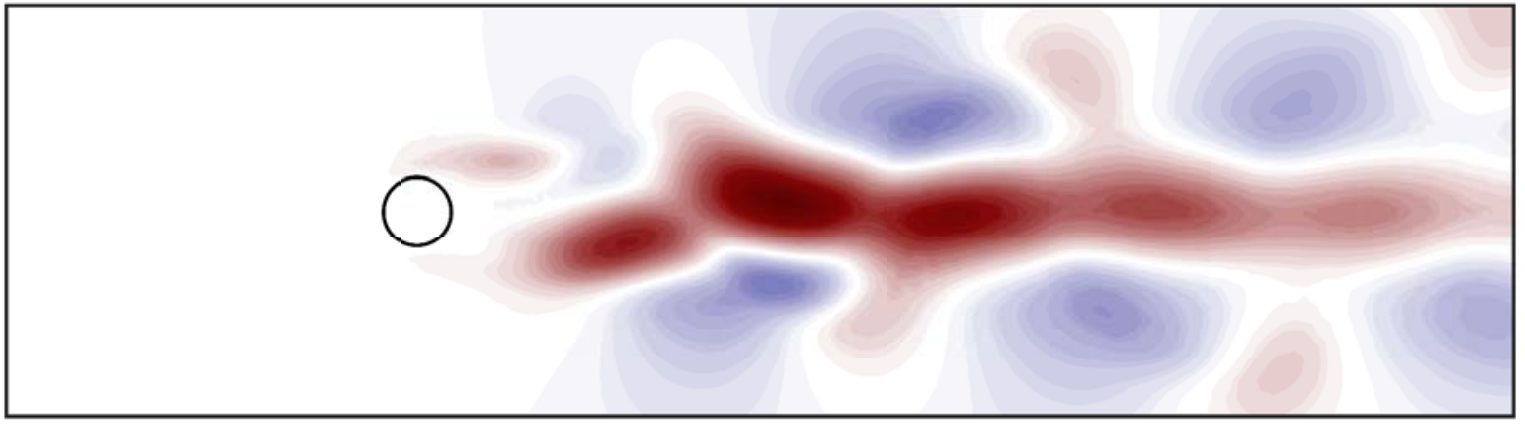}
\end{center}
\end{subfigure}

\vspace{1mm}

\begin{subfigure}{.48\textwidth}
\begin{center}
\includegraphics[width=2.4in]{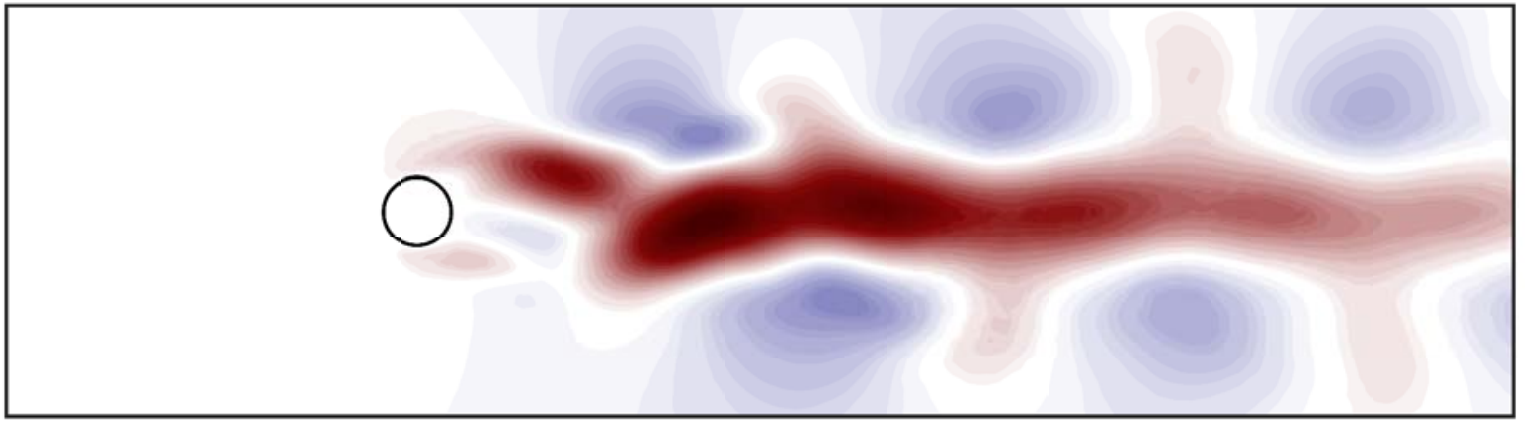}
\end{center}
\end{subfigure}
\begin{subfigure}{.48\textwidth}
\begin{center}
\includegraphics[width=2.4in]{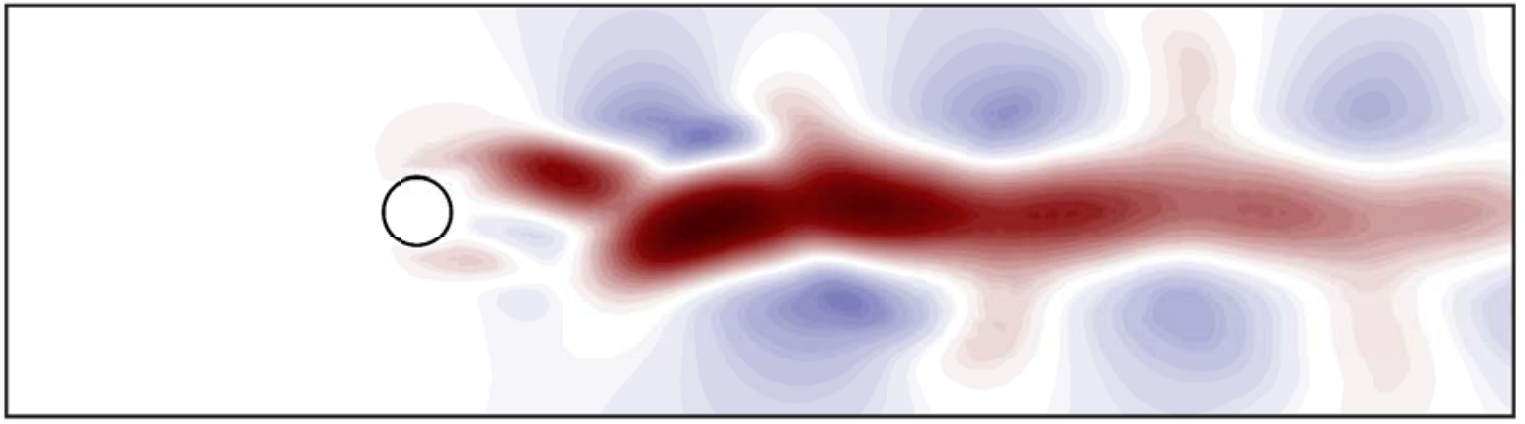}
\end{center}
\end{subfigure}

\vspace{1mm}

\begin{subfigure}{.48\textwidth}
\begin{center}
\includegraphics[width=2.4in]{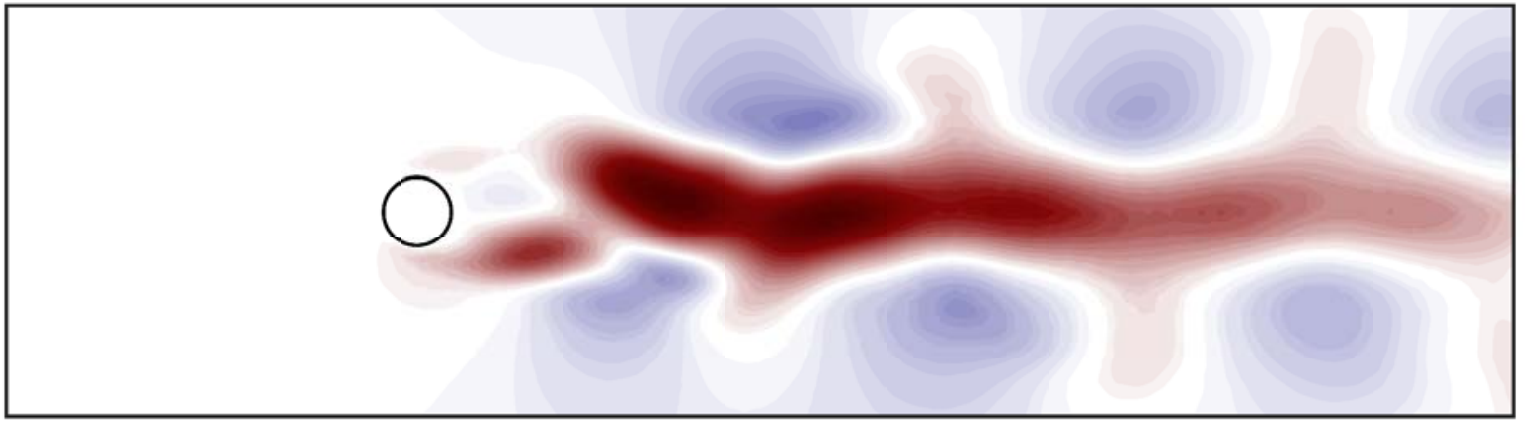}
\end{center}
\end{subfigure}
\begin{subfigure}{.48\textwidth}
\begin{center}
\includegraphics[width=2.4in]{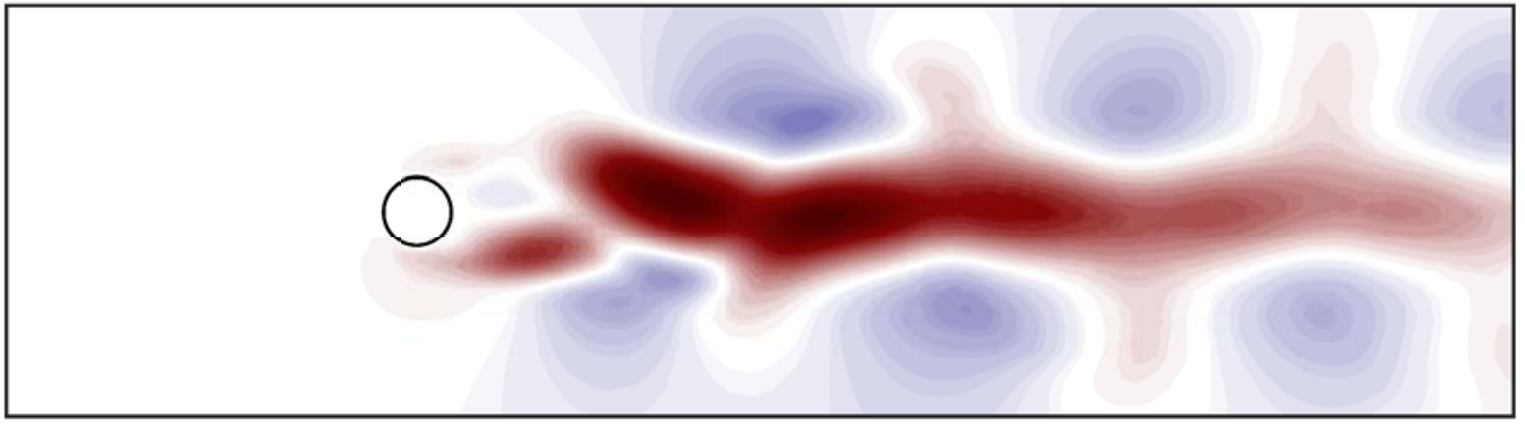}
\end{center}
\end{subfigure}

\vspace{1mm}

\begin{subfigure}{.48\textwidth}
\begin{center}
\includegraphics[width=2.4in]{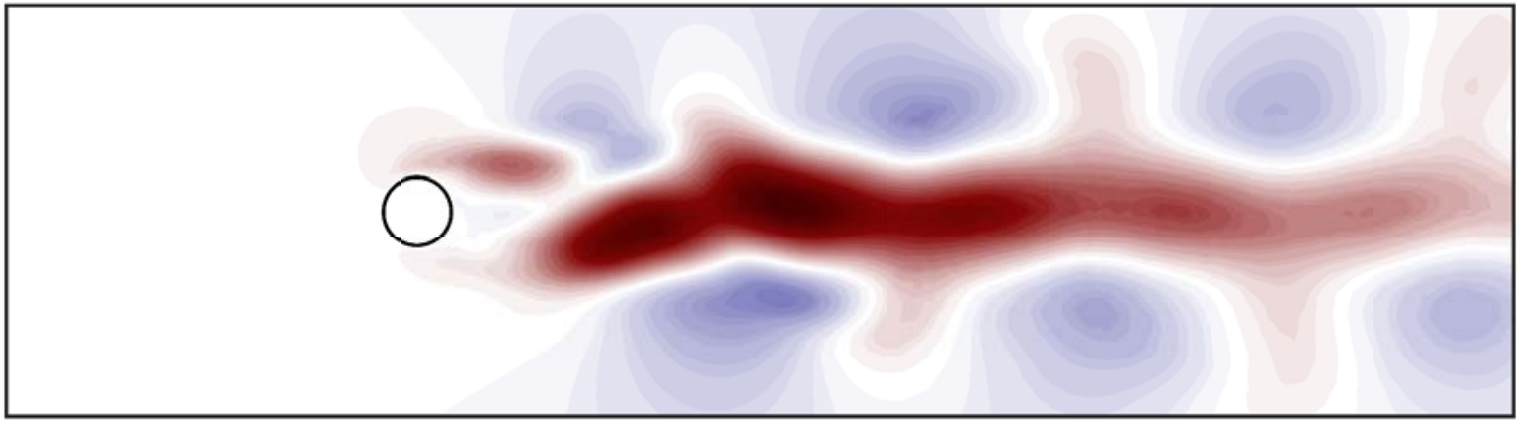}
\end{center}
\end{subfigure}
\begin{subfigure}{.48\textwidth}
\begin{center}
\includegraphics[width=2.4in]{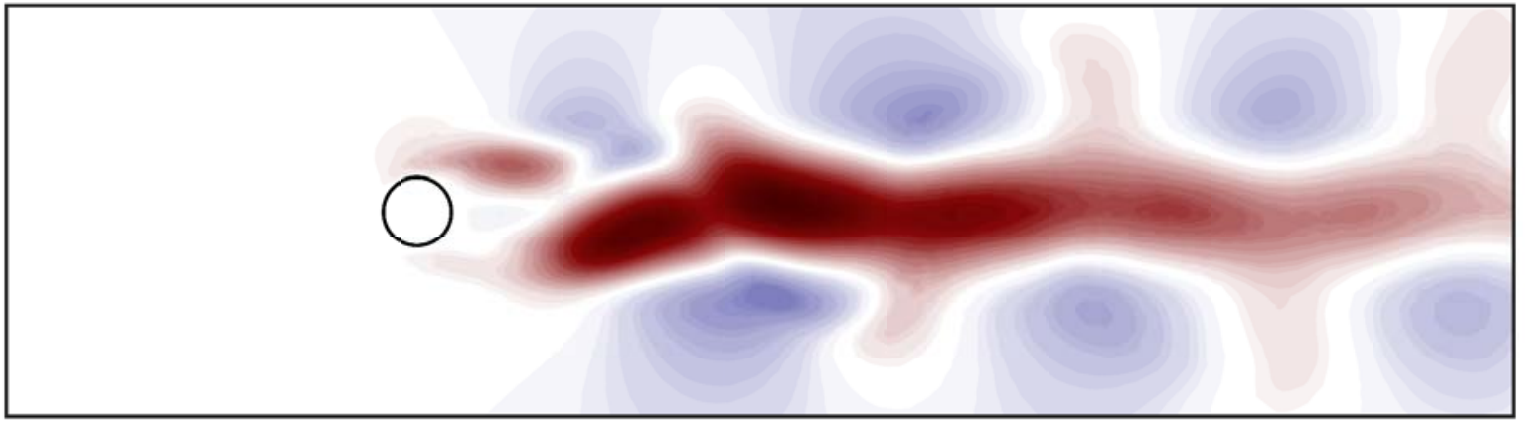}
\end{center}
\end{subfigure}

\vspace{1mm}

\begin{subfigure}{.48\textwidth}
\begin{center}
\includegraphics[width=2.4in]{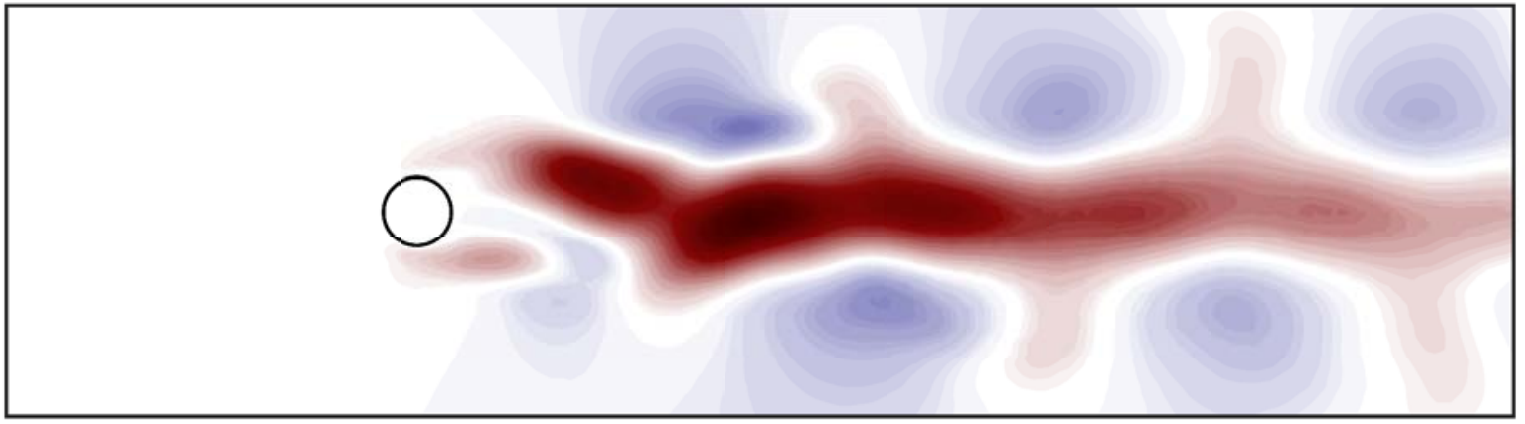}
\end{center}
\end{subfigure}
\begin{subfigure}{.48\textwidth}
\begin{center}
\includegraphics[width=2.4in]{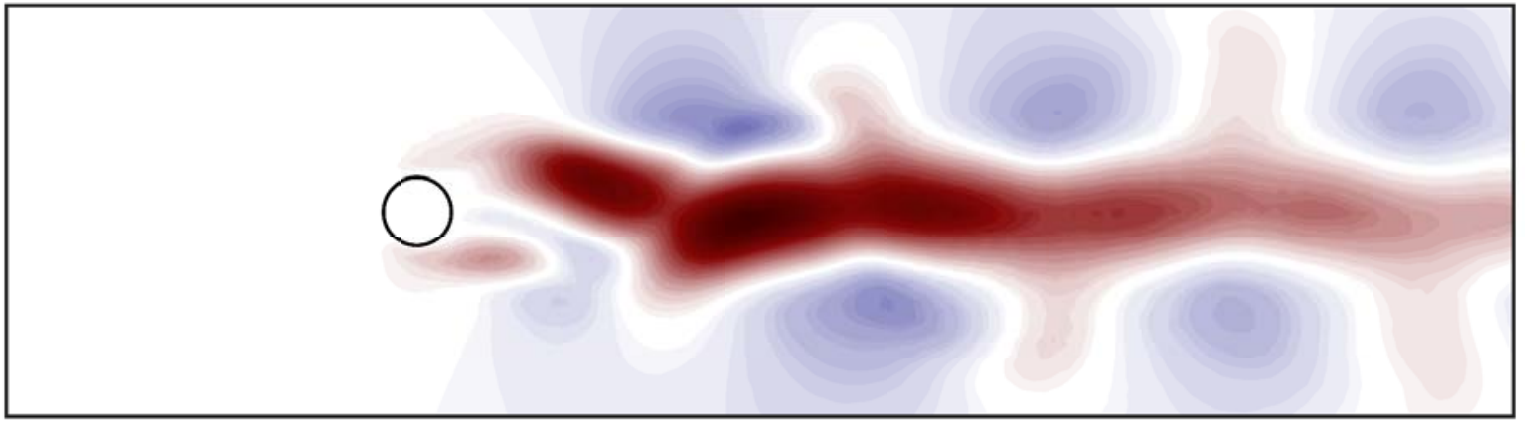}
\end{center}
\end{subfigure}

\vspace{1mm}

\begin{subfigure}{.48\textwidth}
\begin{center}
\includegraphics[width=2.4in]{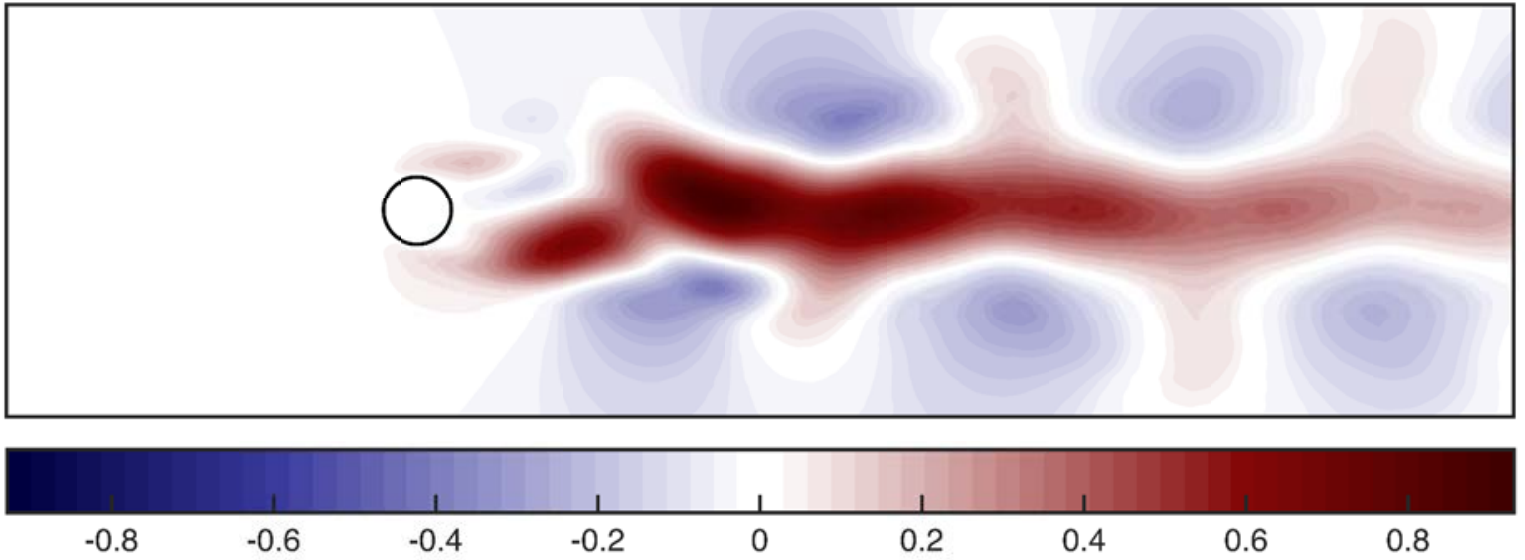}
\end{center}
\end{subfigure}
\begin{subfigure}{.48\textwidth}
\begin{center}
\includegraphics[width=2.4in]{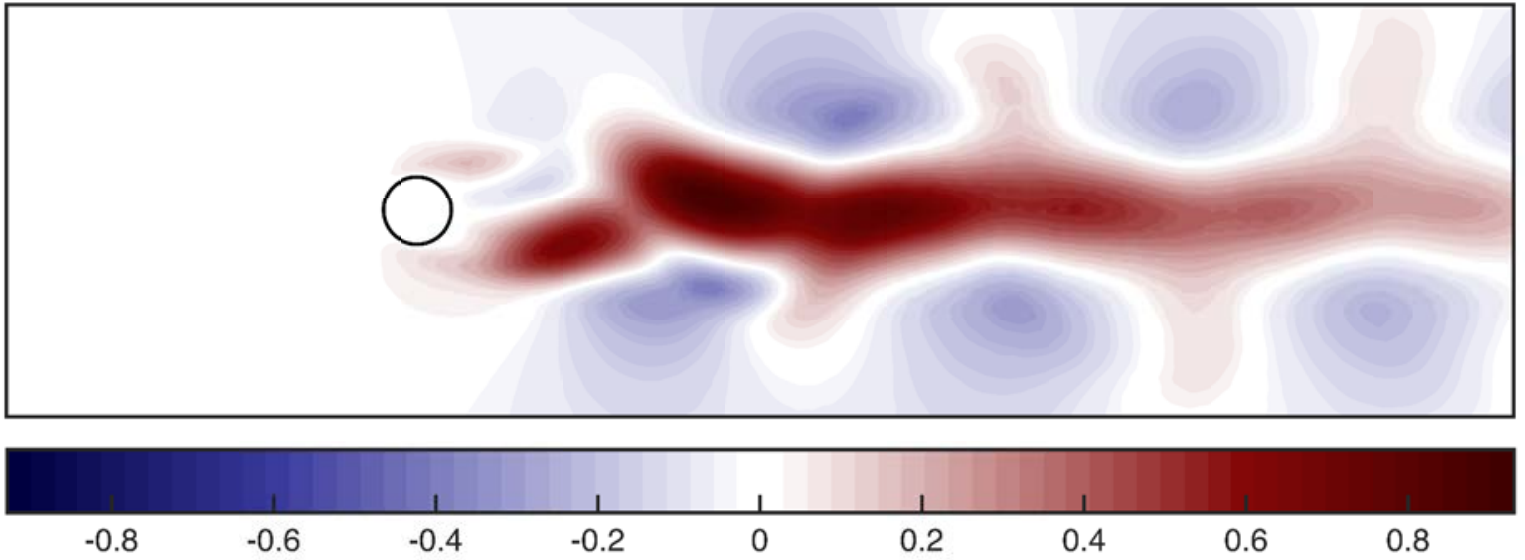}
\end{center}
\end{subfigure}

\vspace{1pc}

\caption{\textbf{Approximation of \texttt{StreamVel} via~\cref{eqn:Ahat-fixed}.}
(Sparse maps, approximation rank $r = 10$, storage budget $T = 48 (m+n)$.)
The columns of the matrix \texttt{StreamVel} describe the fluctuations
of the streamwise velocity field about its mean value as a function of time.
From top to bottom, the panels show columns $1, 1001, 1501, 2001, 2501, 3001, 3501, 4001$.
The \textbf{left-hand side} displays the approximation~\cref{eqn:Ahat-fixed} of the flow field,
and the \textbf{right-hand side} displays the exact flow field.
The heatmap indicates the magnitude of the fluctuation.
See \cref{sec:flow-field}.}
\label{fig:flow-field-time}
\end{center}
\end{figure}

\begin{figure}[htp!]

\begin{center}
\begin{subfigure}{.48\textwidth}
\begin{center}
\includegraphics[width=2.4in]{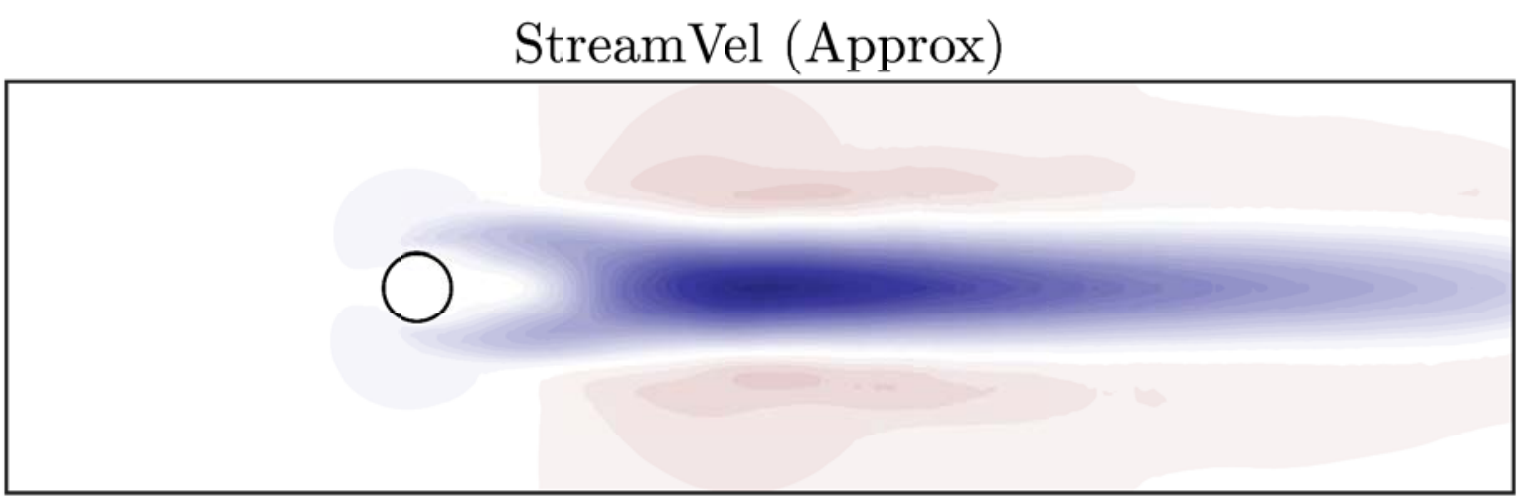}
\end{center}
\end{subfigure}
\begin{subfigure}{.48\textwidth}
\begin{center}
\includegraphics[width=2.4in]{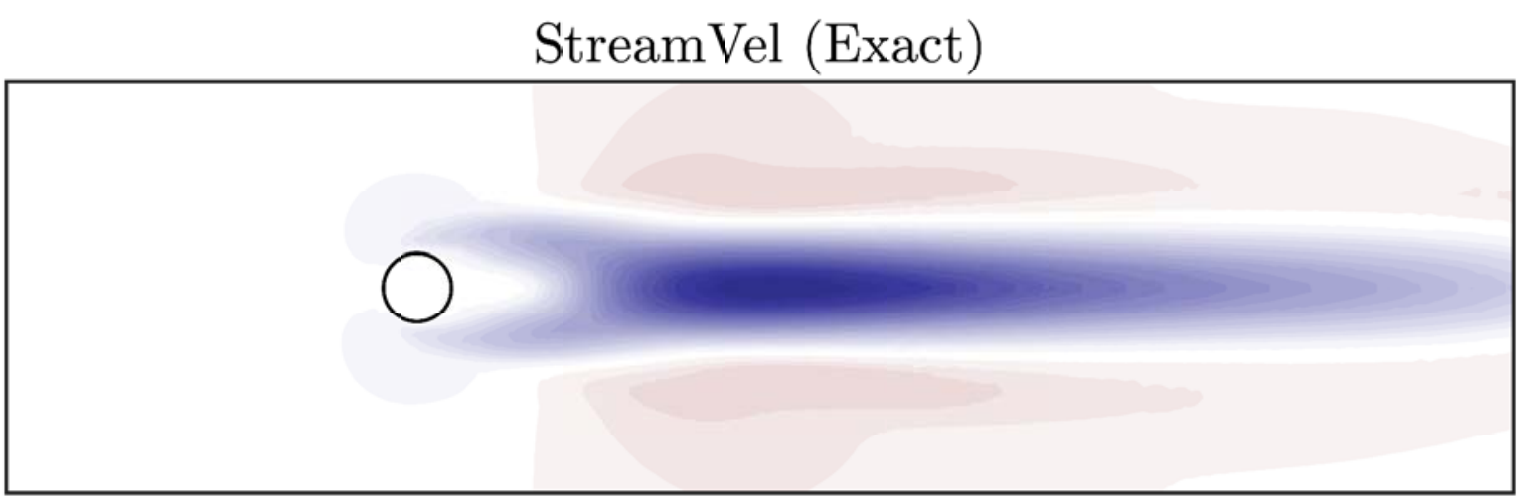}
\end{center}
\end{subfigure}

\vspace{1mm}

\begin{subfigure}{.48\textwidth}
\begin{center}
\includegraphics[width=2.4in]{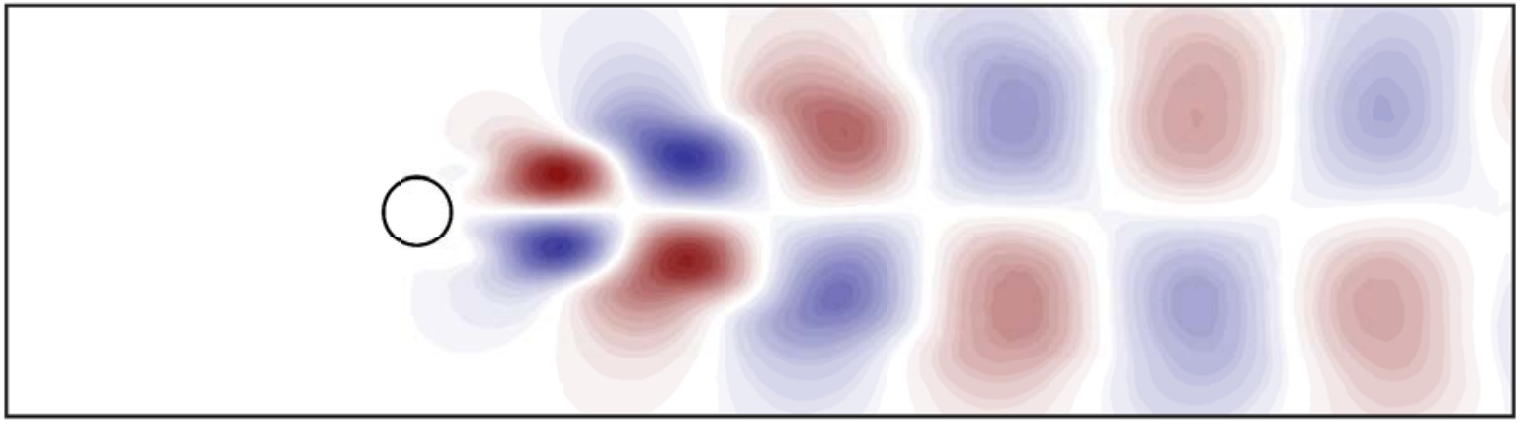}
\end{center}
\end{subfigure}
\begin{subfigure}{.48\textwidth}
\begin{center}
\includegraphics[width=2.4in]{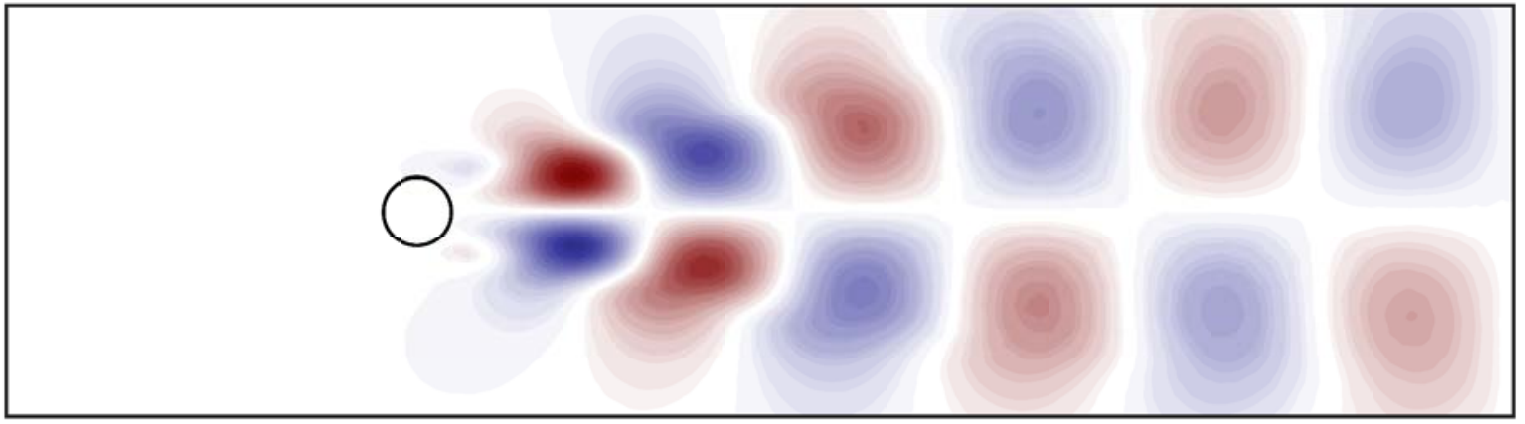}
\end{center}
\end{subfigure}

\vspace{1mm}

\begin{subfigure}{.48\textwidth}
\begin{center}
\includegraphics[width=2.4in]{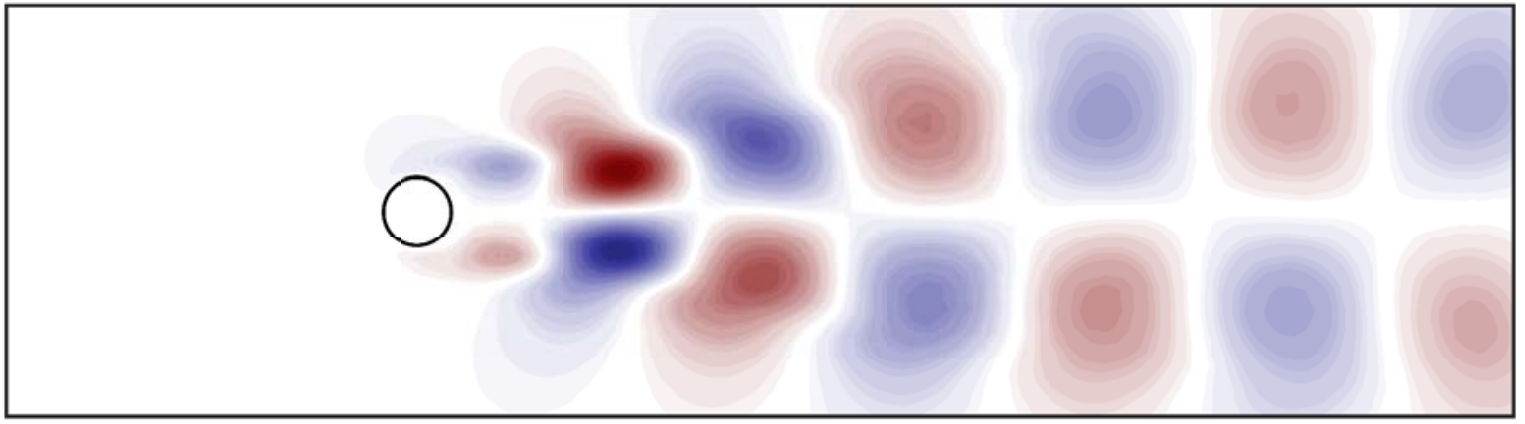}
\end{center}
\end{subfigure}
\begin{subfigure}{.48\textwidth}
\begin{center}
\includegraphics[width=2.4in]{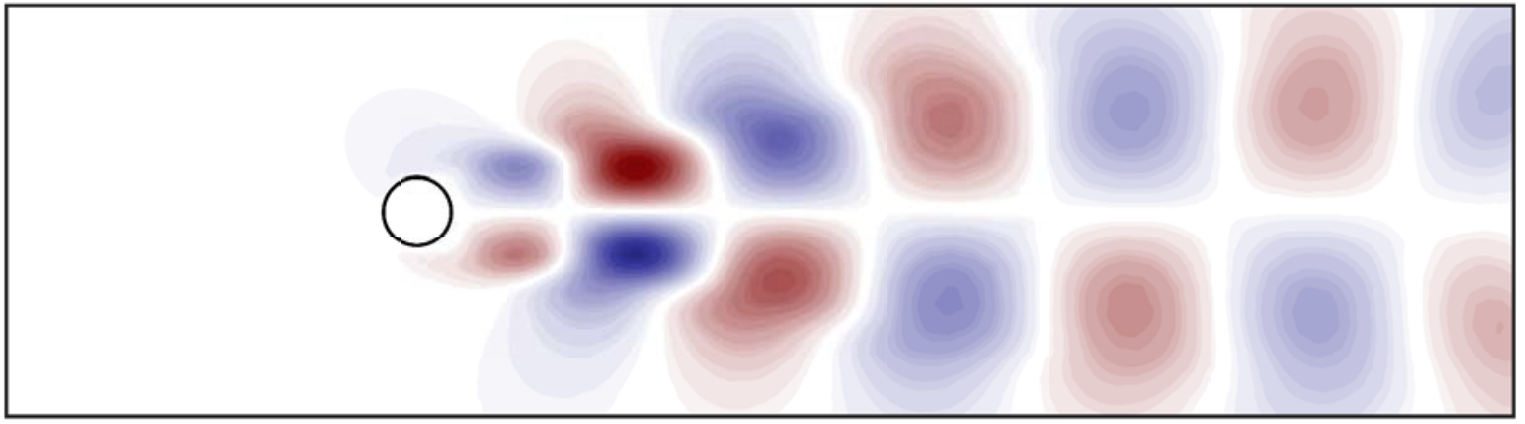}
\end{center}
\end{subfigure}

\vspace{1mm}

\begin{subfigure}{.48\textwidth}
\begin{center}
\includegraphics[width=2.4in]{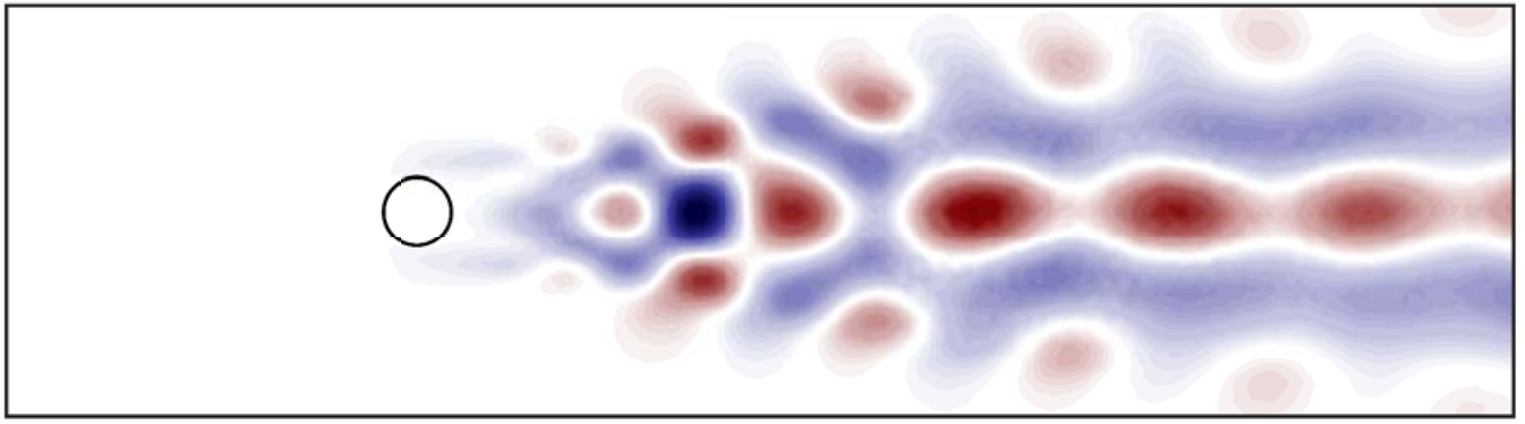}
\end{center}
\end{subfigure}
\begin{subfigure}{.48\textwidth}
\begin{center}
\includegraphics[width=2.4in]{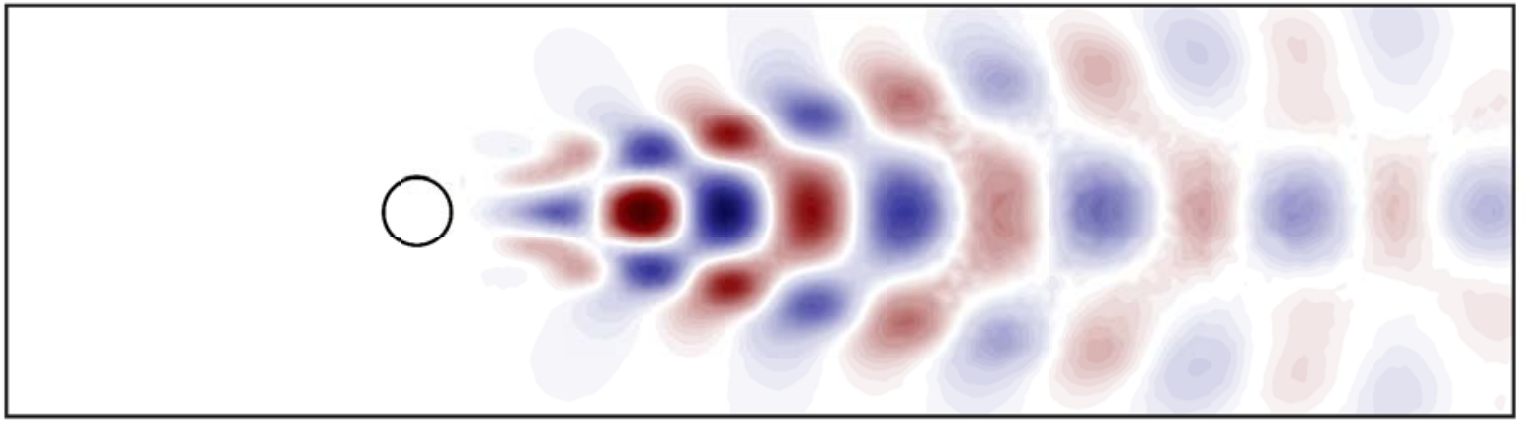}
\end{center}
\end{subfigure}

\vspace{1mm}

\begin{subfigure}{.48\textwidth}
\begin{center}
\includegraphics[width=2.4in]{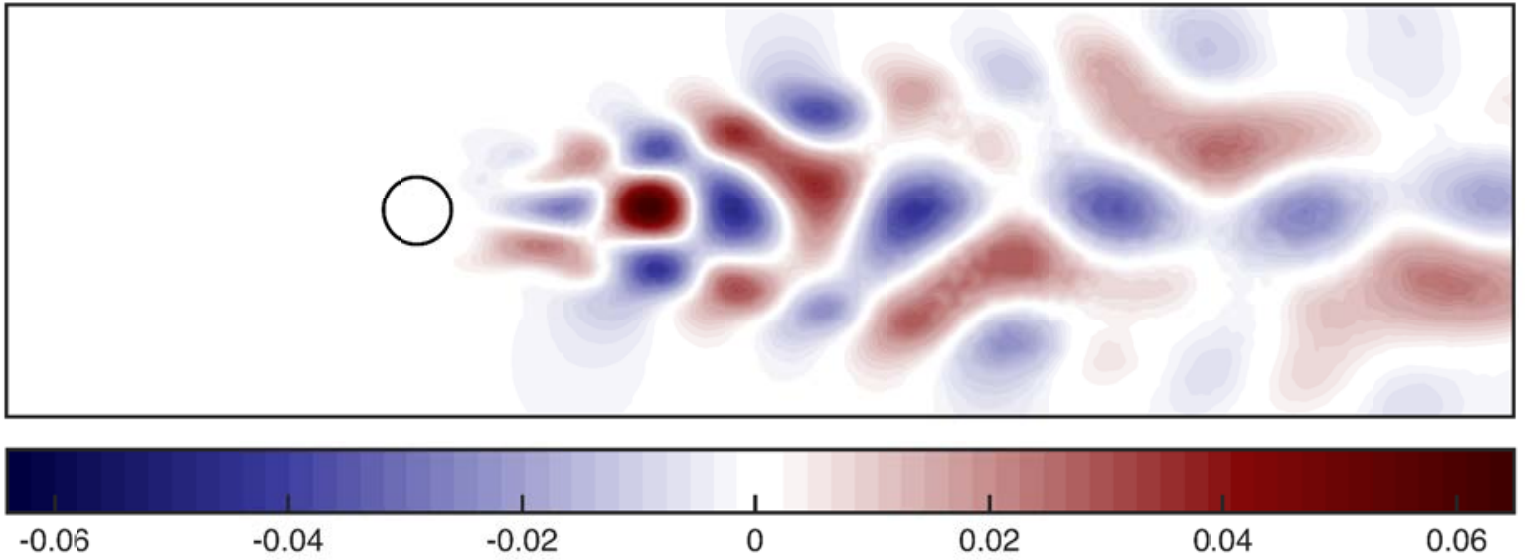}
\end{center}
\end{subfigure}
\begin{subfigure}{.48\textwidth}
\begin{center}
\includegraphics[width=2.4in]{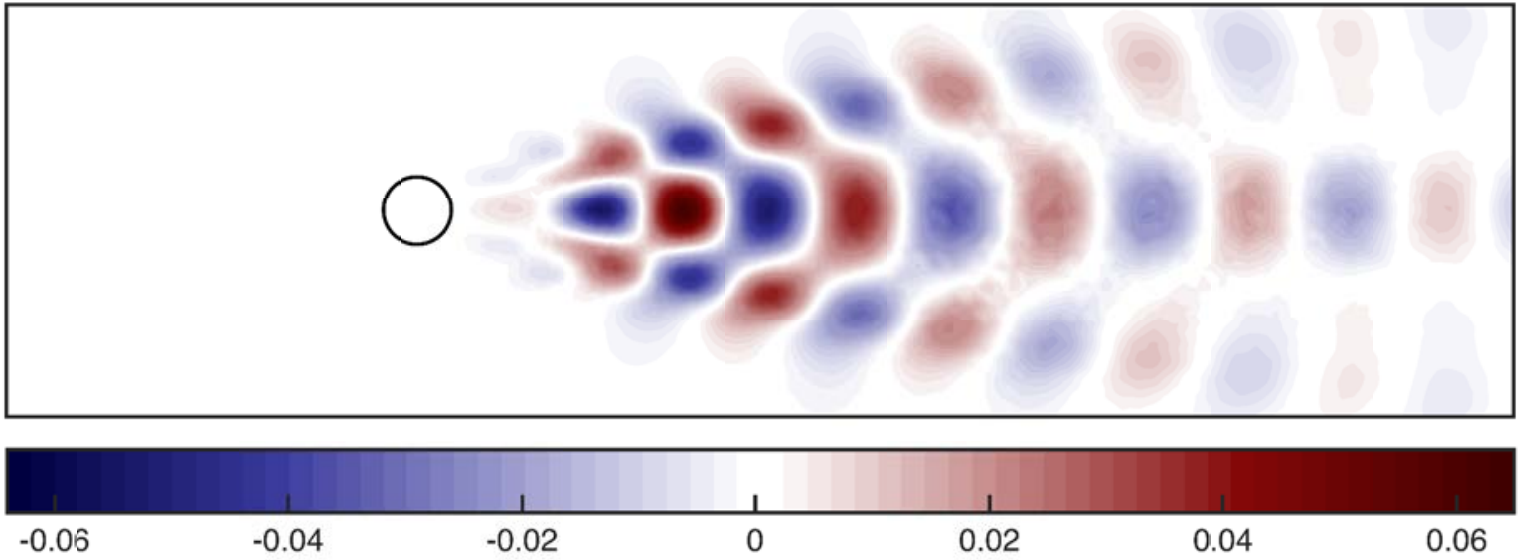}
\end{center}
\end{subfigure}

\vspace{1pc}

\caption{\textbf{Left singular vectors of \texttt{StreamVel} via [HMT11].}
(Sparse maps, \textbf{approximation rank $r = 5$}, storage budget $T = 48 (m+n)$.)
The columns of the matrix \texttt{StreamVel} describe the fluctuations
of the streamwise velocity field about its mean value as a function of time.
From top to bottom, the panels show the first nine computed left singular
vectors of the matrix.
The \textbf{left-hand side} is computed using [HMT11],
while the \textbf{right-hand side} is computed from the exact flow field.
The heatmap indicates the magnitude of the fluctuation.
See \cref{sec:flow-field}.}
\label{fig:flow-field-svec-hmt5}
\end{center}
\end{figure}

\begin{figure}[htp!]

\begin{center}
\begin{subfigure}{.48\textwidth}
\begin{center}
\includegraphics[width=2.4in]{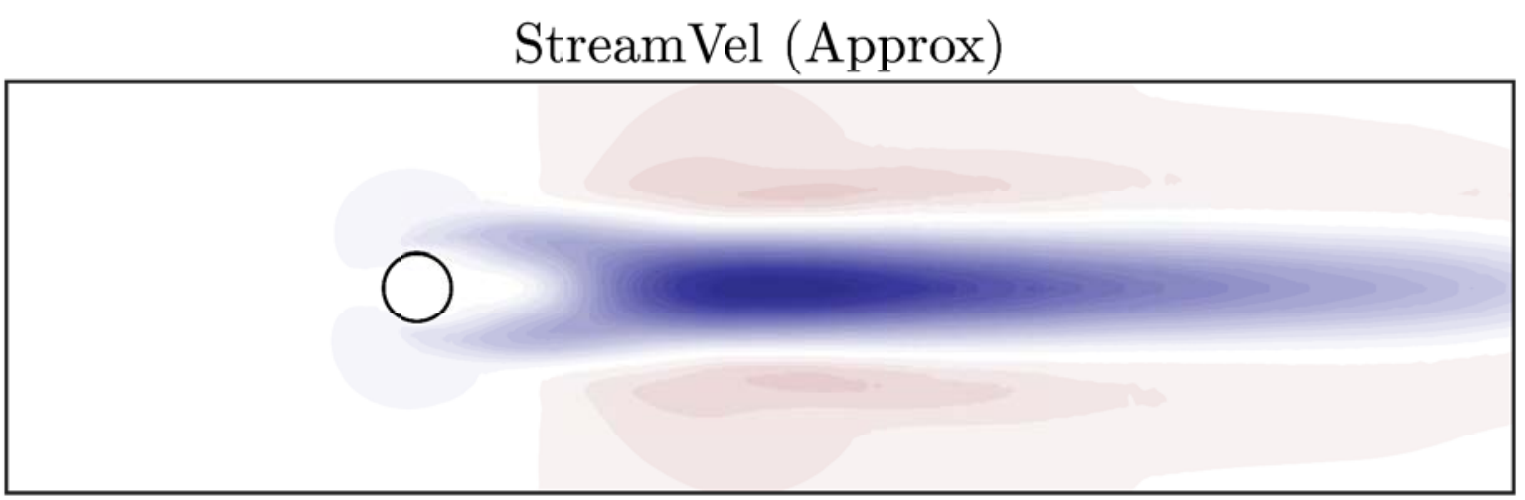}
\end{center}
\end{subfigure}
\begin{subfigure}{.48\textwidth}
\begin{center}
\includegraphics[width=2.4in]{figures/dns-svec/Fluctuation_exact_rank1.pdf}
\end{center}
\end{subfigure}

\vspace{1mm}

\begin{subfigure}{.48\textwidth}
\begin{center}
\includegraphics[width=2.4in]{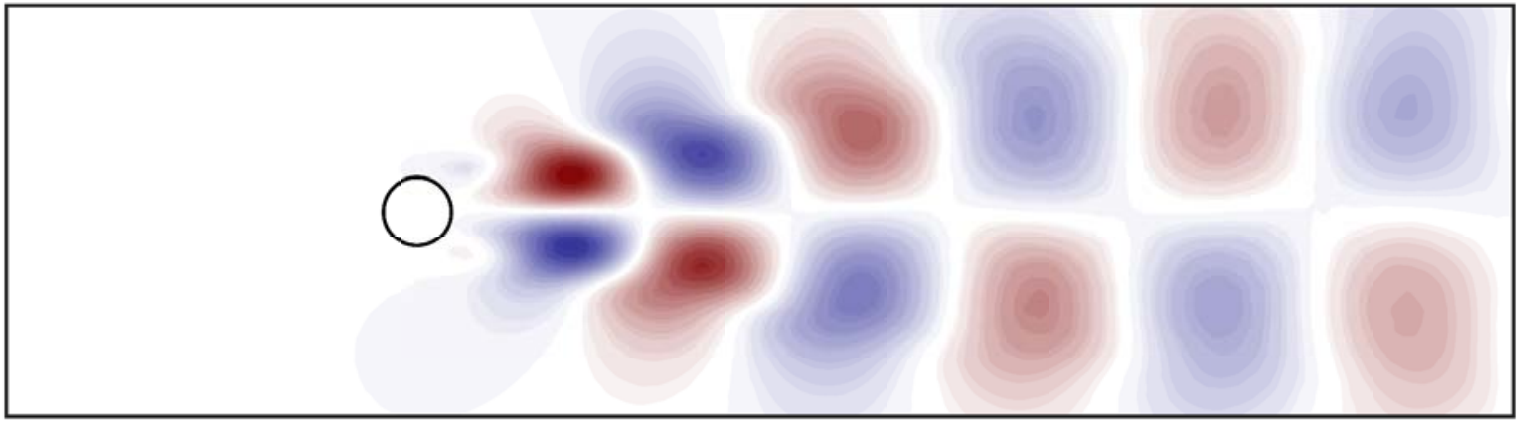}
\end{center}
\end{subfigure}
\begin{subfigure}{.48\textwidth}
\begin{center}
\includegraphics[width=2.4in]{figures/dns-svec/Fluctuation_exact_rank2.pdf}
\end{center}
\end{subfigure}

\vspace{1mm}

\begin{subfigure}{.48\textwidth}
\begin{center}
\includegraphics[width=2.4in]{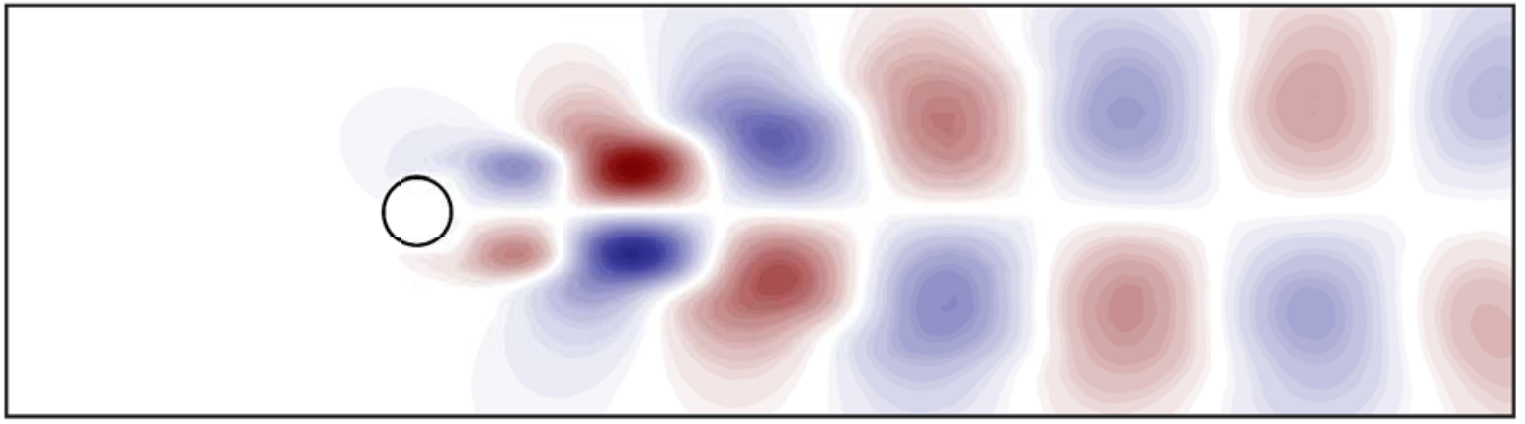}
\end{center}
\end{subfigure}
\begin{subfigure}{.48\textwidth}
\begin{center}
\includegraphics[width=2.4in]{figures/dns-svec/Fluctuation_exact_rank3.pdf}
\end{center}
\end{subfigure}

\vspace{1mm}

\begin{subfigure}{.48\textwidth}
\begin{center}
\includegraphics[width=2.4in]{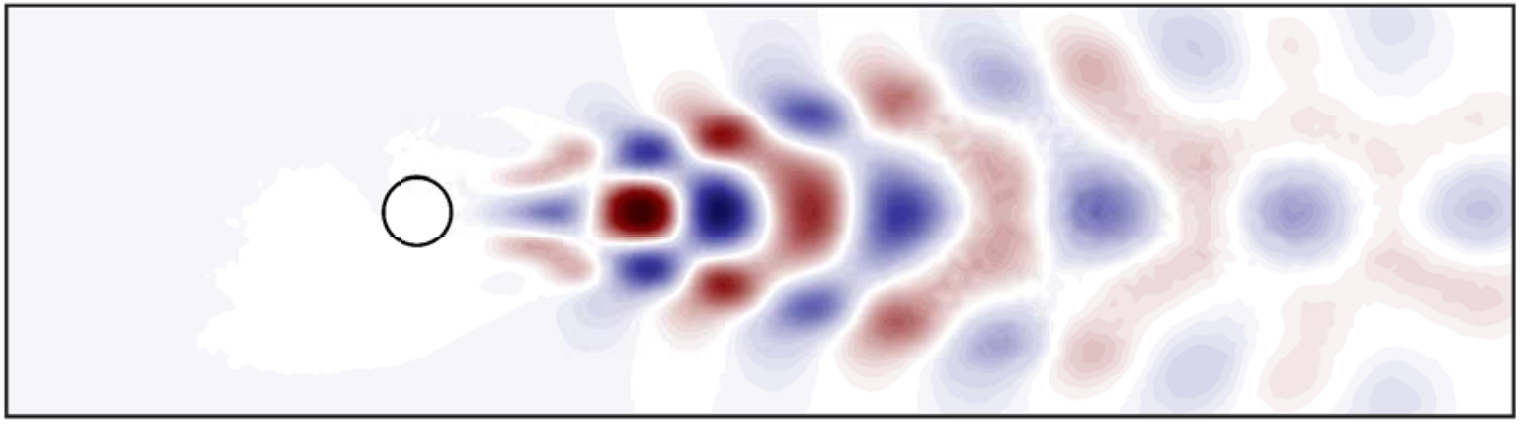}
\end{center}
\end{subfigure}
\begin{subfigure}{.48\textwidth}
\begin{center}
\includegraphics[width=2.4in]{figures/dns-svec/Fluctuation_exact_rank4.pdf}
\end{center}
\end{subfigure}

\vspace{1mm}

\begin{subfigure}{.48\textwidth}
\begin{center}
\includegraphics[width=2.4in]{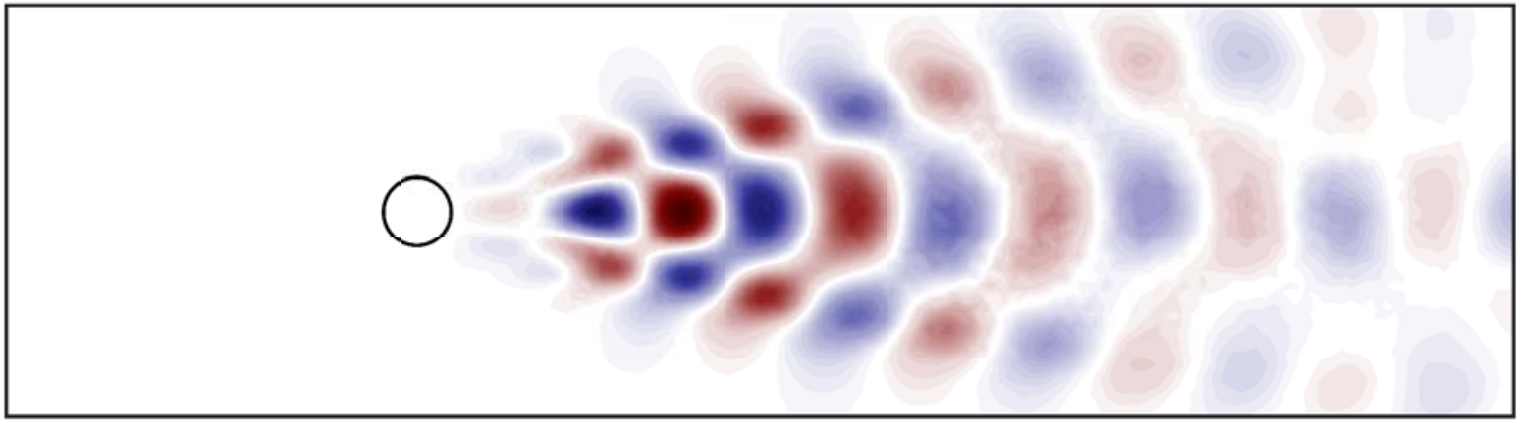}
\end{center}
\end{subfigure}
\begin{subfigure}{.48\textwidth}
\begin{center}
\includegraphics[width=2.4in]{figures/dns-svec/Fluctuation_exact_rank5.pdf}
\end{center}
\end{subfigure}

\vspace{1mm}

\begin{subfigure}{.48\textwidth}
\begin{center}
\includegraphics[width=2.4in]{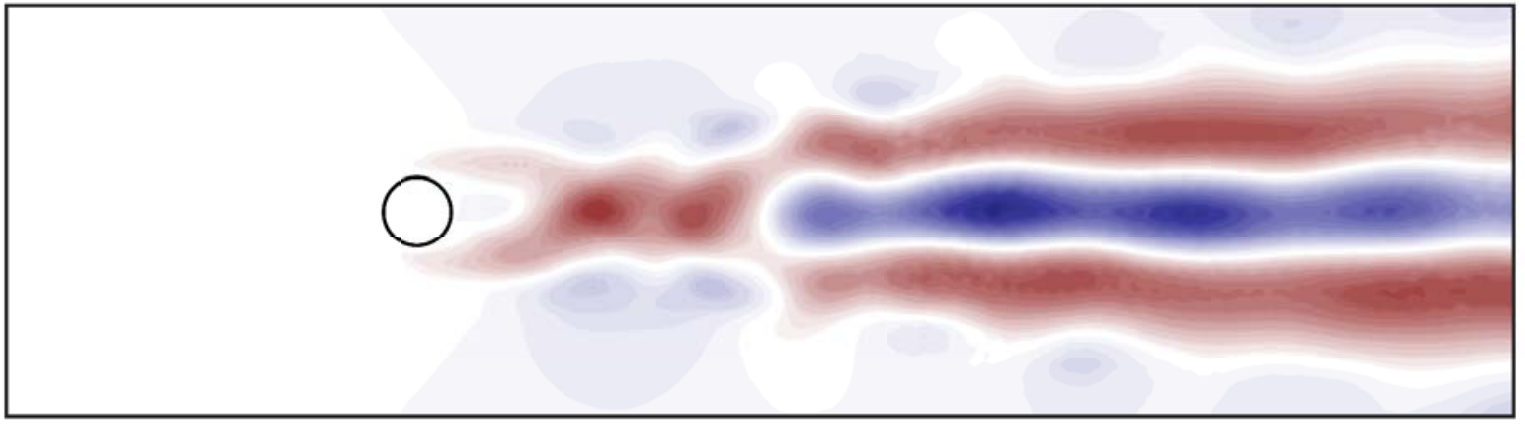}
\end{center}
\end{subfigure}
\begin{subfigure}{.48\textwidth}
\begin{center}
\includegraphics[width=2.4in]{figures/dns-svec/Fluctuation_exact_rank6.pdf}
\end{center}
\end{subfigure}

\vspace{1mm}

\begin{subfigure}{.48\textwidth}
\begin{center}
\includegraphics[width=2.4in]{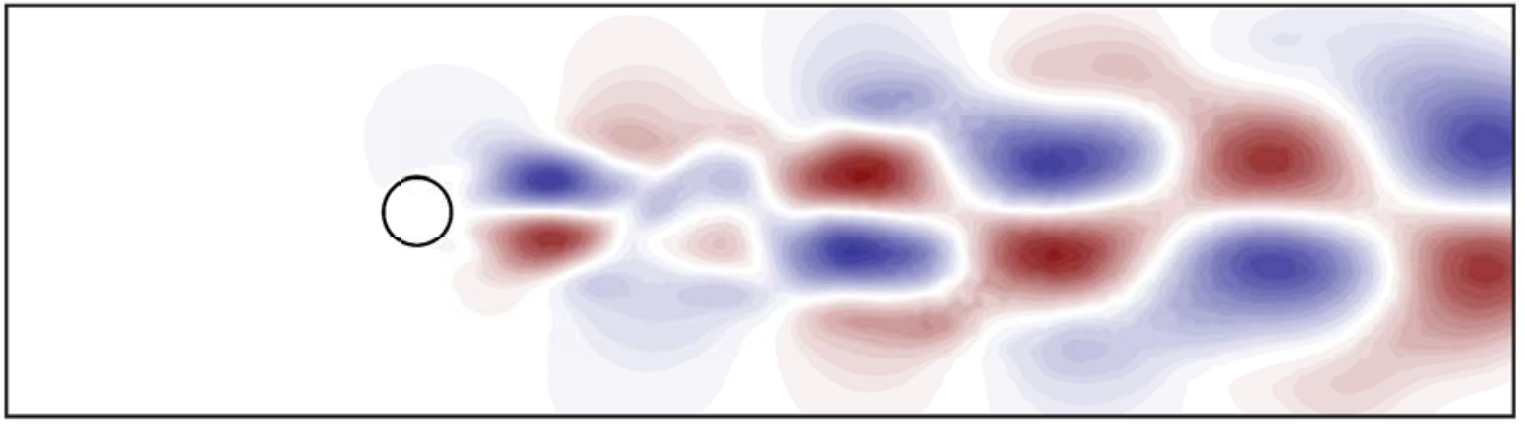}
\end{center}
\end{subfigure}
\begin{subfigure}{.48\textwidth}
\begin{center}
\includegraphics[width=2.4in]{figures/dns-svec/Fluctuation_exact_rank7.pdf}
\end{center}
\end{subfigure}

\vspace{1mm}

\begin{subfigure}{.48\textwidth}
\begin{center}
\includegraphics[width=2.4in]{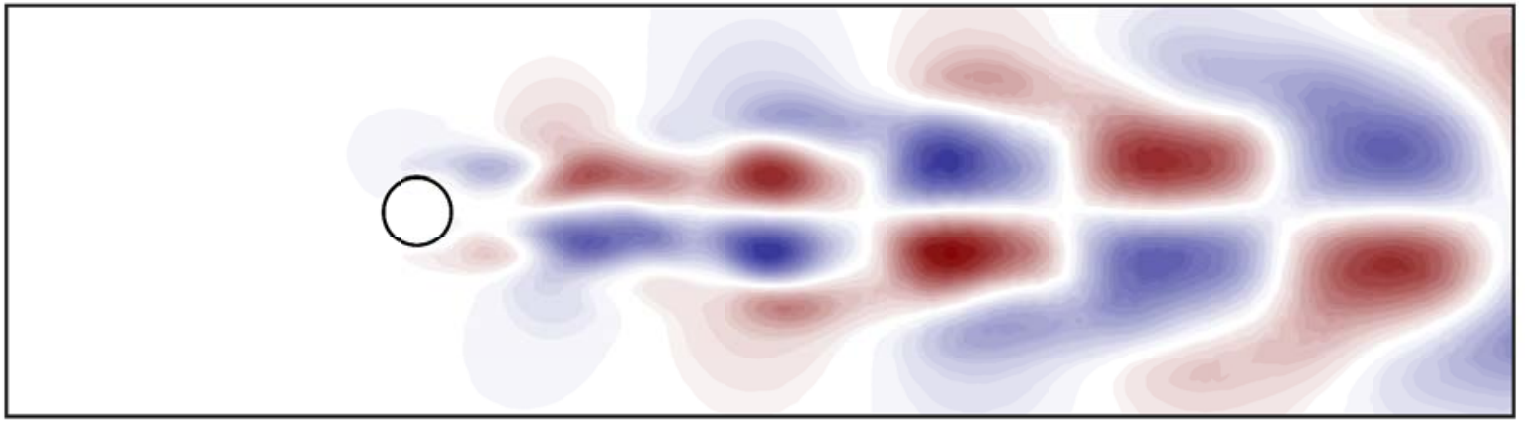}
\end{center}
\end{subfigure}
\begin{subfigure}{.48\textwidth}
\begin{center}
\includegraphics[width=2.4in]{figures/dns-svec/Fluctuation_exact_rank8.pdf}
\end{center}
\end{subfigure}

\vspace{1mm}

\begin{subfigure}{.48\textwidth}
\begin{center}
\includegraphics[width=2.4in]{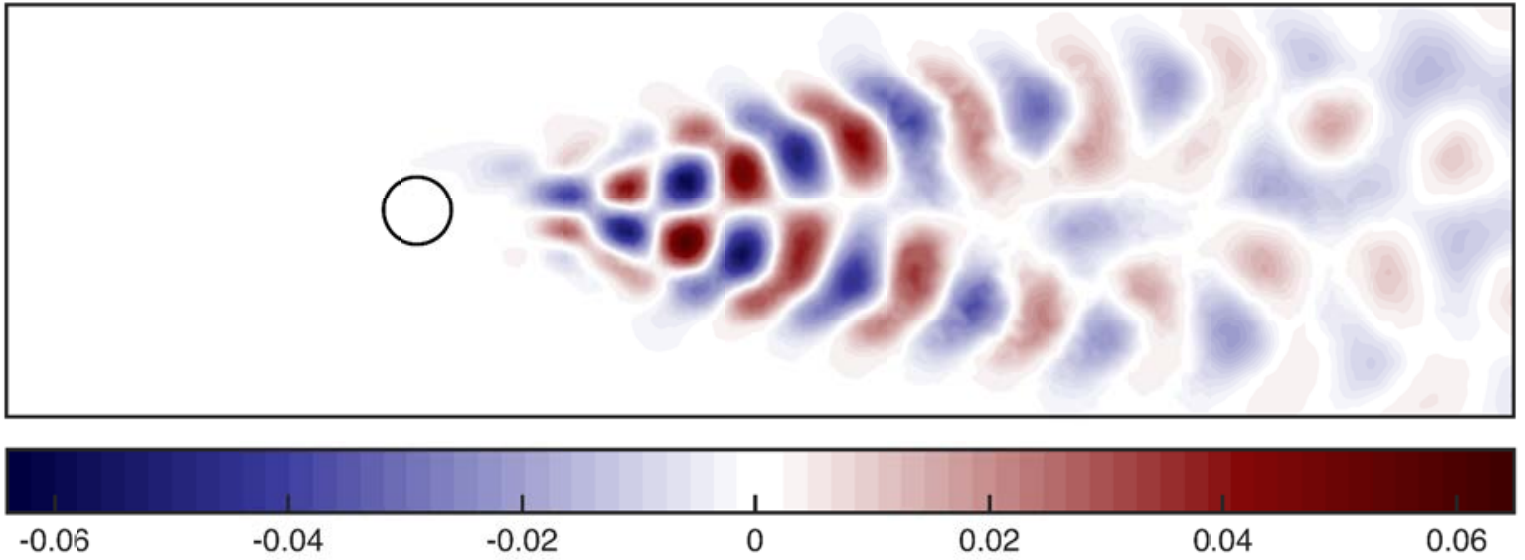}
\end{center}
\end{subfigure}
\begin{subfigure}{.48\textwidth}
\begin{center}
\includegraphics[width=2.4in]{figures/dns-svec/Fluctuation_exact_rank9.pdf}
\end{center}
\end{subfigure}

\vspace{1pc}

\caption{\textbf{Left singular vectors of \texttt{StreamVel} via [HMT11].}
(Sparse maps, \textbf{approximation rank $r = 10$}, storage budget $T = 48 (m+n)$.)
The columns of the matrix \texttt{StreamVel} describe the fluctuations
of the streamwise velocity field about its mean value as a function of time.
From top to bottom, the panels show the first nine computed left singular
vectors of the matrix.
The \textbf{left-hand side} is computed using [HMT11],
while the \textbf{right-hand side} is computed from the exact flow field.
The heatmap indicates the magnitude of the fluctuation.
See \cref{sec:flow-field}.}
\label{fig:flow-field-svec-hmt}
\end{center}
\end{figure}

\begin{figure}[htp!]

\begin{center}
\begin{subfigure}{.48\textwidth}
\begin{center}
\includegraphics[width=2.4in]{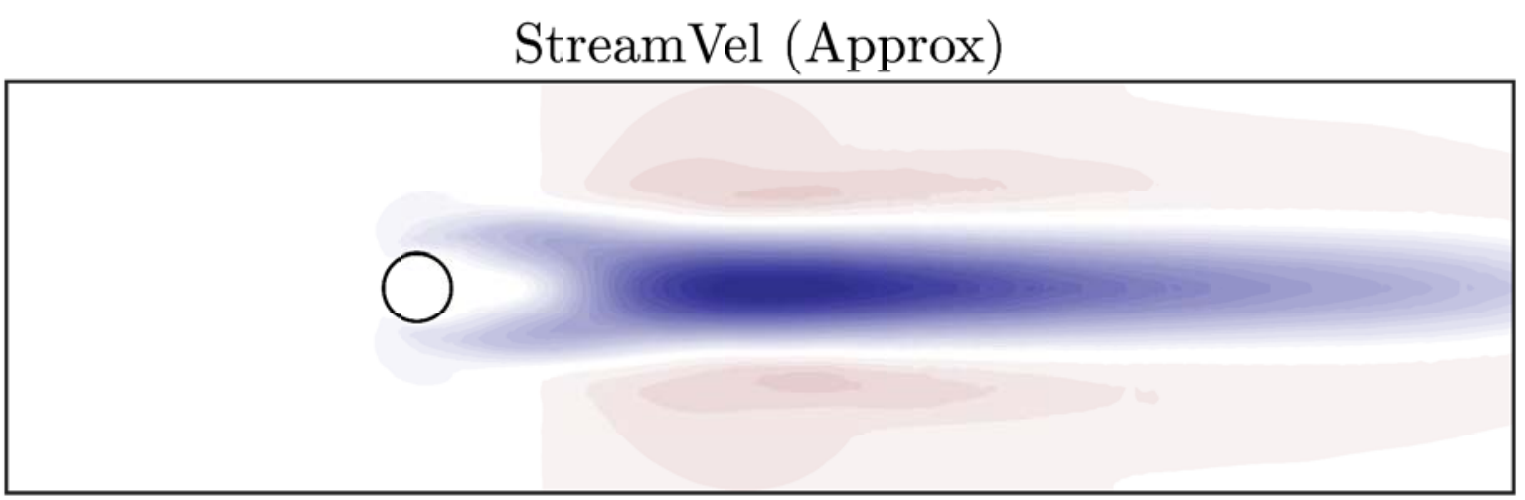}
\end{center}
\end{subfigure}
\begin{subfigure}{.48\textwidth}
\begin{center}
\includegraphics[width=2.4in]{figures/dns-svec-rank5/Fluctuation_exact_rank1.pdf}
\end{center}
\end{subfigure}

\vspace{1mm}

\begin{subfigure}{.48\textwidth}
\begin{center}
\includegraphics[width=2.4in]{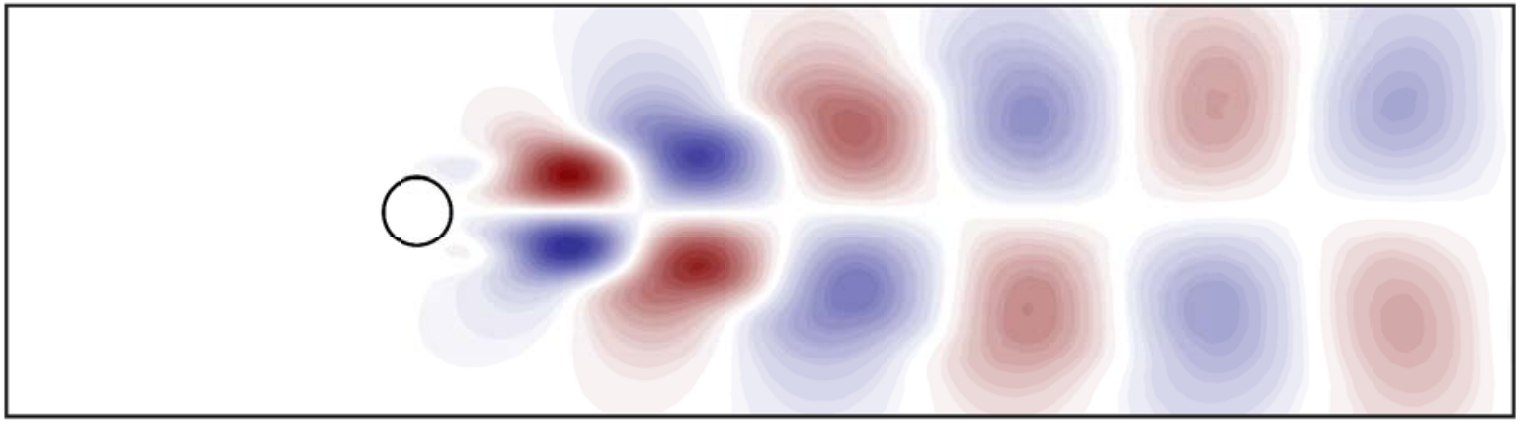}
\end{center}
\end{subfigure}
\begin{subfigure}{.48\textwidth}
\begin{center}
\includegraphics[width=2.4in]{figures/dns-svec-rank5/Fluctuation_exact_rank2.pdf}
\end{center}
\end{subfigure}

\vspace{1mm}

\begin{subfigure}{.48\textwidth}
\begin{center}
\includegraphics[width=2.4in]{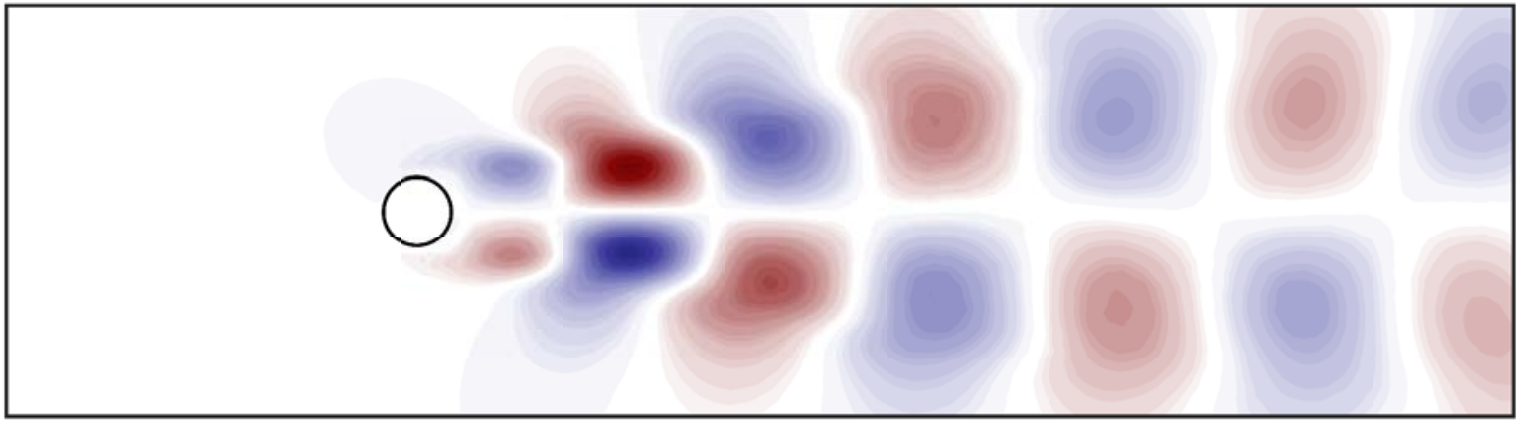}
\end{center}
\end{subfigure}
\begin{subfigure}{.48\textwidth}
\begin{center}
\includegraphics[width=2.4in]{figures/dns-svec-rank5/Fluctuation_exact_rank3.pdf}
\end{center}
\end{subfigure}

\vspace{1mm}

\begin{subfigure}{.48\textwidth}
\begin{center}
\includegraphics[width=2.4in]{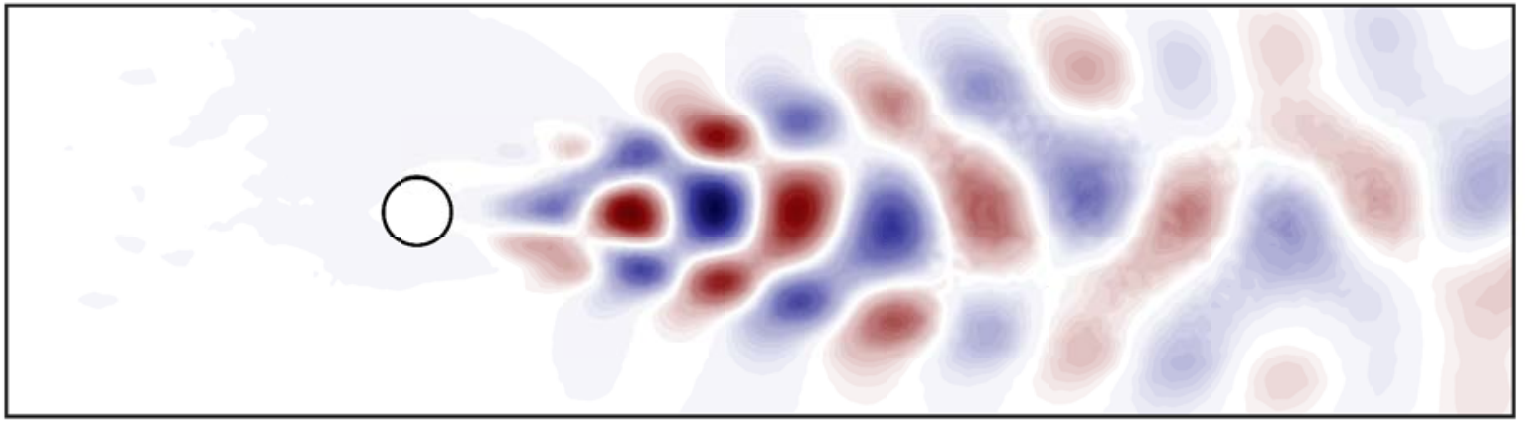}
\end{center}
\end{subfigure}
\begin{subfigure}{.48\textwidth}
\begin{center}
\includegraphics[width=2.4in]{figures/dns-svec-rank5/Fluctuation_exact_rank4.pdf}
\end{center}
\end{subfigure}

\vspace{1mm}

\begin{subfigure}{.48\textwidth}
\begin{center}
\includegraphics[width=2.4in]{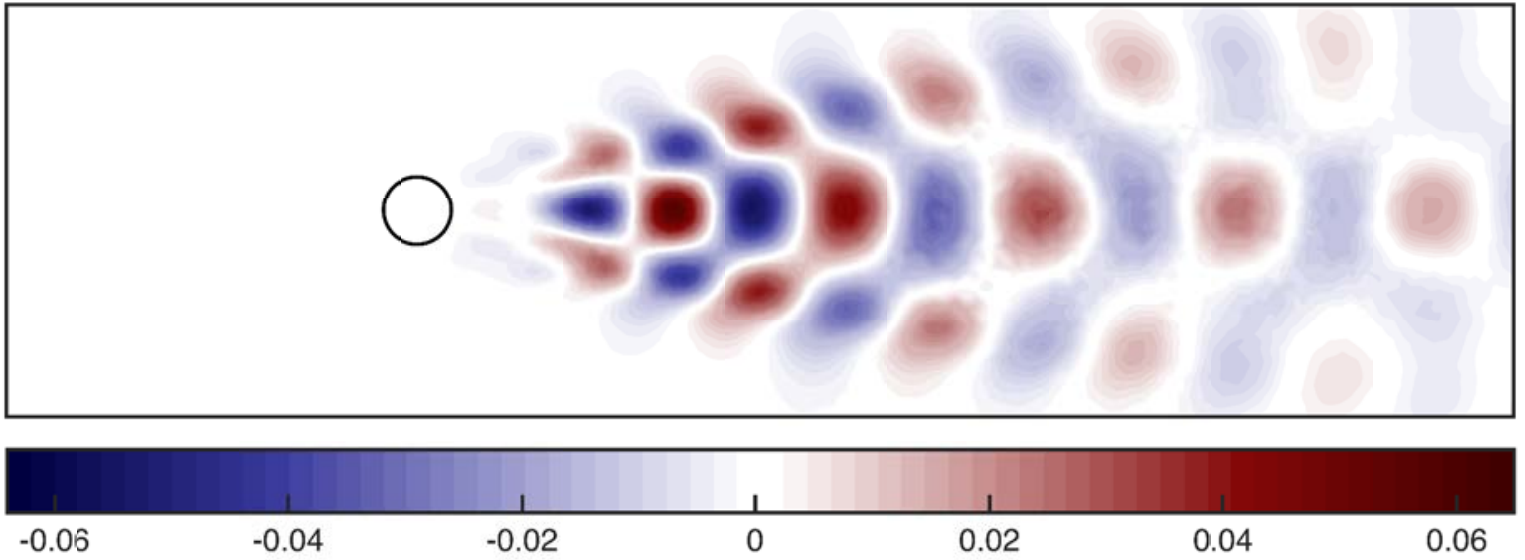}
\end{center}
\end{subfigure}
\begin{subfigure}{.48\textwidth}
\begin{center}
\includegraphics[width=2.4in]{figures/dns-svec-rank5/Fluctuation_exact_rank5.pdf}
\end{center}
\end{subfigure}

\vspace{1pc}

\caption{\textbf{Left singular vectors of \texttt{StreamVel} via [Upa16].}
(Sparse maps, \textbf{approximation rank $r = 5$}, storage budget $T = 48 (m+n)$.)
The columns of the matrix \texttt{StreamVel} describe the fluctuations
of the streamwise velocity field about its mean value as a function of time.
From top to bottom, the panels show the first nine computed left singular
vectors of the matrix.
The \textbf{left-hand side} is computed using [Upa16],
while the \textbf{right-hand side} is computed from the exact flow field.
The heatmap indicates the magnitude of the fluctuation.
See \cref{sec:flow-field}.}
\label{fig:flow-field-svec-upa5}
\end{center}
\end{figure}

\begin{figure}[htp!]

\begin{center}
\begin{subfigure}{.48\textwidth}
\begin{center}
\includegraphics[width=2.4in]{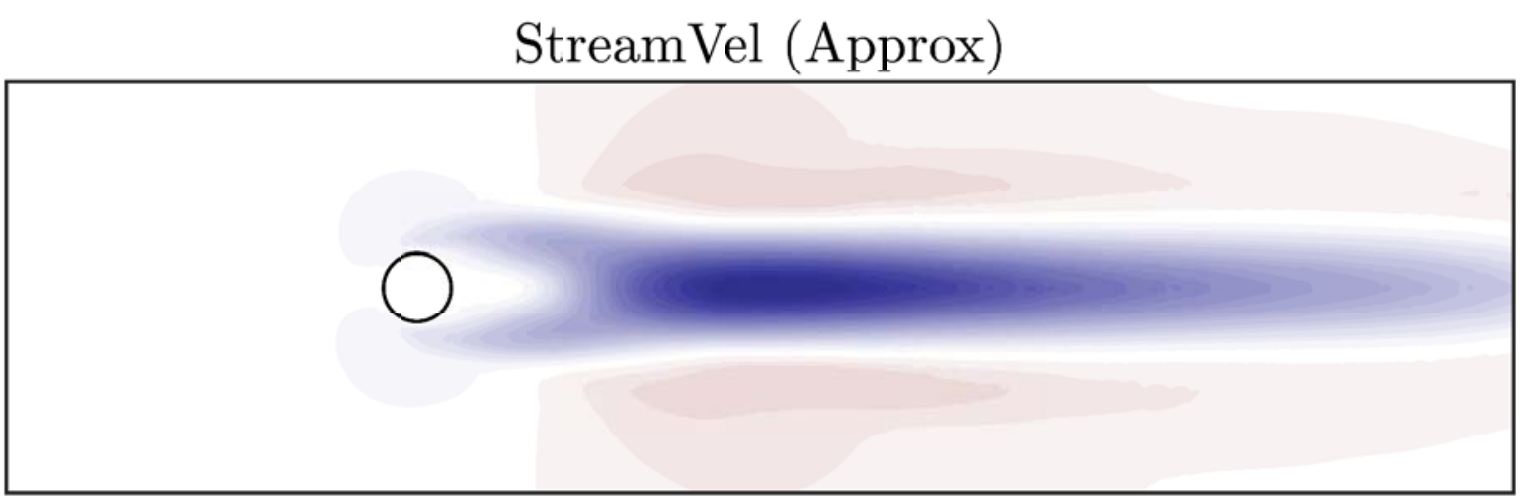}
\end{center}
\end{subfigure}
\begin{subfigure}{.48\textwidth}
\begin{center}
\includegraphics[width=2.4in]{figures/dns-svec/Fluctuation_exact_rank1.pdf}
\end{center}
\end{subfigure}

\vspace{1mm}

\begin{subfigure}{.48\textwidth}
\begin{center}
\includegraphics[width=2.4in]{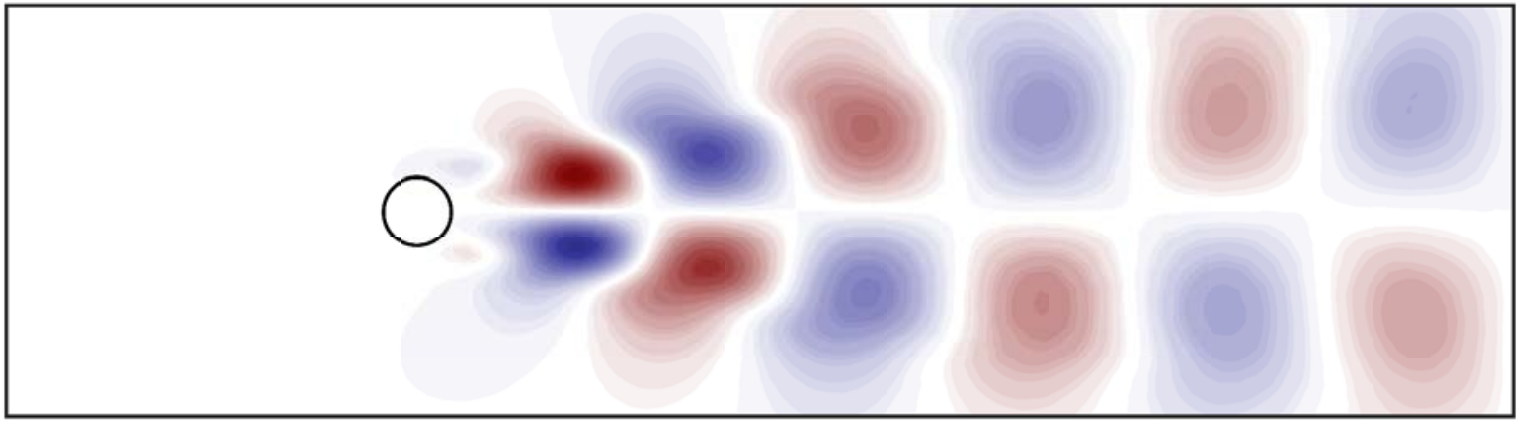}
\end{center}
\end{subfigure}
\begin{subfigure}{.48\textwidth}
\begin{center}
\includegraphics[width=2.4in]{figures/dns-svec/Fluctuation_exact_rank2.pdf}
\end{center}
\end{subfigure}

\vspace{1mm}

\begin{subfigure}{.48\textwidth}
\begin{center}
\includegraphics[width=2.4in]{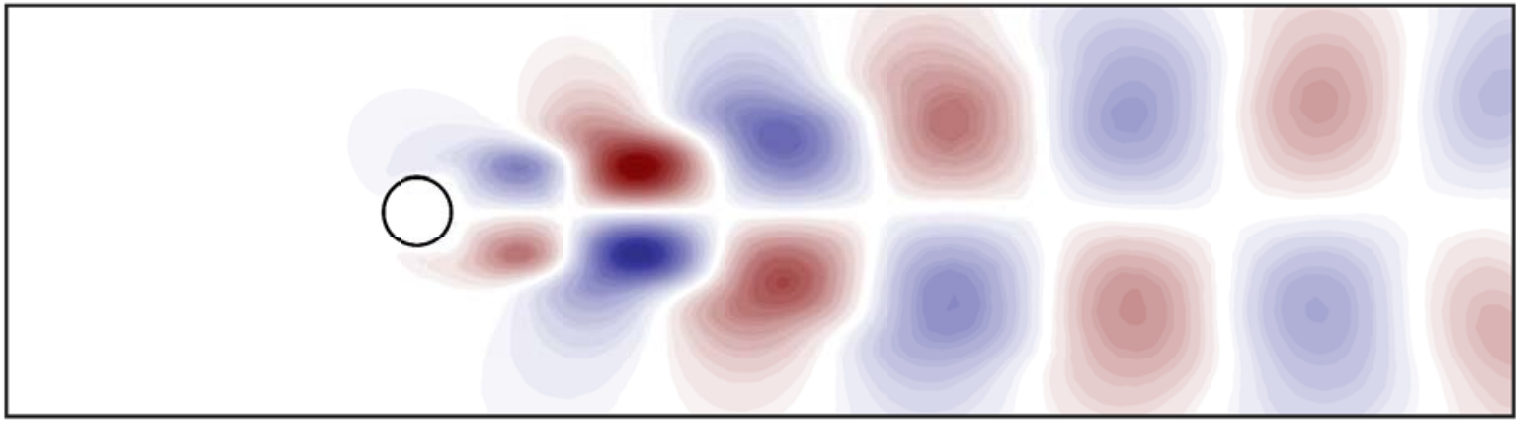}
\end{center}
\end{subfigure}
\begin{subfigure}{.48\textwidth}
\begin{center}
\includegraphics[width=2.4in]{figures/dns-svec/Fluctuation_exact_rank3.pdf}
\end{center}
\end{subfigure}

\vspace{1mm}

\begin{subfigure}{.48\textwidth}
\begin{center}
\includegraphics[width=2.4in]{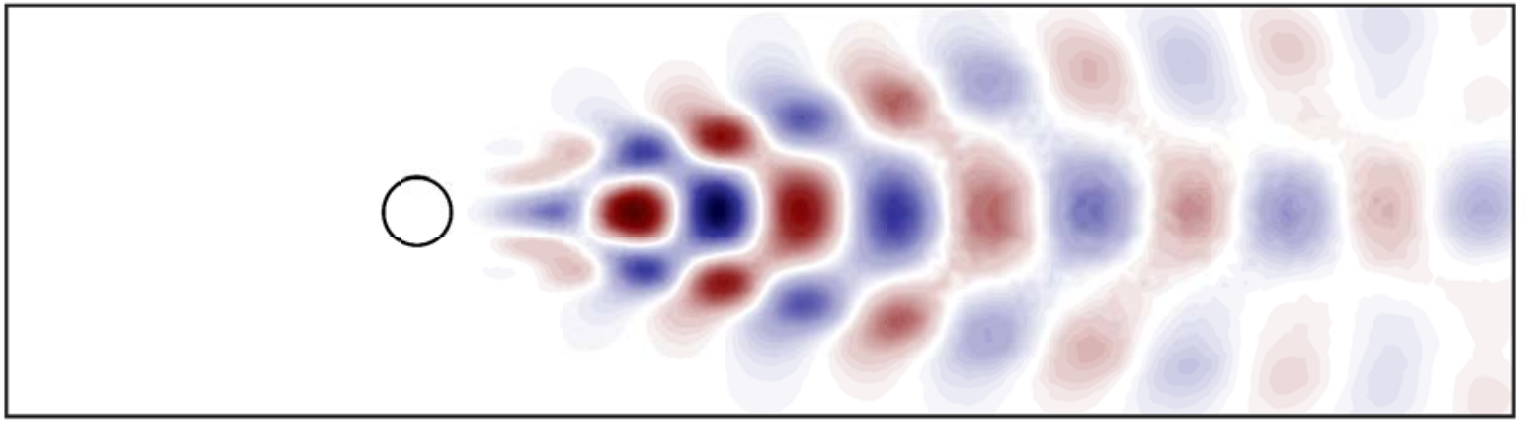}
\end{center}
\end{subfigure}
\begin{subfigure}{.48\textwidth}
\begin{center}
\includegraphics[width=2.4in]{figures/dns-svec/Fluctuation_exact_rank4.pdf}
\end{center}
\end{subfigure}

\vspace{1mm}

\begin{subfigure}{.48\textwidth}
\begin{center}
\includegraphics[width=2.4in]{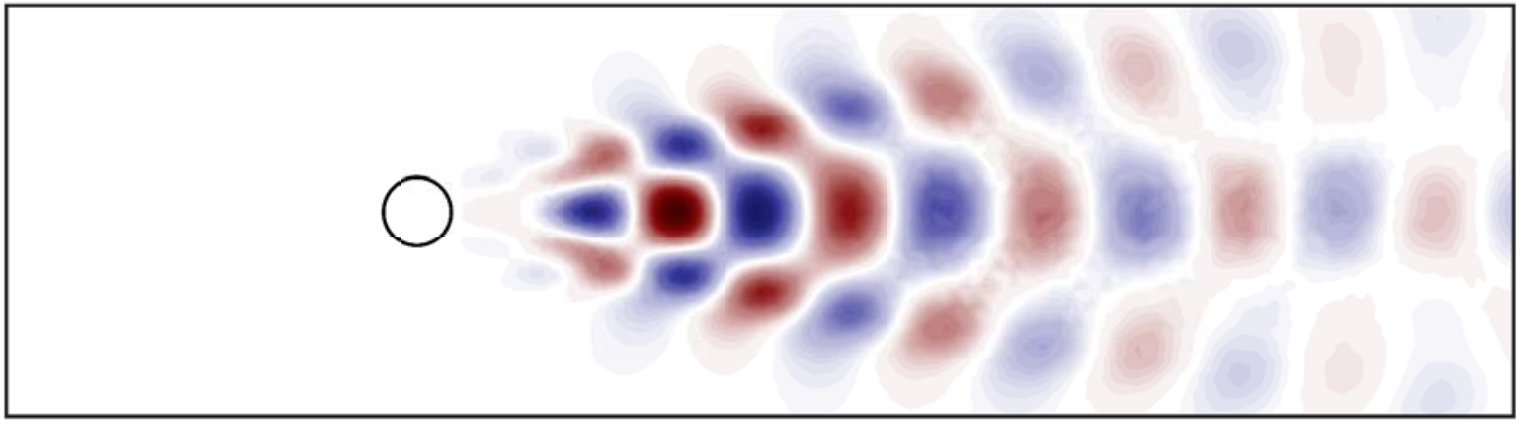}
\end{center}
\end{subfigure}
\begin{subfigure}{.48\textwidth}
\begin{center}
\includegraphics[width=2.4in]{figures/dns-svec/Fluctuation_exact_rank5.pdf}
\end{center}
\end{subfigure}

\vspace{1mm}

\begin{subfigure}{.48\textwidth}
\begin{center}
\includegraphics[width=2.4in]{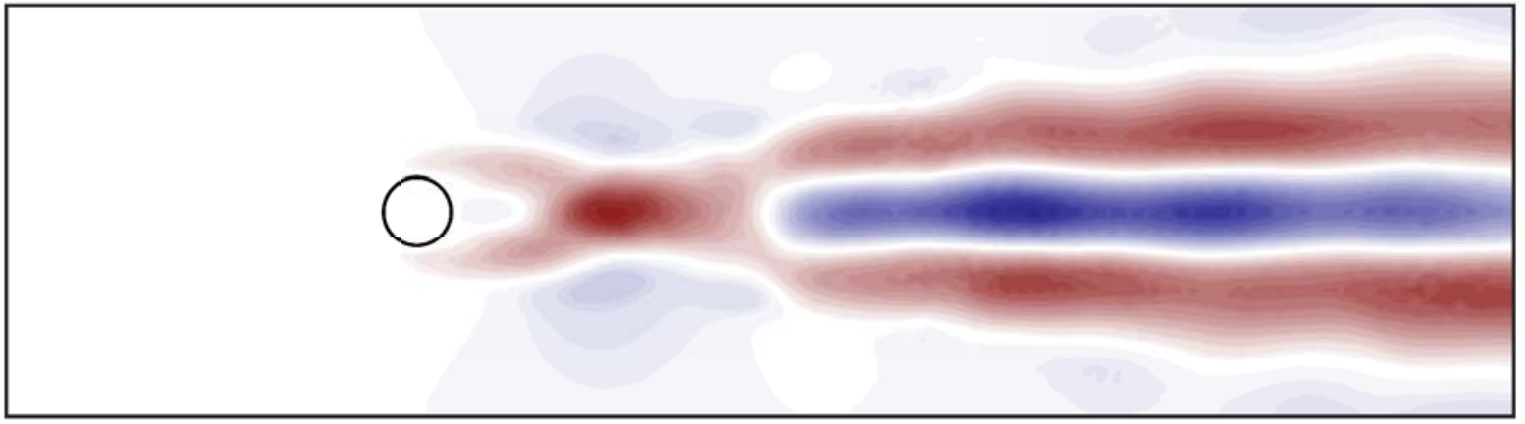}
\end{center}
\end{subfigure}
\begin{subfigure}{.48\textwidth}
\begin{center}
\includegraphics[width=2.4in]{figures/dns-svec/Fluctuation_exact_rank6.pdf}
\end{center}
\end{subfigure}

\vspace{1mm}

\begin{subfigure}{.48\textwidth}
\begin{center}
\includegraphics[width=2.4in]{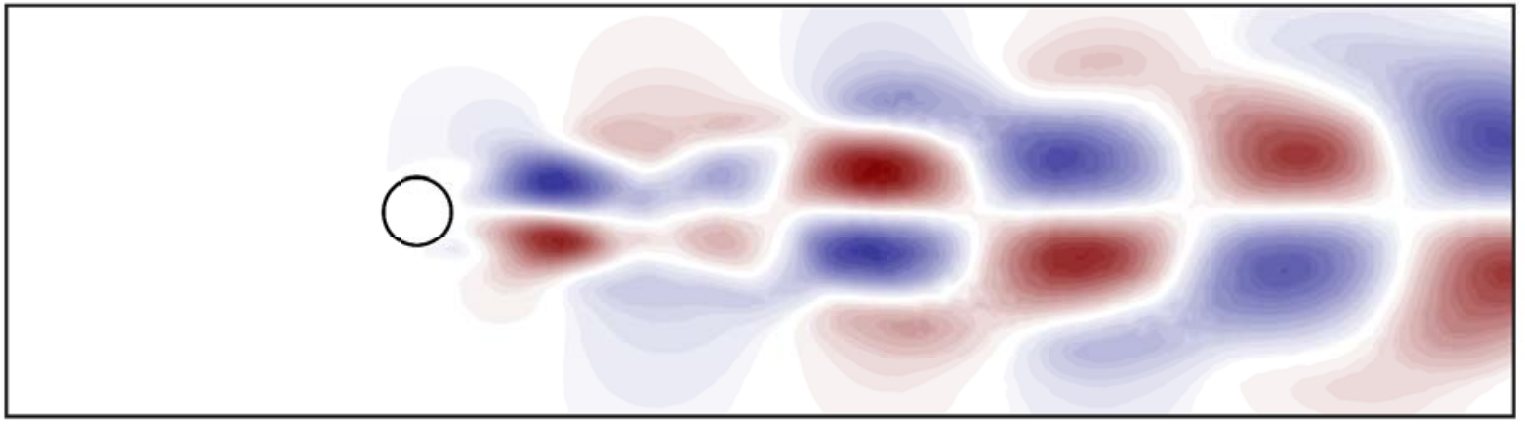}
\end{center}
\end{subfigure}
\begin{subfigure}{.48\textwidth}
\begin{center}
\includegraphics[width=2.4in]{figures/dns-svec/Fluctuation_exact_rank7.pdf}
\end{center}
\end{subfigure}

\vspace{1mm}

\begin{subfigure}{.48\textwidth}
\begin{center}
\includegraphics[width=2.4in]{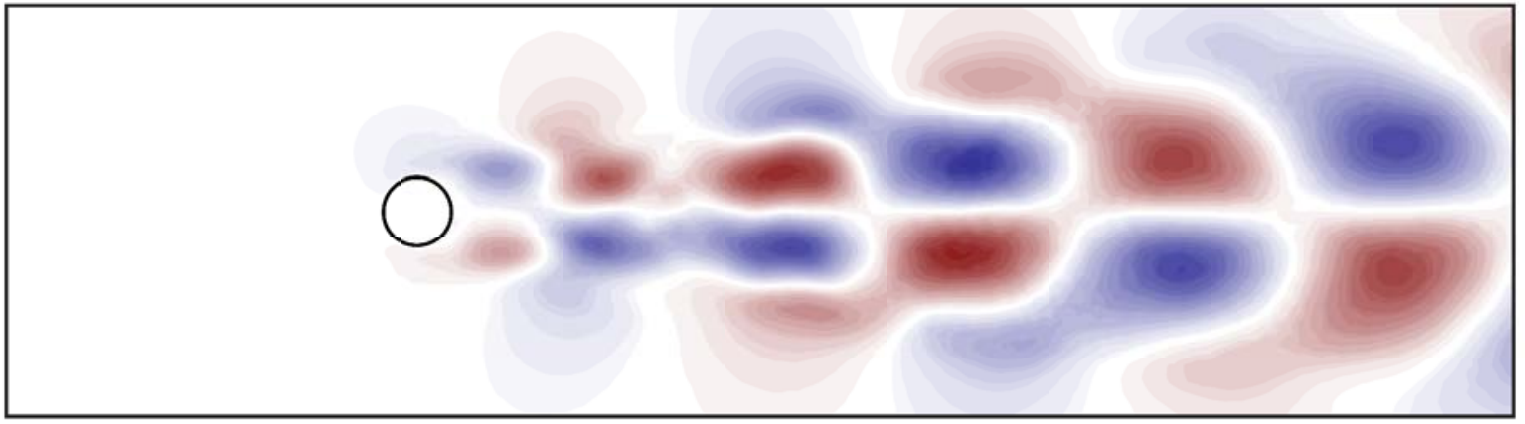}
\end{center}
\end{subfigure}
\begin{subfigure}{.48\textwidth}
\begin{center}
\includegraphics[width=2.4in]{figures/dns-svec/Fluctuation_exact_rank8.pdf}
\end{center}
\end{subfigure}

\vspace{1mm}

\begin{subfigure}{.48\textwidth}
\begin{center}
\includegraphics[width=2.4in]{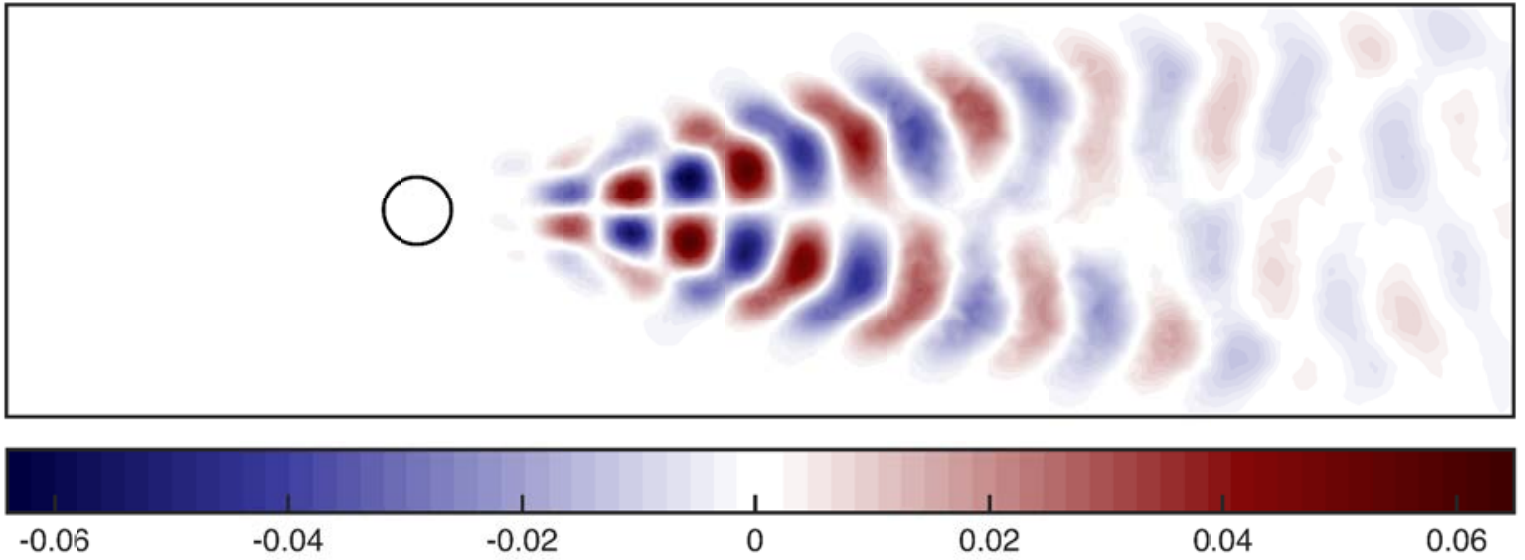}
\end{center}
\end{subfigure}
\begin{subfigure}{.48\textwidth}
\begin{center}
\includegraphics[width=2.4in]{figures/dns-svec/Fluctuation_exact_rank9.pdf}
\end{center}
\end{subfigure}

\vspace{1pc}

\caption{\textbf{Left singular vectors of \texttt{StreamVel} via [Upa16].}
(Sparse maps, \textbf{approximation rank $r = 10$}, storage budget $T = 48 (m+n)$.)
The columns of the matrix \texttt{StreamVel} describe the fluctuations
of the streamwise velocity field about its mean value as a function of time.
From top to bottom, the panels show the first nine computed left singular
vectors of the matrix.
The \textbf{left-hand side} is computed using [Upa16],
while the \textbf{right-hand side} is computed from the exact flow field.
The heatmap indicates the magnitude of the fluctuation.
See \cref{sec:flow-field}.}
\label{fig:flow-field-svec-upa}
\end{center}
\end{figure}

\begin{figure}[htp!]

\begin{center}
\begin{subfigure}{.48\textwidth}
\begin{center}
\includegraphics[width=2.4in]{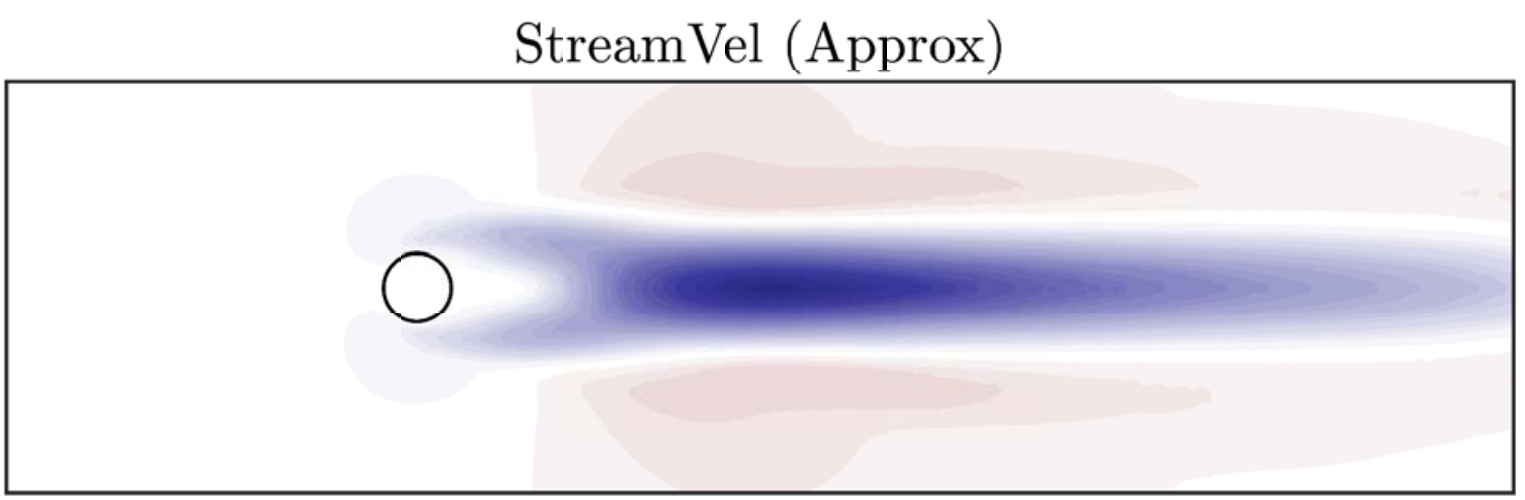}
\end{center}
\end{subfigure}
\begin{subfigure}{.48\textwidth}
\begin{center}
\includegraphics[width=2.4in]{figures/dns-svec/Fluctuation_exact_rank1.pdf}
\end{center}
\end{subfigure}

\vspace{1mm}

\begin{subfigure}{.48\textwidth}
\begin{center}
\includegraphics[width=2.4in]{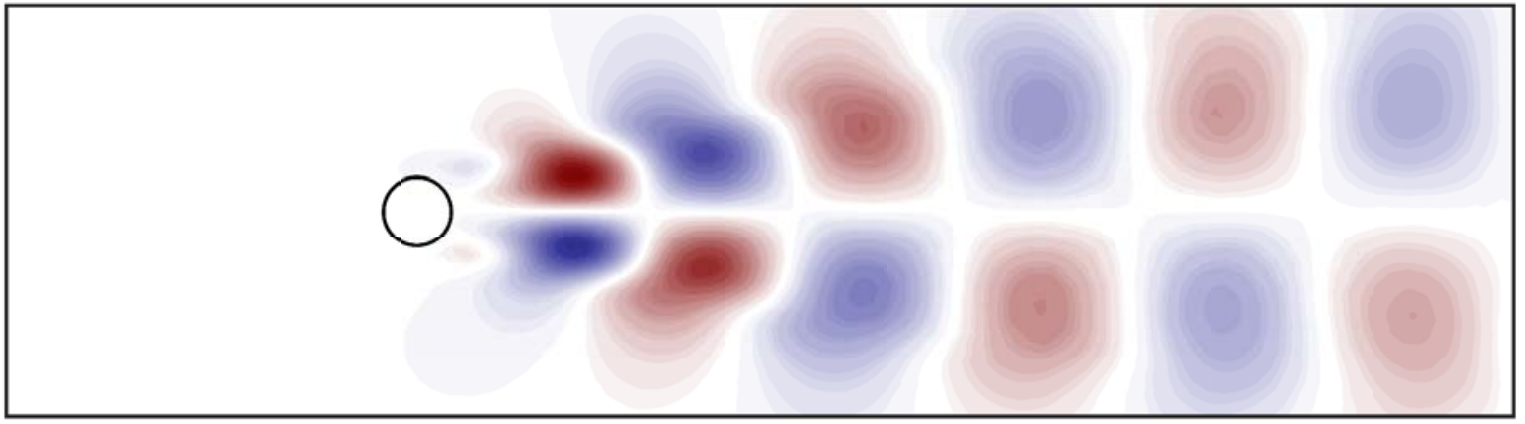}
\end{center}
\end{subfigure}
\begin{subfigure}{.48\textwidth}
\begin{center}
\includegraphics[width=2.4in]{figures/dns-svec/Fluctuation_exact_rank2.pdf}
\end{center}
\end{subfigure}

\vspace{1mm}

\begin{subfigure}{.48\textwidth}
\begin{center}
\includegraphics[width=2.4in]{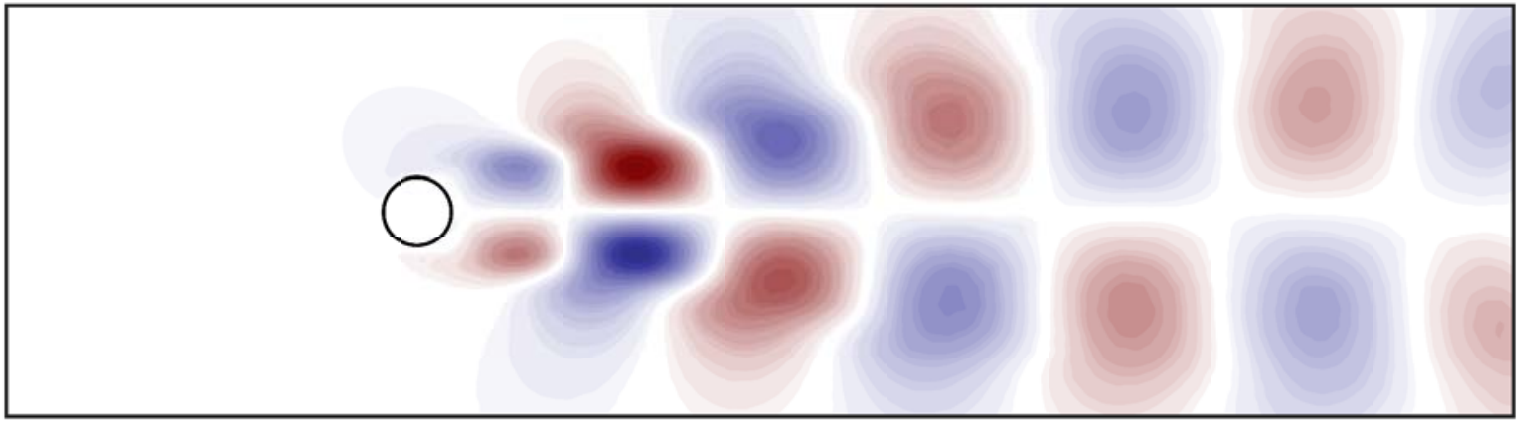}
\end{center}
\end{subfigure}
\begin{subfigure}{.48\textwidth}
\begin{center}
\includegraphics[width=2.4in]{figures/dns-svec/Fluctuation_exact_rank3.pdf}
\end{center}
\end{subfigure}

\vspace{1mm}

\begin{subfigure}{.48\textwidth}
\begin{center}
\includegraphics[width=2.4in]{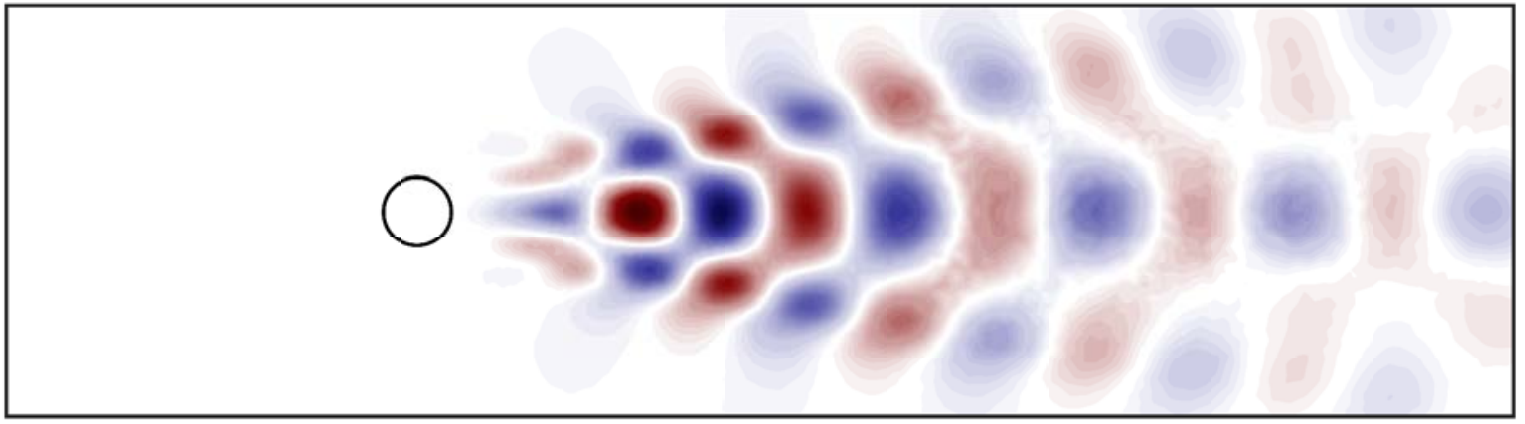}
\end{center}
\end{subfigure}
\begin{subfigure}{.48\textwidth}
\begin{center}
\includegraphics[width=2.4in]{figures/dns-svec/Fluctuation_exact_rank4.pdf}
\end{center}
\end{subfigure}

\vspace{1mm}

\begin{subfigure}{.48\textwidth}
\begin{center}
\includegraphics[width=2.4in]{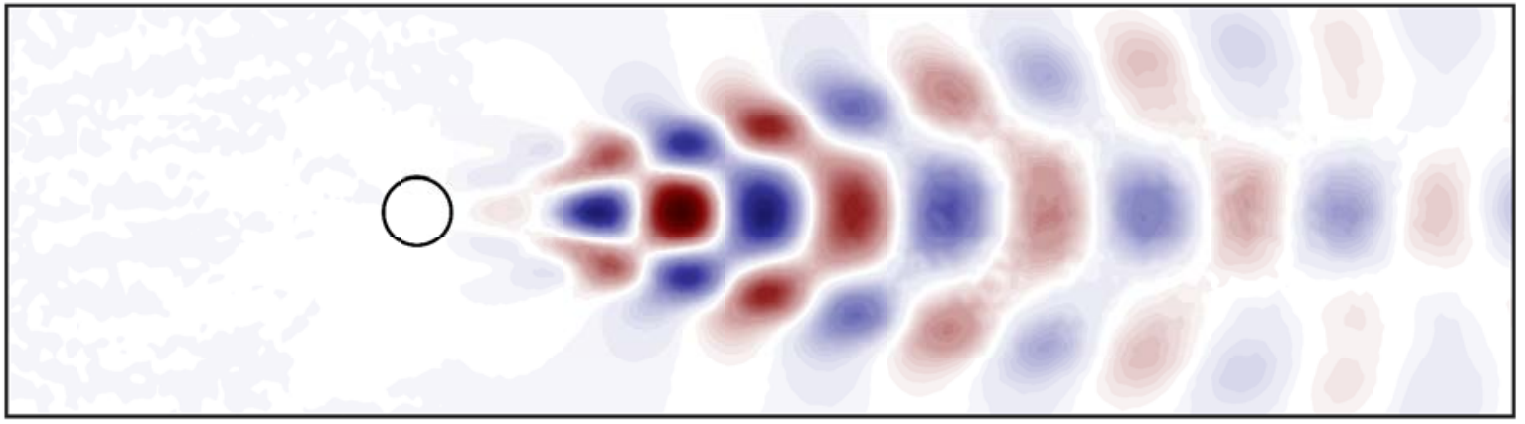}
\end{center}
\end{subfigure}
\begin{subfigure}{.48\textwidth}
\begin{center}
\includegraphics[width=2.4in]{figures/dns-svec/Fluctuation_exact_rank5.pdf}
\end{center}
\end{subfigure}

\vspace{1mm}

\begin{subfigure}{.48\textwidth}
\begin{center}
\includegraphics[width=2.4in]{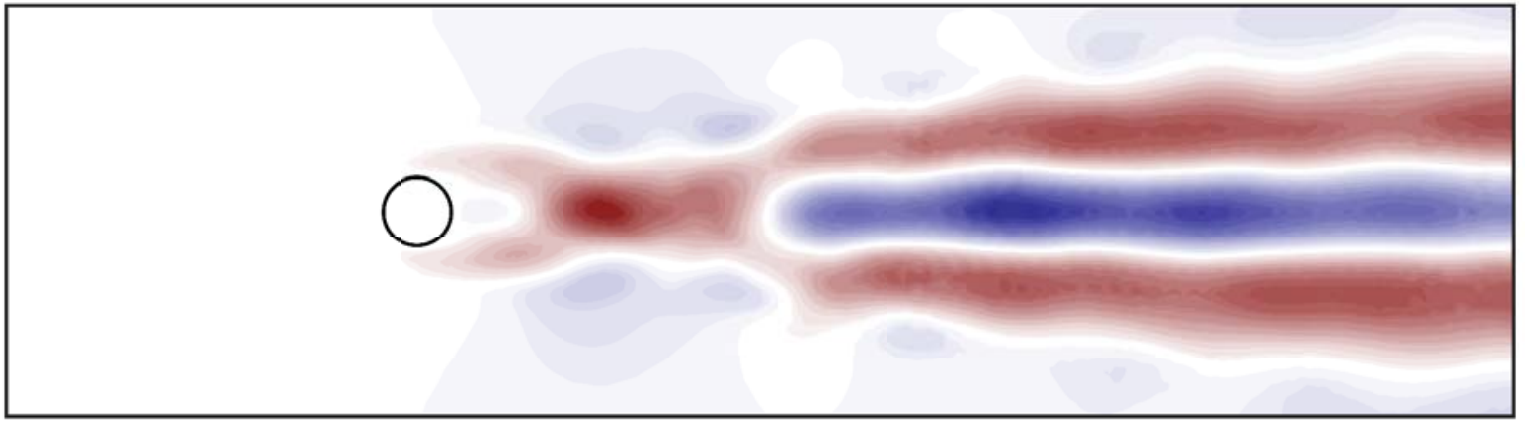}
\end{center}
\end{subfigure}
\begin{subfigure}{.48\textwidth}
\begin{center}
\includegraphics[width=2.4in]{figures/dns-svec/Fluctuation_exact_rank6.pdf}
\end{center}
\end{subfigure}

\vspace{1mm}

\begin{subfigure}{.48\textwidth}
\begin{center}
\includegraphics[width=2.4in]{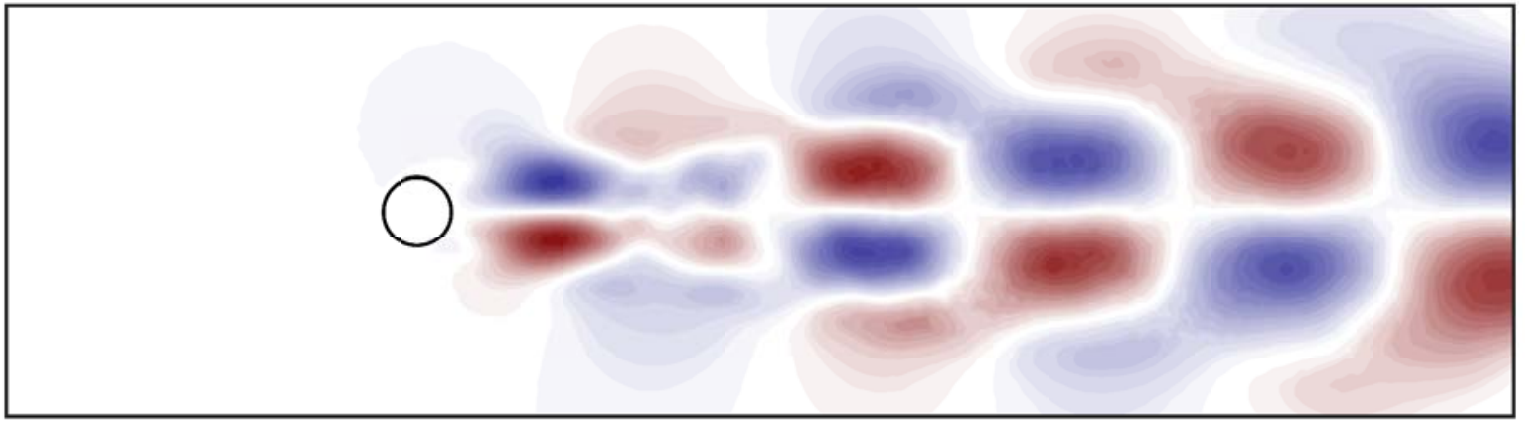}
\end{center}
\end{subfigure}
\begin{subfigure}{.48\textwidth}
\begin{center}
\includegraphics[width=2.4in]{figures/dns-svec/Fluctuation_exact_rank7.pdf}
\end{center}
\end{subfigure}

\vspace{1mm}

\begin{subfigure}{.48\textwidth}
\begin{center}
\includegraphics[width=2.4in]{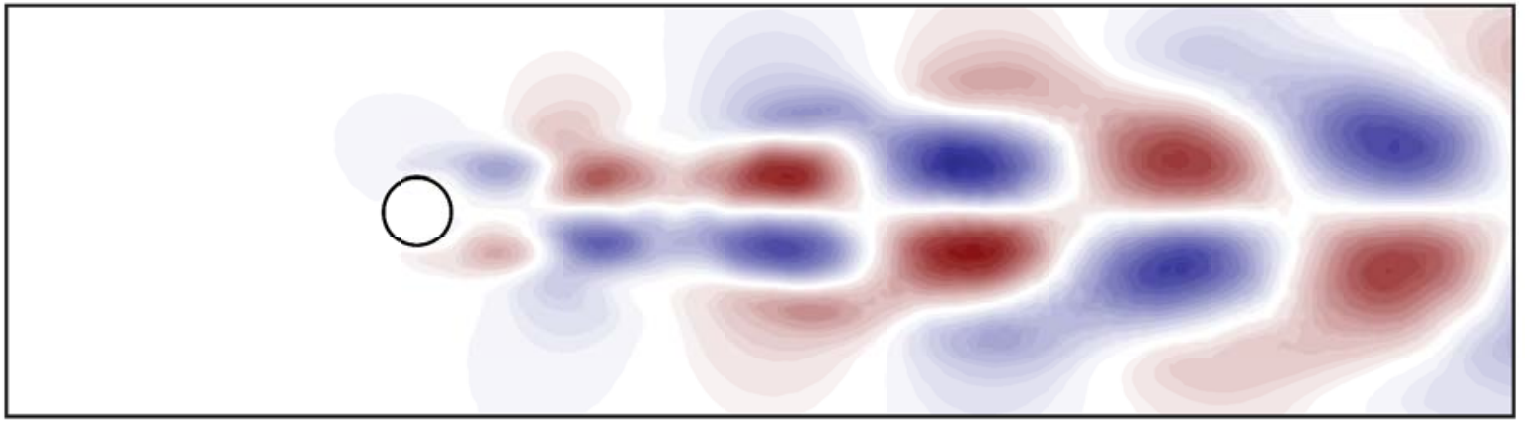}
\end{center}
\end{subfigure}
\begin{subfigure}{.48\textwidth}
\begin{center}
\includegraphics[width=2.4in]{figures/dns-svec/Fluctuation_exact_rank8.pdf}
\end{center}
\end{subfigure}

\vspace{1mm}

\begin{subfigure}{.48\textwidth}
\begin{center}
\includegraphics[width=2.4in]{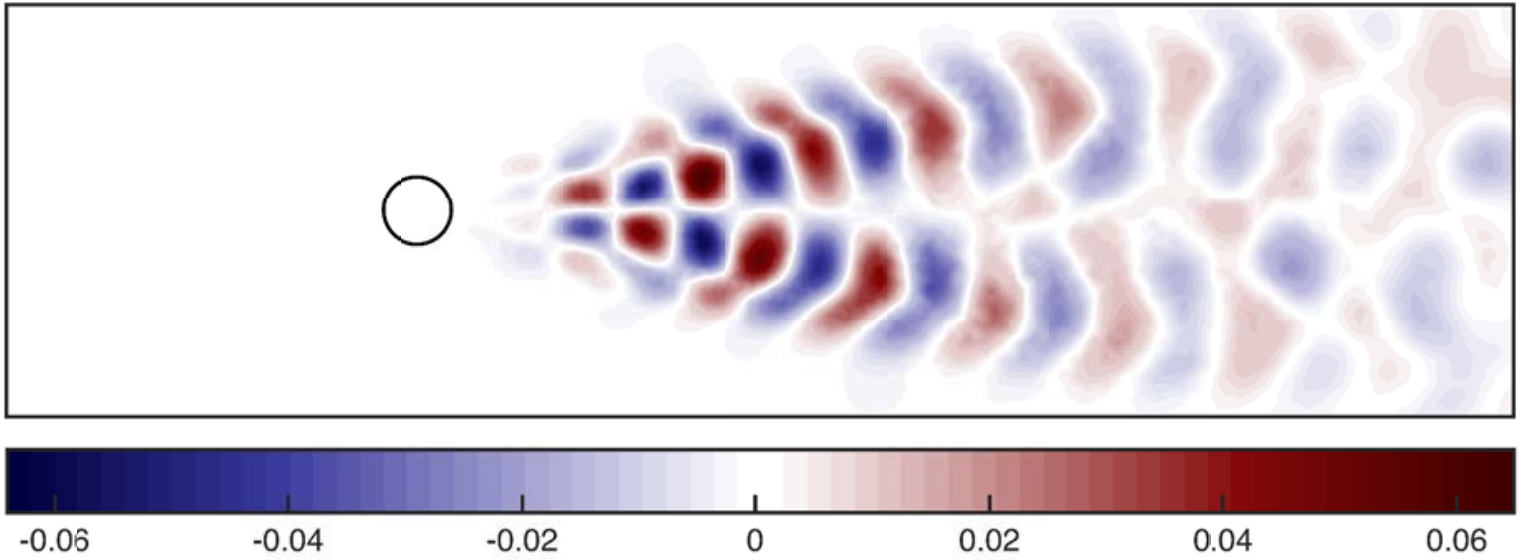}
\end{center}
\end{subfigure}
\begin{subfigure}{.48\textwidth}
\begin{center}
\includegraphics[width=2.4in]{figures/dns-svec/Fluctuation_exact_rank9.pdf}
\end{center}
\end{subfigure}

\vspace{1pc}

\caption{\textbf{Left singular vectors of \texttt{StreamVel} via [TYUC17].}
(Sparse maps, approximation rank $r = 10$, storage budget $T = 48 (m+n)$.)
The columns of the matrix \texttt{StreamVel} describe the fluctuations
of the streamwise velocity field about its mean value as a function of time.
From top to bottom, the panels show the first nine computed left singular
vectors of the matrix.
The \textbf{left-hand side} is computed using [TYUC17],
while the \textbf{right-hand side} is computed from the exact flow field.
The heatmap indicates the magnitude of the fluctuation.
See \cref{sec:flow-field}.}
\label{fig:flow-field-svec-tyuc}
\end{center}
\end{figure}

\end{document}